\documentclass{memo-l}

\usepackage{amsmath,amssymb,mathrsfs,tikz,multirow,hyperref}

\newtheorem{theorem}{Theorem}[chapter]
\newtheorem{lemma}[theorem]{Lemma}

\newtheorem{proposition}[theorem]{Proposition}
\newtheorem{corollary}[theorem]{Corollary}

\theoremstyle{definition}
\newtheorem{definition}[theorem]{Definition}

\theoremstyle{remark}
\newtheorem{remark}[theorem]{Remark}

\numberwithin{section}{chapter}
\numberwithin{equation}{chapter}
\numberwithin{table}{chapter}
\newcommand{\PSL}{\mathrm{PSL}}
\newcommand{\PSU}{\mathrm{PSU}}
\newcommand{\SU}{\mathrm{SU}}
\newcommand{\GU}{\mathrm{GU}}
\newcommand{\PGL}{\mathrm{PGL}}
\newcommand{\GL}{\mathrm{GL}}
\newcommand{\SL}{\mathrm{SL}}
\newcommand{\PSp}{\mathrm{PSp}}
\newcommand{\Sp}{\mathrm{Sp}}
\newcommand{\SO}{\mathrm{SO}}
\newcommand{\GO}{\mathrm{GO}}
\newcommand{\Spin}{\mathrm{Spin}}
\newcommand{\POmega}{\mathrm{P}\Omega}
\newcommand{\cf}{\mathrm{cf}}

\newcommand{\Aut}{\mathrm{Aut}}
\newcommand{\Inn}{\mathrm{Inn}}
\newcommand{\F}{\mathbb{F}}
\newcommand{\Out}{\mathrm{Out}}
\newcommand{\gen}[1]{\langle#1\rangle}
\newcommand{\Hom}{\mathrm{Hom}}
\newcommand{\soc}{\mathrm{soc}}
\newcommand{\rad}{\mathrm{rad}}
\newcommand{\Ext}{\mathrm{Ext}}
\newcommand{\Alt}{\mathrm{Alt}}
\newcommand{\Sym}{\mathrm{Sym}}

\newcommand{\I}{\mathrm{i}}
\newcommand{\subs}{\subseteq}
\newcommand{\normal}{\trianglelefteq}

\newcommand{\bG}{\mathbf{G}}

\newcommand{\bX}{\mathbf{X}}

\begin{document}

\frontmatter

\newcommand{\ms}[1]{\mathscr{#1}}
\newcommand{\mc}[1]{\mathcal{#1}}
\newcommand{\HSpin}{\mathrm{HSpin}}
\renewcommand{\top}{\mathrm{top}}
\newcommand{\mb}[1]{\mathbf{#1}}
\newcommand{\red}[1]{\textcolor{red}{#1}}

\title{On medium-rank Lie primitive and maximal subgroups of exceptional groups of Lie type}
\author{David A. Craven}
\address{School of Mathematics, University of Birmingham, Birmingham, B15 2TT, United Kingdom}
\email{d.a.craven@bham.ac.uk}
\thanks{The author is a Royal Society University Research Fellow, and gratefully acknowledges the financial support of the Society.}

\date{7th December, 2019}




\begin{abstract}
We study embeddings of groups of Lie type $H$ in characteristic $p$ into exceptional algebraic groups $\mb G$ of the same characteristic. We exclude the case where $H$ is of type $\PSL_2$. A subgroup of $\mb G$ is \emph{Lie primitive} if it is not contained in any proper, positive-dimensional subgroup of $\mb G$.

With a few possible exceptions, we prove that there are no Lie primitive subgroups $H$ in $\mb G$, with the conditions on $H$ and $\mb G$ given above. The exceptions are for $H$ one of $\PSL_3(3)$, $\PSU_3(3)$, $\PSL_3(4)$, $\PSU_3(4)$, $\PSU_3(8)$, $\PSU_4(2)$, $\PSp_4(2)'$ and ${}^2\!B_2(8)$, and $\mb G$ of type $E_8$. No examples are known of such Lie primitive embeddings.

We prove a slightly stronger result, including stability under automorphisms of $\mb G$. This has the consequence that, with the same exceptions, any almost simple group with socle $H$, that is maximal inside an almost simple exceptional group of Lie type $F_4$, $E_6$, ${}^2\!E_6$, $E_7$ and $E_8$, is the fixed points under the Frobenius map of a corresponding maximal closed subgroup inside the algebraic group. 

The proof uses a combination of representation-theoretic, algebraic group-theoretic, and computational means.
\end{abstract}

\maketitle

\tableofcontents


\mainmatter

%

\chapter{Introduction}

This paper deals with four interconnected questions, progress on all of which can be achieved with the same analysis:
\begin{enumerate}
\item Classify the maximal subgroups of finite simple groups of Lie type;
\item Understand the subgroup structure of the simple algebraic groups;
\item Study the \emph{morphism extension problem}: given a simple algebraic group $\mb G$ and an algebraic group $\mb H$, and a Frobenius endomorphism $\sigma$ of $\mb H$, when does a homomorphism $\mb H^\sigma\to \mb G$ extend to a morphism $\mb H\to \mb G$ of algebraic groups;
\item Generalize Steinberg's restriction theorem for $\GL_n$ (that every simple module for a simple group of Lie type---in other words, an irreducible subgroup of $\GL_n$---is the restriction of one for the algebraic group) to arbitrary semisimple algebraic groups.
\end{enumerate}
Here we analyse the first problem for $G$ a finite exceptional group of Lie type and a potential maximal subgroup that is the normalizer of a simple group of Lie type $H$ in the same characteristic as $G$, but is not of type $\PSL_2$. We more or less give a complete answer for such maximal subgroups, leaving only eight possible maximal subgroups in $E_8$.

\medskip

The determination of the maximal subgroups of the exceptional groups of Lie type is one of the major outstanding problems in our understanding of the structure of the finite simple groups. This task has occupied the endeavours of mathematicians over several decades; this paper is the third in a series by the author, whose eventual goal is the completion of this determination.

By a result of Borovik \cite{borovik1989} and Liebeck--Seitz \cite{liebeckseitz1990}, all maximal subgroups of almost simple exceptional groups of Lie type that are not themselves almost simple groups are known, so we apply the classification of finite simple groups. Alternating maximal subgroups were nearly classified in \cite{craven2017}, and there has been more work since then (there are no examples in $F_4$, $E_6$ and ${}^2\!E_6$, and there is one conjugacy class of maximal $\Alt(7)$ in $E_7(5)$, but $\Alt(6)$ in $E_7(q)$ and $\Alt(6),\Alt(7)$ in $E_8$ are unresolved), and cross-characteristic Lie type and sporadic groups that embed in exceptional groups were determined in \cite{liebeckseitz1999}, with some analysis of these in various papers in characteristic $0$, and particularly \cite{litterickmemoir} in positive characteristic. (These will be treated in a later paper.) The remaining potential maximal subgroups, which form the bulk of the groups to be examined, are Lie type groups in the same characteristic as the exceptional group.

Broadly speaking, all embeddings of a group of Lie type $H=H(q)$ into an exceptional algebraic group $\mb G$ of characteristic $p\mid q$ are contained in proper, $\sigma$-stable, positive-dimensional subgroups of $\mb G$ if $q>9$ (and $q\neq 16$ for $H$ of type $\PSL_3$ and $\PSU_3$) or if the rank of $H$ is more than half of the rank of $\mb G$. If $H$ is of type $\PSL_2$, ${}^2\!B_2$ or ${}^2\!G_2$ then much larger bounds than $9$ are needed to guarantee the result, for example $q>2624$ for $\mb G=E_8$. (This was proved in \cite{liebeckseitz1998}.)

This paper does not deal with the case of $\PSL_2$ (which for $\mb G=F_4,E_6,E_7$, but not $E_8$, was considered in \cite{craven2015un}), but with all of the other groups $H$. Recall that a finite subgroup $H$ of an algebraic group is called \emph{Lie primitive} if it does not lie in a proper positive-dimensional subgroup.

We prove the following result. In fact, we prove something slightly stronger that involves stability under automorphisms of $\mb G$ (which is needed for the application to maximal subgroups), but this will suffice for applications to the morphism extension problem.

\begin{theorem}\label{thm:maximale8} Let $\mb G$ be the exceptional algebraic group $E_8$ in characteristic $p>0$, over an algebraically closed field. Let $H=H(q)$ be a simple group of Lie type with $p\mid q$, and suppose that $H$ is a subgroup of $\mb G$. Furthermore, suppose that $H$ does not have the same type as $\mb G$, and that $H$ is not $\PSL_2(q)$. If $N_{\mb G}(H)$ is Lie primitive in $\mb G$, then $H$ is one of the groups $\PSL_3(q)$ for $q=3,4$, $\PSU_3(q)$ for $q=3,4,8$,  $\PSp_4(2)'$, ${}^2\!B_2(8)$, and $\PSU_4(2)$.

Moreover, if $\sigma$ is a Frobenius endomorphism on $\mb G$ and $H\leq \mb G^\sigma$, then there is a $\sigma$-stable proper positive-dimensional subgroup $\mb X$ such that $N_{\mb G^\sigma}(H)\leq \mb X^\sigma$ unless $H$ is one of the above groups.
\end{theorem}
(Note that there are no known examples of such groups.)

We will define the slight strengthening of this result now. Let $\mb G$ be a simple, simply connected algebraic group and $\sigma$ a Frobenius endomorphism of $\mb G$. A subgroup $K$ of $G=\mb G^\sigma$ is \emph{strongly imprimitive} if there is some $\sigma$-stable positive-dimensional subgroup $\mb X$ of $\mb G$ such that $N_{\bar G}(K)$ is contained in $N_{\bar G}(\mb X^\sigma)$ for all almost simple groups $\bar G$ with socle $G/Z(G)$. Of course, if $H=N_{\bar G}(H)$ then strongly imprimitive is the same as Lie imprimitive. The theorem above is also true with strong imprimitivity in place of Lie imprimitivity.

The second statement of this theorem (with strong imprimitivity) yields an application to maximal subgroups of the finite groups $E_8(q)$, namely that the maximal subgroups of the form $N_{\bar G}(H)$ for $H$ a Lie type group in characteristic $p\mid q$ are known except for $H$ on the list above, and $H\cong \PSL_2(q_0)$ for $p\mid q_0$.

The author believes that one may be able to use the techniques in this  paper to solve the outstanding cases of $\PSU_4(2)$ and $\PSU_3(8)$, but that $\PSp_4(2)$ and ${}^2\!B_2(8)$ cannot be resolved using these methods. For the rest he is not sure.

We also consider the other exceptional groups $F_4$, $E_6$ and $E_7$, obtaining the same theorem as above, but with a very short list of possibilities. (We do not need to consider $G_2$ because, by work of Kleidman \cite{kleidman1988} and Cooperstein \cite{cooperstein1981} all maximal subgroups of the finite groups $G_2(q)$ are known, and hence it is easy to determine all Lie primitive subgroups. Indeed, we use this case as an example in Section \ref{sec:example}.) Because of \cite{craven2015un} we can also include $\PSL_2$ subgroups in the list. (In the case of $F_4$ and $p\geq 5$ this result was obtained by Magaard in \cite{magaardphd}, and for $E_6$ and $H\not\cong \PSL_2(11),\PSL_3(3),\PSU_3(3)$ this result was obtained by Aschbacher in \cite{aschbacherE6Vun}.)

\begin{theorem}\label{thm:exceptionalmaximal} Let $\mb G$ be the exceptional algebraic group $F_4$, $E_6$ or $E_7$ in characteristic $p>0$, over an algebraically closed field. Let $H=H(q)$ be a simple group of Lie type with $p\mid q$, and suppose that $H$ is a subgroup of $\mb G$. Furthermore, suppose that $H$ does not have the same type as $\mb G$. If $H$ is not strongly imprimitive in $\mb G$, then $\mb G=E_7$ and $H$ is one of $\PSL_2(7)$, $\PSL_2(8)$, and $\PSL_2(25)$.
\end{theorem}

(The group $\PSL_2(25)$ is always strongly imprimitive, and this case is resolved in work of the author and Alexander Ryba. It will form part of a partial classification of maximal subgroups of groups of type $E_7$. No example is known of the other two possible subgroups.)

The author believes that these theorems represent close to the limit of what can be achieved simply using the techniques developed in the series of papers \cite{craven2017,craven2015un} he has produced (others on $\PSL_2$ in $E_8$ and non-defining characteristic embeddings are planned). More or less these techniques involve studying the restriction of small-dimensional modules for $\mb G$ to $H$, the eigenvalues of semisimple elements and Jordan block structures of unipotent elements, and proper subgroups of $H$ known to be in positive-dimensional subgroups, to find a subspace of a module, normally either the Lie algebra $L(\mb G)$ or the minimal module for $\mb G$, that is stabilized by both $H$ and a positive-dimensional subgroup $\mb X$ of $\mb G$. Thus $\gen{\mb X,H}$ is not all of $\mb G$, so $H$ is contained in a proper positive-dimensional subgroup of $\mb G$. If one chooses the subspace judiciously, one gets the refined statement on strong imprimitivity. To push these theorems still further probably requires more detailed use of the subgroup structure of exceptional algebraic groups, and perhaps computer proofs that explicitly compute inside finite groups of Lie type, as has been done for certain other cases.

\medskip

The structure of this article is as follows: the next chapter sets up the notation and a few preliminary results on elements of algebraic groups, while results on the subgroup structure of exceptional algebraic groups are in Chapter \ref{chap:subgroupstructure}. Chapter \ref{chap:techniques} goes through the techniques used to prove the results in this article. We give information about modules for small-rank groups of Lie type over small fields, i.e., the possible subgroups $H$, in the following chapter, and then we move on to the proof. Chapter \ref{chap:rank4ine8} considers rank-$4$ subgroups of $E_8$, Chapter \ref{chap:rank3ine8} deals with subgroups of rank $3$, and Chapter \ref{chap:rank2ine8} with subgroups of rank $2$. After this, Chapters \ref{chap:subsine7}, \ref{chap:subsine6} and \ref{chap:subsinf4} consider the subgroups of $E_7$, $E_6$ and $F_4$ respectively. Chapter \ref{chap:diff} gives the proofs of the difficult cases that were too long to fit in the general flow of the text. For the cases for $E_6$ that act irreducibly on $M(E_6)$ (or whose normalizer does), so $\PSL_3(3)$, $\PSU_3(3)$ and ${}^2\!G_2(3)'$, we have to approach them using completely different methods, and this takes place in Chapter \ref{chap:trilinear}.

\medskip

There are a large number of supplementary materials that come with this article, consisting of Magma programs, and also list of traces of semisimple elements of various orders on low-dimensional modules. For many sections, these prove the assertions about module structures, actions of unipotent elements, and so on. In Chapter \ref{chap:trilinear} they also provide the matrices of the groups involved and prove the assertions in that chapter.

\medskip

The author would like to thank the anonymous referee for their fast, detailed and thoughtful report, which has greatly improved the exposition of this paper.

\newpage
\chapter{Notation and preliminaries}

Our aim for this chapter is to collate some notation and some general results that have been used in the past, particularly in \cite{craven2017,craven2015un}, to prove that a given subgroup is not maximal in a finite exceptional group of Lie type.

For the rest of this paper, $p$ will be a prime and $\mb G$ will denote a simple, simply connected exceptional algebraic group over an algebraically closed field $k$ of characteristic $p$. Let $\sigma$ denote a Frobenius endomorphism of $\mb G$ and let $G=\mb G^\sigma$ denote the fixed points of $\mb G$ under $\sigma$, a finite group of Lie type. The precise types of groups $G$ that we are interested in are $F_4(q)$, $E_6(q)$, ${}^2\!E_6(q)$, $E_7(q)$ and $E_8(q)$. Note that these are the simply connected groups, so for example $E_7(q)$ has a non-trivial centre for $q$ odd. Write $\ms X$ for the set of maximal closed, positive-dimensional subgroups of $\mb G$, and let $\ms X^\sigma$ denote the set of fixed points $\bX^\sigma$ for $\bX$ a $\sigma$-stable member of $\ms X$. The members of the set $\ms X^\sigma$, together with subgroups of the same type as $\mb G$ (e.g., $E_6(q_0)$ and ${}^2\!E_6(q_0)$ inside $E_6(q)$) form almost all of the maximal subgroups of $G$. If $\bar G$ denotes an almost simple group such that $G/Z(G)\cong F^*(\bar G)$ via some isomorphism $\phi$, we extend our notation of $\ms X^\sigma$ to denote the set of subgroups $N_{\bar G}((Z(G)\mb X^\sigma/Z(G))\phi)$ for $\bX\in \ms X^\sigma$. For $\bX$ an algebraic group, write $\mb X^\circ$ for the connected component of the identity. Write $o(x)$ for the order of an element $x$ of a group.

A subgroup of $\mb G$ is \emph{Lie primitive} if it does not lie in any proper, positive-dimensional subgroup of $\mb G$, it is called \emph{$\mb G$-irreducible} if it does not lie in any proper parabolic subgroup of $\mb G$, and \emph{$\mb G$-completely reducible} (often abbreviated to $\mb G$-cr) if, whenever it lies in a parabolic subgroup, it lies in a corresponding Levi subgroup. Of course, Lie primitive subgroups are $\mb G$-irreducible, and both types of subgroup are vacuously $\mb G$-completely reducible.

For a given (simple) subgroup $H$ of $\mb G$, we prove that $H$ is contained in a member of $\ms X$, that if $H$ is contained in the fixed points $G=\mb G^\sigma$ then $H$ is contained in a member of $\ms X^\sigma$, and furthermore that $N_{\bar G}(H)$ is contained inside a member of the collection $\ms X^\sigma$ for any $\bar G$. This is a strong form of Lie imprimitivity, which we will unimaginatively refer to as \emph{strong imprimitivity}. Normally we will prove this using general theorems---for example, if $H$ centralizes a line on the non-trivial composition factor of the adjoint module then $H$ is strongly imprimitive (Lemma \ref{lem:fix1space})---but occasionally we will have to prove strong imprimitivity via the stabilization of more complicated subspaces than lines.

\medskip

Our notation for modules is the same as in the previous papers in this series \cite{craven2017,craven2015un}. Write $\soc(V)$ for the socle of a module $V$ and $\soc^i(V)$ for the preimage in $V$ of the socle $\soc(V/\soc^{i-1}(V))$ of the quotient module $V/\soc^{i-1}(V)$, $\rad(V)$ for the radical of a module $V$, $\top(V)$ for the top of a module $V$, i.e., the quotient module $V/\rad(V)$, $V{\uparrow^H}$ for the induction of $V$ to $H$ and $V{\downarrow_H}$ for the restriction of $V$ to $H$. Write $V^H$ for the fixed points of $V$ under the action of $H$. Write $P(V)$ for the projective cover of $V$. The \emph{heart} of a module $V$ is the quotient module $\rad(V)/\soc(\rad(V))$. Write $k$ or $1$ for the trivial $kH$-module. If $V$ is a $kH$-module, we often consider its $1$-cohomology, i.e., the module $\Ext_{kH}^1(k,V)$. For brevity we sometimes conflate the $1$-cohomology itself with its dimension, as it is only the dimension of $1$-cohomology that is of interest to us.

In order to save space, when we give a module's socle layers, we distinguish between the layers using `/', so that a module with socle $A$ and second socle $B\oplus C$ would be written $B,C/A$. We also use the notation $B/A$ for the quotient of $B$ by $A$, but whenever we do so we say that we mean the quotient. Write $\cf(V)$ for the multiset of composition factors of $V$.

If $I$ is a set of irreducible $kH$-modules, write $I'$ for the set of (isomorphism class representatives of) irreducible $kH$-modules that are not in $I$. The \emph{$I$-radical} of a $kH$-module $V$ is the largest submodule of $V$ whose composition factors lie in $I$, and the \emph{$I$-residual} is the smallest submodule whose quotient only has composition factors in $I$. The \emph{$I$-heart} of $V$ is the $I'$-residual of the quotient of $V$ by its $I'$-radical. (Thus the socle and top of the $I$-heart consist solely of modules in $I$. Note that the $I$-heart is only an analogue of the heart, and there is no choice of $I$ for which the $I$-heart coincides with the heart.) Write $\Lambda^i(V)$ and $S^i(V)$ for the $i$th exterior and symmetric power of $V$ respectively. Let $V^*$ denote the dual of $V$, and if we mean either $V$ or its dual, write $V^\pm$.

An important exception to this is when we consider submodules of projectives. We will often be interested in the $I$-radical of a projective cover $P(V)$ of an irreducible $kH$-module $V$, but with $V$ not itself in $I$, so of course the $I$-radical is $0$. What we want is the preimage in $P(V)$ of the $I$-radical of the quotient module $P(V)/V$, but rather than writing out that whole phrase, we simply write the $I$-radical of $P(V)$, accepting that this means the slightly different module above.

For a given dominant weight $\lambda$, write $L(\lambda)$ for the simple module and $W(\lambda)$ for the Weyl module with that highest weight. Denote by $M(\mb G)$ a non-trivial Weyl module of smallest dimension for $\mb G$, and by $L(\mb G)$ the adjoint module for $\mb G$. For most primes $M(\mb G)$ and $L(\mb G)$ are irreducible, but for $p=2,3$ and certain $\mb G$ they have a single trivial composition factor and a single non-trivial composition factor. Moreover, if $p=3$ and $\mb G=G_2$, or $p=2$ and $\mb G=F_4$, then $L(\mb G)$ does not have a unique non-trivial composition factor. In these cases the adjoint module has two composition factors of the same dimension, one of which is $M(\bG)$. If $M(\mb G)$ has a unique non-trivial composition factor, then denote it by $M(\mb G)^\circ$, and similarly for $L(\mb G)^\circ$, and set $L(\mb G)^\circ=L(\mb G)$ otherwise. The module $M(\bG)^\circ$ is called the \emph{minimal module} in many of our references, and here. We give the description of these in terms of highest weights and the dimension of their simple constituent now. (The labelling of highest weights is consistent with \cite{bourbakilie2}, and is used by most of our references.)
\begin{center}\begin{tabular}{ccccc}
\hline $\mb G$ & $M(\mb G)$ & $L(\mb G)$ & $\dim(M(\mb G)^\circ)$ & $\dim(L(\mb G)^\circ)$
\\ \hline $G_2$ & $W(10)$ & $W(01)$ & $7-\delta_{p,2}$ & $14$
\\ $F_4$ & $W(\lambda_4)$ & $W(\lambda_1)$ & $26-\delta_{p,3}$ & $52$
\\ $E_6$ & $L(\lambda_1)$ or $L(\lambda_6)$ & $W(\lambda_2)$ & $27$ & $78-\delta_{p,3}$
\\ $E_7$ & $L(\lambda_7)$ & $W(\lambda_1)$ & $56$ & $133-\delta_{p,2}$
\\ $E_8$ & N/A & $L(\lambda_8)$ & N/A & $248$
\\\hline
\end{tabular}\end{center}

For both $M(\mb G)$ and $L(\mb G)$, the decomposition of the action of a unipotent element $u$ into Jordan blocks was given in \cite{lawther1995}, although when $M(\mb G)$ or $L(\mb G)$ has a trivial composition factor the action on $M(\mb G)^\circ$ or $L(\mb G)^\circ$ was not recorded there, but was given in \cite[Lemmas 6.1--6.3]{craven2015un}. Table \ref{t:unipotentF4} lists the unipotent classes for $F_4$ in characteristic $3$ and their actions on $M(F_4)^\circ$, $M(F_4)$ and $M(E_6)$. Table \ref{t:unipotentE6} lists some unipotent classes of $E_6$ in characteristic $3$, and for all other classes one obtains the action on $L(E_6)^\circ$ by removing a block of size $1$ from the action on $L(E_6)$. For $E_7$ in characteristic $2$, one always obtains the action of a unipotent element on $L(E_7)^\circ$ by removing a block of size $1$ from the action on $L(E_7)$. We use the Bala--Carter--Pommerening label for unipotent classes as stated in \cite{lawther1995} without any further comment throughout this text.

\medskip

We will now define the concept of encapsulation, which we need occasionally in our arguments. If $(a_1,\dots,a_r)$ and $(b_1,\dots,b_s)$ are sequences with $a_i\geq a_{i+1}>0$ and $b_i\geq b_{i+1}>0$, then $(a_i)$ \emph{encapsulates} $(b_i)$ if $r\geq s$ and $a_i\geq b_i$ for all $1\leq i\leq s$. If $u$ is a unipotent element of a group $G$ and $V$ and $W$ are $kG$-modules, then the action of $u$ on $V$ \emph{encapsulates} that of $W$ if, when written as non-increasing sequences, the sizes of Jordan blocks of the action of $u$ on $V$ encapsulates that of $W$. (Note that this is not the same as the dominance order. Indeed, no partition encapsulates any other partition of the same size.)

The reason we introduce this is the following easy, but important, lemma.

\begin{lemma}\label{lem:encapsulates} Let $G$ be a finite group and let $u$ be a $p$-element of $G$. Let $V$ be a $kG$-module and let $W$ be a subquotient of $V$. The action of $u$ on $V$ encapsulates that of $W$.
\end{lemma}
\begin{proof} We may of course assume that $G=\gen u$. Let $V$ be a minimal counterexample to the statement. As encapsulation is transitive, we may assume that $W$ is a submodule or quotient of dimension $1$ less than that of $V$. By taking duals if necessary, $W$ is a submodule of $V$. By choice of minimal counterexample, we may remove any summand of $V$ that is contained in $W$. Thus $W$ is indecomposable (equivalent to having $1$-dimensional socle), and in this case it is clear from Jordan normal form that there are exactly two modules of dimension $n+1$ that contain a block of size $n$ as a submodule, namely $(n+1)$ and $(n,1)$, both of which encapsulate $(n)$.
\end{proof}

\begin{table}
\begin{center}
\begin{tabular}{lccc}
\hline Class in $F_4$ & Act. on $M(F_4)^\circ$ & Act. on $25/1$ & Act. on $M(E_6)=1/25/1$
\\\hline $A_1$ & $2^6,1^{13}$ & $2^6,1^{14}$ & $2^6,1^{15}$
\\ $\tilde A_1$ & $3,2^8,1^6$ & $3,2^8,1^7$ & $3,2^8,1^8$
\\ $A_1+\tilde A_1$ & $3^3,2^6,1^4$ & $3^3,2^6,1^5$ & $3^3,2^6,1^6$
\\ $A_2$ & $3^6,1^7$ & $3^6,1^8$ & $3^6,1^9$
\\ $A_2+\tilde A_1$ & $3^7,2^2$ & $3^7,2^2,1$ & $3^7,2^2,1^2$
\\ $\tilde A_2$, $\tilde A_2+A_1$ & $3^8,1$ & $3^8,2$ & $3^9$
\\ \hline $B_2$ & $5,4^4,1^4$ & $5,4^4,1^5$ & $5,4^4,1^6$
\\ $C_3(a_1)$ & $5^2,4^2,3,2^2$ & $5^2,4^2,3,2^2,1$ & $5^2,4^2,3,2^2,1^2$
\\ $F_4(a_3)$ & $5^3,3^3,1$ & $5^3,3^3,1^2$ & $5^3,3^3,1^3$
\\ $B_3$ & $7^3,1^4$ & $7^3,1^5$ & $7^3,1^6$
\\ $C_3$, $F_4(a_2)$ & $9,6^2,3,1$ & $9,6^2,3,2$ & $9,6^2,3^2$
\\ $F_4(a_1)$ & $9^2,7$ & $9^2,7,1$ & $9^2,7,1^2$
\\ \hline $F_4$ & $15,9,1$ & $15,9,2$ & $15,9,3$
\\ \hline
\end{tabular}
\end{center}
\caption{Actions of unipotent elements on $M(F_4)^\circ$ and its extensions for $F_4$ in characteristic $3$. (Horizontal lines separate elements of different orders.)}
\label{t:unipotentF4}\end{table}

\begin{table}
\begin{center}
\begin{tabular}{lcc}
\hline Class in $E_6$ & Action on $L(E_6)^\circ$ & Action on $L(E_6)$
\\\hline $2A_2$&$3^{23},1^8$&$3^{23},2,1^7$
\\$2A_2+A_1$&$3^{24},2^2,1$&$3^{24},2^3$
\\$A_5$&$9^3,8^2,6^4,3^2,1^4$&$9^3,8^2,6^4,3^2,2,1^3$
\\$E_6(a_3)$&$9^4,7,6^4,3^3,1$&$9^4,7,6^4,3^3,2$
\\$E_6(a_1)$&$9^8,5$&$9^8,6$
\\$E_6$&$19,15^2,9^3,1$&$19,15^2,9^3,2$
\\\hline
\end{tabular}
\end{center}
\caption{Actions of unipotent elements on $L(E_6)^\circ$ and $L(E_6)$ for $E_6$ in characteristic $3$, where one does not obtain the former from the latter by removing a trivial Jordan block}
\label{t:unipotentE6}
\end{table}

Looking at the tables in \cite{lawther1995}, one sees that for all but finitely many primes the Jordan blocks of a given unipotent class on $M(\bG)$, the `generic' action, and similarly with $L(\bG)$. If a unipotent class acts on one of these modules with the generic action, then $u$ is said to be \emph{generic} for this module. Note that if $u$ is generic for one of $M(\bG)$ and $L(\bG)$ then it need not be generic for the other. The non-generic classes of elements of order $p$ for $\mb G=E_8$ and $p=3,5,7$ are given in Tables \ref{t:unipe8p3}, \ref{t:unipe8p5} and \ref{t:unipe8p7}, and are transcribed from \cite{lawther1995}, given here for the reader's convenience. (We also include unipotent elements of orders $4$ and $8$ in Tables \ref{t:unipe8p4} and \ref{t:unipe8p8}.)
\begin{small}\begin{table}
\begin{center}
\begin{tabular}{cc}
\hline Class & Action on $L(E_8)$
\\\hline $3A_1$ & $3^{31},2^{50},1^{55}$
\\ $4A_1$ & $3^{44},2^{40},1^{36}$
\\ $A_2$ & $3^{57},1^{77}$
\\ $A_2+A_1$ & $3^{58},2^{20},1^{34}$
\\ $A_2+2A_1$ & $3^{65},2^{16},1^{21}$
\\ $A_2+3A_1$ & $3^{70},2^{14},1^{10}$
\\ $2A_2$ & $3^{78},1^{14}$
\\ $2A_2+A_1$ & $3^{79},2^2,1^7$
\\ $2A_2+2A_1$ & $3^{80},2^4$
\\ \hline
\end{tabular}\caption{Non-generic unipotent classes of elements of order $3$ in $E_8$ in characteristic $3$}\label{t:unipe8p3}
\end{center}
\end{table}

\begin{table}
\begin{center}
\begin{tabular}{cccc}
\hline Class & Action on $L(E_8)$ & Class & Action on $L(E_8)$
\\\hline $A_2$ & $4^2,3^{54},1^{78}$ & $A_3+2A_1$ & $4^{46},2^{30},1^4$
\\ $A_2+A_1$ & $4^{14},3^{30},2^{34},1^{34}$ & $A_3+A_2$ & $4^{50},3^{10},2^6,1^6$
\\ $A_2+2A_1$ & $4^{22},3^{14},2^{52},1^{14}$ & $D_4(a_1)$ & $4^{54},3^2,1^{26}$
\\ $A_2+3A_1$ & $4^{26},3^6,2^{62},1^2$  & $A_3+A_2^{(2)}$ & $4^{54},3^2,2^{10},1^6$
\\ $2A_2$ & $4^{28},3^{36},1^{28}$& $D_4(a_1)+A_1$ & $4^{54},3^2,2^{10},1^6$
\\ $2A_2+A_1$ & $4^{40},3^{12},2^{18},1^{16}$& $A_3+A_2+A_1$ & $4^{54},3^2,2^{12},1^2$
\\ $2A_2+2A_1$ & $4^{44},3^4,2^{28},1^4$ & $D_4(a_1)+A_2$ & $4^{56},3^8$
\\ $A_3$ & $4^{46},2^{10},1^{44}$& $2A_3$ & $4^{60},2^4$
\\ $A_3+A_1$ & $4^{46},2^{24},1^{16}$ &&
\\ \hline
\end{tabular}\caption{Unipotent classes of elements of order $4$ in $E_8$ in characteristic $2$}\label{t:unipe8p4}
\end{center}
\end{table}

\begin{table}
\begin{center}
\begin{tabular}{cccc}
\hline Class & Action on $L(E_8)$ & Class & Action on $L(E_8)$
\\\hline $A_3$ & $5^{13},4^{32},1^{55}$ & $D_4(a_1)+A_2$ & $5^{36},3^{20},1^8$
\\ $2A_2+A_1$ & $5^{14},4^{14},3^{23},2^{18},1^{17}$ & $2A_3$ & $5^{38},4^{12},1^{10}$
\\ $2A_2+2A_1$ & $5^{18},4^{12},3^{20},2^{20},1^{10}$ & $A_4$ & $5^{45},1^{23}$
\\ $A_3+A_1$ & $5^{21},4^{16},3^9,2^{14},1^{24}$ & $A_4+A_1$ & $5^{45},3,2^6,1^8$
\\ $A_3+2A_1$ & $5^{25},4^{10},3^{14},2^{14},1^{13}$ & $A_4+2A_1$ & $5^{45},3^4,2^4,1^3$
\\ $D_4(a_1)$ & $5^{29},3^{25},1^{28}$ & $A_4+A_2$ & $5^{46},3^5,1^3$
\\ $D_4(a_1)+A_1$ & $5^{29},4^6,3^{14},2^{14},1^9$ & $A_4+A_2+A_1$ & $5^{46},4^2,3^2,2^2$
\\ $A_3+A_2$ & $5^{30},4^8,3^{13},2^8,1^{11}$ & $A_4+A_3$ & $5^{48},4^2$
\\ $A_3+A_2+A_1$ & $5^{32},4^8,3^{10},2^{10},1^6$ & & 
\\ \hline
\end{tabular}\caption{Non-generic unipotent classes of elements of order $5$ in $E_8$ in characteristic $5$}
\label{t:unipe8p5}
\end{center}
\end{table}

\begin{table}
\begin{center}\begin{tabular}{ccccccc}
\hline Class & Action on $L(E_8)$ & Class & Action on $L(E_8)$
\\ \hline $A_4$ & $7^{13},5^{20},3^{11},1^{24}$ & $D_5(a_1)$ & $7^{28},3^7,2^8,1^{15}$
\\ $A_4+A_1$ & $7^{13},6^6,5^8,4^8,3^8,2^8,1^9$ & $D_5(a_1)+A_1$ & $7^{28},4^2,3^6,2^{10},1^6$
\\ $2A_3$ & $7^{14},5^{10},4^{16},3^6,2^4,1^{10}$ & $D_4+A_2$ & $7^{28},5,3^{13},1^8$
\\ $A_4+2A_1$ & $7^{15},6^4,5^8,4^8,3^9,2^8,1^4$ & $D_5(a_1)+A_2$ & $7^{28},5^3,4^2,3^6,2^4,1^3$
\\ $A_4+A_2$ & $7^{19},5^{11},3^{18},1^6$ & $E_6(a_3)$ & $7^{28},5^7,3,1^{14}$
\\ $A_4+A_2+A_1$ & $7^{19},6^2,5^7,4^8,3^7,2^6,1^3$ & $E_6(a_3)+A_1$ & $7^{28},6^2,5^3,4^2,3^2,2^4,1^3$
\\ $A_5$ & $7^{21},6^{14},1^{17}$ &  $D_6(a_2)$ & $7^{29},5^4,4^4,3,1^6$
\\ $A_4+A_3$ & $7^{24},6^2,5^3,4^6,3^6,2^4,1^3$ &$E_7(a_5)$ & $7^{29},6^2,5,4^4,3^3,1^3$
\\ $A_5+A_1$ & $7^{25},6^6,5^4,3,2^4,1^6$ & $E_8(a_7)$ & $7^{30},5^4,3^6$&
\\ $D_4$ & $7^{28},1^{52}$ & $A_6$ & $7^{35},1^3$
\\ $D_4+A_1$ & $7^{28},3,2^{14},1^{21}$ & $A_6+A_1$ & $7^{35},3$
\\ \hline
\end{tabular}\end{center}
\caption{Non-generic unipotent classes of elements of order $7$ in $E_8$ in characteristic $7$}
\label{t:unipe8p7}
\end{table}

\begin{table}
\begin{center}
\begin{tabular}{cccc}
\hline Class & Action & Class & Action
\\ \hline $D_4$ & $8,6^{27},2^{26},1^{26}$ & $D_6(a_2)$ & $8^{14},6^{18},2^{12},1^4$
\\ $D_4+A_1$ & $8,6^{27},2^{36},1^6$ & $E_6(a_3)$ & $8^{20},5^{12},4^2,3^2,1^{14}$
\\ $A_4$ & $8^2,7^{10},5^{20},4^2,3^{10},1^{24}$ & $E_6(a_3)+A_1$ & $8^{20},6^4,5^4,4^6,3^2,2^6,1^2$
\\ $A_4+A_1$ & $8^4,7^6,6^8,5^8,4^{10},3^6,2^{10},1^8$ & $E_7(a_5)$ & $8^{20},6^6,4^8,3^2,2^6,1^2$
\\ $A_4+2A_1$ & $8^6,7^2,6^{12},5^4,4^{14},3^2,2^{16}$ & $E_8(a_7)$ & $8^{20},7^4,4^{12},3^4$
\\ $A_4+A_2$ & $8^6,7^{10},5^4,4^{20},3^8,1^6$ & $A_6$ & $8^{22},7^6,5^4,3^2,1^4$
\\ $D_4+A_2$ & $8^7,6^{15},4^{14},3^6,2^{14}$ & $A_6+A_1$ & $8^{24},7^2,6^4,4^2,3^2,2^2$
\\ $A_4+A_2+A_1$ & $8^{10},7^2,6^4,5^4,4^{24},2^6,1^2$ & $D_5$ & $8^{27},6,2^6,1^{14}$
\\ $D_5(a_1)$ & $8^{11},6^7,4^{26},1^{14}$ & $D_5+A_1$ & $8^{27},6,2^{12},1^2$
\\ $D_5(a_1)+A_1$ & $8^{11},6^7,4^{26},2^6,1^2$ & $D_5+A_2$ & $8^{27},6,4^4,3^2,2^2$
\\ $D_4+A_2^{(2)}$ & $8^{11},6^7,4^{26},2^6,1^2$ & $D_6(a_1)$ & $8^{27},6,4^5,2,1^4$
\\ $A_4+A_3$ & $8^{12},7^2,6^2,4^{30},3^2$ & $E_7(a_4)$ & $8^{27},6,4^5,2,1^4$
\\ $D_5(a_1)+A_2$ & $8^{13},6^3,4^{30},3^2$ & $D_5+A_2^{(2)}$ & $8^{27},6,4^5,2^3$
\\ $A_5$ & $8^{14},6^{18},2^6,1^{16}$ & $D_7(a_2)$ & $8^{28},6^4$
\\ $A_5+A_1$ & $8^{14},6^{18},2^{12},1^4$ & $A_7$ & $8^{30},4^2$
\\ \hline
\end{tabular}\end{center}
\caption{Unipotent classes of elements of order $8$ in $E_8$ in characteristic $2$}
\label{t:unipe8p8}
\end{table}

\begin{table}
\begin{center}
\begin{tabular}{cccc}
\hline Class & Action on $L(\lambda_1)$ & Action on $L(\lambda_4)$, $p=3$ & Action on $L(\lambda_4)$, $p=5$
\\ \hline $A_1$ & $2^2,1^6$ & $2^4,1^8$ & $2^4,1^8$
\\ $D_2$ & $3,1^7$ & $2^8$ & $2^8$
\\ $2A_1$ & $2^4,1^2$ & $3,2^4,1^5$ & $3,2^4,1^5$
\\ $A_1+D_2$ & $3,2^2,1^3$ & $3^2,2^4,1^2$ & $3^2,2^4,1^2$
\\ $A_2$ & $3^2,1^4$ & $3^4,1^4$ & $3^4,1^4$
\\ $D_3$ & $5,1^5$ & $4^4$ & $4^4$
\\ $A_2+A_1$ & $3^2,2^2$ & $3^4,2,1^2$ & $4,3^2,2^2,1^2$
\\ $A_2+D_2$ & $3^3,1$ & $3^4,2^2$ & $4^2,2^4$
\\ $A_3$ & $4^2,1^2$ & $5,4^2,1^3$ & $5,4^2,1^3$
\\ $A_1+D_3$ & $5,2^2,1$ & $5,4^2,3$ & $5,4^2,3$
\\ $D_4(a_1)$ & $5,3,1^2$ & $5^2,3^2$ & $5^2,3^2$
\\ $D_4$ & $7,1^3$ & $7^2,1^2$ & $7^2,1^2$
\\ $A_4$ & $5^2$ & $7,5,3,1$ & $5^3,1$
\\ $D_5(a_1)$ & $7,3$ & $8,6,2$ & $8,6,2$
\\ $D_5$ & $9,1$ & $9,7$ & $11,5$
\\ \hline
\end{tabular}
\end{center}
\caption{Unipotent classes in $D_5$, with actions on $L(\lambda_1)$ and $L(\lambda_4)$ for $p=3,5$}
\label{t:unipd5}
\end{table}

\begin{table}
\begin{center}
\begin{tabular}{ccc}
\hline Class & Eigenspace dimensions on $L(E_8)$ & Centralizer
\\ \hline 2A & $136,112$ & $E_7A_1$
\\ 2B & $120,128$ & $D_8$
\\\hline 3A & $80,84,84$ & $A_8$
\\ 3B & $86,81,81$ & $E_6A_2$
\\ 3C & $92,78,78$ & $D_7T_1$
\\ 3D & $134,57,57$ & $E_7T_1$
\\ \hline 5A[2] & $64,56,36,36,56$ & $A_7T_1$
\\ 5B[2] & $52,49,49,49,49$ & $A_6A_1T_1$
\\ 5C & $48,50,50,50,50$ & $A_4A_4$
\\ 5D[2] & $54,51,46,46,51$ & $A_2D_5T_1$
\\ 5E[2] & $82,54,29,29,54$ & $E_6A_1T_1$
\\ 5F[2] & $92,64,14,14,64$ & $D_7T_1$
\\ 5G & $68,45,45,45,45$ & $D_6T_2$
\\ 5H[2] & $134,56,1,1,56$ & $E_7T_1$
\\ \hline
\end{tabular}
\end{center}
\caption{Classes of semisimple elements of small orders in $E_8$ and their centralizers. (The number in brackets is the number of conjugacy classes in the subgroup generated by an element.)}
\label{t:semie8}
\end{table}\end{small}

\medskip

In this work we will be considering embeddings of finite groups into algebraic groups $\mb G$. A significant restriction on a homomorphism $H\to \mb G$ is to consider modules $V$ for $\mb G$, and compute the possible composition factors for the restriction $V{\downarrow_H}$ of $V$ to $H$. The traces of semisimple elements of $H$ on $V$ must match those of semisimple classes of $\mb G$, i.e., to each $H$-conjugacy class $C$ we must assign a $\mb G$-conjugacy class $\bar C$ of semisimple elements such that the traces of elements of $C$ and $\bar C$ on $V$ are the same. Moreover, by requiring that the power map is consistent with this assignment, i.e., that if $x\in C$ then the class assigned to $x^n$ is the class containing $y^n$ for $y\in \bar C$, we obtain a slightly stronger restriction than merely consistency with traces. (This is equivalent to requiring that the eigenvalues of $x$ and $y$ on $V$ coincide.)

A putative set of composition factors for $V{\downarrow_H}$ is called \emph{conspicuous} if all semisimple elements satisfy the requirements above. In practice, we will sometimes only be able to check this condition on some element orders, as we only have data on the eigenvalues for some orders, in which case we will say that the set is conspicuous for some set of semisimple elements of $H$, for example of all elements of order at most $m$ for some $m$. Our first restriction when proving our results will normally be to calculate the conspicuous sets of composition factors for $M(\mb G){\downarrow_H}$ or $L(\mb G){\downarrow_H}$.

The conspicuous sets of composition factors for $V{\downarrow_H}$ may be split into two collections. A \emph{warrant} for a set of composition factors for $V{\downarrow_H}$ is a proper positive-dimensional subgroup $\mb X<\mb G$ and a homomorphism $\phi:H\to \mb X$ such that the composition factors of $V{\downarrow_{H\phi}}$ coincide with our given set. If $H<\mb G$ and the composition factors of $V{\downarrow_H}$ are unwarranted, then $H$ is necessarily Lie primitive in $\mb G$.

Thus if we are presented with a warranted set of composition factors we need to show that all corresponding copies of $H$ in $\mb G$ must be (strongly) imprimitive in $\mb G$, and if we have an unwarranted set of composition factors we need to show that there is no copy of $H$ in $\mb G$ with those factors. In practice, in the second case we will often show that $H$ must be imprimitive instead, and therefore indirectly show that $H$ does not exist.

While for small groups $H$ and $\mb G$ it is possible to perform the computations needed to determine the conspicuous sets of composition factors by hand, for larger $H$ and $\mb G$, and in practice even for small $H$ and $\mb G$, we will use a computer to avoid errors in hand solutions to linear algebra problems.

\medskip

A subgroup $H\leq \mb G$ is a \emph{blueprint} for a module $V$ if there exists a positive-dimensional subgroup $\mb X$ of $\mb G$ stabilizing the same subspaces of $V$ as $H$. An element $x\in \mb G$ is called a blueprint for $V$ if $\langle x\rangle$ is a blueprint for $V$. Of course, if $L\leq H\leq \mb G$ and $L$ is a blueprint for $V$, so is $H$.

If $H$ acts irreducibly on $V$ then $H$ is always a blueprint, but this information is often not helpful. From \cite{liebeckseitz2004a}, we know all possibilities for a finite subgroup acting irreducibly on either $M(\mb G)^\circ$ or $L(\mb G)^\circ$, and since in this article we are interested in $H\leq \mb G$ with $H$ a Lie type group in characteristic $p$ other than $\PSL_2(q)$, the only possibilities for $H$ are $\PSL_3(3)$, $\PSU_3(3)$ or ${}^2G_2(3)'$ with $p=3$ and $\mb G=E_6$. These cannot be dealt with using blueprints, and indeed cannot be dealt with at all using the standard methods in this paper, and we have to resort to more underhand means. The case $H\cong {}^2\!G_2(3)$ and its derived subgroup was proved to always lie in an irreducible $G_2$ subgroup of $E_6$ in \cite{aschbacherE6Vun}, but the proof involves an explicit analysis of the trilinear form on $M(E_6)$, and there does not appear to be a way to deal with it representation-theoretically. These three groups will be considered in Chapter \ref{chap:trilinear}.

\medskip

If $u$ is a unipotent element, then there is an easy sufficient criterion for $u$ to be a blueprint for either $M(\mb G)$ or $L(\mb G)$.

\begin{lemma}[{{\cite[Lemmas 1.2 and 1.3]{craven2017}}}]\label{lem:genericmeansblueprint} Let $u$ be a non-trivial unipotent element in $\mb G$ and let $V$ be either $M(\mb G)\oplus M(\mb G)^*$ or $L(\mb G)$. If $u$ is generic for $V$ then $u$ is a blueprint for $V$.
\end{lemma}

Work on whether semisimple elements are blueprints has been done in the past, and we are able to give some results that prove that a semisimple element $x\in \mb G$ is always a blueprint for $M(\mb G)$ or $L(\mb G)$ if its order is sufficiently large. The next result is a combination of results from the author in \cite[Theorem 5.9 and Proposition 6.10]{craven2015un} for $M(\mb G)$, and Liebeck--Seitz and Lawther from \cite{liebeckseitz1998} and \cite{lawther2014} for $L(\mb G)$.

\begin{definition}\label{defn:T(G)} Let $\mb G$ be one of $G_2$, $F_4$, $E_6$, $E_7$, $E_8$. Define the set $T(\mb G)$ of positive integers as follows: the odd elements of $T(\mb G)$ are
\[ T(\mb G)_{\mathrm{odd}}=\begin{cases} \{1,\dots,9\} & \mb G=G_2,\\\{1,\dots,57\} & \mb G=F_4,\\\{1,\dots,105\} & \mb G=E_6,\\\{1,\dots,317\} & \mb G=E_7,\\ \{1,\dots,1093,1097,1099,1103,1105,1113,1115,1117,
\\ 1121,1123,1127,1129,1147,1153,1165,1189\} & \mb G=E_8,\end{cases}\]
and the even elements of $T(\mb G)$ are
\[ T(\mb G)_{\mathrm{even}}=\begin{cases} \{2,\dots,12\} & \mb G=G_2,\\\{2,\dots,68\} & \mb G=F_4,\\\{2,\dots,120,124\} & \mb G=E_6,\\\{2,\dots,364,370,372,388\} & \mb G=E_7,\\ \{2,\dots,1262,1268,1270,1284,1298,1312\} & \mb G=E_8.\end{cases}\]
\end{definition}

\begin{theorem}\label{thm:largeorderss} Let $x$ be a semisimple element in $\mb G$.
\begin{enumerate}
\item If the image of $x$ in the adjoint form of $\mb G$ has order not in $T(\mb G)$, then $x$ is a blueprint for $L(\mb G)$.
\item If $\mb G=E_7$ and $o(x)$ is odd and at least $75$, then $x$ is a blueprint for $M(E_7)$.
\item If $\mb G=E_7$ and $o(x)>30$, and $x$ centralizes a $6$-space on $M(E_7)$, then $x$ is a blueprint for $M(E_7)$.
\item If $\mb G=E_6$ and $o(x)$ is odd and at least $75$, then $x$ is a blueprint for $M(E_6)\oplus M(E_6)^*$.
\item If $\mb G=E_6$, $x$ is real, and $o(x)>18$, then $x$ is a blueprint for $M(E_6)\oplus M(E_6)^*$.
\item If $\mb G=F_4$ and $o(x)>18$, then $x$ is a blueprint for $M(F_4)$.
\item If $\mb G=G_2$ and $o(x)>4$, then $x$ is a blueprint for $M(G_2)$.
\end{enumerate}
\end{theorem}

We also consider the concept of pressure, and we take the definition and main result from \cite[Lemma 2.2]{craven2015un}. Let $\mc M$ be a collection of simple modules for a group $H$, such that $\Ext_{kH}^1(M,M')=0$ for all $M,M'\in \mc M$. The \emph{$\mc M$-pressure} of a $kH$-module $V$ is the quantity
\[ \left(\sum_{W\in \cf(V)} \sum_{M\in \mc M} \dim(\Ext_{kH}^1(M,W))\right)-|\{W\in \cf(V)\mid W\in \mc M\}|,\]
where as mentioned above $\cf(V)$ denotes the multiset of composition factors of $V$.

\begin{proposition}\label{prop:pressure} Let $\mc M$ be a set as above, and let $V$ be a module.
\begin{enumerate}
\item If the $\mc M$-pressure of $V$ is negative then $V$ has a simple submodule that is a member of $\mc M$.
\item Suppose that $\Ext_{kH}^1(A,B)=\Ext_{kH}^1(B,A)$ for all simple modules $A,B$ in $\cf(V)$. (This is the case whenever there is an automorphism that sends all simple modules to their duals, hence in all simple groups of Lie type in defining characteristic, where the graph automorphism or the identity does this.)
\begin{enumerate}
\item If the $\mc M$-pressure of $V$ is $0$ then $V$ has either a submodule or quotient that is a member of $\mc M$.
\item If the $\mc M$-pressure of $V$ is equal to $n>0$, and there are no submodules or quotients of $V$ that are members of $\mc M$, then the $\mc M$-pressure of any subquotient of $V$ lies in the interval $[-n,n]$.
\end{enumerate}
\end{enumerate}
\end{proposition}

This is normally used in the case where $\mc M$ is simply the set $\{k\}$ containing the trivial module, in which case the $\mc M$-pressure is just referred to as \emph{pressure} (see \cite[Lemma 1.8]{craven2017}, and also \cite[Lemma 1.2]{lst1996} for what might be the original use of the idea). Note that if $H$ has no quotient of order $p$, i.e., $O^p(H)=H$, then $\Ext_{kH}^1(k,k)=0$ and so $\{k\}$ satisfies the required condition. In particular, if $H$ is a simple group and $V$ is a module for $H$ with negative pressure, then $H$ stabilizes a line on $V$. This is important because the line stabilizers of $M(\mb G)$ and $L(\mb G)$ are known to be positive dimensional.

\medskip

Given $H\leq \mb G$, and a proper subgroup $L$ of $H$, we might already know that $L$ is imprimitive. Thus we want to find all possibilities for $L$ (knowing that it lies inside some $\mb X$), and we want to use information about $L$ to restrict the possibilities for $H$. The next proposition shows that we normally need only consider connected positive-dimensional subgroups $\mb X$, rather than for example the normalizer of a torus. This result may be found in \cite[Proposition 3.11]{litterickmemoir}.

\begin{proposition}\label{prop:inconnected} Let $H$ be a simple subgroup of $G$. If $H$ is contained inside a member of $\ms X^\sigma$, then there is some $\sigma$-stable member $\mb Y$ of $\ms X$ such that $H\leq \mb Y^\circ$, possibly unless $\mathbf{G}=E_8$ and $H\cong\mathrm{Alt}(5),\mathrm{PSL}_2(7)$, $\mathbf{G}=E_7$ and $H\cong\mathrm{PSL}_2(7),\mathrm{PSL}_2(8),\mathrm{PSU}_3(3),\mathrm{PSp}_6(2)$, or $\mathbf{G}=E_6$ and $H\cong\mathrm{PSU}_4(2)$.
\end{proposition}

The next result finds a subquotient in the `middle' of a self-dual module.

\begin{lemma}\label{lem:oddandodd} Let $V$ be a self-dual module for a finite group $G$ over a field $k$, and suppose that $S_1,\dots,S_r$ are non-isomorphic, simple, self-dual $kG$-modules, each appearing in $V$ with odd multiplicity. Then $S_1\oplus S_2\oplus \cdots\oplus S_r$ is a subquotient of $V$.
\end{lemma}
\begin{proof} Proceed by induction on $\dim(V)$, the case $\dim(V)=1$ being clear. We may assume without loss of generality that $V$ is equal to its $\{S_1,\dots,S_r\}$-heart. If $V\cong V_1\oplus V_2$ then each $S_i$ appears with odd multiplicity in exactly one of the $V_j$, whence induction applied to each $V_j$ yields a subquotient of each $V_j$, whose sum is a subquotient of $V$ and has the required form.

Therefore we may assume that $V$ is indecomposable. In particular, $\soc(V)\leq \rad(V)$, and $\soc(V)$ and $\top(V)$ are isomorphic. Thus the heart of $V$, which is the quotient $\rad(V)/\soc(V)$, again has each of the $S_i$ as a composition factor with odd multiplicity. The heart of $V$ is a module of smaller dimension satisfying the conditions of the lemma, hence satisfies the conclusion, and therefore $V$ does. (Note that $V$ cannot have only two socle layers as it contains an odd number of composition factors isomorphic to each $S_i$, and hence the heart of $V$ is non-zero.)
\end{proof}

This has a nice corollary, using pressure.

\begin{corollary}\label{cor:trivialoddmult}  Let $V$ be a self-dual module for a finite group $G=O^p(G)$ over a field $k$, and suppose that $S_1,\dots,S_r$ are non-isomorphic, simple, self-dual $kG$-modules, each appearing in $V$ with odd multiplicity. If none of the $S_i$ is the trivial module, and the pressure of the sum of the $S_i$ is greater than the pressure of $V$, then $G$ stabilizes a line on $V$.
\end{corollary}
\begin{proof} This is immediate upon application of Lemma \ref{lem:oddandodd}, and then Proposition \ref{prop:pressure}.
\end{proof}

We can also obtain a result about modules that will be useful occasionally, which is that in a self-dual module, `half of the composition factors are in the bottom of the module'. This incredibly vague statement can be firmed up as the following.

\begin{proposition}\label{prop:bottomhalf} Let $V$ be a self-dual $kH$-module, let $I$ be a set of simple $kH$-modules, closed under duality, and let $W$ denote the $I$-heart of $V$. Let $V_0$ denote the $I'$-radical of $V$. Then $W$ is a submodule of the quotient $V/V_0$, and if $S$ is a self-dual simple module not in $I$ and $W$ contains no copies of $S$, then the multiplicity of $S$ in $V$ is at most twice that of $S$ in $V_0$.
\end{proposition}
\begin{proof} It is clear that $W$ is a submodule of $V/V_0$. For the second part, since $V$ is self-dual there is a submodule $V_1$ such that the quotient $V/V_1$ is isomorphic to $V_0^*$, i.e., $V_1$ is the $I$-residual of $V$. We have that the quotient $V_1/(V_0\cap V_1)$ is isomorphic to $W$, hence $V$ has a filtration
\[ 0\subs V_0\cap V_1\subs V_1\subs V.\]
All composition factors of $V$ isomorphic to $S$ appear either in $V_0\cap V_1\subset V_0$ or in $V/V_1\cong V_0^*$, and so the result holds.
\end{proof}

We include a lemma, which is certainly well-known, but which evaded the author's brief search of the literature. (The lemma is trivial for $p\neq 2$.)

\begin{lemma}\label{lem:orthogfixpoint} Let $p=2$. The module $S^2(M(D_n))$ has a trivial submodule. Consequently, if $V$ is a self-dual $kH$-module for a finite group $H$ and $S^2(V)^H=0$, then there is no homomorphism from $H$ to $D_n$ via the module $V$.
\end{lemma}
\begin{proof} There are many proofs of this. The quickest is to note that $\Lambda^2(M(D_n))$ of course does have a trivial submodule (as $p=2$), and that $S^2(M(D_n))$ has a submodule (the image under the Frobenius endomorphism of) $M(D_n)$ with quotient $\Lambda^2(M(D_n))$. Since the natural module for $D_n$ has no $1$-cohomology (see \cite{jonesparshall1976}) it must be that $S^2(M(D_n))$ has a trivial submodule. The consequence is clear.\end{proof}

We need one easy result from non-abelian cohomology, the generalization to arbitrary $p$-groups of the standard result from cohomology that all complements in a semidirect product are conjugate if and only if the $1$-cohomology is zero. Although we state it in the case where $P$ is a finite $p$-group, the same proof works for algebraic groups.

\begin{lemma}\label{lem:nonabcohomzero} Let $P$ be a finite $p$-group and let $H$ be a finite group that acts on $P$, with $G=P\rtimes H$. Suppose that there is an $H$-stable series
\[ 1=P_0\leq P_1\leq \cdots \leq P_r=P\]
of normal subgroups of $P$ such that $P_i/P_{i-1}$ is elementary abelian for all $1\leq i\leq r$. If, viewed as an $\F_pH$-module, we have that $H^1(H,P_i/P_{i-1})=0$ for all $1\leq i\leq r$, then all complements to $P$ in $G$ are conjugate.
\end{lemma}
\begin{proof} Let $H_1$ and $H_2$ be complements to $P$ in $G$. Proceed by induction on $r$, the case $r=1$ being the standard result from cohomology theory. Note that $P_1$ is normal in $G$, so by induction the result holds for $G/P_1$. Since $H_1P_1/P_1$ and $H_2P_1/P_1$ are complements to $P/P_1$ in $G/P_1$, there exists $g\in G$ such that $(H_1P_1)^g=H_2P_1$. Thus we may assume that $H_1P_1=H_2P_1$. But then $H_1$ and $H_2$ are complements to $P_1$ in $H_1P_1$, whence since $H^1(H,P_1)=0$ we have that $H_1$ and $H_2$ are conjugate, as claimed.
\end{proof}

\newpage

\chapter{Subgroup structure of exceptional algebraic groups}
\label{chap:subgroupstructure}

Let $\mb G$ be a simple, simply connected, exceptional algebraic group, so one of $G_2$, $F_4$, $E_6$, $E_7$ and $E_8$. The collection of maximal closed, positive-dimensional subgroups of $\mb G$, i.e., the set $\ms X$, is known by work of Liebeck and Seitz \cite{liebeckseitz2004}\footnote{There is a maximal subgroup missing in \cite{liebeckseitz2004}, namely a copy of $F_4$ lying in $E_8$ in characteristic $3$. It acts on $L(E_8)$ as the direct sum of $L(\lambda_1)$ of dimension $52$ and $L(\lambda_3)$ of dimension $196$. This subgroup was very recently discovered by the author, and details are in \cite{cst2022}.}. Broadly speaking, it consists of parabolic subgroups, maximal-rank subgroups, and a small collection of $\mb G$-irreducible reductive subgroups.

To move from subgroups of $\mb G$ to subgroups of $G$, and more generally subgroups of $\bar G$, we need the following theorem.

\begin{theorem}[Borovik, Liebeck--Seitz \cite{borovik1989,liebeckseitz1990}]\label{thm:classmaximalsubgroup} Let $H$ be a maximal subgroup of $\bar G$. One of the following holds:
\begin{enumerate}
\item[(IMP)] $H$ is a member of $\ms X^\sigma$;
\item[(EX)] $H$ is an \emph{exotic $r$-local subgroup} of $\bar G$.
\item[(BOR)] $H$ is the Borovik subgroup $(\Alt(5)\times \Sym(6))\rtimes 2$;
\item[(AS)] $H$ is an almost simple group.
\end{enumerate}
\end{theorem}

The members of $\ms X$ are known from \cite{liebeckseitz2004}, and in theory from here it is possible to write down a complete list of the members of (IMP). This has been done in the case of maximal-rank subgroups in \cite{liebecksaxlseitz1992}, but there are some minor errors in the tables there with regards to the `decorations' that are placed on the group. (For example, a maximal-rank subgroup of type $A_8$ in $E_8$ is of the form $(3\cdot \PSL_9(q)\cdot 3).2$, if $3\mid (q-1)$, and the decorations are the $3$s and $2$.) The members of (EX) are known from \cite{clss1992}, and the Borovik subgroup is unique up to conjugation.

A priori, there are infinitely many possibilities in (AS), as $\PSL_2(p^a)$ lies inside $E_8$ in characteristic $p$ for all $a$. If $H$ is not a Lie type group in characteristic $p$, the (finitely many) possibilities are given in \cite{liebeckseitz1999} (note that ${}^2\!B_2(8)$ for $p=5$ is missing from $E_7$). The work of Litterick \cite{litterickmemoir} and the author \cite{craven2017} has eliminated some of the subgroups in the list from \cite{liebeckseitz1999}, and various papers that have appeared in the literature have considered specific examples, often in characteristic $0$. 

If $H$ is a Lie type group in characteristic $p$ in (AS) above, then in \cite{liebeckseitz1998} Liebeck and Seitz proved that $H$ must be one of the following:
\begin{enumerate}
\item[(MED)] $H=H(q)$ for $q\leq 9$ and $H$ of untwisted rank (i.e., the rank of the untwisted group) at most half of that of $\mb G$;
\item[(LU3)] $H\cong \PSL_3(16)$ or $H\cong \PSU_3(16)$;
\item[(RK-1)] $H\cong \PSL_2(q)$, ${}^2\!B_2(q)$ or ${}^2\!G_2(q)$ for $q\leq \gcd(2,p-1)\cdot t(\mb G)$.
\end{enumerate}

In \cite{craven2015un} the author considered $H\cong \PSL_2(q)$ for $\mb G=F_4$, $E_6$ and $E_7$, proving that there are no such examples in (AS) for $\mb G=F_4$ and $\mb G=E_6$, and for $\mb G=E_7$ we may restrict $q$ to be one of $7$, $8$ and $25$. (There are still examples of $\PSL_2(q)$ in (IMP) for $F_4$ and $E_7$.) Recent work of Alexander Ryba and the author has eliminated the case $q=25$.

This article eliminates possibility (LU3), eliminates ${}^2\!G_2(q)$ from (RK-1), severely restricts the examples that may occur in (MED) to fewer than ten groups, all for $E_8$, and eliminates ${}^2\!B_2(q)$ from (RK-1) unless $\mb G=E_8$ and $q=8$.

In order to do this, we must show that for any simple subgroup $H$ of $\bar G$, $N_{\bar G}(H)$ lies inside a member of $\ms X^\sigma$, or an exotic $r$-local subgroup, or in the Borovik subgroup. In fact, we will always show that $H\leq \mb G$ is strongly imprimitive, so we require some techniques to prove this.

\medskip

The primary method is to prove that $H$ stabilizes a subspace of some module $V$ for $\mb G$ whose stabilizer is proper and positive dimensional. This proves that $H$ lies in a member of $\ms X$, and if $H=H^\sigma$ and if the subspace is $\sigma$-stable then $H$ lies inside a member of $\ms X^\sigma$, but this still says nothing about $N_{\bar G}(H)$, the potential maximal almost simple group in Theorem \ref{thm:classmaximalsubgroup}. For that we need to prove that $H$ stabilizes each subspace in an orbit under the action of $\bar G$ of subspaces of some module $V$, and the intersection of the stabilizers of all subspaces in the orbit is positive dimensional. Note that in order for this to be even well defined, the module $V$ must be stabilized by $\bar G$; as diagonal automorphisms stabilize all highest weight modules, and field automorphisms act as semilinear transformations of highest-weight modules, in practice this means checking that the module is graph-stable for $E_6$ and $F_4$ ($p=2$), and there is no condition for $E_7$ and $E_8$. Because semilinear transformations are allowed, a \emph{graph-stable} module can be taken to mean a semisimple $k\bG$-module $V$ whose highest weights are invariant, up to applying a power of the Frobenius endomorphism, under the graph automorphism.

Thus for $E_6$ we examine $L(\mb G)$ or $M(\mb G)\oplus M(\mb G)^*$, which is graph-stable, and for $F_4$ and $p=2$ we choose the sum of the composition factors of $L(\mb G)$, i.e., $L(\lambda_1)\oplus L(\lambda_4)$, which is also graph-stable. (Compare with the modules mentioned in Theorem \ref{thm:largeorderss} above.) Notice that in both of the later cases the graph automorphism will interchange the two submodules.

The theorem which proves this in its most general form is found in \cite[Proposition 4.3]{craven2015un}, but it is based entirely on work of Liebeck and Seitz \cite{liebeckseitz1998}, and was also featured in the work of Litterick \cite{litterickmemoir}. Define $\Aut^+(\mb G)$ to be the group generated by inner, diagonal, graph, and $p$-power field automorphisms of $\mb G$. As in \cite[Chapter 3]{craven2015un}, when $p=2$ and $\mb G=F_4$ we need to be a bit more careful, as the graph automorphism powers to a field automorphism, so we can add a single graph automorphism to $\Aut^+(\mb G)$, but not all of them simultaneously. This distinction is purely academic, but it means that when we consider $N_{\Aut^+(\mb G)}(H)$-stability for a subgroup $H\leq \mb G^\sigma$, then the definition of $\Aut^+(\mb G)$ technically depends on $\sigma$. This will never be of importance from now on, and so we will never mention it again.

\begin{theorem}\label{thm:intersectionorbit} Let $\mb G$ be a simple algebraic group over an algebraically closed field, and let $V$ be a graph-stable module. If $\phi\in N_{\Aut^+(\mb G)}(H)$ then $\phi$ permutes the $H$-invariant subspaces of $V$. If $\mathcal W$ is a $\phi$-orbit of $H$-invariant subspaces and $\mb G_{\mathcal W}$ denotes the intersection of the stabilizers $\mb G_W$ for $W\in \mathcal W$, then $H\leq \mb G_{\mathcal W}$ and $\mb G_{\mathcal W}$ is $\phi$-stable.
\end{theorem}

Equipped with this, we must find some subspaces for $H$ to stabilize. If $H$ is a blueprint for a graph-stable module $V$ (and does not stabilize exactly the same subspaces as $\mb G$), then certainly $H$ is contained in a member of $\ms X$, and the intersection of all subspaces it stabilizes, in particular a single orbit, is positive dimensional. In addition, if $H=H^\sigma$ then $N_{\bar G}(H)$ is also contained in a member of $\ms X^\sigma$ (of the almost simple group $\bar G$).

\begin{corollary} Let $H$ be a subgroup of $\mb G$, and let $V$ be a graph-stable $k\mb G$-module. If $H$ is a blueprint for $V$ then either $H$ is strongly imprimitive or $H$ and $\mb G$ stabilize the same subspaces of $V$. In particular, if $V$ is irreducible then $H$ is strongly imprimitive or $H$ acts irreducibly on $V$.
\end{corollary}

The other possibility is that $H$ stabilizes a line on $M(\mb G)^\circ$ or $L(\mb G)^\circ$. If $H$ stabilizes a line on $L(\mb G)^\circ$ then we can say something about $H$.

\begin{lemma}[{{\cite[(1.3)]{seitz1991}}}]\label{lem:maxrankorpara} If $H$ stabilizes a line on $L(\mb G)^\circ$ then $H$ is contained in either a maximal-rank subgroup or a maximal parabolic subgroup of $\mb G$.
\end{lemma}

If $H$ stabilizes a line on either $L(\mb G)^\circ$ or $M(\mb G)^\circ$, \cite[Propositions 4.5 and 4.6]{craven2015un} prove that $H$ is strongly imprimitive, with some mild conditions, which are always going to be satisfied by an almost simple subgroup such that $N_{\bar G}(H)$ is maximal.

\begin{lemma}\label{lem:fix1space} Let $H$ be a finite subgroup of $\mb G$. Suppose that $C_{\mb G}(H)=Z(H)$ and $H$ has no subgroup of index $2$. If $H$ stabilizes a line on $L(\bG)^\circ$ or $M(\bG)^\circ$ then $H$ is strongly imprimitive. If $\bG=F_4$ and $p=2$, and either $H$ or its image under the graph automorphism stabilizes a line on $M(F_4)$, then $H$ stabilizes a line or hyperplane on $L(F_4)$ and is strongly imprimitive.
\end{lemma}

(The consequence is simply because if $H$ stabilizes a line on $M(F_4)$ then it also stabilizes a hyperplane. The composition factors of $L(F_4)$ are $M(F_4)$ and its image under the graph automorphism, and so $H$ must also stabilize a line or hyperplane on $L(F_4)$.)

This means that if the composition factors of $M(\mb G){\downarrow_H}$ or $L(\mb G){\downarrow_H}$ are unwarranted, and also force $H$ to stabilize a line on the module, then $H$ cannot embed with those factors. One example is the following.

\begin{corollary}\label{cor:no246} There does not exist a finite subgroup of $E_8$ over any field whose composition factors on $L(E_8)$ have dimensions $246,1,1$.
\end{corollary}
\begin{proof} Let $H$ be such a subgroup. Clearly $H$ stabilizes a line on $L(E_8)$, so by Lemma \ref{lem:fix1space} lies in a member of $\ms X$. However, no proper positive-dimensional subgroup of $E_8$ has a composition factor of dimension at least $246$ (see for example the tables in \cite{thomas2016}) so we obtain a contradiction.
\end{proof}

On one occasion we will show that a subgroup is contained in a parabolic subgroup but not inside a Levi complement of it (it is not $\mb G$-completely reducible). The next result comes from \cite[Proposition 2.2 and Remark 2.4]{liebeckmartinshalev2005}, and is stated in our language in \cite[Proposition 4.1]{craven2015un}.

\begin{lemma}\label{lem:nonGcr} Let $H$ be a finite subgroup of $\mb G$. If $H$ is not $\mb G$-completely reducible then $H$ is strongly imprimitive.
\end{lemma}

When $p=2$ and $\mb G=F_4$, the presence of a graph automorphism means we often have to be careful not just to find a positive-dimensional subgroup containing our given subgroup $H$, but to find one that is stable under the automorphisms in $N_{\bar G}(H)$, to eliminate almost simple groups from being potential maximal subgroups. The next lemma gives us a criterion that enables us to ignore the graph automorphism when proving normalizer-stability. We could always prove the results a different way, but using this makes it significantly easier.

\begin{lemma}\label{lem:F4ignoregraph} Let $p=2$. Let $H$ be a subgroup of the algebraic group $F_4$. If $H$ is stabilized by $\phi\in\Aut^+(\mb G)$ such that $\phi$ induces a graph automorphism on $\mb G^\sigma$ (i.e., $\phi$ does not act as a semilinear automorphism on $M(F_4)=L(\lambda_4)$), then there is an automorphism $\tau$ of $H$ such that the module $L(\lambda_4){\downarrow_H}$ is isomorphic to the conjugate module $L(\lambda_1){\downarrow_{H^\tau}}$.

In particular, if the composition factors of $H$ on $M(F_4)$ and the quotient module $L(F_4)/M(F_4)$ do not have the same dimensions then $N_{\Aut^+(\mb G)}(H)$ cannot induce a graph automorphism on $G$.
\end{lemma}
\begin{proof} Immediate.\end{proof}

In a similar vein, we need to understand a bit more carefully how semilinear transformations act on vector spaces, which are needed in the case where $\bar G$ induces field automorphisms on $G$.

\begin{lemma}\label{lem:semilinearfield}
Let $H$ be a finite group and let $V$ be a faithful $kH$-module. Suppose that $\sigma$ is a semilinear transformation of $V$ that is $H$-equivariant. If $W$ and $W'$ are $kH$-submodules of $V$ such that $W^\sigma=W'$, then there exists an automorphism $\phi$ of $H$ such that $W'\cong W^\phi$, and then $V=V^\phi$, and in particular the composition factors of $V$ are invariant under $\phi$.
\end{lemma}
\begin{proof} Since $\sigma$ is $H$-equivariant, we may form the subgroup $\gen{H,\sigma}$ of the group of semilinear transformations of $V$. This group has $H$ as a normal subgroup, and hence $\sigma$ induces an automorphism $\phi$ of $H$. Furthermore, $V^\sigma=V$, so $V^\phi=V$ and the result follows.
\end{proof}

With some rank-$2$ subgroups $H$ of $E_8$ for $q\geq 8$ we will find an element $x$ that certainly isn't in $N_{\mb G}(H)$ and stabilizes a subspace of $L(E_8)$ also stabilized by $H$. Hence we can take the subgroup $\gen {H,x}$, which is larger than $H$ and smaller than $\mb G$, to show that $H$ is not maximal in a finite group. With a bit more care, this can be used to show that $H$ is strongly imprimitive. We do this now.

The proof of this uses forward references. The reason for this is that we place $H$ inside a finite group $K$, and want that $K$ is strongly imprimitive. We will only use this result in Chapter \ref{chap:rank2ine8}, so can use results for ranks $3$ and $4$ from the previous chapters. Forward referencing is not strictly required, but proving this result without it takes considerably longer, as one has to determine precisely which groups can contain which, and then prove results for these groups.

\begin{lemma}\label{lem:allcasesstronglyimp} Let $H$ be a finite simple group. Let $\mb G=E_8$, and let $x$ be a semisimple element of $\mb G$. Suppose that one of the following holds:
\begin{enumerate}
\item $H\cong \PSL_3(8)$ and $x$ has order $189$;
\item $H\cong \PSL_3(9),\PSU_3(9)$ and $x$ has order $160$;
\item $H\cong \PSU_3(16)$ and $x$ has order $765$;
\item $H\cong \PSp_4(8)$ and $x$ has order $195$;
\item $H\cong {}^2\!B_2(32)$ and $x$ has order $93$;
\item $H\cong {}^2\!B_2(128)$ and $x$ has order $435$;
\item $H\cong \PSL_2(81)$ and $x$ has order $164$;
\item $H\cong \PSU_5(4)$ and $x$ has order $195$, with $x^3\in H$.
\end{enumerate}
If $\mc O$ is a $\sigma$-stable, $N_{\Aut^+(\mb G)}(H)$-orbit of subspaces of $L(E_8)$, each subspace stabilized by $H$, and the members of $\mc O$ are also stabilized by $x$, then $H$ is strongly imprimitive. If $\mc O$ is stabilized by $H$ and $x$, but not necessarily by $N_{\Aut^+(\mb G)}(H)$, then $H$ is Lie imprimitive, but not necessarily strongly imprimitive (i.e., $H$ lies in a member of $\ms X$).
\end{lemma}

\begin{proof} We first check that $x$ cannot normalize $H$, i.e., $H\not\normal \gen{H,x}$.

By \cite[pp.74, 78, 79]{atlas}, we see that $\Aut(H)$ contains no element of order $o(x)$ for $H$ one of $\PSL_3(8)$, $\PSL_3(9)$, $\PSU_3(9)$, and since $3\nmid |\Out(\PSU_3(16))|$ we see that the same holds for $H\cong \PSU_3(16)$. For $H\cong \PSp_4(8)$, $\Out(H)$ is cyclic of order $3$ so $|H\cap \gen x|$ is either $65$ or $195$, so $65$ as $195$ is not the order of an element of $H$. However, $x$ induces the field automorphism on $H$, whose centralizer is $\PSp_4(2)$ and must contain $x^3$ of order $65$, a contradiction.

The Suzuki groups famously do not have elements of order $3$, and their automorphisms consist only of field automorphisms, which have orders $5$ and $7$ in the cases above, so this proves the claim for them. For $H\cong \PSL_2(81)$, the intersection $H\cap \gen x$ must have order $41$; $\Out(H)$ is dihedral of order $8$, with the field automorphism of order $2$ as the centre of it. But the centralizer in $H$ of the field involution is $\PSL_2(9)$, which of course does not contain an element of order $41$, so we are done in this case as well. Finally, $3\nmid |\Out(\PSU_5(4))|$ and $\PSU_5(4)$ does not contain an element of order $195$.

Write $K=\gen{H,x}$, and therefore by the above $K\not\leq N_{\mb G}(H)$. Let $\bar K$ be the intersection of the stabilizers of the members of $\mc O$, and note that $H\leq K\leq \bar K$.

If $\bar K$ is infinite then we apply Theorem \ref{thm:intersectionorbit} to see that $H$ is strongly imprimitive. Thus we may assume that $K$ and $\bar K$ are both finite. By Theorem \ref{thm:classmaximalsubgroup}, if $\bar K$ (and hence $H$) is not strongly imprimitive, then $\bar K$ lies inside an exotic local subgroup, or the Borovik subgroup, or $\bar K$ is almost simple.

The non-cyclic composition factors of the exotic local subgroups (from \cite{clss1992}) are $\PSL_5(2)$ and $\PSL_3(5)$, neither of which contains $H$, so cannot contain $\bar K$, and similarly $H$ is not contained in the Borovik subgroup either. Thus $\bar K$ is almost simple. If $\bar K$ is not of Lie type in characteristic $p$, then the possibilities for $\bar K$ appear in \cite{liebeckseitz1999}; none of these contains $H$ (certainly they don't contain elements of order $o(x)$) and so $\bar K$ must be Lie type in characteristic $p$, $\bar K(q)$ for some $q$, and this is not $H$. If $q>9$ or the rank of $\bar K$ is greater than $4$ then $\bar K$ is a blueprint for $L(E_8)$ as we saw above, and we are done. If $\bar K$ stabilizes a line on $L(E_8)$ then so does $H$ and again $H$ is strongly imprimitive by Lemma \ref{lem:fix1space}. If the rank of $\bar K$ is $3$ or $4$ then either $\bar K$ is strongly imprimitive or it is $\PSU_4(2)$, which does not contain $H$, by the collection of propositions from Chapters \ref{chap:rank4ine8}, \ref{chap:rank3ine8} and \ref{sec:difl52}, so $\bar K$ has rank $2$. The only possibility now is that $q=8,9$ and $\bar K=G_2(q)$, but this stabilizes a line on $L(E_8)$ by Proposition \ref{prop:g2ine8} below, hence so does $H$ and we are done again by Lemma \ref{lem:fix1space}.
\end{proof}

For $H\cong \PSL_3(16)$, we want the order of $x$ to be $255$, but there is an element of order $255$ in $\PGL_3(16)$ (but not in $\PSL_3(16)$). We have to be slightly more careful: since $\Out(H)$ has a normal subgroup of order $3$ and index $8$, all elements of order $255$ lie in $\PGL_3(16)\leq \Aut(H)$. Furthermore, the centralizer in $\PGL_3(16)$ of a cyclic subgroup $\gen{x^3}$ of order $85$ (unique up to conjugacy) is of order $255$ (it is a Singer cycle) and hence there are exactly two elements of order $255$ in $\Aut(H)$ cubing to a given element of order $85$ in $\PSL_3(16)$. The analogue of Lemma \ref{lem:allcasesstronglyimp} is therefore the following, with an identical proof. Similar statements also hold for $H\cong \PSp_4(9)$.

\begin{lemma}\label{lem:stronglyimppsl316} Let $H$ be one of $\PSL_3(16)$ and $\PSp_4(9)$. Suppose that $H$ is a subgroup of $\mb G=E_8$, and let $x_1,x_2,x_3$ be three distinct elements of $\mb G$ of order $255$ if $H\cong \PSL_3(16)$, and order $82$ if $H\cong \PSp_4(9)$. If $\mc O$ is a $\sigma$-stable, $N_{\Aut^+(\mb G)}(H)$-orbit of subspaces of $L(E_8)$, each subspace stabilized by $H$, and the members of $\mc O$ are also stabilized by all $x_i$, then $H$ is strongly imprimitive.
\end{lemma}

\begin{remark}\label{rem:howusestrongimp} This is how we will use Lemmas \ref{lem:allcasesstronglyimp} and \ref{lem:stronglyimppsl316}. Let $H$ be a subgroup of $\mb G$, and let $S_1,\dots,S_r$ be a set of representatives for the composition factors of $L(E_8){\downarrow_H}$. For some of these, say $S_i$ for $i\in I$, we can find an element $x_i\in \mb G$ satisfying the order hypothesis in one of these two lemmas, and stabilizing the eigenspaces of $y$ on $L(E_8){\downarrow_H}$ that comprise $S_i$, where $y$ is a generator of $\gen{x_i}\cap H$. Thus if $L(E_8){\downarrow_H}$ possesses a submodule $S$ isomorphic to $S_i$ then $x_i$ stabilizes it. (We may have that different $S_i$ can be stabilized by the same $x_i$.)

By one of the two lemmas, this means that $H$ is Lie imprimitive. If $S$ is stable under $N_{\Aut^+(\mb G)}(H)$ then $H$ is also strongly imprimitive, but this will not always be the case. For this to not be the case, by Lemma \ref{lem:semilinearfield} the composition factors of $L(E_8){\downarrow_H}$ must be stable under an outer automorphism of $H$ (in fact one induced by an element of $N_{\Aut^+(\mb G)}(H)$). Furthermore, the image of $S$ under the elements of $N_{\Aut^+(\mb G)}(H)$ are other simple submodules that must, by symmetry, also appear in $I$. If there is an element $x$ that stabilizes all of the eigenspaces comprising all of these different modules $S$ simultaneously, then $H$ is strongly imprimitive. Otherwise we may assume that both $S$ and its image under the potential element of $N_{\Aut^+(\mb G)}(H)$ are submodules of $L(E_8){\downarrow_H}$ when analysing further, and indeed there is an automorphism of $H$ that maps $S$ to this other module in its induced action on $L(E_8){\downarrow_H}$.

In conclusion, if the composition factors of $L(E_8){\downarrow_H}$ are not stable under an outer automorphism of $H$, then we obtain that $H$ is strongly imprimitive or $\soc(L(E_8){\downarrow_H})$ consists solely of simple submodules $S_i$ for $i\not\in I$. If there is such an outer automorphism, then we still obtain this conclusion if there is an element stabilizing all of the eigenspaces from the orbit of $S$ simultaneously, and otherwise we may only conclude that $H$ is Lie imprimitive or the condition above on the socle holds. In this case though, we may assume that every simple module in the orbit of $S$ under this outer automorphism appears in the socle of $L(E_8){\downarrow_H}$.
\end{remark}

We now describe the composition factors of copies of $\Alt(6)$ in $E_8$ on $L(E_8)$ for $k$ of characteristic $2$. The simple modules for $\Alt(6)$ are $1$, $4_1$, $4_2$, $8_1$ and $8_2$, with the two $8$s merging into a $16$ in $\Sym(6)$. If we have equal numbers of $8_1$ and $8_2$, we write $16$ to make reading the list easier.

\begin{proposition}\label{prop:compfactorsalt6} Let $H\cong \Alt(6)$ be a subgroup of $\mb G=E_8$ in characteristic $2$. The warranted sets of composition factors of $L(E_8){\downarrow_H}$ are as follows:
\[ 16^9,4_1^9,4_2^9,1^{32},\quad  16^5,4_1^{14},4_2^{20},1^{32},\quad 
16^5,4_1^{20},4_2^{14},1^{32},\quad 16^4,4_1^{19},4_2^{19},1^{32},\]
\[ 16^4,4_1^{22},4_2^{16},1^{32},\quad 16^4,4_1^{16},4_2^{22},1^{32},\quad
4_1^{33},4_2^{12},1^{68},\quad
4_1^{12},4_2^{33},1^{68}\]
\[ 8_1^8,8_2^3,4_1^{16},4_2^{16},1^{32},\quad 8_1^3,8_2^8,4_1^{16},4_2^{16},1^{32},\quad 8_1^{25},4_1^2,4_2^2,1^{32},\quad 8_2^{25},4_1^2,4_2^2,1^{32}\]
If $H$ embeds in $E_8$ then the composition factors of $L(E_8){\downarrow_H}$ are either warranted or possibly $16^6,4_1^{16},4_2^{16},1^{24}$. (No example is known with these factors.)
\end{proposition}
\begin{proof}By \cite[Proposition 6.3]{craven2017}, the composition factors of $L(E_8){\downarrow_H}$ are either warranted or $16^6,4_1^{16},4_2^{16},1^{24}$, thus we may assume that $H$ lies inside a positive-dimensional subgroup $\mb X$ of $\mb G$.

If $\mb X$ is $D_8$ then we use a computer to determine the action of $H$ on the half-spin module. Since $L(D_8)$ has the same composition factors as $\Lambda^2(M(D_8))$ this is easy to compute, and $L(E_8)$ is the sum of $L(D_8)$ and a half-spin module. This yields the following table, where representations of $\Alt(6)$ (taken up to automorphism for brevity) that extend to $\Sym(6)$ are given first, and then those that cannot (inside $D_8$) are given afterwards.
\begin{center}
\begin{tabular}{cc}
\hline Factors on $M(D_8)$ & Factors on $L(E_8)$
\\ \hline $16$ & $16^4,4_1^{19},4_2^{19},1^{32}$
\\ $4_1^4$ & $16^4,4_1^{22},4_2^{16},1^{32}$ or $4_1^{12},4_2^{33},1^{68}$
\\ $4_1^2,4_2^2$ & $16^4,4_1^{19},4_2^{19},1^{32}$
 or $16^9,4_1^9,4_2^9,1^{32}$
\\ $4_1^3,1^4$ & $16^4,4_1^{19},4_2^{19},1^{32}$
\\ $4_1^2,4_2,1^4$ & $16^4,4_1^{22},4_2^{16},1^{32}$
\\ $4_1^2,1^8$ & $4_1^{33},4_2^{12},1^{68}$
\\ $4_1,4_2,1^8$ & $16^9,4_1^9,4_2^9,1^{32}$
\\ $4_1,1^{12}$ & $4_1^{12},4_2^{33},1^{68}$
\\ \hline $8_1^2$ & $8_1^8,8_2^3,4_1^{16},4_2^{16},1^{32}$
\\ $8_1,4_1^2$ & $16^5,4_1^{14},4_2^{20},1^{32}$
\\ $8_1,4_1,1^4$ & $16^5,4_1^{20},4_2^{14},1^{32}$
\\ $8_1,1^8$ & $8_1^{25},4_1^2,4_2^2,1^{32}$
\\ \hline 
\end{tabular}
\end{center}
(There is a third set of factors, $16^9,4_1^{12},4_2^6,1^{32}$, that has a correct Brauer character on $L(E_8)$ for the second row of the table, but since the factors on $\Lambda^2(M(D_8))$ are $4_1^6,4_2^{16},1^{32}$, a submodule of $L(E_8){\downarrow_{D_8}}$, we see that this cannot occur.)

If $H\leq A_8$ then, although $H$ need not stabilize a line on $M(A_8)$, there is another copy of $H$ with the same composition factors on $M(A_8)$---and hence on $L(E_8)$---that does stabilize a line on $M(A_8)$, hence inside an $A_7$-parabolic subgroup of $A_8$, which is inside $D_8$.

If $H$ is contained in an $E_7$-parabolic subgroup of $E_8$ then as we only consider composition factors we may assume that $H$ is contained in the $E_7$-Levi subgroup. By \cite[Proposition 6.3]{craven2017}, $H$ stabilizes a line on $M(E_7)$, hence lies inside a $D_6$-parabolic of $E_7$---contained in $D_8$---or an $E_6$-parabolic of $E_7$, thus an $E_6$-Levi by the arguments above. The same proposition shows that $H$ stabilizes a line on $M(E_6)$ as well, hence lies inside a $D_5$-parabolic of $E_6$---again, inside $D_8$---or inside $F_4$. Once again, the same \cite[Proposition 6.3]{craven2017} states that up to graph automorphism $H$ stabilizes a line on $M(F_4)$, hence lies in a parabolic subgroup of $F_4$ or $C_4$, hence inside $C_4$.

There are four embeddings of $H$ into $C_4$, and we list the actions of $H$ on $M(C_4)$, together with the $16$- and $26$-dimensional modules $L(\lambda_4)$ and $L(\lambda_2)$. As $L(E_8)$ restricts to $F_4$ as the sum of eight copies of $M(F_4)$, one of $M(F_4)$ twisted by the graph automorphism, and fourteen trivial factors, it is easy to compute the factors of $H$ on $L(E_8)$. (We use the fact that $C_4$ acts on $M(F_4)$ as $L(\lambda_2)$ and on $L(F_4)$ with factors $L(\lambda_2)$, $L(\lambda_4)$, $L(0)^2$ and $L(2\lambda_1)$, note the twist on this last module.) We again need only work up to automorphism. Thus we need the composition factors on $M(C_4)$, $L(\lambda_4)$ and $L(\lambda_2)$, from which we can deduce the two possible actions on $L(E_8)$. These are as follows.
\begin{center}\begin{small}
\begin{tabular}{ccccc}
\hline $M(C_4)$ & $L(\lambda_4)$ & $L(\lambda_2)$ & First action on $L(E_8)$ & Second action on $L(E_8)$
\\ \hline $4_1^2$ & $4_1^2,4_2,1^4$ & $4_1,4_2^4,1^6$ & $4_1^{33},4_2^{12},1^{68}$ & $4_1^{12},4_2^{33},1^{68}$
\\ $4_1,4_2$ & $16$ & $16,4_1,4_2,1^2$ & $16^9,4_1^9,4_2^9,1^{32}$ & $16^9,4_1^9,4_2^9,1^{32}$
\\ $4_1,1^4$ & $4_2^4$ & $4_1^4,4_2,1^6$ & $4_1^{12},4_2^{33},1^{68}$& $4_1^{33},4_2^{12},1^{68}$
\\ \hline $8_1$ & $8_2^2$ & $8_1,4_1^2,4_2^2,1^2$ & $8_2^{25},4_1^2,4_2^2,1^{32}$ & $8_1^8,8_2^3,4_1^{16},4_2^{16},1^{32}$
\\ \hline
\end{tabular}
\end{small}
\end{center}

Finally, suppose that $H$ embeds in any other maximal positive-dimensional subgroup $\mb X$. If $\mb X$ has a factor $A_1$, $A_2$ or $G_2$ then it may be removed, and if the resulting subgroup may be placed inside a previously considered subgroup, then we are done. This completes the proof for $\mb X$ one of $E_7A_1$, $E_6A_2$, $F_4G_2$, and all but one parabolic subgroup, leaving us with $A_4A_4$ and the $A_4A_3$-parabolic of $A_4A_4$, so just $A_4A_4$. In this case, $H$ stabilizes a line on both $M(A_4)$s, so $H$ is contained in an $A_3A_3$-parabolic of $A_4A_4$, which is contained in $A_7$, completing the proof for embeddings of $\Alt(6)$.

We need also consider the triple cover $3\cdot \Alt(6)$, embedding in a Lie type group whose simply connected version has a centre of order $3$ but that does not occur in $\mb G$, i.e., $A_8$ and $E_6A_2$. The faithful simple modules for $3\cdot \Alt(6)$ are $3_1$, $3_2$ and $9$ (and their duals, on which the centre acts as the inverse). If $H$ acts irreducibly on $M(A_8)$, then the composition factors of $L(E_8){\downarrow_H}$ are $16^4,4_1^{19},4_2^{19},1^{32}$, which we have seen before. Otherwise $H$ lies in a copy of $(A_2)^3$ which lies inside $E_6A_2$.

Inside $E_6A_2$, we obtain from the proof of Proposition \ref{prop:sp4ine6} below the composition factors of $3\cdot\Alt(6)$ on $M(E_6)$ and $L(E_6)$, and together with the action of $H$ on $M(A_2)$ being $3_1^*$ or $3_2^*$, we obtain the following.
\begin{center}
\begin{tabular}{ccc}
\hline Action on $M(E_6)$ & Action on $L(E_6)$ & Action on $L(E_8)$
\\ \hline $3_1^8,3_2$ & $8_1^8,1^{14}$ & $8_1^8,8_2^3,4_1^{16},4_2^{16},1^{32}$ or $8_1^{25},4_1^2,4_2^2,1^{32}$
\\ $9,3_1^3,3_2^3$ & $8_1,8_2,4_1^6,4_2^6,1^{14}$ & $8_1^8,8_2^3,4_1^{16},4_2^{16},1^{32}$ or $8_1^3,8_2^8,4_1^{16},4_2^{16},1^{32}$
\\ \hline
\end{tabular}
\end{center}
These have already occurred, so we gain no new possibilities.
\end{proof}

We also need the stabilizers of lines on the minimal module for $E_6(q)$ and $E_7(q)$. (See \cite[Lemmas 5.4 and 4.3]{liebecksaxl1987}.)

\begin{lemma}\label{lem:e6linestabs} Let $G=E_6(q)$. There are three orbits of lines of the action of $G$ on $M(E_6)$, with line stabilizers as follows:
\begin{enumerate}
\item $F_4(q)$ acting on $M(E_6)$ as $L(\lambda_4)\oplus L(0)$ ($L(0)/L(\lambda_4)/L(0)$ in characteristic $3$);
\item a $D_5$-parabolic subgroup; $q^{16}D_5(q).(q-1)$, acting uniserially as \[L(\lambda_1)/L(\lambda_4)/L(0);\]
\item a subgroup $q^{16}.B_4(q).(q-1)$ acting indecomposably as \[L(0)/L(\lambda_4)/L(0),L(\lambda_1).\]
\end{enumerate}\end{lemma}

\begin{lemma}\label{lem:e7linestabs}  Let $G=E_7(q)$. There are five orbits of lines of the action of $G$ on $M(E_7)$, with line stabilizers as follows:
\begin{enumerate}
\item $E_6(q).2$ (the graph automorphism) acting semisimply with composition factors of dimensions $54,1,1$;
\item ${}^2\!E_6(q).2$ (the graph automorphism) acting semisimply with composition factors of dimensions $54,1,1$;
\item an $E_6$-parabolic subgroup $q^{27}.E_6(q).(q-1)$ acting uniserially as \[L(0)/L(\lambda_1)/L(\lambda_6)/L(0);\]
\item a subgroup $q^{1+32}.B_5(q).(q-1)$ acting indecomposably as \[L(\lambda_1),L(0)/L(\lambda_5)/L(\lambda_1),L(0);\]
\item a subgroup $q^{26}.F_4(q).(q-1)$ acting indecomposably as \[L(0),L(0)/L(\lambda_4)/L(\lambda_4)/L(0),L(0).\]
\end{enumerate}\end{lemma}

In the algebraic group, there are still three orbits for $E_6$, with stabilizers: $F_4$; one of the two conjugacy classes of $D_5$-parabolic subgroup, and the preimage of $B_4$ in the Levi subgroup of the other class of $D_5$-parabolic subgroups. There are four orbits of lines on $M(E_7)$, with stabilizers: an $E_6$-Levi subgroup; an $E_6$-parabolic subgroup; the preimage of $B_5$ in the Levi subgroup of a $D_6$-parabolic; the preimage of $F_4$ in the Levi subgroup of an $E_6$-parabolic subgroup. This can easily be seen from the previous two lemmas.

\newpage

\chapter{Techniques for proving the results}
\label{chap:techniques}

This short chapter, as with \cite[Chapter 9]{craven2015un}, aims to give a general overview of the ideas used in proving Theorems \ref{thm:maximale8} and \ref{thm:exceptionalmaximal}. We will outline the strategy in this subsection, and then use the example of $G_2$ to illustrate the methods in the next subsection.

\section{The strategy}

First, we use Theorem \ref{thm:classmaximalsubgroup} to define which particular subgroups we need to show are strongly imprimitive. These are: groups of semisimple rank at most half that of $\mb G$ and field size at most $9$, $\PSL_3(16)$ and $\PSU_3(16)$, and groups $\PSL_2(q)$, ${}^2\!B_2(q)$ and ${}^2\!G_2(q)$ for all $q$ such that the groups do not contain a semisimple element not in $T(\mb G)$ (see Definition \ref{defn:T(G)}) using Theorem \ref{thm:largeorderss}.

For those groups $H$ that remain, we show that $H$ always stabilizes the members of some $N_{\Aut^+(\mb G)}(H)$-orbit of subspaces of either $M(\mb G)$ or $L(\mb G)$ that is also stabilized by a positive-dimensional subgroup. We then apply Theorem \ref{thm:intersectionorbit} to see that $H$ is strongly imprimitive, as desired. In particular, if $H$ is a blueprint for some module then this is usually enough.

If $\mb G$ is not $E_8$ then we can use Theorem \ref{thm:largeorderss} to prove that some of the groups above are blueprints for the minimal module. However, there are two issues with this plan. Now it is possible for our subgroup $H$ to act irreducibly on $M(\mb G)$, whereas it cannot act irreducibly on $L(\mb G)$, and so it is obviously a blueprint. (Such groups are enumerated in \cite{liebeckseitz2004a}, and include for example $\PSL_3(3)$ in $E_6$ in characteristic $3$.) In addition, the minimal module is not always stable under $\Aut^+(\mb G)$, in particular when there is a graph automorphism, so we cannot always use this. With these two caveats, the remaining groups that possess a semisimple element as described in Theorem \ref{thm:largeorderss} are strongly imprimitive.

Thus we assume that we have a `small' group $H$. Using the tables of known character values that semisimple elements of $\mb G$ can take on $M(\mb G)$ and $L(\mb G)$, we produce a list of candidate possibilities for the composition factors for $M(\mb G){\downarrow_H}$ and $L(\mb G){\downarrow_H}$, the conspicuous sets. (If $\mb G$ is $E_8$ then we just have $L(\mb G)$.) Note that these have to be consistent in two different ways. First, the traces need to be consistent with the power map. This is equivalent to stating that the eigenvalues of semisimple elements of $H$ on the modules need to match those of semisimple classes of $\mb G$, not just the traces. The second is that specifying traces of elements on $M(\mb G)$ will force the elements of $H$ to belong to particular classes of $\mb G$, possibly with some ambiguity if more than one $\mb G$-class has the same eigenvalues on $M(\mb G)$. These classes have corresponding eigenvalues on $L(\mb G)$, and those must correspond to a conspicuous set of factors as well.

If we have elements $x$ of large order, say much larger than $30$ or $40$, we do not tend to have tables of their eigenvalues, as there are too many to compute easily. What we can do, if the order of $x$ is composite, is use the \emph{roots trick}: if $x^n$ has smaller order than $x$, then find an element of a maximal torus $\mb T$ in the class to which $x^n$ belongs (which we know from the eigenvalues of $x^n$) and then consider all $n$th roots of $x^n$ in $\mb T$ of order that of $x$. This much smaller set can be enumerated and eigenvalues found. They can then be used to compute the conspicuous sets of factors. If this is still too many, if $x$ has order $mn$ then we can specify the eigenvalues of $x^n$ and $x^m$, run through all elements of order $n$ in $\mb T$ with eigenvalues that of $x^m$ and multiply them by our given element of $\mb T$ corresponding to $x^n$ to find all possibilities for the eigenvalues of $x$. (This is used exactly once in this article, for ${}^2\!B_2(512)$. Usually $n$ is one of $2,3,5,7$, and so there are few enough $n$th roots to just use all of them. In the one case where we don't, $n=13$.)

We therefore obtain a list of sets of composition factors for $M(\bG)$ and/or $L(\bG)$. We then must go through each in turn, so we will call them `Case 1', `Case 2', and so on. Many proofs subdivide into these cases, and we will write \textbf{Case 1} and so on at the start of each new set of composition factors.

\medskip

Thus at this point we may assume that we are given a conspicuous set of composition factors for $M(\mb G)$ or $L(\mb G)$. We first compute its pressure, so the sum of the $1$-cohomologies of the composition factors, minus the number of trivial composition factors. If we have a trivial composition factor and this pressure is non-positive, then $H$ stabilizes a line or hyperplane on the module by Proposition \ref{prop:pressure}, and then $H$ is strongly imprimitive by Lemma \ref{lem:fix1space}. Many conspicuous sets of composition factors fail this easy test.

\medskip

We are then left with those cases where we need to study the potential module structure of $M(\mb G){\downarrow_H}$ or $L(\mb G){\downarrow_H}$. For this we have a few tools: importantly, $L(\bG)^\circ$ is always self-dual (unless $\bG=F_4$ and $p=2$), and $M(\bG)^\circ$ is as well unless $\bG=E_6$. In particular, if $H$ does not stabilize a line on $L(\bG)^\circ$ then $H$ does not stabilize a hyperplane either. Often we argue by contradiction that $H$ must stabilize a line on $L(\bG)^\circ$, and assuming the contrary will necessarily mean that $H$ does not stabilize a hyperplane either. We will usually not explicitly note this each time we use this extra information, as we will use it so frequently.

The possible Jordan block structures of unipotent elements of $\mb G$ are tabulated in \cite{lawther1995}, so our unipotent elements of $H$ must act in a way in those tables. Furthermore, Lemma \ref{lem:genericmeansblueprint} states that if $H$ contains an element in some of the unipotent classes of $\mb G$, the generic ones, then $H$ is always a blueprint for that module. Thus we may assume that $H$ acts irreducibly on the module in question, or is strongly imprimitive.

We also have subgroups of $H$, particularly centralizers of semisimple elements. The centralizers of semisimple elements of low order in exceptional groups are known, and tabulated in various places, for example \cite{frey1998a} for $E_8$. If a proper subgroup $L$ of $H$ is known to be imprimitive, either because we will prove it in this paper or for some other reason, one may work through the members of $\ms X$ to try to understand the module structures of $M(\mb G){\downarrow_L}$ and $L(\mb G){\downarrow_L}$. This must be compatible with the module structures for $H$.

In particular, if we can show that some $kL$-submodule of the space is stabilized by both $H$ and a positive-dimensional subgroup $\mb X$ (but not the whole of $\mb G$), then $\gen{H,\mb X}$ also stabilizes that space, and must be proper in $\mb G$. This proves that $H$ is not Lie primitive, and if the same holds for an orbit of spaces under $N_{\Aut^+(\mb G)}(H)$, then this shows that $H$ is strongly imprimitive by Theorem \ref{thm:intersectionorbit}. One good way to do this is in the case $L=\gen x$: we look for elements of $\mb T$ that stabilize all of the eigenspaces of the action of $x$ that appear in the $k\gen x$-submodule. Then those elements of $\mb T$ must also stabilize the submodule. Find enough such elements and the stabilizer of the subspace must be positive dimensional. The roots trick is a useful tool to look for these elements of $\mb T$.

To constrain the structure of $M(\mb G){\downarrow_H}$ and $L(\mb G){\downarrow_H}$, we often must use the computer package Magma, which allows us to construct the largest possible module with a given socle and certain composition factors or other restrictions. The main idea is that of a `pyx'. If $W$ is a $kH$-module, then a \emph{pyx} for $W$ is simply any module $V$ such that $W$ is a subquotient of $V$. Normally, we will construct a pyx $V$ for $W$ by forcing $\soc(V)=\soc(W)$, and then build $W$ according to some universal property. The most obvious one is that $V$ is the $\cf(W)$-radical of $P(\soc(W))$, which of course must be an overmodule of $W$. More generally, if we have a filtration of $W$, then we may construct $V$ by taking repeated radicals with composition factors according to the filtration of $W$. This will in general be smaller than the $\cf(W)$-radical, and hopefully small enough that one can gain information about $W$ from $V$.

We will now see these techniques in action, as we (almost) prove that every defining-characteristic subgroup of $G_2$ is strongly imprimitive.

\section{An example}
\label{sec:example}
We use the techniques and preliminary lemmas on a toy example: $G_2$. While the maximal subgroups of $G_2(q)$ are known \cite{cooperstein1981,kleidman1988}, this will illustrate some of the points involved. We also deliberately do not do it in the easiest way in order to show more of the techniques we will use.

Let $\mb G$ be of type $G_2$ in characteristic $p$, and let $H\leq \mb G$ be a simple group that is of Lie type in characteristic $p$. We suppose that $H\cong \PSL_2(q)$ with $q$ a power of $p$.

\medskip

\noindent\textbf{Step 1}: Use Theorem \ref{thm:largeorderss} to bound $q$.

\medskip \noindent Theorem \ref{thm:largeorderss} states that $H$ is a blueprint for $L(G_2)$, and hence strongly imprimitive, if $H$ contains a semisimple element of order at least $13$. Since $H$ contains a semisimple element of order $(q+1)/2$ (unless $p=2$, in which case it possesses an element of order $q+1$), we have that $H$ is blueprint for $L(G_2)$ whenever $q\geq 25$. Thus we may assume from now on that $q\leq 23$ (and we can exclude $q=16$ as well). So our set of $q$ is
\[\{ 4,5,7,8,9,11,13,17,19,23\}.\]

\noindent\textbf{Step 2}: Switch to $M(G_2)$ and use Theorem \ref{thm:largeorderss} to exclude more cases.

\medskip \noindent We cannot use Theorem \ref{thm:largeorderss} when $p=3$ as $M(G_2)$ is not graph stable in this case. Also, if $p\geq 7$ then from \cite[Table 1.2]{liebeckseitz2004a} we see that $H$ can act irreducibly on $M(G_2)$. If $H$ contains a semisimple element of order at least $5$ then Theorem \ref{thm:largeorderss} applies, but that is true for $\PSL_2(q)$ for all $q$ except $q=5,7$. Thus in all other cases we may assume that $H$ acts irreducibly on $M(G_2)$ (and thus this excludes $q=4,8$).

\medskip

\noindent \textbf{Step 3}: Use unipotent actions to eliminate more cases.

\medskip\noindent Let $u$ be an element of order $p$ in $H$. If $u$ is generic for $L(G_2)$ then $u$ is a blueprint for $L(G_2)$ by Lemma \ref{lem:genericmeansblueprint}, and hence so is $H$. Thus $H$ is strongly imprimitive. Consulting \cite[Table 2]{lawther1995}, we see that this is the case whenever $p\geq 11$. If $q=5$, we see from \cite[Table 1]{lawther1995} that $u$ is generic for $M(G_2)$, and $H$ cannot act irreducibly on $M(G_2)$, so $H$ is a blueprint for $M(G_2)$ and again $H$ is strongly imprimitive.

Thus $q$ must be either $7$ or $9$.

\medskip

\noindent \textbf{Step 4}: Completing $q=9$ using pressure and graph automorphism.

\medskip\noindent We now look at the conspicuous sets of composition factors for $M(G_2){\downarrow_H}$. The simple modules for $\PSL_2(9)$ are of dimensions $1$, $3$, $3$, $4$ and $9$. If there is a trivial factor then the dimensions must be $3^2,1$ or $3,1^4$; since the $3$s have zero $1$-cohomology, in both cases $H$ has negative pressure on $M(G_2)$ and so $H$ stabilizes a line on $M(G_2)$ by Proposition \ref{prop:pressure}, hence is strongly imprimitive by Lemma \ref{lem:fix1space}. (In fact, we note that if $H$ stabilizes a line on $M(G_2)$ then its stabilizes a line on $L(G_2)$, and hence is strongly imprimitive. This distinction will be used for $F_4$ and $p=2$ in Chapter \ref{chap:subsinf4}.)

The other case is that the composition factors are of dimensions $4$ and $3$, and up to field automorphism of $H$ there is a unique such set of composition factors. As $M(G_2)$ is self-dual, the restriction to $H$ must be semisimple.

\begin{table}\begin{center}\begin{tabular}{cccccc}
\hline Module &1A & 2A & 4A&5A&5B
\\ \hline $1$&$1$&$1$&$1$&$1$&$1$
\\ $3_1$&$3$&$-1$&$1$&$\lambda+1$&$\bar\lambda+1$
\\ $3_2$&$3$&$-1$&$1$&$\bar\lambda+1$&$\lambda+1$
\\ $4$&$4$&$0$&$-2$&$-1$&$-1$
\\ $9$&$9$&$1$&$1$&$-1$&$-1$
\\ \hline
\end{tabular}\end{center}
\caption{Brauer character table for $\PSL_2(9)$ and $p=3$. Here, $\lambda$ is the sum of a fifth root of unity and its inverse.}
\label{t:brauercharsofpsl29}
\end{table}
The Brauer character table of $\PSL_2(9)$ is given in Table \ref{t:brauercharsofpsl29}. If $H$ acts on $M(G_2)$ as $3_1\oplus 4$, then we can use this table, and the traces of semisimple elements on $M(G_2)$ and $L(G_2)$, to deduce the composition factors of the image of $H$ under the graph automorphism on $M(G_2)$. This is because $L(G_2)$ has as composition factors $M(G_2)$ and its image under the graph automorphism.

The traces of elements of orders $1$, $2$, $4$, $5$ and $5$ of $H$ on $M(G_2)$ are $7$, $-1$, $-1$, $\lambda$ and $\bar\lambda$ respectively. This uniquely determines the semisimple classes of $G_2$ to which these elements of $H$ belong, and their traces on $L(G_2)$ are given by $14$, $-2$, $2$, $\bar\lambda+2$ and $\lambda+2$ respectively. These traces yield composition factors $1,3_1,3_2,3_2,4$, and so the action of $H$ on the factor of $L(G_2)$ other than $M(G_2)$ is $1\oplus 3_2\oplus 3_2$. In particular, either $H$ or its image under the graph automorphism stabilizes a line on $M(G_2)$. Hence $H$ always stabilizes a line on $L(G_2)$, and is strongly imprimitive by Lemma \ref{lem:fix1space}.

\medskip

\noindent \textbf{Step 5}: Completing $q=7$ using $L(G_2)$.

\medskip\noindent As $H$ acts irreducibly on $M(G_2)$ we cannot do much, so we switch to $L(G_2)$. As with Step 4, we compute the traces of elements of $H$ on $M(G_2)$, which are $7$, $-1$, $1$ and $-1$ for orders $1$, $2$, $3$ and $4$ respectively. These correspond to traces on $L(G_2)$ of $14$, $-2$, $-1$ and $2$ respectively.

\begin{table}\begin{center}\begin{tabular}{ccccc}
\hline Module &1A & 2A & 3A&4A
\\ \hline $1$&$1$&$1$&$1$&$1$
\\ $3$&$3$&$-1$&$0$&$1$
\\ $5$&$5$&$1$&$-1$&$-1$
\\ $7$&$7$&$-1$&$1$&$-1$
\\ \hline
\end{tabular}\end{center}
\caption{Brauer character table for $\PSL_2(7)$ and $p=7$.}
\label{t:brauercharsofpsl27}
\end{table}
We normally use a computer to calculate the conspicuous sets of composition factors, as it is well suited to linear algebra problems.  In this case the set of composition factors is $5,3,3,3$, as we can see from Table \ref{t:brauercharsofpsl27}.

To understand how these composition factors assemble into the module for $L(G_2){\downarrow_H}$, we use the action of the unipotent element $u$. This acts on the module $7$ with a single Jordan block, hence lies in class $G_2$ by \cite[Table 1]{lawther1995}, i.e., it is regular unipotent. The action of this class on $L(G_2)$ is given in \cite[Table 2]{lawther1995}, and the blocks are $7^2$. Thus $H$ acts projectively on $L(G_2)$, and thus must act as the projective module $P(3)$, which has structure
\[ 3/3,5/3.\]

At this stage the representation theory has done all that it can: it has determined completely the actions of $H$ on the minimal and adjoint modules. This type of subgroup was referred to as a \emph{Serre embedding} in \cite{craven2015un}, and is not amenable to our methods, being solved in general in \cite{burnesstesterman2019}. (Note that this case is mentioned in \cite{liebeckseitz2004a} but does not seem to appear in their proof, although it is of course handled in \cite{kleidman1988}.)

\medskip

With such a small case we cannot showcase all of our techniques, in particular the construction of a pyx. In the final step the projective was the whole module, so the pyx is just the projective module itself.

We also cannot use the roots trick to find elements of large order that stabilize the subspace $3$ above. Although they exist (the maximal $A_1$ subgroup stabilizes a subspace decomposition $3/8/3$) since they do not stabilize \emph{all} copies of $3$ in the module, one will normally not be able to prove that some semisimple element in $\mb T$ of large order stabilizes exactly the $3$ in the socle. Indeed, in this case we cannot.

\newpage

\chapter{Modules for groups of Lie type}
\label{chap:labelling}

In this chapter we will discuss facts about the subgroups $H$ that we will consider. In particular, we will give the dimensions of the simple modules for $H$ when $q=2,3,5,7$ (the dimensions for $q=4,8,9,16$ can be deduced from Steinberg's tensor product theorem) and the dimensions of the $1$-cohomologies of the simple modules for $q$ a prime, and for $q=8,9,16$ for those groups that we need. For $H$ of ranks $2$ and $3$, and rank $4$ for $q=2,3,5$, we give the dimensions of the modules and the dimensions of the $1$-cohomologies of simple modules of dimension at most $124$, and of self-dual modules of dimension at most $248$. These are given in Tables \ref{t:modules2}, \ref{t:modules3}, \ref{t:modules5} and \ref{t:modules7}.

For $\mb G=E_7$, we also need the dimensions of all faithful modules (of dimension at most $28$, or self-dual and at most $56$) for a central extension $2\cdot H$ for $H$ of rank at most $3$ and $q$ odd, which we give in Table \ref{t:mods2centre}. Finally, for $\mb G=E_6$, we need dimensions of faithful modules (of dimension at most $27$) for a central extension $3\cdot H$ for $H$ of rank at most $3$ and $q$ not a power of $3$, which are in Table \ref{t:mods3centre}.

One may compute $\Ext^1$ between simple modules in Magma easily: if $\mathtt{A}$ and $\mathtt{B}$ are modules, $\mathtt{Ext(A,B)}$ computes the space of extensions. Furthermore, extensions between simple modules for $\PSL_3(q)$, $\PSU_3(q)$, $\PSp_4(q)$ and ${}^2\!B_2(q)$ for $q$ even are given in \cite{sin1992,sin1992b}, and extensions for $G_2(q)$ and ${}^2\!G_2(q)$ for $q$ a power of $3$ in \cite{sin1993}. We tabulate $\Ext^1(M,k)$ here for $M$ a simple module, as these are the extensions that we need the most. However, we will use information about extensions between other simple modules, and we use Magma to compute this information, together with modules realizing the extensions, using the commands
\begin{verbatim}
V,rho:=Ext(A,B); E:=MaximalExtension(A,B,V,rho);
\end{verbatim} 
(This requires that $B$, but not necessarily $A$, is simple.)

\begin{table}
\begin{center}
\begin{tabular}{cc}
\hline Group & Modules of dimension at most $248$
\\ \hline $\PSL_3(2)$&$1,\mb{3}^\pm,8$
\\ $\PSp_4(2)'$&$1,\mb{4_1},\mb{4_2},8_1,8_2$
\\ $G_2(2)'$&$1,\mb{6},14,32_1,32_2$
\\\hline$\PSL_4(2)$&$1,4^\pm,\mb{6},\mb{14},\mb{20}^\pm,64$
\\$\PSU_4(2)$&$1,\mb{4^\pm},6,\mb{14},20^\pm,64$
\\$\PSp_6(2)$&$1,\mb{6},8,14,\mb{48},64,112$
\\\hline $\PSL_5(2)$&$1,5^\pm,10^\pm,24,\mb{40_1^\pm},40_2^\pm,\mb{74}$
\\$\PSU_5(2)$&$1,5^\pm,\mb{10^\pm},24,40_1^\pm,40_2^\pm,\mb{74}$
\\$\PSp_8(2)$&$1,\mb{8},16,\mb{26},\mb{48},128,160,\mb{246}$
\\ $\POmega^+_8(2)$&$1,8_1,8_2,8_3,\underline{26},\mb{48_1},\mb{48_2},\mb{48_3},160_1,160_2,160_3,\mb{246}$
\\ $\POmega^-_8(2)$&$1,8_1,8_2,8_3,\underline{26},\mb{48_1},48_2,48_3,160_1,\mb{160_2},\mb{160_3},\mb{246}$
\\ ${}^3\!D_4(2)$&$1,8_1,8_2,8_3,\underline{26},48_1,48_2,48_3,\mb{160_1},\mb{160_2},\mb{160_3},\mb{246}$
\\ $F_4(2)$ & $1,26_1,26_2,\mb{246_1},\mb{246_2}$
\\ ${}^2\!F_4(2)'$ & $1,26,246$
\\ \hline
\end{tabular}\end{center}
\caption{Simple modules of dimension at most $248$ (or dual pair up to dimension $124$) for groups over a field of two elements. Bold indicates a $1$-dimensional $1$-cohomology, underlined indicates a $2$-dimensional.}
\label{t:modules2}
\end{table}

\begin{table}\begin{center}
\begin{tabular}{cc}
\hline Group & Modules of dimension at most $248$
\\ \hline $\PSL_3(3)$&$1,3^\pm,6^\pm,\mb 7,\mb{15}^\pm,27$
\\ $\PSU_3(3)$&$1,3^\pm,\mb 6^\pm,\mb 7,15^\pm,27$
\\$\PSp_4(3)$&$1,\mb 5,10,\mb{14},25,81$
\\ $G_2(3)$&$1,7_1,7_2,27_1,27_2,\underline{49},189_1,189_2,729$
\\ ${}^2\!G_2(3)'$&$1,\mb 7,9_1,9_2,9_3$
\\\hline$\PSL_4(3)$&$1,6,10^\pm,15,\mb{19},\underline{44},45^\pm,\mb{69},156$
\\$\PSU_4(3)$&$1,15,\mb{\underline{19}},45^\pm,\mb{69},156$
\\$\Omega_7(3)$&$1,7,21,27,35,\mb{63},\mb{\underline{132}},189_1,189_2$
\\$\PSp_6(3)$&$1,\mb{13},21,57,84,\mb{90},189$
\\\hline $\PSL_5(3)$&$1,5^\pm,10^\pm,15^\pm,24,30^\pm,45_1^\pm,45_2^\pm,51,65^\pm,105^\pm,\mb{199}$
\\$\PSU_5(3)$&$1,5^\pm,10^\pm,15^\pm,24,30^\pm,\mb{45_1^\pm},45_2^\pm,51,65^\pm,105^\pm,\mb{199}$
\\$\Omega_9(3)$&$1,9,36,\mb{43},84,126,147$
\\$\PSp_8(3)$&$1,27,36,\mb{41}$
\\ $\POmega^+_8(3)$&$1,28,35_1,35_2,35_3,\mb{\underline{195}}$
\\ $\POmega^-_8(3)$&$1,8,28,35_1,35_2,35_3,56,104,\mb{195},224_1,224_2$
\\ ${}^3\!D_4(3)$&$1,8_i,28,35_i,56_i,104_i,195,224_j$
\\ $F_4(3)$&$1,\mb{25},52,196$
\\ \hline
\end{tabular}\end{center}
\caption{Simple modules of dimension at most $248$ (or dual pair up to dimension $124$) for groups over a field of three elements. Bold indicates a $1$-dimensional $1$-cohomology, underlined indicates $2$-dimensional, and bold and underlined indicates $3$-dimensional. Here $1\leq i\leq 3$ and $1\leq j\leq 6$ for ${}^3\!D_4(3)$.}
\label{t:modules3}
\end{table}

\begin{table}
\begin{center}
\begin{tabular}{cc}
\hline Group & Modules of dimension at most $248$
\\ \hline $\PSL_3(5)$&$1,3^\pm,6^\pm,8,10^\pm,15_1^\pm,15_2^\pm,18^\pm,19,35^\pm,\mb{39}^\pm,60^\pm,\mb{63},90^\pm,125$
\\ $\PSU_3(5)$&$1,8,10^\pm,\underline{19},35^\pm,\mb{63},125$
\\$\PSp_4(5)$&$1,5,10,\mb{13},30,35_1,35_2,55,\mb{68},86,105,\mb{115}$
\\ $G_2(5)$&$1,7,14,27,64,77_1,77_2,97,182,189$
\\\hline$\PSL_4(5)$&$1,15,20,35^\pm,45^\pm,\mb{83},85,175$
\\$\PSU_4(5)$&$1,6,10^\pm,15,20,35^\pm,45^\pm,50,58,\mb{68^\pm},70^\pm,\mb{83},85,175,184$
\\$\Omega_7(5)$&$1,7,21,27,35,77,105,141,168,182$
\\$\PSp_6(5)$&$1,14,21,63,70,90,126,189$
\\\hline $\PSL_5(5)$&$1,5^\pm,10^\pm,15^\pm,\mb{23},35^\pm,40^\pm,45^\pm,50^\pm,$
\\ &$70_1^\pm,70_2^\pm,75,103^\pm,105^\pm,121^\pm,200$
\\$\PSU_5(5)$&$1,5^\pm,10^\pm,15^\pm,\mb{23},35^\pm,40^\pm,45^\pm,50^\pm,$
\\ &$70_1^\pm,70_2^\pm,75,103^\pm,105^\pm,121^\pm,200$
\\$\Omega_9(5)$&$1,9,36,44,84,126,156,231$
\\$\PSp_8(5)$&$1,27,36,42$
\\ $\POmega^+_8(5)$&$1,28,35_1,35_2,35_3$
\\ $\POmega^-_8(5)$&$1,8,28,35_1,35_2,35_3,56,104,160,168_1,168_2$
\\ ${}^3\!D_4(5)$&$1,8_i,28,35_i,56_i,104_i,160_i,168_j,$
\\ $F_4(5)$&$1,26,52$
\\ \hline
\end{tabular}\end{center}
\caption{Simple modules of dimension at most $248$ (or dual pair up to dimension $124$) for groups over a field of five elements. Bold indicates a $1$-dimensional $1$-cohomology, underlined indicates a $2$-dimensional. Here $1\leq i\leq 3$ and $1\leq j\leq 6$ for ${}^3\!D_4(5)$.}
\label{t:modules5}
\end{table}

\begin{table}
\begin{center}
\begin{tabular}{cc}
\hline Group & Modules of dimension at most $248$
\\ \hline $\PSL_3(7)$&$1,8,10^\pm,27,28^\pm,35^\pm,37,\mb{71}^\pm,117,215$
\\ $\PSU_3(7)$&$1,3^\pm,6^\pm,8,10^\pm,15_1^\pm,15_2^\pm,21^\pm,24^\pm,27,28^\pm,33^\pm,35^\pm,$
\\ &$\mb{36}^\pm,37,42^\pm,63^\pm,71^\pm,75^\pm,105^\pm,114^\pm,117,\mb{215}$
\\$\PSp_4(7)$&$1,5,10,14,25,35_1,35_2,\mb{54},71,84,91,105,140,\mb{149},154,199,206,231$
\\ $G_2(7)$&$1,7,14,\mb{26},38,77_1,77_2,182,189,248$
\\\hline$\PSL_4(7)$&$1,6,10^\pm,15,20,35^\pm,45^\pm,50,64,70^\pm,84_1,84_2^\pm,105,175,196,231,236$
\\$\PSU_4(7)$&$1,15,20,35^\pm,45,45^\pm,84,105,175,231$
\\$\Omega_7(7)$&$1,7,21,\mb{26},35,77,105,168,182,189$
\\$\PSp_6(7)$&$1,14,21,70,84,\mb{89},126,171,189$
\\ \hline
\end{tabular}\end{center}
\caption{Simple modules of dimension at most $248$ (or dual pair up to dimension $124$) for groups over a field of seven elements. Bold indicates a $1$-dimensional $1$-cohomology.}
\label{t:modules7}
\end{table}

\begin{table}
\begin{center}
\begin{tabular}{cc}
\hline Group & Modules with non-zero $1$-cohomology
\\\hline  $\PSL_5(4)$ & $50_{i,j}^\pm,74_i$
\\ $\PSp_8(4)$ & $8_1,8_2,26_1,26_2,246_1,246_2$
\\ $\Omega_8^\pm(4)$, ${}^3\!D_4(4)$ & $\boldsymbol{26_i},208_{i,j},246_i$
\\ $\PSL_4(q)$, $q=4,8$ & $14_i,24_{i,i-1}^\pm,84_{i,i-1}$
\\ $\PSU_4(q)$, $q=4,8$ & $14_i,24_{i,i-1}^\pm,84_{i,i-1}$
\\ $\PSp_6(q)$, $q=4,8$ & $6_i,48_i,84_{i,i-1}$
\\ $\PSL_3(4)$ & $\mb{9^\pm}$
\\ $\PSL_3(q)$, $q=8,16$ & $9_{i,i+1}^\pm$
\\ $\PSU_3(q)$, $q=4,8,16$ & $9_{12}$ and images under field automorphism
\\ $\PSp_4(q)$, $q=4,8$ & $4_i$
\\ ${}^2\!B_2(q)$, all $q$ & $4_i$
\\ $G_2(4)$ & $6_i,36,84_{i,j}$
\\ $G_2(8)$ & $6_i,84_{i,i-1}$
\\ \hline $\PSL_3(9)$, $\PSU_3(9)$ & $7_i,21_{i,j}^\pm$
\\ $\PSp_4(9)$ & $64_{i,j},125_{i,j}$
\\ $G_2(9)$, ${}^2\!G_2(q)$, $q>3$ & $49$ (not all, see \cite{sin1993})
\\ \hline
\end{tabular}\end{center}
\caption{Simple modules of dimension at most $248$ (or dual pair up to dimension $124$) with non-zero $1$-cohomology for groups over a field of four, eight, nine and sixteen elements. Bold indicates a $2$-dimensional $1$-cohomology. Only those groups for which this information is needed in this article are listed.}
\label{t:modules48916}
\end{table}

\begin{table}
\begin{center}
\begin{tabular}{cc}
\hline Group & Modules of dimension at most $56$
\\\hline $\Sp_4(3)$&$4,16,40$
\\ $\Sp_4(5)$&$4,12,20,40,52$
\\ $\Sp_4(7)$&$4,16,20,24,44,56$
\\\hline $\SL_4(3)$&$4^\pm,16^\pm$
\\ $2\cdot \PSL_4(5)$&$6,10^\pm,50$
\\ $\SL_4(7)$&$4^\pm,20_1^\pm,20_2^\pm$
\\ $2\cdot \PSU_4(3)$&$6,10^\pm,44$
\\ $\SU_4(5)$&$4^\pm,20_1^\pm,20_2^\pm$
\\ $2\cdot \PSU_4(7)$&$6,10^\pm,50$
\\ $\Sp_6(3)$&$6,14,50$
\\ $\Sp_6(q)$, $q=5,7$&$6,14,56$
\\ $\Spin_7(q)$, $q=3,5$&$8,48$
\\ $\Spin_7(7)$&$8,40$
\\ \hline
\end{tabular}\end{center}
\caption{Simple modules of dimension at most $56$ (or dual pair up to dimension $28$) for $2$-fold central extensions of groups of rank at most $3$.}
\label{t:mods2centre}
\end{table}

\begin{table}
\begin{center}
\begin{tabular}{cc}
\hline Group & Modules of dimension at most $27$
\\ \hline $3\cdot \PSL_3(4)$&$3_1,3_2^*,24_1,24_2^*$
\\ $3\cdot \PSL_3(7)$&$3,6,15_1,15_2,21,24$
\\ $3\cdot \PSU_3(5)$&$3,6,15_1,15_2,18$
\\ $3\cdot \PSp_4(2)'$&$3_1,3_2^*,9$
\\ \hline
\end{tabular}\end{center}
\caption{Simple modules of dimension at most $27$ for $3$-fold central extensions of groups of rank at most $3$. (We only include modules from the one of the two (dual) faithful blocks.)}
\label{t:mods3centre}
\end{table}

\medskip

If a simple module $M$ is of dimension $n$, and is the only $kH$-module of that dimension (e.g., the Steinberg module) then we just denote it by $n$.

If $M$ is not self-dual but it and its dual are the only modules of this dimension then we label $M$ and its dual by $n$ and $n^*$, and use $n^\pm$ to mean both or either $n$ and $n^*$ (it is clear from context which we mean). However, when giving precise module structures, we will need to be specific about exactly which of the two modules is $n$ and which is $n^*$. Of course, we can make a global choice for one module (as duality induces an automorphism, so we can never fix modules precisely, only up to duality of all modules simultaneously), but then all other modules are fixed. We now describe, for each group, the choices that we make in this article. (We only include the modules that actually appear in our sets of composition factors, as these are all that is required. Thus, for example, for $\PSU_3(7)$ we do not need a label for $33^\pm$.)

If there are multiple modules of the same dimension up to duality then we use subscripts to distinguish them. We place self-dual modules first, so for example $20_1$, $20_2$ and $20_2^*$ would be three modules of dimension $20$.

For $\PSL_3(q)$ and $\PSU_3(q)$, we can be guided by the fact that $3$ lies in a non-principal block for $q=7$, and so we could like to label $6$ and $3$ so as to lie in the same block. With these kind of guidelines, we make the following choices.

\begin{itemize}
\item $\PSL_3(2)$: only $3,3^*$ to fix, so nothing to do.
\item $\PSL_3(3)$ and $\PSU_3(3)$. Fix $3$. Then $6=S^2(3^*)$ and $15=\Lambda^2(6^*)$.
\item $\PSL_3(5)$. Fix $3$. Then $6=S^2(3^*)$, $10$ lies in $6\otimes 3^*$, $15_1$ lies in $S^2(6^*)$, $15_2=\Lambda^2(6^*)$, $18^*$ is the socle of $3\otimes 15_1$, and $35$ lies in $\Lambda^2(10^*)$. Finally, $39$ is the socle of $\Lambda^2(15_1^*)$.
\item $\PSU_3(5)$. Fix signs so that $3,6,15_1,15_2,18,39$ all lie in the same block. Then $10$ lies in $6\otimes 3^*$, $15_1$ lies in $S^2(6^*)$, $15_2=\Lambda^2(6^*)$ and then $35$ lies in $\Lambda^2(10^*)$.
\item $\PSL_3(7)$. Fix $10$. Then $28$ lies in $S^2(10^*)$ and $35$ lies in $\Lambda^2(10^*)$. The module $\Lambda^2(27)$ has a summand $10^*/71/10^*$.
\item $\PSU_3(7)$. Fix $3$. Set $6=S^2(3^*)$, $10$ lies in $3\otimes 6^*$, $15_1$ lies in $S^2(6^*)$ and $15_2=\Lambda^2(6^*)$. Both $21$ and $24$ lie in $3^*\otimes 15_1^*$, and $28$ lies in $S^2(10^*)$. We have that $35$ lies in $\Lambda^2(10^*)$, $42$ lies in $\Lambda^2(15_1^*)$, and $71$ lies in $\Lambda^2(28^*)$.
\item $\PSL_4(2)$ and $\PSU_4(2)$. Fix $4$. Then $20$ lies in $4\otimes 6$.
\item $\PSL_4(3)$. In the principal block, we only need to fix $10^\pm$ and $45^\pm$. Fix $10$. Then $45=\Lambda^2(10^*)$.
\item $\PSU_4(3)$. In the principal block, only $45$ needs to be fixed.
\item $\PSL_4(q)$ and $\PSU_4(q)$ for $q=5,7$. We do not need to specify the modules here.
\item $\PSL_5(2)$ and $\PSU_5(2)$. Fix $5$. Set $10=\Lambda^2(5^*)$. Then $40_1$ lies in $\Lambda^2(10)$, and $40_2$ lies in $5\otimes 10^*$.
\item $\PSL_5(3)$ and $\PSU_5(3)$. Fix $5$. Set $10=\Lambda^2(5^*)$ and set $30$ via $5\otimes 10^*=10/30/10$. We have $S^2(10)=5/45_1/5$ and $\Lambda^2(10)=45_2$.
\item $\PSL_5(q)$ and $\PSU_5(q)$ for $q=5,7$. We do not need to specify the modules here.
\end{itemize}

We also have some labelling difficulties for $\Sp_4$ and $G_2$.

For $G_2(5)$ and $G_2(7)$, there are two $77$-dimensional simple modules, which we label $77_1$ and $77_2$. The first of these has highest weight label 30 and the second 02. For $G_2(3)$, fix $7_1$ and let $27_i$ be defined by $S^2(7_i)\cong 27_i\oplus 1$.

For $\Sp_4(5)$ and $\Sp_4(7)$ there are two simple modules of dimension $35$. The module $35_1$ has highest weight 21 and the module $35_2$ has highest weight 40. For $\Sp_4(2^n)$ and the Suzuki groups, fix $4_1$ and let a generator for the outer automorphism group cycle the $4_i$, so that $\Lambda^2(4_i)\cong 1/4_{i+1}/1$. This is the one case where we deviate from the convention below that the image under the standard Frobenius map of $4_i$ should be $4_{i+1}$, it is $4_{i+2}$ in this case.

\medskip

When $p=2,3$, we will also have to consider modules for groups $H(q)$ for $q$ a proper power of $p$. In this case we need a notation for the modules. As in \cite{craven2015un}, and above, we give the dimension of the module, and with a subscript if there is more than one module of that dimension. The subscript is chosen as follows:

If $M$ is a $p$-restricted module of dimension $n$, we write $n_1$ for this module (and $n_1^*$ for its dual if it is not self-dual). The image of $n_1$ under the field automorphism will be denoted $n_2$, then $n_3$, and so on.

Modules that are not twists of $p$-restricted modules are a tensor product of twists of $p$-restricted modules. If $M$ is the tensor product of $n_i$ and $m_j$, and $n\leq m$, then $M$ is denoted by $nm_{i,j}$, and so on for products of more than two modules. The dual of this module is again denoted by $nm_{i,j}^*$. (It looks like it is impossible to recover $n$ and $m$ from this, but the dimensions of simple modules are sparse and $n$ and $m$ are both dimensions of simple modules.)

If neither $n_i$ nor $m_j$ is self-dual, then we need to distinguish between the four modules
\[ n_i\otimes m_j,\quad n_i^*\otimes m_j^*,\quad n_i\otimes m_j^*,\quad n_i^*\otimes m_j.\]
We denote the first two by $nm_{i,j}$ and $nm_{i,j}^*$, as stated before. The other two are denoted by $\overline{nm}_{i,j}$ and $\overline{nm}_{i,j}^*$ respectively. Thus, for $\PSL_3(9)$, the module $\overline{18}_{12}$ denotes the tensor product $3_1\otimes 6_2^*$, and $\overline{18}_{21}$ denotes the tensor product $3_2\otimes 6_1^*$.

None of our subscripts exceeds $9$, and therefore we may also omit the comma in the subscript and write $18_{12}$, for example. We will usually do so for space reasons.

Although this is of course not enough to determine all modules uniquely for all groups of Lie type, it is enough to determine all modules in this paper uniquely, except for $\PSL_5(4)$, where we state exactly what we mean in that section.

(We could of course use highest-weight notation, which is unambiguous. However, highest-weight notation suffers from a difficulty, in that it does not easily show you the dimension of the module, which is important to us here. Since we also occasionally deal with groups that are not of Lie type, we prefer to keep a general module notation that is not tied to algebraic groups. When discussing positive-dimensional subgroups of algebraic groups $\mb G$ on the other hand, we will usually stick to highest-weight notation, with the exception of the minimal and Lie algebra modules $M(\mb G)$ and $L(\mb G)$.)

\medskip

We occasionally deviate from these guidelines when expedient, which is when the presence of a centre for the simply connected group means that not all modules are faithful. For example, in $\PSL_3(4)$, there are only two modules of dimension $9$, swapped by the graph automorphism, and we can label them simply as $9$ and $9^*$, rather than as $9_{12}$ and $9_{12}^*$. One other case is for $\PSL_3(q)$ and $\PSU_3(q)$ for $q=8,16$ as well, where we label the $27$-dimensional modules using different conventions depending on the group, and each time we describe our conventions.

Finally, for $\PSL_5(4)$ and $\PSU_5(4)$ we state in Chapter \ref{chap:rank4ine8} what our notation is, as there are four simple modules for $\PSL_5(2)$ and $\PSU_5(2)$ of dimension $40$, so our above notation cannot be used.

\newpage

\chapter{Rank 4 groups for \texorpdfstring{$E_8$}{E8}}
\label{chap:rank4ine8}
As rank-$4$ subgroups of $E_8$ are the easiest case to consider, we do these first. This chapter considers Lie type groups of untwisted rank $4$, so $\PSL_5$, $\PSU_5$, $\POmega_9$, $\PSp_8$, $\POmega_8^+$, $\POmega_8^-$, ${}^3\!D_4$, $F_4$ and ${}^2\!F_4$. In many of these cases, there are elements of order roughly $q^4$ (but note that we might need to divide by the order of the centre). Theorem \ref{thm:largeorderss} states that if $H$ has an element of order greater than $1312$ then $H$ is a blueprint for $L(E_8)$, and $7^4=2401$. Thus for fields of size $7$, $8$ and $9$, there is a good chance that we can find an element of order greater than $1312$.

In this chapter we prove that $H$ is one of the groups above over a field of size at most $9$, and $H\leq \mb G$, then either $H$ is strongly imprimitive or $H\cong \PSL_5(2)$. This latter case is significantly more difficult than the other cases dealt with in this chapter, indeed the whole article, and is postponed until Section \ref{sec:difl52}. In that section we will prove that $\PSL_5(2)$ is also strongly imprimitive.

Throughout this chapter, and the rest of the article, $u$ will denote a unipotent element of $H$ of order $p$ coming from the smallest class, i.e., with the largest centralizer. (This is the class most likely to belong to a generic class of $\mb G$.)

\section{\texorpdfstring{$\PSL_5$}{PSL 5}}

\begin{proposition}\label{prop:sl5ine8} Let $H\cong \PSL_5(q)$ for some $q\leq 9$.
\begin{enumerate}
\item If $q=2$ then either $H$ stabilizes a line on $L(E_8)$ or the composition factors of $L(E_8){\downarrow_H}$ are
\[ 40_1,40_1^*,40_2,40_2^*,24^2,10,10^*,(5,5^*)^2.\]
\item If $q=3,5,7,8,9$ then $H$ is a blueprint for $L(E_8)$.
\item If $q=4$ then either $H$ stabilizes a line on $L(E_8)$ or $H$ is a blueprint for $L(E_8)$.
\end{enumerate}
\end{proposition}
\begin{proof} $\boldsymbol{q=7,8,9}$: These contain an element of order $(q^5-1)/(q-1)$, which is not in $T(E_8)$ (see Definition \ref{defn:T(G)}). Hence $H$ is a blueprint for $L(E_8)$ by Theorem \ref{thm:largeorderss}.

\medskip

\noindent $\boldsymbol{q=5}$: We use the traces of elements of order up to $13$ to find conspicuous sets of composition factors and obtain two possibilities for $H$ on $L(E_8)$, namely
\[ 23,(10,10^*)^5,(5,5^*)^{10},1^{25},\qquad 45,45^*,40,40^*,23^2,10,10^*,5,5^*,1^2.\]
From Table \ref{t:modules5}, we see that $L(E_8){\downarrow_H}$ has non-positive pressure, and therefore in both cases $H$ stabilizes a line on $L(E_8)$ by Proposition \ref{prop:pressure}.

In fact, we can prove more. The only extension between the composition factors above is between $23$ and $1$, so most of these factors are summands. The action of the unipotent element $u$ on $(10\oplus 10^*)^{\oplus 5}\oplus (5\oplus 5^*)^{\oplus 10}\oplus 1^{\oplus 23}$ and $45\oplus 45^*\oplus 40\oplus 40^*\oplus 10\oplus 10^*\oplus 5\oplus 5^*$ has Jordan blocks $2^{50},1^{123}$ and $3^{12},2^{52},1^{60}$ respectively. This means that $u$ must lie in class $A_1$ and $2A_1$ respectively by \cite[Table 9]{lawther1995}, both of which are generic (see also Table \ref{t:unipe8p5} to see that $u$ must be generic). Thus $H$ is a blueprint for $L(E_8)$ by Lemma \ref{lem:genericmeansblueprint}.

\medskip

\noindent $\boldsymbol{q=3}$: We use the traces of elements of order up to $13$ to find conspicuous sets of composition factors, and obtain two conspicuous sets of composition factors for $H$ on $L(E_8)$, namely
\[ 24,(10,10^*)^5,(5,5^*)^{10},1^{24},\qquad 45_2,45_2^*,30,30^*,24^2,(10,10^*)^2,5,5^*.\]
This yields two cases.

\medskip

\noindent \textbf{Case 1}: None of the composition factors has an extension with any other, so $L(E_8){\downarrow_H}$ is semisimple. The element $u$ acts on $L(E_8)$ with Jordan blocks $3,2^{56},1^{133}$. Hence $u$ comes from the generic class $A_1$ (see \cite[Table 9]{lawther1995}), and therefore $H$ is a blueprint for $L(E_8)$ by Lemma \ref{lem:genericmeansblueprint}.

\medskip

\noindent \textbf{Case 2}: The modules $24$, $45_2$ and $45_2^*$ have no extensions with other composition factors so split off as summands. The only extensions arise from $\Ext_{kH}^1(30,5)$ and $\Ext_{kH}^1(30,10)$ (and their duals), both $1$-dimensional, and so we must have two other summands, one consisting of $30,10,10,5$, and one of $30^*,10^*,10^*,5^*$.

Thus $45_2\oplus 45_2^*\oplus 24^{\oplus 2}$ is a summand of $L(E_8){\downarrow_H}$, with $u$ acting on this module with blocks $3^8,2^{36},1^{42}$. The two remaining summands are dual to one another, so the remaining blocks in the action of $u$ on $L(E_8)$ must come in pairs. The only possibility in \cite[Table 9]{lawther1995} is the generic class $2A_1$, and hence $H$ is a blueprint for $L(E_8)$ by Lemma \ref{lem:genericmeansblueprint} again.

\medskip

\noindent $\boldsymbol{q=2}$: There are three conspicuous sets of composition factors for $L(E_8){\downarrow_H}$:
\[ 24,(10,10^*)^5,(5,5^*)^{10},1^{24},\quad 74,40_1,40_1^*,(10,10^*)^3,(5,5^*)^3,1^4,\] \[40_1,40_1^*,40_2,40_2^*,24^2,10,10^*,(5,5^*)^2.\]
(The second of these is an artefact of the fact that there is a non-rational class of elements of order $5$ in $\mb G$ with trace $2$ on $L(E_8)$, so `looks' rational. The second set of factors therefore cannot exist.)

From Table \ref{t:modules2}, $40_1$, $40_1^*$ and $74$ are the only composition factors with non-zero $1$-cohomology. Thus the first two cases have negative pressure (and the second does not exist), and hence $H$ stabilizes a line on $L(E_8)$ in the first case by Proposition \ref{prop:pressure}. Thus we are left with the third case, as claimed.

\medskip

\noindent $\boldsymbol{q=4}$: Up to field automorphism there are three sets of composition factors for $L(E_8){\downarrow_H}$ that are conspicuous for elements of order up to $11$:
\[ 24_1,(10_1,10_1^*)^5,(5_1,5_1^*)^{10},1^{24},\quad 40_{1,1},40_{1,1}^*,40_{1,2},40_{1,2}^*,24_1^2,10_1,10_1^*,(5_1,5_1^*)^2,\] \[50_{1,2},50_{1,2}^*,\overline{50}_{2,1},\overline{50}_{2,1}^*,24_1,24_2,\]
where $40_{1,i}^\pm$ are the four $2$-restricted, $40$-dimensional modules---$L(1010)$, $L(0101)$, $L(0011)$ and $L(1100)$---and up to field automorphism the four $50$-dimensional modules are $L(2100)$, $L(0012)$, $L(1020)$ and $L(0201)$, consistent with our notation from Chapter \ref{chap:labelling}.

\medskip

\noindent \textbf{Case 1}: None of $24_1$, $10_1^\pm$ and $5_1^\pm$ has non-zero $1$-cohomology, so $H$ has negative pressure, and thus stabilizes a line on $L(E_8)$ by Proposition \ref{prop:pressure}.

\medskip

\noindent \textbf{Cases 2 and 3}: Let $x\in H$ have order $315$. In these cases, $x^{45}$, $x^9$, $x^3$ and $x$, of orders $7$, $35$, $105$ and $315$ are all uniquely determined up to conjugacy by their eigenvalues on $L(E_8)$, by a computer check using the roots trick (see Chapter \ref{chap:techniques}).

We find a copy of $H$ inside the $A_4A_4$ maximal-rank subgroup acting with these composition factors (in the second case the actions of $H$ on each $M(A_4)$ factor are $L(0002)$ and $L(0001)$, and in the third case $H$ acts as $L(0001)$ on both factors), and it clearly lies inside a copy of $\SL_5(16)$. We therefore find, in $\SL_5(16)$, an element of order $4095=315\times 13$ that powers to our element in $H$ and with the same number of distinct eigenvalues---$171$ in the second case and $149$ in the third---as the element of order $315$ on $L(E_8)$. This proves that $x$, and hence $H$, is a blueprint for $L(E_8)$, as claimed.
\end{proof}

\section{\texorpdfstring{$\PSU_5$}{PSU 5}}

We continue with our definition of $u$ from the start of the chapter.

\begin{proposition}\label{prop:su5ine8} Let $H\cong \PSU_5(q)$ for some $q\leq 9$.
\begin{enumerate}
\item If $q=2$ then either $H$ stabilizes a line on $L(E_8)$ or $H$ stabilizes a $10$-space of $L(E_8)$ stabilized by a positive-dimensional subgroup of $\mb G$, and in addition $H$ is strongly imprimitive.
\item If $q=3,4,5,7,8,9$ then $H$ is a blueprint for $L(E_8)$.
\end{enumerate}
\end{proposition}
\begin{proof}
$\boldsymbol{q=7,8,9}$: Note that $\PSU_5(7)$ contains elements of order $2400$, the group $\PSU_5(8)$ contains elements of order $3641$, and $\PSU_5(9)$ contains elements of order $1181$. None of these lies in $T(E_8)$ (see Definition \ref{defn:T(G)}), so if $H\cong \PSU_5(q)$ for $q=7,8,9$ then $H$ is a blueprint for $L(E_8)$ by Theorem \ref{thm:largeorderss}.

\medskip

\noindent $\boldsymbol{q=5}$: There are two sets of composition factors for $L(E_8){\downarrow_H}$ conspicuous for elements of order at most $13$, which are
\[ 23,(10,10^*)^5,(5,5^*)^{10},1^{25},\qquad 45,45^*,40,40^*,23^2,10,10^*,5,5^*,1^2.\]
There is an extension between $23$ and $1$, but there are no other extensions between composition factors of $L(E_8){\downarrow_H}$, so in both cases $H$ stabilizes a line on $L(E_8)$, and moreover most of the module is semisimple. The action of $u$ on the summand consisting of all factors that are not $23$ or $1$ has blocks $2^{50},1^{100}$ and $3^{12},2^{52},1^{60}$ respectively. Thus $u$ comes from class $A_1$ in the first case, and $3A_1$ in the second, as we can see from \cite[Table 9]{lawther1995}. (See Table \ref{t:unipe8p5} to see that they are generic, at least.) These two classes are generic, so $H$ is a blueprint for $L(E_8)$ by Lemma \ref{lem:genericmeansblueprint}.

\medskip

\noindent $\boldsymbol{q=3}$: We use the traces of elements of order up to $10$ to find conspicuous sets of composition factors, and obtain two sets of factors:
\[ 24,(10,10^*)^5,(5,5^*)^{10},1^{24},\qquad 45_2,45_2^*,30,30^*,24^2,(10,10^*)^2,5,5^*.\]
\noindent \textbf{Case 1}: There are no extensions between the composition factors of $L(E_8){\downarrow_H}$, so $L(E_8){\downarrow_H}$ is semisimple; then $u$ acts on $L(E_8)$ with blocks $3,2^{56},1^{133}$, so comes from class $A_1$, which is generic. Thus $H$ is a blueprint for $L(E_8)$ by Lemma \ref{lem:genericmeansblueprint}.

\medskip

\noindent \textbf{Case 2}: The extensions are slightly different to the $\PSL_5(3)$ case: the $45$- and $24$-dimensional factors have no extensions with other composition factors, so break off as summands, as before. However, this time the only extensions are between $5$ and $30$, $10$ and $30$, and $30$ and $30^*$ (and their duals), so we do not necessarily obtain two separate summands.

The action of $u$ on $45_2\oplus 45_2^*\oplus 24^{\oplus 2}$ is $3^8,2^{36},1^{42}$. If $u$ comes from the generic class $2A_1$ then $H$ is a blueprint for $L(E_8)$ by Lemma \ref{lem:genericmeansblueprint}, so we assume not. Examining \cite[Table 9]{lawther1995} (or Table \ref{t:unipe8p3}), we see that $u$ comes from class $3A_1$, acting with Jordan blocks $3^{31},2^{50},1^{55}$. Thus $u$ acts on the remaining summand $V$ of $L(E_8){\downarrow_H}$---with composition factors $30,30^*,(10,10^*)^2,5,5^*$---with Jordan blocks $3^{23},2^{14},1^{13}$. Suppose that $30/30^*$ or its image under the graph automorphism is not a subquotient of $V$: then as in the case of $\PSL_5(3)$, $V$ breaks up as two dual summands, $V_1\oplus V_2$, but then we cannot have thirteen blocks of size $1$, a contradiction. Thus (without loss of generality) $30/30^*$ is a subquotient of $V$.

We now construct a module $W$ by first building the largest submodule of $P(30^*)$ with $30/30^*$ as a submodule and whose quotient has composition factors only from $5,5^*,10,10^*$, yielding the module $5^*,10^*,30/30^*$. We then take the dual of this, so $30/5,10,30^*$, and perform the same action, placing as many copies of $5,5^*,10,10^*$ on top of this module as we can, while remaining a submodule of $P(5\oplus 10\oplus 30^*)$. This yields the module
\[ W=5^*,10^*,30/5,10,30^*.\]
A pyx for $V$ is a sum of $W$ and a semisimple module consisting of copies of $10^\pm$ and $5^\pm$, on which $u$ acts with blocks $2^3,1^4$ and $2,1^3$ respectively. As $u$ acts on $W$ with Jordan blocks $3^8,2^{25},1^{16}$, it is not possible to produce the 23 blocks of size $3$ needed for the action of $u$ on $V$, which is a contradiction. Thus $u$ comes from the generic class $2A_1$, as needed.

\medskip

\noindent $\boldsymbol{q=2}$: There are two conspicuous sets of composition factors for $L(E_8){\downarrow_H}$:
\[ 24,(10,10^*)^5,(5,5^*)^{10},1^{24},\quad 40_1,40_1^*,40_2,40_2^*,24^2,10,10^*,(5,5^*)^2.\]
The first has pressure $-14$ and hence $H$ stabilizes a line on $L(E_8)$ by Proposition \ref{prop:pressure}, so we concentrate on the second case. Let $L$ be a copy of $3\times \SU_4(2)$ inside $H$: the trace of the central element $z\in L$ is $14$ on $L(E_8)$, and so $C_{\mb G}(z)$ has type $D_7T_1$ (see Table \ref{t:semie8}). Thus $L'\cong \SU_4(2)$ lies inside $D_7$. The composition factors of $D_7$ on $L(E_8)$ are two copies of $M(D_7)$, one of $L(D_7T_1)$ (which in characteristic $2$ has the form $1/90/1$, but this is not important), and one copy of each of the half-spin modules. The $1$-eigenspace of the action of $z$ on $L(E_8)$ must be the $1/90/1$, so we can compute which actions of $L'$ on $M(D_7)$ yield the correct composition factors on $L(D_7T_1)$ and on $L(E_8)$. The composition factors $14$ and $6,4,4^*$ for $M(D_7){\downarrow_{L'}}$ yield the correct composition factors for $L(D_7T_1){\downarrow_{L'}}$, namely
\[ 20,20^*,14^2,6^2,4,4^*,1^4.\]
However, the trace of elements of order $3$ in $L'$ on $L(E_8)$ does not match the case where $M(D_7){\downarrow_{L'}}=14$, so the composition factors are $6,4,4^*$. There is a module $4/6/4^*$, but the symmetric square of this module does not have a trivial submodule, whence $L'$ does not lie in $D_7$ with this module by Lemma \ref{lem:orthogfixpoint}. The only other self-dual module with those composition factors is the semisimple module $6\oplus 4\oplus 4^*$.

In particular, this means that $L'$ lies inside the $D_3A_3$-Levi subgroup of $D_7$, and this can be conjugated to lie inside the $A_7$-Levi subgroup of $E_8$. (This is so we can more easily understand the action of $L'$ on the half-spin modules; it is very possible to do this directly.) It's even easier to consider $L'$ inside $A_8$, which acts on $L(E_8)$ with summands $\Lambda^3(M(A_8))$, $\Lambda^3(M(A_8)^*)$ and $M(A_8)\otimes M(A_8)^*$ minus a trivial summand. The action of $L'$ on $M(A_8)$ is either $4\oplus 4\oplus 1$ or $4\oplus 4^*\oplus 1$, both yield the same action on $L(E_8)$. Embedding $L'$ into a copy $\mb X$ of $A_3$ acting as $L(100)\oplus L(100)\oplus L(000)$ or $L(100)\oplus L(001)\oplus L(000)$, in either case the action of $\mb X$ on $L(E_8)$ is
\begin{multline*} (000/101/000)^{\oplus 4}\oplus (010/200/010)\oplus (010/002/010)\\\oplus (100\oplus 001)^{\oplus 6}\oplus (010)^{\oplus 4}\oplus (110\oplus 011)^{\oplus 2}.\end{multline*}
(Here we have suppressed the `$L(-)$' to save space. The module $000/101/000$ is just $L(100)\otimes L(001)$, and the module $010/200/010$ is $L(100)^{\otimes 2}$.)

The reason for writing this like this is that we see that $L'$ acts in exactly the same way except that the $200$ and $002$ in the second socle layer are isomorphic to $001$ and $100$, so in particular any semisimple submodule of $L(E_8){\downarrow_{L'}}$ is stabilized by $\mb X$. If there is a submodule of $L(E_8){\downarrow_H}$ that restricts to $L'$ (equivalently $L$) semisimply then we have proved that $H$ is not Lie primitive, and since our condition is invariant under automorphisms, strongly imprimitive via Lemma \ref{lem:semilinearfield}.

The restrictions to $L'$ of $5$, $5^*$, $10$, $10^*$, $40_1$ and $40_1^*$ are all semisimple, and we in fact show that there is a submodule $10$ of $L(E_8){\downarrow_H}$. The module $10$ has extensions with $5$, $24$, $40_1$ and $40_1^*$, so we show that none of these extensions is allowed in $L(E_8){\downarrow_H}$. The restriction of $10/24$ to $L'$ is
\[ 4^*/1/6,14/1,4,4^*,\]
which obviously is not a subquotient of $L(E_8){\downarrow_{L'}}$. The restrictions of $10/40_1$ and $10/40_1^*$ to $L'$ are
\[ (4^*/6)\oplus (6/20)\oplus 14,\qquad (4^*/6)\oplus (6/14)\oplus 20^*,\]
and since $6/20$ and $6/14$ are not subquotients of $L(E_8){\downarrow_{L'}}$, these extensions cannot appear in $L(E_8){\downarrow_H}$ either. Finally the restriction of $10/5$ to $L'$ is
\[ 1\oplus 4^*\oplus (6/4),\]
which is a subquotient of $L(E_8){\downarrow_{L'}}$. However, the restriction of $40_2^\pm$ to $L'$ already contains a summand $6/4^\pm/6$, so there cannot be more than one subquotient $6/4$, and therefore $10/5$ is not a subquotient of $L(E_8){\downarrow_H}$ either. Since the only composition factors of $L(E_8){\downarrow_H}$ that $10$ has an extension with are $5$, $24$, $40_1$ and $40_1^*$, we have proved that $10$ is actually a summand of $L(E_8){\downarrow_H}$, and this is stabilized by $\mb X$, as needed.

\medskip

\noindent $\boldsymbol{q=4}$: The simple modules for $H$ of dimension at most $124$, or self-dual and of dimension at most $248$, are of dimension $1$, $24$ (two of these), $50$ (four) and $74$ (two). There is a single set of composition factors for $L(E_8){\downarrow_H}$ conspicuous for elements of order $3$ and $5$, and it is the set of all the simple modules of dimension $24$ or $50$, i.e.,
\[50_1,50_1^*,50_2,50_2^*,24_1,24_2,\]
where $50_i^\pm$ are the four simple modules for $\PSU_5(4)$ of that dimension.
(Such an embedding exists in the $A_4A_4$ maximal-rank subgroup, and we also saw it for $\PSL_5(4)$.)

We first show that $L(E_8){\downarrow_H}$ possesses a submodule of dimension $24$. Note that $24_1$ has a single extension with $50_1^\pm$, and none with $50_2^\pm$, for some labelling for the modules of dimension $50$. However, one cannot construct a module of the form $50_1/24_1/50_1^*$ or $50_1^*/24_1/50_1$, and so $H$ must always stabilize a $24$-space on $L(E_8)$. (To see this last step, if the socle is $50_1\oplus 50_2^\pm$ then it is clear, and if it is simple, remove the socle and top to leave a module that must have $24_1\oplus 24_2$ as a submodule.) Hence $24_1\oplus 24_2$ is in fact a summand of $L(E_8){\downarrow_H}$.

Let $x$ be an element of order $65$. By a computer check, the class of $x$ in $E_8$ is determined by its eigenvalues on $L(E_8)$. Next, we check the $6560$ possible elements of order $195=65\times 3$ that power to $x$ inside the normalizer of a torus, and find that there are four elements, $\hat x_1,\hat x_2,\hat x_3,\hat x_4$, of order $195$, such that $\hat x_i^3=x$ and with exactly $91$ distinct eigenvalues on $L(E_8)$. (This is the minimal number for a root of $x$.) Two of these stabilize all eigenspaces of $x$ on $24_1$, and two of them stabilize all eigenspaces of $x$ on $24_2$. Therefore we see that if either $24_i$ is a submodule of $L(E_8){\downarrow_H}$ (and we know that both are from above) then $H$ is contained in the stabilizer of these $24$-spaces, which contains elements of order $195$ and a copy of $\PSU_5(4)$. Thus $H$ is contained in a member of $\ms X$ by Lemma \ref{lem:allcasesstronglyimp}, and in a member of $\ms X^\sigma$ if $k$ contains $\F_4$, as otherwise the $24_i$ might not be defined.

We check now that $H$ is a blueprint for $L(E_8)$, completing the proof. Let $\mb X$ be a member of $\ms X$ containing $H$. The dimensions of the composition factors of $\mb X$ on $L(E_8)$ must be coarser than $50^4,24^2$, and this eliminates all parabolic subgroups, $E_7A_1$ and $D_8$ (as they have trivial factors on $L(E_8)$ in characteristic $2$), $A_8$ (has factors of dimension $84^2,80$), $A_2E_6$ (has a factor of dimension $8$) and $F_4G_2$ (has a factor of dimension $14$), leaving the $A_4A_4$ maximal-rank subgroup, which has factors of dimension $50^4,24^2$. As this $A_4A_4$ is semisimple, so is $L(E_8){\downarrow_H}$, and $\mb X$ and $H$ stabilize the same subspaces of $L(E_8)$, as needed.
\end{proof}

\section{\texorpdfstring{$\Omega_9$}{Omega 9}}

The following proposition is easy, assuming that we have shown that $\Omega_7(q)$ is a blueprint for $L(E_8)$, which we do in Proposition \ref{prop:omega7}.

\begin{proposition}\label{prop:omega9ine8} Let $H\cong \Omega_9(q)$ for some odd $q\leq 9$. Then $H$ is a blueprint for $L(E_8)$.
\end{proposition}
\begin{proof} As $L\cong \Omega_7(q)$ is contained in $H$, and $L$ is a blueprint for $L(E_8)$ by Proposition \ref{prop:omega7}, so is $H$. This completes the proof.
\end{proof}

\section{\texorpdfstring{$\PSp_8$}{PSp 8}}

We continue with our definition of $u$ from the start of the chapter.

\begin{proposition}\label{prop:sp8ine8} Let $H\cong \PSp_8(q)$ for some $q\leq 9$.
\begin{enumerate}
\item If $q=2$ then either $H$ stabilizes a line on $L(E_8)$, or $H$ is contained in an $A_7$-parabolic subgroup of $\mb G$ and is strongly imprimitive.
\item If $q=3,5,7,8,9$ then $H$ is a blueprint for $L(E_8)$.
\item If $q=4$ then $H$ stabilizes a line on $L(E_8)$.
\end{enumerate}
\end{proposition}
\begin{proof}
$\boldsymbol{q=7,8,9}$: The group $\PSp_8(7)$ contains elements of order $1201$, $\PSp_8(8)$ contains elements of order $4097$, and $\PSp_8(9)$ contains elements of order $3281$. None of these lies in $T(E_8)$ (see Definition \ref{defn:T(G)}), so $H$ is a blueprint for $L(E_8)$ by Theorem \ref{thm:largeorderss}.

\medskip

\noindent $\boldsymbol{q=5}$: The only set of composition factors for $L(E_8){\downarrow_H}$ that is conspicuous for elements of orders $2$ and $3$ is $42,36,27^6,1^8$, and there are no extensions between any of the composition factors, so $L(E_8){\downarrow_H}$ must be semisimple. Furthermore, the unipotent element $u$ acts on this module with Jordan blocks $3,2^{56},1^{133}$, so lies in the generic class $A_1$. Thus $H$ is a blueprint for $L(E_8)$ by Lemma \ref{lem:genericmeansblueprint}.

\medskip

\noindent $\boldsymbol{q=3}$: There is a unique set of composition factors for $L(E_8){\downarrow_H}$ that is conspicuous for elements of orders $2$ and $4$, which is
\[ 41,36,27^6,1^9.\]
This has at least seven trivial summands as the only extension between composition factors is between $41$ and $1$.

To show that $H$ is a blueprint for $L(E_8)$, note that $36\oplus 27^{\oplus 6}\oplus 1^{\oplus 7}$ must be a summand of $L(E_8){\downarrow_H}$ and an element $u$ from the smallest unipotent class of $H$ acts on this with Jordan blocks $3,2^{42},1^{118}$. There is no non-generic class with this many trivial summands from Table \ref{t:unipe8p3} so $u$ is generic, and indeed $u$ must lie in class $A_1$ by \cite[Table 9]{lawther1995}. (In fact, using other unipotent classes, one can check that the only embedding must be
\[ 36\oplus 27^{\oplus 6}\oplus (1/41/1)\oplus 1^{\oplus 7},\]
as the semisimple possibility has unipotent elements acting with blocks that do not occur in \cite[Table 9]{lawther1995}.)

Thus $H$ is a blueprint for $L(E_8)$ by Lemma \ref{lem:genericmeansblueprint}.

\medskip

\noindent $\boldsymbol{q=2}$: There are five conspicuous sets of composition factors for $L(E_8){\downarrow_H}$:
\[ 246,1^2,\quad 128,48,26^2,8^2,1^4,\quad 48^2,26^4,8^5,1^8,\quad 26^8,16,8,1^{16},\quad 26,16^8,8^8,1^{30}.\]

\noindent \textbf{Cases 1, 4 and 5}: The first case does not occur by Corollary \ref{cor:no246}. The modules $48$, $26$ and $8$ each have $1$-dimensional $1$-cohomology, with $16$ and $1$ having none (see Table \ref{t:modules2}), and so the fourth and fifth set of composition factors have pressures $-7$ and $-21$ respectively. Thus $H$ stabilizes a line on $L(E_8)$ in these cases by Proposition \ref{prop:pressure}.

\medskip

\noindent \textbf{Case 2}: By Table \ref{t:modules2}, we see that $L(E_8){\downarrow_H}$ has pressure $1$, so its socle must be either $26$ or $8$, if we assume that $H$ does not stabilize a line on $L(E_8)$. Since $128$ has no extensions with the other composition factors, it must split off as a summand. The other summand $W$ must have the form
\[ 8/1/26/1/48/1/26/1/8\quad\text{or}\quad 26/1/8/1/48/1/8/1/26.\]

Suppose first that the socle of $W$ is $26$, and we construct a pyx for $W$ in the obvious way: add all copies of $1$, then $8$, then $1$, then $48$ on top of $26$. Already at this stage the result of this is
\[ 1/1,8,48/1,8,48/26,\]
so a pyx cannot be built. Similarly, we reverse the roles of $8$ and $26$, and again attempt to build a pyx. This time we start with $8$, then add all copies of $1$, then $26$, then $1$, then $48$, then $1$, and finally $26$. This yields a module
\[ 1,26/1,26,48/1,48,48/1,26/1,26/8\]
so again a pyx cannot be built (as there are no $26$s in the $7$th socle layer). Thus $H$ stabilizes a line on $L(E_8)$, as needed.

\medskip

\noindent \textbf{Case 3}: Assume that $H$ does not stabilize a line on $L(E_8)$, and that $H$ does not lie in an $A_7$-parabolic subgroup. We will show that neither $26$ nor $48$ may lie in the socle of $L(E_8){\downarrow_H}$, so that it must be $8^{\oplus i}$ for some $i$.

Let $L$ denote a copy of $\SO_8^+(2)$ in $H$, and let $M$ denote a copy of $\Sym(10)$ in $H$, arranged so that $L\cap M\cong \Sym(9)$. Note that each of $1$, $8$, $26$ and $48$, the simple modules for $H$, restrict irreducibly to each of $L$, $M$ and $L\cap M$. By Proposition \ref{prop:omega8+} below, and \cite[Section 10]{craven2017}, we know that $L$ and $M$ both stabilize lines on $L(E_8)$, and hence lie inside members of $\ms X$. We claim that both $L$ and $M$ lie inside (different) $A_7$-parabolic subgroups of $\mb G$. (If they were the same $A_7$-parabolic subgroup, then $H=\gen{L,M}$ also lies in it.)

We prove that any copy of $L$, $M$ or $H$ with these composition factors that lies inside a member of $\ms X$ lies inside an $A_7$-parabolic subgroup, so assume that one of these does not. Notice that $L$, $M$ and $H$ all require at least rank $4$ for an overgroup, and do not embed in a subgroup of $A_6$ or $C_3$, so if those factors appear in $\mb X$ in $\ms X$, they may be removed. Thus we reduce to the case where, if $\mb X$ is a parabolic subgroup of $\mb G$, then it is of type either $E_7$ or $D_7$. If $H$, $L$ or $M$ lies inside a $D_7$-parabolic subgroup then it must act on $M(D_7)$ with factors $8,1^6$, but $M(D_7)$ appears twice in $L(E_8){\downarrow_{D_7}}$, so $H$ cannot embed in the $D_7$-parabolic subgroup with these factors.

If $H$, $L$ or $M$ embeds in the $E_7$-parabolic subgroup, then the image in the $E_7$-Levi subgroup must embed inside a proper positive-dimensional subgroup of $E_7$ by Theorem \ref{thm:classmaximalsubgroup} and the remarks after for $H$ and $L$, and \cite[Theorem 1]{litterickmemoir} for $M$. This proper positive-dimensional subgroup can be the $A_7$ maximal-rank subgroup, with our subgroup acting irreducibly on $M(A_7)$. But this $A_7$ acts on $L(E_8)$ with, up to duality, four copies of $\Lambda^2(M(A_7))$ and one of $L(A_7)$. Together they contribute six copies of $26$, which is too many. Our subgroup can embed in the $D_6$-parabolic, but then this lies in the $D_7$-parabolic above, so we embed in the $E_6$-parabolic subgroup. But this acts on $L(E_8)$ with (up to duality) six copies of $M(E_6)$, and no non-trivial composition factor appears that often.

Thus $H$, $L$ or $M$ embeds in a maximal-rank subgroup, and it suffices to rule out $A_8$ and $D_8$ as options. If $\mb X=A_8$ then $H$, $L$ or $M$ stabilize either a line or hyperplane on $M(A_8)$, so we lie in an $A_7$-parabolic, and if $\mb X=D_8$ then it embeds as the sum of two $8$s on $M(D_8)$ (else we lie in a $D_7$-parabolic), but this lies in an $A_7$-parabolic again, and we are done.

\medskip

Therefore $L$ and $M$ lie in different parabolics, so that $L\cap M$ lies in the intersection of two parabolics, which are described in \cite[Section 2.8]{carterfinite} using double cosets in the Weyl group. However, in this case $L\cap M$ must act irreducibly on $M(A_7)$, so $L\cap M$ projects irreducibly onto the Levi factor of the $A_7$-parabolic subgroup. Hence the subgroup containing $L\cap M$ must have $A_7$-Levi factor, and therefore the two parabolics involved are either the same (explicitly ruled out) or opposite (as the corresponding double coset must lie in the normalizer of $W(A_7)$ inside $W(E_8)$, and this is $2\times \Sym(8)$); in this latter case, $L\cap M$ lies in the $A_7$-Levi subgroup. In particular, $L\cap M$ has a known action on $L(E_8)$ via $A_7$, given by
\[ L(\lambda_1)\oplus L(\lambda_2) \oplus L(\lambda_3) \oplus L(\lambda_5) \oplus L(\lambda_6) \oplus L(\lambda_7) \oplus (L(\lambda_1)\otimes L(\lambda_7)),\]
and this restricts to $L\cap M$ as
\[ 8^{\oplus 5}\oplus 48^{\oplus 2}\oplus (1/26/1)^{\oplus 2}\oplus (1/26/1/1/26/1).\]
The $\Sp_8(2)$ subgroup of the $A_7$-Levi subgroup containing $L\cap M$ acts as
\[ 8^{\oplus 4}\oplus 48^{\oplus 2}\oplus (1/26/1)^{\oplus 2}\oplus (1/26/1/8/1/26/1),\]
and the $C_4$ subgroup of the $A_7$ acts similarly, as
\[ 8^{\oplus 4}\oplus 48^{\oplus 2}\oplus (1/26/1)^{\oplus 2}\oplus (1/26/1/8^\sigma/1/26/1),\]
where $8^\sigma$ denotes the image of the natural module under the Frobenius automorphism, so $L(2\lambda_1)$, where $M(C_4)=L(\lambda_1)$.

Therefore we see the following facts:
\begin{enumerate}
\item $26$ cannot be a composition factor of $\soc(L(E_8){\downarrow_H})$, as it does not appear in $\soc(L(E_8){\downarrow_{L\cap M}})$;
\item if $48$ lies in $\soc(L(E_8){\downarrow_H})$ then its stabilizer in $\mb G$ is positive dimensional, containing $C_4$, and this subgroup stabilizes all $48$-dimensional simple submodules of $L(E_8){\downarrow_H}$;
\item if $8^{\oplus 2}$ lies in $\soc(L(E_8){\downarrow_H})$ then one of the $8$s in this socle must also be stabilized by the $C_4$ subgroup, so its stabilizer is positive dimensional. (We cannot say anything about the orbit of $8$-dimensional submodules, so we only obtain that $H$ is contained in a member of $\ms X$. However, we have already shown that this means that $H$ is contained in an $A_7$-parabolic subgroup.)
\end{enumerate}
Thus in order to not satisfy any of these criteria, $\soc(L(E_8){\downarrow_H})$ is exactly $8$.

Recall that $M$ stabilizes a line on $L(E_8)$ by \cite[Section 10]{craven2017}, and therefore by Frobenius reciprocity there is a map from the permutation module $P_M=1_M{\uparrow^H}$ to $L(E_8){\downarrow_H}$. We thus consider $P_M$, which has dimension $13056$. The quotient by the $\{1,8,26,48\}$-residual of this is much smaller, and it and its dual have socle layers
\[ 1/48/1/26/1,26/8,8/1,1\quad \text{and}\quad 1/1,8/1,8/26/1,26/48/1.\]
We cannot have a copy of $26$ in the second socle layer of $L(E_8){\downarrow_H}$ without $H$ stabilizing a line on $L(E_8)$, by the structures above, so with this condition we obtain a unique quotient
\[ 1/48/1/26/1/8\]
of this module. Thus this module is a submodule of $L(E_8){\downarrow_H}$.

We therefore try to add modules on top of this to build a pyx, and look for a contradiction. Let $W$ denote this submodule.

We have $\Ext_{kH}^1(1,W)=0$ and $\Ext_{kH}^1(48,W)=0$, so $1$ and $48$ cannot lie in the socle of the quotient module $L(E_8){\downarrow_H}/W$. We also have $\Ext_{kH}^1(26,W)\cong k$, with the extension being
\[ 1/48/1/26/1,26/8,\]
which is not allowed, as we saw above. Thus the socle of the quotient module $L(E_8){\downarrow_H}/W$ must consist solely of $8$s. There is an extension with submodule $8$ and quotient $W$, yielding a module
\[ 1/8,48/1/26/1/8;\]
this still has no extensions with $1$, and $\Ext_{kH}^1(48,8)=0$ so it cannot have an extension with $48$ either. Hence there must be an extension with $26$. There is now a $2$-dimensional $\Ext_{kH}^1$-group, yielding the module
\[ 1,26/8,48/1/26/1,26/8;\]
the restriction of this to $L\cap M$ results in both of the new $26$s splitting off as summands, so since $L(E_8){\downarrow_{L\cap M}}$ has no $26$ submodule, neither of these can occur in $L(E_8){\downarrow_H}$.

Hence we cannot construct $L(E_8){\downarrow_H}$ with a single $8$ in the socle. This proves that $H$ is contained in a member of $\ms X$, and as we saw above, this implies that $H$ lies in an $A_7$-parabolic subgroup, as claimed.

To obtain the statement that $H$ is strongly imprimitive, we simply apply Lemma \ref{lem:nonGcr}, which states that if $H$ is non-$\mb G$-cr, then $H$ is strongly imprimitive. (In fact, since $\Out(H)=1$ we only need $\sigma$-stability to obtain strong imprimitivity.) If $H$ lies in the $A_7$-Levi subgroup though, $H$ stabilizes a line on $L(E_8)$ and is strongly imprimitive again, this time by Lemma \ref{lem:fix1space}.

\medskip

\noindent $\boldsymbol{q=4}$: The modules of dimension at most $248$ with non-zero $1$-cohomology have dimension $8$, $26$ and $246$. Letting $L=\Sp_8(2)\leq H$, we see that the composition factors of $8\otimes 8$ for $L$ are $26^2,8,1^4$, while $8\otimes 16=128$ and $8\otimes 26=48\oplus 160$. Hence the composition factors of $L(E_8){\downarrow_H}$ have dimension $1$, $8$, $16$, $26$, $48$ and $64$, with only the last of these restricting reducibly to $L$.

Note that, for $H$, although $8_i$ and $26_i$ still have $1$-dimensional $1$-cohomology, $48_i$ no longer does, and $64$ does not either. If $L(E_8){\downarrow_H}$ has positive pressure, then $L(E_8){\downarrow_L}$ must have factors $128,48,26^2,8^2,1^4$ or $48^2,26^4,8^5,1^8$, and so $L(E_8){\downarrow_H}$ has composition factors of dimension one of the following:
\[ 128,64,48,8,\quad 48^2,26^4,8^5,1^8,\quad 64,48^2,26^2,8^4,1^4,\quad 64^2,48^2,8^3.\]
If the factors are $128,64,48,8$ then $L(E_8){\downarrow_H}$ is semisimple, and the restriction to $L$ is determined uniquely. The action of $u$ on this module has Jordan blocks $2^{93},1^{62}$, which does not appear in \cite[Table 9]{lawther1995}. Thus this case cannot occur.

The only set of composition factors with either of the other two multisets of dimensions that are conspicuous for elements of order at most $9$ is
\[ 48_1^2,26_1^4,8_1^4,8_2,1^8,\]
up to field automorphism. The only extensions between composition factors in $L(E_8){\downarrow_H}$ arise from $\Ext_{kH}^1(1,8_i)$, $\Ext_{kH}^1(1,26_1)$ and $\Ext_{kH}^1(26_1,48_1)$, so $L(E_8){\downarrow_H}$ has pressure $1$. Suppose that $H$ does not stabilize a line on $L(E_8)$. Thus, quotienting out by any $48_1$ in the socle, we may assume that the socle is either $8_1$ or $26_1$.

The $\{48_1,26_1,8_1,8_2,1\}$-radical of $P(8_1)$ is
\[ 8_1/1,1/8_1,26_1/1,1,1/8_1,8_2,26_1/1,48_1/26_1/1/8_1.\]
Hence the socle cannot be $8_1$. Similarly, the $\{48_1,26_1,8_1,8_2,1\}$-radical of $P(26_1)$ is
\[ 8_1/1,1/8_1,26_1/1,1,1/8_1,8_2,26_1/1,48_1/26_1,\]
so the socle cannot be $26_1$ either. This completes the proof that $H$ stabilizes a line on $L(E_8)$.
\end{proof}

\section{\texorpdfstring{$\Omega_8^+$}{Omega 8+}}

We continue with our definition of $u$ from the start of the chapter.

\begin{proposition}\label{prop:omega8+}
Let $H\cong \POmega_8^+(q)$ for some $q\leq 9$.
\begin{enumerate}
\item If $q=2,4$ then $H$ stabilizes a line on $L(E_8)$.
\item if $q=3,5,7,8,9$ then $H$ is a blueprint for $L(E_8)$.
\end{enumerate}
\end{proposition}
\begin{proof} $\boldsymbol{q=3,5,7,9}$: In these cases $\Omega_7(q)$ is contained in $H$, and $\Omega_7(q)$ is a blueprint for $L(E_8)$ by Proposition \ref{prop:omega7} below, hence $H$ is.

\medskip

\noindent $\boldsymbol{q=8}$: In this case $H$ contains an element of order $8^4-1=4095$, not in $T(E_8)$, so $H$ is a blueprint for $L(E_8)$ by Theorem \ref{thm:largeorderss}.

\medskip

\noindent $\boldsymbol{q=2}$: There are (up to automorphism) five conspicuous sets of composition factors for $L(E_8){\downarrow_H}$:
\[26,8_1^8,8_2^8,8_3^8,1^{30},\quad 26^8,8_1,8_2,8_3,1^{16},\quad 48_1^2,26^4,8_1^5,1^8,\]
\[48_1,48_2,48_3,26^2,8_1^2,8_2^2,8_3^2,1^4,\quad 246,1^2.\]
The pressures of the five cases are $-28$, $0$, $2$, $3$ and $-1$ respectively.

\medskip

\noindent \textbf{Cases 1, 2 and 5}: The fifth case cannot occur by Corollary \ref{cor:no246}, and the first two cases lead to $H$ stabilizing a line on $L(E_8)$ by Proposition \ref{prop:pressure}.

\medskip

\noindent \textbf{Case 3}: Suppose that $L(E_8){\downarrow_H}$ has no trivial submodule, and let $W$ denote the $\{26,48_1\}$-heart of $L(E_8){\downarrow_H}$. The socle of $W$ must be $26$ or $48_1$ (as $26$ has $2$-dimensional $1$-cohomology, and so the socle cannot be $26\oplus 48_1$, and if it were $48_1^{\oplus 2}$ then this would be the entirety of $W$).

Suppose first that $\soc(W)=48_1$. The $\{1,8_1,26\}$-radical of $P(48_1)$ is
\[ 1/26/1,1,8_1/26/1,8_1/48_1,\]
which does not have eight trivial factors. Since this is a pyx for $W$---with the $48_1$ at the top removed---we obtain a contradiction. Thus $\soc(W)=26$.

The $\{1,8_1,26,48_1\}$-radical of $P(26)$ is
\[ 1/26/1,1,1,8_1/26,48_1/1,1/26,48_1/1,1,8_1/26\]
and must be a pyx for $W$. This pyx has eight trivial composition factors but has a trivial quotient. Any trivial quotient of $W$ becomes a trivial quotient of $L(E_8){\downarrow_H}$, so $W$ cannot have eight trivial composition factors in the third case, another contradiction. Thus $H$ stabilizes a line on $L(E_8)$.

\medskip

\noindent \textbf{Case 4}: Suppose that $H$ does not stabilize a line on $L(E_8)$. There are three conjugacy classes of subgroups of $H$ isomorphic with $\Alt(9)$, and for exactly one of them the restriction of $48_i$ to it remains irreducible. Let $L_i$ denote an $\Alt(9)$ subgroup such that $48_i$ remains irreducible on restriction to $L_i$. Since $L(E_8){\downarrow_{L_i}}$ has a trivial submodule by \cite[Proposition 9.2]{craven2017}, there is a non-trivial homomorphism from the permutation module $P_{L_i}$ of $H$ on $L_i$ to $L(E_8){\downarrow_H}$.

The module $P_{L_i}$ has dimension $960$, but the quotient by the $\{1,8_i,26,48_i\}$-residual has dimension $302$. The quotient modulo its second radical layer is the (unique up to isomorphism) uniserial module $1/48_i$, and so $1/48_i$ must be a subquotient of $L(E_8){\downarrow_H}$ as $H$ does not stabilize a line on $L(E_8)$. This shows that $48_i$ cannot lie in the socle of $L(E_8){\downarrow_H}$, since then it would be a summand.

Let $W$ denote the quotient of $L(E_8){\downarrow_H}$ by its $\{1,8_i,26\}$-radical. This has socle $48_1\oplus 48_2\oplus 48_3$, and so since each $L_i$ must stabilize a line on $L(E_8)$, we actually see that $W$ must have the submodule
\[ (1/48_1)\oplus (1/48_2)\oplus (1/48_3).\]
From here it is easy to obtain a contradiction, and therefore see that $H$ stabilizes a line on $L(E_8)$. The most obvious way is to recall that $L(E_8)$ is self-dual, and hence there is also a subquotient $48_i/1$, so we need at least six trivial composition factors in $L(E_8){\downarrow_H}$, which is a contradiction.

\medskip

\noindent $\boldsymbol{q=4}$: Suppose that $H$ does not stabilize a line on $L(E_8)$. By restricting to $L=\Omega_8^+(2)$, we see that there are no composition factors of dimension $160$ or $208$ (this is $8_i\otimes 26$, which restricts to $L$ as $48_i\oplus 160_i$), or $246$. Thus we may assume that the factors of $L(E_8){\downarrow_H}$ are of dimensions $1$, $8$, $26$, $48$ and $64$. As $H^1(H,26_i)$ has dimension $2$ and all other simple modules have zero $1$-cohomology, we need more than half as many $26$-dimensional composition factors as $1$s in $L(E_8){\downarrow_H}$. This means we need $64$s in $L(E_8){\downarrow_H}$, otherwise the dimensions of the factors of $L(E_8){\downarrow_H}$ match those of $L(E_8){\downarrow_L}$, and we are done.

For $L$, $8_1\otimes 8_1$ has factors $26^2,8_2,1^4$, and $8_1\otimes 8_2$ has factors $48_3,8_3^2$. Hence it is easy to see that we need to have no trivial factors or $26$-dimensional factors in $L(E_8){\downarrow_H}$, and so the dimensions of the factors are one of
\[ 64^3,48,8,\quad 64^2,48^2,8^3,\quad 64,48^3,8^5.\]
However, none of these is conspicuous for elements of order at most $5$, and so $H$ must always stabilize a line on $L(E_8)$.
\end{proof}

\section{\texorpdfstring{$\Omega_8^-$}{Omega 8-}}

We now turn to $\POmega_8^-(q)$, where we use the same embedding and element order as for $\POmega_8^+(q)$ to deal with the cases $q=3,5,7,8,9$. We continue with our definition of $u$ from the start of the chapter.

\begin{proposition}
Let $H\cong \POmega_8^-(q)$ for some $q\leq 9$.
\begin{enumerate}
\item If $q=2,4$ then $H$ stabilizes a line on $L(E_8)$.
\item if $q=3,5,7,8,9$ then $H$ is a blueprint for $L(E_8)$.
\end{enumerate}
\end{proposition}
\begin{proof} $\boldsymbol{q=3,5,7,9}$: In these cases $\Omega_7(q)$ is contained in $H$, and $\Omega_7(q)$ is a blueprint for $L(E_8)$ by Proposition \ref{prop:omega7}, hence $H$ is.

\medskip

\noindent $\boldsymbol{q=8}$: In this case $H$ contains an element of order $8^4-1=4095$, which is not in $T(E_8)$, so $H$ is a blueprint for $L(E_8)$ by Theorem \ref{thm:largeorderss}.

\medskip

\noindent $\boldsymbol{q=2}$: There are six conspicuous sets of composition factors up to automorphism, given by
\[ 26,8_1^8,8_2^8,8_3^8,1^{30},\quad 26^8,8_1,8_2,8_3,1^{16},\quad 48_1^2,26^4,8_1^5,1^8,\]
\[ 48_2^2,26^4,8_2^4,8_3,1^8,\quad 48_1,48_2,48_3,26^2,8_1^2,8_2^2,8_3^2,1^4,\quad 246,1^2.\]
As $26$ has $2$-dimensional $1$-cohomology, and $48_1$ has $1$-dimensional $1$-cohomology (see Table \ref{t:modules2}), in all but the third case $H$ must stabilize a line on $L(E_8)$ by pressure arguments (Proposition \ref{prop:pressure}). (The last case does not occur by Corollary \ref{cor:no246}, and while the fifth case has pressure $1$, it has a composition factor with $2$-dimensional $1$-cohomology.)

In the third case, suppose that $H$ does not stabilize a line on $L(E_8)$. The pressure is $2$ so the socle of $L(E_8){\downarrow_H}$, modulo the $\{8_1\}$-radical, is either $26$ or $48_1$. Let $W$ denote the quotient by the $\{8_1\}$-radical. The $\{1,8_1,26,48_1\}$-radical of $P(26)$ is
\[ 1,8_1/26,48_1/1,1/26,48_1/1,1,8_1/26,\]
and if $\soc(W)=26$ then this must be a pyx for $W$. But this module does not have eight trivial factors. Thus $26$ cannot be the socle. On the other hand, the $\{1,8_1,26\}$-radical of $P(48_1)$ is
\[ 1/26/1,1,8_1/26/1,8_1/48_1.\]
Again, if $\soc(W)=48_1$ then this must be a pyx for $W$, and does not have enough trivial factors either. Hence $48_1$ cannot be the socle. Thus $H$ stabilizes a line on $L(E_8)$, as needed.

\medskip

\noindent $\boldsymbol{q=4}$: This is similar to the proof for $\Omega_8^+(4)$, so assume that $H$ does not stabilize a line on $L(E_8)$. The simple modules of dimension at most $64$ that have non-zero $1$-cohomology are of dimension $26$, with $2$-dimensional $1$-cohomology. Thus, just as for $\Omega_8^+(4)$, we see that the possible dimensions of the composition factors of $L(E_8){\downarrow_H}$ are
\[ 64^3,48,8,\quad 64^2,48^2,8^3,\quad 64,48^3,8^5.\]
As with $\Omega_8^+(4)$, there are no sets of composition factors that are conspicuous for elements of order at most $9$, so $H$ must stabilize a line in this case as well, completing the proof.\end{proof}

\section{\texorpdfstring{${}^3\!D_4$}{3D4}}

We continue with our definition of $u$ from the start of the chapter.

\begin{proposition}
Let $H\cong {}^3\!D_4(q)$ for some $q\leq 9$.
\begin{enumerate}
\item If $q=2,4$ then $H$ stabilizes a line on $L(E_8)$.
\item if $q=3,5,7,8,9$ then $H$ is a blueprint for $L(E_8)$.
\end{enumerate}
\end{proposition}
\begin{proof} $\boldsymbol{q=7,8,9}$: The group ${}^3\!D_4(q)$ contains elements of order $\Phi_{12}(q)=q^4-q^2+1$, which for $q=7,8,9$ is equal to $2353$, $4033$ and $6481$, and of course none of these is in $T(E_8)$. Thus $H$ is a blueprint for $L(E_8)$ by Theorem \ref{thm:largeorderss}.

\medskip

\noindent $\boldsymbol{q=5}$: The group $H$ contains $G_2(5)$, which is a blueprint for $L(E_8)$ by Proposition \ref{prop:g2ine8} below.

\medskip

\noindent $\boldsymbol{q=3}$: The simple modules of dimension at most $248$ for $H$ are $1$, $8_i$, $28$, $35_i$, $56_i$, $104_i$ and $224_j$, where $1\leq i\leq 3$ and $1\leq j\leq 6$. There are six sets of composition factors that are conspicuous for elements of order at most $13$, four up to field automorphism, and these are
\[ 56_1,56_2,56_3,28^2,8_1,8_2,8_3,\qquad 56_1^2,35_1,28^3,8_1^2,1.\]
\[ 35_1,35_2,35_3,28^5,1^3,\qquad 28,8_1^8,8_2^8,8_3^8,1^{28}.\]
Each of these is semisimple because there are no extensions between composition factors. The unipotent element $u$ acts on $L(E_8)$ with blocks $3^{14},2^{64},1^{78}$ in the first three cases and $3,2^{56},1^{133}$ in the last case. This means that $u$ lies in class $A_1$ or $2A_1$ by \cite[Table 9]{lawther1995}, both of which are generic. Thus $H$ is a blueprint for $L(E_8)$ by Lemma \ref{lem:genericmeansblueprint}.

\medskip

\noindent $\boldsymbol{q=2}$: There are up to field automorphism five conspicuous sets of composition factors for $L(E_8){\downarrow_H}$:
\[ 246,1^2,\qquad 48_1,48_2,48_3,26^2,8_1^2,8_2^2,8_3^2,1^4,\qquad 48_1^2,26^4,8_1^4,8_2,1^8,\]
\[ 26^8,8_1,8_2,8_3,1^{16},\qquad 26,8_1^8,8_2^8,8_3^8,1^{30}.\]
Each of these has non-positive pressure (see Table \ref{t:modules2}), so $H$ stabilizes a line on $L(E_8)$ by Proposition \ref{prop:pressure}, as needed. (The first case cannot occur by Corollary \ref{cor:no246}.)

\medskip

\noindent $\boldsymbol{q=4}$: Assume that $H$ does not stabilize a line on $L(E_8)$. Let $L$ be the subgroup ${}^3D_4(2)$ of $H$. As with $\Omega_8^+(2)$, we use the restriction to $L$ to determine the factors of $L(E_8){\downarrow_H}$.

From Table \ref{t:modules48916} we have $H^1(H,M)=0$ if $\dim(M)=1,8,48$ and $H^1(H,26_i)$ is $2$-dimensional. Hence if the composition factors involved in $L(E_8){\downarrow_H}$ restrict irreducibly to $L$ then $L(E_8){\downarrow_H}$ has non-positive pressure, so has a trivial submodule.

For $L$, we have $8_1^{\otimes 2}=1/26/1,1,8_2/26/1$, and $8_1\otimes 8_2=8_3/48_3/8_3$, and the $64$-dimensional modules for $H$ have zero $1$-cohomology, so in order for $H$ to not stabilize a line on $L(E_8)$ there must be no trivial composition factors. Thus the dimensions of the composition factors of $L(E_8){\downarrow_H}$ are
\[ 64,48^3,8^5,\quad 64^2,48^2,8^3,\quad 64^3,48,8.\]
(This should look familiar from $\Omega_8^\pm(4)$.) None of the $185868$ such possibilities for $L(E_8){\downarrow_H}$ is conspicuous for elements of order $5$, so $H$ must stabilize a line on $L(E_8)$, as needed.\end{proof}

The current version of Magma (V2.24) at the time of writing will not directly compute inside ${}^3\!D_4(4)$ as it is given, so we explain how to coerce it to do so.

Assume that we implement ${}^3\!D_4(4)$ using \texttt{ChevalleyGroup("3D",4,4)}: this is generated by $x$ of order $63$ and $y$ of order $12$. If one tries to compute $\Ext^1$ between simple modules using \texttt{Ext} it will return an error, but instead implement ${}^3\!D_4(4)$ as follows:

\begin{verbatim}
G1:=ChevalleyGroup("3D",4,4);
M:=CompositionFactors(ExteriorSquare(GModule(G1)))[2];
G:=WriteOverSmallerField(sub<GL(26,64)|ActionGenerators(M)>,GF(4));
G:=sub<GL(26,4)|
    ActionGenerators(CompositionFactors(GModule(G))[1])>;
\end{verbatim}
For this copy of \texttt{G} one may now use \texttt{Ext} between modules.

However, one still cannot use the command \texttt{ConjugacyClasses} on \texttt{G}. Thus we take a specific element, namely \texttt{g:=(xy)\textasciicircum9}, where \texttt{x} and \texttt{y} are the generators of the group as given in Magma.

This element has order $5$: it has a $4$-dimensional $1$-eigenspace on each $8$-dimensional simple module, and a $16$-dimensional $1$-eigenspace on all $48$- and $64$-dimensional simple modules. From the dimensions above, this means that the $1$-eigenspace of $g$ on $L(E_8)$ has dimension $76$ or $84$. However, from the list of semisimple classes of $E_8$, we find that the dimension of the $1$-eigenspace is one of 48, 52, 54, 64, 68, 82, 92 and 134. This contradiction means that there is no such conspicuous set of composition factors.

\section{\texorpdfstring{$F_4$ and ${}^2\!F_4$}{F4 and 2F4}}

Let $H\cong F_4(q)$ for some $2\leq q\leq 9$. We also require $H\cong {}^2\!F_4(q)$ for $q=2,8$. We continue with our definition of $u$ from the start of the chapter.

\begin{proposition} Let $H\cong F_4(q)$ for some $q\leq 9$, or $H\cong {}^2F_4(q)'$ for some $q=2,8$.
\begin{enumerate}
\item If $q=2,4,5$ then $H$ stabilizes a line on $L(E_8)$.
\item If $q=3,7,8,9$ then $H$ is a blueprint for $L(E_8)$.
\end{enumerate}
\end{proposition}
$\boldsymbol{q=7,8,9}$: There exists an element of order $\Phi_8(q)=q^4+1$ in $H$, and so if $q\geq 7$ then this element has order greater than $1312$, so that $H$ a blueprint for $L(E_8)$ for $q=7,9$ by Theorem \ref{thm:largeorderss}.

\begin{proof} $\boldsymbol{q=2,4,8}$: The simple modules for $H$ of dimension at most $248$ have dimensions $1$, $26$ and $246$. If there is a $246$ then the factors are $246,1^2$, and $H$ cannot occur by Corollary \ref{cor:no246}.

Thus there are at least fourteen trivial composition factors and at most nine $26$-dimensional factors, since $248\equiv 14\bmod 26$. (Using a rational element of order $5$, which exists in $H$ and always has trace $+1$ on any simple module of dimension $26$, and trace $-2,3,23$ on $L(E_8)$, we see that $26^9,1^{14}$ is the only possibility, but we do not need this.)

If the $1$-cohomology of $26$-dimensional modules is at most $1$ then $H$ has negative pressure and we are done. For $H\cong {}^2\!F_4(2)'$ this is an easy computer calculation, for $F_4(q)$ for $q=2,4,8$ this is in \cite[Section 6]{jonesparshall1976}, and this leaves ${}^2\!F_4(8)$. Magma currently cannot compute this $1$-cohomology, so we need to show directly that $H$ is a blueprint for $L(E_8)$.

Up to field automorphism, the unique set of composition factors conspicuous for an element of order $7$ (a torus element, a generator of $H$ in the Magma implementation) is $26_2^8,26_1,1^{14}$. In the Magma implementation of $H$, the product $g$ of the two generators has order $91$, and has $47$ distinct eigenvalues on $L(E_8)$.

The traces of semisimple elements of $\mb G$ of order $13$ are known, and a conjugacy class is determined by its eigenvalues on $L(E_8)$. Fixing an element $g'$ in the group $X=(C_{2^{12}-1}^8)\rtimes W$ for $W$ the Weyl group of type $E_8$ with the same eigenvalues on $L(E_8)$ as $g$, we check the $5764801$ elements of order $91$ in $X$ whose seventh power is $g'$, find $120$ that have the same eigenvalues as $g$, and check that they are all conjugate in $X$, so let $g''$ denote one of them. Thus $g$ is determined up to conjugacy by its eigenvalues on $L(E_8)$. Furthermore, we can check the $6561$ roots of $g''$ of order $91\times 3=273$ and see that there are $26$ non-conjugate elements that have the same number of distinct eigenvalues on $L(E_8)$ as $g''$, so each stabilize the same subspaces of $L(E_8)$, so clearly $H$ will be a blueprint for $L(E_8)$. We make sure though, and for each one of these we can construct an element of order $1365=273\times 5$ in $X$ whose $15$th power is $g''$ and with $47$ eigenvalues on $L(E_8)$. Since $1365\not\in T(E_8)$, this proves that $H$ is a blueprint for $L(E_8)$ by Theorem \ref{thm:largeorderss}. (There are in fact many such elements powering to each of the $26$ roots.)

\medskip

\noindent $\boldsymbol{q=3}$: We have that $\dim(\Ext_{kH}^1(25,1))=1$ \cite[Section 6]{jonesparshall1976}, and $\Ext_{kH}^1(52,1)=0$ by \cite[Corollary 2]{volklein1989} (see also Table \ref{t:modules3}). If the composition factors of $L(E_8){\downarrow_H}$ are not $196,52$, then by simple counting we must have more trivial composition factors than $25$-dimensional ones. Thus $L(E_8){\downarrow_H}$ has negative pressure, so $H$ stabilizes a line on $L(E_8)$ by Proposition \ref{prop:pressure}.

To prove that $H$ also is a blueprint for $L(E_8)$, first note that $\Ext_{kH}^1(25,52)=0$ and so any $52$-dimensional composition factors must break off into their own summand. The largest module with composition factors $1$ and $25$ is $1/25/1$, and any indecomposable module with these factors is a subquotient of this. This is the action of $H$ on $M(E_6)$, so we know how the unipotent elements in $H$ act on it, namely according to Table \ref{t:unipotentF4}; in particular, the element $u\in H$ acts with no blocks of size $3$ on the summand of $L(E_8){\downarrow_H}$ whose composition factors have dimension $1$ and $25$. However, the only conspicuous set of composition factors is $52,25^7,1^{21}$, and so $u$ has exactly one block of size $3$ on $L(E_8)$ by \cite[Tables 4 and 5]{lawther1995}. This proves that $u$ lies in class $A_1$ of $\mb G$, which is generic, so $H$ is a blueprint for $L(E_8)$ by Lemma \ref{lem:genericmeansblueprint}.

For $196,52$, let $L$ denote the subgroup $\Spin_9(3)$ of $H$. This centralizes an involution of $H$, and it has trace $-8$ on $L(E_8)$, so $L\leq D_8$. The composition factors of $L$ on $L(E_8)$ have dimensions $16$, $36$, $84$ and $112$. This determines $L$ uniquely up to conjugacy inside $D_8$, as $L$ acts irreducibly on the $16$-dimensional module $M(D_8)$. Of course, $L$ is therefore contained in a $B_4$ subgroup $\bX$, which also has composition factors of the dimensions on $L(E_8)$. Thus $\gen{H,\bX}$ stabilizes the $52\oplus 196$ decomposition of $L(E_8)$. Hence $H$ is a blueprint for $L(E_8)$, as claimed.

\medskip

\noindent $\boldsymbol{q=5}$: The proof for $q=5$ is easier than for $q=3$: now the simple modules of dimension at most $248$ are $1$, $26$ and $52$, and all three of these have zero $1$-cohomology by \cite[Section 6]{jonesparshall1976} and \cite[Corollary 2]{volklein1989} (see also Table \ref{t:modules5}). Since $248\equiv 14\bmod 26$ there are at least fourteen trivial composition factors in $L(E_8){\downarrow_H}$, which must be trivial summands, as claimed.
\end{proof}

We do not prove that $H$ is a blueprint for $q$ even, since for the Tits group there is only one simple module of dimension $26$, yielding many diagonal subspaces that are not stabilized by an algebraic $F_4$ that acts with two different $26$-dimensional composition factors on $L(E_8)$.

To show that $F_4(5)$ is a blueprint for $L(E_8)$, one first shows that it must be contained in the $E_7$-parabolic, then inside the $E_7$-Levi subgroup; at this point, a summand of $L(E_8){\downarrow_H}$ must be $26^{\oplus 2}\oplus 1^{\oplus 14}$, and $u$ must act on $L(E_8)$ with at least twelve blocks of size $2$ and $42$ blocks of size $1$, proving that it lies in a generic class by consulting \cite[Table 9]{lawther1995}. Hence $H$ is a blueprint for $L(E_8)$ by Lemma \ref{lem:genericmeansblueprint}.

Alternatively, one takes an $\SL_2(5)$-Levi subgroup $L$ of $H$ containing $u$, and notes that the composition factors of this on $L(E_8)$ are $3,2^{56},1^{133}$. The only way to construct a block of size $5$ in the action of $u$ is from a module $1/3/1$ in the action of $L$, hence $u$ can have at most a single block of size $5$ on $L(E_8)$. Thus $u$ is generic by consulting \cite[Table 9]{lawther1995}, and again a blueprint by Lemma \ref{lem:genericmeansblueprint}.

\newpage

\chapter{Rank 3 groups for \texorpdfstring{$E_8$}{E8}}
\label{chap:rank3ine8}
This chapter considers the groups $\PSL_4(q)$, $\PSU_4(q)$, $\Omega_7(q)$ and $\PSp_6(q)$ for $2\leq q\leq 9$ (with $q$ odd for $\Omega_7(q)$).

We can prove that there are no Lie primitive copies of $H$ except for $\PSU_4(2)$. 
 This probably does stabilize a line on $L(E_8)$, but it could be difficult to prove using these methods. There is a single set of composition factors that remains unresolved, and it is warranted.

As with Chapter \ref{chap:rank4ine8}, $u$ will denote a unipotent element of $H$ of order $p$ coming from the smallest class, i.e., with the largest centralizer.

\section{\texorpdfstring{$\PSL_4$}{PSL 4}}

We cannot use large element orders any more, unlike rank $4$, so we need to work in every case. For $\PSL_4(q)$ we have a complete answer.

\begin{proposition}\label{prop:sl4ine8}Let $H\cong \PSL_4(q)$ for some $q\leq 9$.
\begin{enumerate}
\item If $q=2,4,8$ then $H$ stabilizes a line on $L(E_8)$.
\item If $q=3,7$ then $H$ is a blueprint for $L(E_8)$.
\item If $q=5,9$ then $H$ does not embed in $\mb G$.
\end{enumerate}
\end{proposition}
\begin{proof}
$\boldsymbol{q=2}$: The result was proved in \cite{craven2017}.

\medskip

\noindent $\boldsymbol{q=3}$: There are two conspicuous sets of composition factors for $L(E_8){\downarrow_H}$, which are
\[19,15^7,(10,10^*)^2,6^{12},1^{12},\qquad 45,45^*,44^2,15^2,(10,10^*)^2.\]
The former of these has pressure $-11$ (as only $19$ has non-zero $1$-cohomology), so $L(E_8){\downarrow_H}$ has at least ten trivial summands.

Furthermore, the seven $15$s must break off as summands as they have no extensions with the other composition factors. For $6$, we have that $\Ext_{kH}^1(6,19)$ has dimension $2$, and there are no other extensions involving it, so $L(E_8){\downarrow_H}$ has at least eight $6$s as summands as well; thus we have $15^{\oplus 7}\oplus 6^{\oplus 8}\oplus 1^{\oplus 10}$ as a summand. The unipotent element $u\in H$ acts on this module with blocks $3^7,2^{44},1^{54}$, so the unipotent class of $\mb G$ to which $u$ belongs can only be $2A_1$ or $3A_1$ from \cite[Table 9]{lawther1995}. The class $2A_1$ is generic, but $3A_1$ is not: since $3A_1$ acts on $L(E_8)$ with blocks $3^{31},2^{50},1^{55}$, and $u$ acts on each of $1$, $6$ and $10$ with a block of size $1$, the remaining composition factors $19,(10,10^*)^2,6^4,1^2$ must form an indecomposable module if $u$ is from class $3A_1$. All extensions between these factors involve $19$, so the only indecomposable module with these composition factors must be of the form
\[ 1,6,6,10,10^*/19/1,6,6,10,10^*.\]
While it is possible to make the module $19/1,6,6,10,10^*$, one cannot construct a module with three socle layers and with these first two socle layers. Thus $u$ comes from class $2A_1$ in $\mb G$, which is generic, and so $H$ is a blueprint for $L(E_8)$ by Lemma \ref{lem:genericmeansblueprint}.

For the second case, we restrict to $L\cong \PSp_4(3)$ inside $H$, and see that it has no trivial composition factors. In Proposition \ref{prop:sp4ine8} below we see that $L$ is a blueprint for $L(E_8)$, and so therefore is $H$.

\medskip

\noindent $\boldsymbol{q=9}$: There are many simple modules of dimension at most $248$. However, although there are 30245 possible sets of composition factors for the self-dual module $L(E_8){\downarrow_H}$, none of these is conspicuous for traces of elements of order at most $13$ in $\mb G$. Thus $H$ does not embed in $\mb G$.

\medskip

\noindent $\boldsymbol{q=5}$: There are $549$ possible sets of composition factors for a module of dimension $248$, none of which is conspicuous for elements of orders $2$, $3$ and $4$.

\medskip

\noindent $\boldsymbol{q=7}$: Since there are no modules with non-zero $1$-cohomology of dimension at most $248$, any trivial composition factors must be summands. Using the traces of elements of order up to $6$, there are two conspicuous sets of composition factors for $H$,
\[ 64^2,45,45^*,15^2,\qquad 20,15^7,(10,10^*)^2,6^{12},1^{11};\]
both of these are semisimple because there are no extensions between the composition factors. The element $u$ acts with Jordan blocks $4^8,3^{28},2^{48},1^{36}$, thus lies in the generic class $4A_1$ of $\mb G$, and $3^{14},2^{64},1^{78}$, thus lies in the generic class $2A_1$, respectively. Thus $H$ is a blueprint for $L(E_8)$ in both cases by Lemma \ref{lem:genericmeansblueprint}.

\medskip

\noindent $\boldsymbol{q=4,8}$: If $L=\PSL_4(2)$, then we saw in \cite[Proposition 4.2]{craven2017} that the possible composition factors of $L(E_8){\downarrow_L}$ are
\[ 14,6^{10},(4,4^*)^{16},1^{46},\quad 14^8,6^{17},(4,4^*)^2,1^{18},\] \[(20,20^*)^2,14^4,6^8,(4,4^*)^7,1^8,\quad 64^2,20,20^*,14^4,6^2,4,4^*,1^4.\]
The dimensions of the composition factors of the restrictions to $L$ of simple $kH$-modules are as follows.
\begin{center}
\begin{tabular}{ccc}
\hline Dimension of module for $H$ & Is self-dual? & Restriction of module to $L$
\\ \hline $16$ & No & $6^2,4$ or $14,1^2$
\\ $24$ & No & $20,4$
\\ $36$ & Yes & $14^2,6,1^2$
\\ $56$ & No & $20^2,14,1^2$
\\ $64$ ($q=8$) & No & $20^2,6^2,4^3$ or $20^2,14,4^2,1^2$
\\ $80$ & No & $64,6^2,4$ or $20,14^3,6^2,4,1^2$
\\ $84$ & Yes & $64,6^2,4^2$
\\ $96$ ($q=8$) & No & $20,14^4,6^2,4,1^4$ or $64,6^4,4^2$
\\ $120$ & No & $20^4,14^2,4^2,1^4$
\\ $196$ & Yes & $20^4,14^5,6^4,4^4,1^6$
\\\hline
\end{tabular}
\end{center}
From this it is easy to see, first that modules of dimension $196$, $120$, $96$ and $80$ cannot occur in $L(E_8){\downarrow_H}$ (as factors of dimensions $120$, $96$ and $80$ must occur in pairs), and second that in the first two sets of composition factors for $L(E_8){\downarrow_L}$, any such module for $H$ has negative pressure (from Table \ref{t:modules48916} the modules with non-zero $1$-cohomology are of dimensions $14$, $24$ and $84$), so $H$ stabilizes a line on $L(E_8)$ by Proposition \ref{prop:pressure}.

For $q=4$, with these restrictions on the dimensions of the composition factors of $L(E_8){\downarrow_H}$ one manages to fairly easily compute the conspicuous sets of composition factors for elements of order at most $21$, finding six up to field automorphism. These are
\[14_1,6_1^{10},(4_1,4_1^*)^{16},1^{46},\quad 14_1^8,6_1^{16},6_2,(4_2,4_2^*)^2,1^{18},\] \[(20_1,20_1^*)^2,14_1^4,6_1^8,(4_1,4_1^*)^6,4_2,4_2^*,1^8,\]
\[ 36,(16_{12},16_{12}^*)^2,(\bar{16}_{12},\bar{16}_{12}^*)^2,14_1,14_2,6_1^4,6_2^4,1^8,\quad 64_1^2,24_{21},24_{21}^*,14_1^4,6_2^2,1^4,\]
\[ 24_{12},24_{12}^*,24_{21},24_{21}^*,16_{12},16_{12}^*,\bar{16}_{12},\bar{16}_{12}^*,14_1,14_2,6_1^2,6_2^2,(4_1,4_1^*)^2,(4_2,4_2^*)^2,1^4,\]
Since the only simple modules above with non-zero $1$-cohomology are $14_i$ and $24_{i,j}^\pm$ (each $1$-dimensional) we see that the pressures of these modules are $-45$, $-10$, $-4$, $-6$, $2$ and $2$ respectively. Thus in the first four cases $H$ stabilizes a line on $L(E_8)$ by Proposition \ref{prop:pressure}. (This tallies with our earlier remarks as the first two and fourth sets of factors for $H$ come from the first two sets for $L$.)

For $q=8$ there are many more possible $248$-dimensional modules, but imposing the requirement that the restriction to $L$ must be one of the two remaining cases above brings the number down to something manageable. Up to field automorphism, there are three sets of composition factors that are conspicuous for elements of order at most $21$:
\[(20_1,20_1^*)^2,14_1^4,6_1^8,(4_1,4_1^*)^6,4_2,4_2^*,1^8,\quad  64_1^2,24_{21},24_{21}^*,14_1^4,6_2^2,1^4.\]
\[ 24_{12},24_{12}^*,24_{21},24_{21}^*,16_{12},16_{12}^*,\bar{16}_{12},\bar{16}_{12}^*,14_1,14_2,6_1^2,6_2^2,(4_1,4_1^*)^2,(4_2,4_2^*)^2,1^4,\]
(Notice that these are Cases 3, 5 and 6 from the $q=4$ case above.) As $24_{12}^\pm$ no longer has non-zero $1$-cohomology when $q=8$ (but $24_{21}^\pm$ still does, see Table \ref{t:modules48916}), we see that we need only consider the set of factors involving $64_1$ for $q=8$.

Thus we assume in the remaining cases that $H$ does not stabilize a line on $L(E_8)$. Let $W$ denote the $\{14_i,24_{i,j}^\pm\}$-heart of $L(E_8){\downarrow_H}$, with all simple summands removed. Note that, if $H$ does not stabilize a line on $L(E_8)$ then $W$ contains all trivial composition factors in $L(E_8){\downarrow_H}$. Note also that $\soc(W)$ consists solely of modules of dimensions $14$ and $24$.

\medskip

\noindent \textbf{$\boldsymbol{q=4}$, Case 5, and $\boldsymbol{q=8}$}: The same proof works in both cases. Since $64_1$ has no extensions with the other composition factors, the two copies break off into a separate summand. The $\{1,6_2,14_1,24_{21}^\pm\}$-radical of $P(14_1)$ is
\[ 14_1,14_1/1,1,6_2,6_2/14_1,24_{21},24_{21}^*/1,6_2/14_1,\]
so $14_1$ cannot be the socle of the module $W$ defined above. Suppose next that $14_1\oplus 14_1$ is the socle of $W$, and attempt to build a pyx for $W$. Since there are only two copies of $6_2$ in $L(E_8){\downarrow_H}$, we cannot have both copies of $6_2$ in the second socle layer of $W$ without them both being in the second radical layer (as $W$ is self-dual), whence we remove a quotient $(14_1/6_2)^{\oplus 2}$ from $W$. The $\{1,24_{21}^\pm\}$-radical of $P(14_1)$ is simply $1/14_1$, so since $W$ should have four trivial factors, we obtain a contradiction. This also shows that the second socle layer cannot have no copies of $6_2$ either, for then we cannot build a pyx for $W$ at all.

Thus we can start with the module $(6_2/14_1)\oplus 14_1$ and then add as many copies of $1$, $24_{21}$ and $24_{21}^*$ as possible. Doing so yields the module
\[ (1,1/24_{21},24_{21}^*/1,6_2/14_1)\oplus (1/14_1),\]
but this has three trivial quotients, so one of the trivial factors cannot occur for pressure reasons, using Proposition \ref{prop:pressure}. Thus we cannot build a pyx in this case either.

Thus, up to graph automorphism we may assume that the socle of $W$ is either $24_{21}$ or $24_{21}\oplus 14_1$. For $q=4,8$ the $\{1,6_2,14_1\}$-radical of $P(24_{21})$ is simply $14_1/1,6_2/24_{21}$, so the socle cannot be $24_{21}$, as we cannot build a pyx with four trivial factors. If the socle is $24_{21}\oplus 14_1$ then the $\{1,6_2,14_1\}$-radical of $P(24_{21})\oplus P(14_1)$ is
\[ (14_1/1,6_2/24_{21})\oplus (14_1/1,6_2/14_1),\]
and again we fail to build a pyx for $W$. Thus $W$ cannot have four trivial composition factors, and hence $H$ stabilizes a line on $L(E_8)$.
\medskip

\noindent \textbf{$\boldsymbol{q=4}$, Case 6}: 
We first show that $L(E_8){\downarrow_H}$ possesses a submodule of dimension $1$, $4$ or $6$. Thus suppose that none of $1$, $4_i^\pm$ and $6_i$ lies in the socle of $L(E_8){\downarrow_H}$. Since the $14_i$ appear with multiplicity $1$, removing from the socle all composition factors with zero $1$-cohomology, and all simple summands, we obtain a module $W$, which has up to field and graph automorphism one of $24_{12}$, $24_{12}\oplus 24_{21}$ and $24_{12}\oplus 24_{21}^*$ as socle and four trivial composition factors.

If the socle of $W$ is $24_{12}$ then we consider the $\{1,4^\pm_i,6_i,14_i,16^\pm_{12},\bar{16}_{12}^\pm,24_{21}^\pm\}$-radical of $P(24_{12})$, which is of the form
\[24_{21}^*/6_2/4_1^*/1,6_1,6_2/4_2,14_1,14_2,24_{21}^*/1,6_1,\bar{16}_{12}^*/24_{12}:\]
this has only two trivial composition factors. Thus there must be two $24$-dimensional factors in the socle.

The $\{1,4^\pm_i,6_i,14_i,16_{12}^\pm,\bar{16}_{12}^\pm\}$-radical of $P(24_{12})$ is
\[ 1,6_1/4_2,14_1,14_2/1,6_1,\bar{16}_{12}^*/24_{12}.\]
On this we can place $24_{21}^*$ or $24_{12}^*$.

Notice that only $4_2$ appears here out of the $4$-dimensional modules. Thus $W$ cannot have both $4_2$ and $4_2^*$ in it (as $W$ is a submodule of the sum of this and (up to graph automorphism) the image of this under the field automorphism), so must have neither as $W$ is self-dual. Removing these, we take the $\{1,6_i,14_i,16_{12}^\pm,\bar{16}_{12}^\pm\}$-radical instead of the $\{1,4^\pm_i,6_i,14_i,16_{12}^\pm,\bar{16}_{12}^\pm\}$-radical, and this is
\[ 1/14_1,14_2/1,6_1,\bar{16}_{12}^*/24_{12}.\]

We now let $V$ denote the $\{4^\pm_i,6_i,16^\pm_{i,j}\}$-radical of $L(E_8){\downarrow_H}$. As $W$ possesses at most two $6$-dimensional composition factors, $V$ must have at least one composition factor of dimension $6$. The $\{4^\pm_i,6_i,16^\pm_{i,j}\}$-radicals of $P(16_{12})$ and $P(\bar{16}_{12})$ are simply
\[ 4_1^*,4_2^*/16_{12},\qquad \bar{16}_{12},\]
and so there must be a $4$- or $6$-dimensional submodule of $V$, hence of $L(E_8){\downarrow_H}$, as we originally claimed.

\medskip

We now check that there exist elements of order $255=85\times 3$ that cube to a given element $x$ of order $85$ in $\PSL_4(4)$ and that stabilize the eigenspaces comprising the composition factors $4_i^\pm$ and $6_i$. Using the roots trick, we find 80 elements of order $255$ powering to $x$ and stabilizing the eigenspaces comprising a given $4$-dimensional module, and 2186 elements that stabilize the eigenspaces for each $6_i$ (728 of which stabilize the eigenspaces for both $6_i$). Indeed, one obtains more, and finds an element of order $1785=7\times 255$ that powers to an element of order $255$, has the same number of distinct eigenvalues on $L(E_8)$ as its seventh power, and stabilizes the eigenspaces comprising both $4_1$ and $6_1$. Thus in fact the stabilizer of a $1$-, $4$- or $6$-space is positive dimensional, and therefore $H$ is certainly contained in an element of $\ms X$.

\medskip

In fact, we wish to show that $H$ actually stabilizes a line on $L(E_8)$, so assume that this is not the case. Since we have shown that $H$ lies in a member of $\ms X$, but not necessarily $\ms X^\sigma$ (and shown nothing for $N_{\mb G}(H)$), this extra step will allow us to extend our results to $N_{\mb G}(H)$ and $N_{\bar G}(H)$.

Let $\mb X\in \ms X$ and suppose that $H\leq \mb X$. By Proposition \ref{prop:inconnected} we may assume that $\mb X$ is connected. If $\mb X$ is a maximal parabolic subgroup, then $\mb X$ cannot be an $E_7$-parabolic, either because there are not enough composition factors appearing with multiplicity $2$ in $L(E_8){\downarrow_H}$ to appear in the two copies of the $56$-dimensional factor of $L(E_8){\downarrow_{\mb X}}$, or because $\mb X$ stabilizes a line on $L(E_8)$ anyway, and that is what we are trying to prove. This also deals with the $A_1E_6$, $A_1A_6$, $A_2D_5$ and $A_1A_2A_4$-parabolics, since the projection of $H$ onto the Levi subgroup must lie inside a factor contained in the $E_7$-Levi subgroup (up to conjugation). Similarly, for $A_3A_4$, $H$ must lie inside the $A_3A_3$-parabolic, hence inside the $A_7$-parabolic subgroup.

If $\mb X$ is the $A_7$-parabolic subgroup, then it is an easy check that the embedding of $H$ into the $A_7$-Levi subgroup must have factors $4_1\oplus 4_2$ or $4_1\oplus 4_2^*$ on the natural module. There is a filtration of $L(E_8){\downarrow_{\mb X}}$ with the layers being as follows:
\[L(\lambda_1),L(\lambda_6),L(\lambda_3),L(\lambda_1)\otimes L(\lambda_7),L(\lambda_5),L(\lambda_2),L(\lambda_7).\]
The layers $L(\lambda_1),L(\lambda_6)$ have no trivial factors and no factors with non-zero $1$-cohomology, and the same for their duals at the top of the module, so whether $H$ stabilizes a line on $L(E_8)$ depends on the middle three layers. The middle layer restricts to $H$ with structure
\[ (1/14_1/1)\oplus (1/14_2/1)\oplus 16\oplus 16^*,\]
where $16$ is one of the $16$-dimensional modules (depending on the action of $H$ on $M(A_7)$). The layer above and below are semisimple, with $24$-dimensional and $4$-dimensional factors. Thus in our proof above, when we construct the module $W$, the $6_1$ cannot lie in $W$. This is enough to remove the possibility that those three layers can be combined to form a module with no trivial submodule or quotient, as there is no module of the form
\[ 24/1/14/1,16/24.\]
Thus there is a trivial submodule (and quotient) of those middle three layers, and hence of the whole module $L(E_8){\downarrow_H}$.

The last parabolic is the $D_7$-parabolic, so let $\mb X$ be a copy of this in $E_8$. Here the action of $\mb X$ on $L(E_8)$ has factors the two half-spin modules $L(\lambda_6)$ and $L(\lambda_7)$, a trivial module, two natural modules $L(\lambda_1)$, and the exterior square $W(\lambda_2)$ of the natural (this is $L(\lambda_2)$ and $L(0)$).

The only possibility for the action of $H$ on $L(\lambda_1)$ with the right composition factors even on the sum $L(\lambda_1)^{\oplus 2}\oplus L(\lambda_2)$ is (up to field automorphism) $4_1\oplus 4_1^*\oplus 6_2$. (This is the only self-dual module with these factors.) This therefore lies in a parabolic subgroup of $D_7$, since $4_1$ must be a totally isotropic subspace of $M(D_7)$ and stabilizers of totally isotropic subspaces are parabolic subgroups of $D_7$. Hence $H$ lies inside a different parabolic subgroup of $\mb G$.

If $\mb X$ is a (connected) maximal-rank subgroup, then $\mb X=A_8$ places $H$ inside the $A_7$-parabolic, $A_4A_4$ places $H$ inside the $A_3A_3$-parabolic, and $A_1E_7$ and $A_2E_6$ place $H$ inside the $E_7$-parabolic, whence $H$ cannot embed in $\mb X$. For $G_2F_4$, $H$ must lie in the $F_4$ factor, hence inside $E_7$, again a contradiction. The last case is $\mb X=D_8$, which in characteristic $2$ stabilizes a line on $L(E_8)$, so $H$ stabilizes a line on $L(E_8)$.

Thus $H$ stabilizes a line on $L(E_8)$ in all cases.
\end{proof}

\section{\texorpdfstring{$\PSU_4$}{PSU 4}}

For $H\cong \PSU_4(q)$, unless $q=2$ we are able to prove that $H$ is always strongly imprimitive in $\mb G$. However, for this last case of $q=2$, there is one set of composition factors for $L(E_8){\downarrow_H}$, which is warranted, that we cannot deal with here. We continue with our definition of $u$ from the start of the chapter.

\begin{proposition}\label{prop:su4ine8} Let $H\cong \PSU_4(q)$ for some $q\leq 9$.
\begin{enumerate}
\item If $q=2$ then either $H$ stabilizes a line on $L(E_8)$ or the composition factors of $L(E_8){\downarrow_H}$ are
\[ (20,20^*)^2,14^4,6^8,(4,4^*)^7,1^8.\]
\item If $q=3,7$ then $H$ does not embed in $\mb G$.
\item If $q=4,8$ then $H$ stabilizes a line on $L(E_8)$.
\item If $q=5,9$ then $H$ is a blueprint for $L(E_8)$.
\end{enumerate}
\end{proposition}
\begin{proof} $\boldsymbol{q=3,7}$: There are no sets of composition factors for $L(E_8){\downarrow_H}$ that are conspicuous for elements of order at most $7$, so $H$ does not embed in $\mb G$.

\medskip

\noindent $\boldsymbol{q=5}$: There are two sets of composition factors for $L(E_8){\downarrow_H}$ that are conspicuous for elements of order at most $8$:
\[ 20,15^7,(10,10^*)^2,6^{12},1^{11},\quad 58^2,45,45^*,15^2,6^2.\]
In the first case, there are no extensions between any of the composition factors, so $L(E_8){\downarrow_H}$ is semisimple, and $u$ acts on $L(E_8)$ with blocks $4^8,3^{28},2^{48},1^{36}$. Thus $u$ lies in the generic class $4A_1$ of $\mb G$, so $H$ is a blueprint for $L(E_8)$ by Lemma \ref{lem:genericmeansblueprint}.

In the second case, if $L(E_8){\downarrow_H}$ is semisimple then $u$ acts with blocks $3^{14},2^{64},1^{78}$, so lies in the generic class $2A_1$, and again we apply Lemma \ref{lem:genericmeansblueprint}. Furthermore, $\Ext_{kH}^1(6,58)\cong \Ext_{kH}^1(58,6)\cong k$, and all other extensions between composition factors are split, so the $15$s and $45$s must split off. As $u$ acts on $45\oplus 45^*\oplus 15^{\oplus 2}$ with blocks $4^4,3^{12},2^{24},1^{20}$, and there are no non-generic unipotent classes with those blocks (see Table \ref{t:unipe8p5}), we see that $u$ must belong to a generic class. Hence $H$ is a blueprint for $L(E_8)$, as needed.

\medskip

\noindent $\boldsymbol{q=9}$: There are five sets of composition factors that are conspicuous for elements of order up to $16$, three up to field automorphism. These are
\[ 19_1,15_1^7,(10_1,10_1^*)^2,6_1^{12},1^{12},\quad 36,(16_1,16_1^*)^2,(16_2,16_2^*)^2,15_1,15_2,6_1^4,6_2^4,1^6,\] \[45_1,45_1^*,44_1^2,(10_1,10_1^*)^2,15_1^2.\]
Using a computer program and the roots trick (see Chapter \ref{chap:techniques}), we check that in each case an element $x$ of order $328$ in $H$ is determined by its eigenvalues up to conjugacy in $E_8$, at least up to taking powers, and then that there is an element of order $1640=324\times 5$ with the same number of distinct eigenvalues on $L(E_8)$ as $x$. Since $1640\not\in T(E_8)$ (see Definition \ref{defn:T(G)}), this shows that $H$ is a blueprint for $L(E_8)$ by Theorem \ref{thm:largeorderss}, as claimed.

\medskip

\noindent $\boldsymbol{q=2}$: There are seventeen conspicuous sets of composition factors for $L(E_8){\downarrow_H}$, eleven of which have positive pressure. These are
\[ 64^2,20,20^*,14^4,6^2,4,4^*,1^4,\quad
64,(20,20^*)^2,14^4,6^5,(4,4^*)^2,1^2,\]
\[ 64,(20,20^*)^2,14^4,6^6,4,4^*,1^4,\quad
64,(20,20^*)^2,14^3,6^4,(4,4^*)^4,1^6,\]
\[ 64,20,20^*,14^5,6^6,(4,4^*)^4,1^6,\quad 64,20,20^*,14^4,6^5,(4,4^*)^6,1^{10},\]
\[ (20,20^*)^3,14^4,6^9,(4,4^*)^2,1^2,\quad (20,20^*)^3,14^3,6^8,(4,4^*)^4,1^6,\]
\[ (20,20^*)^2,14^5,6^{10},(4,4^*)^4,1^6,\quad (20,20^*)^2,14^4,6^8,(4,4^*)^7,1^8,\]
\[ (20,20^*)^2,14^4,6^9,(4,4^*)^6,1^{10}.\]
We begin by using Proposition \ref{prop:compfactorsalt6}, and restricting these sets of composition factors to the copy of $L\cong \Alt(6)$ in $H$. The restrictions are as follows:
\begin{center}
\begin{tabular}{cccc}
\hline Case & Restriction & Case & Restriction
\\ \hline $1$ & $16^4,4_1^{22},4_2^{16},1^{32}$ & $7$ & $16^6,4_1^{17},4_2^{14},1^{28}$
\\$2$ & $16^5,4_1^{19},4_2^{16},1^{28}$ & $8$ & $16^6,4_1^{14},4_2^{17},1^{28}$
\\$3$ & $16^5,4_1^{20},4_2^{14},1^{32}$ & $9$ & $16^4,4_1^{20},4_2^{17},1^{36}$
\\$4$ & $16^5,4_1^{16},4_2^{19},1^{28}$ & $10$ & $16^4,4_1^{16},4_2^{22},1^{32}$
\\$5$ & $16^3,4_1^{22},4_2^{19},1^{36}$ & $11$ & $16^4,4_1^{17},4_2^{20},1^{36}$
\\$6$ & $16^3,4_1^{19},4_2^{22},1^{36}$ &&
\\ \hline
\end{tabular}
\end{center}
From this we can see that only Cases 1, 3 and 10 can exist.

\medskip

\noindent \textbf{Cases 1 and 3}: Assume that $H$ does not stabilize a line on $L(E_8)$, and let $W$ denote the $\{4^\pm,14\}$-heart of $L(E_8){\downarrow_H}$. Since $4,4^*,14$ are the modules with non-zero $1$-cohomology, this module also has no trivial submodule or quotient, and is self-dual. To attack the first and third cases, which have pressure $2$, we construct the following submodule $W'$ of $P(14)$:
\begin{enumerate}
\item Take the $\{1,6,14,20^\pm\}$-radical of $P(14)$;
\item On this add as many copies of $4$ as possible;
\item Again, add as many copies of $1$, $6$, $14$, $20$ and $20^*$ as possible;
\item On this add as many copies of $4^*$ as possible;
\item Again, add as many copies of $1$, $6$, $14$, $20$ and $20^*$ as possible.
\end{enumerate}
Up to application of the graph automorphism (which swaps $4$ and $4^*$), if $\soc(W)=14$ then $W$ is a submodule of $W'$ in the first and third cases, so $W'$ is a pyx for $W$. This module has structure
\[ 14/1,6/14,14,14,20,20^*/1,1,6,6,6/4,4^*,14,20,20^*/1,6/14,\]
so has four trivial composition factors, but also has a trivial quotient. Thus $\soc(W)\neq 14$, so must be (up to graph automorphism) $4\oplus 14$ or $14^{\oplus 2}$. If it is $4\oplus 14$, then we take the $\{1,6,14,20^\pm\}$-radical of $P(4)\oplus P(14)$ to obtain
\[ (14/6/14,20,20^*/6,6/14,20,20^*/1,6/14)\oplus(1,6/14,20^*/1,6/4),\]
which has only three trivial composition factors. Thus the socle must be $14^{\oplus 2}$: the $\{1,4^\pm,6,20^\pm\}$-radical of $P(14)$, then with as many $14$s placed on top as possible, then with all quotients other than $14$ removed,
is
\[ 14,14/1,6/4,4^*,14,20,20^*/1,6/14.\]
Two copies of this would be a pyx for $W$, but all four trivial composition factors in this pyx would need to be present in $W$. However, we now take the $\{1\}'$-residual of this pyx, which must lie inside $W$, and this is
\[ (1/4,4^*/1/14)^{\oplus 2};\]
this has too many composition factors of dimension $4$, which is a contradiction. Thus $H$ stabilizes a line on $L(E_8)$ in the first and third cases.

\medskip

\noindent \textbf{Case 10}: This is the case remaining in the proposition. 

\medskip

\noindent $\boldsymbol{q=4}$: There are, up to field automorphism, six sets of composition factors for $L(E_8){\downarrow_H}$ that are conspicuous for elements of order at most $17$. These are
\[ 14_1,6_1^{10},(4_1,4_1^*)^{16},1^{46},\quad 14_1^8,6_1^{16},6_2,(4_2,4_2^*)^2,1^{18},\] \[(20_1,20_1^*)^2,14_1^4,6_1^8,(4_1,4_1^*)^6,4_2,4_2^*,1^8,\]
\[ 36,(16_{12},16_{12}^*)^2,(\bar{16}_{12},\bar{16}_{12}^*)^2,14_1,14_2,6_1^4,6_2^4,1^8,\quad  64_1^2,24_{21},24_{21}^*,14_1^4,6_2^2,1^4,\]
\[24_{12},24_{12}^*,24_{21},24_{21}^*,16_{12},16_{12}^*,\bar{16}_{12},\bar{16}_{12}^*,14_1,14_2,6_1^2,6_2^2,(4_1,4_1^*)^2,(4_2,4_2^*)^2,1^4.\]

The only simple modules from these with non-zero $1$-cohomology are $14_i$ and $24_{i,j}^\pm$. Thus the fifth and sixth sets of factors have pressure $2$, and the others have negative pressure, so $H$ stabilizes a line in the first four cases.

\medskip

\noindent \textbf{Case 5}: Suppose that $H$---which has pressure 2---does not stabilize a line on $L(E_8)$, and let $W$ denote the $\{14_i,24_{21}^\pm\}$-heart of $L(E_8){\downarrow_H}$. The copies of $64_1$ must split off, as they have no extensions with the other composition factors. Suppose that $\soc(W)$ is either $14_1$ or $24_{21}$. We consider the $\{1,6_2,14_1,24_{21}^\pm\}$-radical of $P(14_1)$, which is
\[ 14_1,14_1/1,1,6_2,6_2/14_1,24_{21},24_{21}^*/1,6_2/14_1,\]
and the $\{1,6_2,14_1\}$-radical of $P(24_{21})$, which is
\[ 14_1/1,6_2/24_{21}.\]
As neither of these contains enough trivial composition factors, we cannot have that $\soc(W)$ is simple.

If it is $14_1\oplus 24_{21}$, then we simply note that the $\{1,6_2,14_1\}$-radical of $P(14_1)$ is $14_1/1,6_2/14_1$, so a pyx for $W$ cannot have enough composition factors. Finally, if $\soc(W)$ is $14_1^{\oplus 2}$, then we need the $\{1,6_2,24_{21}^\pm\}$-radical of  $P(14_1)$, which is
\[ 1,1,6_2/24_{21},24_{21}^*/1,6_2/14_1.\]
Notice that this means that we must have a submodule $(1/14_1)^{\oplus 2}$, and a quotient module $(14_1/1)^{\oplus 2}$, and that the kernel of the quotient module contains the submodule (so that one may remove them both).

Let $v$ denote an element of order $4$ in $H$ that acts with a single Jordan block (of size $4$) on the $4$-dimensional simple modules $4_i^\pm$. Its actions on $64_1$ and $24_{21}^\pm$ have Jordan blocks $4^{16}$ and $4^6$, and on $14_1$ and $6_2$ it acts with blocks $4^3,2$ and $4,2$. Therefore $v$ acts on $L(E_8)$ with at least 58 blocks of size $4$. Consulting Table \ref{t:unipe8p4}, we see that $v$ belongs to class $2A_3$ and acts on $L(E_8)$ with blocks $4^{60},2^4$. However, the action of $v$ on $1/14_1$ (and its dual) has blocks $4^3,3$, so the action of $v$ on $L(E_8)$ has at least four blocks of size at least $3$, in addition to the 58 of size $4$. However, this now contradicts Lemma \ref{lem:encapsulates}, so we cannot build a pyx for $W$.

Thus $H$ stabilizes a line on $L(E_8)$ in this case.

\medskip

\noindent \textbf{Case 6}: We follow the corresponding proof of the case for $\PSL_4(4)$, at least initially. Again, $H$ has pressure 2 on $L(E_8)$. Again, we will show that $L(E_8){\downarrow_H}$ possesses a submodule of dimension $1$, $4$ or $6$, so suppose the contrary.

Since the $14_i$ appear with multiplicity $1$, removing from the socle all composition factors with zero $1$-cohomology (which are all but $14_i$ and $24_{i,j}^\pm$, see Table \ref{t:modules48916}), and all simple summands, we obtain a module $W$, which has up to field and graph automorphism one of $24_{12}$, $24_{12}\oplus 24_{21}$ and $24_{12}\oplus 24_{21}^*$ as socle.

Suppose first that $\soc(W)$ is simple, and therefore $24_{12}$, so $\top(W)\cong 24_{12}^*$. The $\{1,4_i^\pm,6_i,14_i,16_{12}^\pm,\bar{16}_{12}^\pm,24_{21}^\pm\}$-radical of $P(24_{12})$ is
\[ 1,6_1/4_2,4_2^*,14_1,14_2/1,6_1,\bar{16}_{12}^*/24_{12}.\]
Since this has only two trivial factors, we cannot produce a pyx for $W$, so the socle cannot be simple. Thus we may assume that the socle of $W$ has two factors. Hence we need the $\{1,4_i^\pm,6_i,14_i,16_{12}^\pm,\bar{16}_{12}^\pm\}$-radical of $P(24_{12})$, which of course is exactly the same module.

On this we cannot place a copy of $24_{21}$ or $24_{21}^*$ (of course), but we can place a single copy of $24_{12}^*$. This copy therefore must be on top of it in $W$. This yields a module
\[ 24_{12}^*/1,6_1/4_2,4_2^*,14_1,14_2/1,6_1,\bar{16}_{12}^*/24_{12},\]
but it has quotients other than $24_{12}^*$. Indeed, removing all quotients that are not of dimension $24$ (which cannot occur in $W$) leaves the much smaller module
\[24_{12}^*/1/14_1/1,6_1/24_{12}.\]
Thus a pyx for $W$ is the sum of this module and its image under either the field or field-graph automorphism. In fact, this must be $W$ itself, not just a pyx for $W$. Hence $W$ contains no $4$-dimensional factors and at most two $6$-dimensional factors. Thus we must find at least one $6$-dimensional composition factor in the $\{4_i^\pm,6_i,16_{12}^\pm,\bar{16}_{12}^\pm\}$-radical of $L(E_8){\downarrow_H}$.

The $\{4_i^\pm,6_i,16_{12}^\pm,\bar{16}_{12}^\pm\}$-radicals of $P(16_{12})$ and $P(\bar{16}_{12})$ are simply $4_1^*/16_{12}$ and $4_2/\bar{16}_{12}$, so as with $\PSL_4(4)$, there must be a $4$- or $6$-dimensional submodule of $L(E_8){\downarrow_H}$, as we claimed.

\medskip

We check that there exist elements of order $195=65\times 3$ that cube to a given element $x$ of order $65$ in $\PSU_4(4)$ and that stabilize eigenspaces comprising the composition factors $4_i^\pm$ and $6_i$. We find two elements of order $195$ powering to $x$ and stabilizing the eigenspaces comprising a given $4$-dimensional module, and 80 elements that stabilize the eigenspaces for each $6_i$. Although this is enough to show that $H$ is contained in a member of $\ms X$, one even finds elements of order $1365=195\times 7$ powering to one of the elements of order $195$ and with the same number of distinct eigenvalues on $L(E_8)$ as the element of order $195$. This element is a blueprint for $L(E_8)$ by Theorem \ref{thm:largeorderss}, and so the stabilizer of a $4$- or $6$-dimensional submodule of $L(E_8){\downarrow_H}$ is positive dimensional.

\medskip

We now have to do exactly the same as for $\PSL_4(4)$, and check all members of $\ms X$ to see that $H$ stabilizes a line on $L(E_8)$, not just a $4$- or $6$-space. The exact same proof works for $\PSU_4(4)$ as for $\PSL_4(4)$, so we see that $H$ stabilizes a line on $L(E_8)$.

\medskip

\noindent $\boldsymbol{q=8}$: We look for sets of composition factors that are conspicuous for elements of order at most $19$, and that have either no trivials or more modules of dimension $14$, $24$ and $84$ than of dimension $1$. (All simple modules of dimension at most $128$, or self-dual with dimension at most $248$, with non-zero $1$-cohomology have these dimensions, as we see from Table \ref{t:modules48916}.) Since there are still 9320 possible dimension sets for the composition factors, and billions of possible modules of dimension $248$, we have to narrow down the possibilities. The six cases with non-positive pressure for $L\cong \PSU_4(2)$ that were omitted before are
\[ 14,6^{10},(4,4^*)^{16},1^{46},\quad 14^8,6^{17},(4,4^*)^2,1^{18},\quad 64,(20,20^*)^2,14^2,6^4,(4,4^*)^5,1^{12},\]
\[ 64^2,20,20^*,14^2,(4,4^*)^5,1^{12},\quad 64,20,20^*,14^4,6^5,(4,4^*)^6,1^{10},\] \[64^2,20,20^*,14^3,6^2,(4,4^*)^2,1^{10}.\]

The dimensions of the composition factors of the restrictions of simple modules for $H$ to $L$ are slightly different from $\PSL_4(8)$, and are as follows.
\begin{center}
\begin{tabular}{ccc}
\hline Dimension of module for $H$ & Is self-dual? & Restriction of module to $L$
\\ \hline $16$ & No & $6^2,4$ or $14,1^2$
\\ $24$ & No & $20,4$
\\ $36$ & Yes & $14^2,6,1^2$
\\ $56$ & No & $20^2,6^2,4$
\\ $64$ & No & $20^2,6^2,4^3$ or $20^2,14,4^2,1^2$
\\ $80$ & No & $64,6^2,4$ or $20,14^3,6^2,4,1^2$
\\ $84$ & Yes & $64,6^2,4^2$
\\ $96$ & No & $20,14^4,6^2,4,1^4$ or $64,6^4,4^2$
\\ $120$ & No & $20^4,6^4,4^4$
\\ $196$ & Yes & $20^4,14^5,6^4,4^4,1^6$
\\\hline
\end{tabular}
\end{center}
We see from this that to remove two trivial composition factors from $L(E_8){\downarrow_H}$ that appear in $L(E_8){\downarrow_L}$, we need at least one $14$. We can in particular note that of the six sets of factors of non-positive pressure for $L(E_8){\downarrow_L}$ above, only the fifth could yield positive pressure for $L(E_8){\downarrow_H}$, and then only with the dimensions
\[84,24^2,16^4,6^3,4^8,1^2.\]
There are no such conspicuous sets of composition factors. We therefore assume that $L(E_8){\downarrow_L}$ comes from one of the three possible sets of factors of positive pressure from earlier in this proof.

From this it is easy to see first that modules of dimension $196$, $120$, $96$ and $80$ cannot occur in $L(E_8){\downarrow_H}$, then that the number of composition factors of dimension $14$ and $36$ put together does not exceed 4, that there are at most eight trivial factors, and so on. These sorts of easy restrictions can bring the number of possible sets of dimensions down to something manageable, a couple of dozen. The total number of modules is around 4 million, and we can easily check the traces of these modules.

We find, up to field automorphism, two such sets of factors, but one, namely
\[ 24_{12},24_{12}^*,24_{21},24_{21}^*,16_{12},16_{12}^*,\bar{16}_{12},\bar{16}_{12}^*,14_1,14_2,6_1^2,6_2^2,(4_1,4_1^*)^2,(4_2,4_2^*)^2,1^4,\]
has pressure $0$ because $24_{12}^\pm$ does not have non-zero $1$-cohomology.

The other conspicuous set of composition factors for $L(E_8){\downarrow_H}$ is
\[ 64_1^2,24_{21},24_{21}^*,14_1^4,6_2^2,1^4,\]
which appeared in the case $q=4$, again having pressure $2$. The proof is identical to the $q=4$ case, even down to the radicals of $P(14_1)$ and $P(24_{21})$ being the same. Hence $H$ stabilizes a line on $L(E_8)$, as needed.
\end{proof}

\section{\texorpdfstring{$\Omega_7$}{Omega 7}}

In this section we let $H\cong \Omega_7(q)$ for some $q\leq 9$. We continue with our definition of $u$ from the start of the chapter.

\begin{proposition}\label{prop:omega7}
If $H\cong \Omega_7(q)$ for some odd $q\leq 9$, then $H$ is a blueprint for $L(E_8)$.
\end{proposition}
\begin{proof} $\boldsymbol{q=3}$: There is only one set of composition factors for $L(E_8){\downarrow_H}$ that is conspicuous for elements of orders $2$ and $4$, and that is
\[ 35^2,27,21^5,7^6,1^4.\]
Since none of these composition factors has an extension with any other, $L(E_8){\downarrow_H}$ is semisimple. The unipotent element $u$ acts on $L(E_8)$ with blocks $3^{14},2^{64},1^{78}$; thus $u$ comes from the generic class $2A_1$, proving that $H$ is a blueprint for $L(E_8)$ by Lemma \ref{lem:genericmeansblueprint}.

\medskip

\noindent $\boldsymbol{q=9}$: As $\Omega_7(3)$ is contained in $\Omega_7(9)$, $H$ is a blueprint for $L(E_8)$.

\medskip

\noindent $\boldsymbol{q=5}$: There is a single set of composition factors that is conspicuous for elements of orders $2$ and $3$, and it is the same as for $q=3$. Again, this is semisimple, and $u$ acts on $L(E_8)$ with blocks $3^{14},2^{64},1^{78}$; thus $u$ comes from the generic class $2A_1$, proving that $H$ is a blueprint for $L(E_8)$ via Lemma \ref{lem:genericmeansblueprint}.

\medskip

\noindent $\boldsymbol{q=7}$: Again, elements of orders $2$ and $3$ yield a single conspicuous set of composition factors,
\[ 35^2,26,21^5,7^6,1^5.\]
The only extension this time is between $26$ and $1$, so we have at least $35^{\oplus 2}\oplus 21^{\oplus 5}\oplus 7^{\oplus 6}\oplus 1^{\oplus 3}$, on which $u$ acts with blocks $3^{11},2^{58},1^{71}$. Examining \cite[Table 9]{lawther1995}, we see that $u$ must come from the generic class $2A_1$ again, so $H$ is a blueprint for $L(E_8)$ by Lemma \ref{lem:genericmeansblueprint}.
\end{proof}

\section{\texorpdfstring{$\PSp_6$}{PSp 6}}

In this section we consider the groups $\PSp_6(q)$ for $q\leq 9$.

\begin{proposition} Let $H\cong \PSp_6(q)$ for some $q\leq 9$.
\begin{enumerate}
\item If $q$ is even then $H$ stabilizes a line on $L(E_8)$.
\item If $q$ is odd then $H$ does not embed in $\mb G$.
\end{enumerate}
\end{proposition}
\begin{proof} $\boldsymbol{q}$\textbf{ odd}: For $q=3,5,7$, there are no sets of composition factors for $L(E_8){\downarrow_H}$ that are conspicuous for elements of order at most $4$, so $H$ does not embed in $\mb G$ for those values. Since $\PSp_6(3)$ does not embed in $\mb G$, neither does $\PSp_6(9)$.

\medskip

\noindent $\boldsymbol{q=2}$: There are seven conspicuous sets of composition factors for $L(E_8){\downarrow_H}$, which are
\[ 14,8^{16},6^{10},1^{46},\quad 14^8,8^2,6^{17},1^{18},\quad 48^2,14^4,8^4,6^9,1^{10},\quad 48^2,14^4,8^5,6^8,1^8,\] \[64,48,14^5,8^3,6^6,1^6,\quad 
64,48^2,14^3,8^2,6^4,1^6,\quad 
64^2,48,14^4,6^2,1^4.\]
The $1$-cohomologies of these modules are given in Table \ref{t:modules2}. The pressures of these sets of composition factors are $-36$, $-1$, $1$, $2$, $1$, $0$ and $-1$ respectively.

\medskip

\noindent \textbf{Cases 1, 2, 6 and 7}:  These all have non-positive pressure,  so $H$ must stabilize a line on $L(E_8)$ by Proposition \ref{prop:pressure}.

\medskip

\noindent \textbf{Case 3}: Assume that $H$ does not stabilize a line on $L(E_8)$. Let $W$ denote the $\{6,48\}$-heart of $L(E_8){\downarrow_H}$. The socle of $W$ is either $6$ or $48$ since $W$ has pressure $1$, by Proposition \ref{prop:pressure}.

The $\{1,6,8,14\}$-radical of $P(48)$ is
\[ 1,1,8/6,6/1,1,14,14/6,6/1,8,8/6,6/1,14,14/6/1/48.\]
This has only seven trivial composition factors, thus cannot be a pyx for $W$ (without the copy of $48$ on top). Thus the socle cannot be $48$.

The dimension of $\Ext^1_{kH}(M_1,M_2)$ for $M_i$ simple is $0$ or $1$, and it is $1$ only for $(M_1,M_2)$ (or $(M_2,M_1)$) one of the following pairs: $(1,6)$, $(1,48)$, $(6,8)$, $(6,14)$, $(14,64)$, $(64,64)$. Thus if $64$ is not a factor then the socle layers alternate between being $6$ and $48$, and being $1$, $8$ and $14$. Thus as the pressure is $1$, every other socle layer of $W$ is either $6$ or $48$.

The $\{1,6,8\}$-radical of $P(6)$ is $1/6/1,8/6$, but both trivials are quotients of this, hence there cannot be three socle layers of $W$ of the form $6/1,8/6$ or $6/1/6$. Thus whenever there are two $6$s in the $i$th and $(i+2)$th socle layers, there must be a $14$ in between. For the third case, this situation occurs more than four times in the socle layers of $W$, but there are only four $14$s in $W$, a contradiction. Thus $H$ stabilizes a line on $L(E_8)$.

\medskip

\noindent \textbf{Case 5}: Suppose that $H$ does not stabilize a line on $L(E_8)$. This has pressure $1$, so let $W$ be as in Case 3, and note that as there is a unique $48$ in the composition factors of $W$, $\soc(W)=6$. Ignoring any $8$s and $64$s, the socle structure of $L(E_8){\downarrow_H}$ must be
\[ 6/1,14/6/1,14/6/1/48/1/6/1,14/6/1,14/6,\]
with another $14$ somewhere in the module. In particular, consider the quotient of $L(E_8){\downarrow_H}$ by the $\{48\}'$-radical of $L(E_8){\downarrow_H}$. We see that it cannot contain $64$ as both $48$ and $64$ appear exactly once. This quotient must have the socle structure $6/1,14/6/1,14/6/1/48$, ignoring $8$s. This is not a submodule of $P(48)$, as we can see from the submodule in Case 3, hence $H$ stabilizes a line on $L(E_8)$.

\medskip

\noindent \textbf{Case 4}: Suppose that $H$ does not stabilize a line on $L(E_8)$, and let $W$ be as before. Let $L$ denote a copy of $\Alt(8)$ inside $H$, and note that the composition factors of $L(E_8){\downarrow_L}$ are
\[ (20,20^*)^2,14^4,8^9,6^8,1^8.\]
Since $L$ stabilizes a line on $L(E_8)$ by \cite[Proposition 8.5]{craven2017}, it lies inside a maximal-rank or parabolic subgroup of $\mb G$ by Lemma \ref{lem:maxrankorpara}.

If $L$ lies inside an $A_7$-parabolic subgroup, it acts as either $4^{\oplus 2}$ or $4\oplus 4^*$ on $M(A_7)$, and in either case the module $(1/14/1)^{\oplus 4}$ is a summand of the action of the projection of $L$ inside the $A_7$-Levi subgroup on $L(E_8)$. In particular, this means that no $14$ may lie in the socle of $L(E_8){\downarrow_L}$, and this severely restricts the possible submodules of $L(E_8){\downarrow_H}$. We suppose that this is the case, and prove that $H$ must stabilize a line on $L(E_8)$ in this situation.

Let $V_1$ denote the $\{1,6,8\}$-radical of $W$, and $V_2$ the preimage in $W$ of the $\{14\}$-radical of the quotient $W/V_1$. The corresponding submodule of $P(6)$ has structure
\[ 1,14/6/1,8,14,14/6,\]
and on restriction to $L$ there is no subquotient $14/1$. Therefore none of these $14$s may appear in $W$, i.e., $V_1=V_2$. In particular, $\soc(W/V_1)$ is either $48$ or $48^{\oplus 2}$.

Since $W$ has pressure $2$, and the submodule $1/6/1,8/6$ of $P(6)$ has $1/6/8/6$ as a submodule, we see that $V_1$ may have at most two trivial composition factors. Writing $\bar W$ for the $\{48\}$-heart of $L(E_8){\downarrow_H}$, we therefore see that it must possess all four $14$s and at least four of the eight trivial composition factors.

Therefore we consider the $\{1,6,8,14\}$-radical of $P(48)$ above, and note that it only possesses four $14$s, so all must occur in $\bar W$. Remove any quotients that are not $14$ from this, and we find a module
\[ 14,14/6,6/8,8/6,6/1,14,14/6/1/48,\]
which must be a submodule of $\bar W$. This however cannot be a subquotient of $L(E_8){\downarrow_H}$, since it has a pressure $4$ subquotient, namely $6,6/8,8/6,6$. This yields a contradiction, and so $H$ stabilizes a line on $L(E_8)$ whenever $L$ lies in an $A_7$-parabolic subgroup of $\mb G$.

\medskip

Thus we may assume that $L$ does not lie inside an $A_7$-parabolic subgroup of $\mb G$. Let $\mb X$ be an $A_3A_3$-Levi subgroup of $\mb G$, and suppose that $L$ lies in it. If $L$ lies in one of the $A_3$ factors then $L$ has at least 24 composition factors on $L(E_8)$, so this is not correct. Thus $L$ embeds diagonally in $\mb X$. In particular, it acts on the two modules $M(A_3)$ as either $L(100)$ or $L(001)$. Since $L(E_8){\downarrow_{\mb X}}$ has composition factors $(101,000)$, $(000,101)$, and $(100,100)$, $(100,001)$ and their duals, we see that $L(E_8){\downarrow_L}$ contains four subquotients $1/14/1$ (which is $4\otimes 4^*$). Thus we may assume that $L$ does not lie in any $A_3A_3$-parabolic subgroup of $\mb G$.

Since $L$ does not lie in $A_1$, $A_2$, $B_2$ or $G_2$, if $L$ lies in a central product of one of these with another group $\mb X$, then $L$ lies in $\mb X$. By \cite[Proposition 4.2]{craven2017}, $L$ is not contained in an $E_7$-parabolic subgroup, and so $L$ is not contained in $E_7A_1$ or $E_6A_2$ either. If $L\leq A_4A_4$, then $L$ must have a trivial summand on $M(A_4)$, so $L$ lies inside an $A_3A_3$-Levi subgroup, and we use the previous paragraph. If $L$ lies inside $A_8$ then $L$ stabilizes a line on $M(A_8)$, so again lies inside an $A_7$-parabolic subgroup. This deals with all maximal-rank subgroups except for $D_8$.

For the parabolic subgroups, we have eliminated $E_6$, can eliminate $A_3A_4$ as above, and excluding $A_1$ and $A_2$ factors we see that lying inside $A_1A_2A_4$-parabolic means $L$ lies in the $A_7$-parabolic, lying in the $A_2D_5$- or $A_1E_6$-parabolics means $L$ lies in the $E_7$-parabolic subgroup. Since $L$ must stabilize a line or hyperplane on $M(A_6)$, hence if $L$ lies in the $A_1A_6$-parabolic then it lies in the $A_7$- or $E_7$-parabolic subgroup.

If $L$ lies inside the $D_7$-parabolic subgroup, then as we have considered all other parabolics of $\mb G$, we may assume that $L$ is $D_7$-irreducible. In particular, any subspace of $M(D_7)$ that $L$ stabilizes and acts irreducibly on must be non-singular, as totally isotropic subspaces have parabolic stabilizers, contradicting the $D_7$-irreducibility of $L$. In particular, there cannot be a subspace $4^\pm$ of $M(D_7){\downarrow_L}$, as $L$ cannot act non-singularly on a non-self-dual module, and there cannot be a trivial submodule.

Thus $L$ acts with composition factors on $M(D_7)$ as $14$ or $6^2,1^2$ or $6,1^8$. In the first case the composition factors of $L$ on $L(E_8)$ are $64^2,20,20^*,14^4,6^2,4,4^*,1^4$, by checking traces of elements of orders $3$. (In other words, $L$ acts irreducibly on the half-spin modules as well as $M(D_7)$.) In the second and third, $L$ stabilizes a line on $M(D_7)$ so is not $D_7$-irreducible.

If $L$ lies inside the $D_8$ maximal-rank subgroup, then we may assume that $L$ is $D_8$-irreducible, hence again $4$ is not in the socle of $M(D_8){\downarrow_L}$, and neither is $1$. There are no sets of composition factors that can satisfy these conditions, and so we have proved that $L$ always lies in an $A_7$-parabolic subgroup, completing the proof.

\medskip

\noindent $\boldsymbol{q=4,8}$: The composition factors of tensor products of simple modules for $L=\Sp_6(2)$ whose product has dimension at most $248$ are in the table below.
\begin{center}
\begin{tabular}{cc}
\hline Tensor product & Factors
\\ \hline $6\otimes 6$ & $14^2,6,1^2$
\\ $6\otimes 8$ & $48$
\\ $8\otimes 8$ & $14^2,8,6^4,1^4$
\\ $6\otimes 14$ & $64,8,6^2$
\\ $8\otimes 14$ & $112$
\\ $14\otimes 14$ & $48^2,14^5,6^4,1^6$
\\ $6^{\otimes 3}$ & $64^2,14^2,8^2,6^7,1^2$
\\\hline\end{tabular}
\end{center}
By comparing with the conspicuous sets of composition factors for $L(E_8){\downarrow_L}$, we therefore see that the dimension of any factor of $L(E_8){\downarrow_H}$ is at most $84$. Furthermore, if $M$ is a simple $kH$-module then the dimension of $H^1(H,M)$ is at most $1$ and it is non-zero only if $M$ has dimension $6$, $48$ or $84$ (not all modules of dimension $48$ and $84$ have non-zero $1$-cohomology).

If $L(E_8){\downarrow_H}$ has no trivial composition factors, then from the possibilities for $L(E_8){\downarrow_L}$ and the table above, we see that the dimensions of the factors of $L(E_8){\downarrow_H}$ are
\[48^2,64^2,8^3,\qquad 64^2,48,36,14,8^2,6,\qquad 64^2,48,36^2.\]
There are no sets of composition factors with these dimensions that are conspicuous for elements of order at most $13$.

If $L(E_8){\downarrow_H}$ has trivial factors then we can consider its pressure, so we eliminate all sets of dimensions that must have non-positive pressure. For each set of factors for $L(E_8){\downarrow_L}$ one obtains a small possible set of dimensions for $L(E_8){\downarrow_H}$, and it is easy to check for conspicuous sets of factors for each of these in turn.

For $q=4,8$ we obtain, up to field automorphism, two sets of composition factors for $L(E_8){\downarrow_H}$ that are conspicuous for elements of order up to $13$ and have more factors of dimension $6$, $48$ and $84$ than trivial factors. The first can be written down the same for both $q=4$ and $q=8$, and it is
\[ 48_{21}^2,14_2^4,8_1^4,8_2,6_2^8,1^8;\]
this has pressure $0$, because the $48$-dimensional modules that have non-zero $1$-cohomology are $48_i=6_i\otimes 8_i$, not $48_{i,j}=6_i\otimes 8_j$.

The second set of composition factors changes a little between $q=4$ and $q=8$: for $q=8$ it is
\[64_{13},48_1,48_{23},14_1,14_2,8_1^2,8_3^2,6_1^2,6_2^2,1^4.\]
We aim to show that $H$ stabilizes a line on $L(E_8)$, so suppose that this is not the case. This has pressure $1$, and as $64_{13}$, $48_1$, $48_{23}$, $14_1$ and $14_2$ all appear with multiplicity $1$, if they are in the socle then they are summands, so can be ignored. Thus we may assume that the socle is either $6_1$ or $6_2$, by quotienting out by any $8_i$. The $\{1,6_1,6_2,8_1,8_3\}$-radicals of $P(6_1)$ and $P(6_2)$ are $6_1/1,8_1/6_1$ and $1/6_2$ respectively. Notice that the $\{1,6_1,6_2,8_1,8_3\}$-heart of $L(E_8){\downarrow_H}$ is a module whose socle consists of modules with multiplicity $1$, hence is semisimple. But this means that $H$ must stabilize a line on $L(E_8)$ because there are not enough trivials above the $6_i$. This completes the proof for $q=8$.

\medskip

When $q=4$, the other set of factors is
\[64_{12},48_1,48_2,14_1,14_2,8_1^2,8_2^2,6_1^2,6_2^2,1^4.\]
One can check that there are elements of order $255$ cubing to a given element $x$ of order $85$ in $H$, and that stabilize all of the constituent eigenspaces of $6_1$, $6_2$, $8_1$ and $8_2$. Indeed, there are 2186 roots $\hat x$ such that $\hat x^3=x$ and $\hat x$ stabilizes the eigenspaces comprising $6_i$, and 80 for $8_i$. Moreover, 728 roots stabilize both $6_1$ and $6_2$. Indeed, to show that the stabilizer of each $6_i$ and each $8_i$ is positive dimensional, one finds an element of order $1985=7\times 255$ whose seventh power $\hat x$ has the same number of distinct eigenvalues on $L(E_8)$ as it, and such that $\hat x$ is an element stabilizing the eigenspaces comprising both $6_1$ and $8_1$.

Thus if $L(E_8){\downarrow_H}$ has $1$, $6_i$ or $8_i$ as a submodule then $H$ is contained in a member of $\ms X$, but of course these are the only composition factors that appear more than once. Hence at least one of them must be in the socle of $L(E_8){\downarrow_H}$, so $H$ lies in a member of $\ms X$.

We want to show that $H$ stabilizes a line on $L(E_8)$. To see this, we run through the members of $\ms X$, and find that $H$ must lie inside a $D_8$ maximal-rank subgroup, which stabilizes a line on $L(E_8)$.

As with $\PSL_4(4)$, $H$ cannot lie inside the $E_7$-parabolic subgroup as there are not enough factors with multiplicity $2$ in $L(E_8){\downarrow_H}$. The possibilities for embedding $H$ in the $A_7$-parabolic subgroup---acting on the natural module with factors $6_1,1^2$ or $8_1$---yield the wrong factors on $L(E_8)$. For the $D_7$-parabolic subgroup, the only embedding of $H$ acts on the natural as $8_1\oplus 1^{\oplus 6}$, clearly yielding the wrong factors. As with $\PSL_4(4)$, we eliminate all other parabolics by projecting onto the Levi factor, embedding in a subgroup, and using containment of parabolics.

For the reductive subgroups, if $H\leq G_2F_4$ then $H\leq F_4\leq E_7$, and $H$ is not contained in $A_4A_4$, or the $A_8$ subgroup as it is not in the $A_7$-parabolic subgroup, and of course not in $A_1E_7$ or $A_2E_6$ either. The only subgroup left is $D_8$, which of course stabilizes a line on $L(E_8)$, so we are done. (In this case, $H$ must act on the natural module as $8_1\oplus 8_2$.)
\end{proof}

\newpage
\chapter{Rank 2 groups for \texorpdfstring{$E_8$}{E8}}
\label{chap:rank2ine8}

This chapter considers the groups $\PSL_3(q)$, $\PSU_3(q)$, $\PSp_4(q)$ and $G_2(q)$ for $2\leq q\leq 9$, together with the Suzuki groups ${}^2\!B_2(q)$ for $q=8,32,128,512$ and ${}^2\!G_2(q)$ for $q=3,27$. In addition, we consider the derived groups when these are not simple, which are $\PSp_4(2)$, $G_2(2)$, ${}^2\!G_2(3)$. We do not consider $\PSU_3(2)$, which is soluble. 


As to be expected, there are more groups in this situation that we cannot prove are always strongly imprimitive. In particular, we cannot for $\PSL_3(q)$ for $q=3,4$, for $\PSU_3(q)$ for $q=3,4,8$, and for $\PSp_4(q)$ for $q=2$. In addition, for ${}^2\!B_2(q)$ we do not manage to prove any results for $q=8$. With the exception of these seven groups, we prove that every one of the subgroups above is strongly imprimitive when embedded in $E_8$. One case of $\PSU_3(4)$ is delayed until Section \ref{sec:diffpsu34} (although there is another case that we cannot do), and the proof for the most difficult case of $\PSL_3(5)$ is delayed until Section \ref{sec:diffpsl35}.

Throughout this chapter, as with the previous two, we let $u$ denote an element of order $p$ in the smallest conjugacy class of $H$.

\section{\texorpdfstring{$\PSL_3$}{PSL 3}}

For $\PSL_3(q)$, when $q=2,7$ we can prove that $H$ always stabilizes a line on $L(E_8)$, but for $q=3,4$ we cannot show that $H$ is always Lie imprimitive, never mind strongly imprimitive. For $q=5,8,9,16$, there are some cases where $H$ is a blueprint for $L(E_8)$, some where it stabilizes a line on $L(E_8)$, and others where all we can show is that there is an orbit of subspaces under the action of $N_{\Aut^+(\mb G)}(H)$ whose simultaneous stabilizer is positive dimensional, i.e., it is strongly imprimitive via Theorem \ref{thm:intersectionorbit}.

This proposition proves all of the results that we want except for the single case for $q=5$, which is delayed until Section \ref{sec:diffpsl35}.

\begin{proposition}\label{prop:sl3ine8} Let $H\cong \PSL_3(q)$ for some $2\leq q\leq 9$ or $q=16$.
\begin{enumerate}
\item If $q=2,7$ then $H$ stabilizes a line on $L(E_8)$.
\item If $q=3$ then either $H$ stabilizes a line on $L(E_8)$ or the composition factors of $L(E_8){\downarrow_H}$ are one of
\[ (15,15^*)^5,7^8,6,6^*,(3,3^*)^3,1^{12},\quad (15,15^*)^3,7^{10},(6,6^*)^3,(3,3^*)^7,1^{10},\] \[27,(15,15^*)^4,7^8,6,6^*,(3,3^*)^4,1^9.\]
\item If $q=4$ then either $H$ stabilizes a line on $L(E_8)$ or the composition factors of $L(E_8){\downarrow_H}$ are one of
\[ (9,9^*)^8,8_1^8,8_2^3,1^{16},\quad 64^3,(9,9^*)^2,8_1,8_2,1^4.\]
\item If $q=5$ then either $H$ stabilizes a line on $L(E_8)$ or the composition factors of $L(E_8){\downarrow_H}$ are
\[ 39,39^*,35,35^*,18,18^*,15_2,15_2^*,8^2,6,6^*,3,3^*.\]
\item If $q=8,9,16$ then $H$ is strongly imprimitive.
\end{enumerate}
\end{proposition}
\begin{proof} $\boldsymbol{q=7}$: The conspicuous sets of composition factors for $L(E_8){\downarrow_H}$ are
\[ 71,71^*,37,27,(10,10^*)^2,1^2,\;\; 117,(28,28^*)^2,8^2,1^3,\;\; 37,(35,35^*)^2,27,(10,10^*)^2,1^4,\]
\[ 35,35^*,27^6,8,1^8,\quad 27^3,(10,10^*)^5,8^8,1^3,\quad 10,10^*,8^{25},1^{28}.\]
Each of these has non-positive pressure (see Table \ref{t:modules7}), so $H$ stabilizes a line on $L(E_8)$ by Proposition \ref{prop:pressure}.

\medskip

\noindent $\boldsymbol{q=5}$: There are fifteen conspicuous sets of composition factors for $L(E_8){\downarrow_H}$, fourteen of which have a trivial composition factor. The fourteen that have a trivial factor are:
\[ 8,(3,3^*)^{27},1^{78},\quad 10,10^*,8^{25},1^{28},\quad 8^8,(6,6^*)^3,(3,3^*)^{21},1^{22},\] \[15_2,15_2^*,8^9,(6,6^*)^7,(3,3^*)^8,1^{14},\quad  19,(15_2,15_2^*)^3,(10,10^*)^2,8^2,(6,6^*)^6,1^{11},\] \[35,35^*,19^6,8^7,1^8,\quad 63,(35,35^*)^2,(10,10^*)^2,1^5,\]
\[39,39^*,19^2,(15_1,15_1^*)^2,10,10^*,8^3,(6,6^*)^2,1^4,\] \[63,39,39^*,18,18^*,15_1,15_1^*,10,10^*,(3,3^*)^3,1^3,\]
\[ 19,18,18^*,(15_2,15_2^*)^3,(10,10^*)^2,8^3,(6,6^*)^2,(3,3^*)^2,1^3,\] \[125,19^4,8,(6,6^*)^2,(3,3^*)^2,1^3,\quad 63^3,15_2,15_2^*,10,10^*,3,3^*,1^3,\]
\[ 19^3,(10,10^*)^5,8^{11},1^3,\quad (15_2,15_2^*)^3,10,10^*,8^8,(6,6^*)^3,(3,3^*)^6,1^2.\]
As only $39^\pm$ and $63$ have non-zero $1$-cohomology, and each is $1$-dimensional (see Table \ref{t:modules5}), we see that all of these sets of composition factors have non-positive pressure, hence $H$ stabilizes a line on $L(E_8)$ in those cases by Proposition \ref{prop:pressure}. 

We thus have the remaining case, where the composition factors of $L(E_8){\downarrow_H}$ are
\[ 39,39^*,35,35^*,18,18^*,15_2,15_2^*,8^2,6,6^*,3,3^*,\]
the unique conspicuous set with no trivial factor. This is as stated in the proposition (and dealt with in Section \ref{sec:diffpsl35}).
\medskip

\noindent $\boldsymbol{q=3}$: There are fifteen conspicuous sets of composition factors for $L(E_8){\downarrow_H}$:
\[ 7,(3,3^*)^{27},1^{79},\quad 7^{27},3,3^*,1^{53},\quad 7^8,(6,6^*)^3,(3,3^*)^{21},1^{30},\]
\[ 15,15^*,7^9,(6,6^*)^7,(3,3^*)^8,1^{23},\quad 27^5,(15,15^*)^2,7^3,6,6^*,3,3^*,1^{14},\] \[(15,15^*)^5,7^8,6,6^*,(3,3^*)^3,1^{12},\quad 27,(15,15^*)^3,7^5,(6,6^*)^6,(3,3^*)^2,1^{12},\]
\[27^3,7^{18},(3,3^*)^5,1^{11},\quad 27^6,15,15^*,7^3,6,6^*,(3,3^*)^2,1^{11},\] \[(15,15^*)^3,7^{10},(6,6^*)^3,(3,3^*)^7,1^{10},\quad 27,(15,15^*)^4,7^8,6,6^*,(3,3^*)^4,1^9,\]
\[27^2,(15,15^*)^2,7^5,(6,6^*)^6,(3,3^*)^3,1^9,\quad 27^7,7^3,6,6^*,(3,3^*)^3,1^8,\]
\[ 27^3,(15,15^*)^3,7^3,(6,6^*)^4,1^8,\quad 27^3,15,15^*,7^5,(6,6^*)^6,(3,3^*)^4,1^6.\]
The pressures of the sets of composition factors are $-78$, $-26$, $-22$, $-12$, $-7$, $6$, $-1$, $7$, $-6$, $6$, $7$, $0$, $-5$, $1$ and $1$ respectively.

\medskip

\noindent \textbf{Cases 1, 2, 3, 4, 5, 7, 9, 12 and 13}:  In these cases $H$ stabilizes a line on $L(E_8)$ by Proposition \ref{prop:pressure}.

\medskip

\noindent \textbf{Cases 14 and 15}: If the pressure of $L(E_8){\downarrow_H}$ is $1$ then, upon quotienting out by the $\{7,15^\pm\}'$-radical, we obtain a submodule of (up to duality) either $P(7)$ or $P(15)$. There are five and four trivial composition factors of these modules respectively, so there must be a trivial module in the $\{7,15\}'$-radical. This means there is a trivial submodule, so $H$ stabilizes a line on $L(E_8)$ in these two cases as well.

\medskip

\noindent \textbf{Case 8}: Of course, as $27$ is the Steinberg module and hence projective, it splits off as a summand. The $\{1,3^\pm,7\}$-radical of $P(7)$ is
\[ 3,3^*/7,7/1,3,3^*/7,\]
so we need at least twice as many $7$s as $1$s in order for $H$ not to stabilize a line. In this case we do not have that many, so $H$ stabilizes a line on $L(E_8)$.

\medskip

\noindent \textbf{Cases 6, 10 and 11}: These are left over, and are stated in the proposition.

\medskip

\noindent $\boldsymbol{q=9}$: Up to field automorphism, there are 27 sets of composition factors for $L(E_8){\downarrow_H}$ that are conspicuous for elements of order at most $20$. Thirteen of these are of non-positive pressure (the simple modules for $H$ with non-zero $1$-cohomology have dimensions $7$ and $21$), so $H$ stabilizes a line on $L(E_8)$ by Proposition \ref{prop:pressure}.

First, we note that by a computer calculation, the eigenvalues on $L(E_8)$ of an element $x\in H$ of order $80$ determine the semisimple class of $\mb G$ containing $x$. If there exists an element $\hat x$ in $\mb G$ that squares to $x$ and shares the same eigenspaces of $L(E_8)$ as $x$, then we can use Lemma \ref{lem:allcasesstronglyimp} to conclude that $H$ is strongly imprimitive.

Fourteen of the 27 sets of factors have this property, but there is some overlap with the sets of non-positive pressure, and we are left with eight sets of factors where there is no such element $\hat x$, and $L(E_8){\downarrow_H}$ has positive pressure. These are
\[ 45_{12},45_{12}^*,27_2,(18_{12},18_{12}^*)^2,7_1,7_2^5,(3_1,3_1^*)^2,1^5,\] \[\bar{45}_{12},\bar{45}_{12}^*,27_2,(\bar{18}_{12},\bar{18}_{12}^*)^2,7_1,7_2^5,(3_1,3_1^*)^2,1^5,\] 
\[\begin{split}21_{21},21_{21}^*,18_{21},18_{21}^*,\bar{18}_{21},\bar{18}_{21}^*,15_1,15_1^*,&9_{12},9_{12}^*,\bar 9_{12},\bar 9_{12}^*,
\\ &7_1^3,7_2,6_1,6_1^*,(3_1,3_1^*)^2,(3_2,3_2^*)^2,1^4,\end{split}\]
\[\begin{split}21_{12},21_{12}^*,21_{21},21_{21}^*,18_{12},18_{12}^*,\bar{18}_{21},\bar{18}_{21}^*,9_{12},&9_{12}^*,\bar 9_{12},\bar 9_{12}^*,
\\ &7_1^2,7_2^2,(3_1,3_1^*)^2,(3_2,3_2^*)^2,1^4,\end{split}\]
\[ 21_{12},21_{12}^*,18_{12},18_{12}^*,(15_2,15_2^*)^3,(\bar 9_{12},\bar 9_{12}^*)^2,7_2^4,(3_1,3_1^*)^2,1^4,\]
\[ 49,42_{12},42_{12}^*,21_{12},21_{12}^*,15_1,15_1^*,7_1^2,7_2^2,(3_1,3_1^*)^2,1^3,\] \[ 49,45_{12},45_{12}^*,18_{12},18_{12}^*,18_{21},18_{21}^*,7_1,7_2^3,3_1,3_1^*,1^3,\]
\[49,\bar{45}_{12},\bar{45}_{12}^*,\bar{18}_{12},\bar{18}_{12}^*,\bar{18}_{21},\bar{18}_{21}^*,7_1,7_2^3,3_1,3_1^*,1^3.\]
(Recall the labelling conventions from Chapter \ref{chap:labelling}.)

The pressures of these modules are $1$, $1$, $2$, $4$, $2$, $3$, $1$ and $1$ respectively. Let $W$ denote the $\{7_i,21_{21}^\pm\}$-heart of $L(E_8){\downarrow_H}$.

\medskip

\noindent \textbf{Cases 1, 2, 7 and 8}: These all have pressure $1$, and by Lemma \ref{lem:oddandodd} they each contain a subquotient $7_1\oplus 7_2$, which has pressure $2$. Thus $H$ stabilizes a line on $L(E_8)$ by Proposition \ref{prop:pressure}.

\medskip

\noindent \textbf{Case 3}: Suppose that $H$ does not stabilize a line on $L(E_8)$. If $21_{21}^\pm$ lies in the socle of $W$, then we need to take the $\{1,3_i^\pm,6_1^\pm,7_i,9_{12}^\pm,\bar 9_{12}^\pm,15_1^\pm,18_{21}^\pm,\bar{18}_{21}^\pm\}$-radicals of $P(7_1)$ and $P(21_{21})$, and these are
\[ 7_1,7_1/1,3_2,3_2^*/7_1,\qquad 1/7_1,7_2/1,3_2,3_2^*/21_{21}.\]
The sum of these two must be a pyx for $W$, but only has two trivial composition factors, so $H$ must stabilize a line on $L(E_8)$, a contradiction. Hence $\soc(W)$ must consist solely of copies of $7_1$, so in fact only one as we only have three copies of $7_1$ in $L(E_8){\downarrow_H}$. Thus we take the $\{1,3_i^\pm,6_1^\pm,7_i,9_{12}^\pm,\bar 9_{12}^\pm,15_1^\pm,18_{21}^\pm,\bar{18}_{21}^\pm,21_{21}^\pm\}$-radical of $P(7_1)$, which is
\begin{equation} 1/7_1,7_1,7_2/1,1,3_2,3_2,3_2^*,3_2^*/7_1,7_1,21_{21},21_{21}^*/1,3_2,3_2^*/7_1.\label{eq:smallermoduleforCase3}\end{equation}

In fact, this is not quite true. This module has an extension with $21_{21}^\pm$, but when you place either of those modules on top and then take the $\{21_{21}^\pm\}'$-residual, one finds that both copies of $21_{21}^\pm$ remain in the residual. Since only one copy of $21_{21}^\pm$ appears in $L(E_8){\downarrow_H}$, this second copy at the top cannot appear in $W$.

Thus we may work in the smaller module in (\ref{eq:smallermoduleforCase3}). This module does not have enough trivial composition factors, so $L(E_8){\downarrow_H}$ has more trivial factors than $W$. Hence $H$ stabilizes a line on $L(E_8)$.

\medskip

\noindent \textbf{Case 5}: Assume that $H$ does not stabilize a line on $L(E_8)$. This case has composition factors a subset of the composition factors of the third case (twisted by the field automorphism), so the radicals above eliminate $7_2$, $21_{12}$ and $7_2\oplus 21_{12}$ from being socles of $W$, leaving only $7_2^{\oplus 2}$. In this case we take the $\{1,3_1^\pm,21_{12}^\pm\}$-radical (we don't need the other modules by the above radical) of $P(7_2)$ to obtain the module
\[ 1/21_{12},21_{12}^*/1,3_1,3_1^*/7_2.\]
Two copies of $7_2$ may be placed on top of this module, but they fall into the third socle layer, so the top trivial remains a quotient in $W$, a contradiction. Thus $H$ stabilizes a line on $L(E_8)$.

\medskip

The remaining two cases have pressures $3$ and $4$. The easiest way to proceed here is to find elements of order $160$ that stabilize the eigenspaces comprising certain composition factors of $L(E_8){\downarrow_H}$, as we have done before.

\medskip

\noindent \textbf{Case 4}: There are sixteen elements of order $160$ stabilizing the eigenspaces comprising $3_1\oplus 3_1^*\oplus 7_2$, and four that stabilize the eigenspaces comprising $3_2\oplus 3_2^*\oplus 7_1\oplus \bar 9_{12}\oplus \bar 9_{12}^*\oplus 21_{21}\oplus 21_{21}^*$. Furthermore, note that $18_{12}^\pm$ and $\bar{18}_{21}^\pm$ lie in different $\Out(H)$-orbits, so no element of $N_{\Aut^+(\mb G)}(H)$ can swap such subspaces. In particular, $N_{\Aut^+(\mb G)}(H)$ cannot induce a field automorphism on $L(E_8)$ so, for example, an $N_{\Aut^+(\mb G)}(H)$-orbit of a subspace of type $3_1$ can only contain types $3_1$ and $3_1^*$. In particular, this means that we may use Lemma \ref{lem:allcasesstronglyimp} to state that $H$ is strongly imprimitive if any of the modules $3_i^\pm$, $7_i$, $\bar 9_{12}^\pm$ and $21_{21}^\pm$ lie in the socle of $L(E_8){\downarrow_H}$.

First, $9_{12}^\pm$ has no extension with any composition factor in $L(E_8){\downarrow_H}$, so they must split off as summands. The $\cf(L(E_8){\downarrow_H})$-radicals of $P(18_{12})$, $P(\bar{18}_{21})$ and $P(21_{12})$ are
\[ 18_{12}/\bar 9_{12}^*,\bar{18}_{21}/18_{12},\quad \bar{18}_{21}/3_1^*/21_{12}^*/3_1^*,18_{12}/\bar{18}_{21},\]
\[ 7_2/1,3_1,3_1^*/21_{12},21_{12}^*/1/7_1,7_2,\bar{18}_{21}^*/1,3_1,3_1^*/21_{12}.\]
The sum of these modules only has three trivial factors and $L(E_8){\downarrow_H}$ has four, so the socle must contain a module other than one of these. Thus $H$ is strongly imprimitive by Lemma \ref{lem:allcasesstronglyimp}, as needed.

\medskip

\noindent \textbf{Case 6}: There are four elements of order $160$ stabilizing the eigenspaces comprising $1\oplus 3_1\oplus 3_1^*\oplus 7_1\oplus 7_2\oplus 21_{12}\oplus 21_{12}^*$, so if any of these modules lie in the socle of $L(E_8){\downarrow_H}$ then $H$ is strongly imprimitive by Lemma \ref{lem:allcasesstronglyimp}.

Thus the socle of $L(E_8){\downarrow_H}$ must consist of some of $49$, $42_{12}^\pm$ and $15_1^\pm$ if $H$ is not strongly imprimitive, which we will assume. If $49$ is a submodule of the socle then it splits off as a summand as it appears exactly once and is self-dual, so we may ignore this. The $\cf(L(E_8){\downarrow_H})\setminus\{15_1^\pm\}$-radical of $P(15_1)$ is
\[ 1,3_1,3_1^*/7_1,7_2/42_{12}^*,49/15_1,\]
and the $\cf(L(E_8){\downarrow_H})\setminus\{42_{12}^\pm\}$-radical of $P(42_{12})$ is
\[ 1/3_1^*,7_1,15_1^*/42_{12}.\]
These clearly have far too few composition factors, so $\soc(L(E_8){\downarrow_H})$ contains some other module, and we are done.

\medskip

\noindent $\boldsymbol{q=2}$: There are seven conspicuous sets of composition factors for $L(E_8){\downarrow_H}$, which are
\[ 8,(3,3^*)^{27},1^{78},\quad 8^{22},(3,3^*)^{6},1^{36},\quad 8^{25},(3,3^*)^{3},1^{30},\quad 8^{28},1^{24},\]
\[ 8^{8},(3,3^*)^{27},1^{22},\quad 8^{11},(3,3^*)^{24},1^{16},\quad 8^{14},(3,3^*)^{21},1^{10}.\]
The $8$s split off, and the projective cover $P(3)$ is
\[ 3/\left((3^*/3/3^*)\oplus 1\right)/3.\]
From this we can see that the only indecomposable modules that possess a trivial composition factor but no trivial submodule or quotient are $P(3)$ and $P(3^*)$; thus we need at least five times as many $3$-dimensional factors as trivial factors, else we stabilize a line on $L(E_8)$. Examining the sets above, this is not the case, so $H$ stabilizes a line on $L(E_8)$.

\medskip

\noindent $\boldsymbol{q=4}$: There are eight conspicuous sets of composition factors for $L(E_8){\downarrow_H}$ up to field automorphism:
\[ 9,9^*,8_1^{25},1^{30},\quad 64,(9,9^*)^7,8_1^5,1^{18},\quad (9,9^*)^8,8_1^8,8_2^3,1^{16},\]
\[ 64^2,(9,9^*)^5,8_1,8_2,1^{14},\quad 64,(9,9^*)^6,8_1^4,8_2^4,1^{12},\quad (9,9^*)^7,8_1^7,8_2^7,1^{10},\]
\[ 64^3,(9,9^*)^2,8_1,8_2,1^4,\quad 64^2,(9,9^*)^3,8_1^4,8_2^4,1^2.\]

We begin by letting $M$ denote a copy of $\Alt(6)$ inside $H$, restricting these to $M$, and then applying Proposition \ref{prop:compfactorsalt6}.
\begin{center}
\begin{tabular}{cc}
\hline Case & Restriction to $M$
\\ \hline $1$ & $8_1^{25},(4_1,4_2)^2,1^{32}$
\\ $2$ & $8_1^6,8_2,(4_1,4_2)^{19},1^{40}$
\\ $3$ & $8_1^8,8_2^3,(4_1,4_2)^{16},1^{32}$
\\ $4$ & $8_1^3,8_2^3,(4_1,4_2)^{20},1^{40}$
\\ $5$ & $8_1^5,8_2^5,(4_1,4_2)^{17},1^{32}$
\\ $6$ & $8_1^7,8_2^7,(4_1,4_2)^{14},1^{24}$
\\ $7$ & $8_1^4,8_2^4,(4_1,4_2)^{19},1^{32}$
\\ $8$ & $8_1^6,8_2^6,(4_1,4_2)^{16},1^{24}$
\\\hline
\end{tabular}
\end{center}
\noindent \textbf{Cases 2, 4, 5 and 6}: These all have restrictions to $M$ that do not exist by Proposition \ref{prop:compfactorsalt6}, so $H$ does not embed in $\mb G$ with these composition factors.

\medskip

\noindent\textbf{Case 1}: This has pressure $-26$, so $H$ stabilizes a line on $L(E_8)$ by Proposition \ref{prop:pressure}.

\medskip

\noindent \textbf{Case 8}: Note that the restriction to $M$ is not warranted (but might exist), so we show that $H$ does not exist. We may number the conjugacy classes $L_1,L_2,L_3$ of subgroups $\PSL_3(2)$ and $M_1,M_2,M_3$ of subgroups $\Alt(6)$ so that the permutation module $P_{L_i}$ on the cosets of $L_i$ is the direct sum of the permutation module $P_{M_i}$ on $M_i$ and a copy of $64$. Thus if we exclude the summands $64$ on $L(E_8){\downarrow_H}$ to make a module $W$, the fixed-point spaces of $L_i$ and $M_i$ on $W$ have the same dimension. Fixing $i$, and writing $L$ and $M$ for $L_i$ and $M_i$ respectively, since $M$ must be Lie primitive, this means that the fixed-point spaces satisfy $W^M=W^L=0$.

As the composition factors of $W{\downarrow_L}$ are $(3,3^*)^9,8^8,1^2$, we see that $W{\downarrow_L}$ is exactly
\[ P(3)\oplus P(3^*)\oplus 8^{\oplus 8}\oplus A,\]
where $A$ is a module with composition factors $(3,3^*)^4$. Since an element of order $4$ in $L$ acts on $L(E_8)$ with at least 56 blocks of size $4$, from Table \ref{t:unipe8p4} we see that it acts with Jordan blocks $4^{56},3^8$ or $4^{60},2^4$. This means that $A$ is either semisimple or the sum of four indecomposable modules of the form either $3/3^*$ or $3^*/3$. In particular, since $9$ restricts to $L$ as $3/3^*/3$, this means that $\soc(W)$ is a submodule of $9\oplus 9^*$ (ignoring copies of $8_i$). Indeed, we may take the $\{8_i\}'$-heart of $W$ and this still holds (but with a different number of copies of $8$), so we replace $W$ by its $\{8_i\}'$-heart.

From this we can almost determine the exact structure of $L(E_8){\downarrow_L}$: it is
\[ P(3)^{\oplus 3}\oplus P(3^*)^{\oplus 3}\oplus P(1)^{\oplus 2}\oplus 8^{\oplus 14}\oplus A,\]
where $A$ is the module above.

We attempt to find subgroups of $\mb G$ isomorphic to $L$ with this module structure, which has composition factors on $L(E_8)$ given by
\[ 8^{14},(3,3^*)^{21},1^{10}.\]
Since $L$ stabilizes a line on $L(E_8)$, $L$ is contained in a maximal-rank or parabolic subgroup by Lemma \ref{lem:maxrankorpara}.

Suppose that $L$ is contained in a subgroup $\mb X$ of $\mb G$; we may remove any factors $A_1$ or $B_2$ from $\mb X$ as $L$ does not embed in these, and if $L$ is contained in $A_4$ then it is contained in an $A_2$-parabolic subgroup of $A_4$. In fact, we saw in the case $q=2$ above that if $V$ is an indecomposable module for $L$ with a trivial composition factor, then either it has a trivial quotient or submodule or $V=P(3)$ of dimension $16$ (and $V$ is not self-dual).

If $\mb X$ is $A_8$, then we consider the action of $L$ on $M(A_8)$, and we obtain the following table.
\begin{center}
\begin{tabular}{cc}
\hline Factors on $M(A_8)$ & Factors on $L(E_8)$
\\ \hline $3,3,3$ & $8^{25},(3,3^*)^3,1^{30}$
\\ $3,3,3^*$ or $8,1$ & $8^{11},(3,3^*)^{24},1^{16}$
\\ $3,3,1^3$ or $3,3^*,1^3$ &  $8^8,(3,3^*)^{27},1^{22}$
\\ $3,1^6$ & $8,(3,3^*)^{27},1^{78}$
\\ \hline
\end{tabular}
\end{center}
Thus $L$ is not contained in $A_8$, or any parabolic subgroup whose Levi subgroup is contained in $A_8$, eliminating the $A_7$- and the $A_1A_2A_4$-parabolics, and also $A_4A_4$ and the $A_3A_4$-parabolic (as $L$ must then lie in an $A_3A_3$-parabolic) and the $A_1A_6$-parabolic (as $L$ must stabilize a line or hyperplane on $M(A_6)$, so lie in an $A_1A_5$-parabolic, whose Levi lies inside $A_8$).

If $\mb X=E_7A_1$ then $L\leq E_7$. We see from Proposition \ref{prop:sl3ine7} below, and the fact that the composition factors of $L(E_8){\downarrow_{E_7}}$ are $M(E_7)^2,L(E_7)^\circ,1^4$, that $L$ must have at least twelve trivial composition factors on $L(E_8)$, so $L$ cannot embed in $\mb X$. Thus $L$ does not embed in $E_7A_1$ or the $E_7$-parabolic subgroup, or the $E_6A_1$-parabolic (as then it would embed in an $E_6$-parabolic).

If $\mb X=E_6A_2$, then $\mb X$ has a summand of dimension $81$ on $L(E_8)$ (the tensor product of the minimal modules), but all summands of $L(E_8){\downarrow_L}$ are of even dimension, so $L\not\leq \mb X$.

Suppose that $\mb X$ is a $D_7$-parabolic subgroup: the composition factors of $L(E_8){\downarrow_{\mb X}}$ are two copies of $M(D_7)$, one copy each of the two (dual to one another) half-spin modules, two trivial modules, and a copy of the $90$-dimensional simple part $L(D_7)^\circ$ of the Lie algebra.

If $L$ lies inside a $D_6$-parabolic subgroup of $\mb X$ then $L$ lies inside an $E_7$-parabolic, which is a contradiction. The composition factors of $L$ on $M(D_7)$ with at least one trivial factor are $8,1^6$, $(3,3^*)^2,1^2$ and $3,3^*,1^8$. In all cases there is a subgroup $L$ in $D_7$ acting semisimply on $M(D_7)$, with these factors, hence lies inside a $D_6$-parabolic subgroup. Thus no copy of $L$ with a trivial factor on $M(D_7)$ can have the correct composition factors on $L(E_8)$. Therefore $L$ cannot have a trivial factor on $M(D_7)$, so the composition factors are $8,3,3^*$. (These do yield a copy of $L$ with the correct factors on $L(E_8)$.) The possible actions of $L$, up to graph automorphism, on $M(D_8)$ are
\[ 8\oplus 3\oplus 3^*,\qquad 8\oplus (3/3^*).\]
By placing $L$ inside a $D_4D_3$ subgroup, and noting that an irreducible copy of $L$ in $D_4$ acts irreducibly on $M(D_4)$ and the two half-spin modules (if it didn't then up to graph automorphism $L$ would stabilize a line on $M(D_4)$, hence lie in a parabolic, thus cannot act irreducibly on $M(D_4)$ up to graph automorphism as the graph map permutes the parabolics), we see that the action of $L$ on the half-spin modules for $D_7$ is the tensor product of $8$ with a module, so is projective. Using trace data, the composition factors can be determined, and the action on the half-spin modules is
\[ 8^{\oplus 8}\oplus P(3)^{\oplus 2}\oplus P(3^*)^{\oplus 2}.\]

Thus the remainder of $L(E_8){\downarrow_{\mb X}}$ must, upon restriction to $L$, form copies of $8$, together with $P(3)$, $P(3^*)$, $P(1)$, and the module $A$. We may construct $L(D_7)^\circ$ by taking the exterior square of $M(D_7)$ and removing a trivial summand, and the resulting module has restriction to $L$ given by
\[ P(3)\oplus P(3^*)\oplus 8^{\oplus 4}\oplus (3,3^*/1,3,3^*/3,3^*)\oplus B,\]
where $B$ is either $3\oplus 3^*$ or $3^*/3$, depending on the action of $L$ on $M(D_7)$ given above. However, this clearly yields a contradiction, as there is no possible extension of the above module by $3$s and trivials that can make the module $P(1)^{\oplus 2}\oplus A$, which does not have a $3$ above the second socle layer. Thus $L$ cannot embed in the $D_7$-parabolic subgroup.

If $L$ embeds in a $D_5A_2$-parabolic subgroup, then it embeds in a $D_4A_2$-parabolic, and this is contained in a $D_7$-parabolic subgroup. To see this, note that the action on $M(D_5)$ must have at least two trivial summands, hence $L$ lies in a $D_4$-parabolic subgroup. Thus $L$ cannot lie in any parabolic subgroup of $\mb G$.

Finally, if $\mb X=D_8$, then $L$ must be $\mb X$-irreducible, and so the action of $L$ on $M(D_8)$ is $8^{\oplus 2}$. The action on $L(D_8)^\circ$, which is the exterior square of $M(D_8)$ with a trivial removed from top and bottom, has two trivial submodules, so as $L(E_8){\downarrow_{\mb X}}$ has structure
\[ (L(0)/L(D_8)^\circ/L(0))\oplus L(\lambda_7),\]
we see that the fixed space $L(E_8)^L$ has dimension at least $3$, a contradiction.

This proves that there is no such embedding of $L$ into $\mb G$, and therefore $H$ cannot embed in $\mb G$, as needed.

\medskip

\noindent \textbf{Cases 3 and 7}: These are left over, and are as in the proposition.

\medskip

\noindent $\boldsymbol{q=8}$: Write $x$ for an element in $H$ of order $63$, and we will use Lemma \ref{lem:allcasesstronglyimp} to show strong imprimitivity in most cases. If there are no trivial composition factors of $L(E_8){\downarrow_H}$, then up to field automorphism there are six sets of composition factors that are conspicuous for elements of order at most $21$, all of which are also conspicuous for elements of order $63$. Write
\[ 27_1=3_1\otimes 3_2\otimes 3_3,\quad 27_2=3_1\otimes 3_2\otimes 3_3^*,\quad 27_3=3_1\otimes 3_2^*\otimes 3_3,\quad 27_4=3_1^*\otimes 3_2\otimes 3_3.\]
These cases are:
\begin{small}
\[ 24_{31},24_{31}^*,\bar 9_{12},\bar 9_{12}^*,(9_{13},9_{13}^*)^2,(\bar 9_{13},\bar 9_{13}^*)^2,9_{23},9_{23}^*,\bar 9_{23},\bar 9_{23}^*,8_1^3,8_3,(3_1,3_1^*)^4,(3_2,3_2^*)^2,3_3,3_3^*,\]
\[24_{13},24_{13}^*,24_{31},24_{31}^*,(9_{13},9_{13}^*)^2,(\bar 9_{13},\bar 9_{13}^*)^2,9_{32},9_{32}^*,8_1^2,8_3^2,(3_1,3_1^*)^3,3_2,3_2^*,3_3,3_3^*,\]
\[ 27_2,27_2^*,27_4,27_4^*,24_{23},24_{23}^*,(\bar 9_{13},\bar 9_{13}^*)^2,8_1,8_2,8_3^2,(3_1,3_1^*)^2,(3_2,3_2^*)^2,\]
\[ (27_4,27_4^*)^2,(9_{12},9_{12}^*)^2,(\bar 9_{13},\bar 9_{13}^*)^2,9_{32},9_{32}^*,8_2^3,8_3,(3_1,3_1^*)^2,3_2,3_2^*,\]
\[ 27_1,27_1^*,27_4,27_4^*,24_{13},24_{13}^*,\bar 9_{12},\bar 9_{12}^*,(9_{23},9_{23}^*)^2,8_1,8_2,8_3^2,3_1,3_1^*,\]
\[ 27_2,27_2^*,27_3,27_3^*,24_{13},24_{13}^*,9_{12},9_{12}^*,(9_{32},9_{32}^*)^2,8_1,8_2,8_3^2,3_1,3_1^*.\]
\end{small}
\noindent \textbf{Case 1}: Unlike the other five cases, in this case there is no composition factor that simultaneously splits off because it has no extensions with the other factors of $L(E_8){\downarrow_H}$, and also has elements of order $189$ cubing to $x$ that preserve the eigenspaces that comprise it. By Lemma \ref{lem:semilinearfield}, since the composition factors of $L(E_8){\downarrow_H}$ are not invariant under any field automorphisms in $\Out(H)$, the $N_{\Aut^+(\mb G)}(H)$-orbits of simple submodules can only include a module and its dual.

There are 27 elements of order $189$ that stabilize the eigenspaces that comprise $3_1\oplus 3_1^*$ and nine each that stabilize the eigenspaces that comprise $3_2\oplus 3_2^*$ and $3_3\oplus 3_3^*$. If there is a $3$-dimensional submodule in $L(E_8){\downarrow_H}$ then $H$ is strongly imprimitive by Lemma \ref{lem:allcasesstronglyimp}.

In addition there are $27$ such elements that stabilize the $8_1$ factor, $81$ that stabilize $8_3$, nine that stabilize $9_{13}\oplus 9_{13}^*$ and three that stabilize $\bar 9_{13}\oplus \bar 9_{13}^*$. Thus if $H$ is strongly imprimitive then  none of these modules can lie in the socle of $L(E_8){\downarrow_H}$. We therefore may assume for a contradiction that $\soc(L(E_8){\downarrow_H})$ consists solely of a subset of $\bar 9_{12}^\pm$, $9_{23}^\pm$, $\bar 9_{23}^\pm$ and $24_{31}^\pm$. The $\cf(L(E_8){\downarrow_H})\setminus\{V^\pm\}$-radicals of $P(V)$ are as follows for $V$ one of $\bar 9_{12}$, $9_{23}$, $\bar 9_{23}$ and $24_{31}$:
\[ 3_1^*/3_2,\bar 9_{13},9_{23},\bar 9_{23}^*,24_{31}/3_1^*,3_3^*,8_3/3_2,9_{13},9_{23}^*/\bar 9_{12},\]
\[ 24_{31}^*/3_1/3_2^*/3_1/8_3,\bar 9_{12}^*,\bar 9_{13}^*/9_{23},\]
\[ 3_3,\bar 9_{12}^*,9_{13}/3_1^*,3_2^*,8_1/3_3,9_{13}^*/\bar 9_{23},\]
\[ \bar 9_{13}/3_1^*,9_{23}^*/3_2,\bar 9_{13}/3_1^*,\bar 9_{12}/3_2/3_1^*/24_{31}.\]
There are three copies of $8_1$ in $L(E_8){\downarrow_H}$, and only one in the sum of modules above. Therefore $\soc(L(E_8){\downarrow_H})$ must consist of modules other than those above, and $H$ is strongly imprimitive.

\medskip

\noindent \textbf{Cases 2, 3 and 4}: In each case the $8_3$ must split off as it has no extensions with other composition factors. It is stabilized by nine, 27 and nine elements of order $189$ cubing to $x$ respectively. Hence $H$ is strongly imprimitive by Lemmas \ref{lem:semilinearfield} and \ref{lem:allcasesstronglyimp}.

\medskip

\noindent \textbf{Case 5}: The modules $8_1$ and $8_2$ split off as summands as they have no extensions with other composition factors in $L(E_8){\downarrow_H}$. There are 729 elements of order $189$ cubing to $x$ and stabilizing the eigenspaces of $8_1$, so we are done again.

\medskip

\noindent \textbf{Case 6}: This time there are no extensions between $3_1$ and the other composition factors, so this and its dual split off as summands. There are 81 elements of order 189 cubing to $x$ and stabilizing the eigenspaces comprising $3_1\oplus 3_1^*$, so we are done as before.

\medskip

We may therefore assume that $L(E_8){\downarrow_H}$ has a trivial composition factor. Up to field automorphism, there are eleven sets of composition factors for $L(E_8){\downarrow_H}$ that are conspicuous for elements of order at most $21$ and have positive pressure.
\[ (9_{13},9_{13}^*)^8,\bar 9_{23},\bar 9_{23}^*,8_1^8,8_3,1^{14},\quad 64_{13},(9_{13},9_{13}^*)^6,(\bar 9_{23},\bar 9_{23}^*)^2,8_1^4,1^8,\]
\[ 27_2,27_2^*,(\bar 9_{12},\bar 9_{12}^*)^3,(9_{13},9_{13}^*)^3,(9_{23},9_{23}^*)^3,8_1,8_2,8_3,1^8,\] \[(27_1,27_1^*)^2,(9_{12},9_{12}^*)^2,(9_{13},9_{13}^*)^2,(9_{23},9_{23}^*)^2,8_1,8_2,8_3,1^8,\]
\[ 64_{13},27_4,27_4^*,9_{12},9_{12}^*,(9_{13},9_{13}^*)^3,(\bar 9_{23},\bar 9_{23}^*)^2,8_1^2,1^6,\] \[(9_{12},9_{12}^*)^4,(\bar 9_{12},\bar 9_{12}^*)^3,8_1^3,8_2^5,(3_1,3_1^*)^4,(3_2,3_2^*)^4,3_3,3_3^*,1^4,\]
\[ 64_{13},27_3,27_3^*,9_{12},9_{12}^*,(\bar 9_{13},\bar 9_{13}^*)^3,(9_{23},9_{23}^*)^2,(3_1,3_1^*)^2,3_2,3_2^*,1^4,\] \[24_{13},24_{13}^*,(9_{13},9_{13}^*)^2,(\bar 9_{13},\bar 9_{13}^*)^3,8_1^3,8_3^3,(3_1,3_1^*)^5,3_2,3_2^*,(3_3,3_3^*)^4,1^2,\]
\[24_{13},24_{13}^*,24_{23},24_{23}^*,(9_{13},9_{13}^*)^2,(\bar 9_{13},\bar 9_{13}^*)^2,9_{23},9_{23}^*,8_1^5,8_3,3_1,3_1^*,3_2,3_2^*,1^2,\] \[24_{12},24_{12}^*,24_{21},24_{21}^*,(9_{12},9_{12}^*)^2,(\bar 9_{12},\bar 9_{12}^*)^2,9_{13},9_{13}^*,8_1^2,8_2^4,3_1,3_1^*,3_2,3_2^*,1^2,\]
\[ 27_1,27_1^*,27_4,27_4^*,24_{12},24_{12}^*,(9_{23},9_{23}^*)^2,8_1,8_2^2,8_3^3,3_1,3_1^*,1^2.\]
The fourth, sixth and ninth cases are not conspicuous for elements of order $63$, leaving eight cases to consider.

\medskip

\noindent \textbf{Cases 1, 2, 5 and 8}: There are elements of order $189$ that cube to $x$ and have the same number of distinct eigenvalues on $L(E_8)$, so we may apply Lemma \ref{lem:allcasesstronglyimp} to obtain the result.

\medskip

\noindent \textbf{Cases 3, 10 and 11}: In these cases there are elements of order $189$ with only two or four more distinct eigenvalues on $L(E_8)$ than $x$, so will stabilize many of the composition factors of $L(E_8){\downarrow_H}$.

In the third case, there are elements stabilizing the constituent eigenspaces of $V\oplus V^*$ for each $V$ a composition factor of $L(E_8){\downarrow_H}$ except for $27_2^\pm$, and furthermore, $8_2$ has no extensions with other composition factors of $L(E_8){\downarrow_H}$, so it splits off as a summand, and $H$ is strongly imprimitive. (In fact, there are nine elements stabilizing all eigenspaces of all factors simultaneously, except for $27_2^\pm$ of course, and $8_1$.)

In the tenth case, there are elements stabilizing the constituent eigenspaces of $V\oplus V^*$ for each $V$ a composition factor of $L(E_8){\downarrow_H}$ except for $24_{21}\oplus 24_{21}^*$. However, $24_{21}^\pm$ has no extensions with other composition factors, and so these must split off as summands, proving the result.

The eleventh case is very similar to the tenth: there are elements stabilizing the constituent eigenspaces of $V\oplus V^*$ for each $V$ a composition factor of $L(E_8){\downarrow_H}$ except for $27_1^\pm$. However, both $8_1$ and $8_2$ have no extensions with other composition factors, proving that $H$ is strongly imprimitive. (In fact, there are nine elements stabilizing all eigenspaces of all factors simultaneously, except for $27_1^\pm$ of course, and $8_1$.)

\medskip

\noindent \textbf{Case 7}: The $64_{13}$ splits off as it has no extensions with other composition factors in $L(E_8){\downarrow_H}$. We prove that $H$ stabilizes a line on $L(E_8)$, so assume this is not the case.

Since $\bar 9_{13}^\pm$ has zero $1$-cohomology, the pressure of this module is $2$. Let $W$ denote the $\{9_{12}^\pm,9_{23}^\pm\}$-heart of $L(E_8){\downarrow_H}$. Since $H$ does not stabilize a line on $L(E_8)$, the module $W$ has all four trivial composition factors in it. The socle of $W$ is at most two modules from $9_{12}^\pm$, $9_{23}^\pm$ and $9_{23}^\pm$.

Suppose that $9_{12}^\pm$ does not lie in the socle of $W$. The  $\cf(L(E_8){\downarrow_H})\setminus \{9_{23}^\pm\}$-radical of $P(9_{23})$ is
\[ 27_3/\bar 9_{13}/3_1^*/3_2/3_1^*,3_2^*,9_{12}/3_1,9_{12}^*,\bar 9_{13}/1,\bar 9_{13}^*,27_3/9_{23}.\]
We immediately see that $\soc(W)$ cannot have two factors, as then a pyx for $\rad(W)$ is a sum of two of these modules (up to graph automorphism). However, $W$ has four trivial factors, and the pyx only has two.

In fact, $\soc(W)$ cannot be $9_{23}$ either: one cannot place a copy of $9_{23}^*$ on top of the radical above, and can only place $9_{23}$ again. However, then the quotient of $W/\soc(W)$ by its $\{9_{23}\}'$-radical is again a submodule of the radical above, so we still cannot place $9_{23}^*$ in it. Thus $9_{23}^*$ does not lie in $W$ at all, a clear contradiction.

Thus $9_{12}$ lies in $\soc(W)$ (up to application of a graph automorphism). However, the $\cf(L(E_8){\downarrow_H})\setminus \{9_{12}^\pm\}$-radical of $P(9_{12})$ is
\[ 1/9_{23}/9_{23}^*,27_3/1,\bar 9_{13}/9_{12}.\]
But on here we cannot place a copy of $9_{12}^*$, so the $\{9_{12}^\pm\}$-heart of $W$ is $9_{12}\oplus 9_{12}^*$. In particular, this means that $9_{12}$ cannot lie in the socle of $W$, which is a final contradiction.

Thus $H$ stabilizes a line on $L(E_8)$, and in particular is strongly imprimitive by Lemma \ref{lem:fix1space}.

\medskip

\noindent $\boldsymbol{q=16}$: Up to field automorphism, there are ten sets of composition factors that are conspicuous for elements of order at most $17$, namely
\[ 9_{12},9_{12}^*,8_1^{25},1^{30},\quad (9_{12},9_{12}^*)^8,8_1^8,8_2^3,1^{16},\quad (\bar 9_{13},\bar 9_{13}^*)^8,9_{23},9_{23}^*,8_1^8,8_3,1^{14},\]
\[ (9_{14},9_{14}^*)^8,\bar 9_{24},\bar 9_{24}^*,8_1^8,8_4,1^{14},\quad 64_{14},(9_{14},9_{14}^*)^6,(\bar 9_{24},\bar 9_{24}^*)^2,8_1^4,1^8,\]
\[27_{123},27_{123}^*,(9_{12},9_{12}^*)^3,(\bar 9_{13},\bar 9_{13}^*)^3,(9_{23},9_{23}^*)^3,8_1,8_2,8_3,1^8,\]
\[ 64_{14},27_{124},27_{124}^*,9_{12},9_{12}^*,(9_{14},9_{14}^*)^3,(\bar 9_{24},\bar 9_{24}^*)^2,8_1^2,1^6,\] \[ 64_{12}^3,9_{12},9_{12}^*,9_{23},9_{23}^*,8_1,8_2,1^4,\quad  64_{13}^3,9_{12},9_{12}^*,9_{34},9_{34}^*,8_1,8_3,1^4,\] \[ 64_{24},27_{124},27_{124}^*,(9_{12},9_{12}^*)^2,9_{14},9_{14}^*,(\bar 9_{24},\bar 9_{24}^*)^3,9_{34},9_{34}^*,1^4.\]
(Here, $27_{a,b,c}$ is the module for $H$ that is the product of $3_a$, $3_b^\pm$ and $3_c^\pm$ (the others are faithful modules for $\SL_3(16)$), so $27_{123}$ is $3_1\otimes 3_2^*\otimes 3_3$ and so on.)

Let $x$ be an element of order $85$ in $H$. Using Lemma \ref{lem:stronglyimppsl316}, if we can find more than two elements of order $255$ cubing to $x$ in $\mb G$ and stabilizing the same subspaces of $L(E_8)$ as $x$, then $H$ is strongly imprimitive.

We can easily check that there are such elements for the first nine cases, leaving only the tenth. In this final case, there are at least eight elements of order $255$ stabilizing the eigenspaces associated to each of $9_{12}^\pm$, $9_{14}^\pm$, $\bar 9_{24}^\pm$ and $9_{34}^\pm$. Thus if $\soc(L(E_8){\downarrow_H})$ consists of modules other than $64_{24}$ and $27_{124}^\pm$, then we may apply Lemma \ref{lem:stronglyimppsl316} to obtain that $H$ is strongly imprimitive. Furthermore, the module $64_{24}$ splits off as a summand because it has no extensions with other composition factors, so the socle of $L(E_8){\downarrow_H}$ must be a sum of $64_{24}$ and (up to duality) $27_{124}$.

Thus we take the $\cf(L(E_8){\downarrow_H})\setminus\{27_{124}^\pm\}$-radical of $P(27_{124})$, and this is
\[\bar 9_{24}/1,9_{14}^*,\bar 9_{24}^*,9_{34}^*/9_{12}^*,\bar 9_{24}/27_{124}.\]
Since this obviously does not have enough trivial composition factors, we must have more modules in the socle, and so $H$ is indeed strongly imprimitive, as needed.
\end{proof}

When $q=3$, the three remaining conspicuous sets of composition factors are warranted. A warrant for the first set of factors is given by a diagonal $\PSL_3(3)$ in $E_6A_2$ acting irreducibly on both $M(A_2)$ and $M(E_6)$, a warrant for the second set of factors is given by a $\PSL_3(3)$ subgroup of $D_8$ acting with factors $7,3,3^*,1^3$ on $M(D_8)$, and a warrant for the third set of factors is given by a copy of $\PSL_3(3)$ inside $A_8$ acting on $M(A_8)$ as $3\oplus 6^*$.

\medskip

For $q=4$, both remaining sets of composition factors are warranted, lying in $A_8$. For the first, act on $M(A_8)$ as $1\oplus 8_1$. For the second, $H$ does not act on $M(A_8)$, but rather $\SL_3(4)$ acts faithfully and irreducibly, which yields a representation of $\PSL_3(4)$ in $\mb G$.

\section{\texorpdfstring{$\PSU_3$}{PSU 3}}
\label{sec:su3ine8}
The next proposition proves almost all that we know about $\PSU_3(q)$. We leave open cases for $q=3,4,8$, two cases for the first two and one for the third. The second open case for $\PSU_3(4)$, that with no trivial factors, is solved in Section \ref{sec:diffpsu34} below, but at the time of writing no solutions exist for the other four sets of composition factors given here. We continue with our definition of $u$ from the start of the chapter.

\begin{proposition}\label{prop:su3ine8} Let $H\cong \PSU_3(q)$ for some $3\leq q\leq 9$ or $q=16$.
\begin{enumerate}
\item If $q=3$ then either $H$ stabilizes a line on $L(E_8)$ or the composition factors of $L(E_8){\downarrow_H}$ are one of
\[(15,15^*)^4,7^8,(6,6^*)^3,(3,3^*)^5,1^6,\quad
27,(15,15^*)^4,7^6,(6,6^*)^2,(3,3^*)^5,1^5.\]
\item If $q=4$ then either $H$ stabilizes a line on $L(E_8)$ or the composition factors of $L(E_8){\downarrow_H}$ are one of
\[ 24_{12},24_{12}^*,(9_{12},9_{12}^*)^3,(\bar 9_{12},\bar 9_{12}^*)^2,8_1^3,8_2^3,(3_1,3_1^*)^5,(3_2,3_2^*)^5,1^2,\]
\[ 24_{12},24_{12}^*,24_{21},24_{21}^*,(9_{12},9_{12}^*)^2,(\bar 9_{12},\bar 9_{12}^*)^2,8_1^2,8_2^2,(3_1,3_1^*)^4,(3_2,3_2^*)^4.\]
\item If $q=5$ then $H$ stabilizes a line on $L(E_8)$.
\item If $q=7$ then $H$ either stabilizes a line on $L(E_8)$ or is a blueprint for $L(E_8)$.
\item If $q=8$ then either $H$ is strongly imprimitive or the composition factors of $L(E_8){\downarrow_H}$ are
\[ 27,27^*,(9_{12},9_{12}^*)^3,(9_{13},9_{13}^*)^3,(9_{23},9_{23}^*)^3,8_1,8_2,8_3,1^8.\]
\item If $q=9,16$ then $H$ is strongly imprimitive.
\end{enumerate}
\end{proposition}
\begin{proof} $\boldsymbol{q=7}$: The fourteen sets of composition factors for $L(E_8){\downarrow_H}$ that are conspicuous for elements of order up to $24$ are
\[
8,(3,3^*)^{27},1^{78},\quad
10,10^*,8^{25},1^{28},\quad
8^8,(6,6^*)^3,(3,3^*)^{21},1^{22},\]
\[15_2,15_2^*,8^9,(6,6^*)^7,(3,3^*)^8,1^{14},\quad 
27,(15_2,15_2^*)^3,(10,10^*)^2,8,(6,6^*)^6,1^{11},\]
\[ 35,35^*,27^6,8,1^8,\quad
37,(35,35^*)^2,27,(10,10^*)^2,1^4,\]
\[27,24,24^*,(15_2,15_2^*)^3,(10,10^*)^2,8^2,6,6^*,(3,3^*)^2,1^3,\quad
27^3,(10,10^*)^5,8^8,1^3,\]
\[117,(28,28^*)^2,8^2,1^3,\quad
(15_2,15_2^*)^3,10,10^*,8^8,(6,6^*)^3,(3,3^*)^6,1^2,\]
\[ 42,42^*,37,27,21,21^*,15_1,15_1^*,10,10^*,3,3^*,1^2,\quad
71,71^*,37,27,(10,10^*)^2,1^2,\]
\[42,42^*,35,35^*,24,24^*,15_2,15_2^*,8^2.\]
As the only simple modules with non-zero $1$-cohomology have dimension $36$ (see Table \ref{t:modules7}), we see that all but the last set of composition factors has negative pressure, so that $H$ stabilizes a line on $L(E_8)$. In the last case, there are no extensions between the factors, so that $L(E_8){\downarrow_H}$ is semisimple, and $u$ acts on $L(E_8)$ with blocks $6^4,5^{10},4^{16},3^{20},2^{20},1^{10}$. Hence $u$ lies in the generic class $2A_2+2A_1$ by \cite[Table 9]{lawther1995}, so $H$ is a blueprint for $L(E_8)$ by Lemma \ref{lem:genericmeansblueprint}.

\medskip

\noindent $\boldsymbol{q=5}$: The nine conspicuous sets of composition factors for $L(E_8){\downarrow_H}$ are
\[ 10,10^*,8^{25},1^{28},\quad 63,19^6,10,10^*,8^5,1^{11},\quad 35,35^*,19^6,8^7,1^8,\]
\[ 125,19^4,8^5,1^7,\quad 63,(35,35^*)^2,(10,10^*)^2,1^5,\quad 19^3,(10,10^*)^5,8^{11},1^3,\]
\[ 63^3,(10,10^*)^2,8^2,1^3,\quad (35,35^*)^3,10,10^*,8^2,1^2,\quad 63^2,35,35^*,10,10^*,8^4.\]
The pressures of these modules are $-28$, $2$, $4$, $1$, $-4$, $3$, $0$, $-2$ and $2$ respectively.

\medskip

\noindent \textbf{Cases 1, 4, 5, 7 and 8}: In the first, fifth, seventh and eighth cases $H$ stabilizes a line on $L(E_8)$. It also does in the fourth case since $H^1(H,19)$ is $2$-dimensional and $L(E_8){\downarrow_H}$ has pressure $1$, by Proposition \ref{prop:pressure}. (See Table \ref{t:modules5}.)

\medskip

\noindent \textbf{Case 3}: Suppose that $H$ does not stabilize a line on $L(E_8)$, and let $W$ denote the $\{19\}$-heart of $L(E_8){\downarrow_H}$. As $W$ has pressure $4$, there are at most two copies of $19$ in $\soc(W)$. The $\{1,8,19,35^\pm\}$-radical of $P(19)$ is
\begin{equation} 8/19/1,1,1,8,8,8,8/19,19,19,19,35,35^*/1,1,8,8,8/19,\label{eq:19radical}\end{equation}
and the sum of two of these must be a pyx for $W$. This has enough trivial factors, but we can also take the $\{19\}'$-residual of this module, and two copies of it still must be a pyx. This residual is
\[ 19/8/19,19,19,19,35,35^*/1,1,8,8,8/19,\]
and two copies of it is clearly no longer a pyx for $W$. Thus $H$ stabilizes a line on $L(E_8)$.

\medskip

\noindent \textbf{Case 2}: Suppose that $H$ does not stabilize a line on $L(E_8)$. The module $10$ only has an extension with $63$ from the composition factors of $L(E_8){\downarrow_H}$, so either $10$ or $10^*$ is a submodule of $L(E_8){\downarrow_H}$. Let $W$ denote the $\{10^\pm,63\}'$-radical of $L(E_8){\downarrow_H}$. This needs to contain at least six trivial composition factors by Proposition \ref{prop:bottomhalf}, and since $L(E_8){\downarrow_H}$ has pressure $2$, we cannot have a subquotient $1^{\oplus 3}$ or $19^{\oplus 2}$ by Proposition \ref{prop:pressure}. Since $W$ has at least six trivial factors and at most two in each socle layer, it must have trivial factors in at least three different socle layers.

By quotienting out by any $8$s, we may assume that $\soc(W)=19$, so a pyx for $W$ is the $\{1,8,19\}$-radical of $P(19)$. This is a submodule of the module in (\ref{eq:19radical}), and this does not have trivial modules in three distinct layers. Thus $H$ stabilizes a line on $L(E_8)$.

\medskip

\noindent \textbf{Case 6}: Suppose that $H$ does not stabilize a line on $L(E_8)$. The composition factors $10$ and $10^*$ have no extensions with $1,8,19$, so $L(E_8){\downarrow_H}$ splits into two summands. Since the pressure of $L(E_8){\downarrow_H}$ is $3$ and $H^1(H,19)$ is $2$-dimensional, we cannot have $19^{\oplus 2}$ as a subquotient of $L(E_8){\downarrow_H}$. Quotient out by the $\{8,10^\pm\}$-radical of $L(E_8){\downarrow_H}$ to leave a module $W$, whose socle must be copies of $19$, and hence simply $19$. The $\{1,8,19\}$-radical of $P(19)$ is
\[ 1,1,1,8,8/19,19,19,19/1,1,8,8,8/19,\]
and again this is a pyx for $W$. However, $W$ cannot have a trivial quotient, so those three trivial quotients can be removed from this and it is still a pyx for $W$. However, $W$ has three trivial factors and the pyx has two, which is a contradiction. Thus $H$ stabilizes a line on $L(E_8)$, as claimed.

\medskip

\noindent \textbf{Case 9}: We show that this does not exist. Let $L$ denote a copy of $\PSL_2(5)$ in $H$: the restrictions of $8$, $10^\pm$, $35^\pm$ and $63$ to $L$ are the sum of a projective module, and $3$, the zero module, the zero module, and $3$ respectively. Hence if $v$ denotes an element of order $5$ in $L$ then there are at least 46 blocks of size $5$ for the action of $v$ on $L(E_8)$, and as the only modules that can be made from $3$ are $3$ and $3/3$, the action of $v$ on $L(E_8)$ is $5^{46+i},3^{6-2i},1^i$ for some $0\leq i\leq 3$. There is no class acting with these blocks in \cite[Table 9]{lawther1995} (see also Table \ref{t:unipe8p5}), hence $H$ cannot embed in $\mb G$ with these factors.

\medskip

\noindent $\boldsymbol{q=3}$: There are fifteen conspicuous sets of composition factors for $L(E_8){\downarrow_H}$:
\[ 7,(3,3^*)^{27},1^{79},\quad 7^{27},3,3^*,1^{53},\quad 7^8,(6,6^*)^3,(3,3^*)^{21},1^{30},\]
\[ 15,15^*,7^9,(6,6^*)^7,(3,3^*)^8,1^{23},\quad
27,15,15^*,7^7,(6,6^*)^6,(3,3^*)^8,1^{22},\] \[27^2,15,15^*,7^{11},(6,6^*)^5,(3,3^*)^2,1^{15},\quad
27,(15,15^*)^3,7^5,(6,6^*)^6,(3,3^*)^2,1^{12},\]
\[27^3,7^{18},(3,3^*)^5,1^{11},
\quad 27^6,15,15^*,7^3,6,6^*,(3,3^*)^2,1^{11},\]
\[(15,15^*)^3,7^{10},(6,6^*)^3,(3,3^*)^7,1^{10},\quad
27,(15,15^*)^3,7^7,(6,6^*)^6,1^{10},\]
\[27^2,(15,15^*)^3,7^5,(6,6^*)^5,1^9,\quad
 27^3,(15,15^*)^3,7^3,(6,6^*)^4,1^8,\]
\[(15,15^*)^4,7^8,(6,6^*)^3,(3,3^*)^5,1^6,\quad
27,(15,15^*)^4,7^6,(6,6^*)^2,(3,3^*)^5,1^5.\]
The pressures of the above modules are $-78$, $-26$, $-16$, $0$, $-3$, $6$, $5$, $7$, $-6$, $6$, $9$, $6$, $3$, $8$ and $5$ respectively.

\medskip

\noindent \textbf{Cases 1, 2, 3, 4, 5 and 9}: For these cases $H$ stabilizes a line on $L(E_8)$ by Proposition \ref{prop:pressure}.

\medskip

The socle layers of $P(6)$ and $P(7)$ are
\[\begin{array}{c}
6
\\1
\\7
\\1
\\6\oplus6^*\oplus15^*
\\1\oplus3
\\6\oplus7
\\1\oplus15^*
\\6
\end{array}\qquad
\begin{array}{c}
7
\\1
\\6\oplus6^*
\\1\oplus3\oplus3^*
\\7\oplus7\oplus7\oplus15\oplus15^*
\\1\oplus1\oplus3\oplus3\oplus3^*\oplus3^*
\\6\oplus6^*\oplus7\oplus7\oplus15\oplus15^*
\\1\oplus3\oplus3^*
\\7
\end{array}
\]
If $W$ denotes the $\{6^\pm,7\}$-heart of $L(E_8){\downarrow_H}$, then $W$ is a submodule of copies of $P(6)$, $P(6^*)$ and $P(7)$. Therefore, trivial composition factors can occur on the second, fourth, sixth and eighth socle layers of $W$, with the eighth layer occupied if and only if there is a projective summand.

\medskip
\noindent \textbf{Cases 11, 12 and 13}: In these cases the module $3$ does not lie in $L(E_8){\downarrow_H}$. The $\{3^\pm\}'$-radicals of $P(6)$ and $P(7)$ are
\[ 1/6,6^*/1/6,7/1,15^*/6,\qquad 7/1/6,6^*/1/7.\]
From this it is easy to see that either $H$ stabilizes a line on $L(E_8)$ or we need at least twice as many copies of $6$, $6^*$ and $7$ as trivial factors in $W$, i.e., the pressure needs to be at least the number of trivial factors. This latter statement does not hold, so $H$ stabilizes a line on $L(E_8)$ in all three cases.

\medskip

\noindent \textbf{Case 8}: We have no $6$ or $15$, so we need the $\{1,3^\pm,7\}$-radical of $P(7)$, which is 
\[ 3,3^*/7,7/1,3,3^*/7.\]
Either $H$ stabilizes a line on $L(E_8)$ or we need at least twice as many $7$s as $1$s. This latter statement is not the case, so $H$ indeed stabilizes a line on $L(E_8)$.

\medskip

We are left with the sixth, seventh, tenth, fourteenth and fifteenth cases, all of which have all modules from the principal block of $H$ as composition factors, so all $\cf(L(E_8){\downarrow_H})$-radicals are the entire projective modules and are therefore not of much use. To help the reader, we repeat them:
\[  27^2,15,15^*,7^{11},(6,6^*)^5,(3,3^*)^2,1^{15},
\quad 27,(15,15^*)^3,7^5,(6,6^*)^6,(3,3^*)^2,1^{12},\]
\[ (15,15^*)^3,7^{10},(6,6^*)^3,(3,3^*)^7,1^{10},\quad(15,15^*)^4,7^8,(6,6^*)^3,(3,3^*)^5,1^6,\]
\[27,(15,15^*)^4,7^6,(6,6^*)^2,(3,3^*)^5,1^5.\]
In these cases we let $M$ denote a copy of $\PSL_3(2)$ in $H$. All simple modules for $H$ other than $1$ and $7$ restrict projectively to $M$, and these two restrict irreducibly. The projective covers have structures $1/7/1$ and $7/1/7$, so it is easy to determine the possible actions of $M$ on $L(E_8)$ given the unipotent class to which $v\in M$ of order $3$ belongs.

We also require the permutation module $P_M$ on the cosets of $M$, which has structure
\[ 1/6,6^*/((1/7/1)\oplus 1)/6,6^*/1.\]
By Frobenius reciprocity, the spaces $L(E_8)^M$ and $\Hom_{kH}(P_M,L(E_8))$ have the same dimension, hence we are interested in the quotients of $P_M$. Assuming that $H$ does not stabilize a line on $L(E_8)$, we have quotients
\[ 1/6^\pm,\quad 1/6,6^*,\quad 1/6,6^*/1/7,\quad 1/6,6^*/1/7/1/6^\pm,\quad P_M/\soc(P_M).\]
The space of homomorphisms from $P_M$ to these quotients has dimension $1$, $1$, $1$, $2$ and $3$ respectively.

We claim that the dimension of $L(E_8)^M$ is at most the multiplicity of $6$ as a composition factor of $L(E_8){\downarrow_H}$. To see this, first note that from the structure of $kM$-modules, we see that if we quotient out a $kM$-module by a trivial submodule, the fixed-point space of it always decreases by exactly $1$.

\medskip

Let $V$ be a self-dual $kH$-module such that $V^H=0$. Let $U_1\cong 1/6,6^*/1/7$, $U_2\cong 1/6,6^*/1/7/1/6$ and $U_3$ be the quotient $P_M/\soc(P_M)$. Note that $\Ext_{kH}^1(1,U_i)=0$ and $\Ext_{kH}^1(6^\pm,U_i)=0$ for all $i=1,2,3$, so if $U_i$ is a submodule of $V$, then every submodule $1$ or $6^\pm$ of the quotient $V/U_i$ comes from a $1$ or $6^\pm$ submodule of $V$.

Consider the quotient module $V/U_i$. We see that the dimension of $(V/U_i)^M$ is $\dim(V^M)-\dim(U_i^M)$, and the number of factors $6$ in $V/U_i$ is the number of those in $V$ minus the number in $U_i$. However, if the socle of $U_i$ contains $6^*$, then there is a corresponding quotient $6$ of $V$ that may be removed as well. Thus we obtain a module $V'/U_i$ such that the dimension of $(V'/U_i)^M$ is at most the number of copies of $6$ in $V'/U_i$.

We therefore quotient out by any copies of $U_i$ (and remove any necessary $6$s from the top of $V$) to obtain a new module $V_1$, which is no longer self-dual but we may still repeat the process. If $U_4\cong 1/6^\pm$ then we can quotient out by these, and again remove $6$s from the top for every $1/6^*$ in the socle, as again $\dim(\Ext_{kH}^1(1,U_4))=\dim(\Ext_{kH}^1(6^\pm,U_4))=0$. The last remaining case is $U_5\cong 1/6,6^*$, which of course now has $\dim(\Ext_{kH}^1(1,U_5))=1$, but removing a single copy of this from the socle of $V_1$ both reduces $\dim(V_1^M)$ by $1$ and removes a copy of $6$ from $V_1$, without having to worry about any new trivial submodules. (Indeed, if a new trivial submodule appears, then there was a submodule $(1/6)\oplus (1/6^*)$ in the first place, which would have been removed at the previous stage.)

The conclusion is that one can continually remove submodules until $\dim(V^M)=0$, and at each stage remove a single $6$ as well. Thus the number of $6$s must be at least $\dim(V^M)$, as claimed.

\medskip

We therefore have that if the dimension of $L(E_8)^M$ is greater than half of the number of $6$-dimensional factors (equivalently greater than the number of copies of $6$ that are factors) then $L(E_8)^H\neq 0$, as needed. For most choices of the set of composition factors and unipotent class of $v$, this is impossible, as we will show.

\medskip

\noindent \textbf{Case 6}: Excluding those projective summands that arise from composition factors of $L(E_8){\downarrow_H}$, the composition factors of $L(E_8){\downarrow_M}$ are $7^{11},1^{15}$. The projectives that arise from composition factors have dimension $156$, so yield blocks $3^{52}$. Each copy of $7$ yields another two blocks of size $3$, so there are at least $52+22=74$ blocks of size $3$ in the action of $v$ on $L(E_8)$. Thus $v$ lies in one of $2A_2$ (acts on $L(E_8)$ with blocks $3^{78},1^{14}$), $2A_2+A_1$ (acts with blocks $3^{79},2^2,1^7$) and $2A_2+2A_1$ (acts with blocks $3^{80},2^4$), and the dimension of the fixed-point space $L(E_8)^M$ in the three cases must be $11$, $9$ and $7$ respectively.

(For example, if $v$ lies in $2A_2+A_1$ then the composition factors $7^{11},1^{15}$ must form a module on which $v$ acts with blocks $3^{27},2^2,1^7$. The two blocks of size $2$ must arise from $(7/1)\oplus (1/7)$, and the seven blocks of size $1$ arise from seven simple summands. The rest must form five projective summands. Each projective $1/7/1$ or $7/1/7$ has a single trivial composition factor not in the socle, and $1/7$ gives another trivial not in the socle, so there are six not in the socle. Thus regardless of the exact projectives involved, the dimension of $L(E_8)^M$ is always $9$, as there are fifteen trivials in total. A similar argument works for the other two classes.)

Thus there are only five copies of $6$ in $L(E_8){\downarrow_H}$, but the dimension of $L(E_8)^M$ is at least $7$. Hence by the argument above, $H$ stabilizes a line on $L(E_8)$.

\medskip

\noindent \textbf{Case 7}: This is very similar; instead of $7^{11},1^{15}$ we have $7^5,1^{12}$. If $v$ lies in class $2A_2$ or $2A_2+A_1$ then $\Hom_{kH}(V,L(E_8){\downarrow_H})$ has dimension $11$ or $9$ respectively. Notice that $v$ cannot belong to class $2A_2+2A_1$ as the $2^4$ contribution to $v$ is $(1/7)^{\oplus 2}\oplus (7/1)^{\oplus 2}$, and then we must construct a projective module from $7,1^8$, which is not possible. There are six copies of $6$ in $L(E_8){\downarrow_H}$, so $H$ must stabilize a line on $L(E_8)$.

\medskip

\noindent \textbf{Case 10}: The proof is again very similar. We have $7^{10},1^{10}$, and the dimension of $\Hom_{kH}(V,L(E_8){\downarrow_H})$ is $8$, $6$ and $4$ if $v$ lies in class $2A_2$, $2A_2+A_1$ and $2A_2+2A_1$ respectively. This time there are three copies of $6$ in $L(E_8){\downarrow_H}$, and thus $H$ stabilizes a line on $L(E_8)$ again.

\medskip

\noindent\textbf{Cases 14 and 15}: These are left over, and are stated as in the proposition.

\medskip

\noindent $\boldsymbol{q=9}$: There are, up to field automorphism, 28 sets of composition factors for $L(E_8){\downarrow_H}$ that are conspicuous for elements of order at most $20$. Letting $x$ denote an element of order $80$ in $H$, note that the eigenvalues of $x$, $x^2$ and $x^4$ on $L(E_8)$ determine the semisimple class in $\mb G$, by a computer check.

We will apply Lemma \ref{lem:allcasesstronglyimp} to prove that some cases are strongly imprimitive: for thirteen of the 28 cases, we actually find elements of order $160$ that have the same number of distinct eigenvalues on $L(E_8)$ as $x$, so $H$ is certainly strongly imprimitive in these cases.

The remaining fifteen cases are as follows:
\[ 18_{21},18_{21}^*,(9_{12},9_{12}^*)^7,7_1^8,7_2,1^{23},\quad \bar{18}_{21},\bar{18}_{21}^*,(\bar 9_{12},\bar 9_{12}^*)^7,7_1^8,7_2,1^{23},\]
\[ 21_{21},21_{21}^*,(9_{12},9_{12}^*)^3,(\bar 9_{12},\bar 9_{12}^*)^3,7_1^2,7_2,(6_1,6_1^*)^3,(3_1,3_1^*)^3,(3_2,3_2^*)^2,1^{11},\]
\[ 49,(21_{12},21_{12}^*)^3,7_1,7_2^2,(6_1,6_1^*)^3,3_1,3_1^*,1^{10},\quad 49^3,7_1^6,7_2^6,3_1,3_1^*,3_2,3_2^*,1^5,\]
\[ 45_{12},45_{12}^*,27_2,(18_{12},18_{12}^*)^2,7_1,7_2^5,(3_1,3_1^*)^2,1^5,\] \[\bar{45}_{12},\bar{45}_{12}^*,27_2,(\bar{18}_{12},\bar{18}_{12}^*)^2,7_1,7_2^5,(3_1,3_1^*)^2,1^5,\]
\[\begin{split}21_{21},21_{21}^*,18_{21},18_{21}^*,\bar{18}_{21},\bar{18}_{21}^*,&15_1,15_1^*,9_{12},9_{12}^*,\bar 9_{12},\bar 9_{12}^*,
\\&7_1^3,7_2,6_1,6_1^*,(3_1,3_1^*)^2,(3_2,3_2^*)^2,1^4,\end{split}\]
\[ \begin{split}21_{12},21_{12}^*,21_{21},21_{21}^*,18_{12},18_{12}^*,\bar{18}_{21},\bar{18}_{21}^*,&9_{12},9_{12}^*,\bar 9_{12},\bar 9_{12}^*,\\&7_1^2,7_2^2,(3_1,3_1^*)^2,(3_2,3_2^*)^2,1^4,\end{split}\]
\[ \begin{split}21_{12},21_{12}^*,21_{21},21_{21}^*,18_{21},18_{21}^*,\bar{18}_{12},\bar{18}_{12}^*,&9_{12},9_{12}^*,\bar 9_{12},\bar 9_{12}^*,\\&7_1^2,7_2^2,(3_1,3_1^*)^2,(3_2,3_2^*)^2,1^4,\end{split}\]
\[ 21_{12},21_{12}^*,\bar{18}_{12},\bar{18}_{12}^*,(15_2,15_2^*)^3,(9_{12},9_{12}^*)^2,7_2^4,(3_1,3_1^*)^2,1^4,\]
\[ 45_{12},45_{12}^*,27_2,21_{12},21_{12}^*,(\bar 9_{12},\bar 9_{12}^*)^2,7_2,(6_2,6_2^*)^3,3_1,3_1^*,1^4,\]
\[ 49,42_{12},42_{12}^*,21_{12},21_{12}^*,15_1,15_1^*,7_1^2,7_2^2,(3_1,3_1^*)^2,1^3,\]
\[ 49,45_{12},45_{12}^*,18_{12},18_{12}^*,18_{21},18_{21}^*,7_1,7_2^3,3_1,3_1^*,1^3,\]
\[49,\bar{45}_{12},\bar{45}_{12}^*,\bar{18}_{12},\bar{18}_{12}^*,\bar{18}_{21},\bar{18}_{21}^*,7_1,7_2^3,3_1,3_1^*,1^3.\]

The pressures of these are $-14$, $-14$, $-6$, $-1$, $7$, $1$, $1$, $2$, $4$, $4$, $2$, $-1$, $3$, $1$ and $1$ respectively.
\medskip

\noindent \textbf{Cases 1, 2, 3, 4 and 12}: In these cases $H$ stabilizes a line on $L(E_8)$ by Proposition \ref{prop:pressure}.

\medskip

\noindent \textbf{Cases 6, 7, 14 and 15}: In these cases the pressure is $1$. In each case, $H$ stabilizes a line on $L(E_8)$ by Corollary \ref{cor:trivialoddmult}. 

\medskip

\noindent \textbf{Cases 5, 8, 9, 10, 11 and 13}: For the other cases, the quickest way to prove the result is to find all composition factors whose constituent eigenspaces for the action of $x$ are stabilized by elements of order $160$ that square to $x$. (Since semisimple elements in $E_8$ are real, if $M$ is stabilized so is $M^*$.) We obtain the following table.

\begin{center}
\begin{tabular}{cc}
\hline Case number & Modules stabilized by an element
\\ \hline $5$ & $1$, $3_1^\pm$, $3_2^\pm$, $7_1$, $7_2$
\\ $8$ & $1$, $3_1^\pm$, $3_2^\pm$, $6_1^\pm$, $7_1$, $7_2$, $15_1^\pm$
\\ $9$ & $1$, $3_1^\pm$, $3_2^\pm$, $7_1$, $7_2$
\\ $10$ & $1$, $3_1^\pm$, $3_2^\pm$, $7_1$, $7_2$, $9_{12}^\pm$, $\bar 9_{12}^\pm$, $21_{12}^\pm$, $21_{21}^\pm$
\\ $11$ & $1$, $3_1^\pm$, $7_2$, $9_{12}^\pm$, $15_2^\pm$
\\ $13$ & $1$, $3_1^\pm$, $7_1$, $7_2$, $21_{12}^\pm$
\\\hline\end{tabular}
\end{center}
In Case $5$, there are, in fact, sixteen elements of order $160$ that stabilize all of the eigenspaces comprising $3_1^\pm$, $3_2^\pm$, $7_1$ and $7_2$, so we even obtain $N_{\Aut^+(\mb G)}(H)$-stability if one of these is a submodule of $L(E_8){\downarrow_H}$. However, in Cases $9$ and $10$ this is not the case, so we will have to be more careful.

In the other cases, if any of the modules above lie in the socle of $L(E_8){\downarrow_H}$, then $H$ is strongly imprimitive by Lemma \ref{lem:allcasesstronglyimp}.

\medskip

\noindent \textbf{Case 5}: The socle of $L(E_8){\downarrow_H}$ must be $49$ if $H$ is not strongly imprimitive. We need the $\{1,3_i^\pm,7_i\}$-radical of $P(49)$, which is
\[ 1,3_1,3_1^*,3_2,3_2^*/7_1,7_2/49.\]
Clearly we cannot produce a module with socle $49$ and six copies of $7_1$ in it, so $H$ is strongly imprimitive in this case.

\medskip

\noindent \textbf{Case 8}: The only extensions that involve $9_{12}^\pm$ or $6_1^\pm$, and other composition factors of $L(E_8){\downarrow_H}$, are $\Ext_{kH}^1(6_1,9_{12})$ (and its cognates such as $6_1^*,9_{12}^*$). Hence the factors $9_{12}$ and $6_1$ break off as a separate summand. Thus, in particular, either $6_1$ or $6_1^*$ is a submodule of $L(E_8){\downarrow_H}$, and hence $H$ is strongly imprimitive.

\medskip

\noindent \textbf{Cases 9 and 10}: We first show that $H$ is Lie imprimitive (we cannot show strong imprimitivity directly because we cannot apply Lemma \ref{lem:allcasesstronglyimp} with an $N_{\Aut^+(\mb G)}(H)$-orbit of spaces), and then use the collection $\ms X$ to show that $H$ in fact stabilizes a line on $L(E_8)$.

In Case 9, the $9$-dimensional factors must split off as summands because they have no extensions with the other factors. For the remaining summand, the socle must consist of at most one copy each of $18_{12}^\pm$, $\bar{18}_{21}^\pm$, $21_{12}^\pm$ and $21_{21}^\pm$, else we may apply Lemma \ref{lem:allcasesstronglyimp} to obtain that $H$ is Lie imprimitive.

The $\cf(L(E_8){\downarrow_H})\setminus\{V^\pm\}$-radicals of $P(V)$, for $V$ each of $18_{12}$, $\bar{18}_{21}$, $21_{12}$ and $21_{21}$, are
\[3_2/21_{21}/3_2/18_{12}\quad 3_1^*/21_{12}^*/3_1^*/\bar{18}_{21},\]
\[ 1/7_1,7_2,\bar{18}_{21}^*/1,3_1,3_1^*/21_{12},\quad 1,7_1,7_2,18_{12}/1,3_2,3_2^*/21_{21}.\]
As $H$ acts on $L(E_8)$ with eight $3$-dimensional factors we need all four of these modules (up to duality) in the socle of $L(E_8){\downarrow_H}$. However, then the $21_{21}$ appearing in the third layer of the first module cannot appear in $L(E_8){\downarrow_H}$, and we consequently lose a $3$-dimensional factor. Since we now have only seven in total, there must be another module in the socle, and $H$ is Lie imprimitive.

\medskip

Case 10 is easier. There are no extensions between the elements of the set $\{9_{12}^\pm,\bar 9_{12}^\pm,18_{21}^\pm,\bar{18}_{12}^\pm\}$ and the other composition factors of $L(E_8){\downarrow_H}$, and therefore there is a summand of the module that has none of these as composition factors. The others are all in the table above, and therefore $H$ is at least Lie imprimitive by Lemma \ref{lem:allcasesstronglyimp}.

\medskip

We now prove that in fact $H$ always stabilizes a line on $L(E_8)$, and is therefore strongly imprimitive by Lemma \ref{lem:fix1space}. We will use the fact that $H$ already lies in a member of $\ms X$.

If $H$ lies in $E_6A_2$ then $H$ stabilizes a line on $L(E_8)$, as $E_6A_2$ does (as $p=3$). Similarly, if $H$ lies in an $E_7$-parabolic subgroup then $H$ stabilizes a line on $L(E_8)$. Since $H$ does not embed in $A_1$, if $H$ lies in $A_1E_7$ then $H$ lies in $E_7$ and again stabilizes a line on $L(E_8)$.

Suppose that $H$ embeds in $A_8$. Note that it is easy to write down all composition factors of $kH$-modules of dimension at most $9$, and therefore compute the factors of $H$ on $L(E_8)$, which splits as $L(\lambda_3)$, $L(\lambda_6)=L(\lambda_3)^*$ and $L(\lambda_1+\lambda_8)$. The dimensions of the factors of $H$ on $M(A_8)$ are one of: $9$; $7,1^2$; $6,3$; $6,1^3$; $3^3$; $3^2,1^3$; $3,1^6$. In each case there are multiple possible factors, but it is easy to compute their factors on $L(E_8)$ and none matches Cases 9 or 10. (For example, if $H$ acts irreducibly on $M(A_8)$ then we obtain Case 5, if it acts as $3_1\oplus 6_2$  then we obtain Case 7, and if it acts as $3_1\oplus 3_1^*\oplus 3_2$ then we obtain Case 3.)

If $H$ lies in $A_4A_4$ then it must act as $3_i^\pm\oplus 1^{\oplus 2}$ on each $M(A_4)$, which pushes $H$ into a subgroup of type $A_2A_2$ that has the same composition factors on $L(E_8)$ as one lying in $A_8$. In particular $H$ cannot embed in $A_4A_4$ either.

If $H$ embeds in $D_8$ then there are fewer possible actions on $M(D_8)$ since it is self-dual. The dimension sets are: $7^2,1^2$; $6^2,1^4$; $3^4,1^4$; $3^2,1^{10}$. If $1^4$ lies in $M(D_8){\downarrow_H}$ then there are at least six trivial factors on $\Lambda^2(M(D_8)){\downarrow_H}$, which is a summand of $L(E_8){\downarrow_H}$, and so we cannot have those cases. The actions $7_1^{\oplus 2}\oplus 1^{\oplus 2}$ and $7_1\oplus 7_2\oplus 1^{\oplus 2}$ have $27_1$ and $49$ respectively in their exterior squares, so in fact $H$ cannot embed in $D_8$ either.

This is enough to eliminate all parabolic subgroups: $E_7$ has already been considered; for $A_1E_6$, $H$ lies in an $E_6$-parabolic, hence an $E_7$-parabolic; $H$ cannot lie in a $D_7$-parabolic as $H$ has the same factors as a subgroup $D_8$; similarly, $H$ cannot lie in $A_3A_4$- or $A_1A_2A_4$-parabolics because $H$ would have the same factors as a subgroup of $A_4A_4$; $H$ cannot lie in an $A_7$- or $A_1A_6$-parabolic subgroup as $H$ would have the same composition factors as a subgroup of $A_8$; $H$ cannot lie in an $A_2D_5$-parabolic subgroup because $H$ would have the same composition factors as a subgroup of $A_3D_5$, which is contained in $D_8$.

Suppose next that $H$ is contained in the $F_4$ maximal subgroup $\bX$. The easiest method here appears to be to check which subsets of the composition factors form a $52$-dimensional module that is conspicuous for elements of order at most $8$ for $F_4$. The only sets have dimensions $18^2,7^2,1^2$, and there are several of these. Of course, this has pressure $0$, so $H$ must stabilize a line on the $52$-dimensional summand $L(F_4)$ of $L(E_8){\downarrow_\bX}$, as needed.

The only remaining subgroup is $G_2F_4$. In characteristic $3$, this has a structure 
\[((10,0000)/(10,0001),(01,0000)/(10,0000))\oplus (00,1000).\]
(See, for example, \cite[Table 10.1]{liebeckseitz2004}.) Suppose that $H$ acts as $7_1$ on $M(G_2)$. We can consider the allowed composition factors of $M(F_4){\downarrow_H}$ by seeing which, when tensored by $7_1$, appear on our list (up to field automorphism, since we have fixed $7_1$). Indeed, only $1$ and $3_2$ are allowed, and so $H$ cannot act as $7_1$ on $M(G_2)$. The only other possibility is that $H$ acts as $3_1\oplus 3_1^*\oplus 1$, in which case $H$ stabilizes a line on $L(E_8)$, as claimed.

\medskip

\noindent \textbf{Case 11}: The only extensions between any of $9_{12}$, $15_2$ and $\bar{18}_{12}^*$ and other composition factors of $L(E_8){\downarrow_H}$ are $\Ext_{kH}^1(9_{12},15_2)$ and $\Ext_{kH}^1(9_{12},\bar{18}_{12}^*)$, both of which are $1$-dimensional. Thus, as in Case 8, either $9_{12}$ or its dual is a submodule of $L(E_8){\downarrow_H}$, and $H$ is strongly imprimitive.

\medskip

\noindent \textbf{Case 13}: If $H$ is not strongly imprimitive then the socle of $L(E_8){\downarrow_H}$ consists solely of copies of $15_1^\pm$, $42_{12}^\pm$ and $49$. Note that $15_1$ has extensions with only $42_{12}^*$ and $49$ from the composition factors of $L(E_8){\downarrow_H}$, so if $15_1^\pm$ does lie in the socle we may quotient out by it without altering whether $L(E_8){\downarrow_H}$ only has allowed modules in the socle. Similarly, if $49$ lies in the socle then it is a summand, so in constructing a pyx for $L(E_8){\downarrow_H}$, we may assume that it is a submodule of $P(42_{12})$. We therefore require the $\cf(L(E_8){\downarrow_H})\setminus\{42_{12}^\pm\}$-radical of $P(42_{12})$, which is
\[ 1/3_1^*,7_1,15_1^*/42_{12}.\]
Clearly this cannot work, as it has no $7_2$ for example, and therefore $H$ is strongly imprimitive.

This completes the proof that $H$ is always strongly imprimitive, as needed.

\medskip

\noindent $\boldsymbol{q=4}$: There are 29 conspicuous sets of composition factors for $L(E_8){\downarrow_H}$, of which fourteen have positive pressure or no trivial composition factors, eight up to automorphism of $H$. Note that---with our labelling convention from Chapter \ref{chap:labelling}---the graph automorphism acts as an automorphism of order $4$,
\[ 9_{12}\to \bar 9_{12}^*\to 9_{12}^*\to \bar 9_{12}.\]
The eight sets of composition factors are
\[ (9_{12},9_{12}^*)^8,8_1,8_2^8,(3_1,3_1^*)^2,3_2,3_2^*,1^{14},\]
\[ 24_{21},24_{21}^*,(9_{12},9_{12}^*)^3,(\bar 9_{12},\bar 9_{12}^*)^3,8_1^2,8_2,(3_1,3_1^*)^6,(3_2,3_2^*)^4,1^8,\]
\[ 64,(9_{12},9_{12}^*)^6,8_2^4,(3_1,3_1^*)^4,(3_2,3_2^*)^2,1^8,\]
\[ 64,24_{21},24_{21}^*,(9_{12},9_{12}^*)^3,\bar 9_{12},\bar 9_{12}^*,8_2^2,(3_1,3_1^*)^4,(3_2,3_2^*)^3,1^6,\]
\[ 9_{12},9_{12}^*,(\bar 9_{12},\bar 9_{12}^*)^3,8_1^8,(3_1,3_1^*)^{12},(3_2,3_2^*)^6,1^4,\]
\[ 24_{12},24_{12}^*,24_{21},24_{21}^*,(9_{12},9_{12}^*)^2,(\bar 9_{12},\bar 9_{12}^*)^2,8_1^4,8_2^4,3_1,3_1^*,3_2,3_2^*,1^4,\]
\[ 24_{12},24_{12}^*,(9_{12},9_{12}^*)^3,(\bar 9_{12},\bar 9_{12}^*)^2,8_1^3,8_2^3,(3_1,3_1^*)^5,(3_2,3_2^*)^5,1^2,\]
\[ 24_{12},24_{12}^*,24_{21},24_{21}^*,(9_{12},9_{12}^*)^2,(\bar 9_{12},\bar 9_{12}^*)^2,8_1^2,8_2^2,(3_1,3_1^*)^4,(3_2,3_2^*)^4.\]
Let $W$ denote the $\{3_i^\pm,24_{i,j}^\pm\}'$-heart of $L(E_8){\downarrow_H}$, and $W_0$ the $\{3_i^\pm,8_i,24_{i,j}^\pm\}'$-heart of $L(E_8){\downarrow_H}$. Of course, since the only simple modules with non-zero $1$-cohomology are $9_{12}^\pm$ and $\bar 9_{12}^\pm$ (see \cite{sin1992}) these contain all trivial modules, and have a trivial submodule if and only if $L(E_8){\downarrow_H}$ does. Let $L\cong \Alt(5)\times 5$ be a subgroup of $H$, and write $z$ for a central element of order $5$ in $L$.

We assume in each case that $H$ does not stabilize a line on $L(E_8)$ and derive a contradiction, or the case is left over and stated in the proposition.

\medskip

\noindent \textbf{Case 1}: The pressure is $2$, but the $\{1,3_i^\pm,8_i,9_{12}^\pm\}$-radical of $P(9_{12})$ has only two trivial composition factors. Thus since $\soc(W_0)$ can have at most two $9$s, we cannot build a pyx for $W$. Hence $H$ stabilizes a line on $L(E_8)$ by Proposition \ref{prop:pressure}.

\medskip

\noindent \textbf{Case 3}: The $64$ splits off as it is projective, and $L(E_8){\downarrow_H}$ has pressure $4$. We compute the $\{9_{12}^\pm\}'$-residual of the $\{1,3_i^\pm,8_2,9_{12}^\pm\}$-radical of $P(9_{12})$. This is
\[ 9_{12}/8_2/9_{12}^*,9_{12}^*/3_1,8_2/3_2^*,9_{12}/1,3_1,3_1^*/3_2,9_{12},9_{12}^*/1,3_1^*,8_2/9_{12}.\]
As there are only two $1$s in this module, and eight in $W_0$, we must have four $9$s in the socle of $W_0$. In particular, all eight $1$s in the sum of four modules of the form above are required, so we take the $\{1\}'$-residual of this module, as all of that module is a submodule of $W_0$. This residual is
\[ 1/9_{12},9_{12}^*/1,8_2/9_{12}.\]
In particular, this means that we need eight $9$s in the third socle layer of $W_0$, but $W_0$ only has pressure $4$. Thus $H$ stabilizes a line on $L(E_8)$ by Proposition \ref{prop:pressure}.

\medskip

\noindent \textbf{Case 4}: Again, the $64$ splits off, and this time $W_0$ has pressure $2$ and six trivial factors. The $\{8_1,24_{12}^\pm\}'$-radicals of $P(9_{12})$ and $P(\bar 9_{12})$ have five and three trivial factors respectively, so $\soc(W_0)$ cannot be simple.

If $\bar 9_{12}$ (up to duality) appears in the socle of $W_0$, then it cannot appear in the heart of it. In this case, we replace the $\{8_1,24_{12}^\pm\}'$-radical of $P(9_{12})$ by the $\{1,3_i^\pm,8_2,9_{12}^\pm\}$-radical of $P(9_{12})$, which possesses only two trivial factors (see Case 3), another contradiction. Thus $\soc(W_0)$ is (up to duality) $9_{12}\oplus 9_{12}^\pm$.

Suppose first that the socle of $W_0$ is $9_{12}^{\oplus 2}$. Thus the top of $W_0$ is $(9_{12}^*)^{\oplus 2}$, so we may take the $\{9_{12}^*\}'$-residual of the previous radical and still produce a pyx for $W_0$ (from two copies of it). This residual only contains three trivial factors, so all six of the trivial factors from the pyx lie in $W_0$. Thus we take the $\{1\}'$-residual, as in Case 3, and find a guaranteed submodule of $W_0$. This is
\[ 1/1,9_{12}/3_1^*,\bar 9_{12},\bar 9_{12}^*/3_2,3_2^*/1,3_1^*/9_{12}.\]
Of course, the subquotient $9_{12}\oplus \bar 9_{12}\oplus \bar 9_{12}^*$ of this already has pressure $3$, so since $W_0$ has pressure $2$, $H$ must stabilize a line on $L(E_8)$ by Proposition \ref{prop:pressure}. Thus $\soc(W_0)$ must be $9_{12}\oplus 9_{12}^*$.

There are five trivial composition factors in the $\{8_1,24_{12}^\pm\}'$-radical $W'$ of $P(9_{12})$, in the second, fourth, fifth, sixth and eighth socle layers. We need to prove that three of these cannot lie in $W_0$ to obtain a contradiction.

Since there are two $9_{12}^\pm$ in both the socle and the top of $W_0$, that leaves only one more in the heart of it. Thus we replace our $\{8_1,24_{12}^\pm\}'$-radical of $P(9_{12})$ by the submodule obtained by adding all copies of $1$, $3_i^\pm$, $\bar 9_{12}^\pm$ and $24_{21}^\pm$ on top of $9_{12}$, then either $9_{12}$ or $9_{12}^*$ (but not both), then $1$, $3_i^\pm$, $\bar 9_{12}^\pm$ and $24_{21}^\pm$, then the other from $9_{12}$ and $9_{12}^*$, then $1$, $3_i^\pm$, $\bar 9_{12}^\pm$ and $24_{21}^\pm$. Doing so yields a module with four trivial factors, not five. Furthermore, in both cases there is a $1^{\oplus 3}$ subquotient. Since the pyx for $\rad(W_0)$ is the sum of this and its image under the graph automorphism, there is a subquotient $1^{\oplus 6}$ in this, so we need to remove another four trivials, so two from each summand. This means we are left with four trivial factors from the two summands combined, and that is too few for $W_0$. Thus $H$ stabilizes a line on $L(E_8)$.

\medskip

\noindent \textbf{Cases 2, 5 and 6}: These cases, and Case 8, which is delayed until Section \ref{sec:diffpsu34}, require the subgroup $L$, the centralizer of an element $z$ of order $5$ in $H$. We split $L(E_8)$ according to the eigenspaces of the action of $z$, as $L(E_8){\downarrow_L}$ splits into five summands, one for each eigenvalue. Each eigenspace is then a module for the subgroup $L'\cong \Alt(5)$, so we can use the simple $k\Alt(5)$-modules $1$, $2_1$, $2_2$ and $4$ to describe it. Obviously we have to make a choice as to which eigenvalue is labelled by the fifth root of unity $\zeta$, and we do so in such a way as to make the restrictions to $L'$ as in Table \ref{t:actioneigpsu34}.
\begin{table}
\begin{center}
\begin{tabular}{cccccc}
\hline Module & $1$-eig. & $\zeta$-eig. & $\zeta^2$-eig. & $\zeta^3$-eig. & $\zeta^4$-eig.
\\\hline $1$ & $1$ &&&&
\\ $3_1$ & & $1$ & $2_1$ &&
\\ $3_1^*$ & & & &$2_1$&$1$
\\ $3_2$ & & & $1$ && $2_2$
\\ $3_2^*$ && $2_2$ & & $1$ &
\\ $8_1$ & $1/2_2/1$ & $2_1$ &&& $2_1$
\\ $8_2$ & $1/2_1/1$ & & $2_2$ & $2_2$ &
\\ $9_{12}$ & $2_2$ & $4$ & & $1$ & $2_1$
\\ $9_{12}^*$ & $2_2$ & $2_1$ & $1$& & $4$
\\ $\bar 9_{12}$ & $2_1$ & & $2_2$ & $4$ & $1$
\\ $\bar 9_{12}^*$ & $2_1$ & $1$ & $4$ & $2_2$ &
\\ $24_{12}$ & $4$ & $1/2_1/1$ & $P(2_1)$ & $2_2$ & $2_2\oplus 4$
\\ $24_{12}^*$ & $4$ & $2_2\oplus 4$ & $2_2$ & $P(2_1)$ & $1/2_1/1$
\\ $24_{21}$ & $4$ & $2_1$ & $1/2_2/1$ & $2_1\oplus 4$ & $P(2_2)$
\\ $24_{21}^*$ & $4$ & $P(2_2)$ & $2_1\oplus 4$ & $1/2_2/1$ & $2_1$
\\ \hline
\end{tabular}
\end{center}
\caption{Actions of $L'\cong \Alt(5)$ on each $\zeta^i$-eigenspace of the action of $z$}
\label{t:actioneigpsu34}
\end{table}

The permutation module $P_L$ on the cosets of $L$ is of the form
\[ \begin{split}(1,8_1,8_2/9_{12},9_{12}^*,\bar 9_{12}&,\bar 9_{12}^*/3_1,3_1^*,3_2,3_2^*/3_1,3_1^*,3_2,3_2^*
\\ &/1,1,3_1,3_1^*,3_2,3_2^*/9_{12},9_{12}^*,\bar 9_{12},\bar 9_{12}^*/1,8_1,8_2)\oplus 64,\end{split}\]
but (ignoring the $64$) the only quotient of this not involving either $8_i$ at the top is simply $1$ itself. The largest quotient not involving $8_2$ is
\[ 1,8_1/\bar 9_{12},\bar 9_{12}^*/3_2,3_2^*/1,3_1,3_1^*,\]
so the image in $W$ is a quotient of $1,8_1/\bar 9_{12},\bar 9_{12}^*$ and one may check that all such quotients have $L$-fixed points of dimension $1$.

In fact, this is generally true. One may construct \emph{all} 34 (non-zero) quotients of $P_L$ with only $9$-dimensional factors in the socle, and find that they all have $1$-dimensional $L$-fixed points. To compute all such quotients, first quotient out by the $1$, $8_1$ and $8_2$. Then choose one of the $9$-dimensional modules to quotient out by, and then also quotient out by any other simple submodules not of dimension $9$. Repeatedly do this, and at each stage it turns out that one never has two copies of the same $9$-dimensional module in the socle of the module. Thus there is an explicit finite list of such quotients (even over the algebraically closed field), and they can all be restricted to $L$. It turns out that each of these has $L$-fixed points of dimension $1$. Furthermore, every such quotient has either two or six socle layers, and if it has six then the top is always $1\oplus 8_1\oplus 8_2$, and all $8_i$ lie in the fifth or sixth socle layer.

\medskip

From this we can prove that, if $V$ is a module with $V^H=0$, then there must be at least as many $8$-dimensional composition factors in $V$ as the dimension $n$ of $V^L$, using induction on $n$ and noting that the case $n=1$ is above. If $V$ has an $8$-dimensional submodule then quotient it out: both the dimension of $V^L$ and the number of $8_i$ in $V$ decrease by $1$, so we may assume that $V$ has no $1$- or $8$-dimensional submodules. We can also remove any $3$-dimensional submodules without affecting either quantity, so we assume that the socle of $V$ has factors all of dimension $9$.

We may replace $V$ by the $kH$-submodule generated by $V^L$ without altering the truth of the statement. Thus $V$ is a (not necessarily direct) sum of $V_1,\dots,V_m$, where each $V_i$ is an image of $P_L$. Thus the top of $V$ is a sum of trivial and $8$-dimensional simple modules. Let $V'$ denote the $\{8_i\}$-residual of $V$. Since $V'$ has no trivial submodules, or $8_i$ submodules or quotients, it cannot have any $L$-fixed points. Thus $V^L$ has dimension at most that of the $L$-fixed points $(V/V')^L$ of the quotient $V/V'$. This dimension is equal to the number of $8$-dimensional factors in $V/V'$, proving the claim.

We will use this fact, that if $V^H=0$ then $\dim(V^L)$ is bounded above by the number of $8$-dimensional composition factors, in the remaining cases.

\medskip

\noindent \textbf{Case 2}: We assume that $L(E_8)^H=0$ and work towards a contradiction. We show that $L(E_8)^L$ has dimension at least $3$. The central element $z$ of $L$ has trace $3$ on $L(E_8)$, so from \cite[Table 1.16]{frey1998a} we see that the centralizer of $z$ is $A_1A_6T_1$, and this is a Levi subgroup of $\mb G$. Since $A_1A_6\leq A_8$, we can understand the action of $L'$ on $L(E_8)$ via $A_8$, and we see that $L'$ must act with factors $2_1^2,2_2,1^3$ on $M(A_8)$ in order to have the correct factors on $L(E_8)$. Including the action of $\gen z$, and noting that $\gen z$ acts trivially on the summand $L(A_6)\oplus (L(0)/L(A_1)^\circ/L(0))$ of $L(E_8){\downarrow_{A_1A_6}}$, we obtain that $L'$ acts as $2_1$ on $M(A_1)$ and with factors $2_1,2_2,1^3$ on $M(A_6)$. Since the $1$-eigenspace of the action of $z$ on $L(E_8)$ has the structure above, and we want to compute the fixed-point space $L(E_8)^L$, we need to understand the action of $L$ on $M(A_6)$, not just the factors.

From Table \ref{t:actioneigpsu34}, we see that for the principal block, excluding copies of the projective $4$, $L$ acts the same on both $L(E_8)$ and $W$ (but not $W_0$). We claim that $L$ must centralize at least a $3$-space on $L(E_8)$, hence on $W$. To see this, note that $L$ centralizes a $1$-space on $L(0)/L(A_1)^\circ/L(0)$, on which it acts as $1/2_2/1$. As $L(A_6)\oplus L(0)=M(A_6)\otimes M(A_6)^*$, the fixed-point space of $L$ on $M(A_6)\otimes M(A_6)^*$ is the same as that on $L(E_8)$, so we count that quantity.

Also, $(M(A_6)\otimes M(A_6)^*)^L\cong \Hom_{kL}(M(A_6),M(A_6))$, and this is easier to understand. The composition factors of $M(A_6){\downarrow_L}$ are $2_1,2_2,1^3$, so it has pressure $-1$, hence $L$ stabilizes a line and a hyperplane on $M(A_6)$. If all trivial composition factors lie in the socle or top then we have at least two different homomorphisms from the top to the socle, plus the identity, so we have a $3$-space of endomorphisms. Thus there is a subquotient $2_1/1/2_2$ (up to duality). In this case we can either have two trivial summands, so easily enough $kL$-homomorphisms, or we place a trivial on the $2_1/1/2_2$, to make $1/2_1/1/2_2$, but no trivial may be added to the top or socle of this, so we have a trivial summand and again a $3$-space of homomorphisms, as needed.

\medskip

We now claim that $8_1^{\oplus 2}\oplus 8_2$ is a subquotient of $W$. To see this, notice that all of them must lie in the $kH$-submodule generated by $W^L$, that this is a sum of quotients of $P_L$, and all $8_i$ must lie in the top of this submodule. This proves the result.

We see therefore that $W$ is the sum of $W_0$ and some $8$-dimensional simple summands.

We now compute a specific submodule of both $P(9_{12})$ and $P(\bar 9_{12})$. In each case, we take the sum of two submodules: the first is simply the $\{1,3_i^\pm,9_{12}^\pm,\bar 9_{12}^\pm,24_{21}^\pm\}$-radical of the projective; the second is the $\{1,3_i^\pm,9_{12}^\pm,\bar 9_{12}^\pm\}$-radical of the projective, on which we place as many copies of $8_1$ and $8_2$ as we can. On top of this sum, we place as many copies of $1$, $3_i^\pm$, $9_{12}^\pm$ and $\bar 9_{12}^\pm$ as we can, and then take the $\{9_{12}^\pm,\bar 9_{12}^\pm\}'$-residual of the resulting module. A sum of these must form a pyx for $W_0$. The two modules have dimension $211$ and in both cases all trivial and $8$-dimensional composition factors lie in either the second socle or the second top. Since there are eight trivial composition factors in $W_0$ and $W_0$ has pressure $4$, we see therefore that there must be exactly four $9$-dimensional factors in the socle (and four in the top).

In particular, this means that $8_2$, which occurs exactly once in $W$, is either a summand of $W$ or a summand of the heart of $W_0$. But now we use the fact from earlier, that any quotient of $P_L$ with $9$-dimensional factors in the socle has either two or six socle layers, and if it has six then $8_2$ lies in the fifth or sixth layer. Since this does not happen in our case, we conclude that all $8_i$ must in fact lie in the second socle \emph{and} the second radical layer of $W_0$. Thus one may construct a pyx for $W_0$ as a sum of submodules of $P(9_{12}^\pm)$ and $P(\bar 9_{12}^\pm)$, formed by taking the $\{1,3_i^\pm,9_{12}^\pm,\bar 9_{12}^\pm,24_{21}^\pm\}$-radical of $P(9_{12}^\pm)$ and $P(\bar 9_{12}^\pm)$, then add on the $8_j$ in the second layer, then add on any $9$-dimensional modules on top of this, then take the $\{9_{12}^\pm,\bar 9_{12}^\pm\}'$-residual of this module.

The two modules have the form
\[ 9_{12}/3_1^*/3_2/3_1^*/9_{12},24_{21}^*/3_1^*,\bar 9_{12},\bar 9_{12}^*/3_2,3_2^*,9_{12},9_{12}^*/1,3_1^*,8_2/9_{12},\]
\[ \bar 9_{12}/3_2/1,3_1^*,\bar 9_{12}/3_2,9_{12},9_{12}^*,24_{21}^*/3_1,3_1^*,\bar 9_{12},\bar 9_{12}^*/1,3_2,8_1/\bar 9_{12}.\]

We see that we cannot obtain a pyx for $W_0$ as no sum of at most four of these, at most three of which have socle $\bar 9_{12}^\pm$, can have eight trivial factors. This contradiction means that $H$ does indeed stabilize a line on $L(E_8)$, as needed.

\medskip

\noindent \textbf{Case 5}: We proceed as in Case 2, so again assume that $L(E_8)^H=0$, and again the trace is $3$, so $L\leq A_1A_6T_1$. This time the composition factors must be $2_1$ for $M(A_1)$ and $2_1^3,1$ for $M(A_6)$. There are only two modules for $M(A_6)$ with those factors, up to duality: semisimple and $(1/2_1)\oplus 2_1^{\oplus 2}$.

In the first case, the $1$-eigenspace of the action of $z$ on $L(E_8){\downarrow_L}$ is
\[ (1/2_2/1)^{\oplus 10}\oplus 2_1^{\oplus 6}.\]
Thus $L(E_8)^L$ has dimension $10$, and since there are only eight $8$-dimensional composition factors, we deduce that $L(E_8)^H\neq 0$, as needed.

Thus we are in the second case, where $2_1/1$ is a submodule of $M(A_6){\downarrow_L}$, and the $1$-eigenspace of $z$ on $L(E_8)$ is now
\[ (1/2_2/1)^{\oplus 5}\oplus(2_1/1/2_2/1)^{\oplus 2}\oplus (1/2_2/1/2_1)^{\oplus 2}\oplus (2_1/1/2_2/1/2_1).\]
We may also compute the other eigenspaces, at least up to duality: the module $(0,\lambda_3)$ restricts to $L'$ as
\[ 1\oplus 2_1^{\oplus 2}\oplus (2_1/1)^{\oplus 2}\oplus (2_1/1/2_2/1)^{\oplus 2}\oplus (2_1/1,2_1,2_2/1)\oplus 4,\]
the module $(1,\lambda_2)$ restricts to $L'$ as
\[ 2_1^{\oplus 4}\oplus 4^{\oplus 3}\oplus (2_1/1/2_2/1)\oplus (2_1/1/2_1,2_2/1)^{\oplus 2},\]
and $(1,\lambda_1)$ and $(0,\lambda_1)$ restrict to $L'$ as
\[ (1/2_2/1)^{\oplus 2}\oplus (1/2_2/1/2_1)\quad\text{and}\quad  2_1^{\oplus 2}\oplus (2_1/1).\]
From this we see that $2_2$ does not appear as a submodule or quotient in any eigenspace of the action of $x$, and hence none of $9_{12}^\pm$, $\bar 9_{12}^\pm$ and $3_2^\pm$ lies in the socle of $L(E_8){\downarrow_H}$. Since we have no trivial submodules by assumption, the socle consists of copies of $8_1$ and $3_1^\pm$.

From this we may determine the possible quotients of $P_L$ that can appear in $L(E_8){\downarrow_H}$. The module
\[ 1,8_1/\bar 9_{12},\bar 9_{12}^*/3_2,3_2^*/3_1,3_1^*\]
has a copy of $2_2$ appearing in the $\zeta$- and $\zeta^{-1}$-eigenspaces of $z$, so this cannot occur. Removing $3_1$ or $3_1^*$ from the socle, then any modules that may not appear in $\soc(L(E_8){\downarrow_H})$ yields modules
\[ 1,8_1/\bar 9_{12}/3_2/3_1,\quad 1,8_1/\bar 9_{12}^*/3_2^*/3_1^*,\]
and these have a copy of $2_2$ in the $\zeta^{-1}$- and $\zeta$-eigenspace respectively. Thus only $1$ and $8_1$ may be quotients of $P_L$ lying in $L(E_8){\downarrow_H}$. We see from the $1$-eigenspace that there are at most five $8_1$ summands, but unless there is a trivial submodule, exactly seven $8_1$ submodules, and therefore seven $8_1$ quotients. This requires nine copies of $8_1$, which is too many, and therefore $H$ stabilizes a line in this case.

\medskip

\noindent \textbf{Case 6}: Now we have that the trace of $z$ is $23$, and so $L\leq D_6T_2$. Again, we assume that $H$ does not stabilize a line on $L(E_8)$. The actions of $L'$ on the eigenspaces of $z$ are given below.

\begin{center}
\begin{tabular}{cc}
\hline Eigenspace & Action
\\ \hline $1$ & $4^4,2_1^8,2_2^8,1^{20}$.
\\ $\zeta$ & $4^3,2_1^9,2_2^4,1^7$
\\ $\zeta^2$ & $4^3,2_1^4,2_2^9,1^7$
\\ \hline
\end{tabular}
\end{center}
The $1$-eigenspace comprises $L(D_6)^\circ$ and four trivial factors, and the only actions of $L'$ on $M(D_6)$ with those factors are $2_1^2,2_2^2,4$, and  $2_1^4,2_2,1^2$ up to field automorphism. The former case is consistent with the other eigenspaces, which are the sum of a trivial, a copy of $M(D_6)$ and a half-spin module, whereas the trace of an element of order $5$ in $L'$ is inconsistent with the latter possibilities. Thus $L'$ acts as $4\oplus 2_1^{\oplus 2}\oplus 2_2^{\oplus 2}$ on $M(D_6)$.

We can compute the action of $D_6$ on $L(E_8)$, either by a direct calculation in Magma using $L(E_8)$ and the subsystem subgroup, or by using the unipotent class $D_6$ in \cite[Table 9]{lawther1995}, which acts on $M(D_6)$ with blocks $10,2$ and on both half-spins with blocks $14,10,6,2$. A computer calculation shows that it acts on $L(D_6)^\circ$ with blocks $16,14,10^2,8,6$, and so comparing these with the action on $L(E_8)$, we see that we must have a summand $L(0),L(0)/L(D_6)^\circ/L(0),L(0)$.

The action of $L'$ on $\Lambda^2(M(D_6))$ is
\[ P(2_1)^{\oplus 2}\oplus P(2_2)^{\oplus 2}\oplus 4^{\oplus 4}\oplus (1/2_1,2_2/1)\oplus (1/2_1/1)\oplus (1/2_2/1)\oplus 1^{\oplus 4}.\]
As we saw above, we need to add two trivial factors to this to get the contribution to the $1$-eigenspace. There are three unipotent classes in $D_6$ with blocks $2^6$ on $M(D_6)$ \cite[Table 7]{lawther2009}. Using a computer, we determine the actions on the half-spin modules, and $L(D_6)^\circ$ and its extension with the trivials.
\begin{center}
\begin{tabular}{cccccc}
\hline Class & $L(\lambda_1)$ & $L(\lambda_5)$ & $L(\lambda_6)$ & $L(\lambda_2)$ & $L(0),L(0)/L(\lambda_2)/L(0),L(0)$ 
\\ \hline $(3A_1)'$ & $2^6$&$2^{16}$&$2^{12},1^8$&$2^{29},1^6$&$2^{30},1^8$
\\ $(3A_1)''$ & $2^6$&$2^{12},1^8$&$2^{16}$&$2^{29},1^6$&$2^{30},1^8$
\\$2A_1+D_2$ & $2^6$&$2^{16}$&$2^{16}$&$2^{30},1^4$&$2^{32},1^4$
\\ \hline
\end{tabular}
\end{center}
Since $u$ (which can be chosen to lie in $L$) acts on the sum of the composition factors of $L(E_8){\downarrow_H}$ with blocks $2^{116},1^{16}$, we see that $u$ lies in class $2A_1+D_2$, and therefore the extension of $L(D_6)$ by trivials is non-split on restriction to $L$. Thus the action of $L$ on $L(0),L(0)/L(\lambda_2)/L(0),L(0)$ must be
\[ P(2_1)^{\oplus 2}\oplus P(2_2)^{\oplus 2}\oplus 4^{\oplus 4}\oplus (1/2_1/1)^{\oplus 2}\oplus (1/2_2/1)^{\oplus 2}\oplus 1^{\oplus 4}.\]
Hence $L(E_8)^L$ is $8$-dimensional, and there are at most two submodules $8_1$ and two submodules $8_2$. If $W$ is defined as in Case 2, we therefore have (at least) a $4$-space of homomorphisms from $P_L$ to $W$ whose image is not $8_i$. Since we only have eight $9$-dimensional modules, and one is needed for each such submodule (and quotient), this must be exactly what occurs.

Thus $W$ consists of four $8$-dimensional simple summands, and a module with four $9$-dimensional modules in the socle, four in the top, and other modules in between.

\medskip

We now examine the $\zeta^i$-eigenspaces of $z$ on $L(E_8)$. As $L'$ acts on $M(D_6)$ as $4\oplus 2_1^{\oplus 2}\oplus 2_2^{\oplus 2}$, it lies inside $D_2D_4$; there are two classes of subgroups $L'$ inside each of these $D_i$ acting in this way, swapped by the graph automorphism. Their actions on the half-spin modules are
\[ 2_1,\quad 2_2 \qquad \text{and}\qquad 2_1^{\oplus 2}\oplus 2_2^{\oplus 2},\quad 4\oplus 1^{\oplus 4}.\]
The half-spin modules for $D_6$ are sums of tensor products of half-spin modules for $D_2$ and $D_4$, so we get that, up to field automorphism, $L'$ acts on a half-spin module as 
\[ 4^{\oplus 2}\oplus P(2_2)\oplus (1/2_1/1)^{\oplus 2}\oplus 2_1^{\oplus 4}.\]
Adding on a copy of $M(D_6)$ and a trivial, we get the action on the $\zeta$-eigenspace, which is
\[ 4^{\oplus 3}\oplus P(2_2)\oplus (1/2_1/1)^{\oplus 2} \oplus 2_1^{\oplus 6}\oplus 2_2^{\oplus 2}\oplus 1.\]

\medskip

Examining the table at the start of this section, we see that the $P(2_2)$ and one $1/2_1/1$ come from the $24_{i,j}^\pm$, so there are likely very few extensions between modules in $L(E_8){\downarrow_H}$. The $\zeta$-eigenspace of $z$ on the module $1,8_1/\bar 9_{12}^*$ has $L'$-action $2_1/1$, and the $\zeta$-eigenspace of $z$ on the module $3_2^*/\bar 9_{12}^*$ is $2_2/1$, so there can be no extension between $3_2^*$ and $\bar 9_{12}^*$. By choosing different eigenspaces, there can be no extension between $3$s and $9$s. The $\{1,8_i,9_{12}^\pm,\bar 9_{12}^\pm,24_{i,j}^\pm\}$-radical of $P(\bar 9_{12}^*)$ is
\[ 1/\bar 9_{12},\bar 9_{12}^*/1,8_1/\bar 9_{12}^*,\]
so we can remove the top $1$ from this, and just need to decide which $9$-dimensional module we want in the third layer. However, the module $\bar 9_{12}/1,8_1/\bar 9_{12}^*$ has a trivial quotient, so we need $\bar 9_{12}^*/1,8_1/\bar 9_{12}^*$.

Thus $W$ must be the sum of $8_1^{\oplus 2}\oplus 8_2^{\oplus 2}$ and four modules of the form
\[ 9_{12}^\pm/1,8_2/9_{12}^\pm,\quad \bar 9_{12}^\pm/1,8_1/\bar 9_{12}^\pm.\]
The $3$s must split off, and there are no extensions between $24_{i,j}^\pm$ and the modules above, so they must split off, and we obtain the action
\[ \bigoplus_{i=1}^2 (9_{i,3-i}/1,8_{3-i}/9_{i,3-i})\oplus (9_{i,3-i}^*/1,8_{3-i}/9_{i,3-i}^*)\oplus  24_{i,3-i}\oplus 24_{i,3-i}^*\oplus 8_i^{\oplus 2}\oplus 3_i\oplus 3_i^*.\]
(Here, we write $9_{21}=\bar 9_{12}$ to make the expression easier to understand.) This exists inside $E_6A_2$. We show that $H$ is contained in a positive-dimensional subgroup stabilizing all $3$-dimensional summands, hence all members of an $N_{\Aut^+(\mb G)}(H)$-orbit of submodules of $L(E_8){\downarrow_H}$.

To do this, we place $L$ inside a group $A_1\times \gen z$, with the $A_1$ subgroup $\mb Y$ lying in $D_6$, in fact inside $D_2D_4$. Let $\mb Y_1$ denote an $A_1$ subgroup of $D_2$ acting on $M(D_2)$ as $L(3)$, and on the two half-spin modules as $L(1)$ and $L(2)$. Clearly $\mb Y_1$ can be chosen to contain the projection of $L'$ onto the $D_2$ factor.

Let $\mb Y_2$ denote an $A_1$ subgroup of $D_4$ acting on $M(D_4)$ as $L(1)^{\oplus 2}\oplus L(2)^{\oplus 2}$. The actions on the two half-spin modules are as $L(1)^{\oplus 2}\oplus L(2)^{\oplus 2}$ and $L(3)\oplus L(0)^{\oplus 4}$. (To see this, note that we may conjugate $\mb Y_2$ by the graph automorphism of order $3$ to lie in a $D_2$-Levi subgroup, and then it becomes clear.) Again, we may choose $\mb Y_2$ to contain the projection of $L'$ onto $D_4$.

Finally, let $\mb Y$ be a diagonal $A_1$ in $D_2D_4\leq D_6$ acting as $\mb Y_1$ on the one factor and $\mb Y_2$ on the other. It is easy to see that $\mb Y$ and $L'$ stabilize all of the same simple submodules of the $\zeta^i$-eigenspaces of $z$ for $i=1,2,3,4$, and so in particular must stabilize each $3_i^\pm$ in $L(E_8){\downarrow_H}$. This completes the proof.

\medskip

\noindent \textbf{Cases 7 and 8}: These remaining cases are as in the statement of the proposition. Case 8 is solved in Section \ref{sec:diffpsu34}, but Case 7 is unsolved.

\medskip

\noindent $\boldsymbol{q=8}$: There are, up to field automorphism, eight conspicuous sets of composition factors for $L(E_8){\downarrow_H}$:
\[ 9_{12},9_{12}^*,8_1^{25},1^{30},\quad (9_{12},9_{12}^*)^8,8_1^8,8_2^3,1^{16},\]
\[  (\bar 9_{13},\bar 9_{13}^*)^8,9_{23},9_{23}^*,8_1^8,8_3,1^{14},\quad 64_{13},(\bar 9_{13},\bar 9_{13}^*)^6,(9_{23},9_{23}^*)^2,8_1^4,1^8,\] \[27,27^*,(9_{12},9_{12}^*)^3,(\bar 9_{13},\bar 9_{13}^*)^3,(9_{23},9_{23}^*)^3,8_1,8_2,8_3,1^8,\] \[(27,27^*)^2,(9_{12},9_{12}^*)^2,(\bar 9_{13},\bar 9_{13}^*)^2,(9_{23},9_{23}^*)^2,8_1,8_2,8_3,1^8,\] 
\[64_{13},27,27^*,9_{12},9_{12}^*,(\bar 9_{13},\bar 9_{13}^*)^3,(9_{23},9_{23}^*)^2,8_1^2,1^6,\]
\[ 64_{12}^3,9_{12},9_{12}^*,9_{23},9_{23}^*,8_1,8_2,1^4.\]
Note that, with the notation from Chapter \ref{chap:labelling},  $\Ext_{kH}^1(\bar 9_{13},9_{23}^*)$ is non-zero, whereas $\Ext_{kH}^1(\bar 9_{13},9_{23})=0$.

The pressures of these are $-28$, $0$, $4$, $8$, $10$, $4$, $6$ and $0$ respectively.

\medskip

\noindent \textbf{Cases 1, 2 and 8}: These all have non-positive pressure, so $H$ stabilizes a line on $L(E_8)$ by Proposition \ref{prop:pressure}.

\medskip

\noindent \textbf{Case 3}: The $\{1,8_1,8_3,\bar 9_{13}^\pm,9_{23}^\pm\}$-radicals of $P(\bar 9_{13})$ and $P(9_{23})$ are
\[ \bar 9_{13}/8_1/\bar 9_{13}^*,\bar 9_{13}^*/1,8_1,9_{23}/\bar 9_{13},\bar
9_{13},\bar 9_{13}^*/1,8_1,9_{23}^*/\bar 9_{13},\]
\[1,\bar 9_{13}/8_1,9_{23},9_{23}^*/1,8_3,\bar 9_{13}^*/9_{23}.\]
Since $L(E_8){\downarrow_H}$ has pressure $4$, we can support at most eight trivial composition factors above the modules with non-zero $1$-cohomology, which is not enough. Thus $H$ stabilizes a line on $L(E_8)$, as claimed.

\medskip

\noindent \textbf{Cases 4, 6 and 7}: We will work in the centralizer of an element $z$ of order $3$, so we give some information about this now. Let $L=C_H(z)\cong \gen z\times \PSL_2(8)$, and note that $L'\cong \PSL_2(8)$. Any module for $L$ splits as a direct sum of three summands: the $1$-eigenspace, $\omega$-eigenspace, and $\omega^2$-eigenspace of the action of $z$, where $\omega^3=1$. We arrange the modules for $H$, $L$ and the exact power of $z$ so that the actions of $L'$ on the eigenspaces of $z$ are as in Table \ref{t:actioneigpsu38}.
\begin{table}
\begin{center}\begin{tabular}{cccc}
\hline Module & $1$-eigenspace & $\omega$-eigenspace & $\omega^2$-eigenspace
\\ \hline $1$ & $1$ &&
\\ $8_1$&$1/2_1/1$&$2_3$&$2_3$
\\ $8_2$&$1/2_2/1$&$2_1$&$2_1$
\\ $8_3$&$1/2_3/1$&$2_2$&$2_2$
\\ $9_{12}$&$2_1$&$1\oplus 4_{13}$&$2_3$
\\ $9_{23}$&$2_2$&$2_1$&$1\oplus 4_{12}$
\\ $\bar 9_{13}$&$2_3$&$2_2$&$1\oplus 4_{23}$
\\ $9_{12}^*$&$2_1$&$2_3$&$1\oplus 4_{13}$
\\ $9_{23}^*$&$2_2$&$1\oplus 4_{12}$&$2_1$
\\ $\bar 9_{13}^*$&$2_3$&$1\oplus 4_{23}$&$2_2$
\\ $27$&$4_{12}\oplus 4_{23}\oplus 4_{13}$&$1\oplus 8$&$2_1\oplus 2_2\oplus 2_3$ 
\\ $64_{13}$&$4_{23}^{\oplus 2}\oplus (1/2_1,2_3/$& $4_{23}\oplus (2_2/4_{12}/2_2)$&Same as $\omega$-eigenspace
\\ &$/1,1,4_{13}/2_1,2_3/1)$&$\oplus (2_3/1/2_1/1/2_3)$
\\\hline
\end{tabular}\end{center}
\caption{Actions of $L'\cong \PSL_2(8)$ on each $\omega^i$-eigenspace of the action of $z$}
\label{t:actioneigpsu38}
\end{table}
If the trace of $z$ on $L(E_8)$ is $-4$ then $C_{\mb G}(z)$ is the group of type $A_8$, if it is $5$ then $C_{\mb G}(z)$ is of type $E_6A_2$, and if it is $14$ then $C_{\mb G}(z)$ is of type $D_7T_1$ (see Table \ref{t:semie8}).

We also need the permutation module $P_L$ on the cosets of $L$, which has dimension $3648$, so it is too large to describe completely. We give some facts about it:
\begin{itemize}
\item $P_L$ has simple quotients precisely $1$, $8_i$, $64_{i,j}$ and $512$;
\item there is no non-trivial quotient of $P_L$ with trivial top;
\item there is no non-semisimple quotient of $P_L$ with top $8_1\oplus 8_2\oplus 8_3$.
\end{itemize}

Suppose that $V$ is a $kH$-module with composition factors of dimension at most $27$, and with $V^H=0$. We show that $\dim(V^L)$ is bounded above by the number of $8$-dimensional composition factors in $V$, proceeding by induction on the dimension of $V$. We may assume that $\dim(V^L)>0$, of course. This proof will look very similar to the proof for $q=4$.

If $V$ possesses a submodule of dimension $8$ then we may quotient it out: the dimension of the $L$-fixed points must reduce by $1$, as must the number of $8$-dimensional factors. Thus there are no $8$-dimensional quotients. We also may assume that $V$ is generated by $V^L$. Let $V'$ denote the $\{8_i\}$-residual of $V$. Since the $kH$-submodule generated by any $L$-fixed point is either $8_i$ or has a trivial and an $8_i$ quotient, $(V')^L=0$. Thus $\dim(V^L)\leq \dim((V/V')^L)$, where $V/V'$ is the quotient. As the result holds for $V/V'$, the result holds for $V$, as needed.

As a consequence, if $L(E_8){\downarrow_H}$ has no trivial submodules then in Case 6, $L(E_8)^L$ has dimension at most $3$, and in Case 7, $L(E_8)^L$ also has dimension at most $3$ (since the $64_{13}$ splits off as a summand and has a $1$-dimensional $L$-fixed space).

\medskip

\noindent \textbf{Case 4}: The module $64_{13}$ splits off as a summand as it has no extensions with other composition factors. The trace of $z$ on $L(E_8)$ is $-4$, so that $\mb X=C_{\mb G}(z)$ is $A_8$. Note that $\mb X$ acts on $L(E_8)$ as the sum of $M(\mb X)\otimes M(\mb X)^*$ (minus a trivial summand), $\Lambda^3(M(\mb X))$ and $\Lambda^3(M(\mb X)^*)$, where $M(\mb X)$ is the $9$-dimensional minimal module for $\mb X$. The summand $M(\mb X)\otimes M(\mb X)^*$ is the $1$-eigenspace for $z$.

From the composition factors of $L(E_8){\downarrow_L}$, we easily deduce that the composition factors of $V=M(\mb X){\downarrow_{L'}}$ are $4_{23},2_2^2,1$, so up to field automorphism $4_{12},2_1^2,1$. There are two modules with this structure, up to duality:
\[ 4_{12}\oplus 2_1\oplus (2_1/1),\qquad 4_{12}\oplus 2_1^{\oplus 2}\oplus 1.\]
In the former possibility, there is no $4$-dimensional simple summand of $V\otimes V^*$, but there is one in the restriction of the $1$-eigenspace of $z$ on $64_{13}$. Thus the action of $L'$ on $V$ is the second possibility, so $V{\downarrow_{L'}}$ is semisimple. Let $\mb Y$ denote the algebraic $A_1$ containing $L'$ and acting as $L(0)\oplus L(1)^{\oplus 2}\oplus L(3)$ on $V$. Using the formula
\[ \Lambda^3(A\oplus B)=\Lambda^3(A)\oplus (\Lambda^2(A)\otimes B)\oplus (A\otimes \Lambda^2(B))\oplus \Lambda^3(B),\]
we can easily compute the structure of $L(E_8){\downarrow_{\mb Y}}$, and indeed $L'$ and $\mb Y$ stabilize the same subspaces of $L(E_8)$, so that $L'$ and $H$ are blueprints for $L(E_8)$.

\medskip

\noindent \textbf{Case 6}: This time $z$ lies in the class with centralizer $D_7T_1$ (see Table \ref{t:semie8}), so $L'\leq D_7$. The action of $D_7$ on $L(E_8)$ has factors of dimensions $14$, $14$, $64$, $64$, $90$, $1$ and $1$, with the final three organized into the shape $1/90/1$. The $1/90/1$ is the $1$-eigenspace for $z$, as $z$ has trace $14$ on $L(E_8)$, and the $90$ here is a submodule of codimension $1$ in the exterior square of the natural module for $D_7$.

The possible composition factors of $M(D_7){\downarrow_{L'}}$ are 
\begin{itemize}
\item $8,2_1,2_2,2_3$,
\item $4_{12},4_{13},4_{23},1^2$, and
\item $2_1^2,2_2^2,2_3^2,1^2$.\end{itemize}
If the factors are the first possibility, then the exterior square is
\[ P(1)\oplus P(4_{12})\oplus P(4_{13})\oplus P(4_{23})\oplus 4_{12}\oplus 4_{13}\oplus 4_{23}\oplus 1^{\oplus 3}.\]
The restriction of $1/90/1$ to $L'$ must have an extra trivial summand. If the factors are the second possibility, then the exterior square is
\[\begin{split}
P(4_{12})\oplus P(4_{13})\oplus& P(4_{23})\oplus (1/2_1,2_2/1)\oplus (1/2_1,2_3/1)
\\&\oplus (1/2_2,2_3/1)\oplus 4_{12}^{\oplus 2}\oplus 4_{13}^{\oplus 2}\oplus 4_{23}^{\oplus 2}\oplus 1.\end{split}\]
The restriction of $1/90/1$ to $L'$ must have an extra trivial submodule and quotient (but the trivial summand is not there). In both possibilities therefore, $L(E_8)^L$ is at least $4$-dimensional, and we have previously shown that it can be at most $3$-dimensional without $H$ stabilizing a line on $L(E_8)$. Thus $H$ stabilizes a line on $L(E_8)$ in both of these possibilities.

In the final possibility, $M(D_7){\downarrow_{L'}}$ has $2_i$-pressure $0$, so it has a submodule (and quotient as it is self-dual) $2_i$ for each $i$. These each contribute one trivial submodule to $\Lambda^2(M(D_7)){\downarrow_{L'}}$, so we just need one further one to achieve a $4$-dimensional fixed-point space. For this, see that the alternating form that these trivial submodules correspond to is singular, with kernel of codimension $2$, and so any linear combination of them has kernel of codimension at most $6$. However, $M(D_7)$ has a non-singular symmetric, hence alternating form on it, and so $\Lambda^2(M(D_7))^{L'}$ is at least $4$-dimensional. Hence again $L(E_8)^L$ has dimension at least $4$, and so $H$ stabilizes a line on $L(E_8)$, as needed.

\medskip

\noindent \textbf{Case 7}: The $64_{13}$ must split off as a summand. The trace of $z$ on $L(E_8)$ is $5$, and hence $C_{\mb G}(z)$ is an $A_2E_6$ maximal-rank subgroup. The $kL'$-module that consists of the $1$-eigenspace of the action of $z$ on $L(E_8)$ has composition factors
\[ 4_{12}^2,4_{13}^3,4_{23}^4,2_1^6,2_2^4,2_3^8,1^{14}.\]
This is the sum of $L(A_2){\downarrow_{L'}}$ and $L(E_6){\downarrow_{L'}}$, and if the action of $L'$ on $M(A_2)$ has factors $2_i,1$, then we must subtract $1^2,2_i^2,2_{i+1}$ from these factors to get the action on $L(E_6)$. The only $i$ for which this produces a conspicuous set of composition factors is $i=3$, and $L(E_6){\downarrow_{L'}}$ has composition factors
\[ 4_{12}^2,4_{13}^3,4_{23}^4,2_1^5,2_2^4,2_3^6,1^{12}.\]
Using traces of semisimple elements, the corresponding factors on $M(E_6)$ are
\[ 4_{12},4_{23}^2,2_1,2_2^2,2_3^3,1^3.\]

We claim that we may assume that $L'$ is contained in the product of $A_2$ and the $A_5$-Levi subgroup of $E_6$. From \cite{craven2015un}, it is known that any copy of $\PSL_2(8)$ lying in $E_6$ lies inside a ($\sigma$-stable) proper positive-dimensional subgroup of $E_6$, so we can place $L'$ inside $A_2\mb X$ for some maximal $\mb X\leq E_6$. We eliminate all $\mb X$ other than the $A_5$-parabolic subgroup first.

We start with maximal-rank subgroups. It is easy to see that $\mb X\neq A_2A_2A_2$, for example because any such $L'$ in this subgroup would have to have at least nine trivial factors on $M(E_6)$. If $L'$ is contained in $\mb X=A_1A_5$ (and not $A_5$ itself), then $L'$ must act as $2_3$ on $M(A_1)$ and $4_{23}\oplus 1^{\oplus 2}$ on $M(A_5)$. This places $L'$ inside an $A_1A_3$-Levi subgroup, so inside some $A_5$-parabolic subgroup that contains it.

Now the parabolics: as there is no copy of $L'$ inside $A_2A_2A_2$, it cannot be contained in the $A_1A_2A_2$-parabolic subgroup. If $L'$ is contained in an $A_1A_4$-parabolic then the image in the Levi subgroup must be contained in $A_1A_3$, as above. This means that $L'$ lies in (up to conjugacy) the intersection of both $A_1A_3$-parabolic subgroups of $A_1A_4$, hence of $E_6$. Thus up to conjugation we may place $L'$ in an $A_5$-parabolic subgroup of $E_6$.

We are left with the $D_5$-parabolic subgroups. We must distribute the factors among the $1$-, $10$- and $16$-dimensional composition factors of the restriction of $M(E_6)$ to $D_5$. As there is no unipotent class acting with blocks $2^5$ on $M(D_5)$ (see, for example, \cite[Table 7]{lawther2009}), the $1^2$ must lie there. The only way to distribute the composition factors of $M(E_6){\downarrow_{L'}}$ among the factors of $D_5$ is $2_1,2_3^3,1^2$ for $M(D_5){\downarrow_{L'}}$ (and hence $4_{12},4_{23}^2,2_2^2$ for the half-spin module). Note that $L'$ must stabilize a line on $M(D_5)$, else there is certainly a subquotient $1/2/1$. In this case the involution in $L'$ acts projectively on $M(D_5)$, but no such involution lies in $D_5$.

Since $L'$ stabilizes a line on $M(D_5)$, it lies in either $B_4$ or a $D_4$-parabolic subgroup. By checking traces of elements of order $3$ on the two half-spin modules, we find that $L'$ cannot lie in $D_4$, so it must lie in $B_4$, whence $L'$ acts as $2_1\oplus 2_3^{\oplus 3}$ on $M(B_4)$ (which has dimension $8$ in characteristic $2$).

We see that one of the submodules $2_3$ (diagonal if necessary) must be totally isotropic, and therefore has a parabolic subgroup as a stabilizer. This places $L'$ inside a proper Levi subgroup---in particular, $B_2A_1$---of $B_4$, hence a proper Levi subgroup---in particular, $D_3A_1$---of $D_5$. Thus $L'$ lies inside the $D_3A_1$-parabolic subgroup of $E_6$, hence inside the $A_4A_1$-parabolic subgroup, which we have already considered.

Of the reductive subgroups, the $G_2$ maximal subgroup acts on $M(E_6)$ acts with factors $L(01)$, $L(10)$, $L(20)$ and $L(00)$ (of dimensions 14, 6, 6 and 1), and this is incompatible with $L'$. If $L'\leq A_2G_2$ then the only action compatible with $M(E_6)$ is acting on $M(A_2)$ with factors $2_2,1$ and on $M(G_2)$ with factors $2_1,2_3^2$. This places $L'$ inside a parabolic subgroup of $A_2G_2$, and hence inside a parabolic subgroup of $E_6$.

If $L'\leq \mb X=F_4$, then again by \cite{craven2015un} there is a proper positive-dimensional subgroup of $\mb X$ containing $L'$. If it is a parabolic then $L'$ is contained in a parabolic subgroup of $E_6$. If it is $B_4$ then $L'$ is contained in the $D_5$-Levi subgroup. If $L'\leq A_2A_2$ then $L'\leq A_2A_2A_2\leq E_6$, which we know it is not. The only remaining subgroup is $C_4$. Note that from $L(E_6){\downarrow_{L'}}$ and the fact that $L(E_6){\downarrow_{\mb X}}$ is the sum of two copies of $M(F_4)$ and one of the image of $M(F_4)$ under the graph automorphism, we may compute the factors of the image of $L'$ under the graph automorphism on $M(F_4)$, and this subgroup must be contained in $B_4$. The composition factors are $4_{13}^3,2_1^3,1^8$, a module of pressure $-5$, hence centralizes a $5$-space on $M(F_4)$. This places $L'$ inside a very small subgroup of $B_4$, certainly contained in another positive-dimensional subgroup of $F_4$, and so $L'$ is contained in another subgroup of $E_6$, as needed.

We now move inside the $A_5$-parabolic subgroup, and place $L'$ inside an $A_5$-Levi subgroup. Inside the $A_5$-parabolic, we see from above that the composition factors on $M(A_5)$ are $4_{23},2_3$, and there are (up to duality) two such modules: $4_{23}\oplus 2_3$ and $2_3/4_{23}$. As the layers of the unipotent radical of the $A_5$-parabolic subgroup are $L(0)$ and $L(00100)$, and the restrictions of these modules to $L'$ have zero $1$-cohomology, we see from Lemma \ref{lem:nonabcohomzero} that $L'$ lies in the $A_5$-Levi subgroup. In particular, we may write down the structure of $L(E_6){\downarrow_{L'}}$ exactly, as the restriction of
\[ L(0)^{\oplus 2}\oplus L(00100)^{\oplus 2}\oplus (L(10000)\otimes L(00001))\]
to $L'$. As $64_{13}$ is a summand of $L(E_8){\downarrow_H}$, the $1$-eigenspace of $z$ on this module has a summand of the form $1/2_1,2_3/1/1,4_{13}/2_1,2_3/1$, and this does not lie in $L(E_6){\downarrow_{L'}}$ if $M(A_5){\downarrow_{L'}}$ is not semisimple, we obtain that $L(E_6){\downarrow_{L'}}$ is exactly
\[\begin{split} (1/2_1,2_3/1/1,4_{13}/2_1,2_3/1)&\oplus (2_3/1/2_1/1/2_3)^{\oplus 2}
\\&\oplus (2_2/4_{12}/2_2)^{\oplus 2}\oplus (1/2_1/1)\oplus 4_{13}^{\oplus 2}\oplus 4_{23}^{\oplus 4}\oplus 1^{\oplus 2}.\end{split}\]
In addition, the action of $L'$ on $L(A_2)$ is either $(1/2_1/1)\oplus 2_3^{\oplus 2}$ or $2_3/1/2_1/1/2_3$.

This has (at least) a $4$-dimensional fixed space, and we saw earlier that if $L(E_8)^H=0$ then it must have dimension at most $3$. Thus $H$ stabilizes a line on $L(E_8)$ in this case as well.

\medskip

\noindent \textbf{Case 5}: This is as stated in the proposition.

\medskip

\noindent $\boldsymbol{q=16}$: As there are so many possible sets of composition factors, we will only consider conspicuous sets of composition factors that have positive pressure or no trivial composition factors. There are, up to outer automorphism, 28 sets of composition factors that are conspicuous for elements of order at most $17$. Let $x$ denote an element of order $255$. A computer check shows that, in each of these cases, the eigenvalues of $x$ determine the class of $x$ in $\mb G$, and the same holds for $x^3$ of order $85$. Thus we look for elements of $\mb G$ of order $765$ cubing to $x$, aiming to apply Lemma \ref{lem:allcasesstronglyimp}. For 24 of the 28 sets of composition factors we find elements of order $765$ with the same number of distinct eigenvalues on $L(E_8)$ as $x$, so $H$ is strongly imprimitive.

Thus we are left with four sets of composition factors for $L(E_8){\downarrow_H}$.
\[ 24_{13},24_{13}^*,24_{31},24_{31}^*,(9_{13},9_{13}^*)^2,(\bar 9_{13},\bar 9_{13}^*)^2,\bar 9_{14},\bar 9_{14}^*,9_{23},9_{23}^*,8_1^2,8_3^2,3_1,3_1^*,3_3,3_3^*,\]
\[27_{124},27_{124}^*,27_{421},27_{421}^*,24_{12},24_{12}^*,(9_{24},9_{24}^*)^2,\bar 9_{34},\bar 9_{34}^*,8_1,8_2^2,8_4,3_1,3_1^*,\]
\[ 27_{124},27_{124}^*,27_{314},27_{314}^*,27_{321},27_{321}^*,27_{432},27_{432}^*,8_1,8_2,8_3,8_4,\]
\[ 27_{134},27_{134}^*,27_{213},27_{213}^*,27_{324},27_{324}^*,27_{421},27_{421}^*,8_1,8_2,8_3,8_4.\]
(Here, $27_{i,j,k}$ means $3_i\otimes 3_j\otimes 3_k$ if $i<j$ and $3_i\otimes 3_j\otimes 3_k^*$ if $j<i$.)

\medskip

\noindent \textbf{Cases 3 and 4}: By \cite{sin1992}, there are no extensions between $8$- and $27$-dimensional modules, or between $8$- and $8$-dimensional modules, so in the last two cases the $8_i$ must split off as summands. Furthermore, the eigenvalues of $x$ on $8_i$, with the exception of two copies of $1$, are distinct from the eigenvalues of $x$ on all other composition factors of $L(E_8){\downarrow_H}$. Thus every root of $x$ in a maximal torus of $\mb G$ must stabilize each $8_i$ (the $1$-eigenspace of $x$ is $8$-dimensional so also preserved by every root), and hence $H$ is strongly imprimitive by Lemma \ref{lem:allcasesstronglyimp}.

\medskip

\noindent \textbf{Case 2}: The $3_1$, $3_1^*$, $8_1$, $8_2$ and $8_4$ all split off as summands, as they have no extensions with the other composition factors of $L(E_8){\downarrow_H}$. As with the previous cases, the eigenvalues of $x$ on $8_1$, apart from $1$, do not appear in the other factors, and so every element of a maximal torus containing $x$ that powers to $x$ must stabilize $8_1$. (Again, the $1$-eigenspace is $8$-dimensional so is automatically preserved.) Since the factors of $L(E_8){\downarrow_H}$ are not preserved by a field automorphism of $H$, we apply Lemma \ref{lem:semilinearfield} to see that $8_1$ is in its own $N_{\Aut^+(\mb G)}(H)$-orbit. Hence $H$ is strongly imprimitive by Lemma \ref{lem:allcasesstronglyimp}.

\medskip

\noindent \textbf{Case 1}: There are no composition factors whose eigenspaces miss all other factors, but we find exactly 729 elements of order 765 that cube to $x$ and stabilize the eigenspaces that comprise a given composition factor of $L(E_8){\downarrow_H}$. Noting that the composition factors of $L(E_8){\downarrow_H}$ are stable under an outer automorphism $\alpha$ of order $4$ of $H$, in fact we find 81 elements that cube to $x$ and stabilize the eigenspaces that make up all four simple modules in the orbit under $\alpha$ for a composition factor of $L(E_8){\downarrow_H}$. We may therefore apply Lemma \ref{lem:allcasesstronglyimp} via Lemma \ref{lem:semilinearfield} to see that $H$ is strongly imprimitive.
\end{proof}

For $q=3$, both conspicuous sets of composition factors in the statement of the proposition are warranted. A warrant for the first set of factors is given by a diagonal $\PSL_3(3)$ in $E_6A_2$ acting irreducibly on both $M(A_2)$ and $M(E_6)$. A warrant for the second set of factors is given by a copy of $\PSU_3(3)$ inside $A_8$ acting on $M(A_8)$ as $3\oplus 6^*$. (These are the same actions, but not the same composition factors, as the first and third cases for $\PSL_3(3)$ that were described at the end of the previous section.)





\section{\texorpdfstring{$\PSp_4$}{PSp 4}}

When computing with $H\cong \PSp_4(3)$, we will be able to use the subgroup $L=2^4\rtimes \Alt(5)$ (which is an extraspecial type maximal subgroup in the group $\Sp_4(3)$, or a parabolic subgroup in $\PSU_4(2)\cong \PSp_4(3)$) to great effect. Because of this, we will give some information about the group now.

The simple modules for $\PSp_4(3)$ are $1$, $5$, $10$, $14$ and $25$, together with the projective module $81$, which we will ignore. The simple modules for $L$ are $1$, $4$, $5$, $10_1$, $10_2$ and $15$, together with two $3$-dimensional modules that do not appear in restrictions except for $81$. The restrictions of the simple modules for $H$ to $L$ are
\[ 1,\quad 5,\quad 10_1,\quad 4\oplus 10_2,\quad 10_2\oplus 15,\quad 15\oplus P(10_1)\oplus P(10_2)\oplus 3_1\oplus 3_2,\]
and the projective indecomposable modules are
\[ 1/4/1,\quad 4/1/4,\quad 5/5/5,\quad 10_1/10_2/10_1,\quad 10_2/10_1/10_2,\quad 15.\]
The simple modules for $H$ restrict to $L$ semisimply (other than $81$, which is always a summand anyway as it is projective) and with few factors, and the structures of the projectives for $L$ are very easy to understand. These can be used to place significant restrictions on the structure of $L(E_8){\downarrow_H}$, which is difficult just using $H$, as there are lots of extensions between simple modules for $H$ (for example, both $5$ and $10$ have self-extensions).

If $v$ denotes an element of order $3$ in $L$, then $v$ acts on the simple modules for $L$ with block structures
\[ 1,\quad 3,1,\quad 3,1^2,\quad 3^3,1,\quad 3^3,1,\quad 3^5.\]
We continue with our definition of $u$, that it is a element of order $p$ in $H$ from the smallest conjugacy class, from the start of the chapter.

\begin{proposition}\label{prop:sp4ine8}
Let $H\cong \PSp_4(q)$ for some $q\leq 9$.
\begin{enumerate}
\item If $q=2$ then either $H$ stabilizes a line on $L(E_8)$ or the composition factors of $L(E_8){\downarrow_H}$ are \[8_1^6,8_2^6,4_1^{16},4_2^{16},1^{24}.\]
\item If $q=3,9$ then $H$ either stabilizes a line on $L(E_8)$ or is a blueprint for $L(E_8)$.
\item If $q=4,5,8$ then $H$ is strongly imprimitive.
\item If $q=7$ then $H$ is a blueprint for $L(E_8)$.
\end{enumerate}
\end{proposition}
\begin{proof} $\boldsymbol{q=7}$: There are three sets of composition factors for $L(E_8){\downarrow_H}$ that are conspicuous for elements of order at most $12$, namely
\[ 154,84,10,\qquad 35_1^4,14^2,10^6,5^4,\qquad 14,10^{11},5^{20},1^{24}.\]
None of the simple modules in these sets has an extension with any other, and so $L(E_8){\downarrow_H}$ must be semisimple in all three cases. The action of $u$ on $L(E_8)$ has blocks 
$7,6^6,5^{11},4^{16},3^{15},2^{14},1^{13}$ in the first case, $4^8,3^{28},2^{48},1^{36}$ in the second and $3^{14},2^{64},1^{78}$ in the third, which are the actions of the generic classes $A_3+2A_1$, $4A_1$ and $2A_1$ respectively (see \cite[Table 9]{lawther1995}). Thus $H$ is a blueprint for $L(E_8)$ by Lemma \ref{lem:genericmeansblueprint}, and this completes the proof.

\medskip

\noindent $\boldsymbol{q=5}$: There are three sets of composition factors for $L(E_8){\downarrow_H}$ that are conspicuous for elements of order at most $12$, namely
\[ 86,68^2,10^2,5,1,\qquad 35_1^4,13^2,10^6,5^4,1^2,\qquad 13,10^{11},5^{20},1^{25}.\]
The second and third cases have non-positive pressure, so $H$ stabilizes a line on $L(E_8)$ in these two cases.

In the first case, we consider the subgroup $L\cong \SL_2(5)\circ \SL_2(5)$, the derived subgroup of the centralizer of an involution $z$ in $H$. The trace of $z$ on $L(E_8)$ is $-8$, so $L\leq D_8$. The composition factors of $L$ on the $1$-eigenspace of $z$ are
\[ (4,4),(4,2)^2,(2,4)^2,(2,2),(2,0)^4,(0,2)^4,(0,0)^2.\]
(We use algebraic group notation here because we will embed $L$ into an algebraic $A_1A_1$ and so we are trying to avoid confusion.) There is a unique possibility for $M(D_8){\downarrow_L}$ whose exterior square has these factors, which is $(1,3)\oplus (3,1)$. We have an obvious action of $\mb Y\cong A_1A_1$ on $M(D_8)$ extending this action, namely $(1,3)\oplus (3,1)$. Such an $A_1A_1$ subgroup lies in $D_4D_4$, acting on $L(\lambda_1)$ for the two $D_4$ factors as $(1,3)$ and $(3,1)$ respectively. Using the traces of elements of order $3$, one proves that the actions of $L$ on the two modules $L(\lambda_3)$ and $L(\lambda_4)$ are $(1,3)$ and $(2,0)\oplus (0,4)$ when $L(\lambda_1)$ is $(1,3)$, and $(3,1)$ and $(0,2)\oplus (4,0)$ in the other case.

The $(-1)$-eigenspace for $z$ on $L(E_8)$ is a half-spin module for $D_8$, which restricts to $D_4$ as a sum of products of half-spin modules. To keep $z$ acting as $-1$ on this space, we must therefore have (with the embeddings above) the sum of $(\lambda_3,\lambda_4)$ and $(\lambda_4,\lambda_3)$. Thus we need to cosnider the tensor product of $(1,3)$ and $(0,2)\oplus (4,0)$, which is
\[ (1,1)\oplus ((1,3)/(1,5)/(1,3))\oplus ((3,3)/(5,3)/(3,3)).\]
The other tensor product is just like this but with the two copies of $A_1$ swapped. The restriction of this module to $A_1(5)A_1(5)\cong L$ is
\[ (1,1)\oplus ((1,3)/(1,1)/(1,3))\oplus ((3,3)/(1,3)/(3,3)),\]
(where we keep the notation from algebraic groups) so $L$ is not a blueprint, but every semisimple $kL$-submodule of the $(-1)$-eigenspace of $z$ on $L(E_8)$ is a $k\mb Y$-submodule.

We now consider the $1$-eigenspace of $z$, which is $\Lambda^2(M(D_8))$. The exterior square of $(1,3)\oplus (3,1)$ for $\mb Y$ is
\[\begin{split} ((0,2)/(0,6)/(0,2))&\oplus ((2,0)/(6,0)/(2,0))
\\ &\oplus (2,0)\oplus (0,2)\oplus (2,2)\oplus (2,4)^{\oplus 2}\oplus (4,2)^{\oplus 2}\oplus (4,4).\end{split}\]
The restriction of this module to $A_1(5)A_1(5)\cong L$ is
\[\begin{split} ((0,2)/(0,0),(0,2)/(0,2))&\oplus ((2,0)/(0,0),(2,0)/(2,0))
\\&\oplus (2,0)\oplus (0,2)\oplus (2,2)\oplus (2,4)^{\oplus 2}\oplus (4,2)^{\oplus 2}\oplus (4,4).\end{split}\]
Thus again all semisimple $kL$-submodules of the $1$-eigenspace of $L(E_8)$ are $k\mb Y$-submodules, and hence any submodule of $L(E_8){\downarrow_H}$ that restricts semisimply to $L$ is stabilized by $\mb Y$. The restriction of $10$ to $L$ is $(0,2)\oplus (2,0)\oplus (1,1)$, and the restriction of $68$ to $L$ is
\[ (2,0)\oplus (0,2)\oplus (1,3)\oplus (3,1)\oplus (2,4)\oplus (4,2)\oplus (3,3),\]
and so if either $10$ or $68$ lies in the socle of $L(E_8){\downarrow_H}$ then $\gen{H,\mb Y}$  is positive dimensional and not equal to $\mb G$. However, $10$ and $68$ are the only composition factors to appear more than once, so at least one of them (in fact, both), must appear in the socle, and $H$ is strongly imprimitive.

\medskip

\noindent$\boldsymbol{q=3}$: There are seven conspicuous sets of composition factors for $L(E_8){\downarrow_H}$, namely
\[ 14,10^{11},5^{20},1^{24},\quad 25^4,14^5,10^3,5^8,1^8,\quad 81,25^2,14^4,5^{11},1^6,\quad 25^3,14^6,10^3,5^{11},1^4,\]
\[ 25^5,14,10^{10},5,1^4,\quad 81,25^3,10^7,5^4,1^2,\quad 25^4,14^2,10^{10},5^4.\]
We will prove that only the first and last of these seven copies may exist (and in fact do exist, both inside $A_4A_4$, the first embedding only in one factor and the last embedding diagonally in both factors).

\medskip

\noindent \textbf{Cases 2, 3, 4, 5 and 6}:  Case 5 has negative pressure, it is fairly easy to prove directly that Case 3 stabilizes a line on $L(E_8)$ anyway, and more detailed arguments using the subgroup $2^4\rtimes \Alt(5)$ can deal with Case 2, but the others seem difficult without using this approach.

Let $L$ denote a copy of $2^4\rtimes \Alt(5)$ in $H$, with the normal subgroup $2^4$ denoted by $L_0$. The restrictions of the seven cases for $L(E_8){\downarrow_H}$ to $L$ are as follows.
\begin{center}
\begin{tabular}{cc}
\hline Case & Restriction to $L$
\\ \hline $1$ & $10_1^{11},10_2,5^{20},4,1^{24}$
\\ $2$ & $15^4,10_1^3,10_2^9,5^8,4^5,1^8$
\\ $3$ & $15^3,10_1^3,10_2^9,5^{11},4^4,3_1,3_2,1^6$
\\ $4$ & $15^3,10_1^3,10_2^9,5^{11},4^6,1^4$
\\ $5$ & $15^5,10_1^{10},10_2^6,5,4,1^4$
\\ $6$ & $15^4,10_1^{10},10_2^6,5^4,3_1,3_2,1^2$
\\ $7$ & $15^4,10_1^{10},10_2^6,5^4,4^2$
\\ \hline
\end{tabular}
\end{center}
We will show that except for the first and last cases, there is no such embedding of $L$ in $\mb G$, hence no such embedding of $H$ in $\mb G$.

We suppose first that $L$ is contained in a member of $\ms X$, so we therefore consider the members of $\ms X$, and look for all copies of $L$ inside them. Note that $L$ possesses no (perfect) central extensions by a subgroup of order $3$ or $5$ by a result of Schur (see for example \cite[(33.14)]{aschbacher}). Therefore if $L$ is contained in, for example, $A_4A_4$, then we have an action of $L$ on each $M(A_4)$, and not a central extension of $L$. Note that $1$, $3_i$ and $4$ are representations of $\Alt(5)$ and not $L$, so at least one of $M(A_4)$s must not consist of these factors, so is $5$.
\begin{center}
\begin{tabular}{ccc}
\hline First $M(A_4)$ & Second $M(A_4)$ & Action on $L(E_8)$
\\\hline $5$ & $1^5$ & $10_1^{11},10_2,5^{20},4,1^{24}$
\\ $5$ & $3_1,1^2$ & $15^6,10_1^7,10_2^5,5^2,4^2,3_1^5,1^5$
\\ $5$ & $4,1$ & $15^6,10_1^7,10_2^5,5^2,4^5,3_1,3_2,1^2$
\\ $5$ & $5$ & $15^4,10_1^{10},10_2^6,5^4,4^2$
\\ \hline
\end{tabular}
\end{center}
We see that the first and last options yield embeddings of $L$ in $E_8$ with correct sets of composition factors on $L(E_8)$, corresponding to Cases 1 and 7 respectively.

If $L$ is contained in an $E_7$-Levi subgroup then $L$ cannot have a factor $3_i$ on $M(E_7)$, as there are two copies of $M(E_7)$ in $L(E_8){\downarrow_{E_7}}$ and at most one copy of $3_i$ in $L(E_8){\downarrow_H}$. There are only two such conspicuous sets of composition factors for $M(E_7){\downarrow_L}$: $10_1^2,5^6,1^6$ and $15^2,10_2^2,1^6$. Since $1^{\oplus 3}$ and $M(E_7)^{\oplus 2}$ are submodules of $L(E_8){\downarrow_{E_7}}$, we see that we need at least $1^{15}$ in $L(E_8){\downarrow_L}$, so we must be in Case 1, which we have already seen exists.

If $L\leq A_8$ then again we consider its action on $M(A_8)$, which has dimension $9$. Again, we need a $5$, yielding three possible actions: $5,4$; $5,3_1,1$; $5,1^4$. The last one lies in $A_4$, so is as above, and the other two lie in an $A_4A_3$-parabolic, whose Levi subgroup lies in $A_4A_4$, so again they appear in the previous table.

If $L\leq E_6A_2$ acting diagonally then $L$ must act on $M(A_2)$ as $3_1$ (without loss of generality), and as $M(E_6)$ does not have dimension divisible by $5$, we need a copy of $1$, $3_i$ or $4$. Since $M(E_6)\otimes M(A_2)$ appears twice up to duality (which doesn't matter for our group) and each $3_i$ appears at most once in $L(E_8){\downarrow_H}$, we may assume that neither $3_1$ nor $3_2$ appears in $M(E_6)\otimes M(A_2)$. As $1\otimes 3_1=3_1$, $3_1\otimes 3_1=3_1\oplus P(1)$ and $4\otimes 3_1=3_2\oplus P(4)$, none of these may occur. Thus $M(E_6)$ contains $3_2^4$, but then $M(E_6)\otimes M(A_2)$ contains $4^8$, so that we have $4^{16}$ in $L(E_8){\downarrow_L}$, which is too many. Therefore we have no such $L$ in this case.

If $L$ is contained in $D_8$ then we first assume that $L$ acts on $M(D_8)$ and not a central extension of $L$. In this case we need a set of composition factors for $M(D_8){\downarrow_L}$, but then the trace of an involution in $L$ must be $0$ or $\pm 8$. There are only four such sets of composition factors for $M(D_8){\downarrow_L}$ if one renumbers the $3_i$, and these are
\[ 5^2,1^6,\quad 10_1,5,1,\quad 5^2,3_1^2,\quad 5^2,3_1,3_2.\]
The exterior square of $M(D_8)$ in the third case has $3_1^3$ as composition factors, which is incorrect. The fourth set of composition factors yields an exterior square
\[ P(4)\oplus 15^{\oplus 4}\oplus 10_1^{\oplus 3}\oplus 10_2\oplus 4\oplus 3_1\oplus 3_2.\]
The set of factors of this module is not a subset of any of the composition factors for $L(E_8){\downarrow_L}$, so this is also incorrect. The first case lies in $D_5$, and is therefore in $E_7$. Thus we have factors $10_1,5,1$ on $M(D_8)$, and the composition factors on $L(D_8)$ are $15^2,10_1^6,10_2^2,5^2$. This is compatible with Cases 6 and 7. However, the trace of an element of $L$ of order $5$ on $M(D_8)$ is $1$, and this semisimple class acts on half-spin modules with trace $-2$, and $L(D_8)$ with trace $0$, thus on $L(E_8)$ with trace $-2$. This is consistent with Case 7, and not with Case 6, on which $x$ acts with trace $3$. (The Brauer characters of these two cases differ only on $x$.)

%
On the other hand, suppose that $L$ embeds in $\HSpin_{16}$, but its preimage in $\Spin_{16}$ is a central extension. We first claim that in any (non-split) central extension $2\cdot L$ there is an element of order $4$ that squares to the central involution. In this case, we need to find an element of order $4$ in $\Spin_{16}$ whose image in $\HSpin_{16}$ has order $2$, i.e., it has trace $128$ on the half-spin module. However, there is no such conjugacy class, as can be seen by examining the list of the sixteen classes of elements of order $4$ for $D_8$. Thus this central extension cannot occur.

To see the claim, the easiest way is to construct all three central extensions of $L$ using the \texttt{SmallGroups} command (they are numbers \texttt{241002}, \texttt{241003} and \texttt{241004}) and check it manually. One can see it theoretically by noting that they are quotients of the group $2^{1+4}\rtimes (2\cdot \Alt(5))$.

If $L\leq E_7A_1$ acting diagonally, then $L$ acts on $L(A_1)\leq L(E_8)$ as $3_1$ (without loss of generality). Thus we may only be in Cases 3 and 6. We consider all conspicuous sets of composition factors for $L(E_7){\downarrow_L}$ with no $3_1$ and composition factors a subset of those of Cases 3 or 6. We obtain a single possibility,
\[15^2,10_1^2,10_2^6,5^2,4^2,3_2,1^2\]
which can arise as Case 3. However, the action on $M(E_7)$ is incorrect: of the three central extensions, only one has modules of dimension $2$ (needed to project onto $A_1$), which is $\hat L=2^4\rtimes (2\cdot \Alt(5))$. There is up to field automorphism a unique conspicuous set of composition factors for this group (that is a faithful module for $\hat L$). The tensor product of this with either $2$-dimensional module (recall that $M(E_7)\otimes M(A_1)$ is a factor of $L(E_8){\downarrow_{E_7A_1}}$) has three composition factors of dimension $15$, which means we cannot be in Case 3 (as that has only three copies of $15$ in total). Hence we can eliminate this subgroup as well.

We now finish with the positive-dimensional subgroups. Every Levi subgroup is contained in a maximal-rank subgroup above, so we need only consider $F_4$, $G_2F_4$ and any subgroup $\mb X$ such that $L\not\leq \mb X^0$.

For $\bX=F_4$, there are two possibilities for the action of $L$ on $M(F_4)$: $10_1,5^3$ and $15,10_2$. In these cases, there is no elementary abelian $2$-subgroup of order $8$ consisting of (the identity and) elements of trace $1$. By \cite{griess1991}, we see that $L_0$ is toral, and thus $L$ is contained in the normalizer of a torus. But the Weyl group of $\bX$ is soluble, and $\Alt(5)$ is not. Thus $L$ does not embed in $\bX$. 

For $\mb X=G_2F_4$, first note that $M(G_2)$ appears twice as a composition factor in $L(E_8){\downarrow_{\mb X}}$ (as we are in characteristic $3$). Therefore $3_i$ cannot appear in $M(G_2){\downarrow_L}$. However, $L$ does not embed in $G_2$, so we must act as $L/L_0\cong \Alt(5)$ on this factor, and this can only embed via $A_2$, which acts as $M(A_2)$ plus its dual (and a trivial) or $M(A_2)\otimes M(A_2)^*$ minus two trivials. Both of these restrict to $L$ to contain $3_i$, so we obtain a contradiction in this case.

Finally, we consider the normalizer of a torus. Of course, if $\mb T$ denotes the torus of $N_{\mb G}(\mb T)$ then $L\cap \mb T$ is either $L_0$ or $1$. We may use Magma to enumerate the conjugacy classes of subgroups $L$ and $L/L_0$ in $W(E_8)$ and find there is one of the former and four of the latter. All but one of them centralize at least a $2$-space on the reflection representation, so they lie in parabolic subgroups of $\mb G$. The remaining class is a diagonal $\Alt(5)$ lying in $\Sym(5)\times \Sym(5)=W(A_4A_4)$, and so $L\leq A_4A_4$ in this case.

It remains to show that $L$ is always contained in a positive-dimensional subgroup. We will do this by showing that $L_0$ is always a toral subgroup, using results of Griess \cite{griess1991}. In Cases 5 to 7, the trace of an involution in $L_0$ is $-8$, so class 2B in the notation of \cite{griess1991}. Thus from \cite[Table I]{griess1991} we see that $L_0$ is toral. For the rest, we see from \cite{griess1991} that whether $L_0$ is toral depends on the subset of $L_0$ that consists of 2A-elements, i.e., those with centralizer $E_7A_1$. This is the same subset (as the traces of elements in $L_0$ are the same) for the first four cases, and so if $L_0$ is toral in one of the cases then it is toral in all cases. However, clearly the $L_0\leq L\leq H\leq A_4\leq A_4A_4$ is toral, as $L_0$ is a torus of $H$. This completes the proof that only Cases 1 and 7 can occur. (Notice that this chimes with the case $q=5$, where these were the two cases other than the maximal $B_2$.)

\medskip

\noindent \textbf{Case 1}: The pressure of this module is $-3$, so $H$ stabilizes a line on $L(E_8)$.

\medskip

\noindent \textbf{Case 7}: We first prove that either $5$ or $10$ is a submodule of $L(E_8){\downarrow_H}$. We will then restrict to $L$ and show that the previous statement implies that $H$ is contained in a member of $\ms X$, and then check that in fact $H$ is a blueprint for $L(E_8)$.

Suppose that $L(E_8){\downarrow_H}$ has a socle consisting only of copies of $14$ and $25$. If any of the submodules of the socle are summands they may be removed, so we may assume that there is at most one $14$ in the socle, and at most two $25$s in the socle.

If $14$ lies in the socle then we need to take the $\{5,10,25\}$-radical of $P(14)$, and this is
\[ 10,25/5,10/14,\]
so obviously we need more in the socle. The $\{5,10,25\}$-radical of $P(25)$ is
\[ 25/10/10/25/10/10/5,25/5,10,25/5,10/25.\]
Together these only have eight copies of $10$, so we need $14\oplus 25^{\oplus 2}$ in the socle, but then all $25$s lie in the socle or top of $L(E_8){\downarrow_H}$. Therefore we need to take the $\{5,10\}$-radical of $P(25)$, which is
\[ 5,10/5,10/25.\]
This yields a contradiction, and so $14$ cannot be a factor of $\soc(L(E_8){\downarrow_H})$. Examining the Cartan matrix of the principal block of $H$, we find that there are nine copies of $10$ in $P(25)$, so $\soc(L(E_8){\downarrow_H})$ cannot be $25$. Thus we take the $\{5,10,14\}$-radical of $P(25)$, to yield the module
\[ 5/5,10,14/5,10/25.\]
This clearly does not have enough copies of $10$. Thus we must have either $5$ or $10$ in $\soc(L(E_8){\downarrow_H})$, as claimed.

\medskip

We now consider the restriction to $L$; as we saw above, $L$ lies in $N_{\mb G}(\mb T)$ (but then lies in $A_4A_4$), inside $A_4A_4$ or inside $D_8$.

For $A_4A_4$, we have that $L$ acts as $5$ on both copies of $M(A_4)$. Thus $L$ is determined up to conjugacy in $\mb G$, and acts as
\[ (10_1/10_2/10_1)^{\oplus 4}\oplus 15^{\oplus 4}\oplus 10_1^{\oplus 2}\oplus 10_2^{\oplus 2}\oplus 5^{\oplus 4}\oplus 4^{\oplus 2}.\]
(This is because $L$ acts on $L(A_4)$ as $10_1\oplus 10_2\oplus 4$, and on $M(A_4)\otimes \Lambda^2(M(A_4))$ as $(10_1/10_2/10_1)\oplus 15\oplus 5$.)

Of course, there is an algebraic $B_2$ subgroup acting irreducibly on $M(A_4)$ as $L(\lambda_1)$; since $L(E_8){\downarrow_{A_4A_4}}$ restricts to a diagonal $B_2$ subgroup $\mb Y$ acting in this way with two isomorphic $24$-dimensional summands, and four isomorphic $50$-dimensional summands, we just need to check the structure of these two modules. We see that 
\[ L(\lambda_1)^{\otimes 2}\cong L(0)\oplus L(2\lambda_1)\oplus L(2\lambda_2),\]
and since $\Lambda^2(L(\lambda_1))\cong L(2\lambda_2)$ (as we can see from the above formula), we have
\[ L(\lambda_1)\otimes \Lambda^2(L(\lambda_1))=L(\lambda_1)\otimes L(2\lambda_2)\cong L(\lambda_1)\oplus (L(2\lambda_2)/L(\lambda_1+2\lambda_2)/L(2\lambda_2)).\]
In terms of dimensions, this is $5\otimes 5=1\oplus 10\oplus 14$ and $5\otimes 10=5\oplus (10/25/10)$.

Comparing the decompositions for $L$ and $\mb Y$, we see that every submodule $5$ and $10_1$ is stabilized, so either $H$ is contained in a member of $\ms X$ or it does not stabilize any of these subspaces, i.e., neither $5$ nor $10$ is a submodule of $L(E_8){\downarrow_H}$, but we have already shown that this is impossible. Thus $H$ is contained in a member of $\ms X$ in this case.

Thus we may assume that $L$ is contained in $\mb X=D_8$, and above we showed that $L$ acts on $M(D_8)$ as $10_1\oplus 5\oplus 1$. This clearly lies inside $B_7$, and then inside $B_2D_5$, the stabilizer of the $5\oplus 10$ decomposition, and then inside a copy of $B_2B_2$. The diagonal $B_2$ subgroup of this is contained in $A_4A_4$ (it is subgroup 59 in \cite[Table 13]{thomas2017un}) and so $L$ is conjugate to the embedding in $A_4A_4$.

In particular, this means that $H$ is contained in a member of $\ms X$. Furthermore, since $L$ is contained in only $A_4A_4$ and $D_8$ (and the normalizer of a torus), $H$ can only be embedded in $A_4A_4$ or $D_8$. As above, we see that $H$ is the fixed points of the irreducible $B_2$ subgroup 59 from \cite[Table 10]{thomas2017un}, and the structure of this on $L(E_8)$ is easiest to see through $A_4A_4$, as we computed it above. This is also the structure of $H$, and so $H$ is a blueprint for $L(E_8)$, as needed.

\medskip

\noindent $\boldsymbol{q=9}$: Using the traces of elements only of order $41$, we find up to field automorphism exactly seven conspicuous sets of composition factors for $L(E_8){\downarrow_H}$, namely
\[ 14_1,10_1^{11},5_1^{20},1^{24},\quad 25_{12},16^8,10_1,10_2,5_1^6,5_2^6,1^{15},\quad 25_1^3,16^3,14_1^6,5_1^6,5_2^2,1,\]
\[
25_1^4,14_1^2,10_1^{10},5_1^4,\quad 64_{12}^2,50_{12},25_2,10_1,10_2^3,5_1,\quad  50_{12}^2,50_{21}^2,14_1,14_2,10_1,10_2,\]
\[ 64_{12},64_{21},50_{12},50_{21},10_1,10_2.\]

\noindent \textbf{Cases 1, 2 and 3}: In each case there is a trivial composition factor. No simple $kH$-module of dimension at most $25$ has non-zero $1$-cohomology, so in each case $H$ stabilizes a line on $L(E_8)$.

\medskip

\noindent \textbf{Cases 4, 5, 6 and 7}: In each of these cases the restriction to $L\cong \PSp_4(3)$ has composition factors $25^4,14^2,10^{10},5^4$ on $L(E_8)$. Thus $L$ is a blueprint for $L(E_8)$ by the previous proof, and so therefore is $H$.

\medskip

\noindent $\boldsymbol{q=2}$: This was obtained in \cite[Proposition 6.3]{craven2017}.

\medskip

\noindent $\boldsymbol{q=4}$: Label the simple modules $4_1$ to $4_4$ so that the graph automorphism, which has order $4$, cycles them. In particular, this means that $S^2(4_i)=1/4_{i+1}/1/4_{i+2}$ for all $i$.

We have two conjugacy classes of subgroups of $H$ isomorphic to $\Alt(5)\times \Alt(5)$. Write $y_1$ and $z_1$ for elements of orders $3$ and $5$ from the first factor of one of the classes, and $y_2$ and $z_2$ for elements of orders $3$ and $5$ from the first factor of the other class. Thus the centralizer in $H$ of $y_i$ is $\gen{y_i}\times L_i$ and the centralizer in $H$ of $z_i$ is $\gen{z_i}\times L_i$ for $L_i\cong \Alt(5)$. In particular, $C_H(y_i)'=C_H(z_i)'$.

We will normally focus on one element for each proof, whichever works best. Write $V_1$ for the $1$-eigenspace of this element, and $V_{\theta^i}$ and $V_{\zeta^i}$ for the $\theta^i$- and $\zeta^i$-eigenspace of $L(E_8)$ with respect to this element, where $\theta^3=1$ and $\zeta^5=1$. Each of these forms a module for the appropriate $L_i$.

We label the $L_i$ and the modules for them so that the restriction of $4_1$ to $L_1$ is $1^{\oplus 2}\oplus 2_1$, and the restriction of $4_2$ to $L_2$ is $1^{\oplus 2}\oplus 2_1$ (hence $4_1$ restricts to $L_2$ as $2_2^{\oplus 2}$). We also choose the $z_i$ and $\zeta$ so that $4_1$ has zero $\zeta^2$-eigenspace for $z_1$, and $4_2$ has zero $\zeta$-eigenspace for $z_2$.

Because there are many modules for $H$ and two classes of $L_i$, each with eigenspaces with respect to $y_i$ and $z_i$, we do not give complete information about the actions here, and just inform the reader as and when we use a particular piece of information.

Up to automorphism, there are eight conspicuous sets of composition factors for $L(E_8){\downarrow_H}$ with positive pressure:
\[ 16_{12}^4,4_1^{16},4_2^{16},4_3^6,1^{32},\quad 16_{14}^4,16_{24}^2,4_1^{12},4_2^9,4_3^3,4_4^8,1^{24},\]
\[ 16_{13}^4,16_{14}^2,16_{23}^2,4_1^9,4_2^4,4_3^9,4_4^4,1^{16},\quad 64_{234},16_{13}^2,16_{14}^3,16_{24},4_1^7,4_2^5,4_3^5,4_4^2,1^{12},\]
\[64_{134},16_{13}^2,16_{14}^2,16_{23},16_{24},4_1^6,4_2^4,4_3^5,4_4^4,1^{12},\]\[ 64_{134}^2,16_{12},16_{14},16_{24},4_1^5,4_2^5,4_3,4_4^4,1^{12},\]
\[ 64_{124},64_{234},16_{13}^2,16_{14},16_{23},4_1^4,4_2^2,4_3^4,4_4^2,1^8,\] \[ 64_{123},64_{134},16_{14}^2,16_{23}^2,4_1^4,4_2^2,4_3^4,4_4^2,1^8.\]
There are also three with non-positive pressure, which stabilize a line on $L(E_8)$ by Proposition \ref{prop:pressure}:
\[4_1^{32},4_2^{12},4_3,1^{68},\quad 16_{14}, 16_{34}^8,4_1^9,4_2,4_4^8,1^{32},\quad 
16_{13},16_{24}^8,4_1^8,4_2,4_3^8,4_4,1^{32}\]
Hence they are strongly imprimitive by Lemma \ref{lem:fix1space}.

\medskip

\noindent\textbf{Case 1}: The eigenspaces of $z_2$ on $L(E_8)$ have dimensions $64$, $56$, $36$, $36$ and $56$, so $z_2$ lies in the class with centralizer $A_7T_1$ from Table \ref{t:semie8}. The composition factors of $V_1$ are $2_1^{16},1^{32}$, so $L_2$ must act on $M(A_7)$ as $2_2^{\oplus 4}$. It is easy to see that the copy of $\PSL_2(4)$ inside $A_7$ acting this way is a blueprint for $L(E_8)$ with the $A_1$ subgroup of $A_7$ acting on $M(A_7)$ as $L(2)^{\oplus 4}$ stabilizing the same eigenspaces on $\Lambda^i(M(A_7))$ for $i=1,2,3$ (i.e., $L(\lambda_i)$ for $i=1,2,3,5,6,7$) and on $M(A_7)^{\otimes 2}$ (i.e., $L(\lambda_1+\lambda_7)$).

Thus $H$ is a blueprint for $L(E_8)$.

\medskip

\noindent\textbf{Case 2}: The eigenspaces of $z_1$ on $L(E_8)$ have dimensions $54,46,51,51,46$, so $z_1$ has centralizer $A_2D_5T_1$ from Table \ref{t:semie8}. The composition factors of $V_1$ are $2_1^{12},2_2^3,1^{24}$, and this is $L(D_5)\oplus L(A_2)\oplus L(T_1)$. This means that $M(D_5){\downarrow_{L_1}}$ cannot have $4\oplus 1$, $4^{\oplus 2}$ or $2_1\oplus 2_2$ in it. In fact, the only solution is that $L_1$ acts trivially on $M(A_2)$ (so $L_1$ is contained in $D_5$) and with factors $2_1^3,1^4$ on $M(D_5)$. The composition factors of $L(\lambda_4){\downarrow_{L_1}}$ are $2_2^4,4^2$, so the module is always semisimple. (We use this notation for simple modules to alert the reader as to which factor of the centralizer of $z_1$ the simple module is for.)

The action of $D_5$ on the $\zeta^i$-eigenspaces of $z_1$ are $L(0)^{\oplus 3}\oplus L(\lambda_4)^{\oplus 3}$ for $\zeta$ and $L(\lambda_1)^{\oplus 3}\oplus L(\lambda_4)$ for $\zeta^2$, each up to duality (but since $L_1$ acts the same on $L(\lambda_4)$ and $L(\lambda_5)$ we do not work out the exact distribution).

There are five possible modules for $M(D_5){\downarrow_{L_1}}$, combining various summands of the form $1/2_1/1$, $(1/2_1)\oplus (2_1/1)$, $2_1$ and $1$. In each case the obvious $A_1$ subgroup $\mb Y$ is a blueprint for $M(D_5)=L(\lambda_1)$ and $L(D_5)=L(\lambda_2)$. Hence $L_1$ and $\mb Y$ stabilize the same subspaces of $V_1$ and any subspace of $V_\zeta$ whose composition factors are only trivials and copies of $2_1$. In particular, $\mb Y$ stabilizes any submodule $4_1$ or $4_2$ in $L(E_8){\downarrow_H}$, since these have only composition factors $1$ and $2_1$ on the $1$, $\zeta$- and $\zeta^4$-eigenspaces, and $\zeta^{\pm 2}$ is not an eigenvalue of them. Furthermore, since clearly $2_2$ is not a submodule of $L(D_5)$ (as it only appears inside $2_1^{\otimes 2}=1/2_2/1$) we can have no submodule $4_3$ in $L(E_8){\downarrow_H}$.

Thus if $H$ is not Lie imprimitive then the socle and top consist of copies of $4_4$, $16_{14}$ and $16_{24}$. The $\{1,4_i,16_{14},16_{24}\}$-radical of $P(16_{24})$ is
\[ 1,1/4_1,4_3/1/4_2,4_4/16_{24},\]
and removing all quotients that cannot be quotients of $L(E_8){\downarrow_H}$ leaves just $4_4/16_{24}$.

Thus we consider the $\{1,4_i,16_{14},16_{24}\}$-radicals of $P(4_4)$ and $P(16_{14})$, and then take the $\{4_4,16_{14},16_{24}\}'$-residuals, to obtain the modules
\[ 4_4/1/4_1/1/4_2/1/4_1,4_4/1,16_{14}/4_1/1,1/4_2,4_4/1/4_1,4_4/1,16_{24}/4_4\]
and
\[ 16_{14}/4_1/1/4_2,4_4/1,1/4_1,4_3/1/4_2/1/4_1/16_{14}.\]
But we need three copies of $4_3$ in $L(E_8){\downarrow_H}$, so we need $16_{14}^{\oplus 3}$ in the socle, and another $16_{14}^{\oplus 3}$ in the top, but we only have four copies of $16_{14}$ in $L(E_8){\downarrow_H}$. Thus two must be summands, and of course we obtain a contradiction.

To see strong imprimitivity, note that the composition factors of $L(E_8){\downarrow_H}$ are not stable under any outer automorphism of $H$, so we may apply Lemma \ref{lem:semilinearfield}.

\medskip

\noindent\textbf{Case 3}: The trace of $z_2$ on $L(E_8)$ is $-2$, so $z_2$ has centralizer $A_4A_4$. The composition factors of $L(E_8){\downarrow_{A_4A_4}}$ are $(1001,0000)$ and $(0000,1001)$ for the $1$-eigenspace of $z$ (we are suppressing the `$L(-)$'), and the other four modules are
\[ (1000,0100),\quad (0100,0001),\quad (0010,1000),\quad (0001,0010).\]
These four modules can be permuted by taking duals and swapping the two $A_4$-factors, so we may assume that the first one is the restriction of the $\zeta$-eigenspace of $z$ on $L(E_8)$ to $A_4A_4$. With that assignment, the actions of $A_4A_4$ on the $\zeta$, $\zeta^2$, $\zeta^3$ and $\zeta^4$-eigenspaces must be the modules above in the same order.

The action of $L_2$ on $V_1$ has factors $2_1^8,2_2^8,1^{16}$, and the action on $V_\zeta$ has factors $4^4,2_1^9,2_2^4,1^8$. This means that there is only one possibility for the action of $L_2$ on the two modules $M(A_4)$, and that is on the first copy of $M(A_4)$ with factors $2_1^2,1$, and on the other with factors $2_2^2,1$. This yields nine possibilities for the action of $L_2$ on $L(E_8)$, as we can act as $(1/2_i)\oplus 2_i$, $(2_i/1)\oplus 2_i$, or $2_i^{\oplus 2}\oplus 1$.

Suppose that every semisimple submodule of $L(E_8){\downarrow_{L_2}}$ is stabilized by some positive-dimensional subgroup $\mb Y$ of $\mb G$. Any submodule of $L(E_8){\downarrow_H}$ that restricts semisimply to $L_2$ is therefore also stabilized by $\mb Y$. The modules $1$, $4_i$ and $16_{13}$ all restrict semisimply to $L_2$, and so if the condition holds, and any of those are submodules of $L(E_8){\downarrow_H}$, then $H$ is strongly imprimitive by Theorem \ref{thm:intersectionorbit} and Lemma \ref{lem:semilinearfield}. The remaining two composition factors are $16_{14}$ and $16_{23}$: the former has $\zeta^{\pm 1}$-eigenspace $1/2_2/1$ and others semisimple, and the latter has $\zeta^{\pm 2}$-eigenspace $1/2_1/1$ and others semisimple.

Suppose that the two copies of $M(A_4)$ are $2_1^{\oplus 2}\oplus 1$ and $(1/2_2)\oplus 2_2$ or its dual. We set $\mb Y$ to be an $A_1$ subgroup of $A_4A_4$ acting on the first module as $L(1)^{\oplus 2}\oplus L(0)$ and on the second as $(L(0)/L(2))\oplus L(2)$ or its dual. In this case, it is an easy check that all semisimple submodules of $L(E_8){\downarrow_{L_2}}$ are stabilized by $\mb Y$. There is no copy of $2_2$ in the second socle layer of one of $V_\zeta$ or $V_{\zeta^4}$ in this case (depending on which possibility for $L_2$ one takes) so $16_{14}$ cannot be a submodule of $L(E_8){\downarrow_H}$. There is no copy of $2_1$ in the second socle layer of $V_{\zeta^{\pm 2}}$ either, so $16_{23}$ cannot be a submodule of $L(E_8){\downarrow_H}$ either. Thus $H$ is strongly imprimitive. 

On the other hand, if the two copies are $(1/2_1)\oplus 2_1$ or its dual, and $2_2^{\oplus 2}\oplus 1$ then we set $\mb Y$ to act as $(L(0)/L(4))\oplus L(4)$ and $L(2)^{\oplus 2}\oplus L(0)$, and again we stabilize every semisimple submodule of $L(E_8){\downarrow_{L_2}}$. As before, there is no copy of $2_2$ in the second socle layer of $V_{\zeta^{\pm 1}}$ and no copy of $2_1$ in the second socle layer of either $V_{\zeta^2}$ or $V_{\zeta^3}$, so neither $16_{14}$ nor $16_{23}$ are submodules of $L(E_8){\downarrow_H}$. Thus $H$ is strongly imprimitive.

If both modules are semisimple then we again embed in $\mb Y$ acting as $L(1)^{\oplus 2}\oplus L(0)$ and $L(2)^{\oplus 2}\oplus L(0)$ and obtain the same conclusion.

This eliminates five possibilities for the action of $L_2$, leaving the four where neither of the two modules is semisimple. To eliminate these, we show that $L_2$ cannot embed in this fashion, by considering the centralizer of $y_2$, which is $\gen{y_2}\times L_2$. The trace of $y_2$ on $L(E_8)$ is $-4$, hence $y_2$ has centralizer $A_8$. The composition factors of $L_2$ on the $1$-eigenspace of $y_2$ are $4^8,2_1^8,2_2^8,1^{16}$, and these are consistent with the composition factors of $M(A_8){\downarrow_{L_2}}$ being either $4^2,1$ or $2_1^2,2_2^2,1$. But the former has the wrong action on the rest of $L(E_8)$. Since there is a single trivial composition factor in $M(A_8){\downarrow_{L_2}}$, we have very few possible modules: it is easy to see that one $2_1$ and one $2_2$ must split off as summands, and the rest is up to duality one of
\[ 2_1\oplus 2_2\oplus 1,\quad 2_1\oplus (1/2_2),\quad 2_2\oplus (1/2_1),\quad 1/2_1,2_2,\quad 2_1/1/2_2.\]
This yields five actions of $L_2$ on $L(E_8)$. The first three of these appear as well in the actions of $L_2$ on $L(E_8)$ via $A_4A_4$, where one of the two modules is semisimple. However, if neither module $M(A_4){\downarrow_{L_2}}$ is semisimple, then $L(E_8){\downarrow_{L_2}}$ does not match up with any of the five possibilities arising from $A_8$. Thus these four cannot yield embeddings of $L_2$, and the proof is complete.

\medskip

\noindent\textbf{Case 4}: The restriction $L(E_8){\downarrow_{L_1}}$ has composition factors $4^{20},2_1^{29},2_2^{29},1^{52}$. We note that by just checking all possibilities, there is no copy of $L_1$ inside $A_8$ that has those factors on $L(E_8)$. (We recall that the composition factors of $A_8$ on $L(E_8)$ are $M(A_8)\otimes M(A_8)^*$ (minus a trivial), $\Lambda^3(M(A_8))$ and $\Lambda^3(M(A_8)^*)$). Thus $L_1$ cannot embed in $A_8$, or in any parabolic whose Levi subgroup is contained in $A_8$.

However, the trace of $z_1$ on $L(E_8)$ is $3$, so $z_1$ has centralizer $A_6A_1$. Hence $L_1$ embeds in $A_6A_1$, which can be conjugated into $A_8$. (In fact, $L_1$ must have a trivial composition factor on $M(A_6)$ so there is a version with the same composition factors on $L(E_8)$ that lies in $A_1A_5$, directly in $A_8$.) This contradiction means that $H$ cannot embed into $E_8$ with these factors.

\medskip

\noindent\textbf{Case 5}: This time the trace of $z_2$ is $3$, so $z_2$ has centralizer $A_6A_1$. The composition factors of $V_1$ are $4^3,2_1^8,2_2^6,1^{12}$, and this yields composition factors of $L_2$ on $M(A_6)$ and $M(A_1)$ of $4,2_2,1$ and $2_2$ respectively. We therefore have three possibilities for the action of $L_2$ on $L(E_8)$ and on each eigenspace of $z_2$.

The $L_2$-actions of the composition factors of $L(E_8){\downarrow_H}$ on the eigenspaces of $z_2$ are given in Table \ref{t:L2z2actions}.
\begin{table}\begin{center}
\begin{tabular}{cccc}
\hline Module & $1$-eigenspace & $\zeta$-eigenspace & $\zeta^2$-eigenspace
\\\hline $1$ & $1$ & - & -
\\ $4_1$ & - & $2_1$ & -
\\ $4_2$ & $2_2$ & - & $1$
\\ $4_3$ & - & - & $2_2$
\\ $4_4$ & $2_1$ & $1$ & -
\\ $16_{13}$ & - & $4$ & $4$
\\ $16_{14}$ & $2_1^{\oplus 2}$ & $(1/2_2/1)$ & $2_1$ 
\\ $16_{23}$ & $2_2^{\oplus 2}$ & $2_2$ & $(1/2_1/1)$
\\ $16_{24}$ & $4$ & $1\oplus 2_2$ & $1\oplus 2_1$
\\ $64_{134}$ & $4^{\oplus 2}$ & $P(2_2)\oplus 4$ & $P(2_2)\oplus 4^{\oplus 2}$
\\ \hline
\end{tabular}
\end{center}
\caption{Actions of $L_2$ on $z_2$-eigenspaces of simple $kH$-modules.}
\label{t:L2z2actions}\end{table}
The dimensions of $V_\zeta$ and $V_{\zeta^2}$ are the same, but they are not isomorphic modules up to duality. By considering the $L_2$-composition factors we can tell them apart, and we know that the action of $A_6A_1$ on the $\zeta$-eigenspace of $L(E_8)$ is the sum of $(\lambda_3,0)$ or its dual and $(\lambda_1,1)$ or its dual, and the action of $A_6A_1$ on the $\zeta^2$-eigenspace of $L(E_8)$ is the sum of $(\lambda_1,0)$ or its dual and $(\lambda_2,1)$ or its dual.

Suppose that $L_2$ acts semisimply on $M(A_6)$. The actions of $L_2$ on $(\lambda_3,0)$ and $(\lambda_1,1)$ are
\[ \quad P(2_1)\oplus P(2_2)\oplus 4^{\oplus 3}\oplus (1/2_1,2_2/1)\oplus 1\quad \text{and}\quad P(2_1)\oplus (1/2_1/1)\oplus 2_2.\]
From this we see that $16_{14}$ cannot be a submodule of $L(E_8){\downarrow_H}$, since $1/2_2/1$ is not a submodule of $V_\zeta$.

In a similar vein to our proofs for $\PSU_3(4)$, we embed $L_2$ in $\mb Y$ acting on $M(A_6)$ as $L(6)\oplus L(2)\oplus L(0)$ and on $M(A_1)$ as $L(2)$. With these factors, $L_2$ and $\mb Y$ stabilize all simple submodules of $L(E_8){\downarrow_{L_2}}$ other than those isomorphic to $4$ that lie in $V_1$. This means that $\mb Y$ stabilizes any submodule of $L(E_8){\downarrow_H}$ isomorphic to $1$, $4_i$ or $16_{13}$. Thus either $H$ is strongly imprimitive or the socle contains some of $16_{23}$, $16_{24}$ and $64_{134}$, but each such submodule must be a summand. Since this means there is another submodule of $L(E_8){\downarrow_H}$, this proves that $H$ is strongly imprimitive.

Thus we may assume that $L_2$ acts on $M(A_6)$ as either $(1/2_1)\oplus 4$ or $(2_1/1)\oplus 4$; up to replacement of $z_2$ by $z_2^{-1}$ we choose the first one. 

We prove that there is no embedding of $\bar L_2$, the maximal subgroup $\Alt(5)\wr 2$ of $H$ containing $L_2$, with this action. We first prove that if $\bar L_2$ has these composition factors on $L(E_8)$ then it can only be embedded in $D_8$, and then that any such embedding has the previous action on $L(E_8)$.

Since $\bar L_2$ is not an almost simple group, it is contained in a member of $\ms X^\sigma$. We will use algebraic group notation for the modules for the group $\bar L_2'\cong \Alt(5)\times \Alt(5)$, writing $(0,0)$ for the trivial module, $(1,0)$, $(0,1)$, $(2,0)$ and $(0,2)$ for the four $2$-dimensional modules, and so on.

The only normal subgroups of $\bar L_2$ are $1$, $\bar L_2$ itself and $\bar L_2'$ of index $2$. If $\bar L_2$ embeds in a product $\mb X_1\mb X_2$ and there is no subgroup isomorphic to $\bar L_2$ in $\mb X_2$, then $\bar L_2'\leq \mb X_1$. We will prove that there are a few positive-dimensional subgroups that cannot contain $\bar L_2'$, and hence eliminate certain maximal subgroups from containing $\bar L_2$.

For example, there is no set of composition factors for $M(A_8){\downarrow_{\bar L_2'}}$ that yields a subgroup with the correct composition factors on $L(E_8)$, and so $\bar L_2'$ cannot be contained in any subgroup conjugate to a subgroup of $A_8$. Using the remark above as well, this eliminates the $A_7$-, $A_1A_6$- and $A_1A_2A_4$-parabolics. Since $\bar L_2$ must stabilize a line or hyperplane on $M(A_4)$, this means that the $A_4A_4$ maximal-rank subgroup, and $A_4A_3$-parabolic subgroups are also eliminated.

Suppose that $\bar L_2$ lies in $E_7$. There are $24776$ sets of composition factors for $\bar L_2$ on $M(E_7)$, fourteen of which are conspicuous. However, none of these has traces on elements of orders $3$ and $5$ that fuse to the correct classes of $E_8$, so $\bar L_2$ cannot embed in $E_7$. Hence $\bar L_2$ cannot embed in an $E_7$-parabolic either.

This is enough to eliminate the $E_6A_1$- and $D_5A_2$-parabolic subgroups, the $E_6A_2$ maximal-rank subgroup ($\bar L_2'$ must lie inside the $E_6$ factor, but then the composition factors of $L(E_8){\downarrow_{\bar L_2'}}$ must be distributed among $M(E_7)$ and $L(E_7)$ just as if it were extensible to $\bar L_2$, so cannot exist), and the $F_4G_2$ maximal subgroup. It can also eliminate the $E_7A_1$ maximal-rank subgroup, as in characteristic $2$ this is the direct product of $E_7$ and $A_1$, so any subgroup $\bar L_2$ containing $\bar L_2'$ lies entirely inside $E_7$.

The remaining maximal subgroups are the $D_7$-parabolic, the $D_8$ maximal-rank subgroup, and any maximal-rank subgroups whose component group contains an involution where $\bar L_2'$ can embed in the connected component, or $\bar L_2$ can embed in the component group. However, $\bar L_2$ cannot embed in the Weyl group $W(E_8)$ by a direct computer check.

There are four possible sets of composition factors for $M(D_8){\downarrow_{\bar L_2'}}$:
\[ (1,0),(0,1),(2,0),(0,2),(2,2)^2,\quad (2,3),(3,2),\]
\[ (0,0)^2,(1,0),(2,0),(0,1),(2,3),\quad (0,0)^2,(1,0),(0,1),(0,2),(3,2).\]
Only the first two of these can yield a copy of $\bar L_2$ rather than $\bar L_2'$; thus $\bar L_2$ cannot embed in the $D_7$-parabolic subgroup. We show that in the first two cases, $\bar L_2$ does not embed in such a way as the restriction to $L_2\times \gen{z_2}$ is consistent with the remaining possibility discussed above. 

The $L_2$-actions on the two modules $(\lambda_3,0)$ and $(\lambda_1,1)$, given the action $(1/2_1)\oplus 4$ on $M(A_6)$ and $2_2$ on $M(A_1)$, are
\[ 4^{\oplus 3}\oplus P(2_1)\oplus P(2_2)\oplus (1/2_1,2_2/1)\oplus 1\quad \text{and}\quad (2_2/1/2_1/1)\oplus P(2_1).\]

The exterior square $\Lambda^2(M(D_8))$ and $L(D_8)$ differ by the precise extensions of the trivial modules on the non-trivial factor of the Lie algebra. Thus the $\zeta^i$-eigenspaces for $i\neq 0$ of $L(D_8)$ and $\Lambda^2(M(D_8))$ are the same, so we can use this exterior square to understand the $\bar L_2'$-action on $L(E_8)$.

If $\bar L_2'$ acts on $M(D_8)$ as $(2,3)\oplus (3,2)$ then $(2,3)\otimes (3,2)$ is a summand of the $\bar L_2'$-action on the exterior square of  $M(D_8)$. The restriction to $L_2\times \gen{z_2}$ of the $k\bar L_2'$-module $(2,3)\otimes (3,2)$ has $\zeta$-eigenspace $P(2_1)^{\oplus 2}$, which is not a summand of the options for $V_\zeta$ given above.

Similarly, since $(1,0)$ and $(0,1)$ have no extensions with $(2,2)$, in the first option for the $\bar L_2'$-action on $M(D_8)$, $(1,0)\oplus (0,1)$ must be a summand. The tensor product $(1,0)\otimes (0,1)$ restricts to $L_2\times \gen {z_2}$ with $\zeta$-eigenspace simply $2_1$. This therefore must be a summand of $V_\zeta$, but of course it is not.

This deals with $\bar L_2$ being contained in $D_8$. For maximal-rank subgroups whose component group contains an involution, we have the $D_4D_4$, $A_2^4$ and $A_1^8$ maximal-rank subgroups. For each of these, their connected component lies in $D_8$, and in each case $\bar L_2'$ must be contained in $D_8$, so is one of the four options above. Notice that these maximal-rank subgroups cannot stabilize lines on $M(D_8)$, so $\bar L_2'$ must act with one of the first two actions given above; but the second was eliminated completely and the first yielded a contradiction with the action of $L_2$ on $M(A_6)$. Thus these cannot occur, and the result is proved.

\medskip

\noindent\textbf{Case 6}: The $64_{134}$ only has an extension with $16_{14}$, but there is no module of the form $64_{134}/16_{14}/64_{134}$; the $\{64_{134}\}$-heart of $L(E_8){\downarrow_H}$ is either of this form or $64_{134}^{\oplus 2}$, so it is the latter. Thus there are two summands $64_{134}$ in $L(E_8){\downarrow_H}$. Any $16$-dimensional simple submodule is a summand, since it appears with multiplicity $1$.

This time, for ease of calculation we consider the action of $y_1$ rather than $z_1$. The trace of $y_1$ on $L(E_8)$ is $-4$, so $y_1$ has centralizer $A_8$, and the composition factors on $V_1$ are $4^6,2_1^7,2_2^{11},1^{20}$, so $L_1$ acts on $M(A_8)$ with factors $2_1,2_2^3,1$. It is easy to see that any such module is the sum of $2_2^{\oplus 2}$ and one of the nine modules with composition factors $2_1,2_2,1$, as we will see in Case 8 of $\PSU_3(4)$, in Section \ref{sec:diffpsu34}. These nine modules are
\[2_1\oplus 2_2\oplus 1,\quad (1/2_1)\oplus 2_2,\quad (2_1/1)\oplus 2_2,\quad (1/2_2)\oplus 2_1,\quad (2_2/1)\oplus 2_1,\]
\[1/2_1,2_2,\quad 2_1,2_2/1,\quad 2_1/1/2_2,\quad 2_2/1/2_1.\]
However, for six of these the module $(M(A_8)\otimes M(A_8)^*){\downarrow_{L_1}}$ does not have a summand $2_2$, whereas the $1$-eigenspace of $y_1$ on $64_{134}$ has $L$-action isomorphic to $2_2^{\oplus 2}$. The remaining three modules are $2_2^{\oplus 2}$ plus one of $2_1\oplus 2_2\oplus 1$, $(1/2_1)\oplus 2_2$ and $(2_1/1)\oplus 2_2$.

In the first case, we embed $L_1$ into an algebraic $A_1$ subgroup $\mb Y$ of $A_8$ acting on $M(A_8)$ as $L(0)\oplus L(1)\oplus L(2)^{\oplus 3}$, and note that the tensor square and exterior cube of this module are
\[ (0/2/0)\oplus (0/4/0)^{\oplus 9}\oplus 5^{\oplus 6}\oplus 2^{\oplus 6}\oplus 1^{\oplus 2}\]
and
\[ (0/4/0)^{\oplus 3}\oplus (1/5/1)^{\oplus 3}\oplus 5^{\oplus 3}\oplus 6\oplus 1^{\oplus 3}\oplus 2^{\oplus 11}\oplus 0^{\oplus 4}\]
respectively, where we have removed the $L(-)$ for clarity. This is of course the action of $A_1$ on the $kL_1$-module $V_{\theta^i}$, so we see that $L_1$ and $\mb Y$ stabilize the same semisimple submodules that do not have $4$ as a composition factor of $V_{\theta^i}$ and $V_1$. Hence $\mb Y$ stabilizes any submodule of $L(E_8){\downarrow_H}$ whose restriction to $L_1$ is semisimple but does not involve $4$, for example, each $4_i$. But one of these must be a submodule (or $H$ stabilizes lines on $L(E_8)$ and so is strongly imprimitive anyway) as all other factors are multiplicity free, hence $H$ is strongly imprimitive by Theorem \ref{thm:intersectionorbit} (and Lemma \ref{lem:fix1space}).

The other possibility is that $L_1$ acts on $M(A_8)$ as $(1/2_1)\oplus 2_2^{\oplus 3}$ or its dual, in which case we embed $L_1$ into $\mb Y$ acting as $(L(0)/L(4))\oplus L(2)^{\oplus 3}$. The product of this and its dual, and the exterior cube of this module, are
\[ (4/0/8/0/4)\oplus (0/4/0)^{\oplus 9}\oplus 6^{\oplus 6}\oplus 2^{\oplus 6}\]
and
\[ (4/0/8/0/4)^{\oplus 3}\oplus (0/4/0)^{\oplus 3}\oplus (0/4)^{\oplus 3}\oplus 6^{\oplus 4}\oplus 2^{\oplus 11}\oplus 0\]
respectively. Every semisimple submodule of $V_1$ and $V_{\theta^i}$ is therefore stabilized by $\mb Y$ in this case, so again $H$ is strongly imprimitive, as before.

\medskip

\noindent\textbf{Case 7}: The trace of $z_1$ is $-2$, so has centralizer $A_4A_4$ from Table \ref{t:semie8}. The composition factors of $V_1$ are $4^6,2_1^4,2_2^4,1^8$, so $L_1$ acts on each $M(A_4)$ as $4\oplus 1$. This means that, ignoring copies of $4$, $V_1$ is isomorphic to $P(1)^{\oplus 2}$ and each $V_{\zeta^i}$ is isomorphic to \[P(1)\oplus P(2_1)\oplus P(2_2)\oplus (1/2_1,2_2/1).\]
Furthermore, both $64$-dimensional modules have no extensions with the other composition factors, so split off as summands. The action of $L_1$ on the $\zeta^i$-eigenspace of the sum $64_{124}\oplus 64_{234}$ is $P(2_1)\oplus P(2_2)\oplus 4^{\oplus 3}$, so any composition factor of $\soc(L(E_8){\downarrow_H})$ other than the $64$-dimensionals must restrict to $L_1\times \gen{z_1}$ with only trivial and $4$-dimensional simple submodules. The only simple modules for $H$ that restrict in this way are $1$ and $16_{24}$, with the latter not a composition factor of $L(E_8){\downarrow_H}$. Thus $H$ stabilizes a line on $L(E_8)$, as needed.

\medskip

\noindent\textbf{Case 8}: The trace of $y_1$ on $L(E_8)$ is $-4$, hence $y_1$ has centralizer $A_8$. The composition factors of $V_1$ are $4^4,2_1^{12},2_2^{12},1^{16}$, but of the (up to field automorphism) thirteen possible sets of composition factors for $M(A_8){\downarrow_{L_1}}$, none has these composition factors. ($2_1^2,2_2,1^3$ has the correct number of $4$s, but has $2_1^{13},2_2^{10}$.) Thus this case cannot exist.

\medskip

\noindent $\boldsymbol{q=8}$: There are far too many sets of composition factors for a module of dimension $248$, so we need to make some reductions first.

Write $1^a,4^b,16^c,64^d$ for the dimensions of the composition factors of $L(E_8){\downarrow_H}$. By considering a unipotent element of order $4$, which acts on $L(E_8)$ with at most 60 blocks of size $4$ by Table \ref{t:unipe8p4}, and acts only with blocks of size $4$ on non-trivial simple modules, we see that $a\geq 8$. By pressure we may assume that $b>a$, since $H^1(H,M)\neq 0$ if and only if $\dim(M)=4$ for simple modules by \cite{sin1992b}. Finally, the trace of a rational element of order $5$ is either $-2$ or $23$ (see Table \ref{t:semie8}), and is also $a-b+c-d$. These constraints yield $23$ solutions for $(a,b,c,d)$:
\[( 8 , 12 , 4 , 2 )\quad
( 12 , 15 , 3 , 2 )\quad
( 16 , 18 , 2 , 2 )\quad
( 20 , 21 , 1 , 2 )\quad
( 8 , 16 , 7 , 1 )\quad
( 12 , 19 , 6 , 1 )\]
\[( 16 , 22 , 5 , 1 )\quad
( 20 , 25 , 4 , 1 )\quad
( 24 , 28 , 3 , 1 )\quad
( 28 , 31 , 2 , 1 )\quad
( 32 , 34 , 1 , 1 )\quad
( 36 , 37 , 0 , 1 )\]
\[( 8 , 20 , 10 , 0 )\quad
( 12 , 23 , 9 , 0 )\quad
( 16 , 26 , 8 , 0 )\quad
( 20 , 29 , 7 , 0 )\quad
( 24 , 32 , 6 , 0 )\quad
( 28 , 35 , 5 , 0 )\]
\[( 32 , 38 , 4 , 0 )\quad
( 36 , 41 , 3 , 0 )\quad
( 40 , 44 , 2 , 0 )\quad
( 44 , 47 , 1 , 0 )\quad
( 48 , 50 , 0 , 0 )\]

We can eliminate solutions $11,12,20,21,22,23$ by taking a subgroup $\gen x$ of order $7$ stable under the graph automorphism and summing the traces of $x$, $x^2$ and $x^4$ on the simple modules. This yields an integer, which we then compare to the possible such integers from $E_8$, finding that those cases above cannot occur.

We can eliminate another one using module structure: the $\{1,4_i\}$-radical of $P(4_1)$ is
\[ 4_1/1/4_2,4_6/1,1/4_1,4_3,4_5/1,1/4_2,4_6/1/4_1,\]
where $S^2(4_i)=1/4_{i+1}/1/4_{i+2}$. Thus we again see that the 23rd and 12th cases $(48,50,0,0)$ and $(36,37,0,1)$ cannot occur, but these have already been eliminated. If $L(E_8){\downarrow_H}$ has pressure $1$, then we split $L(E_8){\downarrow_H}$ into sections by taking the $\{1,4_i\}$-radical and then any $16$s and $64$s, and then repeating this. By the above structure, we can obtain at most four trivials before we must hit either a $16$ or the top of the module, so we must have $a\leq 4(c+1)$, removing the fourth case $(20,21,1,2)$, a possibility with pressure $1$.

We check the rest of the cases, using the traces of elements of order at most $21$, and find up to automorphism eleven conspicuous sets of composition factors for $L(E_8){\downarrow_H}$ (these are all conspicuous for elements of orders $63$ and $65$ as well). These are
\begin{small}
\[ 64_{126},64_{236},16_{13}^2,16_{14},16_{23},4_1^4,4_2^2,4_3^4,4_4^2,1^8,\;\; 64_{136},64_{346},16_{14}^2,16_{15},16_{24},4_1^4,4_2^2,4_4^4,4_5^2,1^8,\]
\[ 64_{156}^2,16_{12},16_{16},16_{26},4_1^5,4_2^5,4_3,4_6^4,1^{12},\quad
64_{256}^2,16_{13},16_{16},16_{36},4_1^5,4_2,4_3^4,4_4,4_6^4,1^{12},\] \[64_{356}^2,16_{14},16_{16},16_{46},4_1^5,4_2,4_4^4,4_5,4_6^4,1^{12},\quad
64_{246}^2,16_{13},16_{15},16_{35},4_1^4,4_2,4_3^4,4_4,4_5^4,4_6,1^{12},\]
\[64_{156},16_{15}^2,16_{16}^2,16_{25},16_{26},4_1^6,4_2^4,4_3,4_5^4,4_6^4,1^{12},\quad 16_{15}^4,16_{16}^2,16_{25}^2,4_1^9,4_2^4,4_3,4_5^8,4_6^4,1^{16},\]
\[16_{14}^4,16_{15}^2,16_{24}^2,4_1^8,4_2^4,4_3,4_4^8,4_5^4,4_6,1^{16},\quad 16_{16}^4,16_{26}^2,4_1^{12},4_2^9,4_3^3,4_6^8,1^{24}\]
\[ 16_{12}^4,4_1^{16},4_2^{16},4_3^6,1^{32}.\]
\end{small}
We will use Lemma \ref{lem:allcasesstronglyimp}, so we look for elements of order $195$, powering to an element $y\in H$ of order $65$, and stabilizing certain eigenspaces of $y$ on $L(E_8)$. We employ the strategy outlined in Remark \ref{rem:howusestrongimp}.

For the eighth, tenth and eleventh cases, there are elements of order $195$ stabilizing all eigenspaces of $y$, so these cases are certainly strongly imprimitive. For the others, we give the table of which modules can have their eigenspaces stabilized by roots of $y$.

\begin{center}
\begin{tabular}{cc}
\hline Case & Modules
\\ \hline $1$ & $1$, $4_1$, $4_2$, $4_3$, $4_4$
\\ $2$ & $1$, $4_1$, $4_2$, $4_4$, $4_5$
\\ $3$ & $1$, $4_1$, $4_2$, $4_3$, $4_4$, $16_{16}$, $16_{26}$
\\ $4$ & $1$, $4_1$, $4_2$, $4_3$, $4_4$, $4_6$, $16_{16}$
\\ $5$ & $1$, $4_1$, $4_2$, $4_4$, $4_5$, $4_6$, $16_{16}$
\\ $6$ & $1$, $4_1$, $4_2$, $4_3$, $4_4$, $4_5$, $4_6$
\\ $7$ & $1$, $4_1$, $4_2$, $4_3$, $4_5$, $4_6$, $16_{15}$, $16_{16}$, $16_{25}$, $16_{26}$
\\ $9$ & $1$, $4_1$, $4_2$, $4_3$, $4_4$, $4_5$, $4_6$, $16_{14}$
\\ \hline
\end{tabular}
\end{center}

\noindent \textbf{Case 1}: We have no possible automorphisms that permute the factors, so if $H$ is not strongly imprimitive, the socle of $L(E_8){\downarrow_H}$ must consist of summands of $L(E_8){\downarrow_H}$ and $16_{13}$, so we need the $\{1,4_i\}$ radical of $P(16_{13})$. This is
\[ 1/4_2/1/4_1,4_3/16_{13}.\]
We apply Proposition \ref{prop:bottomhalf} with the set $I=\{64_{126},64_{236},16_{14},16_{23}\}$. Since each of these appears in $L(E_8){\downarrow_H}$ with multiplicity $1$, the $I$-heart is semisimple. Thus by the proposition at least four trivials appear in the $I'$-radical of $L(E_8){\downarrow_H}$. But this radical us a submodule of the module above, a clear contradiction. Thus $L(E_8){\downarrow_H}$ contains a submodule from the list in the table, and thus $H$ is strongly imprimitive, using Lemma \ref{lem:allcasesstronglyimp}.

\medskip

\noindent \textbf{Case 2}: Suppose that $H$ is not strongly imprimitive, and we aim to apply Lemma \ref{lem:allcasesstronglyimp} again. Since the composition factors are invariant under an outer automorphism swapping $4_1$ and $4_4$, we remark that there are elements of order $195$ stabilizing all eigenspaces of $y$ on $4_1\oplus 4_4$, but not for $4_2\oplus 4_5$. In addition, the $64$-dimensional factors have no extension with $16_{14}$ and therefore split off as summands. Ignoring these, the socle of $L(E_8){\downarrow_H}$ is a sum of some of $4_2\oplus 4_5$, $16_{14}$, and potentially other $16$-dimensional modules that must be summands. We take the $\{1,4_i\}$-radical of $P(16_{14})$, then add any copies of $16_{15}$ and $16_{24}$ on top, then add any copies of $1$, $4_1$, $4_2$, $4_4$ and $4_5$ on top of that, to obtain the module
\[ 1,1/4_1,4_2,4_4,4_5/1,16_{15},16_{24}/4_1,4_4/16_{14}.\]
Since this clearly does not have enough composition factors, we must have $4_2\oplus 4_5$ in the socle (we need both else one on its own is an orbit under $N_{\Aut^+(\mb G)}(H)$, so $H$ is strongly imprimitive by Theorem \ref{thm:intersectionorbit}). They of course then can only appear in the socle and top of the module.

We take the $\{1,4_1,4_4,16_{14}\}$-radical of $P(4_2)$, then add any copies of $16_{15}$ and $16_{24}$ on top, then add any copies of $1$, $4_1$, $4_4$ and $16_{14}$ on top of that, to obtain the module
\[ 4_4/1,16_{14}/4_1,4_4/1,16_{24}/4_2.\]
We need eight trivial composition factors, and we can see at most seven, so this yields a contradiction. Thus either there must be a copy of one of $1$, $4_1$ or $4_4$ in the socle of $L(E_8){\downarrow_H}$, or a copy of $4_2$ but not $4_5$ (or vice versa). In particular, $H$ is strongly imprimitive.

\medskip

\noindent \textbf{Cases 3, 4, 5 and 6}: In these cases $H$ is definitely Lie imprimitive, because all $16$-dimensional composition factors appear with multiplicity $1$, there are only extensions between modules of dimension $4^i$ and $4^{i\pm 1}$, and $L(E_8)$ is self-dual. Thus there must be a $4$-dimensional factor in $\soc(L(E_8){\downarrow_H})$. In the third, fourth and fifth cases the composition factors are not invariant under any outer automorphism, so these are also strongly imprimitive. In the sixth case, we find no elements that simultaneously stabilize the eigenspaces of all of $4_1$, $4_2$ and $4_4$, so we must use a different strategy.

Suppose that $L(E_8){\downarrow_H}$ is invariant under the field automorphism (as it needs to be for $H$ not to be strongly imprimitive). We show that $H$ in fact stabilizes a line on $L(E_8)$, so suppose that this is not true. As we must have a subquotient $16_{13}\oplus 16_{15}\oplus 16_{35}\oplus 4_2\oplus 4_4\oplus 4_6$ by Lemma \ref{lem:oddandodd}, this subquotient must form the `middle' of the module $L(E_8){\downarrow_H}$, and at least six trivial modules must appear in the $\{1,4_1,4_3,4_5\}$-radical of $L(E_8){\downarrow_H}$ by Proposition \ref{prop:bottomhalf}. (We may ignore the $64$-dimensional modules in the socle as we are interested in whether $H$ stabilizes a line.) We therefore consider the $\{1,4_1,4_3,4_5,64_{246}\}$-radical of $P(4_i)$ ($i=1,3,5$), which is simply $1/4_i$. Thus we need at least six $4$-dimensional factors in the socle to have that many trivial factors, but the pressure of the module is only $3$. Thus $H$ stabilizes a line on $L(E_8)$ by Proposition \ref{prop:pressure}, and in particular is strongly imprimitive by Lemma \ref{lem:fix1space}.

\medskip

\noindent \textbf{Case 7}: This case is very easy, since every composition factor of $\soc(L(E_8){\downarrow_H})$ is in the table apart from $64_{156}$, which must be a summand if it is in the socle. Thus $H$ is always strongly imprimitive.

\medskip

\noindent \textbf{Case 9}: We first show that $H$ is Lie imprimitive, and then that $H$ is strongly imprimitive. We need to consider the $\cf(L(E_8){\downarrow_H})$-radicals of $P(16_{15})$ and $P(16_{24})$, which are
\[ 1,1/4_1,4_3,4_4/1,16_{14}/4_2,4_4/16_{24},\quad 1,1/4_1,4_4,4_6/1,16_{14}/4_1,4_5/16_{15}.\]
We clearly cannot fit sixteen trivial composition factors above two of these modules, and so there is another factor in $\soc(L(E_8){\downarrow_H})$. Thus $H$ is Lie imprimitive, so lies inside a member of $\ms X$.

To prove strong imprimitivity, note that there are elements of order $195$ stabilizing the eigenspaces in $4_1\oplus 4_4$ and $4_3\oplus 4_6$ as well, so we may assume that the only extra modules in the socle are $4_2$ and $4_5$, else $H$ is strongly imprimitive.

We show that there is a subquotient 
\begin{equation}16_{15}^{\oplus 2}\oplus 16_{24}^{\oplus 2}\oplus 4_3\oplus 4_6\label{eq:bigsubquotient}\end{equation}
in $L(E_8){\downarrow_H}$. To see this, we take the $\{4_3,4_6,16_{15},16_{24}\}$-heart $V$ of $L(E_8){\downarrow_H}$, and we show that it is semisimple. Note that $V$ is self-dual. If $4_3$ or $4_6$ is a submodule of $V$ then it splits off as a summand, as we desire, so we may assume that $\soc(V)$, and $\top(V)$, have composition factors only copies of $16_{15}$ and $16_{24}$. We constructed above a pyx for $V$, namely a sum of the radicals of $P(16_{15})$ and $P(16_{24})$. However, the $\{16_{15},16_{24}\}'$-residuals of these modules still contain a pyx for $V$, as $\top(V)$ consists solely of copies of these modules. However, these are just $16_{15}$ and $16_{24}$, so indeed $V$ is semisimple.

We now compute a few radicals. The $\{1,4_1,4_2,4_4,4_5,16_{14}\}$-radicals of $P(4_2)$ and $P(4_5)$ are simply
\[ 1/4_1/1/4_2,\quad 1/4_4/1/4_5.\]
There can be at most four copies of $4_2$ and $4_5$ (that are not summands) in the socle of $L(E_8){\downarrow_H}$, and so the $\{1,4_1,4_2,4_4,4_5,16_{14}\}$-radical $W$ of $L(E_8){\downarrow_H}$ possesses at most two copies of $4_1$. However, the quotient module $L(E_8){\downarrow_H}/W$ must have the subquotient from (\ref{eq:bigsubquotient}) as a submodule, and the quotient by that must be isomorphic to a submodule of $W^*$, by self-duality of $L(E_8)$. We see then that there are not enough copies of $4_1$ in $W$, as there are eight in $L(E_8){\downarrow_H}$, which is a contradiction. Hence one of $4_1$, $4_3$, $4_4$, $4_6$, and $16_{14}$ is a submodule of $L(E_8){\downarrow_H}$.

Thus $H$ is strongly imprimitive in this, and all other, cases.
\end{proof}

This last remaining case of $\PSp_4(2)$ is not warranted. Two possible module actions of $H$ on $L(E_8)$ can be constructed in this case. A unipotent element $v$ of order $4$ in $H$ acts on $L(E_8)$ with Jordan blocks either $4^{56},3^8$ or $4^{60},2^4$. Assuming that $H$ does not stabilize a line on $L(E_8)$, we must have submodules of $P(4_i)$, together with copies of $8_1$ and $8_2$. The structure of $P(4_1)$ is uniserial, and is
\[ 4_1/1/4_2/1/4_1/1/4_2/1/4_1\]
and similarly for $P(4_2)$. Therefore, if the action of $v$ of order $4$ has a block of size $2$, there must be a summand $4_i/1/4_{3-i}/1/4_i$, so we see that there is a unique structure for the action of $H$ on $L(E_8)$ in each case, namely
\[ 8_1^{\oplus 6}\oplus 8_2^{\oplus 6}\oplus (4_1/1/4_2/1/4_1/1/4_2)^{\oplus 4}\oplus (4_2/1/4_1/1/4_2/1/4_1)^{\oplus 4}\]
and
\[ 8_1^{\oplus 6}\oplus 8_2^{\oplus 6}\oplus P(4_1)^{\oplus 2}\oplus P(4_2)^{\oplus 2}\oplus (4_1/1/4_2/1/4_1)^{\oplus 2}\oplus (4_2/1/4_1/1/4_2)^{\oplus 2}.\]

\section{\texorpdfstring{${}^2\!B_2$}{2B2}}

As ${}^2\!B_2(q)$ contains an element of order $q-1$, we need only consider $q=2^{2n+1}$ for $q=8,32,128,512$, as all larger Suzuki groups are blueprints as $q-1\not\in T(E_8)$. Thus here we assume that $n=1,2,3,4$ only.

\begin{proposition}\label{prop:suzine8}
Let $H\cong {}^2\!B_2(2^{2n+1})$ for some $n\leq 4$.
\begin{enumerate}
\item If $n=1$ then $H$ stabilizes a line on $L(E_8)$ or the composition factors, up to field automorphism, are one of
\[ 16_{12}^4,4_1^{16},4_2^{16},4_3^6,1^{32},\quad 16_{12}^3,16_{23}^3,4_1^{12},4_2^{12},4_3^8,1^{24},\quad 16_{12}^4,16_{23}^2,4_1^{12},4_2^{11},4_3^9,1^{24},\]
\[ 64,16_{12}^2,16_{13},16_{23}^2,4_1^8,4_2^7,4_3^7,1^{16},\quad 64^2,16_{12},16_{13},16_{23},4_1^5,4_2^5,4_3^5,1^{12}.\]
\item If $n=2,3$ then $H$ is strongly imprimitive.
\item If $n=4$ then either $H$ stabilizes a line on $L(E_8)$ or $H$ is a blueprint for $L(E_8)$.
\end{enumerate}
\end{proposition}
\begin{proof} We first list the $23$ possible distributions of dimensions of composition factors for $L(E_8){\downarrow_H}$, based on the trace of an element of order $5$ being $-2$ or $23$, and the module having positive pressure, which we saw from the previous section:
\[( 8 , 12 , 4 , 2 )\quad
( 12 , 15 , 3 , 2 )\quad
( 16 , 18 , 2 , 2 )\quad
( 20 , 21 , 1 , 2 )\quad
( 8 , 16 , 7 , 1 )\quad
( 12 , 19 , 6 , 1 )\]
\[( 16 , 22 , 5 , 1 )\quad
( 20 , 25 , 4 , 1 )\quad
( 24 , 28 , 3 , 1 )\quad
( 28 , 31 , 2 , 1 )\quad
( 32 , 34 , 1 , 1 )\quad
( 36 , 37 , 0 , 1 )\]
\[( 8 , 20 , 10 , 0 )\quad
( 12 , 23 , 9 , 0 )\quad
( 16 , 26 , 8 , 0 )\quad
( 20 , 29 , 7 , 0 )\quad
( 24 , 32 , 6 , 0 )\quad
( 28 , 35 , 5 , 0 )\]
\[( 32 , 38 , 4 , 0 )\quad
( 36 , 41 , 3 , 0 )\quad
( 40 , 44 , 2 , 0 )\quad
( 44 , 47 , 1 , 0 )\quad
( 48 , 50 , 0 , 0 )\]

\noindent $\boldsymbol{n=1}$: We assume that the trace of an element of order $5$ is not $23$, as then it must have negative pressure. There are five remaining conspicuous sets of composition factors for $L(E_8){\downarrow_H}$ up to field automorphism:
\[ 16_{12}^4,4_1^{16},4_2^{16},4_3^6,1^{32},\quad 16_{12}^3,16_{23}^3,4_1^{12},4_2^{12},4_3^8,1^{24},\quad 16_{12}^4,16_{23}^2,4_1^{12},4_2^{11},4_3^9,1^{24},\]
\[ 64,16_{12}^2,16_{13},16_{23}^2,4_1^8,4_2^7,4_3^7,1^{16},\quad 64^2,16_{12},16_{13},16_{23},4_1^5,4_2^5,4_3^5,1^{12}.\]
We can do nothing with any of these cases.

\medskip

\noindent $\boldsymbol{n=2}$: Note that $H$ has an element $x$ of order $31$ and an element of order $25$. There are, up to field automorphism, seven sets of composition factors for $L(E_8){\downarrow_H}$ that are conspicuous for elements of orders $25$ and $31$ and have positive pressure:
\[ 64_{134},64_{234},16_{12}^2,16_{15},16_{24},4_1^4,4_2^4,4_4^2,4_5^2,1^8,\]
\[ 64_{345}^2,16_{12},16_{13},16_{23},4_1^5,4_2^4,4_3^4,4_4,4_5,1^{12},\]
\[ 64_{135}^2,16_{13},16_{14},16_{34},4_1^5,4_2,4_3^4,4_4^5,1^{12},\] \[64_{135},16_{13}^2,16_{15}^2,16_{34},16_{45},4_1^6,4_2,4_3^4,4_4^4,4_5^4,1^{12},\]
\[ 16_{13}^2,16_{15}^4,16_{45}^2,4_1^9,4_2,4_3^4,4_4^4,4_5^8,1^{16},\quad 16_{13}^4,16_{34}^2,4_1^{12},4_2^3,4_3^8,4_4^9,1^{24},\]
\[16_{14}^4,4_1^{16},4_2^6,4_4^{16},1^{32}\]
As with $\PSp_4(8)$, we use Remark \ref{rem:howusestrongimp} to yield that either $H$ is strongly imprimitive or the socle of $L(E_8){\downarrow_H}$ consists solely of a subcollection of the composition factors, which do not have elements of order $93$ stabilizing their eigenspaces. Since none of the sets of composition factors is invariant under the field automorphism, we do not need to distinguish between strong and Lie imprimitivity (plus $\sigma$-stability). 

\medskip

\noindent \textbf{Cases 6 and 7}: For these sets of composition factors, there exists an element of order $93$ in $\mb G$, powering to $x$, and stabilizing the same subspaces of $L(E_8)$ as $x$, hence $H$ is strongly imprimitive by Lemma \ref{lem:allcasesstronglyimp}.

\medskip

For the other cases, we consider elements $\hat x$ of order $93$ that cube to $x$ and stabilize the constituent eigenspaces of some composition factors. We include the number of elements of order $93$ stabilizing those eigenspaces in the table.
\begin{center}
\begin{tabular}{ccc}
\hline Case & Factors & Number of elements
\\ \hline $1$ & $4_2$ & $2$
\\ $2$ & $4_2$, $4_3$, $4_4$, $4_5$ & $8$, $8$, $8$, $8$
\\ $3$ & $4_1$, $4_2$, $4_3$, $4_4$ & $2$, $8$, $2$, $2$
\\ $4$ & $4_1$, $4_2$, $4_3$, $4_5$ & $2$, $2$, $8$, $2$
\\ $5$ & $4_1$, $4_2$, $4_3$, $4_4$, $4_5$, $16_{15}$ & $2$, $26$, $2$, $2$, $8$, $2$
\\ \hline
\end{tabular}
\end{center}

\noindent \textbf{Cases 2 and 3}: We may remove any $64$-dimensional modules from the socle of $L(E_8){\downarrow_H}$ without affecting whether the socle contains $1$- or $4$-dimensional modules. Doing this, we note that the $16$s appear with multiplicity $1$, so in fact $L(E_8){\downarrow_H}$ must have a $4$-dimensional (or trivial) submodule. This is enough to complete the proof for Case 3, using Lemma \ref{lem:allcasesstronglyimp}.

For Case 2, we may assume that the socle consists of copies of $4_1$ only (and simple summands of dimension $16$). The $\{1,4_1,4_2,4_3\}$-radical of $P(4_1)$ is
\[ 1/4_3/1/4_1.\]
By Lemma \ref{lem:oddandodd}, the $\{4_4,4_5,16_{ij}\}$-heart of $L(E_8){\downarrow_H}$ is semisimple, whence by Proposition \ref{prop:bottomhalf} there must be at least six trivial composition factors (and two copies of $4_2$) in the $\{4_4,4_5,16_{ij}\}'$-radical of $L(E_8){\downarrow_H}$. However, from the above module we see that there are at most four trivials in this, and no copies of $4_2$ at all.

Thus $\soc(L(E_8){\downarrow_H})$ must contain one of the composition factors in the table above (or a trivial), and so $H$ is strongly imprimitive.

\medskip

\noindent \textbf{Case 4}: Suppose that none of the modules in the table appears in the socle of $L(E_8){\downarrow_H}$, and that there are no trivial submodules either. The composition factors of $L(E_8){\downarrow_H}$ that appear with non-unity multiplicity and are not in the table above are $16_{13}$, $16_{15}$ and $4_4$, and by Lemma \ref{lem:oddandodd} the $\{64_{135},16_{34},16_{45},4_2\}$-heart of $L(E_8){\downarrow_H}$ is semisimple.

Let $I$ be the set $\{16_{13},16_{15},4_1,4_3,4_4,4_5,1\}$. The $I$-radicals of $P(4_4)$, $P(16_{13})$ and $P(16_{15})$ are
\[ 4_3/1/4_1/1/4_4,4_5/1,16_{15}/4_1,4_3/1,16_{13}/4_1/1/4_4,\]
\[ 4_3/1/4_1/1,1/4_1,4_4,4_5/1,16_{15}/4_1/16_{13},\]
\[ 4_3/1/4_1/1/4_4/1,1/4_1,4_3/1,16_{13}/4_1,4_5/16_{15}.\]
(In fact, one may place a copy of $16_{15}$ on top of the last one, but of course the other copy of $16_{15}$ would have to lie in the top, so we can ignore this.) There must be two copies of $4_5$ in the $I$-radical of $L(E_8){\downarrow_H}$ by Proposition \ref{prop:bottomhalf}, but in each of the modules above, each copy of $4_5$ lies directly above a copy of $16_{15}$. However, there are four copies of $4_5$ and only two of $16_{15}$, so we obtain a contradiction.

Thus $L(E_8){\downarrow_H}$ has a module in the socle from the table above (or a trivial), and $H$ is strongly imprimitive.

\medskip

\noindent \textbf{Case 5}: This is similar to, but easier than, the previous case. Now the only possible modules in the socle are $16_{13}$ and $16_{45}$, so set $I=\{16_{13},16_{15},16_{45},4_1,4_3,4_4,4_5,1\}$.

The $I\setminus\{V\}$-radicals of $P(V)$, for $V=16_{13},16_{45}$, are
\[ 4_3/1/4_1,4_3,4_5/1,1,16_{15}/4_1,4_4,4_5/1,1,16_{45}/4_1,4_4,4_5/1,16_{15}/4_1/16_{13},\]
\[ 4_3/1/4_1/1/4_4/1/4_1,4_3/1,16_{13}/4_1,4_5/1,16_{15}/4_4,4_5/16_{45}.\]
Neither of these has eight trivial composition factors, which they would need to do in order for the socle to be simple. (This uses Proposition \ref{prop:bottomhalf}, just as in Case 4.)

Thus we see that the socle cannot be simple, so in fact must be $16_{13}\oplus 16_{45}$. In this case we take the $\{1,4_1,4_3,4_4,4_5,16_{15}\}$-radicals of the two projectives, which are
\[ 4_3/1/4_1/1,1/4_1,4_4,4_5/1,16_{15}/4_1/16_{13},\quad 4_3/1/4_1,4_5/1,16_{15}/4_4,4_5/16_{45}.\]
There are only six trivials in this submodule, so we apply Proposition \ref{prop:bottomhalf}, as in Case 4, to conclude that this cannot be the socle either.

Thus $L(E_8){\downarrow_H}$ has a module in the socle from the table above, and $H$ is strongly imprimitive.

\medskip

\noindent \textbf{Case 1}: The two $64$-dimensional modules have no extensions with $16_{12}$, and so must split off as summands. We will show that $L(E_8){\downarrow_H}$ always possesses either a $1$ or a $4_2$ in the socle, so assume that this is not the case. Note that if $16_{15}$ or $16_{24}$ lie in the socle then they are summands, so we may ignore them. Let $W$ be obtained from $L(E_8){\downarrow_H}$ by removing any simple summands, so that $\soc(W)$ is a submodule of $16_{12}\oplus 4_1^{\oplus 2}\oplus 4_4\oplus 4_5$. As the pressure of $L(E_8){\downarrow_H}$ is $4$, and the $16_{12}$ must support a $4$-dimensional module above it, there are at most four factors in $\soc(W)$. To begin with, we suppose that $16_{12}$ does not lie in the socle of $W$. At each stage we will be unable to build a pyx for $W$ with enough trivial composition factors.

In turn we eliminate all possible socles, starting with $4_4$, $4_5$ and $4_4\oplus 4_5$. The next step is to eliminate $4_1$ and $4_1^{\oplus 2}$. After that we eliminate $4_1^{\oplus 2}$ plus a submodule of $4_4\oplus 4_5$. This means that the socle is $4_1$ plus a non-trivial submodule of $4_4\oplus 4_5$, and we eliminate the three possibilities there. We finally briefly consider the case where $16_{12}$ lies in the socle, and eliminate this quickly.

\medskip

We construct a submodule of $P(4_i)$ for $i=1,4,5$, by taking all copies of $1$, $16_{12}$ and $4_j$ for $j\neq i$, together with $4_1$ if $i=1$, on top of $4_i$, then any copies of $16_{15}$ and $16_{24}$, then any copies of the $1$, $16_{12}$, and so on again. Finally, we add copies of $4_i$ on top of that, and remove any quotients of the form $1$, $4_2$, $16_{15}$ and $16_{24}$.
This forms the following two modules for $i=4,5$:
\[ 4_4/1/4_1,4_2,4_4/1,1,16_{12}/4_1,4_2,4_4,4_5/1,1,16_{15}/4_1,4_2,4_4,4_5/1,1,16_{12}/4_1,4_2/1/4_4,\] \[4_1,4_5/1,16_{12}/4_1,4_2,4_4/1,16_{12},16_{24}/4_1,4_2,4_5/1,16_{15}/4_5.\]

Let $v$ denote an element of order $4$ in $H$: the action of $v$ on the sum of the composition factors of $L(E_8){\downarrow_H}$ is $4^{60},1^8$, so from Table \ref{t:unipe8p4} we see that $v$ must have Jordan blocks $4^{60},2^4$. The element $v$ acts on the first module with blocks $4^{29}$, so four of the trivial composition factors of that module cannot appear in $W$ if $4_4$ lies in the socle of $W$. Since there are eight trivial factors in $W$, this means that none of $4_4$, $4_5$ and $4_4\oplus 4_5$ is the socle of $W$.

\medskip

If there are only copies of $4_1$ in $\soc(W)$ then instead of the $\{1,4_2,16_{15},16_{24}\}$-residual of this module for $P(4_1)$, we may take the $\{4_1\}'$-residual. So we take the $\{1,4_1,4_2,4_4,4_5,16_{12}\}$-radical of $P(4_1)$, which is
\[4_1/1/4_4/1,1,1/4_1,4_2,4_4/1,1/4_1,4_2,4_4/1,16_{12}/4_1,\]
then add on copies of $16_{15}$ and $16_{24}$, then take all copies of $1$, $4_1$, $4_2$, $4_4$, $4_5$ and $16_{12}$ on top of this module, and finally take the $\{4_1\}'$-residual of this. This yields the module
\[\begin{split} 4_1/1/4_1,4_1,4_4&/1,1,16_{12},16_{12}/4_1,4_1,4_1,4_2,4_2,4_4
\\ &/1,1,16_{12},16_{15},16_{24}/4_1,4_1,4_2,4_4/1,16_{12},16_{15}/4_1.\end{split}\]
Thus the socle of $W$ cannot be $4_1$ as there are only six trivials in this. However, if $\soc(W)\cong 4_1^{\oplus 2}$ then we cannot have any copies of $4_1$ in the middle of this module, yielding a smaller module:
\[ 4_1/1/4_4/1/4_1,4_2/1/4_1,4_1,4_4/1,16_{12},16_{15}/4_1.\]
As $v$ acts projectively on this module as well, then as before two of the trivials in this module cannot appear in $W$. However, this means that $4_1^{\oplus 2}$ cannot be the socle of $W$ either.

\medskip

Thus we need both a $4_1$ and either a $4_4$ or a $4_5$ in $\soc(W)$. Suppose first that $4_1^{\oplus 2}$ lies in the socle of $W$. Similarly to above, we construct submodules of $P(4_1)$, $P(4_4)$ and $P(4_5)$, but now imposing that all copies of $4_1$ lie in the socle or top. This yields modules
\[ 4_1/1/4_4,4_5/1,1/4_1,4_2,4_2/1,16_{24}/4_1,4_1,4_2,4_4,4_5/1,16_{12},16_{15}/4_1,\]
\[ 4_4/1/4_2/1/4_4,4_5/1/4_1,4_2/1/4_4,\quad 4_1,4_5/1,16_{12}/4_2,4_4/1,16_{24}/4_1,4_2,4_5/1,16_{15}/4_5.\]

The element $v$ acts with Jordan blocks $4^{25},1$ on the first module, $4^8$ on the second and $4^{20},2,1$ on the third. Hence only two trivial composition factors may appear in $W$ from each of these. Thus to achieve eight trivials in $W$, the socle must be $4_1^{\oplus 2}\oplus 4_4\oplus 4_5$. However, in this case we have to remove the copies of $4_1$, $4_4$ and $4_5$ from the middle of the modules, and this yields the three modules
\[ 4_5/1/4_2/16_{24}/4_1,4_1,4_2,4_4,4_5/1,16_{12},16_{15}/4_1,\quad 4_4,4_5/1/4_1,4_2/1/4_4,\] \[4_1,4_5/1,16_{12}/4_2,4_4/1,16_{24}/4_1,4_2,4_5/1,16_{15}/4_5.\]
The action of $v$ on these three modules has Jordan blocks $4^{20},1^2$ for the first, $4^5,2$ for the second, and $4^{20},2,1$ for the third. Now we see that we cannot build a pyx in this case either, since $v$ must act on a submodule of this with blocks $4^{60},2^4$, and the trivial composition factors only form three blocks of size $2$. Thus there is a single $4_1$ in the socle of $W$.

\medskip

If the socle is $4_1\oplus 4_4\oplus 4_5$, then we allow copies of $4_1$ in the interior of the module, but not $4_4$ or $4_5$. Taking the $\{1,4_1,4_2,16_{12}\}$-radical of $P(4_i)$, then adding on any copies of $16_{15}$ and $16_{24}$, and then again any copies of $1$, $4_1$ and $4_2$, and finally all copies of $4_4$ and $4_5$ on that, yields three modules. The $\{4_1,4_4,4_5\}'$-residuals of them are 
\[ 4_1,4_5/1,16_{12}/4_1,4_1,4_2,4_4,4_5/1,16_{12},16_{15},16_{24}/4_1,4_1,4_2,4_4,4_5/1,16_{12},16_{15}/4_1,\]
\[ 4_4/1/4_1,4_2,4_4,4_5/1,1,16_{12}/4_1,4_2/1/4_4,\] \[ 4_1,4_5/1,16_{12}/4_1,4_2,4_4/1,16_{12},16_{24}/4_1,4_2,4_5/1,16_{15}/4_5.\]

There are ten trivial factors in the sum of these, and we need eight, so if we can prove that at least three of these factors do not lie in $W$ then we get a contradiction. Take the $\{1\}$-heart of these modules, to produce three much smaller modules. The second is self-dual and has structure
\[ 1/4_1,4_2/1,1,16_{12}/4_1,4_2/1,\]
while the first and third are isomorphic, and
\[ 1/4_2/1,16_{24}/4_2/1,\quad 1/4_2/16_{24}/4_2/1,1\]
are the structures of the module and its dual. Call these modules $A_1$ and $A_2$ respectively, so that the $\{1\}$-heart of $W$ is a submodule of $A_1\oplus A_2^{\oplus 2}$. This now looks good for a contradiction: of course, $W$ has a single $16_{24}$, and modulo projective blocks the action of $v$ of order $4$ on $A_1$ is $2^2$ and $A_2$ is $2,1$. We end up that $W$ loses a trivial factor that contributes to a block of size $2$ in $A_2$, but we need all four of these to make the action of $v$ equal to $4^{60},2^4$ on $L(E_8)$. This contradiction deals with $\soc(W)\cong 4_1\oplus 4_4\oplus 4_5$.

We next consider the socle $4_1\oplus 4_5$, so we allow copies of $4_1$ and $4_4$ in the interior of our modules. As we have done before, we obtain the following submodules of the projectives: 
\[\begin{split} 4_1/1/4_1,4_1,4_4,4_5&/1,1,1,16_{12},16_{12}/4_1,4_1,4_1,4_2,4_2,4_4,4_5
\\ &/1,1,16_{12},16_{15},16_{24}/4_1,4_1,4_2,4_4,4_5/1,16_{12},16_{15}/4_1,\end{split}\]
\[ 4_1,4_5/1,16_{12}/4_1,4_2/1,16_{12},16_{24}/4_1,4_2,4_5/1,16_{15}/4_5\]
The actions of $v$ on these two modules have blocks $4^{47},2,1$ and $4^{24},2,1$ respectively. We see that the trivial factors contribute blocks $4,2^2,1^2$, so in particular we cannot make four blocks of size $2$. Thus we again obtain a contradiction to the socle of $W$ being this module.

Finally, for a socle of $W$ of the form $4_1\oplus 4_4$, we allow copies of $4_5$ in the interior of our modules. Proceeding as before, we obtain submodules of $P(4_1)$ and $P(4_4)$, which are quite large, and contain lots of trivial composition factors, too many for our standard construction above to work. We can cut the submodules of $P(4_1)$ and $P(4_4)$ down by noting that we can take copies of modules in the following order: $\{1,4_2,4_5,16_{12}\}$; $\{4_1\}$; $\{1,4_2,4_5,16_{12}\}$; $\{16_{15},16_{24}\}$; $\{1,4_2,4_5,16_{12}\}$; $\{4_1\}$; $\{1,4_2,4_5,16_{12}\}$. On top of these, we place copies of $4_1$ and $4_4$, then take the $\{4_1,4_4\}'$-residual of these modules, to form much smaller modules:
\[ 4_1/16_{12}/4_1,4_1,4_4/1,16_{12},16_{15}/4_1,4_1,4_2,4_4/1,16_{12},16_{15}/4_1,\]
\[ 4_4/1/4_2,4_4/1,1/4_1,4_2,4_4,4_5/1,1,16_{12}/4_1,4_2/1/4_4.\]

As a pyx for $W$, the sum of these two modules has exactly eight trivial factors, so all must be present in any pyx for $W$. However, the action of $v$ on these modules has blocks $4^{29},2$ and $4^{15},2$ respectively, and therefore the action is incompatible with the actual action of $v$ on $L(E_8)$, namely $4^{60},2^4$. This final contradiction deals with all cases where the socle of $W$ consists solely of $4$-dimensional factors.

\medskip

We have therefore shown that $16_{12}$ must lie in the socle of $W$, and also in the top. Remove these: as $16_{12}$ has extensions with $4_1$ and $4_2$, potentially both of these drop into the socle of $W$ now, so quotient out by $4_2$ if it exists. If there is no trivial submodule of this new module then we obtain a contradiction, as this new module can take the place of $W$ as well. However, we have already shown above that such a module cannot exist. Hence we find a submodule $1/4_2/16_{12}$, but it is easy to show that such a module cannot exist (there is a module $1/4_1,4_2/16_{12}$ but it needs both $4_i$ in the second layer). This completes the proof that the first case yields a strongly imprimitive subgroup, and therefore all cases.

\medskip

\noindent $\boldsymbol{n=3}$: Let $x$ be an element of order $145$ in $H$, so that $x^5$ has order $29$. The traces of elements of order $145$ are not known, but those of order $29$ are. In addition, there are $48384$ elements of order $5$ with trace $-2$ in the subgroup $(C_5)^8$ of a maximal torus $\mb T$, so given an element $y$ of order $29$ in $\mb T$, there are $48384$ elements $\hat y$ of order $145$ in $\mb T$ such that $\hat y^5=y$ and $\hat y^{29}$ has trace $-2$ on $L(E_8)$. These are manageable numbers for a fixed element of order $29$, but there are far too many possible modules for $H$ of dimension $248$ for us to test them all to find the conspicuous ones.

Suppose that $L(E_8){\downarrow_H}$ does not possess a composition factor of dimension $64$. The $1$-eigenspace of $x^5$ on modules of dimension $4$ and $16$ is zero, and there are no elements of order $29$ in $\mb T$ with a $1$-eigenspace of dimension $44$ or $8$, which removes the cases $( 8 , 20 , 10 , 0 )$ and $( 44 , 47 , 1 , 0 )$ above.

For the others, we look whether, given an element $y\in \mb T$ of order $29$, there is an integral solution to the simultaneous linear equations that arise when writing the dimensions of the eigenspaces of $y$ on $L(E_8)$ and of $x^5$ on the simple modules for $H$. There are no solutions when the $1$-eigenspace of $y$ has dimensions $28$, $36$, $40$ or $48$, and so we have so far eliminated Cases 13, 18, 20, 21, 22 and 23 from the list above. For the others, very few elements $y$ have the required property: in all cases where there are no $64$, we find only ten (up to taking powers of $y$) possibilities for the conjugacy class of $y$.

We run through the $48384$ possible roots of $y$ and see if there are solutions using the same method as above, but now with $x$ rather than $x^5$. This yields exactly four solutions:
\[ 16_{47}^4,4_1^6,4_4^{16},4_7^{16},1^{32},\quad 16_{17}^2,16_{47}^4,4_1^9,4_4^{12},4_5^3,4_7^8,1^{24},\]
\[ 16_{16}^2,16_{34}^2,16_{46}^4,4_1^4,4_3^4,4_4^8,4_5,4_6^8,4_7,1^{16},\quad 16_{13}^2,16_{34}^4,16_{47}^2,4_1^4,4_3^8,4_4^9,4_5,4_7^4,1^{16}.\]
In each case, one finds eight elements of order $435$ in $\mb T$ that cube to $x$ and stabilize the same subspaces of $L(E_8)$ as $x$, whence $H$ is strongly imprimitive by Lemma \ref{lem:allcasesstronglyimp}. (In fact, it is easy to show that $H$ is actually a blueprint for $L(E_8)$ from this.)

\medskip

We now look at the cases where there is a $64$ in the composition factors. Using a similar method as above, we find exactly one set of composition factors for $L(E_8){\downarrow_H}$ that is conspicuous for elements of order $145$, and that is
\[ 64_{347},16_{13},16_{17},16_{34}^2,16_{47}^2,4_1^4,4_3^4,4_4^6,4_5,4_7^4,1^{12}.\]
Again, there are eight elements of $\mb T$ of order $435$ that cube to $x$ and stabilize the same subspaces of $L(E_8)$ as $x$, so we are again done and $H$ is strongly imprimitive.

Finally, we consider those sets of dimensions that have two $64$s. We find seven sets of composition factors that are conspicuous for an element of order $145$, which are
\[ 64_{136},64_{156},16_{23},16_{35}^2,16_{57},4_2^2,4_3^4,4_5^4,4_7^2,1^8,\] \[ 64_{256},64_{356},16_{23}^2,16_{27},16_{36},4_2^4,4_3^4,4_6^2,4_7^2,1^8,\]
\[ 64_{126}^2,16_{35},16_{36},16_{56},4_2,4_3^5,4_5^4,4_6^4,4_7,1^{12},\]
\[ 64_{347}^2,16_{14},16_{17},16_{47},4_1^5,4_4^5,4_5,4_7^4,1^{12},\]
\[ 64_{346}^2,16_{13},16_{17},16_{37},4_1^4,4_3^4,4_4,4_5,4_7^5,1^{12},\] \[ 64_{345}^2,16_{12},16_{17},16_{27},4_1^4,4_2^4,4_4,4_5,4_6,4_7^4,1^{12},\]
\[64_{135}^2,16_{25},16_{27},16_{57},4_2^5,4_4,4_5^4,4_6,4_7^4,1^{12},\]
In all but the first and sixth of these cases, there are eight elements of order $435$ in $\mb T$ that cube to $x$ and stabilize the same subspaces of $L(E_8)$, so we are done in those cases again, and $H$ is strongly imprimitive.

In the first case, there are elements of order $435$ stabilizing the constituent eigenspaces of each simple module in $L(E_8){\downarrow_H}$ but $64_{156}$, and since the socle cannot consist solely of that factor (as it must be a summand if it is a submodule), $H$ is strongly imprimitive by Lemma \ref{lem:allcasesstronglyimp}. In the sixth case, even more is true: there are elements of order $435$ stabilizing the constituent eigenspaces of each composition factor (but of course not simultaneously) and so $H$ is strongly imprimitive by Lemma \ref{lem:allcasesstronglyimp} again, in both cases via Lemma \ref{lem:semilinearfield} (as the composition factors are not stable under any outer automorphisms).

This completes the proof that $H$ is always strongly imprimitive.

\medskip

\noindent $\boldsymbol{n=4}$: We have an element $x$ of order $37$, and a subgroup $L={}^2\!B_2(8)$ of $H$, and these are the only things we can use to compute conspicuous sets of composition factors for $L(E_8){\downarrow_H}$. Our strategy is to determine which classes in $\mb G$ of elements of order $37$ are possibilities in conspicuous sets of composition factors for $L(E_8){\downarrow_H}$, by checking that there is a solution to the linear algebra problem of the various eigenspaces for the action of $x$ having the required dimensions for $x$ to belong to that class, as we range over all possible sets of composition factors for $L(E_8){\downarrow_H}$. (This is a surprisingly stringent requirement.) Given a class for $x$, we then range over all thirteen conspicuous sets of composition factors for $L(E_8){\downarrow_L}$, using elements of orders $13$, $7$ and $5$. This time we use a linear programming solver in Magma to check that there is a solution to the eigenspace dimension problem for elements of orders $37$, $13$, $7$ and $5$ simultaneously, and one with non-negative integer coefficients. This yields a few possible sets of traces, but each can yield many conspicuous sets of composition factors for $L(E_8){\downarrow_H}$. For these, we now fix an element $x'$ of order $37$ in $\mb T$ with the appropriate trace on $L(E_8)$, and then run over all (several million) elements $y'$ of order $13$ with the appropriate trace on $L(E_8)$ and check if the eigenvalues of $x'y'$ match any of the conspicuous sets of composition factors above. This process takes a long, but finite, amount of time.

Very few do. With no $64$-dimensional factor, there are five sets of factors up to field automorphism:
\[16_{16}^4,4_1^{16},4_2^6,4_6^{16},1^{32},\quad  16_{15}^4,16_{56}^2,4_1^{12},4_2^3,4_5^8,4_6^9,1^{24},\]
\[16_{13}^2,16_{17}^4,16_{67}^2,4_1^8,4_2,4_3^4,4_6^4,4_7^8,4_8,1^{16},\quad  16_{14}^2,16_{18}^4,16_{68}^2,4_1^8,4_2,4_4^4,4_6^4,4_8^8,4_9,1^{16},\]
\[ 16_{15}^2,16_{19}^4,16_{69}^2,4_1^9,4_2,4_5^4,4_6^4,4_9^8,1^{16}.\]
(The last three of these fit into the schema
\[ 16_{a,b}^4,16_{a,c}^2,16_{b,d}^2,4_a^8,4_b^8,4_c^4,4_d^4,4_{a+1},4_{b+1},1^{16}.\]
Here, $a,b,c,d$ are distinct integers between $1$ and $9$, the indices are read modulo $9$, and if (say) $a+1=b$ then $4_b$ has multiplicity $9$.)

There is one with a single $64$-dimensional factor:
\[ 64_{126},16_{12}^2,16_{17},16_{26}^2,16_{67},4_1^4,4_2^6,4_3,4_6^4,4_7^4,1^{12}.\]
There are several with two: most fit in a general schema
\[ 64_{a+4,b+4,c+4}^2,16_{a,b},16_{b,c},16_{c,a},4_a^4,4_b^4,4_c^4,4_{a-4},4_{b-4},4_{c-4},1^{12}.\]
(The same rules apply for this schema as the previous one.) This yields a set of factors conspicuous for elements of orders $7$, $13$, $37$ and $481$ for $(a,b,c)$ any such triple. Apart from these, there are up to field automorphism three more conspicuous sets of composition factors for $L(E_8){\downarrow_H}$:
\[ 64_{125},64_{257},16_{13},16_{17}^2,16_{67},4_1^4,4_3^2,4_6^2,4_7^4,1^8,\]
\[ 64_{135},64_{358},16_{14},16_{18}^2,16_{68},4_1^4,4_4^2,4_6^2,4_8^4,1^8,\]
\[ 64_{126},64_{127},16_{27},16_{36},16_{67}^2,4_2^2,4_3^2,4_6^4,4_7^4,1^8,\]
(These are related by being of the form
\[ 64_{a,b,c},64_{b,c,d},16_{a,d}^2,16_{a,b+1},16_{c+1,d},4_a^4,4_d^4,4_{b+1}^2,4_{c+1}^2,1^8,\]
as with the previous case.)

In each case, we check that among all possible elements of order $481$ that are products of a fixed element of order $37$ and all elements of order $13$ that lie in the same maximal torus and have the correct eigenvalues, every element of order $481$ with the correct eigenvalues on $L(E_8)$ has a root of order $481\cdot 3=1443\not\in T(E_8)$ with the same eigenspaces on $L(E_8)$. In other words, the element of order $481$ in $H$, and hence $H$, is a blueprint for $L(E_8)$. (We have only checked those cases with positive pressure to save time, and so we must include the option that $H$ stabilizes a line on $L(E_8)$ as well.)
\end{proof}

It doesn't seem possible, using the techniques from this paper, to resolve any of the cases for $q=8$. Even the first one, where $16_{13}$ and $16_{23}$ do not appear, cannot easily be dealt with, as the $\{1,4_i,16_{12}\}$-radicals of $P(4_1)$, $P(4_2)$ and $P(4_3)$ are all quite sizeable, with around fifteen socle layers. In addition the maximal subgroups of ${}^2\!B_2(8)$ are few: the normalizer of a Sylow $2$-subgroup, and subgroups $13\rtimes 4$, $5\rtimes 4$ and $7\rtimes 2$. The second, third and fourth of these yield no extra information than the action of a unipotent element, and it seems difficult to use the normalizer of a Sylow $2$-subgroup, which is just contained in some parabolic subgroup of $\mb G$. It seems highly likely that different techniques will be needed here.

\section{\texorpdfstring{${}^2\!G_2$}{2G2}}

In this section we simply prove that ${}^2\!G_2(3^{2n+1})'$ must stabilize a line on $L(E_8)$.

\begin{proposition} If $H\cong {}^2\!G_2(3^{2n+1})'$ for some $n\geq 0$, then $H$ stabilizes a line on $L(E_8)$.
\end{proposition}
\begin{proof} $\boldsymbol{n=0}$: There are, up to field automorphism, ten conspicuous sets of composition factors for $L(E_8){\downarrow_H}$. As the Sylow $3$-subgroup of $H$ is cyclic, we can write down explicitly all indecomposable modules: the projective covers of $1$ and $7$ are
\[ P(1)=1/7/1,\qquad P(7)=7/((7/7/7)\oplus 1)/7,\]
and from this it can be seen that the only indecomposable module with a trivial composition factor but no trivial submodule or quotient is $P(7)$. Therefore, in order for $L(E_8){\downarrow_H}$ not to have a trivial submodule, there must be at least five times as many $7$s as $1$s. Up to field automorphism there is a single conspicuous set of factors with this property, namely
\[ 9_1^4,9_2^3,9_3^3,7^{22},1^4,\]
and if $H$ does not stabilize a line on $L(E_8)$, then $L(E_8){\downarrow_H}$ is the sum of a projective module and either $7\oplus 7$ or $7/7$. If $v$ denotes an element of order $9$ in $H$, then $v$ acts on $L(E_8)$ with Jordan blocks $9^{26},7^2$ or $9^{27},5$. Examining \cite[Table 9]{lawther1995}, we see that no unipotent class has either of those Jordan block structures, so $H$ stabilizes a line on $L(E_8)$ in all cases.

\medskip

\noindent $\boldsymbol{n>0}$: Of the 10 conspicuous sets of composition factors in the previous case (up to field automorphism), only three yield conspicuous sets of factors for ${}^2\!G_2(3)=\PSL_2(8).3$, namely
\[ 27^7,7^7,1^{10},\qquad 27^3,7^{23},1^6,\qquad 7^{28},1^{52}.\]
The simple modules for $H$ of dimension at most $248$ are of dimension $1$, $7$, $27$, $49=7\times 7$ and $189=7\times 27$. Those of dimension $49$ restrict to $L\cong {}^2\!G_2(3)$ as $(7/7/7)\oplus 27\oplus 1$, and those of dimension $189$ restrict to $L$ as $27^{\oplus 3}\oplus P(7)^{\oplus 3}$. Moreover, from \cite{sin1993}, we see that $\dim(H^1(H,49))\leq 1$, and $\dim(H^1(H,i))=0$ for $i=1,7,27,189$. We see that the number of trivial factors in $L(E_8){\downarrow_H}$ is at least the number of $49$s, so $H$ always stabilizes a line on $L(E_8)$. (Equality is possible, for $49^3,7^{14},1^3$.)
\end{proof}

\section{\texorpdfstring{$G_2$}{G2}}

In this section we will prove that there are no Lie primitive subgroups $G_2(q)$. In all cases apart from $q=7$ we prove that $H$ stabilizes a line on $L(E_8)$, and for $q=7$ we show that $H$ is a blueprint for $L(E_8)$. We continue with our definition of $u$ from the start of the chapter.

\begin{proposition}\label{prop:g2ine8} Let $H\cong G_2(q)'$ for $q\leq 9$.
\begin{enumerate}
\item If $q=2,3,4,5,8,9$ then $H$ stabilizes a line on $L(E_8)$.
\item If $q=5,7$ then $H$ is a blueprint for $L(E_8)$.
\end{enumerate}
\end{proposition}
\begin{proof} $\boldsymbol{q=7}$: There are four sets of composition factors for $L(E_8){\downarrow_H}$ that are conspicuous for elements of order at most $21$, namely
\[ 77_1,38^2,26^2,14^3,1,\quad 38,26^7,14,1^{14},\quad 26^3,14^5,7^{13},1^9,\quad 14,7^{26},1^{52}.\]

\noindent \textbf{Case 1}: The $77_1$ and $14$s must split off as they have no extensions with the composition factors of $L(E_8){\downarrow_H}$, so $u$ must act on $L(E_8)$ with the blocks $4^4,3^{12},2^{24},1^{19}$, among others. There is no non-generic unipotent class in $\mb G$ with these blocks in its action on $L(E_8)$ from Table \ref{t:unipe8p7}, so $u$, and hence $H$, is a blueprint for $L(E_8)$ by Lemma \ref{lem:genericmeansblueprint}.

\medskip

\noindent \textbf{Case 2}: The $14$ must split off as a summand because it has no extensions with the other composition factors of $L(E_8){\downarrow_H}$. The $\{1,26,38\}$-radicals of $P(1)$ and $P(26)$ are the modules $1/26/1$ and $26/1,38/26$ respectively. The action of $u$ on these modules has blocks $3^3,2^6,1^7$ and $4^2,3^{11},2^{18},1^{14}$, so $u$ is a generic unipotent class as it has no blocks of size $7$ (see Table \ref{t:unipe8p7}). Thus $H$ is a blueprint for $L(E_8)$.

\medskip

\noindent \textbf{Case 3}: The only extension is between $1$ and $26$, and so we have a summand $14^{\oplus 5}\oplus 7^{\oplus 13}\oplus 1^{\oplus 3}$, on which $u$ acts with blocks $3^5,2^{46},1^{57}$. Examining \cite[Table 9]{lawther1995}, we see that $u$ must come from the generic unipotent class $2A_1$. Thus $H$ is a blueprint for $L(E_8)$ by Lemma \ref{lem:genericmeansblueprint}.

\medskip

\noindent \textbf{Case 4}: This must be semisimple, and so $u$ acts on $L(E_8)$ with blocks $3,2^{56},1^{33}$. Hence, from \cite[Table 9]{lawther1995}, $u$ comes from the generic class $A_1$, hence $H$ is a blueprint for $L(E_8)$ by Lemma \ref{lem:genericmeansblueprint}.

\medskip

\noindent $\boldsymbol{q=5}$: There are four sets of composition factors for $L(E_8){\downarrow_H}$ that are conspicuous for elements of order at most $21$, namely
\[ 77_1,64^2,14^3,1,\quad 64,27^6,14,1^8,\quad 27^3,14^5,7^{13},1^6,\quad 14,7^{26},1^{52}.\]
None of these composition factors has an extension with any other, so $L(E_8){\downarrow_H}$ must be semisimple. We can therefore read off the action of $u$, and it acts on these modules with blocks 
\[4^8,3^{28},2^{48},1^{36},\quad 4^2,3^{27},2^{52},1^{55},\quad 3^{14},2^{64},1^{78},\quad 3,2^{56},1^{133}.\]
This means $u$ lies in the generic class $4A_1$, $3A_1$, $2A_1$ and $A_1$ respectively, by \cite[Table 9]{lawther1995}. Hence $H$ both stabilizes a line on $L(E_8)$, and is a blueprint for $L(E_8)$ by Lemma \ref{lem:genericmeansblueprint}, as claimed.

\medskip

\noindent $\boldsymbol{q=3}$: There are (up to graph automorphism) four conspicuous sets of composition factors for $L(E_8){\downarrow_H}$, namely
\[ 7_1^{27},7_2,1^{52},\qquad 49,27_1^6,7_1^2,7_2^2,1^9,\qquad 27_1^3,7_1^{18},7_2^5,1^6,\qquad 49^3,7_1^7,7_2^7,1^3.\]

\noindent \textbf{Cases 1, 2 and 3}: These have pressures $-52$, $-7$ and $-6$ respectively, so in each case $H$ stabilizes a line on $L(E_8)$ by Proposition \ref{prop:pressure}.

\medskip

\noindent \textbf{Case 4}: Suppose that $H$ does not stabilize a line on $L(E_8)$. The pressure of $L(E_8){\downarrow_H}$ is $3$, so there can be no $49^{\oplus 2}$ as a subquotient, and hence if $W$ denotes the $\{49\}$-heart of $L(E_8){\downarrow_H}$ then $W$ is a submodule of $P(49)$ and has three trivial factors. 

Suppose that $1,1/49$ is not a submodule of $W$; then $1/49$ is, and this has pressure $1$, so ignoring the copies of $7_i$ in $W$, we see that $W$ is of the form $49/1,1/49/1/49$, but such a module is clearly not self-dual, a contradiction. Thus (again ignoring the $7_i$) $W$ has the form $49/1/49/1,1/49$, with one of the trivials in the submodule $49/1,1/49$ being a quotient.


Let $W'$ be defined as follows:
\begin{enumerate}
\item Construct the $\{1,7_i\}$-radical of $P(49)$;
\item Take all possible copies of $49$ on top of this module (there are four);
\item Take all possible copies of $1$ on top of this module (there are three).
\end{enumerate}
The $\{7_i,49\}$-residual $\bar W$ of $W$ must lie inside $W'$. Note that $\bar W$ has a quotient of the form $1/49$, and the kernel corresponding to this quotient is $1,1/49$ with some copies of $7_i$ placed on top. This will be important later.

The socle structure of $W'$ is
\[ 1,1/1,49,49/7_1,7_2,49,49/7_1,7_1,7_2,7_2/1,1,7_1,7_2/49,\]
and the only quotients are trivial modules. The $49$ in $W$ that is not either a submodule or quotient therefore occurs on either the fourth or fifth socle layer.

Suppose that it occurs in the fourth, so that the trivial in the fifth layer of $W'$ is also in $W$. We take the $\{7_i,49\}$-residual of $\soc^5(W')$, which is
\[ 1/49/7_1,7_2/1,1,7_1,7_2/49,\]
but all trivial factors of this are quotients, a contradiction. Thus the $49$ in the middle of $W$ lies in the fifth socle layer, and the trivial lies in the sixth.

The second top $W'/\rad^2(W')$ of $W'$ has the form $(1/49)\oplus (1,1/49,49,49)$, and write $V$ for the second, indecomposable, summand. Suppose for a contradiction that $V$ possesses a submodule $V_1$ of the form $1/49$. This is not contained in a submodule of the form $1,1/49$, as then the rest of the socle forms a complement to this and $V$ is decomposable. Obviously there is some indecomposable submodule $V_2$ of the form $1/49,49$; if $V_1\cap V_2=0$ then we have a decomposition for $V$, so $V_1\cap V_2\neq 0$. Certainly $V_1\not\leq V_2$, so $V_1\cap V_2$ is just the $49$. Then $V_1+V_2$ has the form $1,1/49,49$ and is complemented by a simple submodule in the socle, another contradiction.

Thus $V$ has no submodule of the form $1/49$, and therefore the quotient module $W'/\rad^2(W')$ has a unique submodule of the form $1/49$. Letting $W''$ be the $\{1\}'$-residual of the preimage of this submodule in $W'$, we actually see that this is the $\{7_i,49\}$-residual of $\soc^5(W')$ that we saw before. However, the $49$ is in the wrong socle layer, a contradiction.
Thus $H$ stabilizes a line on $L(E_8)$.

\medskip

\noindent $\boldsymbol{q=9}$: From \cite{sin1993}, we see that the only simple modules with non-zero $1$-cohomology of dimension at most $248$ are of dimension $49$, and they have $1$-dimensional $1$-cohomology. A simple $kH$-module of dimension at most $248$ has dimension one of $1$, $7$, $27$, $49$ and $189$, and the restriction to $L\cong G_2(3)$ is either simple or $7_i\otimes 7_i$, or $7_i\otimes 27_i$. The former has factors $27_1,7_1^2,7_2,1$, and the latter has two $49$-dimensional and one $27$-dimensional factor. Thus from the conspicuous sets of factors for $L(E_8){\downarrow_L}$, we see that there are at least three trivial factors and at most three $49$-dimensional factors in $L(E_8){\downarrow_H}$. Thus $H$ has non-positive pressure, and so $H$ stabilizes a line on $L(E_8)$ by Proposition \ref{prop:pressure}.

\medskip

\noindent $\boldsymbol{q=2}$: There are four conspicuous sets of composition factors for $L(E_8){\downarrow_H}$, which are
\[ 14,6^{26},1^{78},\quad 14^8,6^{19},1^{22},\quad 32,32^*,14^7,6^{12},1^{14},\quad (32,32^*)^2,14^6,6^5,1^6;\]
each of these has negative pressure, so $H$ always stabilizes a line on $L(E_8)$ by Proposition \ref{prop:pressure}, as needed.

Note that:
\begin{itemize}
\item $6\otimes 6=1,14/6/1,14$, 
\item $6\otimes 14=(6/1/6/1/6)\oplus 64$,
\item $14\otimes 14$ is a module with factors $14^9,6^{10},1^{10}$, and
\item $6^{\otimes 3}$ is a module with factors $64^2,14^2,6^9,1^6$, where $64=32\oplus 32^*$.
\end{itemize}
These will be needed for $q=4,8$.

\medskip

\noindent $\boldsymbol{q=4}$: There are, up to field automorphism, seven conspicuous sets of composition factors for $L(E_8){\downarrow_H}$:
\[ 14_1,6_1^{26},1^{78},\quad 14_1^8,6_1^{16},6_2^3,1^{22},\quad 36^3,14_1,14_2,6_1^8,6_2^8,1^{16},\quad 64_1,14_1^7,6_1^6,6_2^6,1^{14},\]
\[ 84_{21},36^2,14_1^2,14_2,6_1^4,6_2^3,1^8,\quad 64_1^2,36,14_1^4,6_1^2,6_2^2,1^4,\quad 84_{21},64_1,36,14_1,14_2^3,6_1,1^2.\]
The pressures of these modules are $-52$, $-3$, $3$, $-2$, $2$, $1$ and $1$ respectively.

\medskip

\noindent \textbf{Cases 1, 2, 4 and 7}: In Cases 1, 2 and 4, $H$ must stabilize a line on $L(E_8)$, by Proposition \ref{prop:pressure}. Since the last case has no modules with non-zero $1$-cohomology that occur with multiplicity greater than $1$, we must have a trivial submodule here as well.

\medskip

\noindent \textbf{Case 6:} The copies of $64_1$ have no extensions with other modules, so split off as summands. Suppose that $H$ does not stabilize a line on $L(E_8)$, and let $W$ denote the $\{6_i,36\}$-heart of $L(E_8){\downarrow_H}$. As the pressure is $1$ we may assume that the socle of $W$ is either $6_1$ or $6_2$. As there is a single $36$, if one ignores any composition factors $14_1$, the structure of the module is uniserial, of the form
\[ 6_i/1/6_{3-i}/1/36/1/6_{3-i}/1/6_i.\]
The $\{1,6_i,14_1\}$-radicals of $P(6_1)$ and $P(6_2)$ are
\[ 6_1/1,14_1/6_2/1,6_2/6_1\qquad 6_2/1/6_1/1,6_1,14_1/6_2.\]
We wish to place $36$ on top of this, but in both cases $6_1\oplus 6_2$ is a quotient, so this must be removed to stay pressure $1$; hence there is a unique module of the form $1/6_{3-i}/1/6_i$. For $i=2$, two copies of $36$ may be placed on top of this module, but they fall into the second and third socle layers, and so cannot cover the trivial module. For $i=1$, two copies of $36$ may be placed on top, lying in the second and fifth socle layers, and a single trivial can be placed on top of this new module, in the sixth layer. Thus we can remove a copy of $36$ which appears as a quotient, leaving a unique module of the form
\[ 1/36/1,14_1/6_2/1/6_1.\]
A single $6_j$ can be placed on top of this module for $j=1,2$, but it falls into the fifth or second socle layer, as we saw in the original module above. Hence $L(E_8){\downarrow_H}$ must have a trivial submodule in this case.

\medskip

\noindent\textbf{Case 5:} Suppose that $H$ does not stabilize a line on $L(E_8)$. Since $L(E_8){\downarrow_H}$ has pressure $2$, no subquotient of $L(E_8){\downarrow_H}$ may have pressure greater than $2$ or less than $-2$, by Proposition \ref{prop:pressure}. Let $W$ be the $\{14_1\}'$-heart of $L(E_8){\downarrow_H}$. By Lemma \ref{lem:oddandodd} there is a subquotient $6_2\oplus 14_2\oplus 84_{21}$ in $W$. Thus there are two submodules $V_1$ and $V_2$ of $W$ such that $V_1\leq V_2$ and the quotient $V_2/V_1$ is this module. We will construct a specific submodule $V_1$ now.

Let $V$ denote the $\{6_2,14_2,84_{21}\}$-heart of $W$, which is a submodule of the quotient of $W$ by its $\{6_2,14_2,84_{21}\}'$-radical. The socle of $V$ cannot be a submodule of $14_2\oplus 84_{21}$ as then $V$ is semisimple, so there is at least one $6_2$ in $\soc(V)$. On the other hand, $\soc(V)$ cannot have $6^{\oplus 3}$ as a submodule because $L(E_8){\downarrow_H}$ has pressure $2$. If the socle of $V$ has $6_2^{\oplus 2}$ in it then $\rad(V)\cap \soc(V)$ has a single copy of $6_2$ in it. Add all copies of $1$, $6_1$, $14_1$ and $36$ on top of this that lie in $V$, and then take the preimage of this new module in $W$ to obtain $V_1$.

By construction, the quotient $W/V_1$ has socle $84_{21}\oplus 14_2\oplus 6_2$. Let $V_2$ denote the preimage in $W$ of $\soc(W/V_1)$, so that $V_2/V_1=84_{21}\oplus 14_2\oplus 6_2$. Since $W$ is self-dual, the dual of the quotient $W/V_1$ is isomorphic to a submodule of $W$, say $V_3$.

By construction of $V_1$, it has a single $6_2$ composition factor, as does $V_3$, and we must have that $V_3\leq V_1$. Since $W$ has the same composition factors as $V_2\oplus V_3$, we see that $V_1$ contains at least two copies of $6_1$, and at least one copy of $36$. Since all subquotients have pressure between $-2$ and $2$, and $V_1$ cannot have negative pressure, $V_1$ has pressure $0$. We note that $V_1$ must have exactly one copy of $36$: otherwise $V_1$ contains the entire $\{36\}$-heart of $W$, and then the $\{6_2,36,84_{21}\}'$-residual of $V_2$ must have a quotient $36$, but also $6_2\oplus 84_{21}$ as a quotient as well, hence a quotient of pressure $3$. Since $V_1$ has pressure $0$ it has no quotients $6_i$ or $36$.

\medskip

The socle of $V_1$ is one of $6_1$, $6_2$, $36$, $6_1^{\oplus 2}$, $6_1\oplus 6_2$, $6_1\oplus 36$ and $6_2\oplus 36$. If $\soc(V_1)=6_2\oplus 36$, then $V_1$ is contained in the sum of the $\{1,6_1,14_1\}$-radicals of $P(6_2)$ and $P(36)$, which are
\[ 1/6_1/1,6_1,14_1/6_2,\qquad 1/1,6_1,14_1/36.\]
This has exactly four trivial composition factors. In order to support the $1$ in the third layer of $P(36)$, we need the $6_1$ in the second layer. But then $V_1$ has a submodule $6_2\oplus (6_1/36)$, which has pressure $3$, a contradiction.

Suppose that $\soc(V_1)=36$. We take the $\{1,6_1,14_1\}$-radical of $P(36)$, as given above, then add all copies of $6_2$ on top of this as possible, and then all copies of $1$, $6_1$ and $14_1$ as possible, to create a pyx for $V_1$. We can construct a smaller pyx by removing all $6_i$ quotients. This yields a module
\[ 1,1,14_1/1,6_1,6_2,14_1/1,1,6_2/1,6_1,6_2,14_1/36.\]
Notice that this has exactly two copies of $6_1$, and $V_1$ has at least that many. Thus both of these must lie in $V_1$, and we take the $\{6_1\}'$-residual of this module to find a submodule of $V_1$. This is
\[ 6_1/1/6_1,6_2/36.\]
It has a submodule of pressure $3$, which is a contradiction.

A similar proof works for $\soc(V_1)=6_2$, with the roles of $6_2$ and $36$ reversed. So take the $\{1,6_1,14_1\}$-radical of $P(6_2)$, add copies of $36$ on top, then again all copies of $1$, $6_1$ and $14_1$. Finally, remove all quotients $6_1$ and $36$ to produce the module
\[ 1/6_1/1,1,1,36/1,6_1,6_1,14_1,36/1,6_1,14_1,36/6_2.\]
However, this has a quotient $1^{\oplus 5}$, and so we need to remove three of the trivials, and we obtain a contradiction again.

For $\soc(V_1)=6_1\oplus 36$, we need the $\{1,6_i,14_1\}$-radicals of $P(6_1)$ and $P(36)$, which are
\[ 6_1/1,14_1/6_2/1,6_2/6_1,\quad 6_1,6_2/1,1,6_1,14_1/1,6_1,6_2,14_1/1,1,6_2/1,6_1,6_2,14_1/36.\]
The second and third socle layers contribute $1^{\oplus 4}$, and the fourth and fifth another $1^{\oplus 4}$. Thus, upon removing two from each of these, we obtain exactly four trivial composition factors, the minimum needed for $V_1$.

Also, the $6_1$ and $6_2$ in the second socle layer of $P(36)$ cannot lie in $V_1$, and so the two trivials in the third layer of that module cannot exist in $V_1$. This means that the third layer of $V_1$ must be exactly $6_2$, but now the submodules of $P(36)$ and $P(6_1)$ both rely on a copy of $6_2$ lying in the third layer to cover the trivial. Thus the $6_2$ in $V_1$ must be diagonally embedded across the two modules, but then it cannot support two trivials in the fourth socle layer. This means that a trivial in the fifth layer of $P(36)$ lies in $V_1$, but we saw in the $\soc(V_1)=36$ case that the $6_1$ in the fourth layer of $P(36)$ is not allowed to lie in $V_1$, and we know the $6_2$ lies in the third layer of $V_1$. This means that the trivial cannot lie in $V_1$ after all, a contradiction.

If $\soc(V_1)=6_1^{\oplus 2}$, all copies of $6_1$ lie in $\soc(W)$ or $\top(W)$. Here the first two socle layers of $V_1$ must be $1,1/6_1,6_1$ as there is no extension between $6_i$ and $14_i$. There is no module $36/1/6_i$, so the third socle layer must be $6_2$, and by choosing summands appropriately we have $(6_2/1/6_1)\oplus (1/6_1)$. (The first summand is not unique up to isomorphism.) On top of this we must place a trivial, and get the well-defined module $1/6_2/1/6_1$. There is no uniserial module $36/1/6_2/1/6_1$ (the single $36$ we can place on top of this falls to the second socle layer), but we can place $14_1$ on top and then a $36$, to make a (again, not well-defined) module
\[ 36/1,14_1/6_2/1/6_1.\]
(It is not well defined because there is a $36$ in the second socle layer of $P(6_1)$, and we can take a diagonal $36$.) On top of this we must place a single trivial, and this does make the module well defined, and so $V_1$ must be
\[ (1/36/1,14_1/6_2/1/6_1)\oplus (1/6_1).\]
To make the module $V_2$ we must add on top of this $6_2\oplus 84_{21}$; one may place two copies of $84_{21}$ on top of the first summand (one on the second), but they fall into the second and fifth socle layers, so won't cover a trivial quotient of $V_1$. This yields a contradiction.

The next case is $\soc(V_1)=6_1\oplus 6_2$; the first two socle layers are either $1,1/6_1,6_2$ or $1,1,14_1/6_1,6_2$. There is no module $36/1/6_i$, but there is one $36/1,14_1/6_2$, so if there is no $14_1$ in the second socle layer then one cannot place a $36$ in $V_1$ at all. Thus we can construct the module
\[ (1/6_1)\oplus (36/1,14_1/6_2)\]
so far as a submodule of $V_1$, but the second summand is not well defined. We are still yet to find at least $6_1,1^2$ from $V_1$. The first summand cannot be extended further, and so the remaining factors pile on the second summand. The factor $6_1$ cannot appear in the third socle layer of the second summand as then there is a submodule $6_1\oplus (6_1,36/1,14_1/6_2)$ of $V_1$, which has pressure $3$, a contradiction. Thus $V_1$ contains the submodule
\[ 1/36/1,14_1/6_2,\]
which is well defined. On this we can place two copies of $6_1$, but they fall into the second and third socle layers, so we cannot construct an appropriate module $V_1$.

Thus $\soc(V_1)=6_1$. There is no module $1/6_2/6_1$ or $1/36/6_1$ (there is a module $1/6_2,36/6_1$, which has no non-trivial quotients). If $36/6_1$ is a submodule of $V_1$ then it is also a quotient of $W$. Removing the $6_1/36$ quotient from $W$ yields a module of pressure $0$, and if there is still a submodule $36/6_1$ then removing that yields a module of pressure $-2$ but with only one trivial submodule, which is not possible by Proposition \ref{prop:pressure}. Thus $36$ is a summand of the heart of $W$, but then there is a subquotient $6_2\oplus 84_{21}\oplus 36$ of pressure $3$, another contradiction.

Thus the second socle of $V_1$ is a submodule of $1,6_2/6_1$, and since one cannot place a $36$ on top of this in the third socle layer, the third socle must be a submodule of $6_2/1,6_2/6_1$. On top of this we must place another trivial, yielding
\[ 1/6_2/1,6_2/6_1,\]
and thus the well-defined submodule $1/6_2/1/6_1$ is a submodule of $V_1$. We have therefore identified the location of the $6_2$ in $V_1$. On top of the module $1/6_2/1/6_1$ we place all copies of $1$, $6_1$ and $14_1$, then all copies of $36$, then all copies of $1$, $6_1$ and $14_1$, to yield the module
\[ 1/6_1/1,36/6_1,36/1,1,14_1/6_1,6_2/1,36/6_1.\]
This has a quotient $1^{\oplus 3}$, so one of those trivial quotients needs to be removed, and all others must be present. This means that $V_1$ has exactly four trivial factors, and therefore exactly two copies of $6_1$.

On top of $1/6_2/1/6_1$ one may place a $14_1$, then either $6_1$ or $36$. Suppose that $36$ can be placed on top. One adds on $14_1$ and then can add two copies of $36$. On this we place all copies of $6_1$ and $1$, then remove all quotients $6_1$ and $36$, to obtain the module
\[ 1/36/1,1,14_1/6_1,6_2/1,36/6_1.\]
This has four trivial composition factors, which must therefore all be present in $V_1$. But this module has no non-trivial quotients, and $V_1$ cannot have two copies of $36$ in it.

Thus one must place the second $6_1$ on first, and then $36$. (There's no need to remove quotients $6_1$ and $36$ this time.) Doing it this way round yields the module
\[ 1,36/6_1,36/1,14_1/6_2/1,36/6_1.\]
This only has three trivial composition factors, a final contradiction.

Thus there is no option for the socle of $V_1$, and we have obtained a contradiction to $H$ not stabilizing a line on $L(E_8)$, as needed.

\medskip

\noindent \textbf{Case 3:} Suppose that $L(E_8)^H=0$. As $14_1$ and $14_2$ appear exactly once, they are in the `middle' of the module. Formally, we use Lemma \ref{lem:oddandodd} and Proposition \ref{prop:bottomhalf} with $I$ being $\{14_1,14_2\}$. So we construct the $\{1,6_i,36\}$-radical of $L(E_8){\downarrow_H}$ first, then add on $14_1$ and $14_2$, and this forms an overmodule of the $\{1,6_i,36\}$-residual of $L(E_8){\downarrow_H}$. Note that $L(E_8){\downarrow_H}$ has pressure $3$.

Suppose that $36$ lies in the socle of $L(E_8){\downarrow_H}$. The $\{1,6_i\}$-radicals of $P(36)$ and $P(6_i)$ are
\[ 6_1,6_2/1,1/6_1,6_2/1,1/1,6_1,6_2/36,\quad 6_i/1/6_{3-i}/1,6_{3-i}/6_i.\]
As at least half of the trivial composition factors need to appear in the $\{1,6_i\}$-radical of the quotient module  $L(E_8){\downarrow_H}/\soc(L(E_8){\downarrow_H})$ by Proposition \ref{prop:bottomhalf}, we need the socle to be $36\oplus 6_i\oplus 6_j$; furthermore, this means that we can afford to lose at most one trivial factor from the sum of the three modules above, but there is a subquotient $1^{\oplus 5}$ in this module (the second and third socle layers), so at least two trivials need to be excluded, a contradiction.

Thus $36$ does not lie in the socle of $L(E_8){\downarrow_H}$. Let $W$ denote the $\{36\}$-heart of $L(E_8){\downarrow_H}$. The socle of $W$ consists of copies of $36$, and if there is more than one then at least one is a summand. Suppose that $\soc(W)=36$, so that the heart of $W$ contains (among other factors) a single copy of $36$ and at most one $14_i$. In particular, there can be no extensions between them in the heart. We saw the $\{1,6_i\}$-radical of $P(36)$ before, say $V$. We have that $\Ext_{kH}^1(36,V)=0$, but this is a contradiction, since the heart of $W$ contains a copy of $36$.

Therefore $W$ has a $36$ as a direct summand, so let $W'$ denote the complement. This is a submodule of $P(36)$ again, and since $\Ext_{kH}^1(36,V)=0$, we see that $W'$ needs to contain $14_1$ or $14_2$ (or both). Let $V'$ be constructed by adding as many copies of $14_1$ and $14_2$ on top of $V$, and then again adding $1$- and $6$-dimensional modules on top. The structure of $V'$ is as follows:
\[ 6_1,6_2/1,1,6_1,6_2,14_1,14_2/1,1,6_1,6_2/1,1,6_1,6_2/1,6_1,6_2,14_1,14_2/36.\]
Suppose that the $14_i$ that appear in $W'$ lie in the second socle layer: then the heart of $W'$ has those $14$-dimensional modules as summands, so one may remove the top $36$ from $W'$, then remove the top $14_i$ (as $W'$ is self-dual, so they must be quotients), and we are left with a submodule of $V$ again, except note that $W'$ has pressure at most $2$ (since it is in a sum with $36$). Since it has pressure at most $2$, we cannot have both $6_i$ in the second socle layer of $V$, but each $6_i$ in that layer supports one trivial from the third and fifth layers; hence $W'$ must contain at most three trivial composition factors.

On the other hand, suppose that a $14_i$ appears in the fifth socle layer of $W'$, which we remind the reader is self-dual. The dual of $V'$ is
\[ 36/6_1,6_2/1,1,1,14_1,14_2/6_1,6_1,6_2,6_2/1,1,1,1/6_1,6_1,6_2,6_2,14_1,14_2,\]
and we see that this $14_i$ appears in the second radical layer of $W'$, hence in the second socle layer as $W'$ is self-dual, a contradiction.

Thus $W$ contains at most three trivial composition factors. However, the $\{1,6_i\}$-radical of $P(6_i)$ has the form
\[ 6_i/1/6_{3-i}/1,6_{3-i}/6_i,\]
whence we see that the $\{1,6_i\}$-radical of $L(E_8){\downarrow_H}$ contains at most six trivial factors, and the same holds for the quotient modulo the $\{1,6_i\}$-residual. As this yields the construction of $W$, we find that $L(E_8){\downarrow_H}$ possesses at most $6+6+3=15$ trivial factors, but it really contains $16$, a contradiction.

Thus $H$ stabilizes a line on $L(E_8)$ in all cases.

\medskip

\noindent $\boldsymbol{q=8}$: There are 27 sets of composition factors for $L(E_8){\downarrow_H}$ that are conspicuous for elements of order up to $19$. One (up to field automorphism) has no trivial factors, but has dimensions $14^2,36^2,64,84$. If $v$ denotes the involution from the larger of the two classes, then $v$ acts projectively on $36$-, $64$- and $84$-dimensional modules, and as $2^6,1^2$ on $14$-dimensional modules. Thus $v$ would have to act with at least 122 blocks of size $2$ on $L(E_8)$, which does not appear in \cite[Table 9]{lawther1995}.

Thus $L(E_8){\downarrow_H}$ has trivial factors. The remaining 24 conspicuous sets of composition factors, eight up to field automorphism, all have non-positive pressure (the modules with non-zero $1$-cohomology have dimensions $6$ and $84$, see Table \ref{t:modules48916}), so $H$ stabilizes a line on $L(E_8)$, as claimed.
\end{proof}

In the case $q=7$, we do not prove that $H$ stabilizes a line on $L(E_8)$: the second, third and fourth cases all have negative pressure, so it remains to check the first case. The only possible structure for $L(E_8){\downarrow_H}$ in which $H$ does not stabilize a line is
\[ (26/1,38/26)\oplus 38\oplus 14^{\oplus 3}\oplus 77_1;\]
this does occur in $F_4G_2$, which only for $p=7$ has a diagonal $G_2$ acting irreducibly on $M(F_4)$ and $M(G_2)$.

\newpage

\chapter{Subgroups of \texorpdfstring{$E_7$}{E7}}
\label{chap:subsine7}

Now we let $\mb G$ be the algebraic group of type $E_7$ and $H$ be a subgroup with rank at most $3$, $q\leq 9$, together with a small Ree group, a Suzuki group, $\PSL_3(16)$ or $\PSU_3(16)$. By Theorem \ref{thm:largeorderss}, the following of these are blueprints for $M(E_7)$, since they have semisimple elements of odd order greater than $75$:
\begin{enumerate}
\item $\PSL_4(q)$ for $q=4,7,8,9$;
\item $\PSU_4(q)$ for $q=8,9$;
\item $\PSp_6(q)$ for $q=4,7,8,9$;
\item $\Omega_7(q)$ for $q=7,9$;
\item $\PSL_3(16)$, $\PSU_3(16)$, $\PSL_3(9)$, $G_2(9)$;
\item ${}^2\!B_2(q)$ for $q>32$;
\item ${}^2\!G_2(q)$ for $q>27$.
\end{enumerate}
The groups $\PSL_4(2)$ and $\PSp_4(2)$ are also alternating groups, and so are dealt with in \cite{craven2017}.

We also have groups considered in \cite[Proposition 3.8]{craven2015un}, which were proved to be blueprints for $M(E_7)$:
\begin{enumerate}
\item $\PSL_4(p)$, $\PSU_4(p)$, $\PSp_4(p)$, for $p=5,7$;
\item $\PSp_6(q)$, $\Omega_7(q)$ for $q$ odd.
\end{enumerate}

In addition, if the Schur multipler of $H$ is of odd order and $H$ is a blueprint for $L(E_8)$, then $H$ is a blueprint for $L(E_7)$ and $M(E_7)$ as well: we use this to exclude $G_2(5)$ and $G_2(7)$. 

This leaves:
\begin{enumerate}
\item $\PSL_3(q)$ for $q=2,3,4,5,7,8$;
\item $\PSU_3(q)$ for $q=3,4,5,7,8,9$;
\item $\PSp_4(q)$ for $q=3,4,8,9$;
\item $G_2(q)'$ for $q=2,3,4,8$;
\item ${}^2\!G_2(q)'$ for $q=3,27$;
\item ${}^2\!B_2(q)$ for $q=8,32$;
\item $\PSL_4(3)$;
\item $\PSU_4(q)$ for $q=2,3,4$;
\item $\PSp_6(2)$.
\end{enumerate}

We therefore deal with the remaining groups in turn. Throughout this chapter, $H$ will denote one of the groups above, embedded in the simply connected form $\mb G$ of an algebraic group of type $E_7$, in characteristic $p$ dividing $q$. If $q$ is odd and the Schur multiplier of $H$ is even (for example, $\PSp_4(q)$), let $\bar H$ denote the central extension $2\cdot H$ (note that, for the groups considered, this is unique), embedded in $\mb G$ so that $Z(\bar H)=Z(\mb G)$ (so the image of $\bar H$ in the adjoint algebraic group is simple).

As with the previous three chapters, we let $u$ denote an element of order $p$ in $H$ belonging to the smallest conjugacy class.

\begin{proposition}\label{prop:sl3ine7}  Let $H\cong \PSL_3(q)$ for some $2\leq q\leq 8$.
\begin{enumerate}
\item If $q=2,3$ then $H$ stabilizes a line on either $M(E_7)$ or $L(E_7)^\circ$.
\item If $q=4$ then $H$ stabilizes a line on $M(E_7)$.
\item If $q=5$ then either $H$ is a blueprint for $M(E_7)$ or $H$ acts on $L(E_7)$ as $125\oplus 8$ and is a blueprint for $L(E_7)$.
\item If $q=7,8$ then $H$ is a blueprint for $M(E_7)$.
\end{enumerate}
\end{proposition}
\begin{proof} $\boldsymbol{q=7}$: The four conspicuous sets of composition factors for $M(E_7){\downarrow_H}$ are
\[ 28,28^*,\qquad 27^2,1^2,\qquad (10,10^*)^2,8^2,\qquad 8^6,1^8.\]
If $H$ acts semisimply on $M(E_7)$, then $u$ acts on the four modules as
\[7^2,6^2,5^2,4^2,3^2,2^2,1^2,\qquad 5^2,4^4,3^6,2^4,1^4, \qquad 4^4,3^6,2^8,1^6,\qquad 3^6,2^{12},1^{14},\]
and in each case the unipotent class exists and is generic by \cite[Table 7]{lawther1995}. Therefore $H$ is a blueprint for $M(E_7)$ by Lemma \ref{lem:genericmeansblueprint}. From $\Ext^1$ data, all but the first case must be semisimple. So $H$ is a blueprint in all but the first case. In this case there is a module $28/28^*$ (and its image under the graph automorphism), but the action of $u$ on it has blocks $7^7,1^7$, which is not in \cite[Table 7]{lawther1995}, so $H$ again acts semisimply on $M(E_7)$ and we are done.

\medskip

\noindent $\boldsymbol{q=5}$: There are twelve sets of composition factors for $M(E_7){\downarrow_H}$ that are conspicuous for elements of order at most $12$ (all are in fact conspicuous):
\[6,6^*,(3,3^*)^7,1^2,\quad 8^2,(6,6^*)^2,(3,3^*)^2,1^4,\quad (3,3^*)^6,1^{20},\quad 8^6,1^8,\quad 8^2,(3,3^*)^6,1^4,\]
\[10,10^*,(6,6^*)^3,\quad (10,10^*)^2,8^2,\quad 18,18^*,10,10^*,\quad 19^2,8^2,1^2,\quad 19^2,6,6^*,3,3^*.\]
\[15_2,15_2^*,10,10^*,3,3^*,\quad 15_2,15_2^*,(6,6^*)^2,1^2.\]

The first seven of these and the eleventh and twelfth are semisimple by $\Ext^1$ data, so as in characteristic $7$ consider the action of $u$. We see from this that the action of $u$ on each of the composition factors in these embeddings has no blocks of size $5$, whence $u$ is never non-generic (see \cite[Table 7]{lawther1995}), so $H$ is a blueprint for $M(E_7)$ by Lemma \ref{lem:genericmeansblueprint}.

The eighth case is either semisimple or $(18/10^*)\oplus (10/18^*)$ (up to applying the graph automorphism). In these two cases $u$ acts on $M(E_7)$ with blocks $5^2,4^6,3^4,2^4,1^2$ and $5^8,4^2,1^8$ respectively, neither of which occurs in \cite[Table 7]{lawther1995}. Thus $H$ cannot embed in $\mb G$ with these factors.

For the ninth case, the trivial factors split off, and the remaining composition factors form either $(8/19)\oplus (19/8)$ or a semisimple module. In either possibility, the action of $u$ on $M(E_7)$ would have Jordan blocks $5^2,4^4,3^6,2^4,1^4$, and so lies in the class $2A_2+A_1$ (see \cite[Table 7]{lawther1995}), which is generic. Thus $H$ is a blueprint for $M(E_7)$ by Lemma \ref{lem:genericmeansblueprint}.

For the tenth case, we switch to $L(E_7)$, where the corresponding set of composition factors for $L(E_7){\downarrow_H}$ is $125,8$, which of course must be semisimple. Letting $x$ denote an element of order $31$ in $H$, we note that the conjugacy class of $x$ is determined by its trace on $L(E_7)$. Assuming $x$ lies in a maximal torus $\mb T$ of $\mb G$, we look for elements of order $93$ that power to $x$. None stabilizes all of the eigenspaces of $x$ on $L(E_7)$, but eight stabilize all of the eigenspaces comprising the $8$. The stabilizer in $\mb G$ of the $8$-dimensional subspace of $L(E_7)$ must contain $H$ and these eight elements of order $93$, and stabilize a $125\oplus 8$ decomposition (as $L(E_7)$ is self-dual). In order to show that $H$ is a blueprint, we just need that this stabilizer is contained in a proper, positive-dimensional subgroup $\mb X$ of $\mb G$, as $\mb X$ cannot be irreducible on $L(E_7)$, hence $L(E_7){\downarrow_{\mb X}}$ must be $125\oplus 8$.

However, now Theorem \ref{thm:classmaximalsubgroup} can be used. (This argument is very similar to that of Lemma \ref{lem:allcasesstronglyimp}, but with $E_7$ instead of $E_8$.) If the stabilizer is not positive dimensional, it must be contained in an exotic $r$-local subgroup (but clearly that is not the case for this group, from the list in \cite{clss1992}) or contained in an almost simple subgroup $J$ of $\mb G$ (strictly containing $N_{\mb G}(H)$ and containing an element of order $93$), and $J$ is Lie primitive. The rest of the results from this chapter prove that $J$ cannot be Lie type in defining characteristic. From the tables in \cite{litterickmemoir} we see that $J$ cannot act as $125\oplus 8$, and those that act irreducibly on $L(E_7)$ are listed in \cite{liebeckseitz2004a}, and these are $\PSU_3(8)$, $M_{22}$, $Ru$ and $HS$. None of these contains an element of order $31$, so $J$ cannot exist.

This shows that $H$ is a blueprint for $L(E_7)$, as needed.

\medskip

\noindent $\boldsymbol{q=3}$: There are ten conspicuous sets of composition factors for $H$ on $M(E_7)$:
\[ 27^2,1^2,\quad (3,3^*)^6,1^{20},\quad 7^6,1^{14},\quad 7^2,(6,6^*)^2,(3,3^*)^2,1^6,\]
\[ 7^2,(3,3^*)^6,1^6,\quad 6,6^*,(3,3^*)^7,1^2,\quad  15,15^*,(6,6^*)^2,1^2,\]
\[ 7^6,(3,3^*)^2,1^2,\quad 7^2,(6,6^*)^3,3,3^*,\quad 15,15^*, 7^2,(3,3^*)^2.\]
The pressures of these cases are $-2$, $-20$, $-8$, $-4$, $-4$, $-2$, $0$, $4$, $2$ and $4$, so in all but the last three cases $H$ stabilizes a line on $M(E_7)$ by Proposition \ref{prop:pressure}.


\medskip

\noindent \textbf{Case 8}: We obtain the set of composition factors $27^3,7^6,3,3^*,1^4$ for $L(E_7){\downarrow_H}$, which has pressure $2$. However, the $27$s must split off (being projective), and the $\{1,3^\pm,7\}$-radical of $P(7)$ is
\[ 3,3^*/7,7/1,3,3^*/7.\]
Thus if $H$ does not stabilize a line on $L(E_7)$, the socle of the $\{3^\pm\}'$-heart of $L(E_7){\downarrow_H}$ must contain four copies of $7$, but this is impossible, either by pressure using Proposition \ref{prop:pressure}, or because then there must be summand $7$s (as there are at least four $7$s in the top as well).

\medskip

\noindent \textbf{Case 9}: Now we obtain two possible corresponding sets of composition factors for $L(E_7){\downarrow_H}$,
\[ 27,(15,15^*)^3,7,1^9,\qquad 27^2,(15,15^*)^2,7,3,3^*,1^6,\]
These have pressures $-2$ and $-1$, so $H$ stabilizes a line on $L(E_7)$ by Proposition \ref{prop:pressure}.

\medskip

\noindent \textbf{Case 10}: There is a unique set of composition factors for $L(E_7)$,
\[ 27,(15,15^*)^2,7^4,6,6^*,1^6,\]
which has pressure $2$. Suppose that $L(E_7)^H=0$. Letting $W$ denote the $\{7,15^\pm\}$-heart of $L(E_7){\downarrow_H}$, we consider the $\{1,6^\pm,7,15^\pm\}$-radicals of $P(7)$ and $P(15)$, which are
\[ 7/1/15,15^*/1/7\qquad 1/15,15^*/1/7/1/15.\]
We see that we can only support two trivials above each factor in $\soc(W)$ (as that top trivial in $P(15)$ cannot lie in $W$ as there is no $7$ or $15^\pm$ above it). Thus $\soc(W)$ must have at least three composition factors, whence it has pressure at least $3$, contradicting Proposition \ref{prop:pressure}.

Thus $H$ stabilizes a line on either $M(E_7)$ or $L(E_7)$ in all cases, as claimed.

\medskip

\noindent $\boldsymbol{q=2}$: The four conspicuous sets of composition factors for $M(E_7){\downarrow_H}$ are
\[ (3,3^*)^9,1^2,\qquad (3,3^*)^6,1^{20},\qquad 8^2,(3,3^*)^6,1^4,\qquad 8^6,1^8.\]
The $8$s are projective and so split off as summands. The projective cover $P(3)$ has structure
\[ 3/\left(1\oplus (3^*/3/3^*)\right)/3,\]
so we need five $3$-dimensional factors for every trivial composition factor in order not to stabilize a line on $M(E_7)$. Thus in the second, third and fourth cases $H$ must stabilize a line on $M(E_7)$. In the first case, we switch to the non-trivial composition factor $L(E_7)^\circ$ of the Lie algebra $L(E_7)$, where we find that there are two corresponding Brauer characters. Hence the composition factors of $L(E_7)^\circ{\downarrow_H}$ are one of
\[ 8^{11},(3,3^*)^6,1^8,\qquad 8^8,(3,3^*)^9,1^{14}.\]
Both of these clearly have trivial submodules, and so we are done.

\medskip

\noindent $\boldsymbol{q=4}$: There are, up to field automorphism, just two conspicuous sets of composition factors for $M(E_7){\downarrow_H}$:
\[ 8_1^6,1^8,\qquad (9,9^*)^2,8_1^2,1^4.\]
The first of these is semisimple and so $H$ stabilizes a line on $M(E_7)$. The rest of this proof considers the second case. This has pressure $4$, and we cannot immediately prove that it always stabilizes a line on $M(E_7)$, as there is a module $9^*/1,1/9$. Suppose that $H$ does not stabilize a line on $M(E_7)$, and let $W$ denote the $\{1,9^\pm\}$-heart of $M(E_7){\downarrow_H}$. We note that $\soc(W)$ cannot be $9$ (up to duality) as the $\{1,8_1,9^\pm\}$-radical of $P(9)$ is
\[ 1,1,1/9,9^*,9^*/1,1,8_1/9,\]
and so any pyx for $W$ must have socle either $9\oplus 9$ or $9\oplus 9^*$, and must have a submodule $W_0$, of the form 
\[ (1,1/9)\oplus (1,1/9^\pm).\]

\medskip

\noindent \textbf{Step 1}: Actions of $\Alt(6)$ subgroups and unipotent elements.

\medskip\noindent There are three classes of subgroups $\Alt(6)$ in $H$, labelled $L_1,L_2,L_3$, containing elements $v_1,v_2,v_3$ of order $4$ from three different $H$-classes. The irreducible $kL_i$-modules are $1,4_1,4_2,8_1,8_2$, and $8_1$ for $H$ restricts to $8_1$ for $L_i$. One may arrange the labellings so that the restriction of $9$ to each $L_i$ is $4_1/1/4_2$. Let $L$ denote any of the $L_i$, and $v$ the corresponding $v_i$. Since $8_1$ is a projective $kL$-module, the restriction of $M(E_7)$ to $L$ is isomorphic to the sum of $8_1^{\oplus 2}$ and $W{\downarrow_L}$.

The restriction of $W_0$ to $L$ has the form
\[ 1\oplus (1/4_1/1/4_2)\oplus 1\oplus (1/4_j/1/4_{3-j}),\]
for some $j=1,2$ according as $\soc(W)$ is $9\oplus 9$ or $9\oplus 9^*$. Note that $v$ acts on the sum of the composition factors of $M(E_7){\downarrow_H}$ with Jordan blocks $4^{12},1^8$, and on each summand of $W_0$ with blocks $4^2,2,1$, so we see from \cite[Table 7]{lawther1995} that $v$ acts on $M(E_7)$ with blocks $4^{12},2^4$.

\medskip

\noindent \textbf{Step 2}: The dimension of $W^L$ is equal to $2$.

\medskip\noindent That it is at least $2$ can be seen from the restriction of $W_0$ to $L$ above. For the other direction, note that the permutation module $P_L$ on the cosets of $L$ has structure
\[ 1/9,9^*/1,1,8_1,8_2/9,9^*/1,\]
and the quotient of $P_L$ by its $\{1,8_1,9^\pm\}$-residual is $1/9,9^*/1,1,8_1$. If $\dim(W^L)=3$ then there must be a $3$-dimensional $\Hom$-space from this module to $M(E_7){\downarrow_H}$ and also $W$ by Frobenius reciprocity, and we can of course quotient out by the socle of this as we assume that $W$ has no submodule $1$ or $8_1$. Thus there is now a $3$-dimensional $\Hom$-space from $1/9,9^*$ to $W$. The image of any of these maps in $W$ must lie inside the second socle layer of the $\{1,9^\pm\}$-radical of $W$, so $W_0$. But $W_0^L$ has dimension $2$, as we saw above. This proves that $\dim(W^L)=2$.

\medskip

\noindent \textbf{Step 3}: $L$ is contained in an $A_2$ subgroup.

\medskip\noindent Since $M(E_7)^L$ has dimension $2$, $L$ is contained in an $E_6$-parabolic subgroup or a $B_5$-type subgroup by Lemma \ref{lem:e7linestabs} (which, since $p=2$, is of type $C_5$). The composition factors of $L$ on $M(C_5)$ are $8_1,1^2$ or $4_1,4_2,1^2$ or $4_1^2,1^2$. Unless it is the last option, $M(C_5)^L$ always has dimension at least $2$, hence (since $M(C_5)/1$ is a submodule of the restriction of $M(E_7)$ to the $B_5$-type subgroup) $M(E_7)^L$ has dimension at least $3$, a contradiction. But if the factors are $4_1^2,1^2$ then the traces of elements of order $3$ on the $32$-dimensional module for $C_5$ are not consistent with it having composition factors $8_1^2,4_2^4$, so this case cannot occur.

On the other hand, suppose that $L$ is contained in an $E_6$-parabolic subgroup. By Proposition \ref{prop:sp4ine6} below $L$ stabilizes a line or hyperplane on $M(E_6)$ as well, and so the image of $L$ in the Levi lies in either $F_4$ or a $D_5$-parabolic subgroup by Lemma \ref{lem:e6linestabs}. 

We see that $L$ cannot lie in a $D_5$-parabolic subgroup: the factors on $M(D_5)$ would have to be $8_1,1^2$ or $4_1,4_2,1^2$ or $4_1^2,1^2$ (up to automorphism) and so in particular $L$ would stabilize a line on $M(D_5)$. This means that $L$ lies in $C_4$ or a $D_4$-parabolic, so in either case there is a version of $L$ with the same composition factors in $C_4$. However, in the table in the proof of Proposition \ref{prop:compfactorsalt6} we computed the actions of $\Alt(6)$ on $M(C_4)$ and the $16$-dimensional module, and none matches up with these factors.

Thus we may assume that the image $\bar L$ of $L$ inside the $E_6$-Levi subgroup lies in $F_4$. But then this $F_4$ subgroup acts on $M(E_7)$ as $1,1/M(F_4)/M(F_4)/1,1$, so $L$ has no fixed points on $M(F_4)$. By Proposition \ref{prop:sp4inf4} below $\bar L$ or its image under the graph automorphism stabilizes a line on $M(F_4)$, and so in particular $\bar L$ is not Lie primitive in $F_4$. Thus we consider the proper, positive-dimensional subgroups $\mb X$ of $F_4$. The composition factors of $\bar L$ on $M(F_4)$ are $8_1,4_1^2,4_2^2,1^2$.

We note that $\bar L$ must be $F_4$-irreducible: if $\bar L$ lies in an $A_2A_1$-parabolic then it must lie in an $A_2$-parabolic, but both of the corresponding Levi subgroups are simply connected $A_2$, and $\bar L$ lies in adjoint $A_2$ only; the $B_3$-parabolic subgroup already has four trivial composition factors, so that is not possible; the $C_3$-parabolic subgroup acts on $M(F_4)$ with $M(C_3)$ as a submodule, but of course $\bar L$ must stabilize a line on this if $\bar L\leq C_3$, so this cannot work either.

As $B_4$ stabilizes a line on $M(F_4)$, $\bar L$ cannot lie in this. We are left with $A_2A_2$ and $C_4$. From the table in the proof of Proposition \ref{prop:compfactorsalt6} we mentioned above, we see that if $\bar L$ is contained in $C_4$ then it acts irreducibly on $M(C_4)$, and this lies in an $A_2$ subgroup (subgroup $18$ in \cite[Table 10/10A]{thomas2017un}). If $\bar L$ lies in $A_2A_2$ then it must lie in a diagonal $A_2$, and we see from \cite[Table 10A]{thomas2017un} that this is the same subgroup of $F_4$. This acts on $M(F_4)$ as
\[ L(00)\oplus L(30)\oplus L(03)\oplus L(11).\]
(Note that $L(03)$ has an extension with $L(22)$, but not with $L(11)$. The judicious choice of the $A_2$ allows us not to worry about the $8$-dimensional factors, so we do not need to remove them before examining the possible structures.)

We now note that the restriction of the extension $L(00)/L(30)$ to the subgroup $\Alt(6)$ of $A_2$ is $1/4_1/1/4_2$, so that restriction induces an isomorphism on $1$-cohomology. In particular, this holds for the actions of the $A_2$ subgroup and $\bar L$ on $M(E_6)$, so every conjugacy class of subgroup $L$ in the $E_6$-parabolic with image $\bar L$ in the Levi is contained in an $A_2$ subgroup with the above action on $M(E_6)$.

\medskip

\noindent \textbf{Step 4}: Action of $L$ on $W$ and conclusion.

\medskip\noindent We actually work with the subgroup $J\cong \PSL_3(16)$ of this $A_2$. As a finite group, it is easier to construct the largest modules with certain composition factors, and it contains an element of order $93>75$ (see Theorem \ref{thm:largeorderss}). We find that the copies of $8_1$ must split off, and the $\{1,9_{12}^\pm\}$-radicals of $P(1)$ and $P(9_{12})$ are
\[ 1,1/9_{12},9_{12}^*/1,\qquad 1/9_{12}.\]
In particular, we can obtain useful information about the action of $J$, and hence $L$, on $M(E_7)$ from this. The most important of these is that there can be no $4_i$ in the fifth socle layer of $W{\downarrow_L}$.

We now consider a pyx for $W$, namely
\[ (9,9^*,9^*/1,1,8_1/9)\oplus (9^\pm,9^\mp,9^\mp/1,1,8_1/9^\pm).\]
We restrict the module $9/1,1,8_1/9$ to $L_1$, obtaining the module
\[ 1\oplus 8_1\oplus (4_1/1/4_2/1/4_1/1/4_2).\]
Clearly this contradicts the structure above, and so this $9$ cannot lie in the third layer of the pyx. We then restrict $9^*,9^*/1,1,8_1/9$ to $L$, obtaining the module
\[ 8_1\oplus (1/4_1)\oplus (4_2/1/4_1/1)\oplus (4_2/1/4_1/1/4_2).\]
Again, the $4$-dimensional factor in the fifth layer means that we should remove one of the copies of $9^*$ from the top of the module. Because it lies in the highest socle layer, there is a unique submodule of $9^*,9^*/1,1,8_1/9$ of codimension $9$ that, upon restriction to $L$, has no $4_2$ in the fifth layer.

We now develop a contradiction, proving that $H$ stabilizes a line on $M(E_7)$: if $L=L_2$ then we obtain a different submodule of codimension $9$ from the case when $L=L_1$. But the intersection of these two submodules is simply $1,1,8_1/9$, so we cannot place a copy of $9^*$ on top of this module such that it remains consistent with the known action of $L$ on $W$.

This completes the proof.

\medskip

\noindent $\boldsymbol{q=8}$: There are 49 sets of composition factors for $M(E_7){\downarrow_H}$ that are conspicuous for elements of order at most $21$. Let $x$ be an element of order $63$ in $H$. In each case, we check that for every conjugacy class of elements of order $21$ with the same eigenvalues as $x^3$ on $M(E_7)$ (there are sometimes one and sometimes two), and for which there is an element of order $63$ that powers to that class and has the same eigenvalues on $M(E_7)$ as $x$, there exists an element of order $315$ powering to the element of order $63$ and having the same number of distinct eigenvalues on $M(E_7)$ as $x$. Since $315>75$, this shows that $H$ is a blueprint for $M(E_7)$ in every case using Theorem \ref{thm:largeorderss}. 
\end{proof}

\begin{proposition} Let $H\cong \PSU_3(q)$ for some $3\leq q\leq 9$.
\begin{enumerate}
\item If $q=3$ then $H$ is strongly imprimitive.
\item If $q=4,8$ then $H$ stabilizes a line on $M(E_7)$.
\item If $q=5$ then $H$ either is a blueprint for $M(E_7)$ or stabilizes a line on $M(E_7)$, or $H$ stabilizes a unique $19$-space of $M(E_7)$, and its stabilizer is a maximal $A_2$ subgroup of $\mb G$.
\item If $q=7$ then $H$ is a blueprint for $M(E_7)$.
\item If $q=9$ then $H$ stabilizes a line on either $M(E_7)$ or $L(E_7)$.
\end{enumerate}
\end{proposition}
\begin{proof} $\boldsymbol{q=7}$: There are eleven conspicuous sets of composition factors for $M(E_7)$:
\[  8^2,(6,6^*)^2,(3,3^*)^2,1^4,\qquad 8^6,1^8,\qquad 15_2,15_2^*,(6,6^*)^2,1^2,\]
\[ 27^2,1^2,\qquad 6,6^*,(3,3^*)^7,1^2,\qquad 8^2,(3,3^*)^6,1^4, \qquad (3,3^*)^6,1^{20},\]
\[28,28^*, \qquad (10,10^*)^2,8^2,\qquad 10,10^*,(6,6^*)^3, \qquad 15_2,15_2^*,10,10^*,3,3^*\]
By considering $\Ext^1$, there are no extensions between any two modules in each decomposition, so $M(E_7){\downarrow_H}$ is always semisimple.

This means, as with $\PSL_3(7)$, we consider the element $u$. Only $28$ and $28^*$ in the composition factors above have a block of size $7$ in the action of $u$, and then we get $7^2,6^2,5^2,4^2,3^2,2^2,1^2$, as in $\PSL_3(7)$. As we need blocks of size $7$ for the action of $u$ to be non-generic (see \cite[Table 7]{lawther1995}), this means $u$ always lies in a generic conjugacy class. Hence $H$ is a blueprint for $M(E_7)$ by Lemma \ref{lem:genericmeansblueprint}, as needed.

\medskip

\noindent $\boldsymbol{q=5}$: There are three conspicuous sets of composition factors for $M(E_7){\downarrow_H}$:
\[  8^6,1^8,\qquad (10,10^*)^2,8^2,\qquad 19^2,8^2,1^2.\]
The first case is semisimple, and $u$ acts on $M(E_7)$ with blocks $3^6,2^{12},1^{14}$, which is the generic class $(3A_1)'$ by \cite[Table 7]{lawther1995}. Hence $H$ is a blueprint for $M(E_7)$ by Lemma \ref{lem:genericmeansblueprint}.

In the second case there is an extension between $10$ and $10^*$, but the $8$s must break off, and hence $u$ acts with blocks $3^2,2^4,1^2$ plus some other blocks. However, from \cite[Table 7]{lawther1995} we see that there is no non-generic class of elements of order $5$ that is compatible with this, so $u$ is generic, and again $H$ is a blueprint for $M(E_7)$.

In the third case, suppose that $H$ does not stabilize a line on $M(E_7)$: there is no module of the form $19/1/19$ or $19/1,1/19$, and so if there is a copy of $8$ in $\soc(M(E_7){\downarrow_H})$ then it is a summand, and there is at most one of these. Thus the action of $H$ on $M(E_7)$ has one of the structures
\[ (19/1,1,8/19)\oplus 8,\qquad 19/1,1,8,8/19.\]

Write $L_1$, $L_2$ and $L_3$ for representatives of the three $H$-classes of subgroups $\PGL_2(5)$, all $\Aut(H)$-conjugate, and let $L$ denote one of them. Let $v$ be a non-trivial unipotent element from $L$. The simple $kL$-modules come in pairs, two of each dimension, and we write $1^+$, $1^-$, $3^+$, $3^-$, $5^+$ and $5^-$ for them. The non-simple projectives are
\[ P(1^+)=1^+/3^-/1^+,\quad P(1^-)=1^-/3^+/1^-,\]
\[P(3^+)=3^+/1^-,3^-/3^+,\quad P(3^-)=3^-/1^+,3^+/3^-.\]
The restriction of $8$ to $L$ is $5^+\oplus 3^+$, and the restriction of $19$ to $L$ is $5^+\oplus 5^-\oplus P(1^+)\oplus 3^-\oplus 1^+$. Furthermore, the restriction of $1,1/19$ to $L$ is
\[ 5^+\oplus 5^-\oplus P(1^+)\oplus (1^+/3^-)\oplus 1^+\oplus 1^-.\]
We see that $v$ acts on $M(E_7)$ with at least eight blocks of size $5$, whence from \cite[Table 7]{lawther1995} $v$ must act on $M(E_7)$ with blocks one of
\[ 5^{10},1^6,\quad 5^{10},2^2,1^2,\quad 5^{10},3^2.\]
Combining this with the composition factors of $M(E_7){\downarrow_L}$ yields the possible actions of $L'=\PSL_2(5)$ on $M(E_7)$, which are
\begin{equation} 5^{\oplus 6}\oplus P(1)^{\oplus 2}\oplus
\begin{cases}(3/3)^{\oplus 2}\oplus 1^{\oplus 4},
\\ (1,3/3)\oplus (3/1,3)\oplus 1^{\oplus 2}
\\ (1,3/1,3)^{\oplus 2}\text{ or } P(1)\oplus (1,3/1,3)\oplus 3\text{ or } P(1)^{\oplus 2}\oplus 3^{\oplus 2}.\end{cases}
\label{eq:threeactionsofpgl}\end{equation}
There are multiple possible ways to attach $+$ and $-$ signs to the modules to obtain from this the action of $L$.

Since $1^+/3^-$ is a submodule of $M(E_7){\downarrow_H}$, the element $v$ cannot act as in the first case of (\ref{eq:threeactionsofpgl}). In the second case of (\ref{eq:threeactionsofpgl}), the presence of a submodule $1^+/3^-$ and self-duality forces the signs, and in the third case of (\ref{eq:threeactionsofpgl}) the presence of $1^+/3^-$ and self-duality, together with the fact that $(1^+)^{\oplus 2}$ is a subquotient shows that the middle and last case cannot occur. Thus the non-projective part of the action of $L$ on $M(E_7)$ is
\begin{equation} (1^+,3^+/3^-)\oplus (3^-/1^+,3^+)\oplus (1^-)^{\oplus 2}\;\;\text{or}\;\; (1^+,3^+/1^-,3^-)\oplus (1^-,3^-/1^+,3^+).\label{eq:twooptionspsu33}\end{equation}
Note that $8$ cannot be a summand in either case, since there is no summand $3^+$ in either expression. Thus the structure of $M(E_7){\downarrow_H}$ is definitely $19/1,1,8,8/19$, although such a module is not uniquely determined.

\medskip

In the first of the options in (\ref{eq:twooptionspsu33}), we see that the $1$-dimensional module contributed by each $19$ splits off as a summand. In the module $8,8,8/19$, this is not the case, and so there is a unique submodule $8,8/19$ for which the $1^-$ is a summand, and hence the submodule $V=1,1,8,8/19$ is also unique up to isomorphism.

However, replacing $L$ by a different one of the $L_i$ yields a different module $V$. We obtain three different modules $V_1$, $V_2$ and $V_3$, and the intersection of the $V_i$ is simply $1,1/19$. However, we know that we need a copy of $8$ in the heart of $M(E_7){\downarrow_H}$, and this yields a contradiction.

\medskip

In the second case from (\ref{eq:twooptionspsu33}), we see that one of the two copies of $3^+$ from the two $8$s must become a submodule of $M(E_7){\downarrow_L}$. We again consider the module $8,8,8/19$, which restricts to $L$ as
\[ (5^+)^{\oplus 4}\oplus 5^-\oplus P(1^+)\oplus (3^+/1^-)\oplus (3^+/3^-)\oplus 3^+.\]
There is therefore a unique module $8/19$ such that the $3^+$ in the restriction to $L$ splits off.

As with the previous case, there is a different extension for each $L_i$, and the sum of these extensions is a module $8,8/19$. So we obtain a unique submodule $V=1,1,8,8/19$. For this module, $\Ext_{kH}^1(19,V)$ is $1$-dimensional, and it produces a unique indecomposable module $19/1,1,8,8/19$, which must be $M(E_7){\downarrow_H}$.

\medskip

We now switch to $M=\SL_2(5)$, and note that $M(E_7){\downarrow_M}$ is
\[ 5^{\oplus 2}\oplus P(4)^{\oplus 2}\oplus P(3)^{\oplus 2}\oplus 2^{\oplus 2}\oplus 1^{\oplus 2}.\]
Since $M$ stabilizes a line on $M(E_7)$, by Lemma \ref{lem:e7linestabs} it lies in one of the five line stabilizers, but it clearly cannot lie in $q^{26}.F_4(q).(q-1)$ as that has four trivial quotients and submodules. If $M$ lies in $q^{1+32}.B_5(q).(q-1)$, then by examining traces on $M(B_5)$ we see that the composition factors of $M(B_5){\downarrow_M}$ are $4,3,3,1$. Since $M(B_5)$ is self-dual, we see that $M$ has a trivial summand on $M(B_5)$, and again $M$ would have four trivial quotients and submodules on $M(E_7)$.

We claim that $M$ lies in an $E_6$-Levi subgroup of $\mb G$, but certainly we have at least shown that $M$ lies in an $E_6$-parabolic subgroup. Notice that the projection of $M$ onto the $E_6$-Levi subgroup cannot have a trivial summand on $M(E_6)$. However, this projection of $M$ lies in the centralizer of an involution $z$, an $A_5A_1$ subgroup. The $1$-eigenspace of $z$ is $\Lambda^2(M(A_5))$, and $M$ cannot have a trivial summand on this space. The action of $M$ on $M(A_5)$ is one of $5\oplus 1$ or $3^{\oplus 2}$ or $3/3$, with the remaining possibilities having three trivial composition factors and therefore not being suitable. The first two have $\Lambda^2(M(A_5)){\downarrow_M}$ being projective and the third has a trivial summand in the restriction. Thus $M(E_6){\downarrow_M}$ has projective action on the $1$-eigenspace of $z$. In particular, $H^1(M,M(E_6))=0$, so $M$ lies in a Levi subgroup, and we have that  $M(E_6){\downarrow_M}$ is
\[ 5\oplus P(4)\oplus P(3)\oplus 2.\]

This centralizes an involution with centralizer $A_5A_1$, which acts on $M(E_6)$ as $(L(\lambda_1),L(1))\oplus (L(\lambda_4),L(0))$. Thus $L(\lambda_4)$ restricts to $M$ as $P(3)\oplus 5$, and the only possibility for $M$ embedding in $A_5$ with this action is as $5\oplus 1$. We may embed $M$ into a diagonal $A_1$ subgroup $\mb X$, acting as $L(1)$ on the $A_1$ factor and as $L(4)\oplus L(0)$ on the $A_5$ factor. This acts on $M(E_6)$ as
\[\begin{split} L(1)\otimes\left(L(0)\oplus L(4)\right)\oplus& \Lambda^2\left(L(0)\oplus L(4)\right)
\\ &=L(1)\oplus (L(3)/L(5)/L(3))\oplus L(4)\oplus (L(2)/L(6)/L(2)).\end{split}\]
The embedding of $\mb X$ into $E_7$ now stabilizes the subspace of $M(E_7){\downarrow_M}$ given by $4^{\oplus 2}\oplus 3^{\oplus 2}$, and also every diagonal subspace of $5^{\oplus 2}$. In particular, $\mb X$ stabilizes the $19$ of $M(E_7){\downarrow_H}$, which restricts to $M$ as $5\oplus 4^{\oplus 2}\oplus 3^{\oplus 2}$. Therefore $\gen{\mb X,H}$ contains $H$ and stabilizes a $19$-space, so is positive dimensional and $H$ is not Lie primitive. In fact, $H$ is strongly imprimitive by Theorem \ref{thm:intersectionorbit}.

We now examine the subgroups in $\ms X$ to check the final claim of the proposition: if $H$ lies in a maximal parabolic subgroup $\mb X$, then the dimensions of the factors show that $\mb X$ is not the $D_6$-, $A_6$- or $A_2A_4$-parabolic, and since it does not stabilize a line on $M(E_7)$ it cannot be the $E_6$-parabolic subgroup. All other maximal parabolics, upon removing an $A_1$ factor, lie in one of these four, so $H$ is not contained in a parabolic subgroup. If $H$ is contained in a reductive subgroup $\mb X$ then since $M(E_7){\downarrow_H}$ is indecomposable, $\mb X$ cannot be $D_6A_1$, $A_7$, $G_2C_3$, $A_1G_2$ or $A_1F_4$, and of course $H\not\leq A_1A_1$. This leaves $A_2$, where the factors do have dimensions $19^2,8^2,1^2$, and $H$ is the fixed points under the Frobenius map of the maximal $A_2$ in $\mb G$ (see \cite[Table 10.2]{liebeckseitz2004} for the factors).

\medskip

\noindent $\boldsymbol{q=3}$: There are eleven conspicuous sets of composition factors for $M(E_7){\downarrow_H}$, which are:
\[ 27^2,1^2,\quad (3,3^*)^6,1^{20},\quad (6,6^*)^2,7^4,1^4,\]
\[ 7^2,(6,6^*)^2,(3,3^*)^2,1^6,\quad 7^6,1^{14}, \quad 15,15^*,(6,6^*)^2,1^2,\quad 6,6^*,(3,3^*)^7,1^2,\]
\[ 7^2,(3,3^*)^6,1^6,\quad 7^6, (3,3^*)^2, 1^2,\quad 7^2,(6,6^*)^3,3,3^*,\quad 15,15^*, 7^2,(3,3^*)^2.\]
The simple modules with non-zero $1$-cohomology are $6,6^*,7$ (see Table \ref{t:modules3}), rather than $7,15,15^*$ as with $\PSL_3(3)$, so the pressures are generally higher than before. The pressures of the first eight sets of composition factors are $-2$, $-20$, $4$, $0$, $-8$, $2$, $0$ and $-4$ respectively.

\medskip

\noindent \textbf{Cases 1, 2, 4, 5, 7 and 8}: Proposition \ref{prop:pressure} shows that $H$ always stabilizes a line on $M(E_7)$ in these cases.

\medskip

\noindent \textbf{Case 3}: If the composition factors of $M(E_7){\downarrow_H}$ are $(6,6^*)^2,7^4,1^4$, then the action on $L(E_7)$ is uniquely determined, with composition factors
\[ 27^2,15,15^*,6,6^*,7^3,(3,3^*)^2,1^4;\]
as there are five composition factors with non-zero $1$-cohomology and four trivial factors, the trivials must be contained in a module of the form $7/1/6/1/7/1/6^*/1/7$ or $6^*/1/7/1/7/1/7/1/6$, with the $15,15^*$ and four $3$s connected in some way. In particular, such a module must exist as a submodule of $P(7)$ or $P(6)$, as any other submodules can be quotiented out and the structure not change. However, $P(7)$ and $P(6)$ have exactly nine socle layers, as this module does, and so this \emph{is} $P(7)$---which has dimension $162$, so not possible---or $P(6)$, which is not self-dual. Hence $H$ stabilizes a line on $L(E_7)$, as claimed. Thus $H$ is strongly imprimitive by Lemma \ref{lem:fix1space}.

\medskip

\noindent \textbf{Case 9}: For $7^6,(3,3^*)^2,1^2$, the corresponding action on $L(E_7)$ is determined uniquely, and it has factors $27^3,7^6,3,3^*,1^4$. The $27$s split off as projective summands, and the $\{1,3^\pm,7\}$-radical of $P(7)$ is
\[ 3,3^*/7,7/1,3,3^*/7.\]
Since we have four trivials, either $H$ stabilizes a line on $L(E_7)$ or we need four $7$s below them, and four $7$s above them, but this is impossible as we only have six $7$s. So $L(E_7){\downarrow_H}$ has a trivial submodule, and thus $H$ is strongly imprimitive by Lemma \ref{lem:fix1space}.

\medskip

\noindent \textbf{Case 10}: For $7^2,(6,6^*)^3,3,3^*$, we again switch to the Lie algebra, where the unique corresponding set of composition factors is $27,(15,15^*)^3,7,1^9$, thus $H$ clearly stabilizes a line on $L(E_7)$. Thus $H$ is strongly imprimitive by Lemma \ref{lem:fix1space}.

\medskip

The remaining cases are $15,15^*,(6,6^*)^2,1^2$ and $15,15^*,7^2,(3,3^*)^2$. These are much more difficult than the other cases: let $L$ and $M$ be copies of $\GU_2(3)$ and $\PSL_3(2)$ respectively, containing elements $u$ and $v$ of order $3$ respectively.
\medskip

\noindent \textbf{Case 6}: We deal with $15,15^*,(6,6^*)^2,1^2$ first. The $\{1,6^\pm,15^\pm\}$-radicals of $P(6)$ and $P(15)$ are $6/1,15^*/6$ and $6^*/15$ respectively, hence if $H$ does not stabilize a line on $M(E_7)$ then $M(E_7){\downarrow_H}$ must be
\[ (6/1,15^*/6)\oplus (6^*/1,15/6^*).\]
Note that both $u$ and $v$ act on this module with blocks $3^{18},1^2$, hence lie either in class $2A_2$ or class $2A_2+A_1$, by \cite[Table 7]{lawther1995}. The corresponding set of composition factors for $L(E_7){\downarrow_H}$ is
\[ 27,15,15^*,7^5,(6,6^*)^2,(3,3^*)^2,1^5.\]
Assume that $H$ does not stabilize a line on $L(E_7)$ either.

\medskip

There are eight non-projective irreducible modules for $L$, four of dimension $1$ and four of dimension $2$. Their projective modules are
\[ 1_1/1_2/1_1,\quad 1_2/1_1/1_2,\quad 1_3/1_3^*/1_3,\quad 1_3^*/1_3/1_3^*,\]
\[ 2_1/2_2/2_1,\quad 2_2/2_1/2_2,\quad 2_1^*/2_2^*/2_1^*,\quad 2_2^*/2_1^*/2_2^*.\]
The non-projective parts of the restrictions of the irreducible modules for $H$ to $L$ are as follows:
\begin{center}
\begin{tabular}{cc}
\hline Module & Restriction
\\ \hline $1$ & $1_1$
\\ $3$ & $1_3\oplus 2_1$
\\ $3^*$ & $1_3^*\oplus 2_1^*$
\\ $6$ & $1_2\oplus 2_1$
\\ $6^*$ & $1_2\oplus 2_1^*$
\\ $7$ & $2_2\oplus 2_2^*$
\\ $15$ & $1_3\oplus 2_2^*$
\\ $15^*$ & $1_3^*\oplus 2_2$
\\ \hline
\end{tabular}
\end{center}
From this we obtain that the non-projective parts of the composition factors of $L(E_7){\downarrow_H}$, restricted to $L$, are
\[ 1_1^5,1_2^4,(1_3,1_3^*)^3,(2_1,2_1^*)^4,(2_2,2_2^*)^6.\]
The two possible actions of $u$ on $L(E_7)$ are with blocks $3^{42},1^7$ and $3^{43},2^2$. In the second case, the $1$-dimensional factors of this restriction must assemble themselves into the module
\[ P(1_1)^{\oplus 2}\oplus P(1_2)\oplus P(1_3)\oplus P(1_3^*).\]
In the first case, one of $P(1_1)$ and $P(1_2)$ must become semisimple, and the rest of the module must stay the same. For one of the two (dual) blocks consisting of $2$-dimensional modules, in the first case the structure must be \[ P(2_2)^{\oplus 2}\oplus (2_1/2_2)^{\oplus 2},\]
or with $2_2/2_1$ instead of $2_1/2_2$ for the last two modules. In particular, $P(2_1)$ (and $P(2_1^*)$) are not submodules of $L(E_7){\downarrow_L}$.

We will use this information in the rest of the proof.

\medskip

\noindent \textbf{Step 1}: $v$ belongs to class $2A_2+A_1$.

\medskip\noindent Suppose $v$ lies in class $2A_2$. The action of $v$ on $L(E_7)$ then has blocks $2^{42},1^7$ by \cite[Table 8]{lawther1995}, and the composition factors of $L(E_7){\downarrow_H}$ are
\[27,15,15^*,7^5,(6,6^*)^2,(3,3^*)^2,1^5.\]
(From now on we ignore the $27$, since we know that splits off as a projective summand.) The modules $3^\pm$, $6^\pm$ and $15^\pm$ are all projective for $M$, so we focus on $7^5,1^5$. The projective cover $P(7)$ for $M$ is $7/1/7$, and the projective cover $P(1)$ for $M$ is $1/7/1$, so we see that the $7^5,1^5$ summand of $L(E_7){\downarrow_M}$ must arrange itself into one of two modules:
\[P(1)\oplus 7^{\oplus 4}\oplus 1^{\oplus 3}\quad \text{or} \quad P(7)\oplus 7^{\oplus 3}\oplus 1^{\oplus 4}.\]
Either way, $L(E_7)^M$ has dimension $4$.

However, in the proof of Proposition \ref{prop:su3ine8}, we proved that $\dim(L(E_7)^M)$ was at most the multiplicity of $6$ as a composition factor in $L(E_7)$, which is in this case $2$. This is a contradiction, and hence $v$ cannot lie in class $2A_2$.

\medskip

Let $U$ denote the quotient
\[ 1/6,6^*/1/7/1/6.\]
of the permutation module $P_M$ on the cosets of $M$.

\medskip

\noindent \textbf{Step 2}: $U$ is not a submodule of $L(E_7){\downarrow_H}$.

\medskip\noindent Suppose that $U$ is a submodule of $L(E_7){\downarrow_H}$, and hence $6^*$ is a quotient of $L(E_7){\downarrow_H}$. Removing this $6^*$ from the top does not affect $U$. Suppose first that $u$ belongs to class $2A_2+A_1$. The contribution to the principal block of $U{\downarrow_L}$ is $P(1_1)\oplus P(1_2)$, and if $u$ lies in class $2A_2+A_1$ then the principal block part of $L(E_7){\downarrow_L}$ is $P(1_1)^{\oplus 2}\oplus P(1_2)$. Of course there is a single $1_2$ quotient of this, which appears inside $U$ and also outside the submodule with quotient $6^*$, which is a contradiction.

If $u$ lies in class $2A_2$ however, then since $P(1_2)$ is a submodule of $U{\downarrow_L}$ we see that the principal block part of $L(E_7){\downarrow_L}$ must be
\[ P(1_1)\oplus P(1_2)\oplus 1_1^{\oplus 2}\oplus 1_2.\]
Removing the $6^*$ from the top of $L(E_7){\downarrow_H}$ to yield a module $W$, we see that $W{\downarrow_L}$ has principal block part 
\[ P(1_1)\oplus P(1_2)\oplus 1_1^{\oplus 2}.\]
Removing the $6$ from the socle of $W$ yields a summand $1_2/1_1$, but the quotient $W/6$ must be self-dual, which it clearly is not, another contradiction. Since $u$ cannot come from any class, we must have that $U\not\leq L(E_7){\downarrow_H}$.

\medskip

\noindent \textbf{Step 3}: $u$ belongs to class $2A_2$ and $(1/6,6^*/1/7)\oplus (1/6)$ is a submodule of $L(E_7){\downarrow_H}$.

\medskip\noindent If $u$ belongs to class $2A_2+A_1$, then there is a single copy of $1_2$ in the socle of $L(E_7){\downarrow_L}$, as we saw in the previous step. Hence we cannot have both $6$ and $6^*$, or two copies of $6$, in the socle of $L(E_7){\downarrow_H}$. We know that $L(E_7)^M$ has dimension $2$, hence $\Hom_{kH}(U,L(E_7))$ has dimension $2$. 

The only quotients of $U$ that can occur as submodules of $L(E_7){\downarrow_H}$ are $1/6^\pm$ and $1/6,6^*/1/7$, so we must have the sum of two of these as a submodule. The restriction of this to $L$ has two copies of $1_2$ in the socle, contradicting our choice of the class of $u$.

If $(1/6)\oplus (1/6^\pm)$ is a submodule of $L(E_7){\downarrow_H}$ though, $L$ acts with a submodule $(1_1/1_2)^{\oplus 2}$, which is impossible according to the structure given above. Thus we have, up to graph automorphism, that $(1/6,6^*/1/7)\oplus (1/6)$ is a submodule of $L(E_7){\downarrow_H}$. The restriction of this to $L$ has principal block part
\[ (1_1/1_2/1_1)\oplus (1_1/1_2)\oplus 1_2,\]
so we see that the principal block part of $L(E_7){\downarrow_L}$ must be
\[ P(1_1)\oplus P(1_2) \oplus 1_1^{\oplus 2}\oplus 1_2.\]

\medskip

\noindent \textbf{Step 4}: The permutation module $P_L$ on the cosets of $L$.

\medskip\noindent We have already proved that $(1/6,6^*/1/7)\oplus (1/6)$ (up to graph automorphism) is a submodule of $L(E_7){\downarrow_H}$, and we know that $L(E_7)^L$ is $3$-dimensional. (Recall that we are ignoring the $27$ summand.) The module $P_L$ has structure
\[ 1/7/((1/6,6^*/1)\oplus 3\oplus 3^*)/7/1\]
plus a copy of $27$, which we again ignore. Since we assume that $H$ stabilizes no lines on $L(E_7)$, we obtain a $3$-dimensional $\Hom$-space from $X$ to $L(E_7){\downarrow_H}$, where $X$ is
\[ 1/7/((1/6,6^*/1)\oplus 3\oplus 3^*)/7.\]
Suppose that $X$ is a submodule of $L(E_7){\downarrow_H}$. The principal block part of $X{\downarrow_L}$ is $P(1_1)\oplus 1_1\oplus 1_2$, so we need another quotient of $X$ to be a submodule of $L(E_7){\downarrow_H}$. Since (up to graph automorphism of $H$) $6$ is a submodule and $6^*$ a quotient, we already know where all of the $6$-dimensional modules are, so any $6^\pm$ in this other submodule must be in the socle or top. This submodule cannot contain either $3$ or $3^*$, or $1$, or $6^*$ either, for the reasons that we already have $1_3\oplus 1_3^*$ in the socle of $L(E_7){\downarrow_L}$, $H$ cannot stabilize a line on $L(E_7)$, and we have accounted for both copies of $6^*$ respectively. Thus we have a quotient of $1/7/1/6$, and therefore either
\[ \left(\vphantom{x^2}1/7/((1/6,6^*/1)\oplus 3\oplus 3^*)/7\right)\oplus \left(\vphantom{x^2}1/7/1/6\right)\]
or
\[ \left(\vphantom{x^2}1/7/((1/6,6^*/1)\oplus 3\oplus 3^*)/7\right)\oplus \left(\vphantom{x^2}1/7\right)\oplus \left(\vphantom{x^2}1/6\right)\]
is a submodule of $L(E_7){\downarrow_H}$.

In both cases, we are missing $15,15^*,7^2,6^*,3,3^*$. In order for the trivial module at the top of $X$ to stay in $L(E_7)$, it must be covered by a $7$ or $6^*$. There is no extension with submodule $X$ and quotient $7$ or $6^*$ that does this, so we must add more of the missing modules on top of $X$. We add all copies of $15$, $15^*$ and $7$ that we can, then $3$ and $3^*$, and finally $6^*$ and $7$, but this produces a module
\[ 1/7,7/1,3,3,3^*,3^*/6,6^*,7,7,15,15^*/1,3,3^*/7,\]
which clearly still has a trivial quotient. Thus $X$ cannot be a submodule of $L(E_7){\downarrow_H}$.

\medskip

Hence the possible images of a map from $X$ to $L(E_7){\downarrow_H}$ are
\[ 1/7,\quad 1/7/1/6,\quad 1/7/3^\pm,\quad 1/7/1/6,3^\pm,\quad 1/7/3,3^*,\quad 1/7/1/6,3,3^*.\]
Ignoring any $3^\pm$ in the socle, we still need a sum of three of the modules $1/7/1/6$ and $1/7$. This means we certainly need another copy of $1/7$, possibly with $3^\pm$ or $3,3^*$ underneath it, and another copy of $1/7$, either on top of the $1/6$ or to the side. If it is to the side though, we end up with a subquotient
\[ 6\oplus 6^*\oplus 6\oplus 7^{\oplus 2},\]
which has pressure $5$, so $H$ stabilizes a line on $L(E_7){\downarrow_H}$ by Proposition \ref{prop:pressure}.

Thus we have a submodule
\begin{equation} (1/6,6^*/1/7)\oplus (1/7/1/6)\oplus (1/7),\label{eq:bigsubmodulepsu33}\end{equation}
possibly with some copies of $3^\pm$ under the second and third summands.

\medskip

Let $Y$ denote the $\{3^\pm,15^\pm\}$-heart of $L(E_7){\downarrow_H}$, which is of course a self-dual module.

\medskip

\noindent \textbf{Step 5}: $Y$ contains a single trivial composition factor, and $\soc(L(E_7){\downarrow_H})$ contains exactly one $3$-dimensional submodule.

\medskip\noindent From the composition factors of $L(E_7){\downarrow_H}$, we see that there are at most two composition factors of dimension $3$ in the socle. (None can be summands as there is no summand $1_3^\pm$ in the action of $L$.)

Suppose first that there are no $3$-dimensional factors in the socle. In this case the whole of the module from (\ref{eq:bigsubmodulepsu33}) appears as a submodule of $L(E_7){\downarrow_H}$, and also $6^*\oplus 7^{\oplus 2}$ in the top. This means that $Y$ can only have composition factors $15,15^*,(3,3^*)^2$.

Hence we need to understand the $\{3^\pm,15^\pm\}$-radicals of $P(3)$ and $P(15)$. These are
\[ 3/3^*,15^*/3\quad\text{and}\quad 3^*/15.\]
The former of these is self-dual, so one cannot make a module with three socle layers, six composition factors, and two factors in the third layer, which is necessary to have the reduction to $L$ containing $P(1_3)\oplus P(1_3^*)$. This means that we must be wrong, and there is a $3$-dimensional composition factor in the socle.

Suppose that there is more than one, and therefore that all $3$-dimensional factors lie in the socle or the top. Then one must be able to put possibly $15,15^*$, and then two $7$s on top of the module in (\ref{eq:bigsubmodulepsu33}). However, on top of this module one may only place $15^*$, and then one $7$. This means that one cannot cover two of the three trivial quotients, and this yields another contradiction.

Thus there must be exactly one $3$-dimensional factor in the socle, as claimed. Since one $3^\pm$ lies below either the second or third summand in (\ref{eq:bigsubmodulepsu33}) we see that there must be a trivial factor in $Y$, and there are four trivial factors in the submodule from (\ref{eq:bigsubmodulepsu33}) not lying below any $3^\pm$, so this completes the proof.

\medskip

\noindent \textbf{Step 6}: Contradiction.

\medskip\noindent We will obtain a contradiction with the action of $M$.

First, consider the $\{7\}$-heart of $Y$. Since $Y$ contains factors $7^2,1$ and does not have a trivial submodule, this module must have shape $7/1/7$ with some other factors. There is no module $7/1/7$, so there are other factors in the $\{7\}$-heart. Take the $\{1,3^\pm,15^\pm\}$-radical of $P(7)$, place as many copies of $7$ on top as possible, then take the $\{7\}'$-residual, and this must be a pyx for the $\{7\}$-heart of $Y$. This pyx is
\[ 7,7/1,3,3^*/7,\]
so the $\{7\}$-heart is $7/1,3,3^*/7$, but this module is not uniquely determined by this structure. However, we see that the $\{7\}$-heart of $L(E_7){\downarrow_H}$---note, not the $\{7\}$-heart of $Y$---must have a subquotient $6\oplus 6^*\oplus 3\oplus 3^*$. Since $15$ only has extensions with $3^*$ and $6^*$, the upshot of this is that if one quotients out by the $\{3^\pm,6^\pm\}$-radical of $L(E_7){\downarrow_H}$, there must be some $15^\pm$ in the socle.

We now consider the action of $M$ on $L(E_7)$, aiming to see that it cannot be the sum of a projective module and $(7/1)\oplus (1/7)$, contrary to what we have already established. To prove this, we may of course quotient out by any $kH$-submodule whose restriction to $M$ is projective, and similarly to $kH$-quotients whose restriction to $M$ is projective.

Since $3^\pm$, $6^\pm$ and $15^\pm$ all restrict projectively to $M$, we may remove all of these from the top and socle, so take the $\{1,7\}$-heart of $L(E_7){\downarrow_H}$. This is a module with composition factors
\[ 7^5,6,6^*,3,3^*,1^5,\]
and we can see most of its structure in (\ref{eq:bigsubmodulepsu33}). There is a submodule $1/7/1$, on which $M$ acts projectively and so may be removed as well. We are left with a module with a submodule
\[ (1/6,6^*/1/7)\oplus (1/7),\]
on which we have to place $3\oplus 3^*$ first, and then two copies of $7$.

Notice that the $\{7\}$-heart of $Y$ is now a quotient of the remaining module above, and the restriction to $M$ of it is again projective. Thus we may remove it, and are left with an extension with quotient $7$ and submodule $(1/6,6^*/1/7)$, whose restriction to $M$ must be $(1/7)\oplus (7/1)$ plus projective. However, the unique extension, $7/1/6,6^*/1/7$, has restriction to $M$ given by $(7/1/7)\oplus 7$ plus projective, not $(1/7)\oplus (7/1)$. This contradiction proves that $H$ stabilizes a line on $L(E_7)$.

\medskip

\noindent \textbf{Case 11}: Here $H$ acts on $M(E_7)$ with composition factors $15,15^*,7^2,(3,3^*)^2$---which resemble those of the module $Y$ in the previous case---and on $L(E_7)$ with composition factors $27,(15,15^*)^2,7^2,(6,6^*)^2,3,3^*,1^2$. The element $z$ of order $4$ in $Z(L)$ has trace $0$ on $M(E_7)$ and $-3$ on $L(E_7)$. Consulting \cite[Table II]{frey2001}, for example, we see that $z$ has centralizer $A_1A_3A_3$, inside the centralizer $A_1D_6$ of $z^2$.

Let $\mb X$ be the copy of $A_1D_6$ centralizing the involution in $L'$, and $\mb Y$ the copy of $A_1A_3A_3$ centralizing $z$ in $\mb X$. The summand of $M(E_7){\downarrow_L}$ consisting of $2$-dimensional factors is the half-spin module for $D_6$, and the summand consisting of $1$- and $3$-dimensional factors is the tensor product of the two minimal modules for $D_6$ and $A_1$. Any element of $A_1$ centralizes the half-spin module. Note that the $3$-pressure of $M(E_7){\downarrow_H}$ is $0$, and so either $3$ or $3^*$ is a submodule of $M(E_7){\downarrow_H}$, say $3$ up to duality. Let $W$ denote this $3$-dimensional $kH$-submodule of $M(E_7)$. The action of $L$ on $W$ is $1_3\oplus 2_1$, so we aim to find an infinite subgroup of $\mb G$ stabilizing $W$. We do so by finding an element of $\mb X$ simultaneously stabilizing any given submodule $1_3$ of the product of minimal modules, and any given submodule $2_1$ of the half-spin module.

Of course, not all of $L$ can embed in $A_1$, and so the projection of $L'$ onto $A_1$ cannot be faithful. Thus the composition factors of $L'$ on $M(A_1)$ are $1^2$, and $L'$ acts as $1^{\oplus 2}$ (if $u\in D_6$) or $1/1$ (if $u\not\in D_6$). In particular, this means that $L'$ acts on $M(D_6)$ with composition factors $3^3,1^3$. We therefore see that $u$ acts on $M(D_6)$ with blocks either $3^3,1^3$ or $3^4$ (there is no class in $D_6$ acting with blocks $3^3,2,1$, as we see from \cite[Table 7]{lawther2009}). Therefore $u$ lies in class either $A_2+D_2$ or $2A_2$ of $D_6$ (from the same table). Of course, either $u$ lies in the $D_6$ or acts diagonally, and this yields four options for the unipotent class of $E_7$ containing $u$, using \cite[4.10]{lawther2009}: if $u\in D_6$ then we obtain $A_2+2A_1$ or $2A_2$, and if $u\not\in D_6$ then we obtain $A_2+3A_1$ or $2A_2+A_1$. The actions of $u$ on $M(E_7)$ are as follows, via \cite[Table 7]{lawther1995}.
\begin{center}
\begin{tabular}{cc}
\hline Class & Blocks on $M(E_7)$ 
\\\hline $A_2+2A_1$ & $3^{14},2^4,1^6$
\\ $A_2+3A_1$ & $3^{14},2^7$
\\ $2A_2$ and $2A_2+A_1$ & $3^{18},1^2$
\\ \hline
\end{tabular}
\end{center}
The action of $\mb Y$ on $M(D_6)$ is the sum of the two copies of $\Lambda^2(M(A_3))$ (i.e., $M(D_3)$) for the two $A_3$ factors. The half-spin module for $D_6$ restricts to $\mb Y$ as the sum of two dual modules, each a tensor product of (up to duality) the minimal modules for the two $A_3$ factors. In order for this to contain only $2$-dimensional factors, $L'$ must act on the one copy of $M(A_3)$ with factors $3,1$ and on the other with factors $2,2$. Thus $L'$ acts on the one as $3\oplus 1$ and the other as either $2^{\oplus 2}$ or $2/2$, according as $u$ projects onto the class $A_2+D_2$ or $2A_2$. Notice that $L$ is contained in $N_{\mb X}(\mb Y)$, and in fact $L$ must lie inside the subgroup of index $2$ that stabilizes the two $A_3$ factors (as there is an involution that swaps the two factors). This subgroup still induces a simultaneous graph automorphism on the two $A_3$ factors (since $\GO_6\times \GO_6\leq \GO_{12}$).

\medskip

\noindent \textbf{Step 1}: Eliminating classes $A_2+2A_1$ and $A_2+3A_1$ for $u$.

\medskip\noindent In both of these cases the projection of $u$ onto $D_6$ lies in class $A_2+D_2$, and from the action described above, we see that $L'$ must act on the two copies of $M(A_3)$ as $1\oplus 3$ and $2\oplus 2$. Embed $L'$ in a copy $\mb L$ of $A_1$ acting on the two factors in the obvious way, as $L(0)\oplus L(2)$ and $L(1)\oplus L(1)$ respectively. The exterior square of the second one (a summand of $M(D_6){\downarrow_{\mb L}}$) is $L(2)\oplus L(0)^{\oplus 3}$, so $\mb L$ centralizes the fixed-point space $M(E_7)^{L'}$. On the other hand, the tensor product of the two modules is
\[ L(1)^{\oplus 2}\oplus (L(1)/L(3)/L(1))^{\oplus 2},\]
and we see that $\mb L$ again stabilizes every $2$-space stabilized by $L'$. Thus $\mb L$ definitely stabilizes all submodules $3$ and $3^*$, and $H$ contained in a member of $\ms X$, in fact is strongly imprimitive by Theorem \ref{thm:intersectionorbit}.

\medskip

\noindent \textbf{Step 2}: Eliminating classes $2A_2$ and  $2A_2+A_1$ for $u$.

\medskip\noindent In these cases $u$ acts with Jordan blocks $3^{18},1^2$ on $M(E_7)$, and this determines the action of $L'$ as well. It almost completely determines the action of $L$, and does so if we assume that $3$ is a submodule of $M(E_7){\downarrow_H}$. The action must be
\[ P(1_3)\oplus P(1_3^*)\oplus P(2_1)\oplus P(2_1^*)\oplus P(2_2)\oplus P(2_2^*)\oplus (2_2/2_1)\oplus (2_1^*/2_2^*)\oplus 3_1^{\oplus 2}\oplus 3_2^{\oplus 2}\oplus 3_3\oplus 3_3^*.\]

Since $15$ and $15^*$ contribute $P(2_1)$ and $P(2_1^*)$ respectively to this action, and $3^\pm$ restricts to $L$ with a submodule $2_1^\pm$, there can be at most one submodule $3$ and no submodules $3^*$ of $M(E_7){\downarrow_H}$. Thus, via Theorem \ref{thm:intersectionorbit}, if the stabilizer of $W$ is positive dimensional, $H$ is strongly imprimitive.

Let $\mb M$ denote the stabilizer in $\mb G$ of the $1$-space $1_3$ (which lies in $W$), and note that $L\leq \mb M$. From Lemma \ref{lem:e7linestabs} we see that $\mb M$ is an $E_6T_1$-parabolic subgroup, an $E_6$-Levi subgroup extended by the graph automorphism, a subgroup $\mb U_{26}\cdot F_4T_1$ or a subgroup $\mb U_{1+32}\cdot B_5T_1$.

Suppose that $\mb M$ is an $E_6T_1$-parabolic subgroup. In this case, the stabilizer has dimension $106$. Then $W$ is contained in the $\mb M$-invariant hyperplane of $M(E_7)$, which has dimension $55$ (as all copies of $2_i^\pm$ are contained in this hyperplane).

To compute the dimension of the stabilizer we choose two points at random in the $55$-space, and see if they lie in the $3$-space $W$. The dimension of the variety of $3$-spaces containing the specified line is $55+55-3-3=104$, so the stabilizer of $W$ is always positive dimensional and $H$ is strongly imprimitive.

\medskip

Since there is no $\mb M$-invariant complement to the $\mb M^\circ$-fixed points of $M(E_7)$, $\mb M$ cannot be an $E_6$-Levi subgroup (extended by the graph automorphism).

\medskip

Assume that $\mb M=\mb U_{26}\cdot F_4T_1$. This group acts on $M(E_7)$ with structure $1_a^*,1_b^*/26^*/26/1_a,1_b$ (the duality is for the $T_1$ part). The composition factors of $L'$ on the $26$-dimensional module must be $3^3,2^8,1$ (as on $M(E_7)$ the factors are $3^6,2^{16},1^6$), and therefore $L'$ acts as a sum of $1\oplus 3^{\oplus 3}$ and another module with factors of dimension $2$. In particular, the action of $L'$ on $M(F_4)$ (which is the action on $\mb U_{26}$) has $1$-dimensional $1$-cohomology. This is acted upon by the $T_1$ factor with one regular orbit and one fixed point (see \cite[Lemma 3.2.15]{stewart2013}, or it is a simple calculation in our case). The fixed point is clearly the copy of $L'$ in the $F_4$ subgroup, and this has complemented $1$-spaces, like the $E_6$ case for $\mb M$ above. Thus $L'$ is (up to conjugacy) the other complement. But this cannot be normalized by the torus, so $L\leq \mb U_{26} F_4$. However, this group centralizes a $2$-space on $M(E_7)$ and $L$ does not, a contradiction.

\medskip

The final case is to suppose that $\mb M$ is a subgroup $\mb U_{1+32}\cdot B_5T_1$. The module $M(E_7){\downarrow_{\mb M}}$ has structure $1^*,11^*/32/1,11$, so the $32$-dimensional factor must be self-dual. This means that the toral summand $\mb T$ of the $B_5T_1$ subgroup of $\mb M$ acts trivially on it. Since $\mb T$ is the toral summand from the $D_6T_1$-parabolic containing $\mb M$, it acts as the same scalar on $1$ and $11$. In particular, since $\mb U_{1+32}\cdot B_5$ is the line centralizer, the $B_5T_1$ subgroup is a direct product of the two factors.

If the central involution in $L$ has trace $11$ on $M(B_5)$ then all $2$-dimensional composition factors lie in the submodule $32/1,11$ of $M(E_7){\downarrow_{\mb M}}$, which has dimension $44$. As the dimension of $\mb M$ is $89$, we see that all such $2$-spaces have a positive-dimensional stabilizer in $\mb M$---in particular, the $2_1$ inside $W$---and thus the $W$ has positive-dimensional stabilizer. Hence $H$ is strongly imprimitive.

Thus we may assume that the trace must be $-5$ (as the other class of involutions in $B_5$ have trace $8$ on $M(E_7)$). This means that the composition factors of $M(B_5){\downarrow_{L'}}$ are $3,2^4$. In particular $L'$, and hence $L$ (as all normal subgroups of $L$ contain the central involution) acts faithfully on the $32$-dimensional factor, and thus $L\mb T/\mb T\cong L$. Thus the projection of $L$ onto the $B_5$ factor is $L$ itself, and the image of $L$ in $B_5T_1$ is a subdirect product of $\GU_2(3)$ and a group of order $4$. Therefore there is a copy of $L$ inside the $B_5$ subgroup acting on $M(B_5)$ with composition factors of degrees $3,2^4$. But a non-central involution in $L$ obviously has trace $\pm 1$ on all simple modules of dimension $3$, and $0$ on all simple modules of dimension $2$. The trace of an involution in $B_5$ on $M(B_5)$ is one of $11$, $-5$ and $3$. This is a clear contradiction.

\medskip

\noindent $\boldsymbol{q=9}$: There are 34 sets of composition factors for $M(E_7){\downarrow_H}$ that are conspicuous for elements of order up to order $20$, so we eliminate some first. Only $7_i$ and $21_{i,j}^\pm$ have non-zero $1$-cohomology, so we first eliminate all sets of composition factors with non-positive pressure (and trivial factors) using Proposition \ref{prop:pressure}. Up to field and graph automorphisms, this leaves the following sets of composition factors that are conspicuous for elements of order at most $20$:
\[ 7_1^6,(3_2,3_2^*)^2,1^2,\quad 15_1,15_1^*,7_1^2,3_1,3_1^*,3_2,3_2^*,\quad 7_1^2,(6_1,6_1^*)^3,3_2,3_2^*,\] \[ 18_{21},18_{21}^*,7_1^2,3_2,3_2^*,\quad \bar{18}_{21},\bar{18}_{21}^*,7_1^2,3_2,3_2^*.\]

For each of these, we switch to the Lie algebra:
\[ 27_1^3,7_1^6,3_2,3_2^*,1^4,\quad 27_1,(15_1,15_1^*)^2,\bar 9_{12},\bar 9_{12}^*,7_1^2,6_1,6_1^*,1^2,\quad  27_1,(15_1,15_1^*)^3,7_1,1^9,\]
\[ 45_{21},45_{21}^*,27_1,7_1,7_2,1^2,\quad \bar{45}_{21},\bar{45}_{21}^*,27_1,7_1,7_2,1^2.\]
All but the first case have non-positive pressure, so $H$ stabilizes a line on $L(E_7)$.

The first case has pressure $2$. However, the $\{1,3_2^\pm,7_1,27_1\}$-radical of $P(7_1)$ is
\[ 7_1,7_1/1,3_2,3_2^*/7_1,\]
so clearly we must stabilize a line on $L(E_7)$ in this case as well, just as for the corresponding case for $\PSU_3(3)$.

Thus $H$ stabilizes a line on either $M(E_7)$ or $L(E_7)$ in all cases.

\medskip

\noindent $\boldsymbol{q=4}$: From Table \ref{t:modules48916}, we see that the only modules with non-zero $1$-cohomology are $9$-dimensional, and they always have $1$-dimensional $1$-cohomology, so we need either no trivial composition factors or more $9$-dimensional factors than trivial factors, else $H$ stabilizes a line on $M(E_7)$ by Proposition \ref{prop:pressure}. In fact, there are twenty conspicuous sets of composition factors for $M(E_7){\downarrow_H}$, eighteen of which have negative or zero pressure (and all have trivial factors), so we only have to consider two sets of factors, one up to field automorphism. Thus we reduce to considering
\[ 9_{12},9_{12}^*,\bar 9_{12},\bar 9_{12}^*,(3_1,3_1^*)^2,3_2,3_2^*,1^2.\]
(Note that the field-graph automorphisms act as a cycle 
\[9_{12}\to \bar 9_{12}^*\to 9_{12}^*\to \bar 9_{12}\]
with the labelling convention from Chapter \ref{chap:labelling}.)

Let $W$ denote the $\{3_i^\pm\}'$-heart of $M(E_7){\downarrow_H}$. Suppose first that the socle of $W$ is simple, say $9_{12}$. The $\{1,3_i^\pm,\bar 9_{12}^\pm\}$-radical of $P(9_{12})$ is
\[ 1/3_1^*,\bar 9_{12}^*,\bar 9_{12}/3_2,3_2^*/1,3_1^*/9_{12}.\]
We cannot place a copy of $9_{12}^*$ on top of this module and so we cannot build a pyx for $W$. 

Thus we need two $9$-dimensional modules in the socle, and it cannot be $9_{12}\oplus 9_{12}^*$, so (up to automorphism) it is $9_{12}\oplus \bar 9_{12}$. (To choose this we might need to replace the composition factors by their image under the field automorphism.) This means that we may remove the $\bar 9_{12}$ from the above module, and $W$ is uniquely determined up to isomorphism, as
\begin{equation} (\bar 9_{12}^*/3_2^*/1,3_1^*/9_{12})\oplus (9_{12}^*/3_1/1,3_2/\bar 9_{12}).\label{eq:psu34ine7W}\end{equation}
We cannot place a copy of $3_1$ or $3_2^*$ on top of (\ref{eq:psu34ine7W}) but we can place a single copy of $3_1^*$, and two copies of $3_2$. We obtain two modules by taking the direct sum of (\ref{eq:psu34ine7W}) with either $3_1\oplus 3_1^*$ or $3_2\oplus 3_2^*$. We obtain one more module by extending above by $3_1^*$ and below by $3_1$. Finally, we obtain three more modules by extending above by $3_2$ and below by $3_2^*$.

Let $L\cong \Alt(5)$ be a subgroup of $H$. We can consider the fixed-point space $M(E_7)^L$ for the actions above. For three of these it has dimension $4$ and for three it has dimension $5$. Those with dimension $5$ have a summand of dimension $3$.

Note that $L$ centralizes an element of order $5$ with centralizer $A_4A_1T_2$ \cite[Table II]{frey2001}. Thus $L\leq A_4A_1$, and the composition factors of $A_4A_1$ on $M(E_7)$ are given in \cite[Table 21]{thomas2016}. The best way to proceed though is to embed $A_4A_1$ in $A_6$ and take the sum of $M(A_4)$ and $M(A_1)$. The structure of $M(E_7){\downarrow_{A_6}}$ is easier, and is
\[ M(A_6)\oplus M(A_6)^*\oplus \Lambda^2(M(A_6))\oplus \Lambda^2(M(A_6))^*.\]
The composition factors of $L$ on $M(E_7)$ force the action of $L$ on $M(A_1)$ to be $2_1$, and on $M(A_4)$ to be $2_1,2_2,1$ or $4,1$, or the action of $L$ on $M(A_1)$ to be $2_2$ and on $M(A_4)$ to be $2_1^2,1$. If the action on $M(A_1)$ is $2_1$ then there are six possible module structures for $L$ on $M(A_4)$, up to duality:
\[ 4\oplus 1,\quad 2_1\oplus 2_2\oplus 1,\quad (2_1/1)\oplus 2_2,\quad (2_2/1)\oplus 2_1,\quad 2_1/1/2_2,\quad 2_1,2_2/1.\]
For each of these we can construct the corresponding action of $L$ on $M(E_7)$ and compute the fixed-point space $M(E_7)^L$: their dimensions are $6$, $10$, $7$, $8$, $5$ and $6$ respectively.

If the action on $M(A_1)$ is $2_2$ then up to duality there are two module actions for $L$ on $M(A_4)$, namely
\[ 2_1^{\oplus 2}\oplus 1,\quad (2_1/1)\oplus 2_1,\]
and the dimensions of $M(E_7)^L$ in these cases are $10$ and $7$ respectively. (They appeared in the previous list.)

Comparing to the dimension of $M(E_7)^L$ computed above, we find a unique case of interest, of action (up to graph automorphism) of $2_1\oplus (2_1/1/2_2)$. However, in this one case the precise module structures of the two actions don't match up. The action of this copy of $L$ on $M(E_7)$ is
\[ 1^{\oplus 2}\oplus 2_1^{\oplus 2}\oplus 4^{\oplus 4}\oplus (2_1/1/2_2)\oplus (2_2/1/2_1)\oplus (1/2_1,2_2/1)^{\oplus 2}\oplus (1/2_2/1/2_1)\oplus (2_1/1/2_2/1),\]
and this does not have a summand of dimension $3$.

Thus the module action for $H$ is not correct, and so $H$ stabilizes a line on $M(E_7)$, as needed.

\medskip

\noindent $\boldsymbol{q=8}$: Up to field automorphism there are four conspicuous sets of composition factors for $M(E_7){\downarrow_H}$:
\[ 8_1^6,1^8,\quad (9_{12},9_{12}^*)^2,8_1^2,1^4,\quad (9_{12},9_{12}^*)^2,8_2^2,1^4,\quad 9_{12},9_{12}^*,9_{23},9_{23}^*,\bar 9_{13},\bar 9_{13}^*,1^2.\]
The first three of these are of non-positive pressure, so $H$ stabilizes a line on $M(E_7)$ in these cases.

In the final case, we first assume that the socle of $M(E_7){\downarrow_H}$ is simple, say $9_{12}$. Taking the $\{1,9_{23}^\pm,\bar 9_{13}^\pm\}$-radical of the quotient $P(9_{12})$ then adding copies of $9_{12}^*$ on top, yields a pyx of $M(E_7){\downarrow_H}$, and this is $56$-dimensional, having structure
\[ 9_{12}^*/1,9_{23}^*,\bar 9_{13}^*/1,9_{23},\bar 9_{13}/9_{12}.\]
Thus this must be exactly $M(E_7){\downarrow_H}$. Let $L$ denote a subgroup $\PSL_2(8)$ that centralizes an element of order $3$ in $H$ with trace $-7$; then $L$ lies inside the centralizer of this element, which is $A_6T_1$, so $L\leq A_6$, and $A_6$ acts on $M(E_7)$ as
\[ L(\lambda_1)\oplus L(\lambda_2)\oplus L(\lambda_5)\oplus L(\lambda_6).\]
The restriction of the pyx above to $L$ is a sum of eight $7$-dimensional modules, but for each of them their exterior square is indecomposable, so this is not compatible with the structure above.

\medskip

An element $v\in H$ of order $4$ acts on $9_{i,j}^\pm$ and $\bar 9_{i,j}^\pm$ with Jordan blocks $4^2,1$, so on $M(E_7)$ with at least twelve blocks of size $4$. Thus we see from \cite[Table 7]{lawther1995} that $v$ acts on $M(E_7)$ with blocks $4^{12},2^4$, and in particular there is no simple summand of $M(E_7){\downarrow_H}$. So if a $9$-dimensional module lies in the socle of this module, then its dual does not.

If the socle is $9_{12}\oplus 9_{23}$ then the corresponding pyx to that above is
\[ (9_{23}^*/1,\bar 9_{13}/9_{12})\oplus (9_{12}^*/1,\bar 9_{13}^*/9_{23}),\]
again it has dimension $56$ so must be isomorphic to $M(E_7){\downarrow_H}$, and the restriction of it to $L$ has two summands of dimension $i$ for each $1\leq i\leq 7$. (The modules of dimensions $1$, $2$ and $4$ are simple, the others are not and are not self-dual.) We will try to find which of these summands lie inside $L(\lambda_1)$: it cannot restrict to $L$ as $1\oplus V$ for some $V$, for then $(V\oplus V^*)^{\oplus 2}$ appears in $M(E_7){\downarrow_L}$, which is not possible. The exterior square of the $7$-dimensional summand is indecomposable, so we are left with summands of dimensions $5$ and $2$, and of dimensions $4$ and $3$. The exterior square of the $3$-dimensional summand is also $3$-dimensional, so this is not allowed, and the exterior square of the $5$-dimensional summand is a $10$-dimensional indecomposable module, so there is no compatibility. Hence $H$ cannot embed with this socle.

If the socle is $9_{12}\oplus \bar 9_{13}^*$, however, we can place $1^{\oplus 2}\oplus 9_{23}^{\oplus 2}$ on top of this, but not $9_{23}^*$ and so this cannot be the socle. Up to automorphism these are the only options with two factors in the socle.

If the socle has three $9$-dimensional submodules, however, then the structure must have the form $9,9,9/1,1/9,9,9$, and in particular there is a uniserial module $9/1/9$, which we cannot construct. (For example, it does not appear in the submodule of $P(9_{12})$ above.)

This proves that $H$ must stabilize a line on $M(E_7)$, as needed.
\end{proof}

\begin{proposition}\label{prop:sp4ine7} Let $H\cong \PSp_4(q)$ for $q=3,4,8,9$, or ${}^2\!B_2(q)$ for $q=8,32$, and let $\bar H\cong \Sp_4(q)$ for $q=3,9$.
\begin{enumerate}
\item If $q=4,8,32$ then $H$ stabilizes a line on $M(E_7)$.
\item If $q=3$ then $H$ stabilizes a line on $M(E_7)$ and $\bar H$ does not embed in $\mb G$.
\item If $q=9$ then $H$ is a blueprint for $M(E_7)$, and $\bar H$ does not embed in $\mb G$.
\end{enumerate}
\end{proposition}
\begin{proof} $\boldsymbol{q=3}$: If $H$ embeds in $\mb G$ then there are two conspicuous sets of composition factors for $M(E_7){\downarrow_H}$, and they are $10^2,5^6,1^6$ and $25^2,1^6$. The pressures of these are $0$ and $-6$ respectively, so $H$ stabilizes a line on $M(E_7)$ by Proposition \ref{prop:pressure}. As seen with the case $q=5,7$ in \cite[Proposition 3.8]{craven2015un}, $\bar H$ cannot embed in $\mb G$ with the centres coinciding, because all factors of $M(E_7){\downarrow_H}$ are faithful, but the trace of a non-central involution on faithful modules is $0$, and the trace of an involution in $\mb G$ is $\pm 8$, a contradiction.

\medskip

\noindent $\boldsymbol{q=9}$: we need only consider $H\cong \PSp_4(9)$ embedding in $\mb G$ and not the double cover because $\Sp_4(3)$ does not embed with $Z(\bar H)=Z(\mb G)$. Using traces of elements of order up $8$, we find up to field automorphism three conspicuous sets of composition factors:
\[ 10_1^2,5_1^6,1^6,\quad 25_1^2,1^6,\quad 16^2,5_1^2,5_2^2,1^4.\]
For each of these, an element of order $41$ in $H$ centralizes at least an $8$-space on $M(E_7)$, so by Theorem \ref{thm:largeorderss}, $H$ is a blueprint for $M(E_7)$.

\medskip

\noindent $\boldsymbol{q}$\textbf{ even}: Note that in all cases $H$ contains a rational element $x$ of order $5$. There is a unique rational class of semisimple elements of order $5$ in $E_7$, with trace $6$ on $M(E_7)$. Note that the trace of $x$ on modules of dimension $1$, $4$ and $16$ is $1$, $-1$ and $1$ respectively: as there are at most three $16$-dimensional composition factors in $M(E_7){\downarrow_H}$ (as $\dim(M(E_7))=56$), there must be at least three more trivial factors than $4$-dimensional ones. As the pressure is the number of $4$-dimensional factors minus the number of trivial factors (see \cite{sin1992b}), the pressure must be negative. Thus $H$ stabilizes a line on $M(E_7)$.
\end{proof}

\begin{proposition}\label{prop:g2ine7} Let $H\cong G_2(q)'$ for $q=2,3,4,8$ or ${}^2\!G_2(q)'$ for $q=3,27$.
\begin{enumerate}
\item If $q=2$ or $H\cong {}^2\!G_2(3)'$ then $H$ stabilizes a line on either $M(E_7)$ or $L(E_7)^\circ$.
\item If $q=3,8,27$ but $H$ is not ${}^2\!G_2(3)'$, then $H$ is a blueprint for $M(E_7)$.
\item If $q=4$ then $H$ stabilizes a line on either $M(E_7)$ or $L(E_7)$, or is a blueprint for $M(E_7)$.
\end{enumerate}
\end{proposition}
\begin{proof} $\boldsymbol{3\mid q}$: Letting $H\cong \PSL_2(8)={}^2\!G_2(3)'$, the conspicuous sets of composition factors for $M(E_7){\downarrow_H}$ are
\[7^8,\qquad 7^6,1^{14},\qquad 9_1^2,9_2^2,9_3^2,1^2,\qquad 9_1^5,7,1^4,\qquad 9_1^2,9_2^2,7^2,1^6,\qquad 9_1,9_2^2,9_3^3,1^2.\]
As $H^1(H,7)$ is $1$-dimensional and $9_i$ is projective, we see that in all but the first case $H$ must stabilize a line on $M(E_7)$ by Proposition \ref{prop:pressure}, as the pressures are $8$, $-8$, $-2$, $-3$, $-4$ and $-2$ respectively. For the first case, the composition factors of $L(E_7){\downarrow_H}$ are $9_1^3,9_2^3,9_3^3,7^7,1^3$, and as the only indecomposable module with a trivial composition factor and no trivial submodule or quotient is
\[ P(7)=7/\left((7/7/7)\oplus 1\right)/7,\]
we see that $H$ must stabilize a line on $L(E_7)$.

We must also deal with $H\cong G_2(3)$ and $H\cong {}^2\!G_2(27)$. Both contain $L={}^2\!G_2(3)$, which has composition factors on $M(E_7)$ one of
\[ 27^2,1^2,\qquad 7^6,1^{14},\qquad 7^8.\]
The modules for $H$ of dimension at most $56$ have dimensions $1$, $7$, $27$ and $49$: the first three restrict irreducibly to $L$, and the fourth restricts as $7^{\otimes 2}\cong 27\oplus (7/7/7)\oplus 1$. Thus the composition factors of $M(E_7){\downarrow_H}$ must have the same dimensions as the factors of $M(E_7){\downarrow_L}$. 

The first case is a blueprint for $M(E_7)$ as it lies in an $E_6$-Levi subgroup. The second is a blueprint for $M(E_7)$ as it is semisimple, so $u$ has at least fourteen blocks of size $1$ in its action on $M(E_7)$, hence by \cite[Table 7]{lawther1995} lies in a generic class (using Lemma \ref{lem:genericmeansblueprint}).

For the third case, any semisimple element in $H$ must centralize an $8$-space on $M(E_7)$, since semisimple elements of $H$ are real. As ${}^2\!G_2(27)$ contains an element of order $37>30$, it is a blueprint for $M(E_7)$ by Theorem \ref{thm:largeorderss}. For the other group $H\cong G_2(3)$, first note that, up to graph automorphism, the only conspicuous set of composition factors for $M(E_7){\downarrow_H}$ consisting only of $7$-dimensional modules is $7_1^6,7_2^2$. Then $H$ contains an element $x$ of order $13$, and its trace on $M(E_7)$ shows that it lies in $F_4$. Furthermore, there exist elements of order $26$ in $F_4$ that square to $x$ and have the same number (thirteen) of distinct eigenvalues on $M(F_4)$. Thus $x\in F_4$ is a blueprint for $M(F_4)$ by Theorem \ref{thm:largeorderss}, and hence also for $M(E_7)$ as well. Thus so is $H$. This completes the proof.

\medskip

\noindent $\boldsymbol{q=2}$: The conspicuous sets of composition factors for $M(E_7){\downarrow_H}$ are
\[ 6^9,1^2,\qquad 6^6,1^{20},\qquad 14^2,6^4,1^4.\]
The second and third cases have pressures $-14$ and $0$ respectively, so both stabilize lines on $M(E_7)$ by Proposition \ref{prop:pressure}, as needed. For the first case, the corresponding factors on $L(E_7)^\circ$ are $14^8,6,1^{14}$. Clearly $H$ has at least twelve trivial summands on $L(E_7)^\circ{\downarrow_H}$, as needed.

\medskip

\noindent $\boldsymbol{q=4}$: We see from Table \ref{t:modules48916} that the modules with non-zero $1$-cohomology of dimension at most $56$ are $6_1$, $6_2$ and $36$, the last of which restricts to the subgroup $G_2(2)'$ of $H$ with factors $14^2,6,1^2$. Thus if the restriction of $M(E_7)$ to $G_2(2)'$ is either the first or second case above, then $H$ must stabilize a line on $L(E_7)$ and $M(E_7)$ respectively, by Proposition \ref{prop:pressure}. For the third case, it either has pressure $0$ or has composition factors of dimension $36,6^3,1^2$ and has pressure $2$.

There is a unique (up to field automorphism) such conspicuous set of composition factors for $M(E_7){\downarrow_H}$, which is $36,6_1^2,6_2,1^2$. Suppose that $H$ does not stabilize a line on $M(E_7)$, so that $M(E_7){\downarrow_H}$ is a submodule of $P(6_1)$, perhaps plus summands $6_2$ and $36$. Since there is no module $36/1/6_1$, we must have a submodule $1/6_2/1/6_1$ of $M(E_7){\downarrow_H}$, which is unique up to isomorphism.

Let $L$ denote a copy of $\SL_3(4)$ in $H$. It centralizes an element of order $3$ that acts with trace $2$ on $M(E_7)$ and $-2$ on $L(E_7)$, which means it has centralizer $\mb X=A_5A_2$ (see \cite[Table II]{frey2001}), so $L\leq \mb X$. The composition factors of $M(E_7){\downarrow_L}$ are
\[ 9_{12},9_{12}^*,\bar 9_{12},\bar 9_{12}^*,(3_1,3_1^*)^2,3_2,3_2^*,1^2.\]

Since $Z(L)$ acts faithfully on both $M(A_5)$ and $M(A_2)$, we see that the restrictions $M(A_5){\downarrow_L}$ and $M(A_2){\downarrow_L}$ have factors only of dimension $3$. If $M(A_2){\downarrow_L}=3_1$ then $M(A_5){\downarrow_L}$ has factors either $3_1,3_2$ or $3_1,3_2^*$, and if $M(A_2){\downarrow_L}=3_2$ then $M(A_5){\downarrow_L}$ has factors $3_1,3_1^*$.

If the restriction $M(A_5){\downarrow_L}$ is semisimple then there is an obvious $A_2$ subgroup $\mb Y$ containing $L$, and as $\mb X$ acts on $M(E_7)$ as
\[ (\lambda_1,\lambda_1)\oplus (\lambda_1,\lambda_1)^*\oplus (\lambda_3,0),\]
one can compute immediately the actions of $L$ and $\mb Y$ on these three modules. By placing $L$ inside $\mb Y$, acting on $M(A_2)$ as $L(10)$ and on $M(A_5)$ as either $L(10)\oplus L(20)$ or $L(10)\oplus L(02)$, or on $M(A_2)$ as $L(20)$ and on $M(A_5)$ as $L(10)\oplus L(01)$, we see that $L$ and $\mb Y$ stabilize the same subspaces of $M(A_5)\otimes M(A_2)$ and its dual. As for any modules $M,M'$, 
\[ \Lambda^3(M\oplus M')\cong \Lambda^3(M)\oplus (\Lambda^2(M)\otimes M')\oplus (M\otimes \Lambda^2(M'))\oplus \Lambda^3(M'),\]
we compute that $\mb Y$ acts on $\Lambda^3(M(A_5))$ as
\[ L(0)^{\oplus 2}\oplus L(12)\oplus L(21),\qquad L(0)^{\oplus 2}\oplus L(30)\oplus L(03),\qquad L(0)^{\oplus 2}\oplus L(10)^{\otimes 2}\oplus L(01)^{\otimes 2}.\]
In all cases, $L$ and $\mb Y$ stabilize the same subspaces of $M(E_7)$, so $L$, and $H$, are blueprints for $M(E_7)$.

Thus we may assume that $L$ does not act semisimply on $M(A_5)$, and hence the factors of $M(A_5){\downarrow_L}$ are $3_1,3_2^*$, as there are no extensions between the other factors. First, note that the module $1/6_2/1/6_1$ for $H$ restricts to $L$ with two trivial submodules, hence $M(E_7){\downarrow_L}$ has two trivial summands, not just composition factors. Since $\dim(\Ext_{kL}^1(3_2^*,3_1))=2$, there is a module $3_2^*,3_2^*/3_1$, so $3_2^*/3_1$ is not defined uniquely up to isomorphism. Let $v$ denote an element of order $4$ in $L$, which must act on $M(E_7)$ with at least two blocks of size $1$, as $M(E_7){\downarrow_L}$ has two trivial summands. Then $v$ acts on $3_1$ with a single block of size $3$, hence on $3_1\oplus 3_2^*$ with blocks $3^2$, and on $3_2^*/3_1$ with blocks either $3^2$ or $4,2$. The exterior cube of the blocks $4,2$ has block structure $4^4,2^2$, which is incompatible with there being two trivial summands in $M(E_7){\downarrow_H}$. Thus $v$ acts on $3_2^*/3_1$ with blocks $3^2$, and there is a unique such indecomposable module. However, if $v$ came from a different class of elements of order $4$ then this yields a different extension $3_2^*/3_1$ on which $v$ acts with blocks $3^2$. Thus the only extension on which two different choices of $v$ both act correctly is the split extension $3_2^*\oplus 3_1$.

Thus $H$ is a blueprint for $M(E_7)$ or stabilizes a line on $M(E_7)$, as needed.

\medskip

\noindent $\boldsymbol{q=8}$: Since $\PSL_3(8)$ is a blueprint for $M(E_7)$ by Proposition \ref{prop:sl3ine7}, and $G_2(8)$ contains $\PSL_3(8)$, we are done.
\end{proof}

\begin{proposition}\label{prop:sl4ine7} If $H\cong \PSL_4(3)$ then $H$ is a blueprint for $M(E_7)$. If $\bar H\cong \SL_4(3)$ then $\bar H$ does not embed in $\mb G$.
\end{proposition}
\begin{proof} The two conspicuous sets of composition factors for $M(E_7){\downarrow_H}$ are
\[ 10,10^*,6^6,\qquad 15^2,6^4,1^2.\]
In both cases there are no extensions between the factors so $M(E_7){\downarrow_H}$ is semisimple, and $u$ acts on these with blocks $3^2,2^{16},1^{18}$. This is the generic class $2A_1$, and so $H$ is a blueprint for $M(E_7)$. Since there is no embedding of $\Sp_4(3)$ into $\mb G$ with centres coinciding, the same holds for $\SL_4(3)$.
\end{proof}

For $H\cong \PSU_4(q)$, we know we need to consider $q=2,3,4$.

\begin{proposition}\label{prop:su4ine7} Let $H\cong \PSU_4(q)$ for $q=2,3,4$, and let $\bar H\cong 2\cdot \PSU_4(3)$.
\begin{enumerate}
\item If $q=2,4$ then $H$ stabilizes a line on either $M(E_7)$ or $L(E_7)^\circ$.
\item If $q=3$ then $H$ does not embed in $\mb G$ and $\bar H$ is a blueprint for $M(E_7)$.
\end{enumerate}
\end{proposition}
\begin{proof} $\boldsymbol{q=3}$: As $H$ does not embed in $E_8$ by Proposition \ref{prop:su4ine8}, $H$ does not embed in $E_7$ either. If $\bar H=2\cdot H$ embeds in $\mb G$ with $Z(\bar H)=Z(\mb G)$, then the only conspicuous set of composition factors for $M(E_7){\downarrow_H}$ is $10,10^*,6^6$, and the $6$s must split off as summands. This means that $u$ acts with blocks at least $2^{12},1^2$. Hence this element must come from a generic class, as we see from \cite[Table 7]{lawther1995}, so that $H$ is a blueprint for $M(E_7)$ by Lemma \ref{lem:genericmeansblueprint}.

\medskip

\noindent $\boldsymbol{q=2}$: There are three conspicuous sets of composition factors for $M(E_7){\downarrow_H}$:
\[ 6^8,4,4^*,\quad 14^2,6^4,1^4,\quad 6^2,(4,4^*)^4,1^{12}.\]
The pressures of the second and third modules are $-2$ and $-4$ respectively, so in both cases $H$ stabilizes lines on $M(E_7)$ by Proposition \ref{prop:pressure}. In the first case, we switch to $L(E_7)^\circ$, which has composition factors $14^8,6,1^{14}$, so clearly $H$ stabilizes a line on $L(E_7)^\circ$.

\medskip

\noindent $\boldsymbol{q=4}$: Up to field automorphism, there are four conspicuous sets of composition factors for $M(E_7){\downarrow_H}$:
\[14_1^2,6_1^4,1^4,\quad 16_{12},16_{12}^*,6_1^2,6_2^2,\quad 6_1^8,4_2,4_2^*,\quad 6_1^2,(4_1,4_1^*)^4,1^{12}.\]
The pressures of the first and last sets of factors are $-2$ and $-12$, so $H$ stabilizes lines in these cases by Proposition \ref{prop:pressure}. In the second and third cases we switch to $L(E_7)^\circ$, where we find composition factors of
\[ 36,(\bar{16}_{12},\bar{16}_{12}^*)^2,14_1,14_2,1^4,\quad 14_1^8,6_2,1^{14}\]
respectively. The pressures of these are $-2$ and $-6$ respectively, so $H$ stabilizes a line on $L(E_7)^\circ$, as claimed.
\end{proof}

For $H\cong \PSp_6(q)$, we know from above that we need only check $q=2$.

\begin{proposition} If $H\cong \PSp_6(2)$ then $H$ stabilizes a line on either $M(E_7)$ or $L(E_7)^\circ$.
\end{proposition}
\begin{proof} There are three conspicuous sets of composition factors for $M(E_7){\downarrow_H}$:
\[ 8^4,6^2,1^{12},\quad 14^2,6^4,1^4,\quad 8,6^8.\]
The first two cases have pressures $-10$ and $0$ respectively, so $H$ stabilizes a line on $M(E_7)$ in both cases by Proposition \ref{prop:pressure}. In the third case, we switch to the Lie algebra $L(E_7)^\circ$, on which $H$ must act with composition factors
\[ 14^8,6,1^{14}.\]
Clearly $H$ stabilizes a line on $L(E_7)^\circ$ as it has pressure $-13$.
\end{proof}

\newpage

\chapter{Subgroups of \texorpdfstring{$E_6$}{E6}}
\label{chap:subsine6}

Now we let $\mb G$ be a simply connected algebraic group of type $E_6$ and $H$ be a subgroup of rank at most $3$, $q\leq 9$, together with a small Ree group, a Suzuki group, $\PSL_3(16)$ or $\PSU_3(16)$. By embedding $\mb G$ inside $E_7$, if the Schur multiplier of $H$ is not divisible by $3$ and $H$ is a blueprint for $M(E_7)$ or $L(E_7)$, then it is a blueprint for $M(E_6)\oplus M(E_6)^*$ or $L(E_6)$ respectively. We may also use Theorem \ref{thm:largeorderss} to eliminate $H$ with semisimple elements of odd order greater than $75$, as with the previous chapter. We can also eliminate any subgroup containing a real element of order at least $19$. These conditions together eliminate many possibilities for $H$, leaving the following:
\begin{enumerate}
\item $\PSL_3(q)$ for $q=2,3,4,7$;
\item $\PSU_3(q)$ for $q=3,4,5,8,9$;
\item $\PSp_4(q)$ for $q=3,4$;
\item $G_2(2)'$;
\item ${}^2\!G_2(3)'$;
\item ${}^2\!B_2(8)$;
\item $\PSU_4(q)$ for $q=2,4$;
\item $\PSp_6(2)$.
\end{enumerate}
Note that if $H$ is $\PSL_3(q)$, $\PSU_3(q)$, $G_2(q)$ or ${}^2\!G_2(q)$ for $q$ a power of $3$, then $H$ can act irreducibly on $M(E_6)$, but this action is self-dual. This means that $H$ stabilizes diagonal submodules of $M(E_6)\oplus M(E_6)^*$ that are not stabilized by $\mb G$. Thus the statement that $H$ is a blueprint for $M(E_6)\oplus M(E_6)^*$ actually shows that $H$ is not Lie primitive. This gives a proof that all of these groups are not Lie primitive, hence strongly imprimitive, except for $\PSL_3(3)$, $\PSU_3(3)$ and ${}^2\!G_2(3)'$. These are the three left out of the proof in \cite{liebeckseitz2004a}, where a different proof of imprimitivity for the other subgroups is given.

In this chapter we will eliminate all of the cases above except for $\PSL_3(3)$, $\PSU_3(3)$ and ${}^2\!G_2(3)'$, the first two acting irreducibly on $M(E_6)$, and the third acting as the sum of three modules that are permuted by ${}^2\!G_2(3)$, so that that group (if ${}^2\!G_2(3)'$ extends to it) acts irreducibly on $M(E_6)$. These three cases will be dealt with in Chapter \ref{chap:trilinear}.

As with the previous chapters, in this chapter $u$ denotes an element of order $p$ in $H$ belonging to the smallest conjugacy class.

\medskip

Our first proposition leaves open the possibility of a Lie primitive copy of $\PSL_3(3)$ acting irreducibly on $M(E_6)$, as mentioned above. See Section \ref{sec:irredpsl33} for how to prove strong imprimitivity in this case.

\begin{proposition}\label{prop:sl3ine6} Let $H\cong \PSL_3(q)$ for $q=2,3,4,7$, and let $\bar H\cong \SL_3(q)$ for $q=4,7$.
\begin{enumerate}
\item If $q=2$ then $H$ stabilizes a line or hyperplane on either $M(E_6)$ or $L(E_6)$.
\item If $q=3$ then $H$ stabilizes a line on either $M(E_6)$ or $L(E_6)^\circ$, or is a blueprint for $M(E_6)\oplus M(E_6)^*$, or acts irreducibly on $M(E_6)$.
\item If $q=4$ then $H$ stabilizes a line or hyperplane on $M(E_6)$, and $\bar H$ stabilizes a line on $L(E_6)$.
\item If $q=7$ then $H$ and $\bar H$ are blueprints for $M(E_6)\oplus M(E_6)^*$.
\end{enumerate}
\end{proposition}
\begin{proof} $\boldsymbol{q=7}$: From Proposition \ref{prop:sl3ine7} we see that $H$ is a blueprint for $M(E_6)$, and so we only consider $\bar H$. There are three sets of composition factors for $M(E_6){\downarrow_{\bar H}}$ that are conspicuous for elements of order at most $8$:
\[ 6,3^7,\quad 6^3,3^3,\quad 15_2,6^2.\]
The second of these is not conspicuous for elements of order $16$. The other two must be semisimple, with $u$ acting on $M(E_6)$ with blocks $3,2^8,1^8$ (class $2A_1$) and $4,3^4,2^4,1^3$ (class $A_2+A_1$) respectively, both of which are generic (see \cite[Table 5]{lawther1995}). This completes the proof.

\medskip

\noindent $\boldsymbol{q=3}$: The conspicuous sets of composition factors for $M(E_6){\downarrow_H}$ are
\[ 27,\quad 15,6^2,\quad 7^3,3,3^*,\quad 7^3,1^6,\quad 7,6,6^*,3,3^*,1^2,\]
\[ 7,(3,3^*)^3,1^2,\quad 6,3^7,\quad (3,3^*)^3,1^9.\]

\noindent \textbf{Cases 4, 5, 6 and 8}: The cases with a trivial composition factor all have negative pressure, so $H$ stabilizes a line on $M(E_6)$ in all cases by Proposition \ref{prop:pressure}.

\medskip

\noindent \textbf{Case 2}: It is easiest to switch to the Lie algebra $L(E_6)^\circ$, of dimension $77$, where the corresponding set of composition factors is
\[ 27,7^5,(3,3^*)^2,1^3.\]
As we saw in Proposition \ref{prop:sl3ine8}, the $\{1,3^\pm,7\}$-radical of $P(7)$ is $3,3^*/7,7/1,3,3^*/7$, so we need at least twice as many $7$s as $1$s in order not to stabilize a line on $L(E_6)^\circ$, which we do not have.

\medskip

\noindent \textbf{Case 3}: There is an extension between $7$ and $3^\pm$ but none between $3$-dimensional modules. Let $v$ be from a unipotent class of elements of order $3$ lying in $\PSL_2(3)\leq H$; then $v$ acts on $3$ as simply $3$, and on $7$ as $3^2,1$. So $v$ acts on $M(E_6)$ with at least eight blocks of size $3$, hence on $M(E_6)$ as $3^9$ by \cite[Table 5]{lawther1995}. However, the $\{3^\pm,7\}$-radicals of $P(7)$ and $P(3)$ are $7/3,3^*/7$ and $3^*/7/3$. Thus there must be a subquotient $7\oplus 7$, and this contradicts the action of $v$. Therefore $H$ cannot embed with these factors.

\medskip

\noindent \textbf{Case 7}: This case must be semisimple, and then $u$ acts with Jordan blocks $3,2^8,1^8$ on $M(E_6)$, which is the generic class $2A_1$.

\medskip

\noindent \textbf{Case 1}: This leaves the irreducible case, as claimed.

\medskip

\noindent $\boldsymbol{q=2}$: There are four conspicuous sets of composition factors, namely
\[ 8^3,1^3,\quad 8,(3,3^*)^3,1,\quad (3,3^*)^3,1^9,\quad 3^8,3^*.\]
The projective cover $P(3)$ has structure
\[ 3/(1\oplus (3^*/3/3^*))/3,\]
so in order to have a trivial composition factor but no trivial submodule or quotient we need a copy of $P(3)$ (or $P(3^*)$). Thus the first and third cases yield stabilized lines, and the second case must be, up to duality,
\[ 8\oplus P(3)\oplus 3.\]
The action of $v\in H$ of order $4$ on this module has blocks $4^6,3$, but this is not in \cite[Table 5]{lawther1995}, a contradiction.

For the fourth case, we switch to the Lie algebra, and note that the corresponding set of factors on $L(E_6)$ is $8^8,1^{14}$, which is semisimple and therefore $H$ stabilizes a line on $L(E_6)$.

\medskip

\noindent $\boldsymbol{q=4}$: There are two conspicuous sets of composition factors for $M(E_6){\downarrow_H}$ up to field automorphism, namely
\[ 8_1^3,1^3,\quad 9,9^*,8_1,1.\]
The first case must centralize a $3$-space on $M(E_6)$, but the second case need not stabilize a line on $M(E_6)$, as there is a module $1,1/9$, so we consider this case. There are three $H$-conjugacy classes of subgroups isomorphic to $\PSL_3(2)$, with representatives $L_1$, $L_2$ and $L_3$. Note that the restriction of $M(E_6)$ to each $L_i$ has composition factors $8,(3,3^*)^3,1$ in the second case. We have proved that each $L_i$ stabilizes a line or hyperplane on $M(E_6)$, so we may assume that at least two of them stabilize a line on $M(E_6)$, say $L_1$ and $L_2$. This yields maps from the permutation modules $P_{L_i}$ to $M(E_6){\downarrow_H}$.

There is no module $9^*/1/9$ (there is one of the form $9^*/1,1/9$, but with no trivial quotients), so our module must have the form $9^*/1,8_1/9$. Hence the $L_i$-fixed line must lie in the submodule $1/9$. There is a unique extension $1/9$ that splits on restriction to $L_1$, but is not the same as the extension that splits on restriction to $L_2$. Thus the only consistent option is that $H$ itself stabilizes a line on $M(E_6)$ (up to duality).

If we have $\bar H\cong \SL_3(4)$, there are up to field automorphism and duality three conspicuous sets of composition factors for $M(E_6){\downarrow_{\bar H}}$, which are
\[ 3_1^8,3_2^*,\quad \bar 9_{12}^*,(3_1,3_2^*)^3,\quad 24_{21},3_1.\]
(Here, all modules come from the same faithful $2$-block.)

For the first case, the corresponding factors for $L(E_6){\downarrow_{\bar H}}$ are $8_1^8,1^{14}$, which is semisimple, so $\bar H$ stabilizes a line on $L(E_6)$.

For the second case, the corresponding factors for $L(E_6){\downarrow_{\bar H}}$ are
\[ (9,9^*)^3,8_1,8_2,1^8.\]
The $\{1,9^\pm\}$-radical of $P(9)$ is $9^*/1,1/9$, and on this we may add one copy each of $8_1$ and $8_2$, falling into the second socle layer. On this we add as many copies of $9$, $9^*$ and $1$ as we can, and we obtain a module
\[ 1,1,1,1,1,1/9,9,9^*,9^*,9^*/1,1,8_1,8_2/9.\]
Of course, all of the extra trivials are quotients, so we need as many $9$s as $1$s in order not to stabilize a line on $L(E_6)$, which we do not have. Thus $\bar H$ stabilizes a line on $L(E_6)$.

In the third case, an element $v$ of order $4$ acts on $24_{21}\oplus 3_1$ as $4^6,3$, and so $\bar H$ cannot embed with these composition factors since this action does not appear in \cite[Table 5]{lawther1995}.
\end{proof}

If $H\cong \PSL_3(3)$ acts irreducibly on $M(E_6)$ then the composition factors of $L(E_6)^\circ{\downarrow_H}$ are 
\[ 15,15^*,7^3,6,6^*,(3,3^*)^2,1^2.\]
Since the actual copy of $A_2$ lying in $G_2$ does not lie in a subgroup containing a maximal torus, this $A_2$ cannot stabilize a line on $L(E_6)^\circ$ by Lemma \ref{lem:maxrankorpara}. We prove in Section \ref{sec:irredpsl33} that all such $H$ lie in copies of $G_2(3)$ in $E_6$.

\begin{proposition}\label{prop:su3ine6} Let $H\cong \PSU_3(q)$ for $q=3,4,5,8,9$, and let $\bar H\cong \SU_3(q)$ for $q=5,8$.
\begin{enumerate}
\item If $q=3$ then $H$ stabilizes a line on $M(E_6)$, is a blueprint for $M(E_6)\oplus M(E_6)^*$, stabilizes a line on $L(E_6)^\circ$, or acts irreducibly on $M(E_6)$.
\item If $q=4$ then $H$ either stabilizes a line or hyperplane on $M(E_6)$ or $L(E_6)$, or is a blueprint for $M(E_6)\oplus M(E_6)^*$.
\item If $q=5,8$ then $H$ and $\bar H$ are blueprints for $M(E_6)\oplus M(E_6)^*$, and $\bar H$ stabilizes a line on $L(E_6)$.
\item If $q=9$ then either $H$ is a blueprint for $M(E_6)\oplus M(E_6)^*$ or $H$ stabilizes a line on $L(E_6)^\circ$. If $H$ acts irreducibly on $M(E_6)$ then $H$ is a blueprint for $L(E_6)^\circ$.
\end{enumerate}
\end{proposition}
\begin{proof} $\boldsymbol{q=5}$: There are two conspicuous sets of composition factors for $M(E_6){\downarrow_H}$: $8^3,1^3$ and $19,8$. The first of these is semisimple, so $H$ stabilizes a line on $M(E_6)$, and $u$ acts on $M(E_6)$ with blocks $3^3,2^6,1^6$. This is the generic class $3A_1$ (see \cite[table 5]{lawther1995}), so $H$ is a blueprint for $M(E_6)\oplus M(E_6)^*$ by Lemma \ref{lem:genericmeansblueprint}. The second of these will be the maximal $A_2$ subgroup; note that $\Ext_{kH}^1(8,19)$ is $3$-dimensional, so we cannot easily construct possible modules for $M(E_6){\downarrow_H}$.

The corresponding factors on $L(E_6)$ are $35,35^*,8$, with an extension between $8$ and $35$. Therefore $L(E_6){\downarrow_H}$ is either semisimple or (up to automorphism) $35^*/8/35$. In the first case, $u$ acts on $L(E_6)$ with blocks $5^8,4^2,3^5,2^6,1^3$, class $2A_2+A_1$, which is not generic for $L(E_6)$ but is generic for $M(E_6)$, so that $H$ is a blueprint for $M(E_6)\oplus M(E_6)^*$. The second case, a module $35^*/8/35$, does not actually exist. Therefore $H$ is a blueprint for $M(E_6)\oplus M(E_6)^*$, as needed.

There are two conspicuous sets of composition factors for $M(E_6){\downarrow_{\bar H}}$: $6,3^7$ and $15_2,6^2$. Both of these are semisimple, and $u$ acts on them with blocks $3,2^8,1^8$ and $4,3^4,2^4,1^3$ respectively, which are the generic classes $2A_1$ and $A_2+A_1$ respectively. Thus $\bar H$ is a blueprint for $M(E_6)\oplus M(E_6)^*$ by Lemma \ref{lem:genericmeansblueprint}.

\medskip

\noindent $\boldsymbol{q=3}$:  Up to duality there are eight conspicuous sets of composition factors for $M(E_6){\downarrow_H}$:
\[ (3,3^*)^3,1^9,\quad 7^3,1^6,\quad 7,(3,3^*)^3,1^2,\quad 7,6,6^*,3,3^*,1^2,\]
\[ 6,3^7,\quad 7^3,3,3^*,\quad 15,6^2,\quad 27.\]

\noindent \textbf{Cases 1, 2 and 3}: The first three have negative pressure, so $H$ stabilizes a line on $M(E_6)$ by Proposition \ref{prop:pressure}.

\medskip

\noindent \textbf{Case 4}: This has pressure $1$, so quotienting out by any $3^\pm$ in the socle, if $H$ does not stabilize a line on $M(E_6)$ then the socle is one of $7$, $6$ or $6^*$. The $\{1,3^\pm,6^\pm\}$-radical of $P(7)$ is
\[ 1/6,6^*/1,3,3^*/7,\]
so it cannot be $7$. There is a module $6^*/1/7/1/6$, and neither $3$ nor $3^*$ may attach to this module anywhere, so either $H$ stabilizes a line or hyperplane on $M(E_6)$ or $H$ acts on $M(E_6)$ as one of
\[ (6^*/1/7/1/6)\oplus 3\oplus 3^*\quad (6^*/1/7/1/6)\oplus (3^*/3)\]
up to swapping modules with their duals. The action of $u$ on these modules has blocks $3^6,2^2,1^5$ and $3^6,2^3,1^3$ respectively, with the former not appearing in \cite[Table 5]{lawther1995}, and the latter being class $A_2+A_1$, which acts on $L(E_6)$ with factors $3^{22},2^2,1^8$.

Thus we are in the latter situation. Since $H$ stabilizes a $3$-dimensional subspace, it lies inside a proper positive-dimensional subgroup of $\mb G$ by \cite[Lemma 1.4]{craven2017}, but we require strong imprimitivity. Thus we go through the list of members of $\ms X$, first checking which have module structures that are compatible with the module structure of $M(E_6){\downarrow_H}$.

For the reductive subgroups, $H$ cannot lie in $A_2A_2A_2$ as this has three summands of dimension $9$. If $H$ lies in $A_5A_1$ then $H$ lies in $A_5$, hence in the $A_5$-parabolic subgroup. If $H$ lies in $F_4$ then $H$ stabilizes a line on $M(E_6)$, which it does not. The $G_2$ and $C_4$ subgroups do not intersect the unipotent class $A_2+A_1$ by \cite[Tables 31 and 32]{lawther2009}. This leaves $A_2G_2$.

Of the parabolic subgroups, the $D_5$-parabolic subgroups stabilize a line or hyperplane on $M(E_6)$, so $H$ is not contained in them. If $H$ is contained in an $A_4A_1$-parabolic subgroup then $H$ is contained in an $A_4$-parabolic subgroup, hence in the $D_5$-parabolic subgroup. If $H$ is contained in the $A_2A_2A_1$-parabolic then it lies in the $A_2A_2$-parabolic, hence in the $A_5$-parabolic subgroup. This leaves the $A_5$-parabolic subgroup. This acts on $M(E_6)$ as $6/15/6$, with the top and bottom $6$ being isomorphic and not self-dual. This is therefore not compatible with the module structure either, as $H$ would have to act irreducibly on $M(A_5)$.

Thus we are left with $\mb X\cong A_2G_2$. Here the decomposition $M(E_6){\downarrow_{\mb X}}$ is as the sum of a $21$- and $6$-dimensional module, fitting the structure above, but with the $6$-dimensional module being the symmetric square of a $3$-dimensional module for $A_2$. This can only restrict to $H$ as $1^{\oplus 6}$ or $6^\pm$, not $3^\mp/3^\pm$. Thus $H$ cannot embed with this module structure.

(There is a copy of $H$ with these factors in the $A_2A_2A_2$, $F_4$ and $A_2G_2$ subgroups, but in each case $H$ stabilizes a line on $M(E_6)$.)

\medskip

\noindent \textbf{Case 5}: This must be semisimple and $u$ acts on $M(E_6)$ with blocks $3,2^8,1^8$, whence $u$ lies in the generic class $2A_1$ and $H$ is a blueprint for $M(E_6)\oplus M(E_6)^*$ by Lemma \ref{lem:genericmeansblueprint}.

\medskip

\noindent \textbf{Case 6}: There is no corresponding set of composition factors on $L(E_6)$, so this case cannot exist.

\medskip

\noindent \textbf{Case 7}: We switch to $L(E_6)^\circ$, the composition factors being 
$27,7^5,(3,3^*)^2,1^3$, which has pressure $2$. Of course, the $27$ splits off, and the $\{1,3^\pm,7\}$-radical of $P(7)$ is
\[ 3,3^*/7,7/1,3,3^*/7,\]
so we need at least three $7$s in the socle (ignoring $3,3^*$) to cover the three trivials. Therefore $H$ stabilizes a line on $L(E_6)^\circ$.

\medskip

\noindent \textbf{Case 8}: This is the irreducible case, as for $\PSL_3(3)$.

\medskip

\noindent $\boldsymbol{q=9}$: There are 36 sets of composition factors for $M(E_6){\downarrow_H}$ that are conspicuous for elements of order at most $16$. If $M(E_6){\downarrow_H}$ has real trace of an element of order $73$ in $H$, for example if $M(E_6){\downarrow_H}$ has the same composition factors as its dual, then the element, and hence $H$, is a blueprint for $M(E_6)\oplus M(E_6)^*$ by Theorem \ref{thm:largeorderss}. (This includes the irreducible case. The proof that in this case $H$ is a blueprint for $L(E_6)^\circ$ is given in \cite[Lemma 4.15]{liebeckseitz2004a}.) This is the case for sixteen of the sets of composition factors, leaving ten up to duality, five up to field automorphism:
\[ 6_1,3_1^7,\quad 9_{12}^*,3_1^3,3_2^3,\quad 9_{12},\bar 9_{12},6_1,3_1,\quad 15_1,6_1^2,\quad 21_{12},6_1.\]
The first two of these must be semisimple, and $u$ acts in both cases with blocks $3,2^8,1^8$, which belongs to the generic class $2A_1$. For the third case, only $6_1$ and $9_{12}$ have a non-split extension, so $M(E_6){\downarrow_H}$ is either semisimple or up to automorphism $(6_1/9_{12})\oplus \bar 9_{12}\oplus 3_1$, and $u$ acts in these two possibilities with blocks $3^3,2^6,1^6$ and $3^4,2^5,1^5$ respectively. The second does not appear in \cite[Table 5]{lawther1995}, and the first is class $3A_1$, which is generic. Thus $H$ is a blueprint for $M(E_6)\oplus M(E_6)^*$ in these cases by Lemma \ref{lem:genericmeansblueprint}.

For $15_1,6_1^2$, the corresponding factors on $L(E_6)^\circ$ are
\[ 27_1,7_1^5,(3_2,3_2^*)^2,1^3.\]
The $27_1$ has no extensions with the other factors so must break off, and the $\{7_1,3_2^\pm,1\}$-radical of $P(7_1)$ is
\[ 7_1,7_1/1,3_2,3_2^*/7_1.\]
Thus we need at least twice as many copies of $7_1$ as $1$s in order for $H$ not to stabilize a line on $L(E_6)^\circ$, which we do not have.

In the final case, let $x$ denote an element of order $80$ in $H$. There are 24 distinct eigenvalues for the action of $x$ on $M(E_6)$, and $6_1$ (and $21_{12}$) are sums of eigenspaces (i.e., the eigenvalues of $x$ on $6_1$ and $21_{12}$ are disjoint), hence any semisimple element of $\mb G$ powering to $x$ preserves $6_1$ and $21_{12}$. Thus $H$ is a blueprint for $M(E_6)$, placing $H$ inside a positive-dimensional subgroup $\mb X$ of $\mb G$ stabilizing the same subspaces of $M(E_6)$. Since no composition factor of $M(E_6){\downarrow_H}$ is self-dual, the same is true for $M(E_6){\downarrow_{\mb X}}$, whence $H$ and $\mb X$ stabilize the same subspaces of $M(E_6)\oplus M(E_6)^*$ as well. Thus $H$ is a blueprint for this module, as required.

\medskip

\noindent $\boldsymbol{q=4}$: The conspicuous sets of composition factors for $M(E_6){\downarrow_H}$ split into two groups: those with a trivial factor and those without. Those with a trivial factor are those of $M(F_4){\downarrow_H}$ with a trivial added, and of those, only $9_{12},9_{12}^*,8_1,1$ and $9_{12},9_{12}^*,8_2,1$ have non-negative pressure. Thus in the others $H$ stabilizes a line on $M(E_6)$. In the remaining cases, however, the $\cf(M(E_6){\downarrow_H})$-radicals of $P(9_{12})$ are
\[ 1/9_{12}\quad \text{and}\quad 9_{12}^*/1,8_2/9_{12},\]
and both of these have trivial quotients. Thus if $H$ has a trivial factor on $M(E_6)$, it stabilizes a line or hyperplane on it.

From now we assume that $H$ has no trivial composition factors on $M(E_6)$. There are 16 such sets of composition factors, four up to field and graph automorphism:
\[ 3_1^8,3_2^*,\quad 9_{12}^*,3_1^3,(3_2^*)^3,\quad 9_{12},\bar 9_{12},3_1^2,3_2^*,\quad 24_{21},3_1.\]
Since $H$ stabilizes a $3$-space in all cases $H$ always lies inside a member of $\ms X$ by \cite[Lemma 1.4]{craven2017}, but we need stability under $\sigma$ and also $N_{\Aut^+(\mb G)}(H)$. Switching to $L(E_6)$, the first three sets have corresponding factors
\[ 8_1^8,1^{14},\quad (\bar 9_{12},\bar 9_{12}^*)^3,8_1,8_2,1^8,\quad 24_{21},24_{21}^*,8_1^2,8_2,3_2,3_2^*,\]
and the fourth has no corresponding set. In the first two cases therefore, $H$ stabilizes lines on $L(E_6)$, so we consider the third set. The two $9$-dimensional modules in $M(E_6){\downarrow_H}$ split off as summands, so we consider the summand comprising the $3$-dimensionals. Since $\Ext_{kH}^1(3_1,3_1)=0$ and there is an extension between $3_1$ and $3_2^*$, we have a module $3_1/3_2^*/3_1$, so $M(E_6){\downarrow_H}$ has a summand one of the following three modules, up to duality:
\[ 3_1^{\oplus 2}\oplus 3_2^*,\quad 3_1\oplus (3_2^*/3_1),\quad 3_1/3_2^*/3_1.\]
The action of $v\in H$ of order $4$ on $M(E_6)$ is $4^4,3^3,1^2$ and $4^5,3,2,1^2$ and $4^6,1^3$ respectively. Only the last of these appears in \cite[Table 3]{lawther1995}, so $H$ acts like this on $M(E_6)$.

Now we use the fact that $H$ is not Lie primitive to prove that $H$ is strongly imprimitive. The only connected members of $\ms X$ acting with composition factors of dimensions compatible with $9^2,3^3$ are $A_2A_2A_2$, $A_2G_2$ and $G_2$. In the second and third cases, the $A_2A_2$ subgroup of $A_2G_2$, and the $A_2$ subgroup of $G_2$, also lie in $A_2A_2A_2$ by \cite[Table 11A]{thomas2017un}, so we assume that $H$ is contained in $\mb X=A_2A_2A_2$. This group acts on $M(E_6)$ as the sum of three $9$-dimensional modules, with tensor products $(10,01,00)$, $(00,10,01)$ and $(01,00,10)$. If $H$ maps into the first and second factors as $3_1$ and the third as $3_2$, then $H$ acts on $M(E_6)$ in the required way, and all other permissible embeddings are obtained from this by dualizing and permutations. Thus we see that $H$ and the corresponding $A_2$ subgroup acting as $(10,10,20)$ stabilize the same subspaces of $M(E_6)\oplus M(E_6)^*$, whence $H$ is a blueprint for $M(E_6)\oplus M(E_6)^*$, as needed.

\medskip

\noindent $\boldsymbol{q=8}$: There are eleven conspicuous sets of composition factors for $M(E_6){\downarrow_H}$, nine of which have the same composition factors as their duals, and hence are blueprints for $M(E_6)\oplus M(E_6)^*$ for the same reason as for $q=9$, using a semisimple element of order $21$. The remaining set of composition factors (up to automorphism) is
\[ 9_{12},9_{31},9_{23}^*.\]
If $L$ denotes a copy of $\PSL_2(8)$ in $H$, then $L$ centralizes an element of order $3$ whose eigenvalues on $M(E_6)$ have multiplicities $6$, $6$ and $15$, i.e., $L$ lies in the $A_5$-Levi subgroup. The action of $L$ on $M(E_6)$ is semisimple, with factors
\[ 4_{12},4_{13},4_{23},2_1^2,2_2^2,2_3^2,1^3.\]
Since $M(A_5)$ appears twice in $M(E_6){\downarrow_{A_5}}$, we see that $L$ acts as $2_1\oplus 2_2\oplus 2_3$ on $M(A_5)$, whence $L$ lies in the positive-dimensional $A_1$ subgroup $\mb X$ contained in $A_5$ acting as $L(1)\oplus L(2)\oplus L(4)$ on $M(A_5)$. The action of $\mb X$ on $\Lambda^2(M(A_5))$ is easy to compute, and is
\[ L(0)^{\oplus 3}\oplus L(3)\oplus L(5)\oplus L(6),\]
whence $L$, and therefore $H$, are blueprints for $M(E_6)\oplus M(E_6)^*$.

\medskip

If we have $\bar H$ instead, then there are four conspicuous sets of composition factors for elements of order at most $19$, up to field automorphism:
\[ 3_1^8,3_2^*,\quad 9_{13},3_1^3,3_3^3,\quad 24_{31},3_1,\quad 24_{32},3_1.\]
The third and fourth cases require that an involution in $\bar H$ acts with blocks $2^{13},1$ on $M(E_6)$, which does not appear in \cite[Table 5]{lawther1995}. The first case has composition factors $8_1^8,1^{14}$ on $L(E_6)$, of pressure $-14$, so $\bar H$ stabilizes a line on $L(E_6)$ by Proposition \ref{prop:pressure}. The second case has composition factors $(9_{31},9_{31}^*)^3,8_1,8_3,1^8$, which has pressure $-2$, so $\bar H$ stabilizes a line on $L(E_6)$ in this case as well.
\end{proof}

As with $\PSL_3(3)$, if $H$ acts irreducibly on $M(E_6)$ then its composition factors on $L(E_6)^\circ$ are 
\[ 15,15^*,7^3,6,6^*,(3,3^*)^2,1^2.\]
Again, the actual $H$ acting irreducibly on $M(E_6)$ cannot stabilize a line on $L(E_6)^\circ$. This case is dealt with in Section \ref{sec:irredpsu33}, where as with $\PSL_3(3)$, we prove that all such $H$ lie in copies of $G_2(3)$ in $E_6$.

\begin{proposition}\label{prop:sp4ine6} If $H\cong \PSp_4(q)'$ for $q=2,3,4$, or $H\cong {}^2\!B_2(q)$ for $q=8$, then $H$ stabilizes a line or hyperplane on $M(E_6)$. If $\bar H\cong 3\cdot \PSp_4(2)'$ then $\bar H$ stabilizes a line on $L(E_6)$.
\end{proposition}
\begin{proof} $\boldsymbol{q}$\textbf{ even}: The proof is exactly the same as for Proposition \ref{prop:sp4ine7}, but this time the trace of a rational element of order $5$ is $2$ on $M(E_6)$. If $H\cong \Alt(6)$ then by \cite[Proposition 6.3]{craven2017} we see that $H$ stabilizes a line or hyperplane on $M(E_6)$.

(Note that the proof of the result in \cite{craven2017} does yield this statement, but the claim in \cite[Proposition 6.3]{craven2017} is erroneously that $H$ stabilizes a line on $M(E_6)$, rather than a line \emph{or hyperplane}. Of course, in uses of the result this distinction is usually irrelevant. Moreover, with slightly more work one obtains the original result: if $H$ stabilizes a hyperplane but not a line then $H$ lies inside a $D_5$-parabolic subgroup by Lemma \ref{lem:e6linestabs}, which acts on $M(E_6)$ as $1/16/10$. Thus $H$ cannot stabilize a line on the self-dual module $10=M(D_5)$. But this is clearly impossible, as there is no positive-pressure $10$-dimensional module for $H$.)

In the case of $\bar H$, in each faithful block there are three simple modules, $3_1$, $3_2$ and $9$, labelled so that $3_i\otimes 3_i^*=1\oplus 8_i$, where $8_i$ is one of the two factors into which the Steinberg module splits on restriction to $\Sp_4(2)'$. As in \cite{craven2017} we find up to automorphism four conspicuous sets of composition factors for $M(E_6){\downarrow_{\bar H}}$, namely
\[ 3_1^8,3_2,\quad 3_1^6,3_2^3,\quad 9,3_1^3,3_2^3,\quad 9^2,3_1^2,3_2.\]
Notice that each of these stabilizes a $3$-space on $M(E_6)$, hence lies in a member of $\ms X$ by \cite[Lemma 1.4]{craven2017}, but is not necessarily strongly imprimitive. We go through the members of $\ms X$, and prove that the second and fourth cases do not occur.

Suppose that $\bar H\leq \mb X$ for some $\mb X\in \ms X$. Since $\bar H$ is a non-split extension by the centre, so must $\mb X$ be, so $\mb X$ must be either the $A_5$-parabolic ($\bar H$ cannot map into $A_1$, and $H$ cannot embed in $G_2$, so $A_5A_1$ and $A_2G_2$ need not be considered) or $A_2A_2A_2$. In order to embed $\bar H$ into $A_2A_2A_2$ in the correct manner, the action of $Z(\bar H)$ on each of the three modules $M(A_2)$ must be different. In particular, it must act trivially on one of them, so $\bar H\leq A_2A_2\leq A_5$, and it suffices to check this subgroup. The action of $\bar H$ on $M(A_5)$ has factors either $3_1^2$ or $3_1,3_2$. In the first possibility $M(E_6){\downarrow_{\bar H}}$ has composition factors $3_1^8,3_2$, and in the second possibility $M(E_6){\downarrow_{\bar H}}$ has composition factors $9,3_1^3,3_2^3$, so these are the only two cases that may occur.

In the first case there is a unique possible set of composition factors for $L(E_6){\downarrow_{\bar H}}$, namely $8_1^8,1^{14}$, and clearly $\bar H$ stabilizes a line on $L(E_6)$ the module has pressure $-14$. In the third case the corresponding Brauer character for $L(E_6)$ is not uniquely defined, so we use the fact that $H$ embeds in $A_5$. This has a unique corresponding set of composition factors on $L(E_6)$, and they are $8_1,8_2,4_1^6,4_2^6,1^{14}$. This set of factors has pressure $-2$ so $\bar H$ stabilizes a line on $L(E_6)$, as claimed.

\medskip

\noindent $\boldsymbol{q=3}$: From the composition factors for $M(E_7){\downarrow_H}$, we see that the only conspicuous sets of composition factors for $M(E_6){\downarrow_H}$ are $10,5^3,1^2$ and $25,1^2$. The latter case must be semisimple, but an element from the largest two classes of elements of order $3$ acts on $25$ with blocks $3^8,1$, hence on $M(E_6)$ with blocks $3^8,1^3$. This action does not appear in \cite[Table 5]{lawther1995}, hence $H$ does not embed with these factors.

The first case has pressure $1$, so if $H$ does not stabilize a line or hyperplane on $M(E_6)$, then ignoring the $10$ the module structure would have to be $5/1/5/1/5$. However, the $\{1,5,10\}$-radical of $P(5)$ has the form
\[ 5,5/1,5,10/5.\]
Thus $H$ stabilizes a line or hyperplane on $M(E_6)$, as needed.
\end{proof}

\begin{proposition}\label{prop:g2ine6} Let $H\cong G_2(q)'$ for $q=2$ or ${}^2\!G_2(q)'$ for $q=3$. 
\begin{enumerate}
\item If $q=2$ then $H$ stabilizes a line or hyperplane on $M(E_6)$ or $L(E_6)$, or $H$ stabilizes a unique $14$-space in $L(E_6)$ and its stabilizer is a maximal $G_2$, and $N_{\mb G}(H)=G_2(2)$ is a blueprint for $L(E_6)$.
\item If $q=3$ then either $H$ stabilizes a line on $M(E_6)$, or $H$ acts on $M(E_6)$ and $L(E_6)^\circ$ as
\[ 9_1\oplus 9_2\oplus 9_3,\quad\text{and}\quad 9_1\oplus 9_2\oplus 9_3\oplus P(7)\oplus (7/7)\]
respectively.
\end{enumerate}
\end{proposition}
\begin{proof} $\boldsymbol{q=2}$: There are two conspicuous sets of composition factors for $M(E_6){\downarrow_H}$, namely $6^3,1^9$ and $14,6^2,1$. In the first case $H$ must stabilize a line on $M(E_6)$ by Proposition \ref{prop:pressure} (since it has pressure $-6$), so we focus on the second case. Let $v$ have order $8$ in $H$. Suppose that $H$ does not stabilize a line or hyperplane on $M(E_6)$. If $14$ is a summand then the structure must be $(6/1/6)\oplus 14$, and this module is uniquely determined. The element $v$ acts on this module with blocks $8^2,6,4,1$, which does not appear in \cite[Table 5]{lawther1995}, hence $14$ is not a summand of $M(E_6){\downarrow_H}$. In particular, this means that up to duality we may assume that the socle of $M(E_6){\downarrow_H}$ is $6$.

Either there is a submodule $14/6$ or there is not. If there is, we have a submodule $1,14/6$, on which $6$ must be placed. If there is not then we have a submodule $6/1/6$, on which $14$ must be placed. In each case we can place two copies on top, and we thus obtain two modules,
\[ 6,6/1,14/6,\quad 14/6/1,14/6,\]
one of which must be a pyx for $M(E_6){\downarrow_H}$. The intersection of these two submodules of $P(6)$ is a module $6/1,14/6$ with a $14$ quotient, on which an element $v\in H$ of order $8$ acts with blocks $8^2,6,4,1$. This does not appear in \cite[Table 5]{lawther1995}, and thus this module is not $M(E_6){\downarrow_H}$. (But we have not yet excluded any of the other modules.)

Let $L$ denote a copy of $\PSL_3(2)$ inside $H$. The restrictions of $6$ and $14$ to $L$ are $3\oplus 3^*$ and $3\oplus 3^*\oplus 8$ respectively, so ignoring simple projectives they have the same restriction. The restriction of $6/1/6$ to $L$ has structure
\[ 1\oplus (3^*/3)\oplus (3/3^*),\]
and the restriction of $14/6/1,14/6$ to $L$ has structure
\[ 1\oplus (3/3^*/3)\oplus (3^*/3/3^*)\oplus 3\oplus 3^*\oplus 8^{\oplus 2}.\]
This means that there is a unique extension with quotient $14$ and submodule $6/1/6$ that splits on restriction to $L$, and it turns out that this is the module $6/1,14/6$ eliminated above. Thus whichever of the $kH$-modules $14/6/1/6$ we have (such a module is not unique up to isomorphism), its restriction to $L$ is always 
\begin{equation} 1\oplus (3/3^*/3)\oplus (3^*/3/3^*)\oplus 8.\label{eq:me6restoL}\end{equation}
Up to changing the number of $8$s involved, the exact same statements hold for $1,14/6$ and $6,6/1,14/6$. Thus if $H$ does not stabilize a line or hyperplane on $M(E_6)$ then we may assume that $M(E_6){\downarrow_L}$ has the structure given in (\ref{eq:me6restoL}). 

Either by Proposition \ref{prop:sl3inf4} below, or by computing the composition factors on $L(F_4)$, we see that the image of $L$ under the graph automorphism of $F_4$ stabilizes a line on $M(F_4)$, and in fact acts with composition factors $8^3,1^2$. Certainly the image under the graph automorphism cannot lie in a maximal parabolic subgroup of $F_4$, hence neither can $L$. Thus $L$ lies in one of $B_4$ (which stabilizes a line on $M(F_4)$), $C_4$ (and then, as in the proof of Proposition \ref{prop:sl3ine7}, we find that $L$ lies in an irreducible $A_2$ subgroup lying in $A_2A_2$ from \cite[Table 10/10A]{thomas2017un}) or $A_2A_2$.

Thus $L$ is contained in $A_2A_2$, and obviously a diagonal $A_2$ subgroup of $A_2A_2$. It must be embedded as either $(10,10)$ or $(01,01)$, but one of those acts on $M(F_4)$ with factors $8^3,1^2$ and the other with factors $8,(3,3^*)^3$, so it is of course the latter. The action is then
\[ L(11)\oplus (L(10)/L(02)/L(10))\oplus (L(01)/L(20)/L(01)),\]
and note that we must have that the stabilizer of the $kL$-submodule $3\oplus 3^*$ contains this diagonal $A_2$ subgroup $\mb X$. In particular, this means that $H$ is contained in a positive-dimensional subgroup of $E_6$, namely $\langle H,\mb X\rangle$.

\medskip

This proves that $H$ is not Lie primitive, but since $M(E_6)$ is not graph-stable, we need to work a bit harder to prove strong imprimitivity. It is easier to switch to $L(E_6)$ at this point. Let $\mb Y$ denote a maximal, positive-dimensional subgroup of $E_6$ containing $H$. The action of $H$ on $L(E_6)$ is one of two sets of composition factors:
\[ 32_1,32_2,14,\qquad 14^2,6^7,1^8.\]
In the second case $H$ has pressure $-1$, so $H$ stabilizes a line on $L(E_6)$, so we may assume that $H$ acts as in the first case.

Note that $\mb Y$ cannot act irreducibly on $L(E_6)$, and $H$ acts with composition factors of dimensions $32,32,14$, so if $\mb Y$ acts with three composition factors then they need to be of those dimensions. As $H\not\leq F_4$, we must have that $\mb Y$ is a maximal $G_2$, which acts with dimensions $64,14$, and $N_{\mb Y}(H)=G_2(2)$ is a blueprint for $L(E_6)$. Also, $H$ is contained in an $N_{\Aut^+(\mb G)}(H)$-stable, positive-dimensional subgroup of $\mb G$, namely the stabilizer $\mb Y$ of the $14$-space in $L(E_6)$.

\medskip

\noindent $\boldsymbol{q=3}$: The conspicuous sets of composition factors for $M(E_6){\downarrow_H}$ are
\[ 7^3,1^6,\quad 9_1,9_2,9_3,\quad 9_1,9_2,7,1^2.\]
The first and third sets of composition factors of $H$ on $M(E_6)$ have negative pressure, hence $H$ stabilizes a line on $M(E_6)$ by Proposition \ref{prop:pressure}. For the second set, the action of $H$ on $L(E_6)^\circ$ has composition factors
\[ 9_1,9_2,9_3,7^7,1,\]
and since an element $v\in H$ of order $9$ acts on $M(E_6)$ as $9^3$, it comes from class $E_6(a_1)$, so acts on $L(E_6)$ with blocks $9^8,6$ (see \cite[Tables 5 and 6]{lawther1995}). Hence the action on $L(E_6)^\circ$ is
\[ 9_1\oplus 9_2\oplus 9_3\oplus P(7)\oplus (7/7),\]
as claimed in the proposition.
\end{proof}

The case of ${}^2\!G_2(3)'$ acting as $9_1\oplus 9_2\oplus 9_3$ on $M(E_6)$  was proved to lie in the maximal $G_2$ subgroup in \cite[Theorem 29.3]{aschbacherE6Vun}, using the trilinear form on $M(E_6)$. It appears that pure representation-theoretic methods are not powerful enough to attack this. We will give another proof of this, still using the trilinear form, in Section \ref{sec:irred2g23}.

\begin{proposition}If $H\cong \PSU_4(q)$ for $q=2,4$, then $H$ stabilizes a line on $M(E_6)$.
\end{proposition}
\begin{proof} $\boldsymbol{q=2}$: As we would guess from the conspicuous sets of factors for $M(E_7){\downarrow_H}$, there are two conspicuous sets of factors for $M(E_6){\downarrow_H}$:
\[ 14,6^2,1,\quad 6,(4,4^*)^2,1^5.\]
The first set of factors has pressure $0$, so $H$ stabilizes a line or hyperplane on $M(E_6)$ by Proposition \ref{prop:pressure}. However, the factors do not fit with the factors of the $D_5$-parabolic subgroup in Lemma \ref{lem:e6linestabs}, so $H\leq F_4$ and so stabilizes a line on $M(E_6)$. The second has pressure $-1$, so stabilizes a line on $M(E_6)$.

\medskip

\noindent $\boldsymbol{q=4}$: Again, there are two conspicuous sets of composition factors for $M(E_6){\downarrow_H}$ up to field automorphism:
\[ 14_1,6_1^2,1,\quad 6_1,(4_1,4_1^*)^2,1^5.\]
The same pressure statements apply this time, except now $\Ext_{kH}^1(1,4_1^\pm)$ is zero, so the trivial factors even split off as summands in both cases.
\end{proof}

\begin{proposition} If $H\cong \PSp_6(2)$ is a subgroup of $E_6$, then $H$ stabilizes a line on $M(E_6)$ or $L(E_6)$.
\end{proposition}
\begin{proof} There are two conspicuous sets of composition factors for $M(E_6){\downarrow_H}$, namely $8^2,6,1^5$, of pressure $-4$, hence $H$ stabilizes a line on $M(E_6)$ by Proposition \ref{prop:pressure}, and $14,6^2,1$, which has pressure $1$.

For the second we switch to the Lie algebra $L(E_6)$, where the corresponding set of composition factors is $14^2,8^2,6^5,1^4$. From Proposition \ref{prop:g2ine6} (in particular its proof) above we see that in this case the subgroup $L\cong G_2(2)$ of $H$ stabilizes a line on $L(E_6)$, whence there is a (non-zero) homomorphism from the permutation module $P_L$ to $L(E_6){\downarrow_H}$. Assume that $H$ does not stabilize a line on $L(E_6)$.

Since there is no module $6/1/6$, and $L(E_6){\downarrow_H}$ has pressure $1$, we see that there are two possible structures for the module:
\[ 6/1,14/6/1,8/6/1,8/6/1,14/6,\quad 6/1,8/6/1,14/6/1,14/6/1,8/6.\]
The module $P_L$ has structure
\[ 1,8/6/14/6/1/48/1/6/14/6/1,8,\]
and so we see that the second structure must be the correct one, and the quotient $W=1,8/6$ is a submodule of $L(E_6){\downarrow_H}$. But now $\Ext_{kH}^1(6,W)$ is $1$-dimensional, but the module $6/1,8/6$ has a trivial quotient, which is a contradiction.

This completes the proof.\end{proof}

\newpage

\chapter{Subgroups of \texorpdfstring{$F_4$}{F4}}
\label{chap:subsinf4}

Now we let $\mb G$ be the algebraic group of type $F_4$ and $H$ be a subgroup of rank $2$, $q\leq 9$, together with a small Ree group, a Suzuki group, $\PSL_3(16)$ or $\PSU_3(16)$. We can exclude all groups where $p$ is odd and where $H$ contains a semisimple element of order at least $19$, and where $p=2$ and $H$ contains a semisimple element of order at least $57$ by Theorem \ref{thm:largeorderss}. If $p$ is odd then any subgroup that is a blueprint for $M(E_6)$ is a blueprint for $M(F_4)$, so those cases eliminated in the previous chapter stay eliminated here. The remaining ones are as follows.
\begin{enumerate}
\item $\PSL_3(q)$ for $q=2,3,4$;
\item $\PSU_3(q)$ for $q=3,4,8$;
\item $\PSp_4(q)$ for $q=3,4$;
\item $G_2(q)'$ for $q=2,4$;
\item ${}^2\!B_2(q)$ for $q=8,32$;
\item ${}^2\!G_2(3)'$.
\end{enumerate}

In characteristic $2$ we might prove that either $H$ or its image under the graph automorphism stabilizes a line on $M(F_4)$, and then use Lemma \ref{lem:fix1space}. Note that we can also apply Lemma \ref{lem:F4ignoregraph}, so if we can show that $H$ is a blueprint for $M(F_4)$ and the composition factors of $H$ on $M(F_4)$ and the quotient $L(F_4)/M(F_4)$ have different dimensions, then we are done.

As with all of the previous chapters, here $u$ denotes an element of order $p$ in $H$ belonging to the smallest conjugacy class.

\begin{proposition}\label{prop:sl3inf4} Let $H\cong \PSL_3(q)$ for $2\leq q\leq 4$.
\begin{enumerate}
\item If $q=2,4$ then $H$ stabilizes a line or hyperplane on $L(F_4)$.
\item If $q=3$ then $H$ either stabilizes a line on $M(F_4)^\circ$ or is a blueprint for $M(F_4)^\circ$.
\end{enumerate}
\end{proposition}
\begin{proof} $\boldsymbol{q=2}$: There are three conspicuous sets of composition factors for $M(F_4){\downarrow_H}$, namely
\[ (3,3^*)^3,1^8,\quad 8^3,1^2,\quad 8,(3,3^*)^3.\]
The pressures of the first two are $-2$, so $H$ stabilizes a line on $M(F_4)$ by Proposition \ref{prop:pressure}. An element of order $7$ in $H$ acts with trace $5$ in the first two cases and $-2$ in the third. The graph automorphism swaps the two rational classes of elements of order $7$, and so if $H$ acts as in the third case on $M(F_4)$, the image of $H$ under the graph automorphism acts as in the first or second cases on $M(F_4)$. Thus either $H$ or its image under the graph automorphism stabilizes a line on $M(F_4)$, and $H$ stabilizes a line or hyperplane on $L(F_4)$ by Lemma \ref{lem:fix1space}.

\medskip

\noindent $\boldsymbol{q=4}$: There are, up to field automorphism, only two conspicuous sets of composition factors: $8_1^3,1^2$ and $9,9^*,8_1$. Again the first case has pressure $-2$, and so $H$ stabilizes a line on $M(F_4)$, and again (this time up to field automorphism) the graph automorphism swaps the two classes. (This is because the graph automorphism does not fix any classes for $q=2$.) As for $q=2$ we find that $H$ stabilizes a line or hyperplane on $L(F_4)$, as needed.

\medskip

\noindent $\boldsymbol{q=3}$: There are four conspicuous sets of composition factors for the action of $H$ on the $25$-dimensional module $M(F_4)^\circ$, given by
\[ (3,3^*)^3,1^7,\quad 7^3,1^4,\quad 7,6,6^*,3,3^*,\quad 7,(3,3^*)^3.\]

\noindent \textbf{Cases 1 and 2}: These have pressures $-7$ and $-1$ respectively, so that $H$ stabilizes a line on $M(F_4)^\circ$ by Proposition \ref{prop:pressure}.

\medskip

\noindent \textbf{Case 3}: The only classes of elements of order $3$ that are non-generic for $M(F_4)^\circ$ are $A_2+\tilde A_1$, $\tilde A_2$ and $\tilde A_2+A_1$, which on $M(F_4)^\circ$ act as $3^7,2^2$ or $3^8,1$ by Table \ref{t:unipotentF4}. Since $u$ acts on $3$ as $2,1$, on $6$ as $3,2,1$, and on $7$ as $3,2^2$, we see that if $H$ is not a blueprint for $M(F_4)^\circ$, then there are no simple summands of $M(F_4)^\circ{\downarrow_H}$. Thus the socle is, up to duality, one of: $3$; $6$; $3\oplus 6$; and $3\oplus 6^*$ (note that $\Ext_{kH}^1(3,6)=0$ but $\Ext_{kH}^1(3,6^*)\cong k$). In the first case we consider the $\{3^*,6^\pm,7\}$-radical of $P(3)$, which is
\[ 6/3^*/6^*,7/3,\]
and it is clearly not self-dual. The dual of this is the $\{3^\pm,6,7\}$-radical of $P(6^*)$, so we cannot have a simple socle. If the socle of $M(F_4)^\circ{\downarrow_H}$ is $3\oplus 6$, we take the $\{3^*,6^*,7\}$-radical of $P(3)\oplus P(6)$, then take the $\{7\}$-residual, to obtain
\[ (3^*/6^*,7/3)\oplus (3^*,6^*/6).\]
This has one too many copies of $6^*$, and $u$ acts on it with blocks $3^7,2^4,1^5$, so we cannot get at least seven blocks of size $3$ in the action of $u$ on $M(F_4)^\circ{\downarrow_H}$ this way. Similarly, we do the same with socle $3\oplus 6^*$ to obtain the module
\[ 3^*/6,7/3,6^*,\]
and $u$ acts on this with blocks $3^6,1^7$. This contradiction proves that $u$ is generic on $M(F_4)^\circ$.

\medskip

\noindent \textbf{Case 4}: The $\{3^\pm,7\}$-radical of $P(3)$ is $3^*/7/3$, so we see that at least four $3$-dimensional factors must split off as summands in $M(F_4)^\circ{\downarrow_H}$; since $u$ acts on $3$ with blocks $2,1$, this means that we have blocks at least $2^4,1^4$ in the action of $u$ on $M(F_4)^\circ$, and this is enough to assure that $u$ is generic from \cite[Table 3]{lawther1995}.
\end{proof}

\begin{proposition}\label{prop:su3inf4} Let $H\cong \PSU_3(q)$ for $q=3,4,8$.
\begin{enumerate}
\item If $q=3$ then $H$ stabilizes a line on either $M(F_4)^\circ$ or $L(F_4)$.
\item If $q=4$ then $H$ stabilizes a line or hyperplane on $L(F_4)$ or either $H$ or its image under a graph automorphism is a blueprint for $M(F_4)$ and is not graph-stable.
\item If $q=8$ then either $H$ stabilizes a line or hyperplane on $L(F_4)$ or the stabilizer of all $H$-invariant $8$-spaces on the composition factors of $L(F_4)$ is positive dimensional, and hence $H$ is strongly imprimitive.
\end{enumerate}
\end{proposition}
\begin{proof} $\boldsymbol{q=3}$: There are four conspicuous sets of composition factors for $M(F_4)^\circ{\downarrow_H}$:
\[ (3,3^*)^3,1^7,\quad 7^3,1^4,\quad 7,(3,3^*)^3,\quad 7,6,6^*,3,3^*.\]
The first two have negative pressure, so $H$ stabilizes a line on $M(F_4)^\circ$ by Proposition \ref{prop:pressure}. For the third and fourth cases we switch to $L(F_4)$, and see that $H$ acts with composition factors $7,(6,6^*)^3,1^9$ and $15,15^*,7^2,3,3^*,1^2$ respectively, which have pressures $-2$ and $0$ respectively. Thus $H$ stabilizes a line on $L(F_4)$ in both cases.

\medskip

\noindent $\boldsymbol{q=4}$: There are, up to field automorphism, six conspicuous sets of composition factors for $M(F_4){\downarrow_H}$:
\[ (3_1,3_1^*)^3,1^8,\quad 8_1^3,1^2,\quad 8_1,(3_1,3_1^*)^3,\quad 8_1,(3_1,3_1^*)^2,3_2,3_2^*,\]
\[ 9_{12},9_{12}^*,8_1,\quad \bar 9_{12},\bar 9_{12}^*,8_1.\]

The graph automorphism swaps (up to field automorphism) the first and third cases, the second and fifth, and the fourth and sixth sets of factors. As there are no extensions between the factors, the first two cases are semisimple and so $H$ stabilizes a line or hyperplane on $M(F_4)$, hence either $H$ stabilizes a line on $L(F_4)$ or we may assume that we have either the fourth or sixth case.

In the sixth case, let $x$ denote an element of order $15$ in $H$. There is no element of order $45$ in $\mb G$ cubing to $x$ and having the same eigenspaces on $M(F_4)$ as $x$, but there is an element of order $75=5\times 15$ having the same eigenvalues as $x$, hence $x$ and $H$ are blueprints for $M(F_4)$. Furthermore, we apply Lemma \ref{lem:F4ignoregraph} to see that this is enough to force $\sigma$-stability and $N_{\Aut^+(\mb G)}(H)$-stability, as needed.

\medskip

\noindent $\boldsymbol{q=8}$: Up to field automorphism there are three conspicuous sets of composition factors for $M(F_4){\downarrow_H}$, namely
\[ 8_1^3,1^2,\quad 9_{12},9_{12}^*,8_1,\quad 9_{12},9_{12}^*,8_2.\]
Up to field automorphism, the first and second are swapped by the graph automorphism, so as the first case stabilizes a line on $M(F_4)$, in both cases $H$ stabilizes a line or hyperplane on $L(F_4)$ by Lemma \ref{lem:fix1space}. Thus we consider the third case, which is left invariant (up to field automorphism) by the graph. Since $H$ contains elements of orders $19$ and $21$, $H$ is a blueprint for $M(F_4)$ by Theorem \ref{thm:largeorderss}. The only positive-dimensional subgroup with compatible composition factors for both $H$ and its image under the graph is $\mb X=A_2A_2$, and of course this acts semisimply with composition factors of dimension $8$, $9$ and $9$ on $M(F_4)$. Thus $\mb X$ stabilizes one $8$-dimensional submodule on each composition factor of $L(F_4)$, as does $H$. This completes the proof that $H$ is strongly imprimitive, using Theorem \ref{thm:intersectionorbit}.
%
%
\end{proof}

\begin{proposition}\label{prop:sp4inf4} Let $H\cong \PSp_4(q)'$ for $q=2,3,4$, or ${}^2\!B_2(q)$ for $q=8,32$.
\begin{enumerate}
\item If $q=2$ then $H$ stabilizes a line or hyperplane on $L(F_4)$.
\item If $q=4,8,32$ then $H$ stabilizes a line on $M(F_4)$.
\item If $q=3$ then $H$ does not embed in $\mb G$.
\end{enumerate}
\end{proposition}
\begin{proof} $\boldsymbol{q}$\textbf{ even}: The proof is exactly the same as for Proposition \ref{prop:sp4ine7}, but this time the trace of a rational element of order $5$ is $1$ on $M(F_4)$.

If $q=2$ then we prove that $H$ or its image under the graph automorphism stabilizes a line on $M(F_4)$ in \cite[Proposition 6.3]{craven2017}, hence $H$ stabilizes a line or hyperplane on $L(F_4)$ via Lemma \ref{lem:fix1space}.

\medskip

\noindent $\boldsymbol{q=3}$: The composition factors of $M(F_4)^\circ{\downarrow_H}$ must be $10,5^3$ or $25$, from the composition factors for $M(E_6){\downarrow_H}$ in Proposition \ref{prop:sp4ine6}, but we already proved there that the second case cannot occur.

Note that there are no positive-dimensional subgroups of $\mb G$ whose composition factors are all multiples of $5$, so $H$ does not embed in $\mb G$ if we prove that $H$ cannot be Lie primitive. If $M(F_4)^\circ{\downarrow_H}$ is semisimple, then $u$ acts on $M(F_4)^\circ$ with blocks $3,2^8,1^6$, so lies in the generic class $\tilde A_1$. Thus it is not semisimple. If $10$ is a submodule then it is a summand, and the $\{5\}$-radical of $P(5)$ is $5/5$, so we must have $10\oplus (5/5)\oplus 5$. Then $u$ acts on $M(F_4)^\circ$ with blocks $3^3,2^5,1^6$, which is not a valid unipotent class. Thus $10$ is not a submodule or quotient of $M(F_4)^\circ{\downarrow_H}$, and we consider the $\{5,10\}$-radical of $P(5)$. This is a self-dual module $5/5,10/5$, so this is $M(F_4)^\circ{\downarrow_H}$. In this case $u$ acts with Jordan blocks $3^6,1^7$, so lies in the generic class $A_2$. (In this case a different unipotent element actually has an invalid action, but we do not need this.) This completes the proof that $H$ does not embed in $\mb G$.
\end{proof}

\begin{proposition} Let $H\cong G_2(q)'$ for $q=2,4$, or ${}^2\!G_2(q)'$ for $q=3$.
\begin{enumerate}
\item If $q=2,4$ then $H$ stabilizes a line or hyperplane on $L(F_4)$.
\item If $q=3$ then $H$ stabilizes a line on $M(F_4)^\circ$ or an $N_{\Aut^+(\mb G)}(H)$-orbit of $H$-invariant $9$-spaces on $M(F_4)^\circ$ that has a positive-dimensional stabilizer, and hence is strongly imprimitive.
\end{enumerate}
\end{proposition}
\begin{proof} $\boldsymbol{q=2}$: The conspicuous sets of composition factors for $M(F_4){\downarrow_H}$ are $6^3,1^8$ and $14,6^2$, and these are swapped by the graph automorphism. Since $H$ clearly stabilizes a line on $M(F_4)$ in the first case, $H$ stabilizes a line or hyperplane on $L(F_4)$ in both cases by Lemma \ref{lem:fix1space}.

\medskip

\noindent $\boldsymbol{q=4}$: Let $L=G_2(2)$ be contained in $H$. We see that, up to graph automorphism, we may assume that the composition factors of $M(F_4){\downarrow_L}$ are $6^3,1^8$. Since the only modules of dimension at most $26$ for $H$ have dimensions $1$, $6$ and $14$, we see that $M(F_4){\downarrow_H}$ has composition factors with dimensions $6^3,1^8$, and again $H$ stabilizes a line on $M(F_4)$ up to graph automorphism, i.e., $H$ stabilizes a line on $L(F_4)$.

\medskip

\noindent $\boldsymbol{q=3}$: Here the conspicuous sets of composition factors for $M(F_4)^\circ{\downarrow_H}$ are $7^3,1^4$ and $9_1,9_2,7$ (up to field automorphism). In the first case $H$ stabilizes a line on $M(F_4)^\circ$ as it has pressure $-1$, by Proposition \ref{prop:pressure}, so we assume we are in the second case. We see that $M(F_4)^\circ{\downarrow_H}$ must be semisimple. If $v$ denotes an element of order $9$ in $H$, then $v$ acts on $M(F_4)^\circ$ with blocks $9^2,7$, so lies in class $F_4(a_1)$, which is class $D_5$ in $E_6$. This class acts on $M(E_6)$ with blocks $9^2,7,1^2$, and so $H$ must embed in $E_6$ acting on $M(E_6)$ as
\[9_1\oplus 9_2\oplus 7\oplus 1^{\oplus 2}.\]
Thus $H$ is contained in the intersection of two line stabilizers in $E_6$, one of which is $\mb G$. The intersection is $\mb G\cap\mb X$ for $\mb X$ equal to $F_4$ or a $D_5$-parabolic subgroup, of dimensions $52$ and $62$ respectively, and in either case $\mb G\cap \mb X$ is positive dimensional and contains $H$, so $H$ is contained in a member of $\ms X$, the maximal positive-dimensional subgroups of $F_4$.

Since $H$ lies inside a positive-dimensional subgroup of $\mb G$, and acts on $M(F_4)^\circ$ with composition factors of dimensions $9,9,7$, we see that $H$ can only lie in $B_4$, acting irreducibly on $M(B_4)$. (Such an embedding does exist.) Since the composition factors of $M(F_4)^\circ{\downarrow_H}$ are not stable, each of the two $9$-spaces forms an $N_{\Aut^+(\mb G)}(H)$-orbit of subspaces by Lemma \ref{lem:semilinearfield}. One of these is stabilized by $B_4$, so we may apply Theorem \ref{thm:intersectionorbit} to see that $H$ is strongly imprimitive.
\end{proof}

\newpage

\chapter{Difficult cases}
\label{chap:diff}

This chapter considers three cases that have been postponed from earlier chapters: $\PSL_5(2)$, $\PSL_3(5)$ and $\PSU_3(4)$, all in $E_8$.

\section{\texorpdfstring{$\PSL_5(2)\leq E_8$}{PSL(5,2) < E8}}
\label{sec:difl52}
This proof is probably the hardest in this article, and so deserves this special section all on its own. Because of the many possible paths through to a positive result, there may well be a much shorter and conceptually easier proof, but I have not found it.

\medskip

Let $H\cong \PSL_5(2)$. Proposition \ref{prop:sl5ine8} showed that either $H$ stabilizes a line on $L(E_8)$, or the composition factors of $H$ on $L(E_8)$ are
\[ 40_1,40_1^*,40_2,40_2^*,24^2,10,10^*,(5,5^*)^2.\]
We will prove that $H$ is always strongly imprimitive. We start with some preliminary information, and then proceed to the proof.

The non-split extensions between these simple $kH$-modules are, up to duality and graph automorphism,
\[ (5,10^*),\quad (5,40_1),\quad (5,40_2^*),\quad (10,10^*),\quad (10,40_1),\quad (24,40_1),\quad (24,40_2),\]
and in each case $\Ext^1$ is $1$-dimensional.

Write $L$ for a $\PSL_4(2)$ subgroup of $H$, and write $M$ for the subgroup $3\times \PSL_3(2)$ of $H$. Let $z$ denote a central element of order $3$ in $M$, let $\theta$ denote a primitive cube root of unity. Let $u$, $v$ and $w$ denote unipotent elements of $H$ acting on the natural module with blocks $3,1^2$ and $4,1$ and $5$ respectively. \textbf{Note that this differs from our previous convention for $u$.} We may conjugate $u$ to lie in $M$, so that we may split the blocks of $u$ among the eigenspaces of $z$, i.e., consider the action of $uz$ on $L(E_8)$. Since $z$ has trace $14$ on $L(E_8)$, it lies in the class whose centralizer in $D_7T_1$. By computing traces, we determine that the composition factors of $M(D_7){\downarrow_M}$ are $(3,3^*)^2,1^2$, and any such module has a trivial submodule.

We will write the $1$- and $\theta$-eigenspace of a module $V$ to mean the $1$- and $\theta$-eigenspace of the action of $z$ on $V$. If no module is mentioned we mean the eigenspace of $L(E_8)$, but we will usually state it.

The restrictions to $L$ and $M$ of the simple $kH$-modules that we consider are as follows:

\begin{center}\begin{tabular}{cccc}
\hline Module & Restriction to $L$ & $1$-eigenspace & $\theta$-eigenspace
\\ \hline $5$ & $4\oplus 1$ & $3$ & $1$
\\ $5^*$ & $4^*\oplus 1$ & $3^*$ & $1$
\\ $10$ & $4^*\oplus 6$ & $1\oplus 3$ & $3^*$
\\ $10^*$ & $4\oplus 6$ & $1\oplus 3^*$ & $3$ 
\\ $24$ & $(1/14/1)\oplus 4\oplus 4^*$ & $1^{\oplus 2}\oplus 8$ & $1\oplus 3\oplus 3^*$
\\ $40_1$ & $6\oplus 14\oplus 20^*$ & $(3/3^*/3)\oplus 3^*$ & $3\oplus 3^*\oplus 8$
\\ $40_1^*$ & $6\oplus 14\oplus 20$ & $(3^*/3/3^*)\oplus 3$ & $3\oplus 3^*\oplus 8$
\\ $40_2$ & $(6/4/6)\oplus 4\oplus 20$ & $3^{\oplus 2}\oplus 8$ & $(3^*/3/3^*)\oplus 1\oplus 3$
\\ $40_2^*$ & $(6/4/6)\oplus 4^*\oplus 20^*$ & $(3^*)^{\oplus 2}\oplus 8$ & $(3/3^*/3)\oplus 1\oplus 3^*$
\\ \hline
\end{tabular}\end{center}

\medskip

This proof takes ten steps. The first is to understand the placement of the $14$s, and to prove that there are no submodules $40_1$ and $40_1^*$ of $L(E_8){\downarrow_H}$.

The second analyses the actions of $M$ and $u$ on $L(E_8)$. In this step we construct the actual copy of $H$ that does embed in $E_8$ (inside $A_4A_4$), and show that if $u$ belongs to class $2A_2$ then $H$ is strongly imprimitive. Furthermore, we determine the possible $E_8$-classes for $u$, and how their Jordan blocks distribute amongst the $1$- and $\theta$-eigenspaces of $z$. This information is relied upon heavily in later steps.

Step 3 proves that there must be a subquotient $40_2\oplus 40_2^*$ in $L(E_8){\downarrow_H}$, and Steps 4 and 5 prove that there is a subquotient $40_1\oplus 40_1^*$. Steps 6 and 7 combine these, and prove that $40_1\oplus 40_1^*\oplus 40_2\oplus 40_2^*$ is in fact a subquotient. Step 8 proves that the copies of $24$ do not lie above and below this subquotient in a particular configuration, and therefore $24^{\oplus 2}$ is a subquotient. Finally, Steps 9 and 10 provide the final contradiction.

\medskip

\noindent \textbf{Step 1}: Every composition factor $14$ in $L(E_8){\downarrow_L}$ lies in a subquotient $1/14/1$. In particular, every $kL$-submodule of $L(E_8)$ with a composition factor $14$ has a $14/1$ subquotient. There is no submodule $40_1^\pm$ of $L(E_8){\downarrow_H}$.

\medskip\noindent The subgroup $L$ normalizes a $2$-subgroup of $H$, hence lies inside a parabolic subgroup of $\mb G$. The composition factors of $L(E_8){\downarrow_L}$ are
\[ (20,20^*)^2,14^4,6^8,(4,4^*)^7,1^8,\]
and from \cite[Propositions 4.1 and 4.2]{craven2017} we see that $L$ cannot lie in an $E_7$-parabolic subgroup. Since $L$ cannot lie in an $A_1$ or $A_2$ subgroup, if $L$ lies in any parabolic it must lie in either an $A_7$-, $A_3A_4$- or $D_7$-parabolic subgroup.

If $L$ lies in an $A_7$-parabolic subgroup, the composition factors on $M(A_7)$ are either $4,4$ or $4,4^*$. Either way, $L$ lies in an $A_3A_3$-parabolic subgroup of $A_7$, hence in an $A_3A_4$-parabolic subgroup of $\bG$. If it lies in a $D_7$-parabolic subgroup then the composition factors of $L$ on $M(D_7)$ are $6,4,4^*$, and this lies in a $D_3A_3$-parabolic subgroup, which of course lies in an $A_3A_4$-parabolic subgroup. Hence we may assume that $L$ lies in an $A_3A_4$-parabolic subgroup. In this case the action on $M(A_3)$ must be $4$ (up to automorphism of $L$) and on $M(A_4)$ it must be $1\oplus 4^\pm$.

Recall that the $A_7$-parabolic subgroup of $E_8$ acts on $L(E_8)$ with seven layers. We do not need the precise structure here, but the layers are $\Lambda^i(M(A_8))$ for $i=1,2,3$, their duals, and $M(A_8)\otimes M(A_8)^*$.

If $L$ acts on $M(A_7)$ as $4\oplus 4^*$, then the first four layers of the action are
\[ 4\oplus 4^*,\quad (1/14/1)\oplus 6^{\oplus 2},\quad 20\oplus 20^*\oplus (4\oplus 4^*)^{\oplus 2},\quad (6/4/6)\oplus (6/4^*/6)\oplus (1/14/1)^{\oplus 2}\]
and then the duals of the third to first respectively. If $L$ acts on $M(A_7)$ as $4\oplus 4$, however, then the layers of the action are
\[ 4\oplus 4,\quad (6/4/6)\oplus 6^{\oplus 2},\quad (20^*)^{\oplus 2}\oplus (4^*)^{\oplus 4},\quad (1/14/1)^{\oplus 4},\]
and then the duals of the third to the first. The result is the same: $14$ cannot be a submodule of $L(E_8){\downarrow_L}$.

If $L$ is contained in an $A_3A_4$-parabolic then a similar situation unfolds, but now there are more layers (see, for example, \cite[Table 22]{thomas2016}). The factors include $(101,0000)$ and $(000,1001)$, which as in the previous case must yield a summand $1/14/1$ in both cases. There are also factors $(100,1000)$ and its dual. If the action on $M(A_4)$ is $1\oplus 4^*$ then these two modules for $A_3A_4$ restrict to $L$ with a summand $1/14/1$. There are also factors $(100,0010)$ and its dual. If the action on $M(A_4)$ is $1\oplus 4$ then there is a summand $4^*$ in the restriction of  $L(0010)$ to $L$, hence again there is a summand $1/14/1$ in the $L$-action on $(100,0010)$ and its dual. Either way, we again see four separate subquotients $1/14/1$. Thus the first statement always holds.

\medskip

From the previous table, $40_1$ restricts to $L$ as $6\oplus 14\oplus 20$, so $40_1^\pm$ cannot be a submodule of $L(E_8){\downarrow_H}$.

\medskip

Recall that we have chosen to lie in the centralizer of $z$, which is $D_7T_1$. Since $u$ is unipotent, $u\in D_7$.

\medskip

\noindent \textbf{Step 2}: If $u$ lies in class $2A_2$ of $D_7$ then $H$ is strongly imprimitive. Furthermore, $u$ can only lie in one of the following classes of $E_8$:
\[ 2A_2,\;\; A_3+A_1,\;\; A_3+2A_1,\;\; A_3+A_2,\;\; A_3+A_2^{(2)},\;\;2A_3,\;\; D_4(a_1)+A_1,\;\; D_4(a_1)+A_2.\]

\medskip\noindent The composition factors of $M(D_7){\downarrow_M}$ are $(3,3^*)^2,1^2$, and $u$ acts on $3^\pm$ with a single Jordan block. Since the only indecomposable module with a trivial composition factor and no trivial submodule or quotient is $P(3)$, which has dimension $16$, we see that $M$ stabilizes a line and hyperplane on $M(D_7)$.

As $u$ acts on $3$ with a block of size $3$, if $M(D_7){\downarrow_M}$ is semisimple then $u$ acts on $M(D_7)$ with blocks $3^4,1^2$. Note that $u$ acts on $3/3^*$ with blocks $4,2$, and on $3/1$ with a single block $4$. If the trivial modules are summands then $u$ can only act on $M(D_7)$ with blocks one of $3^4,1^2$ and $4,3^2,2,1^2$ and $4^2,2^2,1^2$ (as $3/3^*/3/3^*$ does not support a symmetric form: to see this, notice that $4^3,1^2$ is not an acceptable unipotent action on $M(D_7)$). On the other hand, if the trivials are not summands then there must be two blocks of size $4$, so $u$ must act as either $4^2,3^2$ or $4^3,2$.

Using \cite{lawther2009} (including Table 8 from there to find the classes of $D_7$), we see that the corresponding classes of $E_8$ are exactly those described in the statement.

\medskip

Now suppose that $u$ lies in class $2A_2$. Since $M$ is contained in the $A_2A_2$-Levi subgroup of $D_7$, it is easier to compute the action of $M$ on $L(E_8)$, and find a suitable positive-dimensional subgroup containing $M$, by placing $M$ inside the $A_8$ maximal-rank subgroup, acting on $L(E_8)$ as the sum of $M(A_8)\otimes M(A_8)^*$ (minus a trivial summand), $\Lambda^3(M(A_8))$ and its dual. The action of $M$ on $M(A_8)$ is either $3\oplus 3\oplus 1^{\oplus 3}$ or $3\oplus 3^*\oplus 1^{\oplus 3}$ (both give the same action on $L(E_8)$) and this action is a sum of modules of the form: $3$; $3^*$; $1$; $3\otimes 3$; $3\otimes 3^*$; $3^*\otimes 3^*$. (Recall that
\[ \Lambda^3(A\oplus B)\cong \Lambda^3(A)\oplus \Lambda^3(B)\oplus (A\otimes \Lambda^2(B))\oplus (\Lambda^2(A)\otimes B),\]
and also $\Lambda^3(3)=1$ and $\Lambda^2(3)=3^*$.)

Let $\mb Y$ denote the algebraic $A_2$ subgroup of $A_8$ acting as $L(10)\oplus L(10)\oplus L(0)^{\oplus 3}$ or $L(10)\oplus L(01)\oplus L(0)^{\oplus 3}$, depending on the action of $M$ on $M(A_8)$. We see that $\mb Y$ contains $M$, and the action of $\mb Y$ on $L(E_8)$ is again a sum of copies of $L(10)$, $L(01)$, $L(0)$, $L(10)\otimes L(10)$, $L(10)\otimes L(01)$ and $L(01)\otimes L(01)$. It is well known that
\[ L(10)\otimes L(10)=L(01)/L(20)/L(01),\quad L(10)\otimes L(01)\cong L(0)\oplus L(11),\]
so that $\mb Y$ stabilizes any semisimple $kM$-submodule of $L(E_8)$.

\medskip

If $u$ lies in class $2A_2$ of $D_7$, then $u$ acts on $M(D_7)$ with blocks $3^4,1^2$, i.e., $M(D_7){\downarrow_M}$ is semisimple. The unipotent class $2A_2$ of $D_7$ is contained in class $2A_2$ of $E_8$ by \cite[4.13]{lawther2009}, so $u$ acts on $L(E_8)$ with blocks $4^{28},3^{36},1^{28}$. This is the same as the action of $u$ on the sum of the composition factors of $L(E_8){\downarrow_H}$, so any extension between composition factors in that module must split on restriction to $\gen u$. Thus the following extensions are not allowed in $L(E_8){\downarrow_H}$:
\[ 10^*/5,\quad 10^*/10,\quad 10/40_1,\quad 5/40_2^*,\quad 24/40_1,\quad 24/40_2.\]
Thus $40_2^\pm$, $24$ and $10^\pm$ must all be summands of $L(E_8){\downarrow_H}$, and the only possibility compatible with the previous step is
\[ (5/40_1/5)\oplus (5^*/40_1^*/5^*)\oplus 40_2\oplus 40_2^*\oplus 24^{\oplus 2}\oplus 10\oplus 10^*.\]
As the restriction of $5$ and $5^*$, and $10$ and $10^*$, to $M$ are semisimple, they are all stabilized by $\mb Y$ as well, and so therefore $H$ is strongly imprimitive by Theorem \ref{thm:intersectionorbit}.

\medskip

We also wish to understand the action of $u$ on the eigenspaces of $z$. Since the $1$-eigenspace of $z$ has the form $1/L(D_7)^\circ/1$, and the $\theta^i$-eigenspaces are the sums of $M(D_7)$ and a half-spin module $HS(D_7)$, we are interested in the actions of unipotent classes of $D_7$ on $M(D_7)$, $1/L(D_7)^\circ/1$ and $HS(D_7)$, together with the action on $L(E_8)$ and which class of $E_8$ contains the unipotent class of $D_7$. These are given in Table \ref{t:classesofd7psl52}.

\begin{table}
\begin{center}\begin{footnotesize}\begin{tabular}{cccccc}
\hline Class in $D_7$ & Class in $E_8$ & Act. $M(D_7)$ & On $1/L(D_7)^\circ/1$ & On $HS(D_7)$ & On $L(E_8)$
\\ \hline $2A_2$&$2A_2$&$3^4,1^2$&$4^{12},3^{12},1^8$&$4^8,3^8,1^8$&$4^{28},3^{36},1^{28}$
\\$A_3+A_1$&$A_3+A_1$&$4^2,2^2,1^2$&$4^{18},2^8,1^4$&$4^{12},2^6,1^4$&$4^{46},2^{24},1^{16}$
\\$A_3+D_2$&$A_3+2A_1$&$4^2,2^2,1^2$&$4^{18},2^{10}$&$4^{12},2^8$&$4^{46},2^{30},1^4$
\\ $A_2+D_3$&$A_3+A_2$&$4,3^2,2,1^2$&$4^{16},3^6,2^4,1^2$&$4^{16}$&$4^{50},3^{10},2^6,1^6$
\\$D_3+A_2^{(2)}$&$A_3+A_2^{(2)}$&$4^2,2^2,1^2$&$4^{18},3^2,2^6,1^2$&$4^{16}$&$4^{54},3^2,2^{10},1^6$
\\$A_3+D_3$&$2A_3$&$4^3,2$&$4^{22},2^2$&$4^{16}$&$4^{60},2^4$
\\$D_4(a_1)+A_1$&$D_4(a_1)+A_1$&$4^2,2^2,1^2$&$4^{18},3^2,2^6,1^2$&$4^{16}$&$4^{54},3^2,2^{10},1^6$
\\$D_4(a_1)+A_2$&$D_4(a_1)+A_2$&$4^2,3^2$&$4^{20},3^4$&$4^{16}$&$4^{56},3^8$
\\ \hline
\end{tabular}\end{footnotesize}\end{center}\caption{Possible unipotent classes for $u$ and their actions on the $1$- and $\theta$-eigenspaces of $L(E_8)$}\label{t:classesofd7psl52}\end{table}

From this we see that if either $24$ or $10\oplus 10^*$ is a summand of $L(E_8){\downarrow_H}$, then $u$ (which acts on $24$ with blocks $4^2,3^4,1^4$ and on $10$ with blocks $3^3,1$) lies either in class $2A_2$---and then $H$ is strongly imprimitive---or $A_3+A_2$, so we may assume the latter case. We also see that $24^{\oplus 2}$ cannot be a summand of $L(E_8){\downarrow_H}$ if $H$ is not strongly imprimitive. Moreover, $40_2$ cannot be a summand since $u$ acts on $40_2\oplus 40_2^*$ with blocks $4^{12},3^8,1^8$, and so $u$ lies in class $2A_2$, so $H$ is strongly imprimitive. If $40_1^\pm$ is a summand of $L(E_8){\downarrow_H}$ then $14$ is a summand of $L(E_8){\downarrow_L}$, contradicting Step 1. Finally, if $5\oplus 5^*$ is a summand of $L(E_8){\downarrow_H}$ then the class of $w$ needs to have at least 20 blocks of size $8$ (the projective parts of all composition factors) and two blocks of size $5$ (from the two summands), together with a block of size at least $7$ (coming from the $10$). The only option for that is class $A_6$, acting with blocks $8^{22},7^6,5^4,3^2,1^4$.

\medskip

\noindent \textbf{Step 3}: The $\{40_2^\pm\}$-heart of $L(E_8){\downarrow_H}$ is $40_2\oplus 40_2^*$.

\medskip\noindent We examine the $\{40_2^\pm\}$-heart of $L(E_8){\downarrow_H}$, aiming to show that it is simply $40_2\oplus 40_2^*$. If it is not, then the socle of it must be (up to automorphism) $40_2$. We compute the $\{5^\pm,10^\pm,24,40_1^\pm\}$-radical of $P(40_2)$, add on top any copies of $40_2^*$ that we can, then take the $\{40_2^*\}'$-residual, which yields the self-dual module $M_1$ with structure
\[ 40_2^*/5/40_1/24/40_1^*/5^*,24/40_2.\]
Thus the $\{40_2^\pm\}$-heart of $L(E_8){\downarrow_H}$ is either $M_1$ (up to graph automorphism of $H$) or $40_2\oplus 40_2^*$; suppose it is $M_1$. The element $w$ of order $8$ in $H$ acts on $M_1$ with blocks $8^{26},5^2$; thus from Table \ref{t:unipe8p8} we see that $M_1$ cannot be a summand of $L(E_8){\downarrow_H}$ as there is no unipotent class that acts with blocks of size $5$ and at least $26$ blocks of size $8$.

Consider the quotient by the $\{40_2^\pm\}'$-radical $N_1$ of $L(E_8){\downarrow_H}$: this cannot be $M_1$ either, as then $M_1$ is a quotient, so is a submodule, and the kernel of the map onto $M_1$ must complement $M_1$ in $L(E_8)$, contradicting the fact that $M_1$ is not a summand. One cannot place $10^\pm$ or $5$ on top of $M_1$, but there is a single non-split extension with $5^*$, yielding a module $M_1'$ with structure
\[ 40_2^*/5/40_1/5^*,24/40_1^*/5^*,24/40_2.\]
If $M_1'$ is the quotient of $L(E_8){\downarrow_H}$ by $N_1$ then $N_1$ has composition factors $5,10,10^*$. If $10$ or $10^*$ is not a submodule of $N_1$ then it must have structure $10/10^*/5$. Therefore we need to be able to place  $10$ underneath $M_1'$, and then $10^*$ underneath that, else one of the $10$-dimensional factors must become a quotient, a contradiction.

There is a unique extension first placing $10$ underneath and then $10^*$, yielding a module $M_1''$ with structure
\[ 40_2^*/5/40_1/5^*,24/40_1^*/5^*,10,24/10^*,40_2.\]
However, $w$ acts on $M_1''$ with blocks $8^{27},7,5^2,4^2,2$. There is no unipotent class in Table \ref{t:unipe8p8} whose action encapsulates this, so this is not a subquotient of $L(E_8){\downarrow_H}$ by Lemma \ref{lem:encapsulates}. Hence $10$ or $10^*$ is a submodule of $L(E_8){\downarrow_H}$.

We therefore take the $\{10^\pm\}'$-heart $M_2$ of $L(E_8){\downarrow_H}$, which has $5$ as a submodule, and $M_1'$ as a quotient by this submodule. There is a unique non-split such extension, a self-dual module, and $w$ acts on it with blocks $8^{26},6^2,4^2$, so $M_2$ cannot be a summand of $L(E_8){\downarrow_H}$ either. Thus the quotient by the $\{10^\pm\}$-radical must also have a single copy of $10^\pm$ in it. Both $10$ and $10^*$ have an extension with $M_2$, so we need to examine each.

We obtain two modules, with socle structures
\[  40_2^*/5,10/40_1/5^*,24/40_1^*/5^*,24/5,40_2\;\text{and}\; 40_2^*/5/40_1/5^*,24/40_1^*/5^*,10^*,24/5,40_2.\]

The former may be excluded because $v$ acts on it with blocks $4^{59},2$ and on $10$ with blocks $4^2,2$, so on $L(E_8)$ with at least $61$ blocks of size $4$, contrary to Table \ref{t:unipe8p4}.

To exclude the latter, we first note that $w$ acts on this module with blocks $8^{27},6,5,4^2,3$. We then place as many copies of $10$ on the bottom of it as we can, which is two: this produces a module with structure
\[ 40_2^*/5/40_1/5^*,24/40_1^*/5^*,10^*,24/5,10,10,40_2,\]
on which $w$ acts with blocks $8^{28},7,5^2,4^3,3,2$. The action of $w$ on $L(E_8)$ must be encapsulated by this, and encapsulate the previous action by Lemma \ref{lem:encapsulates}, but no class in Table \ref{t:unipe8p8} has this property.

We have finally concluded that $M_1$ is not a subquotient of $L(E_8){\downarrow_H}$, so that the $\{40_2^\pm\}$-heart of $L(E_8){\downarrow_H}$ is $40_2\oplus 40_2^*$, as claimed.

\medskip

\noindent \textbf{Step 4}: If $40_1^*/24/40_1$ or $40_1/24/40_1^*$ is a subquotient of $L(E_8){\downarrow_H}$ then $H$ is strongly imprimitive.

\medskip\noindent By applying the graph automorphism if necessary we may assume that $N_1\cong 40_1^*/24/40_1$ is a subquotient of $L(E_8){\downarrow_H}$. There are no extensions of $N_1$ involving $40_2^\pm$, and the $\{40_2^\pm\}$-heart of $L(E_8){\downarrow_H}$ is $40_2\oplus 40_2^*$, so the $\{5^\pm,10^\pm\}'$-heart $N_2$ of $L(E_8){\downarrow_H}$ must be isomorphic to
\[ (40_1^*/24/40_1)\oplus 40_2\oplus 40_2^*\oplus 24.\]
Since the $24$ has no extensions with $5^\pm$ or $10^\pm$, it is a summand of $L(E_8){\downarrow_H}$, so from the end of Step 2, either $H$ is strongly imprimitive or $u$ belongs to class $A_3+A_2$, with $\theta$-eigenspace action $4^{17},3^2,2,1^2$. However, the action of $u$ on the $\theta$-eigenspace of $N_2$ above has blocks $4^{11},3^4,2^3,1^6$, which is not encapsulated by this. This contradiction completes the proof of the step.

\medskip

\noindent \textbf{Step 5}: The $\{40_1^\pm\}$-heart of $L(E_8){\downarrow_H}$ is semisimple.

\medskip\noindent Let $M_1$ denote the $\{40_1^\pm\}$-heart of $L(E_8){\downarrow_H}$, which by applying the graph automorphism of $H$ if necessary we may assume has socle $40_1$ (if it is not $40_1\oplus 40_1^*$). We construct the $\{5^\pm,10^\pm,24,40_2^\pm\}$-radical of the quotient module $P(40_1)/40_1$, then place on top any copies of $40_1^*$ that we can, to obtain a module with structure
\begin{equation} (40_1^*/24)\oplus (5,40_1^*/5^*,10^*,40_2^*/5,10),\label{eq:step5psl52}\end{equation}
with $5$ and $40_2^*$ being quotients. From this we see that if the $\{40_1^\pm\}$-heart is not semisimple, it is (up to graph automorphism of $H$) one of
\[40_1^*/24/40_1,\quad 40_1^*/5^*,10^*/5,10/40_1,\quad 40_1^*/5^*,10^*/5,10,24/40_1,\]
where the last module is not well defined up to isomorphism. The first of these has already been eliminated in the previous step.

Suppose that we are in the second possibility: by Step 1, $40_1^*$ cannot be a quotient of $L(E_8){\downarrow_H}$, and thus the quotient by the $\{40_1\}'$-radical must contain this module, but be strictly larger. Adding as many copies of $5^\pm$, $24$ and $40_2^\pm$ on top of this module as we can, we obtain the module
\[ 24/5,40_1^*/5^*,10^*,40_2^*/5,10,24/40_1.\]
The $14$ in the restriction of the $40_1^*$ to $L$ is still a quotient of this module, so we obtain a contradiction for this case.

\medskip

For the third possibility, $M_1$ is a module of the form $40_1^*/5^*,10^*/5,10,24/40_1$. As before, this cannot be a quotient of $L(E_8){\downarrow_H}$, but we only have a single copy of $24$ (which cannot have any extensions with this module, as in the previous step), $40_2^\pm$ and $5^\pm$ that are not already part of $M_1$. The $\{40_1^\pm,24\}$-heart of $L(E_8){\downarrow_H}$ must be the sum of $M_1$ and $24$, and so the $\{40_1^\pm\}'$-radical of $L(E_8){\downarrow_H}$ must contain the $24$.

While $M_1$ is not determined uniquely up to isomorphism by the socle structure, the quotient $M_1'=M_1/\soc(M_1)$ is. Therefore we may attempt to add copies of $40_2^\pm$ and $5^\pm$ on top of this module to form a quotient of $L(E_8){\downarrow_H}$. However, having done this we obtain a module
\[ 5,40_1^*/5^*,10^*,40_2,40_2^*,40_2^*/5,10,24,\]
which still has a $40_1^*$ quotient, hence we obtain a contradiction again. Thus the $\{40_1^\pm\}$-heart is $40_1\oplus 40_1^*$, as claimed.

\medskip

\noindent \textbf{Step 6}: The $\{5^\pm,10^\pm,24\}'$-heart of $L(E_8){\downarrow_H}$ is not $(40_2^*/5/40_1)\oplus (40_1^*/5^*/40_2)$.

\medskip\noindent Suppose to the contrary that the $\{5^\pm,10^\pm,24\}'$-heart is this module, denoted $M_1$. (This is unique up to isomorphism, as we will see in the next step.) We first claim that there must be a submodule $5^\pm$ in $L(E_8){\downarrow_H}$. Since $24$ has no extensions with $5^\pm$ or $10^\pm$, we may ignore that. Thus assuming the claim is false, $10$ or $10^*$ is a submodule of $L(E_8){\downarrow_H}$, so we take the $\{10^\pm,24\}'$-heart of $L(E_8){\downarrow_H}$, which is the module $M_1$ with $5$ and $5^*$ attached in some way. Since $\Ext_{kH}^1(M_1,5)\cong k$ and $\Ext_{kH}^1(M_1,5^*)=0$, there is a unique way to extend $M_1$ by $5$ and $5^*$, namely
\[ (40_2^*/5/40_1/5)\oplus (5^*/40_1^*/5^*/40_2).\]
One can place a copy of $10^*$ under this module, but it does not go below the $5$, so this remains a submodule. Of course, if the $\{10^\pm,24\}'$-heart is $M_1\oplus 5\oplus 5^*$ instead of an indecomposable module then one must have either $5$ or $5^*$ as a submodule, and so the claim holds.

\medskip

Write $M_2$ for the $\{5^\pm,24\}'$-heart of $L(E_8){\downarrow_H}$. Either the $10$ and $10^*$ extend $M_1$, or $M_2$ is the sum of $M_1$ and a module with factors $10$ and $10^*$.

If $M_2\cong M_1\oplus 10\oplus 10^*$ then we must place a copy of $5$ below this to lie under the $40_1$, else the $14$ is a submodule on restriction to $L$ and we obtain a contradiction by Step 1. There are two such modules, and for each there is a unique self-dual module $M_2'$ that can be obtained by placing $5^*$ on top:
\[\begin{split} (40_2^*/5/40_1/5)\oplus (5^*/40_1^*/5^*/40_2)\oplus& 10\oplus 10^*\quad \text{and}
\\ &(40_2^*/5/10^*,40_1/5)\oplus (5^*/40_1^*/5^*/10,40_2).\end{split}\]
In both cases the action of $u$ on the $1$-eigenspace of $M_2'$ is $4^{12},3^4,2^4,1^4$, and since each $24$ contributes $4^2$ to this eigenspace ($M$ acts on the $1$-eigenspace of $24$ as $1^{\oplus 2}\oplus 8$), the class to which $u$ belongs must have $1$-eigenspace encapsulating $4^{16},3^4,2^4,1^4$. This eliminates $2A_2$, $A_3+A_1$, $A_3+2A_1$, $2A_3$ and $D_4(a_1)+A_2$ from consideration (see Table \ref{t:classesofd7psl52}). The remaining classes all act on the $\theta$-eigenspace with exactly 22 blocks.

In both cases the module $M_2'$ has a $3$-dimensional $\Ext^1$ with $24$, and the $\theta$-eigenspace of this extended module has $u$-action $4^{14},3^5,2^4,1^6$, so 29 blocks in total. Since the $\theta$-eigenspace of $u$ on $24$ has blocks $3^2,1$, removing two copies of $24$ from this module can remove at most six blocks, so any extension by a $24$ has at least 23 blocks. This proves that $u$ cannot belong to the remaining classes. Thus $M_2$ cannot have this structure.

The argument is very similar if $M_2\cong M_1\oplus (10^\mp/10^\pm)$. The same analysis yields two possibilities for $M_2'$ for each of $10/10^*$ and $10^*/10$:
\[ (40_2^*/5/40_1/5)\oplus (5^*/40_1^*/5^*/40_2)\oplus (10^\mp/10^\pm),\quad\text{and}\] \[5^*,40_2^*/5,10,40_1^*/5^*,10^*,40_1/5,40_2\quad\text{or}\quad 5^*,40_2^*/5,40_1^*/5^*,10^*,40_1/5,10,40_2.\]
The action on the $1$-eigenspace is $4^{12},3^4,2^4,1^4$, as before. Again, there is a $3$-dimensional $\Ext^1$ with $24$, and the $\theta$-eigenspace is slightly different, with blocks $4^{15},3^3,2^5,1^6$, but still $29$ blocks, so the conclusion is still the same, that $M_2$ cannot have this form.

\medskip

Thus $M_2$ is obtained from $M_1$ by extending by $10$ and $10^*$ in some way, rather than a direct sum. We have that $\Ext_{kH}^1(M_1,10)=0$, so if $M_1$ is extended by $10$ and $10^*$, it must be $10^*$ on the bottom and $10$ on the top. There is a unique extension involving $10^*$, and under this we must place $5$, else the $40_1$ becomes a submodule of $L(E_8){\downarrow_H}$ modulo the $24$, contrary to Step 1.

There is a $2$-dimensional $\Ext^1$ with $5$, yielding the module
\[ (40_2^*/5/40_1/5)\oplus (40_1^*/5^*,10^*/5,40_2).\]
There are, up to isomorphism, three possible quotients of this by $5$, depending on whether the $5$ is in the first summand, the second, or diagonal. We cannot remove the one under the first summand as that contradicts Step 1, as we have seen before.

Write $M_2'$ for the extension by a single $5$ that lies in $L(E_8){\downarrow_H}$, and suppose that the $5$ in $M_2'$ lies entirely below the $40_1$. In particular, $10^*$ is a submodule of $L(E_8){\downarrow_H}$, so $10$ is a quotient and we can add $5^*$ on top of $M_2'$. There is a unique such module,
\[ (40_2^*/5/40_1/5)\oplus (5^*/40_1^*/5^*/10^*,40_2),\]
but it is possible to place two copies of $10$ on top of this module, one on top of the $40_1$ and one on top of the $10^*$. Since they lie in different summands, we again obtain three (rather than infinitely many) non-isomorphic extensions by $10$, one of which is clearly not self-dual. Write $M_2''$ for whichever of these extensions is a subquotient of $L(E_8){\downarrow_H}$.

For both possibilities for $M_2''$, the element $u$ acts on the $1$-eigenspace of $M_2''$ with blocks $4^{14},3^2,2^4,1^2$. The two $24$s contribute four blocks of size $4$ to this, to yield $4^{18},3^2,2^4,1^2$. The only unipotent classes to encapsulate this action are $A_3+A_2^{(2)}$ and $D_4(a_1)+A_1$, both with action on the $1$-eigenspace of $L(E_8)$ having blocks $4^{18},3^2,2^6,1^2$. In particular, any extension of $M_2''$ involving a $24$ cannot produce any new blocks of size $4$.

The dimension of $\Ext_{kH}^1(M_2'',24)$ is $3$, but in this full extension, $u$ acts on the $1$-eigenspace with $21$ blocks of size $4$. Therefore there is a unique $2$-dimensional subspace of this space consisting of extensions that do not produce any new blocks of size $4$. The action of $u$ on the $\theta$-eigenspace of this module is $4^{14},3^2,2^6,1^4$, which is 26 blocks. Removing one copy of 24 can remove at most three blocks, so $u$ acts with at least 23 blocks on any such quotient. But $u$ acts with only 22 blocks on $L(E_8)$ if $u$ comes from one of the two remaining classes, a clear contradiction. Thus we obtain a contradiction to either structure for $M_2''$, and hence for the particular structure of $M_2'$.

\medskip

Thus the $5$ we added on to make $M_2'$ must be diagonal. We add $5^*/10$ on top of $M_2'$ and $24^2$ as well, and the $1$-eigenspaces of these modules have five blocks of size $4$ in them. Together with the $1$-eigenspace of $M_2'$, which is $4^{12},3^3,2^3,1^2$, this means that we must encapsulate $4^{17},3^3,2^3,1^2$, so again we are forced into $u$ lying in the class $A_3+A_2^{(2)}$ or $D_4(a_1)+A_1$. These two classes act on the $1$- and $\theta$-eigenspaces with blocks $4^{18},3^2,2^6,1^2$ and $4^{18},2^2,1^2$ respectively.

We go back a step now, and consider the $\{5^\pm,10^\pm\}'$-heart $M_3$ of $L(E_8){\downarrow_H}$. This is obtained from $M_1$ by adding a copy of $24$ above and below $M_1$. The space $\Ext_{kH}^1(24,M_1)$ is $3$-dimensional, with the full extension given by
\[ (40_2^*/5,24/40_1)\oplus (24/40_1^*/5^*,24/40_2),\]
but the space $\Ext_{kH}^1(24,M_2')$ is only $2$-dimensional. The missing copy of $24$ is the one in the fourth socle layer, so $M_3$ cannot have that copy in it either.

There are, up to isomorphism, three modules that are extensions with quotient $24$ and submodule $M_1$, and these are
\[ (40_2^*/5,24/40_1)\oplus (40_1^*/5^*/40_2),\quad (40_2^*/5/40_1)\oplus (40_1^*/5^*,24/40_2),\]
\[40_1^*,40_2^*/5,5^*,24/40_1,40_2.\]
Each of these has a $2$-dimensional $\Ext^1$-space with $24$: in the first two cases, there is a unique self-dual module that is an extension with submodule $24$, which must be $M_3$, whereas in the third case all but two possible extensions (of the `$|k|+1$' many) are self-dual, so there are many possibilities for $M_3$. However, the action of $u$ on the $1$- and $\theta$-eigenspaces of these are the same.

Note that we obtain $L(E_8){\downarrow_H}$ from $M_3$ by adding $10^*/5$ on the bottom and $5^*/10$ on the top; $u$ acts on the $1$-eigenspace of each of these with blocks $4,2,1$.

In the first possibility, $M_3$ has structure
\[ (40_2^*/5,24/40_1)\oplus (40_1^*/5^*/24,40_2),\]
and there is a unique extension with submodule $10^*$, it lying in the socle of the second summand. The $\theta$-eigenspace of this larger module has $u$-action $4^{12},3^3,2^6,1^4$, and the remaining modules $5$, $5^*$ and $10$ contribute $1$, $1$ and $3,1$ to this eigenspace respectively. Thus, in particular, $u$ cannot act on the $\theta$-eigenspace of $L(E_8)$ with more than sixteen blocks of size $4$, contradicting $u$ lying in one of the classes above.

In the second possibility, $M_3$ has structure
\[ (40_2^*/5/24,40_1)\oplus (40_1^*/5^*,24/40_2),\]
and $u$ acts on the $1$-eigenspace of this with blocks $4^{16},3^4,1^2$. Adding in the two blocks of size $4$ from the rest of $L(E_8){\downarrow_H}$, we obtain $4^{18},3^4,1^2$, which is not encapsulated by the action of $u$ on the $1$-eigenspace of $L(E_8){\downarrow_H}$, so this case cannot occur either.

Thus the third case occurs: the action of $u$ on the $1$-eigenspace of the module above has blocks $4^{13},3^4,2,1^2$. On this we add $5^*/10$, $10^*/5$ and $24$, which adds one, one and two blocks of size $4$ to this eigenspace respectively. Thus $u$ acts on $L(E_8)$ with blocks encapsulating $4^{17},3^4,2,1^2$, but this has $21$ blocks of size at least $3$, whereas $u$ must act with twenty blocks of size at least $3$. This contradiction proves that $M_2'$ cannot be either possibility, and this completes the proof of the step.

\medskip

\noindent \textbf{Step 7}: The $\{5^\pm,10^\pm,24\}'$-heart of $L(E_8){\downarrow_H}$ is $40_1\oplus 40_1^*\oplus 40_2\oplus 40_2^*$. 

\medskip\noindent If both $40_i$ and $40_i^*$ appear in the socle of the $\{5^\pm,10^\pm,24\}'$-heart $M_1$ of $L(E_8){\downarrow_H}$ then both $40_i^\pm$ are summands of $M_1$, hence $M_1$ is the sum of the $40_i^\pm$ and the $\{40_{3-i}^\pm\}$-heart, which is $40_{3-i}\oplus 40_{3-i}^*$. In particular, this means that $M_1$ is semisimple or (up to graph automorphism) the socle of $M_1$ is either $40_1\oplus 40_2$ or $40_1\oplus 40_2^*$. If the socle is $40_1\oplus 40_2^*$ then we construct the $\{5^\pm,10^\pm,24,40_1^*,40_2\}$-radical of $P(40_2^*)$, which is
\[ 5,24/40_2^*.\]
Hence the $40_2^*$ splits as a summand of $M_1$, and therefore $40_2$ must also lie in the socle, which is a contradiction. Thus if our conclusion does not hold then the socle of $M_1$ is $40_1\oplus 40_2$, so we assume that from now on. The $\{5^\pm,10^\pm,24\}$-radicals of $P(40_i)$, then with any copies of $40_1^*$ and $40_2^*$ placed on top, then with any modules not $40_1^*$ or $40_2^*$ removed from the top, are
\[ 40_1^*/5^*,10^*,40_1^*,40_2^*/5,10,24/40_1,\quad 40_1^*/5^*/40_2.\]
(The former of these can be seen from the radical in (\ref{eq:step5psl52}).) A pyx for $M_1$ is the sum of these two modules. Notice that $M_1$ must contain the $40_2^*$, and hence contain the $5$ in the second socle layer.

\medskip

Suppose that $40_1^*$ lies in the fourth socle layer of $M_1$: then $M_1$ must have the socle structure
\[ 40_1^*/5^*,10^*,40_2^*/5,10,(24)/40_1,40_2,\]
where the $24$ might or might not be present. Although this module is not unique up to isomorphism, the submodule $M_1'$ whose quotient is $40_1^*/5^*$ is unique up to isomorphism, and is
\[ (10^*,40_2^*/5,10,(24)/40_1)\oplus 40_2.\]
Since $M_1$ has $40_1$ as a submodule, and this cannot be a submodule of $L(E_8){\downarrow_H}$ by Step 1, we must be able to place $5$ or $24$ below this. (The only modules in $L(E_8){\downarrow_H}$ but not in $M_1$ are $5$, $5^*$, and one or two copies of $24$.) The $24$s are irrelevant though, since upon restriction to $L$ there is still a submodule $14$. Hence we must be able to place a copy of $5$ below $M_1'$ in such a way as to stop the $14$ being a submodule of the restriction to $L$.

The space $\Ext_{kH}^1(M_1',5)$ is $1$-dimensional, but the unique non-split extension has structure
\[(10^*,40_2^*/5,10,(24)/5,40_1)\oplus 40_2.\]
This proves that $M_1$ cannot have this form, and in particular the $40_1^*$ lies in the third socle layer of $M_1$.

Thus $M_1$ is a submodule of
\[ (40_1^*,40_2^*/5,24/40_1)\oplus (40_1^*/5^*/40_2).\]
The $40_1^*$ cannot lie entirely on the second summand by Step 6 or on the first by self-duality, and since the two summands have no composition factor in common, the diagonal submodule is unique up to isomorphism. The action of $u$ on the $\theta$-eigenspace is $4^{12},2^6,1^3$. Adding the copy of $24$, which must be a summand of $L(E_8){\downarrow_H}$ and on whose $\theta$-eigenspace $u$ acts with blocks $3^2,1$, we obtain a $u$-action with blocks $4^{12},3^2,2^6,1^4$. Since $24$ is a summand, from the end of Step 2 we see that $u$ must belong to class $2A_2$ or $A_3+A_2$, acting with blocks $4^8,3^{12},1^{10}$ or $4^{17},3^2,2,1^2$ respectively. Neither of these actions encapsulates the block structure above. This is a contradiction, and so $M_1$ cannot have this form either. This completes the proof of the step.

\medskip

\noindent \textbf{Step 8}: The $\{24\}$-heart of $L(E_8){\downarrow_H}$ is $24^{\oplus 2}$.

\medskip\noindent Let $M_1$ denote the $\{24\}$-heart of $L(E_8){\downarrow_H}$, and suppose that $M_1$ is not $24^{\oplus 2}$. Thus $M_1$ must be a module that has $24$ in the socle and top, and some $40_i^\pm$ in the heart. There is exactly one copy of $24$ that may be placed on top of $40_1,40_1^*,40_2,40_2^*/24$, and this yields a self-dual module
\[ 24/40_1,40_1^*,40_2,40_2^*/24.\]
Thus we assume from now on that this module is $M_1$. The restriction of $M_1$ to $L$ has a submodule $14^{\oplus 2}$, so we need two copies of $5^\pm$ below $M_1$ in order to comply with Step 1.

If $10^\pm$ is a submodule of $L(E_8){\downarrow_H}$, then we add a copy of $5$ and $5^*$ above and below $M_1$ to form the $\{10^\pm\}'$-heart $M_2$, due to the conclusion of Step 1. There is a unique such module,
\[ 5,5^*,24/40_1,40_1^*,40_2,40_2^*/5,5^*,24,\]
and $u$ acts on this module with blocks $4^{44},3^{14},1^{10}$. From Table \ref{t:classesofd7psl52}, the only block structure that encapsulates this is $4^{50},3^{10},2^6,1^6$, and this comes from $A_3+A_2$. We only have $10$ and $10^*$ to add on, and $M_2$ cannot be a summand of $L(E_8){\downarrow_H}$ from comparison of the $u$-actions, so $10^\pm$ must attach to $M_2$ to form $L(E_8){\downarrow_H}$.

The action of $u$ on the $1$-eigenspaces of $M_2$ and $L(E_8)$ are $4^{16},3^6,1^2$ and $4^{16},3^6,2^4,1^2$ respectively, and $u$ acts on $10$ with blocks $3^3,1$, and $1$-eigenspace $3,1$. We see that the number of blocks of sizes $3$ and $4$ remain the same, so there is an extension of $1^2$ by $3$ that is encapsulated by $2^4,1^2$, which is clearly impossible.

Thus we were incorrect in assuming that $10^\pm$ is a submodule, and hence it cannot be. Write $M_3$ for the $\{5^\pm\}'$-heart of $L(E_8){\downarrow_H}$, which is obtained from $M_1$ by adding a submodule $10$ and a quotient $10^*$. (We may choose this as $M_1$ is invariant under the graph automorphism.) If the $10$ and $10^*$ form a separate summand we have the possibilities that $M_3\cong M_1\oplus 10\oplus 10^*$ and $M_3\cong M_1\oplus (10^*/10)$. Since both $40_1$ and $40_1^*$ are submodules of $M_3$ (ignoring the $24$) we need both $5$ and $5^*$ as submodules (and also quotients) of $L(E_8){\downarrow_H}$, by Step 1.

\medskip

In the first case the actions of $u$ on the $1$- and $\theta$-eigenspaces of $M_3$ are $4^{16},3^4,1^4$ and $4^{14},3^6$ respectively, and we see that the only unipotent class that encapsulates both is $D_4(a_1)+A_2$, so $u$ acts on the $\theta$-eigenspace with blocks $4^{18},3^2$. Comparing this and the $\theta$-eigenspace of $M_3$, we see that no extension involving $M_3$ and the $5^\pm$ can split on restriction to the $\theta$-eigenspace of $\gen u$. 

Note that $\Ext_{kH}^1(5,M_3)$ is $2$-dimensional, and $u$ acts on the $\theta$-eigenspace of this extension with blocks $4^{14},3^6,1^2$. But this means any such extension splits upon restriction, contrary to what is needed.

\medskip

In the second case the actions of $u$ on the $1$- and $\theta$-eigenspaces of $M_3$ are $4^{16},3^4,1^4$ and $4^{15},3^4,2$ respectively, and we see that the only unipotent classes that encapsulate both are $A_3+A_2$, $2A_3$ and $D_4(a_1)+A_2$. We have that $\Ext_{kH}^1(5,M_3)$ is again $2$-dimensional, and the action of $u$ on the $1$-eigenspace of the extension of $M_3$ by $5^{\oplus 2}$ is $4^{17},3^4,2,1^4$, so any extension of $M_3$ with submodule $5$ has $25$ blocks in the action of $u$ on its $1$-eigenspace. This eliminates $2A_3$ and $D_4(a_1)+A_2$, leaving just $A_3+A_2$. Thus $u$ acts on the $1$-eigenspace of $L(E_8)$ as $4^{16},3^6,2^4,1^2$, whereas the action of $u$ on $1$-eigenspace of $M_3$ has blocks $4^{16},3^4,1^4$.

There are two ways to derive a contradiction: the first is to note that there must be a unique extension that does not have an extra block of size $4$, and that is the one where the $5$ is placed on top of $M_3$. Alternatively, recall that $u$ acts on the $1$-eigenspace of $5$ with a single block of size $3$. There is no way to add four blocks of size $3$ to the action for $M_3$ to produce the action for $L(E_8)$.

\medskip

Thus the $10$ and $10^*$ in $M_3$ are attached to $M_1$, and we assume that $10$ is a submodule of $M_3$ by applying the graph automorphism. There is a unique such non-split extension $M_1'$, with structure
\[ 24/40_1,40_1^*,40_2,40_2^*/10,24.\]
One may place two copies of $10^*$ on top of this, one into the second socle layer forming a submodule $10^*/10$, and one in the third socle layer above $40_1^*$. The module $M_3$ is a potentially diagonal submodule of these two copies of $10^*$ (definitely not the one solely on top of the $10$, as this would not be self-dual), and we will compute the action of $u$ on its $1$- and $\theta$-eigenspaces.

The action of $M$ on the $1$-eigenspace of $M_1'$ is
\[ 8^{\oplus 4}\oplus (1/3,3^*/1)^{\oplus 2}\oplus (3/3^*/3)\oplus (3^*/3/3^*)\oplus (3^*/3)\oplus 3\oplus 1,\]
and on the $\theta$-eigenspace is
\[ 8^{\oplus 2}\oplus P(3)\oplus P(3^*)\oplus (1/3,3^*/1)\oplus (3/3^*)\oplus 3\oplus 3^*\oplus 3^*.\]
The $1$-eigenspace of $10$ is $1\oplus 3^*$ and the $\theta$-eigenspace is $3$: the extension by $10^*$ contributing the $10^*/10$ submodule splits on restriction to $M$, so we obtain summands $1\oplus 3^*$ in the first and $3$ in the second. The other copy of $10^*$ places the $3^*$ on top of the summand $3$ with the $1$ splitting off in the $1$-eigenspace, and places the $3$ on top of a summand $3^*$ in the $\theta$-eigenspace.

Thus any extension other than the one contributing the submodule $10^*/10$ yields an extension $1\oplus (3^*/3)$ in the $1$-eigenspace and $3/3^*$ in the $\theta$-eigenspace. The actions of $u$ on the $1$- and $\theta$-eigenspaces of $M_1'$ are $4^{17},3,2,1^3$ and $4^{15},3^3,2$ respectively, so on the $1$- and $\theta$-eigenspaces of $M_3$ the actions must be $4^{18},2^2,1^4$ and $4^{16},3^2,2^2$ respectively. 

Looking at Table \ref{t:classesofd7psl52}, the $1$-eigenspace action is not encapsulated by $2A_2$ or $A_3+A_2$, and the $\theta$-eigenspace is not encapsulated by $A_3+A_1$ or $A_3+2A_1$. We must eliminate the remaining four.

To eliminate $A_3+A_2^{(2)}$ and $D_4(a_1)+A_1$, which have the same action on the $1$-eigenspace, namely $4^{18},3^2,2^6,1^2$, note that we must add four blocks of size $3$ (one from each $5^\pm$) to $4^{18},2^2,1^4$, and this is not possible. To see this, notice that the socle of the extended module is the sum of the socles of the submodule and quotient, so this may be removed without loss of generality, and now we are adding four blocks of size $2$ and we need to increase the number of blocks by six, which is impossible. Thus $u$ lies in class $2A_3$ or $D_4(a_1)+A_2$, each with exactly $24$ blocks in the $1$-eigenspace of $L(E_8)$, the same as $u$ has on $M_3$, and $20$ blocks in the $\theta$-eigenspace.

\medskip

In order for $10$ not to be a submodule of $L(E_8){\downarrow_H}$, there must be a submodule $5^*$. If both copies of $5^*$ are submodules then this module is determined uniquely, because $\Ext_{kH}^1(M_1',5^*)$ is $2$-dimensional (and because neither of the remaining classes for $u$ is compatible with a summand $5^\pm$). This module, which has the form
\[ 24/40_1/10,40_1^*,40_2,40_2^*/5^*,5^*,24,\]
has $\theta$-eigenspace $4^{16},3^2,2,1$, with 21 blocks. This cannot be encapsulated by the two remaining classes for $u$, which have 20 blocks. Thus there must be a $5^*$ in both the socle and top.

The dimension of $\Ext_{kH}^1(5^*,M_1')$ is $2$, so we may place two copies of both $10^*$ and $5^*$ on top of $M_1'$. The action of $u$ on the $1$-eigenspaces of the module
\[5^*,10^*,24/5^*,10^*,40_1,40_1^*,40_2,40_2^*/10,24\]
is $4^{18},3,2^2,1^5$. Removing a copy of $5^*\oplus 10^*$ removes $3^2,1$ from the action of $u$ on the $1$-eigenspace, i.e., at most three blocks. Thus there are at least $25$ blocks in the action of $u$ on the $1$-eigenspace of $L(E_8)$. But the remaining two classes have 24 blocks.

This is a final contradiction, so $10^\pm$ is neither a submodule nor not a submodule, and thus $M_1$ cannot have the structure assumed. Thus $M_1$ is semisimple, as claimed.

\medskip

\noindent \textbf{Step 9}: There is a single $5$-dimensional composition factor of $\soc(L(E_8){\downarrow_H})$.

\medskip\noindent Since the $\{5^\pm\}$-pressure of $L(E_8){\downarrow_H}$ is $2$ and there is a subquotient of $\{5^\pm\}$-pressure $4$, namely the sum of the $40_i^\pm$, the subgroup $H$ stabilizes at least one $5$-space on $L(E_8)$ by Proposition \ref{prop:pressure}. Thus assume that there are two or more.

Let $M_1$ denote the $\{5^\pm,10^\pm\}'$-heart of $L(E_8){\downarrow_H}$, and let $M_2$ denote the $\{5^\pm\}'$-heart of $L(E_8){\downarrow_H}$. We know that $M_1$ has a subquotient $40_1\oplus 40_1^*\oplus 40_2\oplus 40_2^*$, with two copies of $24$ attached somehow. We also know from Step 2 that they do not split off as summands of $M_1$, for then they would be summands of $L(E_8){\downarrow_H}$.

Thus there is a single submodule $24$ of $M_1$, and some of the $40_i^\pm$ are submodules. From Step 8, the $\{24\}$-heart is $24^{\oplus 2}$, and so any $40_i^\pm$ that is not a submodule is a quotient, and therefore there is a submodule $M_1^{(1)}$ consisting of the submodule $24$ and all $40_i^\pm$ that are not submodules of $M_1$. The dual of $M_1^{(1)}$ must be a quotient of $M_1$, and has a $24$ quotient and no other composition factors in common. We see that $M_1$ must in fact be the direct sum of $M_1^{(1)}$ and its dual $M_1^{(2)}$, and possibly some summands $40_i^\pm$. This leads, up to graph automorphism, to four possibilities: $M_1^{(1)}$ contains $40_1$, $40_2$, $40_1\oplus 40_2$, and $40_1^*\oplus 40_2$.

Suppose that $M_1$ is
\[ (40_1/24)\oplus (24/40_1^*)\oplus 40_2\oplus 40_2^*;\]
the action of $u$ on the $\theta$-eigenspace of this module has blocks $4^{10},3^6,2^2,1^6$. From Table \ref{t:classesofd7psl52}, we see that no class can encapsulate this, and so this cannot be the structure of $M_1$.

The other three possibilities for $M_1$ are
\[ 40_1\oplus 40_1^*\oplus (40_2/24)\oplus (24/40_2^*),\quad (40_1,40_2/24)\oplus (24/40_1^*,40_2^*),\]
\[(40_1^*,40_2/24)\oplus (24/40_1,40_2^*);\]
in each of these cases, $u$ acts on the $1$-eigenspace of $M_1$ with blocks $4^{16},3^2,1^2$ and on the $\theta$-eigenspace with blocks $4^{14},3^4$.

In each case, $w$ acts on $M_1$ with blocks $8^{24},5^2,3^2$, so $5$ cannot be a summand of $L(E_8){\downarrow_H}$ by the remark at the end of Step 2. In particular, this means that $L(E_8){\downarrow_H}$ has exactly two $5$-dimensional factors in the socle, so that $M_2$ contains no copy of $5^\pm$.

\medskip

If $M_2$ is the sum of $M_1$ and $10\oplus 10^*$ then in all three possibilities above $u$ acts on the $\theta$-eigenspace of $M_2$ with blocks $4^{14},3^6$. In each possibility $\Ext_{kH}^1(M_2,5)$ is $2$-dimensional. Placing both copies of $5$ on the bottom of $M_2$ yields a module whose $\theta$-eigenspace has $u$-action $4^{14},3^6,1^2$, so any extension of $M_2$ with submodule $5$ has blocks $4^{14},3^6,1$. There is no unipotent class that encapsulates this action, as seen in Table \ref{t:classesofd7psl52}, so we can exclude this structure for $M_2$.

If $M_2$ is the sum of $M_1$ and $10^*/10$, or the sum of $M_1$ and $10/10^*$, then there are six possibilities for $M_2$, and for each of them the proof is the same. Notice that $5$ and $5^*$ must both be submodules of $L(E_8){\downarrow_H}$ as we need to place a copy of $5$ below $40_1$ and a copy of $5^*$ below $40_1^*$, because of Step 1. The action of $u$ on the $1$-eigenspace of $M_2$ is $4^{16},3^4,1^4$ and on the $\theta$-eigenspace is $4^{15},3^4,2$. The space $\Ext_{kH}^1(M_2,5)$ is $2$-dimensional and $u$ acts on the $\theta$-eigenspace of this extension with blocks $4^{15},3^4,2,1^2$, hence on any extension of $M_2$ with submodule $5$ with blocks $4^{15},3^4,2,1$, so $21$ blocks in total. Examining Table \ref{t:classesofd7psl52}, the only class that encapsulates these is $A_3+A_2$, acting on the $1$-eigenspace of $L(E_8)$ with blocks $4^{16},3^6,2^4,1^2$. In particular, when we add the $5$ onto $M_2$, we may make no more blocks of size $4$ in the $1$-eigenspace of $u$. This restriction picks out a unique extension of $M_2$ with submodule $5$, and then a unique extension inside of the $2$-dimensional space of extensions obtained by placing $5^*$ on top of the resulting module. In each case, this results in one of the following modules, depending on $M_1$:
\[ (40_1/5)\oplus (5^*/40_1^*)\oplus (24/40_2)\oplus (40_2^*/24)\oplus (10^\pm/10^\mp),\]
\[ (40_1,40_2/5,24)\oplus (5^*,24/40_1^*,40_2^*)\oplus (10^\pm/10^\mp),\]
\[ (40_1^*,40_2/5^*,24)\oplus (5,24/40_1,40_2^*)\oplus (10^\pm/10^\mp).\]
The action of $u$ on the $\theta$-eigenspace of each of these modules has blocks $4^{15},3^4,2,1^2$.

We now must place a copy of $5^*$ below these modules; in each case there are three potential copies, and placing all of them below we obtain a $u$-action on the $\theta$-eigenspace of $4^{15},3^4,2,1^5$, so any extension by $5^*$ must have action $4^{15},3^4,2,1^3$, which is not encapsulated by the class $A_3+A_2$. This contradiction proves that $M_2$ cannot have this form.

\medskip

Thus $M_2$ is obtained from $M_1$ via a non-split extension involving $10$ and $10^*$. In all three possibilities $\Ext_{kH}^1(M_1,10^\pm)$ is $1$-dimensional, as is $\Ext_{kH}^1(M_1,5)$, and $\Ext_{kH}^1(M_1,5^*)$ is $2$-dimensional.

First we assume that there is a submodule $10^\pm$ in $L(E_8){\downarrow_H}$, so that the $\{5^\pm,10^\pm\}$-radical of $L(E_8){\downarrow_H}$ is semisimple. It is either $5\oplus 5^*\oplus 10^\pm$ or $5^*\oplus 5^*\oplus 10^\pm$. Since there are four possibilities for this radical, and three for $M_1$, there are twelve possibilities in total, but the same proof works for all of these. Since $\Ext_{kH}^1(M_1,10^\pm)$, $\Ext_{kH}^1(M_1,5)$ and $\Ext_{kH}^1(M_1,5^*)$ are $1$-, $1$- and $2$-dimensional respectively, we may form an extension $M_3'$ of $M_1$ with submodule $10^\pm\oplus 5\oplus 5^*\oplus 5^*$, and the $\{5^\pm,10^\pm\}$-residual $M_3$ of $L(E_8){\downarrow_H}$ is the quotient of $M_3'$ by a single $5$ or $5^*$. The twelve possibilities for $M_3$ become six for $M_3'$.

In all six possibilities, $u$ acts on the $1$-eigenspace of $M_3'$ with blocks $4^{17},3^4,2,1^3$ and on the $\theta$-eigenspace with blocks $4^{15},3^3,2,1^3$. If we remove a single $5$ or $5^*$ from the socle, we remove a block of size $3$ from the $1$-eigenspace and a block of size $1$ from the $\theta$-eigenspace. Thus $M_3$ has at least $21$ blocks in its $\theta$-eigenspace, and at least twenty blocks of size at least $3$ in its $1$-eigenspace: the only classes that encapsulate these actions of $u$ are $A_3+A_2$, $A_3+A_2^{(2)}$ and $D_4(a_1)+A_1$, whence $u$ acts on the $1$-eigenspace of $L(E_8)$ with blocks $4^{16},3^6,2^6,1^2$ or $4^{18},3^2,2^6,1^2$.

To obtain a conclusion, note that $u$ acts on the $1$-eigenspace of the extension of $M_1$ by $10$ with blocks $4^{17},3,2,1^3$ (already eliminating the first possible action), on $10\oplus 5$ with blocks $4^{17},3^2,2,1^3$, and on the $1$-eigenspace of the extension of $M_1$ by $10\oplus 5^*\oplus 5^*$ with blocks $4^{17},3^3,2,1^3$, so in fact the action of $u$ on the $1$-eigenspace of $M_3$ has blocks $4^{17},3^3,2,1^3$. This is $23$ blocks in total, five short of the number of blocks needed for $L(E_8)$. However, the $1$-eigenspace of the quotient $L(E_8){\downarrow_H}/M_3$ (which is the sum of $10^\mp$ and two copies of $5$ or $5^*$) has $u$-action $3^3,1$, which is only four blocks in total. This contradiction means that the $\{5^\pm,10^\pm\}$-radical cannot be semisimple, i.e., $10^\pm$ is not a submodule of $L(E_8){\downarrow_H}$.

\medskip

Hence, the $\{5^\pm,10^\pm\}$-radical must contain a summand $10^\pm/5^\mp$, with the remaining summand being either $5$ or $5^*$. Since $u$ always acts on the $1$-eigenspace of $M_1$ with blocks $4^{16},3^2,1^2$ and on the $1$-eigenspace of $10^\pm/5^\mp$ with blocks $4,2,1$, the class containing $u$ must have $1$-eigenspace encapsulating $4^{18},3^2,1^2$. Since we have $5$ and $5^*$ left to add, one must be able to add two more blocks of size $3$ to this, so there must be either more than eighteen blocks of size $4$ or more than two blocks of size $3$. This leaves only $2A_3$ and $D_4(a_1)+A_2$.

We have to be careful now because the actions \emph{almost} work. If the radical is $5\oplus (10^*/5)$, then we take $M_1$, add the single copy of $10^*$ below, add the two copies of $5$ below, add the two copies of $10$ above, and then the three copies of $5^*$ above that. (These are the full extensions in all cases.) This constructs a module $M_4$ that is a pyx for $L(E_8){\downarrow_H}$, in fact an extension of $L(E_8){\downarrow_H}$ with submodule $5^*/10$; $u$ acts on the $1$- and $\theta$-eigenspaces of $M_4$ with blocks $4^{23},2^3,1$ and $4^{18},3^3,1$ respectively. The former does not encapsulate the $1$-eigenspace of $u$ on $L(E_8)$ if $u$ lies in class $D_4(a_1)+A_2$, and the latter does not encapsulate the $\theta$-eigenspace of $u$ on $L(E_8)$ if $u$ lies in class $2A_3$, so we obtain a contradiction.

\medskip

Suppose that the $\{5^\pm,10^\pm\}$-radical of $L(E_8){\downarrow_H}$ is $5^*\oplus (10/5^*)$. If $M_1$ is the third possibility then one cannot place $10$ below $M_1$ and $5^*$ below the $10$, so we may exclude this case. In the other two possibilities, we place $10$ below $M_1$, then three copies of $5^*$ below that, then two copies of $10^*$ above $M_1$, then four copies of $5$ above that, to create a pyx $M_5$ for $L(E_8){\downarrow_H}$. To obtain $L(E_8){\downarrow_H}$ from $M_5$, we must remove a copy of $5^*$ from the bottom of it and $5\oplus (5/10^*)$ from the top of it. The action of $u$ on the $1$-eigenspace of $M_5$ has blocks $4^{23},3^2,2^3,1$, and on the $\theta$-eigenspace it has blocks $4^{18},3^3,1^3$, which does not encapsulate the class $2A_3$, so $u$ must lie in $D_4(a_1)+A_2$. This class acts on the $1$-eigenspace of $L(E_8)$ with blocks $4^{20},3^4$. Note that removing the quotient $5/10^*$ from $M_5$ must remove a block of size $4$ from the action of $u$ on the $1$-eigenspace, and that leaves $u$ acting with exactly $24$ blocks of size at least $3$. Thus the $5$ that is removed from the top, and the $5^*$ that is removed from the bottom of $M_5$, cannot alter the number of blocks of size at least $3$ in this action. However, removing all three copies of $5$ from the bottom of $M_5$ removes three blocks of size at least $3$ from the action of $u$, so removing one must always remove one. Thus two copies of $5^*$ cannot lie in the socle.

\medskip

The third case is where the radical is $5\oplus (10/5^*)$. In this case $\Ext_{kH}^1(M_1,5)$ is $1$-dimensional, so we may place a single $5$ below, and a single $5^*$ above, $M_1$, to make a new module $M_6$, which is a subquotient of $L(E_8){\downarrow_H}$. The presence of the $5^*$ in the top of $M_6$ means that $\Ext_{kH}^1(M_6,10)$ has dimension $2$ rather than $1$, so we add both copies of $10$ to the bottom of $M_6$, then three copies of $5^*$. One of the copies of $10$ remains a submodule of this extension (for the last possibility of $M_1$ both do, so we already obtain a contradiction), so we may remove it as we can identify which $10$ belongs in $L(E_8){\downarrow_H}$. To obtain a subquotient of $L(E_8){\downarrow_H}$ from this module, we must remove two copies of $5^*$, so two blocks of size $2$ from the $\theta$-eigenspace. The $\theta$-eigenspace of this module has blocks $4^{16},3^2,2,1^4$, so the $\theta$-eigenspace of the subquotient has at least 21 blocks in it. However, the $\theta$-eigenspace of $u$ on $L(E_8)$ must have 20 blocks, since we have shown that $u$ lies in class $2A_3$ or $D_4(a_1)+A_2$, both of which act with 20 blocks. This contradiction eliminates this possibility for the radical as well.

\medskip

We are left with $5^*\oplus (10^*/5)$ as the socle. We follow exactly the same strategy. We write out the previous paragraph with altered text in italics.

In this case $\Ext_{kH}^1(M_1,5^*)$ is \emph{now} $2$-dimensional, so we may place \emph{two copies of} $5^*$ below, and \emph{two copies of} $5$ above, $M_1$, to make a new module $M_7$, which is \emph{not quite} a subquotient of $L(E_8){\downarrow_H}$. The presence of the \emph{two copies of} $5$ in the top of $M_7$ means that $\Ext_{kH}^1(M_7,10^*)$ \emph{now} has dimension $3$ rather than $1$, so we add \emph{all three} copies of $10^*$ to the bottom of $M_7$, then \emph{two} copies of $5$. \emph{Two} of the copies of $10^*$ remain a submodule of this extension (\emph{including in the third case}), so we may remove \emph{them} as we can identify which $10^*$ belongs in $L(E_8){\downarrow_H}$. To obtain a subquotient of $L(E_8){\downarrow_H}$ from this module, we must remove \emph{one copy of $5^*$ and two copies of $5$}, so \emph{three} blocks of size $2$ from the $\theta$-eigenspace. The $\theta$-eigenspace of this module \emph{now} has blocks $4^{16},3^2,2,1^5$, so the $\theta$-eigenspace of the subquotient has at least 21 blocks in it. However, the $\theta$-eigenspace of $u$ on $L(E_8)$ must have 20 blocks, since we have shown that $u$ lies in class $2A_3$ or $D_4(a_1)+A_2$, both of which act with 20 blocks. This contradiction eliminates this possibility for the radical as well.

\medskip

Having eliminated all four possible radicals, the proof of the step is complete.

\medskip

\noindent \textbf{Step 10}: The final contradiction.

\medskip\noindent From the previous steps, inside $L(E_8){\downarrow_H}$ we have a submodule $10^\pm$, a single $5$ or $5^*$ in the socle, and two $5$-dimensional factors in the socle modulo the $10^\pm$. We also have three possible modules $M_1$, the $\{5^\pm,10^\pm\}'$-heart of $L(E_8){\downarrow_H}$. Since $10^\pm$ is a submodule, the $\{10^\pm\}'$-heart of $L(E_8){\downarrow_H}$ must have both $5$ and $5^*$ as a submodule, as we have seen in previous steps. In all cases, $\Ext_{kH}^1(M_1,5)$ is $1$-dimensional, as we saw in Step 9, and $\Ext_{kH}^1(M_1,5^*)$ is $2$-dimensional. Construct the module $M_1''$ in all three cases by extending by $5\oplus 5^*\oplus 5^*$.

By Step 9, there can be only one $5^\pm$ in the socle of $L(E_8){\downarrow_H}$, but there must be two $5$-dimensional factors below the $40_1^\pm$ to comply with Step 1. Thus the $10^\pm$ that is a submodule must cover one of the three copies of $5^\pm$ we have placed under $M_1$. However, in all cases $\Ext_{kH}^1(M_1,10^\pm)=\Ext_{kH}^1(M_1'',10^\pm)$, so no copy of $10^\pm$ can be placed underneath any of the $5^\pm$. 

This final contradiction means that $u$ does indeed lie in class $2A_2$, the action of $H$ on $L(E_8)$ is as described in Step 2 (this can be found diagonally in $A_4A_4$) and $H$ is strongly imprimitive, as needed. Thus we have the following.

\begin{proposition}\label{prop:sl52inE8} Let $H\cong \PSL_5(2)$. If $H$ is a subgroup of $\mb G\cong E_8$ in characteristic $2$, then $H$ is strongly imprimitive.
\end{proposition}

\section{\texorpdfstring{$\PSL_3(5)\leq E_8$}{PSL(3,5) < E8}}
\label{sec:diffpsl35}

The case in Proposition \ref{prop:sl3ine8} left over for $q=5$ is where the composition factors on $L(E_8){\downarrow_H}$ are
\[ 39,39^*,35,35^*,18,18^*,15_2,15_2^*,8^2,6,6^*,3,3^*.\]
The proof of this will take some time (although it is easier than the last one!), so we proceed in steps. We start with a few radicals: we take the $\cf(L(E_8){\downarrow_H})\setminus\{V^\pm\}$-radicals of $P(V)$, for $V=3,15_2,18,35,39$:
\begin{equation}\begin{gathered}
35^*/6^*,8,8,15_2^*/18^*,39,39^*/3,\quad 18/3^*,6,8/35,39^*/15_2,\\ 15_2/35,39^*/3^*,6,8/18,\quad  3^*,6,8/18,39^*/6,15_2/35,\quad 18^*,35^*/3,3^*,6^*,8,15_2^*/39.\end{gathered}\label{eq:fiveradicalspsl35}\end{equation}
We will now make a few specific observations that will be useful in the balance of the proof.

\medskip

\noindent\textbf{Step 1}: The socle of $L(E_8){\downarrow_H}$ contains a module not of dimension $3$, $15$ or $18$.

\medskip\noindent Assume this is not the case, and note that if we have two factors of $\soc(L(E_8){\downarrow_H})$ of the same dimension then they are both summands. We remove any simple summands from $L(E_8){\downarrow_H}$ to make a module $W$, whose socle has at most one factor each of dimension $3$, $15$ and $18$. From these radicals we see that the socle of $W$ cannot be simple, as there are not enough $6$-dimensional factors in the radicals above.

If $3^\pm$ lies in the socle then the $35^\mp$ in the top of the radical in (\ref{eq:fiveradicalspsl35}) cannot occur in any pyx for $W$. Thus in order to have both $35$ and $35^*$ in $W$, the socle of $W$ contains $15_2\oplus 18^*$ or its dual. Note that the third radical above is obtained from the second radical by taking duals and then the graph automorphism.

If $W$ has socle $15_2\oplus 18^*$ then $W$, and indeed $L(E_8){\downarrow_H}$, is the sum of the second radical and its dual. However, an element $u$ from the smallest class of elements of $H$ of order $5$ acts on this module with Jordan blocks $5^{40},4^{10},1^8$, which does not appear in \cite[Table 9]{lawther1995} (see also Table \ref{t:unipe8p5}), so the socle of $W$ cannot have at most two composition factors.

Thus the socle of $W$ contains three factors, one each of dimensions $3$, $15$ and $18$. We consider the $\{6^\pm,8,35^\pm,39^\pm\}$-radicals of the $P(V)$ for $V=3,15_2,18$, yielding
\[ 8/39,39^*/3,\quad 6,8/35,39^*/15_2,\quad 35/6,8/18;\]
then place as many copies of $3^\pm$, $15_2^\pm$ and $18^\pm$ (except for $V$ itself) on top of these modules, yielding
\[ 8,15_2^*/18^*,39,39^*/3,\quad 3^*,6,8/35,39^*/15_2,\quad 35/3^*,6,8/18;\]
and then remove all quotients that are simple modules other than $3^\pm$, $15_2^\pm$ and $18^\pm$, to obtain the modules
\[ 15_2^*/18^*,39/3,\quad 3^*/39^*/15_2,\quad 3^*/18.\]
As there is no $35^\pm$ in these, we obtain a final contradiction. (For the second module one may also add on a copy of $18$, but this will extend the $3^*$ that we added on, and so the $3^*$ would not be in the top.)

\medskip

\noindent \textbf{Step 2}: The centralizer of an element of order $4$ in $H$.

\medskip\noindent Let $z$ be an element of order $4$ in $H$ with centralizer $L\cong \GL_2(5)$. The element has trace $0$ on $L(E_8)$, so lies in the class with centralizer $\mb X=D_3D_5\leq D_8$ (see \cite[Table 1.16]{frey1998a} or Table \ref{t:semie8}). The action of $D_3D_5$ on $M(D_8)$ is as $M(D_3)\oplus M(D_5)$, where $M(D_3)$ is $6$-dimensional.

The eigenspaces of $z$ on $L(E_8)$ are modules for $\mb X$, and it is easy to see that the $1$-eigenspace of $z$ is $\Lambda^2(M(D_3))\oplus \Lambda^2(M(D_5))$, the $(-1)$-eigenspace is $M(D_3)\otimes M(D_5)$, and the $(\pm\I)$-eigenspaces are the tensor products of half-spin modules for $D_3$ and $D_5$ (up to duality, and we will see that this doesn't matter in our case).

Since $L\leq C_{\mb G}(z)=\mb X$, we see that each of these eigenspaces is a module for $L$. The simple modules for $L$ that have $z$ acting as $\pm 1$ have dimensions $1$, $3$ and $5$, two of each dimension, and will be labelled $1^+$, $1^-$, $3^+$, $3^-$, $5^+$ and $5^-$, so that
\[ P(1^+)=1^+/3^+/1^+,\qquad P(3^+)=3^+/1^+,3^-/3^+.\]
(There is such a labelling for both sets of modules.) Similarly, the simple modules for $L$ that have $z$ acting as $\pm \I$ have dimensions $2$ and $4$, two of each dimension, and will be labelled $2^+$, $2^-$, $4^+$ and $4^-$, so that 
\[ P(4^+)=4^+/2^+/4^+,\qquad P(2^+)=2^+/4^+,2^-/2^+.\]
(Again, there is such a labelling.) The choice of which is $+$ and which is $-$ is irrelevant as long as it is consistent, and we choose our labellings so that the $L$-action on the simple modules for $H$ is as follows.
\begin{center}
\begin{tabular}{ccccc}
\hline Module & $1$-eigenspace & $(-1)$-eigenspace & $\I$-eigenspace & $-\I$-eigenspace
\\ \hline $3$ & - & $1^+$ & $2^+$ & -
\\ $3^*$ & - & $1^-$ & - & $2^+$
\\ $6$ & $1^-$ & $3^-$ & $2^+$ & -
\\ $6^*$ & $1^-$ & $3^+$ & - & $2^+$
\\ $8$ & $1^+\oplus 3^-$ & - & $2^-$ & $2^-$
\\ $15_2$ & $3^+$ & $1^+\oplus 3^-$ & $2^+\oplus 4^-$ & $2^-$
\\ $15_2^*$ & $3^+$ & $1^-\oplus 3^+$ & $2^-$ & $2^+\oplus 4^-$
\\ $18$ & $3^+$ & $5^-$ & $4^-$ & $2^+\oplus 4^+$
\\ $18^*$ & $3^+$ & $5^+$ & $2^+\oplus 4^+$ & $4^-$
\\ $35$ & $3^-\oplus 5^-$ & $1^-\oplus 3^+\oplus 5^-$ & $2^-\oplus 4^+$ & $P(4^+)\oplus 2^+$
\\ $35^*$ & $3^-\oplus 5^-$ & $1^+\oplus 3^-\oplus 5^+$ & $P(4^+)\oplus 2^+$ & $2^-\oplus 4^+$
\\ $39$ & $3^+\oplus 3^-\oplus 5^+$ & $3^-\oplus 5^-$ & $P(4^-)$ & $2^-\oplus 4^+\oplus 4^-$
\\ $39^*$ & $3^+\oplus 3^-\oplus 5^+$ & $3^+\oplus 5^+$ & $2^-\oplus 4^+\oplus 4^-$ & $P(4^-)$ 
\\ \hline
\end{tabular}
\end{center}

We now determine the possible embeddings of $L'$ into $D_3$ and $D_5$: this yields actions on $M(D_3)$, $M(D_5)$ and the corresponding half-spin modules, and these can be used to produce the actions of $L'$ on each eigenspace of $z$.

The composition factors of $L'$ on $L(E_8)$ are 
\[ 5^{10},4^{20},3^{20},2^{24},1^{10},\]
and this yields two possible sets of composition factors for $M(D_8){\downarrow_{L'}}$, which are $5^2,3,1^3$ and $3^5,1$. Splitting these among $M(D_3)$ and $M(D_5)$, we find two candidates for these that have the correct composition factors on the $1$-eigenspace of $z$, namely $5^4,3^{12},1^4$. These are in the following table, including the composition factors on the half-spin modules, which are determined uniquely by the traces of elements on the minimal modules.
\begin{center}\begin{tabular}{ccccc}
\hline Case & $M(D_3)$ & $M(D_5)$ & $L(010)$ & $L(00010)$
\\ \hline $1$ & $5,1$ & $5,3,1^2$ & $4$ & $5^2,3^2$
\\ $2$ & $3^2$ & $3^3,1$ & $3,1$ & $4^2,2^4$
\\ \hline
\end{tabular}\end{center}

\noindent \textbf{Step 3}: Case 1 with $M(D_5)$ not semisimple.

\medskip\noindent Suppose that we are in Case 1, and $M(D_5)$ is not semisimple. This means that $L'$ acts on $M(D_3)$ as $5\oplus 1$ and on $M(D_5)$ as $5\oplus P(1)$, and on $L(00010)$ as $5^{\oplus 2}\oplus (3/3)$. The actions of $L'$ on the eigenspaces of $z$ are as follows: on the $1$-, $(-1)$- and $(\pm\I)$-eigenspaces $L'$ acts as
\[ P(3)^{\oplus 4}\oplus 5^{\oplus 4},\quad P(1)^{\oplus 2}\oplus P(3)^{\oplus 2}\oplus 5^{\oplus 6},\quad P(4)^{\oplus 4}\oplus P(2)^{\oplus 2}\oplus (2/2)\]
respectively. There is a unique extension of each of the first two modules to $L$ (that have the correct composition factors), namely
\[ P(3^+)^{\oplus 2}\oplus P(3^-)^{\oplus 2}\oplus (5^+\oplus 5^-)^{\oplus 2},\]
and
\[ P(1^+)\oplus P(1^-)\oplus P(3^+)\oplus P(3^-)\oplus (5^+\oplus 5^-)^{\oplus 3}.\]
The action of $L$ on the $(\pm\I)$-eigenspace has composition factors $(4^+,4^-)^5,(2^+,2^-)^6$, and there are two possible extensions to $L$:
\[ P(4^+)^{\oplus 2}\oplus P(4^-)^{\oplus 2}\oplus P(2^+)\oplus P(2^-)\oplus (2^+/2^-),\]
\[ P(4^+)^{\oplus 2}\oplus P(4^-)^{\oplus 2}\oplus P(2^+)\oplus P(2^-)\oplus (2^-/2^+).\]
(Of course, the $\I$-eigenspace is the dual of the $(-\I)$-eigenspace.)

From this we first see that $6$, $6^*$ and $8$ cannot lie in the socle of $L(E_8){\downarrow_H}$, and second that there are no simple summands of $L(E_8){\downarrow_H}$. By Step 1, we know that something other than $3$, $15_2$ and $18$ lies in the socle of $L(E_8){\downarrow_H}$, so either $39$ or $35$ (up to duality).

Suppose that $39$ lies in the socle of $L(E_8){\downarrow_H}$. The modules $35$ and $39^*$ contribute $P(4^+)\oplus P(4^-)$ to the $(-\I)$-eigenspace, and so no other module with $4^+$ or $4^-$ as a submodule of its $(-\I)$-eigenspace can lie in the socle of $L(E_8){\downarrow_{H}}$. This means that none of $15_2^*$, $18$, $18^*$ and $35^*$ can lie in the socle. Similarly, using the $(-1)$-eigenspace we can also exclude $15_2$. This leaves $3$, $3^*$ and $35$ as possibilities for other composition factors of $\soc(L(E_8){\downarrow_H})$.

As we saw from the radical in (\ref{eq:fiveradicalspsl35}), the socle cannot simply be $39$. Assuming it does lie in the socle of $L(E_8){\downarrow_H}$, we are interested in the $\cf(L(E_8){\downarrow_H})\setminus\{V^\pm,39^\pm\}$-radical of $P(V)$, for $V$ one of $3$, $3^*$, $35$ and $39$, (i.e., submodules of the radicals in (\ref{eq:fiveradicalspsl35})) which are
\[ 18^*/3,\quad 18/3^*,\quad 6/18/6,15_2/35, \quad 18^*,35^*/3,3^*,6^*,8,15_2^*/39\]
There is only one copy of $8$ in the sum of these, so the socle of $L(E_8){\downarrow_H}$ cannot be a subset of $\{3,3^*,35,39\}$ containing $39$, i.e., $39^\pm$ does not lie in $\soc(L(E_8){\downarrow_H})$.

Thus $\soc(L(E_8){\downarrow_H})$ must contain $35$, as all other options are exhausted. The other factors of the socle are some of $3$, $15_2$ and $18$ (as we have excluded $6^\pm$, $8$ and $39^\pm$, the $(-1)$-eigenspace excludes $3^*$ and $15_2^*$, and the $\I$-eigenspace excludes $18^*$). Again, we take the radicals of $P(V)$ but exclude $35^\pm$, and this time we obtain
\[ 6^*,8,8,15_2^*/18^*,39,39^*/3,\quad 3^*,8/39^*/15_2,\]
\[ 39^*/3^*,6,8/18,\quad 3^*,6,8/18,39^*/6,15_2/35,\]
where we repeat the case $V=35$ for the reader's convenience. Only the first contains $6^*$, so $3$ must lie in $\soc(L(E_8){\downarrow_H})$. This means that $15_2$ cannot lie in the socle, as otherwise the $(-1)$-eigenspace would have two copies of $1^+$ as submodules. It also means any composition factors of the first module that lie underneath $6^*$ also must be present. The submodule generated by $6^*$ is
\[ 6^*/18^*,39/3,\]
and so in particular $18$ cannot lie in the socle either.

The sum of the two modules corresponding to $3$ and $35$ contains all composition factors needed, but all $8$s are quotients, i.e., will be quotients of $\rad(L(E_8){\downarrow_H})$. However, they are not submodules of $L(E_8){\downarrow_H}/\soc(L(E_8){\downarrow_H})$, so neither of them can exist in $L(E_8){\downarrow_H}$, which yields a contradiction. 
%
Thus this case cannot occur.

\medskip

\noindent\textbf{Step 4}: Case 2 where $M(D_5){\downarrow_{L'}}=3/1,3/3$.

\medskip\noindent Suppose that we are in Case 2, and that $L'$ acts on $M(D_5)$ as $3/1,3/3$. In this case $W=M(D_3){\downarrow_{L'}}$ is either $3\oplus 3$ or $3/3$. Since the exterior square of $P(3)=3/1,3/3$ is $P(3)^{\oplus 3}\oplus 5^{\oplus 3}$, we see that the $1$-eigenspace of the action of $z$ has one of three possible $L$-structures:
\[(5^+\oplus 5^-)^{\oplus 2} \oplus \begin{cases} P(3^+)\oplus P(3^-)\oplus P(3^+)\oplus 3^+\oplus 3^-\oplus 3^-\oplus 1^-& W=3\oplus 3,\\
P(3^+)\oplus P(3^-)\oplus P(3^-)\oplus 3^+\oplus 3^-\oplus 3^+\oplus 1^+& W=3\oplus 3,\\
 P(3^+)^{\oplus 2}\oplus P(3^-)^{\oplus 2}&W=3/3,\end{cases}\]
and in both cases the action on the $(-1)$-eigenspace is as the previous case, 
\[ P(1^+)\oplus P(1^-)\oplus P(3^+)\oplus P(3^-)\oplus (5^+\oplus 5^-)^{\oplus 3}.\]
We also consider the action of $L'$ on the half-spin modules.

For the module $L(010)$ if $L'$ acts semisimply on $M(D_3)$ then it acts as $3\oplus 1$ on $L(010)$, and (up to duality) as $3/1$ if it acts as $3/3$ on $M(D_3)$. On the other hand, the action on $L(00010)$ has composition factors $4^2,2^4$ and $u$ acts with blocks $5^3,1$ by Table \ref{t:unipd5}; the only option is (up to duality) $P(2)\oplus (2/4)$.

We want to take the tensor product of the two modules $L(010)$ and $L(00010)$: the tensor product of $P(2)$ with $3/1$ or $3\oplus 1$ is $P(4)^{\oplus 2}\oplus P(2)^{\oplus 2}$. For the rest of the module, we need the tensor product of $2/4$ with $1/3$, with $3/1$ and with $1\oplus 3$, which are
\[ P(2)\oplus P(4)\oplus 4,\qquad P(4)^{\oplus 2}\oplus (2/2),\qquad P(4)\oplus (2/4)\oplus (2,4/2),\]
respectively.

As every simple $kH$-module in $L(E_8){\downarrow_H}$ has a $2$-dimensional summand on some eigenspace of $z$, we see that there are no simple summands of $L(E_8){\downarrow_H}$, as in Step 3. If $8$ is a submodule then it must be a summand, since both $18/8$ and $39/8$ (which are the only two composition factors that $8$ has an extension with up to duality) fail to have a trivial summand in their $1$-eigenspace. But $8$ cannot be a summand, so $8$ is not a submodule in this case either. However, $6^\pm$ could in this case be a submodule. By Step 1 a module other than $3^\pm$, $15_2^\pm$ and $18^\pm$ lies in the socle.

Suppose that $39$ lies in the socle of $L(E_8){\downarrow_H}$. In all three options for the $(\pm\I)$-eigenspaces, there are exactly four $4$s in the socle and four in the top. As in Step 3, this therefore excludes any module with a $4$ in the $(-\I)$-eigenspace from being a submodule of $L(E_8){\downarrow_H}$, i.e., $15_2^*$, $18$, $18^*$ and $35^*$. The single $3^-$ in the socle of the $(-1)$-eigenspace precludes $6$ and $15_2$ from being in the socle of $L(E_8){\downarrow_H}$ as well. This leaves $3$, $3^*$, $6^*$ and $35$ as possibilities for other composition factors of $\soc(L(E_8){\downarrow_H})$.

In Step 3 we dealt with all of these except for $6^*$, so we may assume now that $6^*$ also lies in the socle. Thus we need the $\cf(L(E_8){\downarrow_H})\setminus\{V^\pm,6^\pm,39^\pm\}$-radical of $P(V)$, for $V$ one of $3$, $3^*$, $6^*$, $35$ and $39$. These are
\[ 18^*/3,\quad 18/3^*,\quad 18^*,35^*/6^*,\quad 15_2/35, \quad 3,3^*,8,15_2^*/39.\]
Again, there is a single $8$, and so this cannot occur.

If $\soc(L(E_8){\downarrow_H})$ contains $35$ (but not $39$ or $39^*$), then the same analysis as in Step 3 excludes $3^*$, $15_2^*$ and $18^*$ from the socle, and $6^*$ can be excluded via the $(-1)$-eigenspace. This leaves $3$, $6$, $15_2$ and $18$. Using Step 3, we may assume that $6$ also lies in the socle. The appropriate radicals for this case are
\[ 8,8,15_2^*/18^*,39,39^*/3,\quad 3^*,8/18,39^*/6,\quad 3^*,8/39^*/15_2,\quad 3^*,8/18,\quad 15_2/35.\]
Only the first module contains $39$, so $3$ must also lie in the socle. From the $(-1)$-eigenspace, one cannot have both $3$ and $15_2$ in the socle at the same time. Thus the $15_2^*$ in the third layer of $P(3)$ must lie in $L(E_8){\downarrow_H}$. The submodule generated by this factor is $15_2^*/39/3$, so there is also an extension $39^*/15_2$ in $L(E_8){\downarrow_H}$. But this does not appear anywhere in these modules, a contradiction.

Thus neither $39^\pm$ nor $35^\pm$ lies in the socle. We must therefore have $6$ (up to duality) in the socle. This has $3^-$ in the $(-1)$-eigenspace, so $15_2$ cannot also lie in the socle. This leaves $3$, $3^*$, $15_2^*$, $18$ and $18^*$. The appropriate radicals for the corresponding projectives are
\[ 8,8,15_2^*/18^*,39,39^*/3,\quad 8,8,15_2/18,39,39^*/3^*,
\quad 3^*,8,15_2/18,35,39^*/6,\]
\[ 3,8/35^*,39/15_2^*,\quad 3^*,8/18,\quad 3,8/18^*.\]
Since $35^*$ appears only once, we must have $15_2^*$ in the socle of $L(E_8){\downarrow_H}$. The presence of $15_2^*$ in the socle means that $3^*$ cannot be there, from the $(-1)$-eigenspace. From the $1$-eigenspace, we see that neither $18$ nor $18^*$ can be in the socle as well. But now the only module with $18^*$ in is the first, so the socle must be $3\oplus 6\oplus 15_2^*$.

Thus we add any copies of $3^*$, $6^*$ and $15_2$ on top of the modules corresponding to $3$, $6$ and $15_2^*$. However, as we said before, $8$ only has extensions with $18^\pm$ and $39^\pm$, so all $8$s in these modules will remain quotients. Since this is not allowed, we obtain a contradiction.

%
%
%
%
%
%
%
%

\medskip

\noindent \textbf{Step 5}: Case 2 where $M(D_5){\downarrow_{L'}}=(3/3)\oplus 3\oplus 1$.

\medskip\noindent In this case, $\Lambda^2(M(D_5){\downarrow_{L'}})$, which is a summand of the $1$-eigenspace of the action of $z$, is of the form
\[ P(3)\oplus 5^{\oplus 3}\oplus (3/3)\oplus (1,3/1,3)\oplus 3^{\oplus 2}.\]
Whether $M(D_3)$ is $3\oplus 3$ or $3/3$, we don't obtain another $1,3/1,3$ in the $1$-eigenspace of $z$. However, the action of $L$ rather than $L'$ now produces a contradiction: the action of $L$ on this summand is either $1^+,3^-/1^-,3^+$ or $1^-,3^+/1^+,3^-$, but neither of these is self-dual, so $L(E_8){\downarrow_L}$ is not self-dual, a contradiction since $L(E_8)$ is self-dual.

\medskip

\noindent\textbf{Step 6}: Conclusion for Case 2 in non-semisimple case.

\medskip\noindent Since the only self-dual modules with composition factors $3^3,1$ are $P(3)$, $(3/3)\oplus 3\oplus 1$ and $3^{\oplus 3}\oplus 1$ we have that $M(D_5){\downarrow_{L'}}$ is semisimple. Thus $M(D_3){\downarrow_{L'}}=3/3$ and $M(D_5){\downarrow_{L'}}=3^{\oplus 3}\oplus 1$. Thus the action of $L'$ on the $1$-eigenspace of $z$ on $L(E_8)$ is
\[ P(3)\oplus 5^{\oplus 4}\oplus 3^{\oplus 9}\oplus 1^{\oplus 3}.\]
The action of $L'$ on the half-spin modules $L(010)$ and $L(00010)$ are (up to duality) $3/1$ and $4^{\oplus 2}\oplus 2^{\oplus 4}$, and we have
\[ (3/1)\otimes 2=2,4/2,\quad (3/1)\otimes 4=P(4)\oplus (2/4).\]
In particular, $2^-$ is not a summand of the $L'$-action on the $(+\I)$-eigenspace of $z$ on $L(E_8)$, hence $8$ is not a summand of $L(E_8){\downarrow_H}$.

The module $8$ has an extension only with $18,18^*,39,39^*$ from our composition factors. The $1$-eigenspaces of $z$ on $39/8$ and $18/8$ have an $L'$-action containing the summand $3^+/1^+,3^-$, so since there is at most one such subquotient of this in the $1$-eigenspace of $L(E_8){\downarrow_{L'}}$ (and then only if the $P(3)$ is a $P(3^+)$) we must have at least one $8$ summand in $L(E_8){\downarrow_H}$. This is a contradiction, so this possibility cannot occur.

\medskip

\noindent \textbf{Step 7}: $M(D_8){\downarrow_{L'}}$ semisimple and conclusion.

\medskip\noindent The remaining case is where $M(D_8){\downarrow_{L'}}$ is semisimple, which of course splits into two cases, depending on the actions of $L'$ on $M(D_3)$ and $M(D_5)$. In Case 1, let $\mb Y$ be an algebraic $A_1$ subgroup of $D_3D_5$ acting on $M(D_3)$ as $L(4)\oplus L(0)$ and on $M(D_5)$ as $L(4)\oplus L(2)\oplus L(0)^{\oplus 2}$, so containing $L'$ and centralizing $z$. It is easy to compute the action of $\mb Y$ on the $(\pm1)$-eigenspaces of $z$, and these are
\begin{align*} \Lambda^2(M(D_5){\downarrow_{\mb Y}})\oplus \Lambda^2(M(D_3){\downarrow_{\mb Y}})&\cong \left((2/6/2)^{\oplus 2}\oplus 4^{\oplus 3}\oplus 2^{\oplus 3}\oplus 0\right)\oplus \left((2/6/2)\oplus 4\right)
\\ &\cong (2/6/2)^{\oplus 3}\oplus 4^{\oplus 4}\oplus 2^{\oplus 3}\oplus 0
\end{align*}
for the $1$-eigenspace (where we suppress the notation `$L(-)$'), and
\[ M(D_5){\downarrow_{\mb Y}}\otimes M(D_3){\downarrow_{\mb Y}}=(2/6/2)^{\oplus 2}\oplus (0/8/0)\oplus 4^{\oplus 5}\oplus 2\oplus 0^{\oplus 2}.\]
The corresponding actions of $\mb Y$ on the half-spin modules are $L(3)$ and $L(4)^{\oplus 2}\oplus L(2)^{\oplus 2}$, and the tensor product of these is
\[ (3/5/3)^{\oplus 4}\oplus (1/7/1)^{\oplus 2}\oplus 1^{\oplus 2}.\]
If $W$ is a simple submodule of one of the eigenspaces of $z$ on $L(E_8){\downarrow_{L'}}$, then $W$ is stabilized by $\mb Y$ unless $W$ has dimension $5$ and we are considering the $(-1)$-eigenspace, where one of the $5$s in the action of $L'$ lies inside the summand $L(0)/L(8)/L(0)$. Thus we immediately see that if $L(E_8){\downarrow_H}$ possesses a submodule $3^\pm$, $6^\pm$, $8$ or $15_2^\pm$ then this is stabilized by $\mb Y$, and $\gen{H,\mb Y}\neq \mb G$, so $H$ is not Lie primitive. If there is an $18^\pm$ in the socle of $L(E_8){\downarrow_H}$, then we need to be a bit more careful. As $\mb Y$ can be embedded in $D_5$ in such a way that the automorphism inducing $\GL_2(5)$ on $\gen{L',Z}$ also normalizes $\mb Y$, the $+/-$-labelling of the $5$s from $L'$ can be extended to the labelling of $L(4)$s, so there are (say) three with $+$ and two with $-$. This means that if $18^*$ is a submodule then it is stabilized by $\mb Y$, and if $18^*$ is a quotient then the kernel of the map $L(E_8){\downarrow_H}\to 18^*$ is stabilized by $\mb Y$. Either way $H$ is not Lie primitive again.

\medskip

We will show the same statement in Case 2, but this is easier. Now the subgroup $\mb Y$ acts on $M(D_3)$ as $L(2)^{\oplus 2}$ and on $M(D_5)$ as $L(2)^{\oplus 3}\oplus L(0)$. This yields actions on the $(\pm1)$-eigenspaces of $z$ of
\[ \Lambda^2(M(D_5){\downarrow_{\mb Y}})\oplus \Lambda^2(M(D_3){\downarrow_{\mb Y}})=L(4)^{\oplus 4}\oplus L(2)^{\oplus 12}\oplus L(0)^{\oplus 4}\]
and
\[ M(D_5){\downarrow_{\mb Y}}\otimes M(D_3){\downarrow_{\mb Y}}=L(4)^{\oplus 6}\oplus L(2)^{\oplus 8}\oplus L(0)^{\oplus 4},\]
so $L'$ is in fact a blueprint for the $(\pm 1)$-eigenspaces of the action of $z$. For the half-spin modules, the action of $\mb Y$ is $L(2)\oplus L(0)$ and $L(3)^{\oplus 2}\oplus L(1)^{\oplus 4}$, and the tensor product of these is
\[ (L(3)/L(5)/L(3))^{\oplus 2}\oplus L(3)^{\oplus 6}\oplus L(1)^{\oplus 10}.\]
In this case, every simple submodule of $L(E_8){\downarrow_{L'}}$ is stabilized by $\mb Y$, and therefore any copy of $3^\pm$, $6^\pm$, $8$, $15_2^\pm$ or $18^\pm$ in the socle of $L(E_8){\downarrow_H}$ is stabilized by $\mb Y$.

\medskip

Thus either $H$ is Lie imprimitive or the socle of $L(E_8){\downarrow_H}$ must only contain copies of $39^\pm$ and $35^\pm$, as all other modules restrict semisimply to $L'$. It cannot be simple, as we saw with the radicals at the start of the proof, and it cannot be $35^*\oplus 39$ as there is no $15_2$ in the sum of the corresponding radicals. Thus the socle must be $35\oplus 39$. But the $\{3^\pm,6^\pm,8,15_2^\pm,18^\pm\}$-radical of $P(35)$ is
\[ 6/18/6,15_2/35\]
and this, together with the submodule of $P(39)$, only possesses a single $8$, a contradiction. Thus there is another factor in the socle of $L(E_8){\downarrow_H}$, and therefore $H$ is Lie imprimitive. It must also be strongly imprimitive, since either an automorphism of $H$ stabilizes this factor of the socle or it interchanges it with another that is stabilized by the same positive-dimensional subgroup above.

This completes the proof of the following proposition.

\begin{proposition}\label{prop:sl35inE8} Let $H\cong \PSL_3(5)$. If $H$ is a subgroup of $\mb G\cong E_8$ in characteristic $5$, then $H$ is strongly imprimitive.
\end{proposition}

\section{\texorpdfstring{$\PSU_3(4)\leq E_8$}{PSU(3,4) < E8}}
\label{sec:diffpsu34}

Let $H\cong \PSU_3(4)$. We left open two possible sets of composition factors for $L(E_8){\downarrow_H}$ in Section \ref{sec:su3ine8}. One we still cannot deal with, but the other, Case 8 from that section, can be solved, and we do so now. The composition factors of $L(E_8){\downarrow_H}$ are
\[ 24_{12},24_{12}^*,24_{21},24_{21}^*,(9_{12},9_{12}^*)^2,(\bar 9_{12},\bar 9_{12}^*)^2,8_1^2,8_2^2,(3_1,3_1^*)^4,(3_2,3_2^*)^4,\]
and we will prove that $H$ is always strongly imprimitive.

The proof of this is long, and involves 81 cases. We proceed step by step, with some preliminaries first to set the stage. Recall that $z$ has order $5$ in $H$, and that $L$ is the centralizer of $z$ in $H$, which is $5\times \Alt(5)$. The trace of $z$ is $-2$, and so $L\leq \mb X=A_4A_4$. The principal block contribution of $L(E_8){\downarrow_L}$ has composition factors $4^4,2_1^6,2_2^6,1^8$, and this requires the composition factors of $L'$ on both $M(A_4)$ to be $2_1,2_2,1$.

Since there is no trivial composition factor in $L(E_8){\downarrow_H}$, any fixed point in $L(E_8){\downarrow_L}$ must come from a submodule $8_i$, as there are no other quotients of the permutation module $P_L$ with no trivial quotient. (We saw the structure of $P_L$ in the proof of Proposition \ref{prop:su3ine8}.) Furthermore, there are nine isomorphism types of module for $M(A_4){\downarrow_{L'}}$, which are the following:
\[2_1\oplus 2_2\oplus 1,\quad (1/2_1)\oplus 2_2,\quad (2_1/1)\oplus 2_2,\quad (1/2_2)\oplus 2_1,\quad (2_2/1)\oplus 2_1,\]
\[ 1/2_1,2_2,\quad 2_1,2_2/1,\quad 2_1/1/2_2,\quad 2_2/1/2_1.\]
We call these modules $1$ to $9$ respectively. In each case, the summands of $L(A_4){\downarrow_{L'}}$ are $4$, $2_i$, $1/2_i/1$ or $P(2_i)$, so the only summands with a trivial submodule or quotient are $1/2_i/1$.

Let $\zeta$ be a primitive $5$th root of unity, and write $V_{\zeta^i}$ for the $\zeta^i$-eigenspace of the action of $z$ on $L(E_8)$, viewed as a module for $L'$. If there are two summands $1/2_i/1$ in $V_1$ then there must be two submodules $8_i$, hence both are summands. To see this, if there is a single submodule $1/2_i/1$ in $V_1$ then there must be a submodule $8_i$, and upon removing it there is no trivial quotient of $V_1$, hence no $8_i$ quotient. This proves that every submodule $8_i$ of $L(E_8){\downarrow_H}$ is a summand.

The composition factors of $L(E_8){\downarrow_{\mb X}}$ are $(1001,0000)$ and $(0000,1001)$ for the $1$-eigenspace of $z$ (we are suppressing the `$L(-)$'), and the other four modules are
\[ (1000,0100),\quad (0100,0001),\quad (0010,1000),\quad (0001,0010).\]
These four modules can be permuted by taking duals and swapping the two $A_4$-factors, so we may assume that the first one is the restriction of the $\zeta$-eigenspace of $z$ on $L(E_8)$ to $\mb X$. With that assignment, the actions of $\mb X$ on the $\zeta$, $\zeta^2$, $\zeta^3$ and $\zeta^4$-eigenspaces must be the modules above in the same order.

The restrictions of the simple $kH$-modules to $L$ are in Table \ref{t:actioneigpsu34again}, which is the same as Table \ref{t:actioneigpsu34}, but has been reproduced here for the reader's benefit.
\begin{table}
\begin{center}
\begin{tabular}{cccccc}
\hline Module & $1$-eig. & $\zeta$-eig. & $\zeta^2$-eig. & $\zeta^3$-eig. & $\zeta^4$-eig.
\\\hline $1$ & $1$ &&&&
\\ $3_1$ & & $1$ & $2_1$ &&
\\ $3_1^*$ & & & &$2_1$&$1$
\\ $3_2$ & & & $1$ && $2_2$
\\ $3_2^*$ && $2_2$ & & $1$ &
\\ $8_1$ & $1/2_2/1$ & $2_1$ &&& $2_1$
\\ $8_2$ & $1/2_1/1$ & & $2_2$ & $2_2$ &
\\ $9_{12}$ & $2_2$ & $4$ & & $1$ & $2_1$
\\ $9_{12}^*$ & $2_2$ & $2_1$ & $1$& & $4$
\\ $\bar 9_{12}$ & $2_1$ & & $2_2$ & $4$ & $1$
\\ $\bar 9_{12}^*$ & $2_1$ & $1$ & $4$ & $2_2$ &
\\ $24_{12}$ & $4$ & $1/2_1/1$ & $P(2_1)$ & $2_2$ & $2_2\oplus 4$
\\ $24_{12}^*$ & $4$ & $2_2\oplus 4$ & $2_2$ & $P(2_1)$ & $1/2_1/1$
\\ $24_{21}$ & $4$ & $2_1$ & $1/2_2/1$ & $2_1\oplus 4$ & $P(2_2)$
\\ $24_{21}^*$ & $4$ & $P(2_2)$ & $2_1\oplus 4$ & $1/2_2/1$ & $2_1$
\\ \hline
\end{tabular}
\end{center}
\caption{Actions of $L'\cong \Alt(5)$ on each $\zeta^i$-eigenspace of the action of $z$}
\label{t:actioneigpsu34again}
\end{table}

Let $M_1$ be the action of $L'$ on the first $M(A_4)$, and $M_2$ be the action of $L'$ on the second $M(A_4)$. A priori we have nine options for each of $M_1$ and $M_2$, so $81$ possible options for the action of $L$ on $L(E_8)$ (some might be isomorphic). We write $(i,j)$ to mean that $M_1$ is module $i$ from the list above, and $M_2$ is module $j$.

Replacing $M_1$ by $M_2$ and $M_2$ by $M_1^*$ yields a cycle of length four on the $kL'$-modules $V_{\zeta^i}$, and this corresponds to taking the field automorphism of $H$ that maps $3_1\mapsto 3_2\mapsto 3_1^*\mapsto 3_2^*$. Since this operation permutes the possible modules $L(E_8){\downarrow_H}$, we may consider a single representative from each orbit under this action. It is easy to see that the action of this $4$-cycle on the $81$ pairs $(i,j)$ has twenty cycles of length $4$ and a single fixed point $(1,1)$. The particular orbit representatives that we use (to make our presentation easier) are
\[ (1,1),\;(2,2),\;(4,4),\;(8,1),\;(8,2),\;(8,3),\;(8,4),\;(8,5),\;(1,6),\;(2,4),\;(2,5),\]
\[(2,6),\;(2,7),\;(1,4),\;(1,2),\;(4,6),\;(4,7),\;(6,6),\;(6,9),\;(8,8),\;(6,8).\]
(The ordering chosen here is the order in which we solve the cases.) All but the last one are easy to exclude, and the majority of the proof will be on Case $(6,8)$.

We are also interested in the $\{9_{12}^\pm,\bar 9_{12}^\pm\}$-heart $W_0$ of $L(E_8){\downarrow_H}$. For this we take the $\{1\}'$-radical of $P(9_{12})$, then take its $\{9_{12}^\pm,\bar 9_{12}^\pm\}'$-residual. This has structure
\[ 9_{12}/3_1^*/3_2,9_{12}^*/3_1,3_1^*,8_2/3_2^*,9_{12},24_{21}^*/3_1,3_1^*/3_2,3_2^*,9_{12}^*/3_1^*,8_2/9_{12}.\]
This has three copies of $9_{12}$ in it, and we must have exactly two in $W_0$. The $\{9_{12}\}'$-residual of it has simple top, and so we cannot have the $9_{12}$ in the highest socle layer. Removing this and again taking the $\{9_{12}^\pm,\bar 9_{12}^\pm\}$-residual yields the module
\begin{equation} 9_{12}^*/3_1,8_2/3_2^*,9_{12}/3_1,3_1^*/3_2,9_{12}^*/3_1^*,8_2/9_{12}.\label{eq:912radical}\end{equation}
Notice that this has no $8_1$, no $\bar 9_{12}^\pm$ and no $24_{i,j}^\pm$. The $\{9_{12}^\pm,\bar 9_{12}^\pm\}$-heart of $L(E_8){\downarrow_H}$ must be a submodule of the sum of the four cognates of this module, so in particular it has no copies of $24_{i,j}^\pm$ in it.

The equivalent module with socle $\bar 9_{12}$ is
\[ \bar 9_{12}^*/3_2^*,8_1/3_1^*,\bar 9_{12}/3_2,3_2^*/3_1,\bar 9_{12}^*/3_2,8_1/\bar 9_{12}.\]
Notice that $W_0$ is the sum of its $\{9_{12}^\pm\}'$- and $\{\bar 9_{12}^\pm\}'$-radicals, so we may just focus on one, say $9_{12}$. If $W_0$ does not have a summand the whole of the above module, then the socle must contain either $9_{12}\oplus 9_{12}^*$ or $9_{12}^{\oplus 2}$, and in neither case is there a $9$-dimensional factor in the heart of the module. Taking the $\{3_i^\pm,8_2\}$-radical (and the $9_{12}$ in the socle) of the module above yields
\[ 3_1^*/3_2/3_1^*,8_2/9_{12}.\]
One may place a $9_{12}$ in the fifth socle layer, on top of the $3$s but not the $8_2$, or a $9_{12}^*$ in the third socle layer, on top of the $8_2$ but not any of the $3$s. Moreover, since $W_0$ is self-dual, if we see the uniserial module
\[ 9_{12}/3_1^*/3_2/3_1^*/9_{12}\]
as a subquotient then we must also have its dual.

Thus we see that the number of $3$s in this module is either six or none, and therefore the number of $3$s in $W_0$ is twelve, six or none.

\medskip

\noindent \textbf{Step 1}: Case $(1,1)$.

\medskip\noindent The first case to deal with is where $L$ acts semisimply on both $M_i$. In this case $V_1$ is
\[ (1/2_1/1)^{\oplus 2}\oplus (1/2_2/1)^{\oplus 2}\oplus 4^{\oplus 4}\oplus 2_1^{\oplus 4}\oplus 2_2^{\oplus 4}.\]
This means that all $8$-dimensional factors of $L(E_8){\downarrow_H}$ are summands, and with that restriction any $248$-dimensional module with the correct composition factors yields this module. The $V_{\zeta^i}$ for $i=1,2,3,4$ are isomorphic, and are
\[ P(2_1)\oplus P(2_2)\oplus 4^{\oplus 3}\oplus (1/2_1/1)\oplus (1/2_2/1)\oplus 2_1^{\oplus 3}\oplus 2_2^{\oplus 3}\oplus 1^{\oplus 2}.\]

Let $\mb Y$ be the $A_1$ subgroup of $A_4A_4$ acting on each $M(A_4)$ as $L(0)\oplus L(1)\oplus L(2)$. The action of $\mb Y$ on the $1$-eigenspace is
\[ (0/2/0)^{\oplus 2}\oplus (0/4/0)^{\oplus 2}\oplus 1^{\oplus 4}\oplus 2^{\oplus 4}\oplus 3^{\oplus 4}.\]
(Here, as often before, we have removed the `$L(-)$'.) The subgroup $\mb Y$ contains $L'$, and we see that the $\{1\}'$-residual as a $kL'$-module and the $\{L(0)\}'$-residual as a $k\mb Y$-module are the same subspace of $L(E_8)$. Inside this subspace, $L'$ and $\mb Y$ stabilize the same subspaces, and so $\mb Y$ stabilizes the $1$-eigenspace of $z$ on any submodule $8_i$.

The action of $\mb Y$ on each $\zeta^i$-eigenspace is also the same, and is
\[ (2/0/4/0/2)\oplus (1/5/1)\oplus (0/2/0)\oplus (0/4/0)\oplus 3^{\oplus 3}\oplus 1^{\oplus 3}\oplus 2^{\oplus 3}\oplus 0^{\oplus 2}.\]
Comparing these two modules, we see that every semisimple $kL$-submodule of $V_\zeta$ is also a semisimple $k\mb Y$-module. We can see that $\mb Y$ and $L'$ do not stabilize exactly the same subspaces, since there is a composition factor $L(5)$ for example, but $\mb Y$ will stabilize any submodule $8_i$ of $L(E_8){\downarrow_H}$.

Thus $\mb Y$ stabilizes all $8$-dimensional simple submodules of $L(E_8){\downarrow_H}$, so $H$ is strongly imprimitive by Theorem \ref{thm:intersectionorbit}.

\medskip

\noindent \textbf{Step 2}: Cases $(2,2)$ and $(4,4)$.

\medskip\noindent We deal with $(4,4)$ first. In the case $M_1\cong M_2\cong (1/2_2)\oplus 2_1$, we see that $L(E_8){\downarrow_H}$ has a summand $8_1^{\oplus 2}$, but has no submodule $8_2$. As in the previous case, we will show that these are stabilized by an $A_1$-subgroup of $A_4A_4$, and therefore $H$ is strongly imprimitive.

Let $\mb Y$ act on each $M(A_4)$ as $(L(0)/L(2))\oplus L(1)$, so that $\mb Y$ contains $L'$. The action of $\mb Y$ on the $1$-eigenspace of $z$ on $L(E_8)$ is
\[ (2/0/4/0/2)^{\oplus 2}\oplus (0/2/0)^{\oplus 2}\oplus 3^{\oplus 4}\oplus 1^{\oplus 4}.\]
As in the previous case, every $kL'$-submodule of the form $1/2_2/1$ is the restriction of a $k\mb Y$-submodule of the form $0/2/0$, so the $1$-eigenspace of $z$ on any $kH$-submodule $8_1$ of $L(E_8){\downarrow_H}$ is stabilized by $\mb Y$.

The $\zeta$-eigenspace of the actions of $L'$ and $\mb Y$ are
\[ P(2_1)\oplus P(2_2)^{\oplus 2}\oplus 4^{\oplus 3}\oplus (1/2_2/1)\oplus (1/2_2) \oplus 2_1^{\oplus 3}\oplus 1\]
and
\[(1/5/1)\oplus (2/0/4/0/2)^{\oplus 2}\oplus 3^{\oplus 3}\oplus (0/2/0)\oplus (0/2)\oplus 1^{\oplus 3}\oplus 0.\]
We see that every $kL'$-submodule of the $\zeta$-eigenspace, such that either it or whose quotient by it is simple, is stabilized by $\mb Y$. Hence, every simple $kL'$-submodule of the $\zeta$- and $\zeta^4$-eigenspaces is stabilized by $\mb Y$, and so therefore is any submodule $8_1$ of $L(E_8){\downarrow_H}$. This proves that $H$ is strongly imprimitive.

For the case $(2,2)$, where $8_2^{\oplus 2}$ is a summand of $L(E_8){\downarrow_H}$, we apply a single Frobenius automorphism to the case $(4,4)$. This means that $M_1$ and $M_2$ are both $(1/2_1)\oplus 2_2$, and $\mb Y$ acts on both $M(A_4)$ as $(L(0)/L(4))\oplus L(2)$. (Note that this embedding works, whereas the more obvious overgroup of $L'$, acting as $(L(0)/L(1))\oplus L(2)$, has a different submodule structure.)

As $\mb Y$ is simply a Frobenius twist of the previous case, the same argument holds for the $1$-eigenspace, so we now consider the $\zeta^2$-eigenspaces of $L'$ and $\mb Y$. These are
\[ P(2_1)\oplus P(2_2)\oplus 4^{\oplus 3}\oplus (1/2_2/1,2_1/1,2_1)\oplus (1/2_1/1)\oplus (2_1/1)\oplus 2_2^{\oplus 3}\]
(the dual of $2_1/1,1/2_1,2_2/1$ is one of these summands) and
\[ (4/0/4/0/4)\oplus (2/10/2)\oplus 6^{\oplus 3}\oplus (0/8/0,4/0,4)\oplus  (0/4/0)\oplus (4/0)\oplus 2^{\oplus 3}.\]
Clearly every simple submodule of the $L'$ action is stabilized by $\mb Y$, and for simple quotients, one needs only to understand the module $0/8/0,4/0,4$, which is the dual of $4/0,0/4,8/0$, so every simple quotient is also stabilized. Therefore in particular each submodule $8_2$ of $L(E_8){\downarrow_H}$ is stabilized by $\mb Y$, and $H$ is strongly imprimitive again.

\medskip

\noindent \textbf{Step 3}: Cases $(8,1)$, $(8,2)$, $(8,3)$, $(8,4)$, $(8,5)$.

\medskip\noindent We are in the case $M_1\cong 2_1/1/2_2$. The exterior square of $M_1$ is $4\oplus (1/2_1,2_2/1)$, and the tensor product of this with the dual of any of our nine modules above is
\[ P(2_1)^{\oplus 2}\oplus P(2_2)^{\oplus 2}\oplus 4^{\oplus 3}\oplus (1/2_1,2_2/1).\]
In particular it has no summand $2_i$. Since this is the $\zeta$-eigenspace, we see that $L(E_8){\downarrow_H}$ can have no simple summand whose $\zeta^2$-eigenspace is $2_2$, for example $8_2$.

The tensor products $M_1\otimes \Lambda^2(M_2)$ (i.e., the $\zeta$-eigenspace) do depend on the isomorphism type of $M_2$, but again in all cases there is no summand $2_i$. Hence there can be no summand $8_i$ in $L(E_8){\downarrow_H}$ in this case, so in particular $M_2$ cannot have a summand $2_i$ for $i=1,2$. However, in the first five possibilities for $M_2$ we do indeed have a summand $1/2_i/1$ for some $i$ in the $1$-eigenspace.

This yields a contradiction.

\medskip

\noindent \textbf{Step 4}: Case $(1,6)$.

\medskip\noindent Since $M_1$ is semisimple, we have both $1/2_i/1$ in the tensor product $M_1\otimes M_1^*$. Thus $8_1\oplus 8_2$ is a summand of $L(E_8){\downarrow_H}$. 

However, the tensor product of $\Lambda^2(M_1)$ with $M_2^*$ does not have a $2$-dimensional summand. This contradicts the fact that $8_2$ is a summand. Thus we obtain a contradiction in this case.

\medskip

\noindent \textbf{Step 5}: Cases $(2,4)$, $(2,5)$, $(2,6)$ and $(2,7)$, and Case $(1,4)$.

\medskip\noindent Let $M_1\cong (1/2_1)\oplus 2_2$. When $M_2$ is either $2_1,2_2/1$ or $1/2_1,2_2$, $V_{\zeta^2}$ has no $2$-dimensional summand. Since $M_1$ has a summand $2_2$, there must be a summand $8_2$ in $L(E_8){\downarrow_H}$, and hence a summand $2_2$ in $V_{\zeta^2}$, which yields a contradiction. A similar statement holds when $M_2\cong (1/2_2)\oplus 2_1$ or its dual, so $(2,4)$ and $(2,5)$. Now we expect to have both an $8_1$ summand and an $8_2$ summand, but in both cases the only $2$-dimensional summand of $M_1\otimes \Lambda^2(M_2)$, which is $V_\zeta$, is $2_2$, whereas we need a $2_1$ coming from the restriction of the $8_1$. Thus these two pairs cannot occur. We also can exclude $M_1$ semisimple and $M_2$ being $(1/2_2)\oplus 2_1$, so pair $(1,4)$, since in this case we obtain no $2_2$ summand in $V_{\zeta^2}$, but we should have a summand $8_2$.

\medskip

\noindent \textbf{Step 6}: Case $(1,2)$.

\medskip\noindent Let $M_1\cong 1\oplus 2_1\oplus 2_2$ and $M_2\cong (1/2_1)\oplus 2_2$. The module $V_{\zeta^2}$ is
\[ P(2_1)\oplus P(2_2)\oplus 4^{\oplus 3}\oplus (1/2_2/1/2_1)\oplus (1/2_1/1)\oplus (2_1/1)^{\oplus 2}\oplus 2_2^{\oplus 3},\]
and the $P(2_1)$ is contributed by the $24_{12}$, so can be ignored. The composition factor $24_{21}$ cannot be a submodule of $L(E_8){\downarrow_H}$ since there is no submodule $1/2_2/1$ in the module above. However, it must be a quotient: the $L$-action of the $\zeta^2$-eigenspace of the module $3_1/24_{21}$ is $2_1/1/2_2/1$, which is not a subquotient of $V_{\zeta^2}$ (excluding the $P(2_1)$). But this is the only extension between $24_{21}$ and another simple $kH$-module, we see that $24_{21}$ is a quotient of $L(E_8){\downarrow_H}$. Hence $24_{21}^*$ is a submodule. However, $24_{21}^*$ contributes a $2_1$ to $V_{\zeta^2}$, and having ignored the $P(2_1)$, there is a single submodule $2_1$ of $V_{\zeta^2}$. Therefore in the quotient module $L(E_8){\downarrow_H}/24_{21}^*$, the $1/2_2/1/2_1$ in the $\zeta^2$-eigenspace of the $L$-action becomes a $1/2_2/1$; i.e., the extension $24_{21}/3_1$ is not a quotient of $L(E_8){\downarrow_H}$, and therefore $24_{21}$ is a summand. But this is a contradiction as $24_{21}{\downarrow_L}$ is not a summand of $L(E_8){\downarrow_L}$ from the description of $V_{\zeta^2}$.

\medskip

\noindent \textbf{Step 7}: Cases $(4,6)$, $(4,7)$, $(6,6)$ and $(6,9)$.

\medskip\noindent These cases proceed very similarly to the previous one: for $(4,7)$, $V_\zeta$ has structure
\[ P(2_1)\oplus P(2_2)\oplus 4^{\oplus 3}\oplus (2_2/1,1/2_1,2_2/1)\oplus (1/2_1,2_2)\oplus (2_1/1/2_2/1)\oplus 2_1,\]
and we may ignore the $P(2_2)$ as it lies inside $24_{21}^*$. This time the lack of a subquotient $1/2_1/1/2_2$ means that $24_{12}$ must be a submodule, hence $24_{12}^*$ a quotient, and this quotient must contain the only $2_2$ in the top of $V_\zeta$. We obtain a contradiction in the same way as the previous case.

For $(4,6)$, $(6,6)$ and $(6,9)$, we highlight what needs to be changed in the previous proof. The module $V_{\zeta^2}$ is
\[ P(2_1)\oplus P(2_2)\oplus (1/2_2/1/2_1)\oplus (2_1,2_2/1)\oplus (1/2_1/1,2_1,2_2/1,2_2)\oplus 4^{\oplus 3},\]
\[ P(2_1)\oplus P(2_2)\oplus (1/2_2/1,2_1,2_2/1,2_1)\oplus (1/2_1/1,2_1,2_2/1,2_2)\oplus 4^{\oplus 3}\]
and
\[ P(2_1)\oplus P(2_2)^{\oplus 2}\oplus 4^{\oplus 3}\oplus (1/2_1,2_2/1,1/2_1,2_2)\oplus (2_1/1)\]
respectively. We ignore the $P(2_1)$ this time, and there is no subquotient $2_1/1/2_2/1$ means $24_{21}$ is a quotient, so $24_{21}^*$ is a submodule, and this removes the (unique) $2_1$ in the socle of the above module. Again, a contradiction ensues.

\medskip

\noindent \textbf{Step 8}: Case $(8,8)$.

\medskip\noindent In this case both $M_1$ and $M_2$ are $2_1/1/2_2$, and all $V_{\zeta^i}$ for $i=1,2,3,4$ are isomorphic to
\[ 4^{\oplus 3}\oplus P(2_1)^{\oplus 2}\oplus P(2_2)^{\oplus 2}\oplus (1/2_1,2_2/1),\]
with the last summand being self-dual. Clearly $24_{i,j}^\pm$ is not a submodule of $L(E_8){\downarrow_H}$ as $1/2_i/1$ is not a submodule or quotient of this, and as the $\{9_{12}^\pm,\bar 9_{12}^\pm\}$-heart of $L(E_8){\downarrow_H}$ cannot contain all the $3$-dimensionals (or the $24$-dimensionals) but does contain all $8_i$, we need a module of dimension $3$ in the socle. We assume that it is $3_1$, but the same proof works in all cases with a change of eigenspace of $z$. In this case, we remove one $P(2_1)$ from $V_{\zeta^2}$ as it lies inside $24_{12}$, and see that there is now a unique $2_1$ in the socle of $V_{\zeta^2}$, so the quotient by this $3_1$ does not have a subquotient $1/2_2/1/2_1$. Thus $24_{21}$ is a submodule of the quotient, so $24_{21}/3_1$ is a submodule of $L(E_8){\downarrow_H}$. In particular, $3_1^*/24_{21}^*$ is a quotient of $L(E_8){\downarrow_H}$---call the kernel of this quotient $W$---and the $\zeta$-eigenspace of $z$ on this module has $L'$-action $2_1\oplus 4$, i.e., the unique $2_1$ quotient in $V_{\zeta^2}$ comes from this quotient module. But now we argue as in the case $(4,7)$, and see that $24_{21}$ must be a quotient of the submodule $W$, hence a quotient of $L(E_8){\downarrow_H}$, which is a contradiction.

\medskip

\noindent \textbf{Step 9}: Case $(6,8)$ and conclusion.

\medskip\noindent This time $V_\zeta$ is as in the previous case,\[ P(2_1)^{\oplus 2}\oplus P(2_2)^{\oplus 2}\oplus (1/2_1,2_2/1)\oplus 4^{\oplus 3},\]
but $V_{\zeta^2}$ is now
\[ P(2_1)^{\oplus 2}\oplus P(2_2)\oplus (1/2_1,2_2/1,1/2_1,2_2)\oplus (2_2/1)\oplus 4^{\oplus 3}.\]
Notice that this means that each $V_{\zeta^i}$ for $i=1,2,3,4$ has a unique trivial submodule.

As with the previous case, we cannot have a submodule $24_{12}^\pm$ or $24_{21}$, or $8_1$ or $8_2$. By the same argument as the previous case, we cannot have a submodule that is $3_2^\pm$ either. However, there could be a submodule $24_{21}^*$.

Let $W$ denote the $\{9_{12}^\pm,\bar 9_{12}^\pm\}'$-radical of $L(E_8){\downarrow_H}$. (This is non-zero since we know that the socle cannot consist solely of $9$s from the start of this proof.) Note that $W$ must contain at least one $3_1^\pm$ and one $3_2^\pm$, and at least one $24_{12}^\pm$ and at least one $24_{21}^\pm$.

For some possible socles we cannot even build a radical that works. For others we will need to (try to) add $W_0$ on top to produce a contradiction.

We now work based on the socle of $W$. Since $V_{\zeta^i}^{L'}$ is $1$-dimensional, we see that the socle of $W$ is one of
\[ 24_{21}^*,\quad 24_{21}^*\oplus 3_1,\quad 24_{21}^*\oplus 3_1\oplus 3_1^*,\quad 24_{21}^*\oplus 3_1^*,\quad 3_1,\quad 3_1^*,\quad 3_1\oplus 3_1^*.\]

\noindent \textbf{The case where $\soc(W)=24_{21}^*$.}

\medskip \noindent The $\{3_i^\pm,24_{i,j}^\pm\}$-radical of $P(24_{21}^*)$ is
\[ 3_1^*/3_2/3_1^*/24_{21}^*.\]
We need more modules in $\soc(W)$, as there is only one $24$-dimensional factor in this radical. Thus $24_{21}^*$ cannot be the socle.

\medskip

\noindent \textbf{The cases where $\soc(W)=24_{21}^*\oplus 3_1$ and $24_{21}^*\oplus 3_1\oplus 3_1^*$.}

\medskip \noindent Suppose that the socle of $W$ contains $3_1\oplus 24_{21}^*$. The $\zeta^2$-eigenspace of this module has $L$-action $2_1^{\oplus 2}$, and therefore in the quotient module $L(E_8){\downarrow_H}/(3_1\oplus 24_{21}^*)$ we find no subquotient $1/2_2/1/2_1$ in the $\zeta^2$-eigenspace. But now this means---as we have seen several times up to this point---that the $24_{21}$ quotient is a summand of $L(E_8){\downarrow_H}/(3_1\oplus 24_{21}^*)$, hence a summand of $L(E_8){\downarrow_H}$, which is a contradiction.

\medskip

\noindent \textbf{The case where $\soc(W)=24_{21}^*\oplus 3_1^*$.}

\medskip \noindent We construct the $\{3_i^\pm,24_{12}^\pm\}$-radical of $P(3_1^*)$, which is
\[ 3_2^*/3_1,3_1^*,24_{12}/3_2,3_2^*/3_1^*.\]
We must have that $24_{12}$ factor in $W$, and this requires the $3_2^*$ in the second socle layer of $P(3_1^*)$. However, the submodule $3_1^*/3_2/3_1^*$ of this radical has a $\zeta^4$-eigenspace $1/2_2/1$, and this is not a submodule of $V_{\zeta^4}$ (not a quotient of $V_{\zeta}=V_{\zeta^4}^*$). Hence that $3_1^*$ in the third layer, hence the $3_2^*$ in the fourth layer, cannot lie in $W$.

Thus we have a guaranteed submodule $W'$ of $W$, $(24_{12}/3_2^*/3_1^*)\oplus 24_{21}^*$, and a pyx for $W$,
\[ (3_1,24_{12}/3_2,3_2^*/3_1^*)\oplus (3_1^*/3_2/3_1^*/24_{21}^*).\]
Note that the $\zeta^3$- and $\zeta^4$-eigenspaces of $W'$ have a trivial submodule, hence there can be no others in $L(E_8){\downarrow_H}$. In particular, this means that $9_{12}$ and $\bar 9_{12}$ cannot be submodules of $L(E_8){\downarrow_H}$.

It is not possible to place a copy of $9_{12}^*$ on top of the pyx for $W$, and so any copy of $9_{12}^*$ in the socle of $W_0$ becomes a copy in the socle of $L(E_8){\downarrow_H}$. Therefore there is at most one copy of $9_{12}^*$ in the socle of $W_0$.

If there is a copy of $9_{12}$ in the socle of $W_0$, then it must lie above $W$. It is possible to place two copies of $9_{12}$ on top of the pyx for $W$. Taking the $\{9_{12}\}'$-radicals of them, and including the $24_{12}$, we obtain the module
\[ 24_{12}/3_2^*,9_{12}/3_1^*,\quad 9_{12}/3_1^*/3_2/3_1^*/24_{21}^*.\]
The $9_{12}$ cannot lie solely in the first module as that $9_{12}$ furnishes the $\zeta^3$-eigenspace with another fixed point. Thus either the $9_{12}$ lies solely over the second module or is diagonally embedded.

One cannot extend the second copy of $9_{12}$ by $3_1^*/3_2/3_1^*$, and hence $W_0$ cannot contain a summand as in (\ref{eq:912radical}). In particular, this means that the socle of $W_0$ cannot be just $9_{12}$, and must contain both $9_{12}$ and $9_{12}^*$, or just $9_{12}^*$.

If $\soc(W_0)=9_{12}^*$ then there must be a submodule $3_2^*/3_1/9_{12}^*$ (by applying the appropriate automorphism to the module in (\ref{eq:912radical})). Adding as many copies of $3_1$ and then $3_2^*$ on top of the pyx for $W$ as possible, we manage no copies of $3_2^*$ at all. Thus $L(E_8){\downarrow_H}$ must have a submodule
\[ (24_{12}/3_2^*/3_1^*)\oplus 24_{21}^*\oplus (3_2^*/3_1/9_{12}^*).\]
The $\zeta^3$-eigenspace of this has two $L'$-fixed points though: one from the second summand and one from the third.

Thus $\soc(W_0)$ must in fact contain $9_{12}\oplus 9_{12}^*$. Note that, in the pyx for $W$, the copy of $3_2$ in the second socle layer of the summand with socle $3_1^*$ cannot occur in $W$, because that contributes another $L'$-fixed point to $V_{\zeta^2}$. So we now have a precise description of $W$, namely
\[ (24_{12}/3_2^*/3_1)\oplus (3_1^*/3_2/3_1^*/24_{21}^*).\]
Together with $9_{12}^*$, this provides the trivial submodule for $V_{\zeta^2}$, $V_{\zeta^3}$ and $V_{\zeta^4}$.

On $W$ we may place one copy of $\bar 9_{12}$, but the trivial from its $\zeta^4$-eigenspace falls into the socle of $V_{\zeta^4}$, so this cannot occur. We may also place one copy of $\bar 9_{12}^*$, and again the trivial from its $\zeta$-eigenspace falls into the socle of $V_\zeta$. In particular, this means that the socle of $W_0$ must be exactly $9_{12}\oplus 9_{12}^*\oplus \bar 9_{12}^*$.

\medskip

We now show that the $\bar 9_{12}^*$ must actually lie in the socle of $L(E_8){\downarrow_H}$. Suppose not: we add on top of $W$ the $\bar 9_{12}^*$, then the $3_2^*$ then the $3_1^*$ that is a submodule of $W_0$. Doing this with all possible copies of the modules yields
\[ (3_1^*/3_2^*/\bar 9_{12}^*,24_{12}/3_2^*/3_1^*)\oplus (3_1^*/3_2/3_1^*/24_{21}^*).\]
However, the first summand of this has $\zeta^4$-eigenspace
\[ (1/2_2/1)\oplus 4,\]
which is not a submodule of $V_{\zeta^4}$.

We also reach a contradiction if $\bar 9_{12}^*$ is a submodule. Then again we take the module $W\oplus \bar 9_{12}^*$, and then add on top the unique copy of $3_2^*$ and then the unique copy of $3_1^*$, and this now forms the module
\[ (24_{12}/3_2^*/3_1^*)\oplus (3_1^*/3_2/3_1^*/24_{21}^*)\oplus (3_1^*/3_2^*/\bar 9_{12}^*).\]
Both the first and third summand have a trivial submodule for the $L'$-action on the $\zeta^4$-eigenspace, so this cannot be a submodule of $L(E_8){\downarrow_H}$ either.

We have finally reached a complete contradiction.

\medskip

\noindent \textbf{The case where $\soc(W)=3_1$.}

\medskip \noindent The $\{3_i^\pm,24_{i,j}^\pm\}$-radical $W'$ of $P(3_1)$ is
\[3_1/3_2,3_2^*,24_{21}/3_1,3_1,3_1^*,24_{12}^*/3_2,3_2^*,24_{21}/3_1.\]
Since $W$ is a submodule of this, the $24_{12}^*$ and one of the copies of $24_{21}$ in the module above lie in $W$. Since $W'$ has simple top and two copies of $24_{21}$, the $3_1$ in the fifth socle layer of $W'$ does not lie in $W$.

We also cannot have the submodule $3_1/3_2^*/3_1$ of $W'$ in $W$ as it has a $\zeta$-eigenspace $1/2_2/1$, which isn't a submodule of $V_\zeta$. This module is a submodule of the minimal submodule of $W'$ supporting the $3_2$ in the fourth socle layer, hence that does not lie in $W$ either, and the same holds for the $24_{21}$ in the fourth socle layer as well. We may therefore replace $W'$ by the module
\[ 3_2^*/3_1,3_1^*,24_{12}^*/3_2,3_2^*,24_{21}/3_1.\]
One cannot place $9_{12}$ on top of this module, and can place a unique copy of $9_{12}^*$ on top of this module, namely
\[ 3_2^*/3_1,3_1^*,24_{12}^*/3_2,3_2^*,9_{12}^*,24_{21}/3_1.\]
This has a $2$-dimensional $L'$-fixed space on its $\zeta^2$-eigenspace: one from the $9_{12}^*$ and one from the sum of the $3_2$ and $24_{21}$. The $3_2$ supports the $24_{12}^*$, so needs to be in $W$, and so we cannot have a $9_{12}^*$ in the socle of $W_0$. Thus the socle of $W_0$ must contain exactly $9_{12}$ and modules $\bar 9_{12}^\pm$. In particular, $9_{12}$ lies in the socle of $L(E_8){\downarrow_H}$.

Since there is a single $9_{12}$ and no $9_{12}^*$ in $\soc(W_0)$, we must have that the $\{9_{12}^\pm\}$-heart of $W_0$ is the full module from (\ref{eq:912radical})
\[ 9_{12}^*/3_1,8_2/3_2^*,9_{12}/3_1,3_1^*/3_2,9_{12}^*/3_1^*,8_2/9_{12}.\]
We build this module from the ground up, on top of the module $W\oplus 9_{12}$. We add all copies of $3_1^*$, then $3_2$, then $3_1^*$, then $9_{12}$, on top of our pyx for $W$, but this does not affect the pyx at all, as there are no such extensions. Thus $L(E_8){\downarrow_H}$ contains a submodule
\[ (24_{12}^*/3_2,24_{21}/3_1)\oplus (3_1^*/3_2/3_1^*/9_{12}).\]
Both of these summands have an $L'$-fixed point on the $\zeta^2$-eigenspace, and so we obtain a contradiction in this case.

\medskip

\noindent \textbf{The case where $\soc(W)=3_1^*$.}

\medskip \noindent Notice that in the previous proof, every statement also holds if we apply a field automorphism, i.e., replace every composition factor of every module by its dual. This swaps the $\zeta$- and $\zeta^4$-eigenspaces, but neither of these has a submodule $1/2_2/1$, and swaps the $\zeta^2$- and $\zeta^3$-eigenspaces, but neither of these has $1^{\oplus 2}$ as a submodule. Thus the exact proof above still works.

\medskip

\noindent \textbf{The case where $\soc(W)=3_1\oplus 3_1^*$.}

\medskip \noindent The presence of $3_1\oplus 3_1^*$ as a submodule means that neither $\bar 9_{12}$ nor $\bar 9_{12}^*$ lies in $\soc(L(E_8){\downarrow_H})$, as we would again have two trivial submodules on the $\zeta^{\pm 1}$-eigenspace of $z$. However, (at least) one of them must lie in the socle of $W_0$, so we have to be able to place one of them on top of $W$. We will show that this is impossible in both cases, providing the final contradiction. We first must restrict the structure of $W$.

The module $W$ is a submodule of $W'$, which is 
\[ (3_2^*/3_1,3_1^*,24_{12}^*/3_2,3_2^*,24_{21}/3_1)\oplus (3_2/3_1,3_1^*,24_{12}/3_2,3_2^*,24_{21}^*/3_1^*).\]
We will assume that $24_{12}^*$ lies in $W$, as either $24_{12}$ or $24_{12}^*$ must do so, and the same proof works with all factors replaced by their duals in the other case. Thus $W$ must contain 
\[(24_{12}^*/3_2,24_{21}/3_1)\oplus 3_1^*\quad\text{or}\quad (24_{12}^*/3_2/3_1)\oplus (24_{21}^*/3_1^*).\]
Each summand of $W'$ can serve as the submodule for an extension with quotient either $\bar 9_{12}$ or $\bar 9_{12}^*$. We deal with $\bar 9_{12}$ first, and note that to support the $\bar 9_{12}$ on top, each summand of $W'$ needs the $3_2$ in the second socle layer.

We examine the $\zeta^4$-eigenspaces, which for $\bar 9_{12}$ is simply $1$, and for the two modules above are
\[ P(2_2)\oplus (1/2_1/1/2_1)\oplus 1\quad\text{and}\quad (1/2_1/1/2_2)\oplus (2_1/1)\]
respectively. Adding on the $3_2$ to the second summand in each case yields
\[ P(2_2)\oplus (1/2_1/1/2_1)\oplus (2_2/1)\quad\text{and}\quad (1/2_1/1/2_2)\oplus (2_1,2_2/1).\]
On each of these four modules we must add one or two trivial modules, but we cannot have two trivial submodules in total, and we cannot have a submodule $1/2_i/1$ for $i=1,2$. Thus we must be in the last case, where we have a submodule $1/2_1,2_2/1$. However, adding both copies of $\bar 9_{12}$ on top of this module yields the $\zeta^4$-eigenspace $(1/2_1/1/2_2)\oplus (1/2_1,2_2/1)\oplus 1$, but the second module has dual $1/2_2/1,2_1$, so this is not the right module. Therefore $\bar 9_{12}$ does not lie in the socle of $W_0$.

\medskip

For $\bar 9_{12}^*$, we may place two copies of $\bar 9_{12}^*$ on top of $W'$; for ease of computation we take for $W'$ the module
\[ (3_2,3_2^*/3_1,3_1,3_1^*,24_{12}^*/3_2,3_2^*,24_{21}/3_1)\oplus (3_2/3_1,3_1^*,3_1^*/3_2,3_2^*,24_{21}^*/3_1^*),\]
(i.e., the $\{3_i^\pm,24_{12}^*,24_{21}^\pm\}$-radical of $P(3_1)\oplus P(3_1^*)$) a module that definitely contains $W$. On this we may place two copies of $\bar 9_{12}^*$, one in the third socle layer of each summand. On this we may place two copies of $3_2^*$, one in the fourth layer of each summand. However, we cannot place any copies of $3_1^*$ on this module, so there can be no submodule $3_1^*/3_2^*/\bar 9_{12}^*$ of $W_0$.

From our work on $W_0$ at the start of this proof, this means that we must have two copies of $\bar 9_{12}^*$ in $\soc(W_0)$, hence both $\bar 9_{12}^*$ that we placed on top of our model for $W'$. We take the $\{3_i^\pm\}$-residual of this to produce a module
\[ (\bar 9_{12}^*,24_{12}^*/3_2,3_2^*,24_{21}/3_1)\oplus (\bar 9_{12}^*/3_2^*,24_{21}^*/3_1^*).\]
This is not necessarily a submodule of $L(E_8){\downarrow_H}$, but if we remove either $24_{21}$ or $24_{21}^*$ then it is. Doing the former yields a module whose $\zeta^3$-eigenspace has two trivial submodules as a $kL'$-module, and doing the latter yields a module whose $\zeta^3$-eigenspace has three $2_1$ submodules as a $kL'$-module, neither of which lies in $V_{\zeta^3}$ from the description above.

Thus we cannot find a configuration that is consistent with the $V_{\zeta^i}$, and therefore we obtain a contradiction, completing the proof.

\medskip

This final case eliminates all $21$ orbit representatives, and therefore the set of composition factors.

\newpage

\chapter{The trilinear form for \texorpdfstring{$E_6$}{E6}}
\label{chap:trilinear}
This chapter concentrates on the case where $\mb G$ is of type $E_6$ in characteristic $3$, although our methods apply to other characteristics. In this chapter we complete the proof that every copy of $\PSL_3(3)$, $\PSU_3(3)$ and $\PSL_2(8)\cong {}^2\!G_2(3)'$ in characteristic $3$ is strongly imprimitive. From the results of Chapter \ref{chap:subsine6}, we are left with $\PSL_3(3)$ and $\PSU_3(3)$ acting irreducibly on $M(E_6)$, and $\PSL_2(8)$ acting as a sum of three non-isomorphic $9$-dimensional modules on $M(E_6)$. (In the third case, the normalizer acts irreducibly on $M(E_6)$.)

This chapter uses completely different methods from the previous chapters, to prove that every such subgroup $H$ is contained in a subgroup $G_2(3)$ acting irreducibly on $M(E_6)$. Note that $S^3(M(E_6))^{\mb G}$ is $1$-dimensional, and this yields a symmetric trilinear form on $\mb G$, as is well known. Use of the trilinear form to study $\mb G$ and its subgroups was the focus of Aschbacher's series of papers \cite{aschbacherE6Vun,aschbacher1987,aschbacher1988,aschbacher1990a,aschbacher1990b}, and was also used in Magaard's thesis on $F_4$ \cite{magaardphd}. It also appears a little in work of Cohen--Wales \cite{cohenwales1997} on Lie primitive subgroups of $E_6(\mathbb{C})$.

We will not define the trilinear form explicitly, because we do not deal with a fixed group $E_6$ but several conjugates of it, in order to give a `nice' description of certain subgroups.

\section{Using the symmetric form}
\label{sec:usingtrilinear}

Suppose that we are given a $27$-dimensional vector space $V$ over an algebraically closed field $k$, and a symmetric trilinear form $f(-,-,-)$ whose symmetry group $\mb G$ is (the simply connected form of) $E_6$. An element $x\in \GL(V)$ lies in $\mb G$ if and only if, for all $u,v,w\in V$,
\[ f(ux,vx,wx)-f(u,v,w)=0.\]
Let $H$ be a finite group that we are attempting to embed in $\mb G$, and view $H$ as a subgroup of $\mb Y=\GL(V)$. If $g\in \mb Y$ and $h\in H$, then $h^g$ lies in $\mb G$ if and only if, for all $u,v,w\in V$,
\[ f(ug^{-1}hg,vg^{-1}hg,wg^{-1}hg)-f(u,v,w)=0.\]
By replacing $u,v,w$ by $ug,vg,wg$ we see that $h^g$ lies in $\mb G$ if and only if, for all $u,v,w\in V$,
\begin{equation} f(uhg,vhg,whg)-f(ug,vg,wg)=0.\label{eq:symform}\end{equation}
This version has the benefit of not having to invert $g$. In particular, given a basis $\{v_1,\dots,v_{27}\}$ for $V$, $g$ may be written as a $27\times 27$-matrix, and since $g$ is not inverted, some of the entries of the matrix can be chosen to be parameters from $k$. If this is done, then (\ref{eq:symform}) becomes a polynomial equation in those parameters of degree at most $3$.

Now suppose that $L$ is a subgroup of $H$ contained in $\mb G$, and we wish to classify the elements of the set $\mc H$ of all conjugates of $H$ via elements of $\mb Y$ that lie in $\mb G$ and still contain $L$. In order to proceed we will assume that all subgroups of $H$ isomorphic to $L$ are $H$-conjugate to $L$. This will always hold for the pairs $(H,L)$ that we consider here and in later work, so it is no restriction.

Let $g\in \mb Y$ be such that $H^g\leq \mb G$ and $L\leq H^g$. Thus $L^{g^{-1}}\leq H$, so since $L^{g^{-1}}$ and $L$ are $H$-conjugate by assumption, there exists $y\in H$ such that $L^{g^{-1}}=L^y$, i.e., $L=L^{yg}$. Clearly $H^g=H^{yg}$, and so if we are attempting to classify elements of $\mc H$, it suffices to consider $g\in N_{\mb Y}(L)$.

While elements of $C_{\mb Y}(L)$ can be written, with respect to an appropriate basis, as matrices with entries $0$ or parameters from $k$, it is less easy to do this for $N_{\mb Y}(L)$. One special case of our situation is where all automorphisms of $L$ induced by elements of $\mb Y$ are also induced by elements of $N_{\mb Y}(H)$, or in other words,
\[ N_{\mb Y}(L)=C_{\mb Y}(L)\cdot N_{N_{\mb Y}(H)}(L).\]
If this is the case then when classifying elements of $\mc H$ it suffices to consider $g\in C_{\mb Y}(L)$. Of the three groups we consider---$\PSL_3(3)$, $\PSU_3(3)$ and ${}^2\!G_2(3)'$---two of these have this extra property. Of course, if $g_1$ and $g_2$ are elements of $N_{\mb Y}(L)$ such that $g_1g_2^{-1}$ lies in $N_{\mb Y}(H)$, then $H^{g_1}=H^{g_2}$ as well, so when constructing $\mc H$ it suffices to consider cosets via the normalizer of $H$.

Our strategy is therefore as follows:
\begin{enumerate}
\item Identify a subgroup $L$ of our subgroup $H\leq \mb Y$;
\item Determine the number of $\mb G$-conjugacy classes of subgroups $L$ with the appropriate action on $V$ (perhaps the action on $L(E_6)$ is also necessary to exclude some $\mb G$-classes, but not in the cases here);
\item Construct a basis of $V$ so that one may write elements of $C_{\mb Y}(L)$ as (reasonably nice) matrices with entries parameters from the field $k$;
\item if $H=\langle L,h\rangle$, and $g\in C_{\mb Y}(L)$ (with parameters), and $u,v,w$ are basis elements of $V$, produce $3654=27\cdot28\cdot29/(3\cdot2\cdot1)$ equations
\[ f(uhg,vhg,whg)-f(ug,vg,wg)=0;\]
\item solve these equations (modulo $C_{\mb Y}(H)$) to yield all elements of $\mc H$.
\end{enumerate}
If we have to consider $N_{\mb Y}(L)$ rather than $C_{\mb Y}(L)$, as we do for $\PSL_3(3)$ in Section \ref{sec:irredpsl33}, then we have to alter this method slightly, and perform it once for every coset of $C_{\mb Y}(L)\cdot N_{N_{\mb Y}(H)}(L)$ in $N_{\mb Y}(L)$. However, if $\Aut_{\mb G}(L)$ covers these cosets then we can avoid this extra work, as all such sets $\mc H$ are related by $N_{\mb G}(L)$-conjugacy.

To solve the systems of cubic equations, we can either work directly with the equations themselves, or simply use Gr\"obner bases to quickly yield all solutions. (We do both of these in the supplementary materials.)

\medskip

We now introduce the groups $H$ and $L$ that we consider in this chapter. Let $H$ be one of the groups $\PSL_3(3)$, $\PSU_3(3)$ and ${}^2\!G_2(3)'=\PSL_2(8)$ acting as mentioned at the start of this chapter. We need the subgroup $L$, its action on $V$ and its normalizer in $\mb Y$, together with the normalizer of $H$ in $\mb Y$.

\medskip

If $H\cong \PSL_3(3)$, then $L\cong 13\rtimes 3$, the normalizer of a Sylow $13$-subgroup of $H$, unique up to $H$-conjugacy. The simple $kL$-modules are $1_1,1_2,1_2^*,3_1,3_1^*,3_2,3_2^*$, and the $3_i^\pm$ are projective. The action of $L$ on $M(E_6)$ is, up to isomorphism,
\[ (1_1/1_1/1_1)\oplus (3_1\oplus 3_1^*)^{\oplus 2}\oplus (3_2\oplus 3_2^*)^{\oplus 2}.\]
The centralizer of this in $\mb Y$ is easy to describe: it is
\[ k^\times \times k^+\times k^+\times \GL_2(k)\times \GL_2(k)\times \GL_2(k)\times \GL_2(k).\]
(The centralizer of $1/1/1$ in $\GL_3(k)$ is $k^\times \times k^+\times k^+$, and the centralizer of each $3_i^\pm\oplus 3_i^\pm$ in $\GL_6(k)$ is isomorphic to $\GL_2(k)$.)
The normalizer of $L$ in $\mb Y$ must normalize the $13$, and so in fact $N_{\mb Y}(L)$ is simply $C_\mb Y(L)\times (13\rtimes 12)$. The normalizer $N_{\mb Y}(H)$ is $k^\times \times H.2$, where the subgroup $k^\times$ is the scalar matrices, so $|\Aut_{\mb Y}(L):\Aut_{N_{\mb Y}(H)}(L)|=2$.

\medskip

If $H\cong \PSU_3(3)$, then $L\cong 4^2\rtimes \Sym(3)$, the normalizer of a `maximal $\Phi_2$-torus' in $H$ (in the language of generic groups of Lie type). This is again unique up to conjugacy in $H$. The simple $kL$-modules are $1_1,1_2,3_1,3_2,3_3^\pm,3_4^\pm,6$. The action of $L$ on $M(E_6)$ is 
\[ (1_1/1_2/1_1)\oplus 3_1\oplus 3_2\oplus 3_3\oplus 3_3^*\oplus 6^{\oplus 2}.\]
As with the previous group, the centralizer of this in $\mb Y$ is easy to describe, and is 
\[ (k^\times \times k^+)\times k^\times\times k^\times\times k^\times\times k^\times\times \GL_2(k).\]
(The contribution from $1_1/1_2/1_1$ is $k^\times \times k^+$, each $3$-dimensional summand contributes $k^\times$, and the contribution from $6^{\oplus 2}$ is $\GL_2(k)$.)

The automizer of $L$ in $\mb Y$ has $L$ as a subgroup of index $2$, as can be checked in Magma as $|\Aut(L):\Inn(L)|=2$ and of course $L\cong\Inn(L)$. Thus $N_{\mb Y}(L)=C_{\mb Y}(L)\times L.2$.  The normalizer $N_{\mb Y}(H)$ is $k^\times \times H.2$, where the subgroup $k^\times$ is the scalar matrices, so $\Aut_{\mb Y}(L)=\Aut_{N_{\mb Y}(H)}(L)$.

\medskip

If $H\cong \PSL_2(8)$, then $L\cong 2^3\rtimes 7$, the normalizer of a Sylow $2$-subgroup of $H$. This is also unique up to conjugacy in $H$. The simple $kL$-modules are $1_i$ for $i=1,..,7$ and $7$. The action of $L$ on $M(E_6)$ is
\[ 7^{\oplus 3}\oplus \bigoplus_{i=2}^7 1_i.\]
The centralizer of this in $\GL_{27}(k)$ is easy to describe, and is
\[\GL_3(k)\times k^\times\times k^\times\times k^\times\times k^\times\times k^\times\times k^\times.\]
Since $\Aut(L)\cong (2^3\rtimes 7)\rtimes 3$, we see that $N_{\mb Y}(L)=C_{\mb Y}(L)\times L.3$. Similarly,
\[ C_{\mb Y}(H)=k^\times \times k^\times \times k^\times,\quad N_{\mb Y}(H)=C_{\mb Y}(H)\times H.3.\]
Thus again $\Aut_{\mb Y}(L)=\Aut_{N_{\mb Y}(H)}(L)$.

\medskip

Our next task is to prove that $L$ is unique up to $\mb G$-conjugacy, or in one case conjugacy in $\Aut(\mb G)$, i.e., allowing the graph automorphism.

\begin{lemma}\label{lem:Lconjugacy} Let $H$ be one of the subgroups $\PSL_3(3)$, $\PSU_3(3)$ and ${}^2\!G_2(3)'$ of the algebraic subgroup $G_2$ of $\mb G=E_6$. Let $L$ be a subgroup $13\rtimes 3$, $4^2\rtimes\Sym(3)$ and $2^3\rtimes 7$ of $H$ respectively. If $\bar L$ is a subgroup of $\mb G$ isomorphic to $L$, and the actions of $L$ and $\bar L$ on $M(E_6)$ are isomorphic, then $L$ and $\bar L$ are conjugate in $\mb G$ for the first two possibilities for $L$, and conjugate in $\mb G$ extended by the graph automorphism in the third case.
\end{lemma}
\begin{proof} We start with $13\rtimes 3$. If $M(E_6){\downarrow_H}$ is irreducible, then an element $x$ of order $13$ in $H$ acts with trace $1$, and is regular semisimple. Hence the centralizer of $x$ is simply a torus of $\mb G$. Inside $N_{\mb G}(\mb T)$, one sees that all powers of $x$ are conjugate, and hence $N_{\mb G}(\gen x)$ is contained in $N_{\mb G}(\mb T)$ as well. But clearly $L$ is contained in $N_{\mb G}(\gen x)$, so is contained in $N_{\mb G}(\mb T)$. The fastest way to proceed now is simply to check that there is a single conjugacy class of subgroups $L$ with the $13$ regular and the $3$ acting with blocks $3^9$ in the group $13^6\rtimes W(E_6)$.

(If verifying this in Magma, the best way to compute this is to find the two classes of elements of order $3$ in $W(E_6)$ that act projectively on $M(E_6)$, and then take power-conjugate presentations of the  two groups $13^4\rtimes 3$ with \texttt{PCGroup}. Then one easily enumerates all groups $13\rtimes 3$ and checks the Brauer character of an element of order $13$. Finally, check that all candidates are conjugate in $26^4\cdot W(E_6)$.)

\medskip

Next, we consider $L\cong 4^2\rtimes\Sym(3)$. First, the $4^2$ subgroup $L_0$ is toral as any rank-$2$ abelian subgroup of a semisimple algebraic group is \cite[Proposition 2.13(vi)]{griess1991}. From the composition factors of $H$ on $L(E_6)$ we see that the restriction of $L(E_6)$ to $L_0$ has a $6$-dimensional centralizer. Hence $C_{\mb G}(L_0)^\circ$ is a maximal torus of $\mb G$. In particular, any element of $\mb G$ that normalizes $L_0$ normalizes $C_{\mb G}(L_0)$, hence normalizes $C_{\mb G}(L_0)^\circ$, and so $L$ is contained in the normalizer of a torus. One may again use Magma to enumerate such subgroups in the normalizer of a torus and check that they are all conjugate. (This is done in the supplementary materials.)

\medskip

Suppose that $L\cong 2^3\rtimes 7$, so that $L$ acts on $M(E_6)$ with the structure
\[ 1_2\oplus 1_3\oplus 1_4\oplus 1_5\oplus 1_6\oplus 1_7\oplus 7^{\oplus 3}.\]
Since $L$ stabilizes, but does not centralize, a line on $M(E_6)$, and $L$ is a $3'$-group, we see that $L$ lies in a $D_5T_1$-Levi subgroup, but not inside the $D_5$, by Lemma \ref{lem:e6linestabs}.

Since $D_5$ acts on $M(E_6)$ with composition factors of dimensions $1$, $16$ and $10$, we see that the projection $L_1$ of $L$ onto $D_5$ acts on $10$ (semisimply) with factors of dimensions $1$, $1$, $1$ and $7$, whence it lies inside $B_3T_1$. Thus $L$ lies inside a $B_3T_2$ subgroup of the $D_4T_2$-Levi subgroup. Indeed, there is a unique conjugacy class of subgroups isomorphic to $L$ in $B_3$, necessarily irreducible on $M(B_3)$. (This can even be seen inside $\SO_7(3)$, although note that $\Omega_7(3)$ possesses two classes.)

Then we have a diagonal subgroup. Note that the image of $L$ in $B_3$ in fact lies in $G_2$ (and there is a unique class of subgroups of order $56$ in $G_2(3)$), with the subgroup $2^3$ being an exotic $2$-local subgroup of $G_2$ (with normalizer $2^3\cdot \PSL_3(2)$). The centralizer of $G_2$ is an $A_2$, and this forms the $A_2G_2$ maximal subgroup of $\mb G$.

Thus $L$ is diagonally embedded in $A_2G_2$. The projection onto $A_2$ is as $7$, and the action on $S^2(M(A_2)^*)$ is as the sum of all non-trivial simple modules. This can be seen as the action of $A_2G_2$ on $M(E_6)$ is
\[ (02,00)\oplus (10,10).\]
As is easy to compute, there is a unique subgroup of order $7$ in $A_2$ with this property.

Conjugacy classes of diagonal subgroups are determined by the automizers of the projections onto each side. The group $L$ is normalized by an element of order $3$ in $G_2$ (this is the whole of $\Out(L)$), as is the $7$ in $A_2$, but the latter is not normalized by an element of order $2$ (the graph automorphism of $A_2$ induces this). Thus there are two conjugacy classes of diagonal subgroups, swapped by the graph automorphism of $\mb G$ (which induces the graph automorphism of $A_2$).
\end{proof}

Once we have understood the embeddings of $H$, i.e., the set $\mc H$, we will prove that each element of $\mc H$ is contained inside a $G_2(3)$ subgroup of $\mb G$ that is contained in a $G_2$ subgroup. Thus $H$ is Lie imprimitive, but we need strongly imprimitive. In all cases, $N_{\mb G}(H)$ is irreducible on $M(E_6)$, and thus is contained in a unique $G_2$ subgroup $\mb X$. If $\phi\in \Aut^+(\mb G)$ normalizes $H$ then it necessarily normalizes $\mb X$, whence $\mb X$ is $N_{\Aut^+(\mb G)}(H)$-stable, and $N_{\mb G}(H)$ and $H$ are strongly imprimitive.

\section{The case \texorpdfstring{$\PSU_3(3)$}{PSU(3,3)}}
\label{sec:irredpsu33}

Now we set $H\cong \PSU_3(3)$ and $L\cong 4^2\rtimes \Sym(3)$. As we saw in the previous section, overgroups of $L$ in $\mc H $ are labelled by cosets of $C_{\mb Y}(H)$ in $C_{\mb Y}(L)$. Since $H$ acts irreducibly, the first centralizer is simply $k^\times$. The second centralizer has ten parameters, as we saw in the previous section, labelled $a,b,c_1,\dots,c_4,m_1,\dots,m_4$. The parameters $m_1,m_2,m_3,m_4$ label the entries of a matrix in $\GL_2(k)$, and so satisfy $m_1m_4\neq m_2m_3$. The parameters $a$ and the $c_i$ are all non-zero, and $b$ can be any element of $k$. The precise matrices we use appear in the supplementary materials. We choose a copy of $L$ to make the structure of $C_{\mb Y}(L)$ clear, and then conjugate $L$ into $E_6$, in particular into the normalizer of a torus. This means that we can pull the trilinear form back to our copy of $L$, and then choose an element of $H$ not in $L$ to test coefficients of the symmetric trilinear form, as in Section \ref{sec:usingtrilinear}.

Using Gr\"obner bases it is very easy to find the two solutions to the resulting equations. (All variables are fixed in terms of $m_4$ except for $m_2$, which satisfies a quadratic, yielding the two solutions.) Thus there are exactly two subgroups $H$ of $\mb G$ containing a given subgroup $L$. In the supplementary materials we also provide a direct proof using specific equations, although that proof is long-winded.

We find exactly two subgroups $H$ containing $L$. We then check that the unique $G_2(3)$ subgroup of $\mb Y$ containing each $\PSU_3(3)$ is also contained in $E_6$, by showing that it leaves the form invariant. This completes the proof of the following.

\begin{proposition} If $H\cong \PSU_3(3)$ acts irreducibly on $M(E_6)$, then $H$ is contained in a copy of $G_2(3)$ inside $\mb G$ and $H$ is strongly imprimitive.
\end{proposition}

Strong imprimitivity was shown at the end of Section \ref{sec:usingtrilinear}.

\section{The case \texorpdfstring{${}^2\!G_2(3)'$}{2G2(3)'}}
\label{sec:irred2g23}

Now we set $H\cong \PSL_2(8)$ and $L\cong 2^3\rtimes 7$. As we saw in Section \ref{sec:usingtrilinear}, overgroups of (one of the two options for) $L$ in $\mc H$ are labelled by cosets of $C_{\mb Y}(H)$ in $C_{\mb Y}(L)$. We also saw there that $C_{\mb Y}(H)$ is simply three copies of $k^\times$, and $C_{\mb Y}(L)$ has fifteen parameters. We label these $a,b,c,d,e,f,m_1,\dots,m_9$. The parameters $m_1,\dots,m_9$ label the entries of a matrix in $\GL_3(k)$, and so satisfy the determinant condition. The parameters $a,b,c,d,e,f$ are all non-zero. We perform the same analysis as in Section \ref{sec:irredpsu33}. Again, we use Gr\"obner bases for an easy proof, and a painstaking elimination using the equations directly as a fast, but difficult, second proof.

The precise matrices we use appear in the supplementary materials. We choose a copy of $L$ to make the structure of $C_{\mb Y}(L)$ clear and easy to use, and then conjugate $L$ into $E_6$, in particular into the subgroup $G_2$, as constructed in \cite{testerman1989}. This means that we can pull the trilinear form back to our copy of $L$, and then choose an element of $H$ not in $L$ to test coefficients of the symmetric trilinear form, as in Section \ref{sec:usingtrilinear}. Of course, we may choose $H$ to also lie in $G_2$, so that the identity matrix conjugates $H$ into our copy of $\mb G$.

As we saw in Section \ref{sec:usingtrilinear}, $L$ is unique up to conjugacy even in the finite group $E_6(q)$, so choosing $L$ to lie in the split Levi subgroup $D_4T_2$, any copy of $L$ has centralizer in $E_6(q)$ of order $(q-1)^2$. (Since $p=3$, $Z(\mb G)=1$, so this centralizer intersects $C_{\mb Y}(H)$ trivially.) On the other hand, of the maximal positive-dimensional subgroups of $\mb G$, $H$ only lies in the $G_2$ subgroup, whence its centralizer in $E_6(q)$ is trivial. Therefore we see at least $(q-1)^2$ different, but $E_6(q)$-conjugate, subgroups $H$ containing $L$. We prove in the supplementary materials that the set $\mc H$ contains exactly $(q-1)^2$ elements in it, and therefore they are all $E_6(q)$-conjugate.

In the supplementary materials we confirm this theoretical calculation. We apply the same method as above to the trivial element of $H$ rather than an element of $H\setminus L$. In other words, we use the trilinear form to determine when a generic element of $C_{\mb Y}(L)$ lies in $\mb G$. Indeed, doing so yields exactly $(q-1)^2$ elements, and modulo $C_{\mb Y}(H)$ they are the same as the elements that conjugate $H$ into $\mb G$ found above. This confirms that all elements of $\mc H$ are conjugate in the finite group $E_6(q)$, and hence are contained in copies of $G_2(3)$.

\begin{proposition} If $H\cong \PSL_2(8)$ acts as the sum of three $9$-dimensional simple modules on $M(E_6)$, then $H$ is contained in a copy of $G_2(3)$ inside $\mb G$ and is strongly imprimitive.
\end{proposition}

Strong imprimitivity was shown at the end of Section \ref{sec:usingtrilinear}.

(This result also appears in \cite[Theorem 29.3]{aschbacherE6Vun}.)

\section{The case \texorpdfstring{$\PSL_3(3)$}{PSL(3,3)}}
\label{sec:irredpsl33}

Now we set $H\cong \PSL_3(3)$ and $L\cong 13\rtimes 3$. As we saw in Section \ref{sec:usingtrilinear}, overgroups of $L$ in $\mc H$ are not labelled by cosets of $C_{\mb Y}(H)$ in $C_{\mb Y}(L)$, because there are $\mb Y$-automorphisms of $L$ that do not come from $\mb Y$-automorphisms of $H$. However, $N_{\mb G}(L)$ induces the whole of $\Aut(L)$ (as mentioned in the proof of Lemma \ref{lem:Lconjugacy}). Thus if $\mc H'$ denotes the $C_{\mb Y}(L)$-conjugates of $H$ that lie in $\mb G$, then all $N_{\mb Y}(L)$-conjugates of $H$ that lie in $\mb G$ are contained in subgroups $G_2(3)$ if and only if the same holds for $\mc H'$.

As we saw in Section \ref{sec:usingtrilinear}, $C_{\mb Y}(H)=k^\times$ and $C_{\mb Y}(L)$ is $19$-dimensional. It now appears difficult to work directly with the equations to determine the set $\mc H'$. In the supplementary materials we use Gr\"obner bases to easily deal with the relations.

We also provide a proof that works directly with the equations, but it uses a subtly different method, that we give a sketch of because it might be of independent interest. The space of $H$-invariant symmetric trilinear forms on $V$ is $3$-dimensional, with basis $\{f_1,f_2,f_3\}$. If $g$ is an element of $C_{\mb Y}(L)$ and $H=\langle L,h\rangle$, then we wish to classify those $g$ such that the space of $H^g$-invariant symmetric trilinear forms on $V$ contains the form determining $\mb G$. Equivalently, we can think of $g$ as a change-of-basis matrix, so we fix $h$ and allow the basis to change, to see when we can push the form of $\mb G$ into the $3$-space of $H$-invariant forms.

Thus we let $g$ be an element of $\mb Y$, and compute the trilinear form on $(ug,vg,wg)$ for $u,v,w$ in $V$. We then compare that value with the value of an arbitrary form in the $3$-space, $\alpha f_1+\beta f_2+\gamma f_3$ for some $\alpha,\beta,\gamma\in k$; if the two values coincide for all $u,v,w$ then $h^{g^{-1}}$ lies in $\mb G$.

\medskip

Both of these methods yield exactly two subgroups $H$ containing $L$ (under the action of $C_{\mb Y}(L)$, so four under the action of $N_{\mb Y}(L)$). We then check that the unique subgroup $G_2(3)$ of $\mb Y$ containing each of these conjugates of $H$ is also contained in $\mb G$, by showing that it leaves the form invariant. We have the following.

\begin{proposition} If $H\cong \PSL_3(3)$ acts irreducibly on $M(E_6)$, then $H$ is contained in a copy of $G_2(3)$ inside $\mb G$ and $H$ is strongly imprimitive.
\end{proposition}

Strong imprimitivity was shown at the end of Section \ref{sec:usingtrilinear}.

\backmatter
\newcommand{\etalchar}[1]{$^{#1}$}
\providecommand{\bysame}{\leavevmode\hbox to3em{\hrulefill}\thinspace}
\providecommand{\MR}{\relax\ifhmode\unskip\space\fi MR }
\providecommand{\MRhref}[2]{%
  \href{http://www.ams.org/mathscinet-getitem?mr=#1}{#2}
}
\providecommand{\href}[2]{#2}

\bibliographystyle{amsplain}
\bibliography{references}

\end{document}